%% file: main_v2.tex
\documentclass[12pt,reqno]{amsart}
\usepackage[margin=1in]{geometry}

\usepackage{amsmath}
\usepackage{amsthm}
\usepackage{times}
\usepackage[T1]{fontenc}
\usepackage[utf8]{inputenc}
\usepackage{numprint}
\npthousandsep{\,}
\usepackage{microtype}
\usepackage{mathrsfs}
\usepackage{latexsym}
\usepackage[dvips]{graphicx}
\usepackage{epsfig}
\usepackage{multicol}
\newcolumntype{t}{>{$}c<{$}}
\usepackage{xcolor}
\usepackage{amsmath,amsfonts,amsthm,amssymb,amscd}
\usepackage[pdfencoding=auto,psdextra,colorlinks]{hyperref}
\hypersetup{
	colorlinks,
	linkcolor={blue},
	citecolor={blue},
	urlcolor={blue},
        pdfauthor={Marcus Appleby, Steven T Flammia, Gene S Kopp},
        pdftitle={A constructive approach to Zauner's conjecture via the Stark conjectures}
}
\usepackage{cleveref}
\usepackage{tikz}
\usetikzlibrary{decorations.pathmorphing}
\tikzset{zigzag/.style={decorate,decoration=zigzag}}
\usepackage{tikz-cd} %
\usepackage{longtable} %
\usepackage{booktabs}
\usepackage{tabularx}
\usepackage{placeins}
\usepackage{cite}
\usepackage{cancel}
\usepackage{caption}
\usepackage{subcaption}
\usepackage{macros}
\usepackage{enumitem}

\begin{document}

\title[A constructive approach to Zauner's conjecture via the Stark conjectures]{A constructive approach to Zauner's conjecture\\ via the Stark conjectures}
\author{Marcus Appleby}
\address{School of Physics, University of Sydney, Sydney Australia}
\email{marcus.appleby@gmail.com}
\author{Steven T. Flammia}
\address{Department of Computer Science, Virginia Tech, Alexandria, VA, USA \& \newline \hphantom{ni}~Phasecraft Inc., Washington DC, USA}
\email{stf@vt.edu}
\author{Gene S. Kopp}
\address{Department of Mathematics, Louisiana State University, Baton Rouge, LA, USA}
\email{kopp@math.lsu.edu}



\date{\today}

\thanks{MA acknowledges the support of NSF PHY grant 2210495 and GSK acknowledges the support of NSF DMS grant 2302514. 
We thank Markus Grassl and Māris Ozols for helpful comments on an earlier draft.}

\begin{abstract} 
We propose a construction of $d^2$ complex equiangular lines in $\mathbb{C}^d$, also known as SICs or SIC-POVMs, which were conjectured by Zauner to exist for all $d$. The construction gives a putatively complete list of SICs with Weyl--Heisenberg symmetry in all dimensions $d > 3$. Specifically, we give an explicit expression for an object that we call a ghost SIC, which is constructed from the real multiplication values of a special function and which is Galois conjugate to a SIC. The special function, the Shintani--Faddeev modular cocycle, is more precisely a tuple of meromorphic functions indexed by a congruence subgroup of ${\rm SL}_2(\mathbb{Z})$. We prove that our construction gives a valid SIC in every case assuming two conjectures: the order 1 abelian Stark conjecture for real quadratic fields and a special value identity for the Shintani--Faddeev modular cocycle. The former allows us to prove that the ghost and the SIC are Galois conjugate over an extension of $\mathbb{Q}(\sqrt{\Delta})$ where $\Delta = (d+1)(d-3)$, while the latter allows us to prove idempotency of the presumptive fiducial projector. We provide computational tests of our SIC construction by cross-validating it with known solutions, particularly the extensive work of Scott and Grassl, and by constructing four numerical examples of nonequivalent SICs in $d=100$, three of which are new. We further consider rank-$r$ generalizations called $r$-SICs given by maximal equichordal configurations of $r$-dimensional complex subspaces. We give similar conditional constructions for $r$-SICs for all $r, d$ such that $r(d-r)$ divides $(d^2-1)$. Finally, we study the structure of the field extensions conjecturally generated by the $r$-SICs. If $K$ is any real quadratic field, then either every abelian Galois extension of $K$, or else every abelian extension for which 2 is unramified, is generated by our construction; the former holds for a positive density of field discriminants. 
\end{abstract}

\maketitle

\setcounter{tocdepth}{3} %
\tableofcontents

\allowdisplaybreaks

\section{Introduction}\label{sec:intro}

SICs (symmetric informationally complete positive operator valued measures, or SIC-POVMs) are complex equiangular tight frames for which the upper bound~\cite{Delsarte:1975} of $d^2$ vectors in dimension $d$ is achieved~\cite[Ch.~14]{Waldron:2018}. 
They have applications to quantum information~\cite{Zauner1999,Renes2004,Scott:2006,Chen:2015,Graydon:2016a,Shang:2018,Dai:2022,Feng:2022,Fuchs:2017,Cuffaro:2024,Szymusiak:2016,Tavakoli:2020,Horodecki:2022}, compressed sensing in radar~\cite{Herman:2009}, classical phase retrieval~\cite{Fannjiang:2020}, and the QBist approach to quantum foundations~\cite{Fuchs:2013,DeBrota:2020}. 
The Stark conjectures~\cite{Stark1,Stark2,Stark3,Stark4,Tate:1981}, by contrast, concern the properties of special values of derivatives of zeta functions in algebraic number theory. They are closely related to Hilbert's twelfth problem~\cite{Hilbert:1900}. 
It turns out that there are some connections between SICs and the Stark conjectures. 
A conjectured construction of SICs in terms of Stark units in odd prime dimensions congruent to $2 \Mod{3}$ is described in \cite{Kopp2019}, while ~\cite{Appleby:2022,Bengtsson:2024} gave a different such construction for dimensions of the form $n^2+3$ that are either prime or $4$ times a prime. 
In this paper, we extend these observations to arbitrary dimensions greater than $3$. 
In particular we show that the Stark conjectures together with a conjectural special function identity imply SIC existence in every finite dimension. 
We describe a practical method for constructing SICs numerically. 
We also describe a larger class of objects called $r$-SICs.

Let $\mcl{L}(\C^d)$ denote the $\C$-algebra of linear operators on a $d$-dimensional complex vector space $\C^d$. 
We say $\Pi \in \mcl{L}(\C^d)$ is a \textit{projector} if $\Pi^2=\Pi$. 
We will often need to contrast Hermitian and certain non-Hermitian projectors, so we introduce the following shorthand. 
\begin{defn}[\hprj]\label{dfn:Hprojector}
An \textit{\hprj} is a Hermitian projector. 
\end{defn}
A set of $n$ distinct rank-$1$ H-projectors $\{\Pi_j\}_{j=1}^n$ is called \textit{equiangular} if the Hilbert--Schmidt inner product is constant on all distinct pairs, $\Tr(\Pi_j\Pi_k) = \alpha$, for some $\alpha$ independent of $j\not=k$ but possibly depending on $d$ and $n$. 
It can be shown~\cite{Delsarte:1975} that $n\le d^2$, with a SIC being the case when $n=d^2$. 
If we drop the the rank-$1$ requirement, we obtain what we will call an $r$-SIC.
\begin{defn}[$r$-SIC]
    \label{def:equiangularcond}
    An \emph{$r$-SIC} is a set of $d^2$ distinct rank-$r$ \hprj s $\{\Pi_j\}_{j=1}^{d^2}$ in $\mcl{L}(\C^d)$ such that for all $j\neq k$ and some fixed constant $\alpha$ we have $\Tr(\Pi_j\Pi_k) = \alpha$.
\end{defn}
\begin{rmkb}
    The terminology \emph{$r$-SIC} is new, but related concepts have appeared in the literature in other contexts under different names; we review this below.
\end{rmkb}

In 1999, Zauner~\cite{Zauner1999} made the following conjecture regarding $1$-SICs. 
\begin{conj}[Zauner's Conjecture]
\label{conj:zauner}
    $1$-SICs exist for all $d$.
\end{conj}
Zauner further conjectured that $1$-SICs should have certain symmetries related to a finite-order Weyl--Heisenberg group (see \Cref{def:WHGroup}), an important point to which we will return. 

Prior work on Zauner's conjecture has proven the existence of $1$-SICs in only a finite number of dimensions $d$. 
Prior authors have constructed $1$-SICs exactly in every dimension $\le 53$ and in many further dimensions up to a maximum of $5\,799$. 
High precision numerical solutions have been calculated in every dimension $\le 193$ and in many further dimensions up to a maximum of $39\,604$. 
These results are the work of many people obtained over a period of $25$ years, starting with the original work of Hoggar~\cite{Hoggar1998} and Zauner~\cite{Zauner1999}. 
For more on the current state of knowledge, and a review of the history, see~\cite{Fuchs:2017,Grassl:2021,Grassl:2021a,Appleby:2022,Bengtsson:2024}. 

As noted above, there seem to be some intimate connections between~\Cref{conj:zauner} and an important open problem in number theory, related to Hilbert's twelfth problem, known as the Stark conjectures~\cite{Stark1,Stark2,Stark3,Stark4}. The Stark conjectures posit the existence of special algebraic units, now called \textit{Stark units}, arising from zeta functions.
In a sequence of papers~\cite{Appleby:2013,Appleby:2020,Kopp2019,Appleby:2022,Bengtsson:2024}, it was shown that, in a variety of special cases of $1$-SICs so far constructed, to very high precision, the expansion coefficients of the elements of a $1$-SIC in a natural matrix basis are proportional to powers of Stark units. 

Based on the close connections documented in prior work, one might ask the question of whether~\Cref{conj:zauner} actually follows from the Stark conjectures, or perhaps a refinement thereof. 
We partially answer that question, by showing that~\Cref{conj:zauner} (Zauner's Conjecture) follows from one of the Stark conjectures together with a related conjectural identity. 
We also show that the signed half-integral powers of the (generalized) Stark units that are needed to calculate $r$-SICs are naturally expressed in terms of a complex analytic function introduced in \cite{Kopp2020d} and defined below (see \Cref{def:shin}), which we term the \SFKFull{} modular cocycle (a generalization of a function originally introduced by Shintani in the context of algebraic number theory~\cite{Shintani1976,Shintani1977b,Shintani1980} and rediscovered by Faddeev and Kashaev in the context of high energy physics~\cite{Faddeev:1994,Faddeev:1995,Kashaev:1997,Woronowicz:2000,Faddeev:2001,Zagier:2007a,Dimofte:2010,Dimofte:2015,Murakami:2018,Closset:2019,Garoufalidis:2023}).

Specifically, we consider four conjectures that we refer to by name throughout the paper. 
The \textit{Stark Conjecture} (\Cref{conj:stark}) is a special case (for real quadratic base field and abelian $L$-functions vanishing to order $1$ at $s=0$) of the conjectures that can be extracted strictly from Stark's original series of papers \cite{Stark1,Stark2,Stark3,Stark4}.
The \textit{Stark--Tate Conjecture} (\Cref{conj:stc}) is a standard refinement of the Stark Conjecture due to Tate~\cite{Tate:1981}, 
stated in the special case we require.\footnote{Tate himself attributes that special case to Stark, but the claim that the square root of the Stark unit is in an abelian extension does not appear as a conjecture in Stark's published work.} 
The \textit{Monoid Stark Conjecture} (\Cref{conj:msc}) is a further mild refinement, involving less-studied zeta functions attached to elements of a certain monoid, not known to follow from the Stark--Tate conjecture.
The fourth conjecture is a new (and rather mysterious) identity involving special values of the \SFKFull{} modular cocycle that we call the \textit{Twisted Convolution Conjecture} (\Cref{cnj:tci}). Together, these conjectures give a remarkably precise refinement of the Stark conjectures as applied to real quadratic fields. 
We establish the following theorem. 
\begin{theorem}
\label{thm:TCCSCimpliesZauner}
The Stark Conjecture and the Twisted Convolution Conjecture together imply Zauner's conjecture. 
\end{theorem}

This result follows as a corollary of a much more precise and stronger theorem stated below, \Cref{thm:rayclassfieldrsicgen}. 
To understand the ideas behind the proof and to see how it extends to certain families of $r$-SICs for $r>1$, we need to establish a few more notions. 
The $r$-SICs we consider all carry a transitive action of the following Weyl--Heisenberg group. 
\begin{defn}[Weyl--Heisenberg group, standard basis, $\rtus_d$, $\rtu_d$, $\db$, displacement operators]
\label{def:WHGroup}
Let $|0\ra, \dots , |d-1\ra$ be the orthonormal \emph{standard basis} for $\C^d$. 
Let $X$, $Z$ be unitary operators acting as
\eag{
X |j\ra &= |j+1\ra, & Z |j\ra &= \rtus^j_d |j\ra, & \rtus_d &= e^{\frac{2\pi i}{d}},
}
where addition of indices in the first equation is performed modulo $d$. 
Also let $\rtu_d = -e^{\frac{\pi i}{d}}$ and $\db = \tfrac{d}{2}\bigl(3+(-1)^d\bigr)$, so that $\db = d$ when $d$ is odd and $\db = 2d$ when $d$ is even, and $\rtu_d$ is a $\db$-th root of unity. 
Then the \emph{Weyl--Heisenberg group} in dimension $d$, denoted $\WH(d)$, is the set of $d^2\db$ operators 
\eag{
\{\rtu_d^{p_0} X^{p_1} Z^{p_2}\colon 0 \le p_0 <\db, 0\le p_1,p_2< d\}.
}
The \emph{displacement operators} are the $d^2$ coset representatives of $\WH(d)/\la \rtu_d I\ra$
\eag{
\label{eq:displacementops}
D_{\mbf{p}} &= \rtu_d^{p_1p_2}X^{p_1}Z^{p_2}, & \mbf{p} &= \bmt p_1 \\ p_2\emt \in \mbb{Z}^2.
}
\end{defn}
\begin{rmkb}
    In the SIC literature, the root of unity we are calling $\rtu_d$ is usually denoted $\tau$. 
    This conflicts with the way $\tau$ is used in the theory of modular forms, to which we make essential appeal.
\end{rmkb}
\begin{defn}[WH-covariant, fiducial]
\label{dfn:whcovrsic}
An $r$-SIC is \emph{WH-covariant} if it is of the form $\{\Pi_{\mbf{p}}\colon 0 \le p_1,p_2< d\}$, where 
\eag{\label{eq:WHCovariantSICTermsFiducial}
\Pi\vpu{\dagger}_{\mbf{p}} &= D\vpu{\dagger}_{\mbf{p}}\Pi D^{\dagger}_{\mbf{p}}
}
for some fixed \hprj{} $\Pi$, called the \emph{fiducial} projector.
\end{defn}

\begin{rmkb}
    In this paper, we are exclusively concerned with WH-covariant $r$-SICs, and we further specialize to $d>3$, henceforth without comment. 
    The known $1$-SICs thus excluded are all in some ways exceptional and are called \emph{sporadic SICs} by Stacey~\cite{Stacey:2021}.
    It is open whether more such examples exist.
\end{rmkb}

With the above restrictions, $r$-SICs split naturally into equivalence classes via an action of the \textit{extended Clifford group}~\cite{Appleby2005}, defined later in~\Cref{ssc:ecdgp}. 
A long-standing problem has been to understand the structure of these classes for the case of $1$-SICs. 
The classes exhibit rather complicated phenomenology, as can be seen from the data tables in, e.g., 
~\cite{Scott2010,Scott:2017,Appleby:2018}. 
We summarize these empirical observations in~\Cref{subsc:sicphenomenology}.
In~\Cref{sec:sicphenexplain} we show that~\Cref{thm:rayclassfieldrsicgen} together with two additional conjectures implies that this phenomenology arises from the class structure of certain integral binary quadratic forms. 
The result is illustrated by the data tables in \Cref{ap:sicdata}. 
Also see the examples in\Cref{sc:ClassificationECdOrbits}, where, among other things, we plot the number of SIC equivalence classes in each dimension up to $d=10^6$. 
In
~\Cref{sec:sicphenexplain} we give proofs for various other aspects of the currently observed phenomenology. 

Our results also show that $r$-SICs can answer questions in number theory and explicit class field theory. 
For example, we show that, under conditional assumptions, every abelian Galois extension of $\Q(\sqrt{5})$ is contained in a field generated by the overlaps of an $r$-SIC and roots of unity. 
The field $\Q(\sqrt{5})$ can be replaced by any real quadratic field with an odd trace unit (see \Cref{thm:cofinal}), and such real quadratic fields make up a positive proportion of all real quadratic fields in the sense of asymptotic density (see \Cref{thm:oddtracecount}).

Our classification scheme and conjectures suggest a new direction to approach the Stark conjectures and Hilbert's twelfth problem for real quadratic fields. Numerical evidence suggests that the polynomial equations defining a WH-covariant $r$-SIC (when $r < \frac{d-1}{2}$) define an algebraic variety of dimension zero. 
A proof of the Twisted Convolution Conjecture would reduce many cases of the Stark conjecture to a claim about the properties of the algebraic variety of WH-covariant $r$-SICs.

\subsection{Generalizing to \texorpdfstring{$r$}{r}-SICs}\label{ssec:GeneralizingTorSICs}

We wish to generalize prior work from $1$-SICs to $r$-SICs, both because this is crucial for the construction of a large family of abelian extensions, and because the richness of the class of $r$-SICs for $r>1$ has been heretofore unappreciated. 
Although general $r$-SICs have received much less attention, they have been studied in other contexts under different names. 
They are also called maximal equichordal tight fusion frames~\cite{Casazza:2011,Fickus:2017,King:2021},
maximal symmetric tight fusion frames~\cite{Appleby:2019b}, or regular quantum designs of degree $1$ and cardinality $d^2$~\cite{Zauner1999}. 
They are instances of structures which have been variously described as SI-POVMs~\cite{Appleby:2007a}, general SIC-POVMs~\cite{Gour:2014}, and SIMs~\cite{Graydon:2016}, and they are special cases of conical designs~\cite{Graydon:2016}. 

Unlike $1$-SICs, there are some known cases where $r$-SICs are proven to exist in infinitely many dimensions. 
Firstly, it has been shown~\cite{Appleby:2007a} that in every odd dimension $d$ there exists an $r$-SIC with $r=(d-1)/2$.
Secondly, it has been shown~\cite{Appleby:2019b} that, to every $1$-SIC in odd dimension $d$ of the kind described in \Cref{dfn:whcovrsic} below, there is a corresponding $r$-SIC with $r=(d-1)/2$ (different from the one constructed in \cite{Appleby:2007a}). 
These constructions described in \cite{Appleby:2007a,Appleby:2019b} are very different from the constructions in this paper, and we do not consider them further.

The connection to the Stark conjectures is via the so-called \textit{\normalizedOverlapsText{}}, which we define below. 
To motivate their definition, we first see that the geometry of an $r$-SIC constrains the value of $\alpha$ in~\Cref{def:equiangularcond} to certain specific values.

\begin{thm}\label{thm:rsicbsc}
Let $\Pi_1, \dots , \Pi_{d^2}$ be an $r$-SIC. 
Then for all $j$, $k$,
\eag{
\Tr(\Pi_j\Pi_k) &= \left(\frac{rd(d-r)}{d^2-1}\right)\delta_{jk} + \frac{r(rd-1)}{d^2-1}.
\label{eq:rsicbscgenolp}
}
Furthermore, the $\Pi_j$ are a basis for $\mcl{L}(\C^d)$, and up to a scale factor the $\Pi_j$ form a resolution of the identity:
\eag{
\sum_{j=1}^{d^2} \Pi_j &= r d I.
\label{eq:rsicbscresid}
}
\end{thm}
\begin{proof}
This result can be proven without assuming WH-covariance or $d>3$; see \Cref{sssec:proofofrsicoverlap}. 
\end{proof}

\begin{thm}\label{thm:sicfidcond}
Let $\Pi$ be an \hprj{} in dimension $d$. 
Then $\Pi$ is a fiducial projector for an $r$-SIC if and only if
\eag{
 \bigl\lvert\Tr\bigl(\Pi D^{\dagger}_{\mbf{p}}\bigr)\bigr\rvert
 &=
 \sqrt{\frac{r(d-r)}{d^2-1}}
\label{eq:olphsedf}
}
for all $\mbf{p}\neq \zero \Mod{d}$. 
\end{thm}
\begin{proof}
    See \Cref{sssec:proofofmodulusofoverlapphase}.
\end{proof}

\begin{defn}[\overlapsText{}; \normalizedOverlapsText{}]
\label{dfn:sicovlp}
Let $\Pi$ be an $r$-SIC fiducial projector.
The numbers $\overlap_{\mbf{p}} = \Tr\bigl(\Pi D^{\dagger}_{\mbf{p}}\bigr)$ are called the \emph{\overlapsText{}}. If $\mbf{p}\not\equiv \zero \Mod{d}$, from the polar decomposition $\overlap_{\mbf{p}} = |\overlap_{\mbf{p}}| e^{i\theta_{\mbf{p}}}$ we define the \emph{\normalizedOverlapsText{}} to be the phases $\normalizedOverlap_{\mbf{p}} = e^{i\theta_{\mbf{p}}}$. 
\end{defn}
It follows from \Cref{thm:sicfidcond} that
\eag{
  \label{eq:normalizedOverlapExpression}
  \normalizedOverlap_{\mbf{p}} &= \sqrt{\frac{d^2-1}{r(d-r)}} \overlap_{\mbf{p}}
}
for $\mbf{p}\not\equiv \zero \Mod{d}$.
    The fact that $\Pi$ is Hermitian means
    \eag{
    \label{eq:normalizedOverlapIdentity}
    \normalizedOverlap_{\mbf{p}} \normalizedOverlap_{-\mbf{p}} &= 1
    }
for all $\mbf{p}$.
Since the displacement operators form a basis for $\mcl{L}(\C^d)$, a fiducial can be recovered from its \normalizedOverlapsText{} using the formula
\eag{
\Pi &= \frac{r}{d} I + \sqrt{\frac{r(d-r)}{d^2(d^2-1)}}\sum_{\mbf{p}\notin d\mbb{Z}^2} \normalizedOverlap_{\mbf{p}} D_{\mbf{p}},
\label{eq:sicfidexpn}
}
where the sum is over any set of coset representatives of $\mbb{Z}^2/d\mbb{Z}^2$ with the representative of $d\mbb{Z}^2$ excluded. 
This means one could equivalently define an $r$-SIC fiducial to be a $\Pi \in \mcl{L}(\C^d)$ that is a rank-$r$ \hprj{} where the $\normalizedOverlap_{\mbf{p}}$ in the representation of \eqref{eq:sicfidexpn} are unit complex numbers. 

The \overlapsText{} and \normalizedOverlapsText{} of known $1$-SIC solutions have been studied in great detail to extract insights that might lead to a resolution of Zauner's conjecture. 
It was realized quickly (see, e.g.,~\cite{Scott2010}) that all known \overlapsText{} are algebraic numbers (excluding, as usual, the case of $d = 3$). 
We define a field generated by these numbers, adjoining a root of unity as well to ensure independence of the choice of fiducial.
\begin{defn}[SIC field]\label{defn:sicfield}
    For a fiducial $r$-SIC projector $\Pi$, the \textit{extended projector SIC field}, or simply the \textit{SIC field}, is the field generated by the entries of $\Pi$ and the $\db$-th root of unity $\rtu_d$, or equivalently, the field generated by the overlaps along with $\rtu_d$:
    \begin{equation}
        E = E_\Pi = \Q(\rtu_d, \Pi_{ij} : 0 \leq i,j < d) = \Q(\rtu_d, \Tr(\Pi D_\p) : \p \in (\Z/d\Z)^2).
    \end{equation}
\end{defn}

In \cite{Appleby:2013} it was found (among other things) that, for known $1$-SICs, the SIC field is an abelian extension of the real quadratic field $ \mbb{Q}\bigl(\sqrt{(d+1)(d-3)}\bigr)$. 
Refs.~\cite{Appleby:2017a,Appleby:2020} made an empirical study of the minimal SIC fields for a large number of dimensions where a full set of exact $1$-SICs had been calculated.
They showed (among other things) that for these examples:
\begin{enumerate}
    \item the minimal SIC field in dimension $d$ is the ray class field over $\mbb{Q}\bigl(\sqrt{(d+1)(d-3)}\bigr)$ with modulus $\db$ and ramification at both infinite places;
    \item the \normalizedOverlapsText{} $\normalizedOverlap_{\mbf{p}} = e^{i\theta_{\mbf{p}}}$ are in fact algebraic units. 
\end{enumerate}
The $1$-SICs generating a ray class field have been explicitly related to Stark units in several examples. 
In \cite{Kopp2019}, the normalized overlaps of the four lowest lying prime dimensions congruent to $5 \Mod{6}$ were shown to be Galois conjugates of square roots of Stark units. 
In \cite{Appleby:2022,Bengtsson:2024}, a different construction was used, in which the components of the fiducial vector were directly related to Stark units for dimensions of the form $n^2+3$ which are either prime~\cite{Appleby:2022} or equal to $4$ times a prime~\cite{Bengtsson:2024}, thereby pushing up the highest dimension in which $1$-SICs have been calculated by an order of magnitude. 

The approach taken in this paper generalizes the method used in ~\cite{Kopp2019} to every $r$-SIC in every dimension. 
We hope to examine the connection with the method used in ~\cite{Appleby:2022,Bengtsson:2024} in a future publication.

SICs are constructed in \cite{Kopp2019} by taking Galois conjugates of half-integral powers of Stark units.
This motivates mimicking the expression~\eqref{eq:sicfidexpn}, but with \textit{real} numbers that we hope to relate to Stark units in place of the \normalizedOverlapsText{}. 
Thus we define a \textit{ghost $r$-SIC} in terms of certain \textit{\normalizedGhostOverlapsText{}} which we expect to be algebraic units:
\begin{defn}[Ghost fiducial, \normalizedGhostOverlapsText{}, \ghostOverlapsText{}, twist]\label{dfn:ghostFiducial}
    A \emph{ghost $r$-SIC fiducial}, or \emph{ghost fiducial} for short, is a rank-$r$ projector $\tilde{\Pi}$ given by
    \eag{
        \tilde{\Pi} &= \frac{r}{d} I + \sqrt{\frac{r(d-r)}{d^2(d^2-1)}}\sum_{\mbf{p}\notin d\mbb{Z}^2}\normalizedGhostOverlap_{G\mbf{p}} D_{\mbf{p}}
        \label{eq:ghstSIC}
    }
    where the sum is over any set of coset representatives of $\mbb{Z}^2/d\mbb{Z}^2$ with the representative of $d\mbb{Z}^2$ excluded, where $G$ is a $\GLtwo{\mbb{Z}/\db\mbb{Z}}$ matrix called the \emph{twist}, and where the $\normalizedGhostOverlap_{\mbf{p}}$, called the \emph{\normalizedGhostOverlapsText{}}, are real numbers satisfying
    \eag{
        \label{eq:ghostanalogofhermiticity}
        \normalizedGhostOverlap_{\mbf{p}} \normalizedGhostOverlap_{-\mbf{p}}
        &=
        1
    }
    for all $\mbf{p}\not\equiv \zero\Mod{\db}$, and which are such that $\normalizedGhostOverlap_{\mbf{p}'}=\normalizedGhostOverlap_{\mbf{p}}$ whenever $\mbf{p}'\equiv \mbf{p} \not\equiv \zero \Mod{\db}$ . Paralleling \eqref{eq:normalizedOverlapExpression} we also define the \emph{\ghostOverlapsText{}}
  \eag{
\ghostOverlap_{\mbf{p}} &=
\begin{cases}
r \qquad & \mbf{p}\equiv \zero \Mod{d},
 \\
    \sqrt{\frac{r(d-r)}{d^2-1}} \normalizedGhostOverlap_{\mbf{p}} \qquad & \text{otherwise.}
\end{cases}
}
(So $\Tr\bigl(\tilde{\Pi} D^{\dagger}_{\mbf{p}}\bigr)=\ghostOverlap_{G\mbf{p}}.)$ 
\end{defn}

\begin{rmkb}
A few observations are in order here. Firstly, we only introduce the matrix $G$ at this stage for the sake of consistency with later discussion. When we subsequently give explicit formulae, it will be found that there is a natural way to define the $\normalizedGhostOverlap_{\mbf{p}}$ to which we want to give special prominence (see \eqref{eq:ghostoverlapformula} below). 
Using this natural definition, we will be able to take $G=I$ when $r=1$, but we will want to take $G \neq I$ for $r>1$.

Secondly, in the case of $r$-SICs we start with a family of $d^2$ projectors, then introduce their overlaps, and finally define the corresponding normalized overlaps. 
In our definition of ghost fiducials we reverse that order and start with the normalized ghost overlaps. 
The reason is that the function of the ghost fiducial, at least for present purposes, is to make a bridge between $r$-SICs and the Stark conjectures. 
The normalized overlaps of the $r$-SICs considered in this paper are conjecturally units in an algebraic number field having absolute value $1$ and satisfying 
~\eqref{eq:normalizedOverlapIdentity}. Conjecturally, they are are also Galois conjugates of a set of real units satisfying 
~\eqref{eq:ghostanalogofhermiticity}. It is these numbers, what we call the normalized ghost overlaps, which provide the connection with the Stark conjectures, and which are thus the objects of primary importance for the purposes of this paper. 
One can then use them in 
~\eqref{eq:ghstSIC} to define a corresponding ghost fiducial. 
For present purposes the latter is only of secondary importance. 

Note that, although one is free to define a family of $d^2$ projectors in analogy with 
~\eqref{eq:WHCovariantSICTermsFiducial}, by defining $\tilde{\Pi}\vpu{\dagger}_{\mbf{p}} = D\vpu{\dagger}_{\mbf{p}} \tilde{\Pi}D^{\dagger}_{\mbf{p}}$, the overlaps of this family are typically not real and do not have constant modulus. 
This construction will therefore play no role in this paper.

Conjecturally, there is a Galois automorphism $g$ acting on a suitable number field such that \eqref{eq:sicfidexpn} and \eqref{eq:ghstSIC} are related by $\tilde{\Pi} = g\bigl(\Pi\bigr)$. 
We therefore use a \textit{tilde} to distinguish ghost objects from their ``live'' counterparts, though this notation does not presume any functional relationship between $\Pi$ and $\tilde{\Pi}$. 
In view of \eqref{eq:normalizedOverlapIdentity}, the condition \eqref{eq:ghostanalogofhermiticity} is implied by such a relationship. 
\end{rmkb}
\begin{defn}[Live fiducial]\label{dfn:liveFiducial}
    To contrast them with the ghost fiducials specified by \Cref{dfn:ghostFiducial}, $r$-SIC fiducials as specified by \eqref{eq:sicfidexpn} will sometimes be referred to as \emph{live} fiducials.
\end{defn}

A ghost fiducial is typically not an \hprj{}. 
It is however a \emph{\pprj{}}, short for parity-Hermitian projector, which we now define. 
\begin{defn}[Parity operator, parity-Hermitian, \pprj{}]\label{def:pProjector}
    The \emph{parity operator} is the unitary matrix $U_P$ acting on the standard basis for $\C^d$ as $U_P|j\ra = |-j\ra$, where arithmetic inside the ket is modulo $d$ (see also \Cref{dfn:parityMatrix}). 
    A matrix $M$ is \emph{parity-Hermitian} if it equals its Hermitian conjugate when conjugated by the parity operator:
    \eag{
    M^{\dagger} &= U\vpu{\dagger}_P M U^{\dagger}_P \,.
    }
    A \emph{\pprj{}} is a parity-Hermitian projection operator. 
\end{defn}
\begin{rmkb}
    The subscript $P$ in the notation $U_P$ stands for the $2 \times 2$ negative-identity matrix $P = \smmattwo{-1}{0}{0}{-1} \in \SLtwo{\Z/\db\Z}$. The unitary matrix $U_P \in \U(d)$ comes from a certain function $(A \mapsto U_A) : \SLtwo{\Z/\db\Z} \to \U(d)$ defining a projective representation $\SLtwo{\Z/\db\Z} \to \U(d)/(\C^\times I)$. The representation is described in \Cref{ssc:ecdgp}.
\end{rmkb}
As examples, observe that the displacement operators are parity-Hermitian: 
\eag{
D^{\dagger}_{\mbf{p}}&= U\vpu{\dagger}_P D\vpu{\dagger}_{\mbf{p}} U^{\dagger}_P\,.
}
The fact that the expansion coefficients on the RHS of \eqref{eq:ghstSIC} are all real means that a ghost fiducial $\tilde{\Pi}$ is a \pprj{}.

\subsection{Refining Stark units}\label{ssec:RefiningStarkUnits}

Informally, the Stark conjectures give concrete formulas relating certain analytic functions with associated algebraic data. 
More specifically, they relate the values of the derivatives of certain partial zeta functions at $s=0$ to the logarithms of absolute values of units in an algebraic number field. 

Our goal is to construct $r$-SICs by first constructing the corresponding \normalizedGhostOverlapsText{} using (conjectural) Stark units. 
What we will actually need are not the Stark units themselves, but rather, certain \textit{square roots} of generalized Stark units. 
This presents a difficulty in that the sign of the square root is \textit{a priori} ambiguous. 
To get around this problem, instead of working with zeta functions as is done in 
~\cite{Kopp2019,Appleby:2022,Bengtsson:2024}, we work with a function we call the \textit{\SFKFull{} modular cocycle}, introduced in 
~\cite{Kopp2020d} based on the approach pioneered by Shintani~\cite{Shintani1976,Shintani1977,Shintani1977b,Shintani1980}. 
We will also need to resolve an ambiguity in roots of unity that arises in this process, and for this we also define the \textit{\SFKFull{} phase}. 
We require some notation and several other definitions before we are ready to define these functions. 

Let $\mbb{H}$ be the \textit{upper half plane}
\eag{\label{eq:upperHalfPlane}
\mbb{H} &= \{z\in \mbb{C}\colon \im(z)>0\}.
}
We require two variants of the $q$-Pochhammer symbol, which in its usual variants is denoted $(a;q)_n$ or $(a;q)_\infty$. 
We will find it convenient to write $q=e^{2\pi i \tau}$ and $a=e^{2\pi i z}$ and to treat $\tau$ and $z$ as the fundamental variables.
\begin{defn}[variant $q$-Pochhammer symbols]\label{dfn:variantqPochhammer}
    The \emph{finite variant $q$-Pochhammer symbol} is defined by
    \eag{\label{eq:finiteqpoch}
    \qp_n (z,\tau) &= 
    \begin{cases}
        \prod_{j=0}^{n-1}\left(1-e^{2\pi i(z+j \tau)}\right)
        \qquad & n>0
        \\
        1 \qquad & n =0 
        \\
        \prod_{j=n}^{-1}\left(1-e^{2\pi i(z+j \tau)}\right)^{-1} 
        \qquad & n<0
    \end{cases}
    }
for  $n\in \Z$, $z, \tau \in \mbb{C}$.  
The (infinite) \emph{variant $q$-Pochhammer symbol} is
\eag{
\qp (z,\tau) &= \prod_{j=0}^{\infty}\left(1-e^{2\pi i(z+j \tau)}\right)
\label{eq:qpochinf}
}
for $z\in \mbb{C}$, $\tau\in \mbb{H}$.
\end{defn}

For $\tau\in \mbb{C}$, $M=\smt{\ma & \mb\\ \mc &\md}\in \GLtwo{\mbb{Z}}$, define the \textit{fractional linear transform} $M\cdot\tau$
and denote the denominator $j_M(\tau)$ respectively by
\begin{align}\label{eq:fractionalineartransform}
    M\cdot\tau = \frac{\ma \tau + \mb}{\mc\tau + \md }\,, 
    \qquad \text{and}\qquad 
    j_M(\tau) = \mc \tau + \md \,.
\end{align}
We say $\tau$ is a \textit{fixed point} of $M$ if $M\cdot\tau = \tau$.  

\begin{defn}[The domain $\DD_M$]
\label{def:sl2ldmndf}
For $M=\smt{\ma & \mb\\ \mc &\md}\in \GLtwo{\mbb{Z}}$ define $\DD_M$ to be the set $\mbb{C}\setminus \{\tau\in \mbb{R}\colon \det(M) j_M(\tau) \le 0\}$, illustrated in the complex plane below for the case $\mc>0$ and $\det(M)=1$.

\spc

\begin{center}
\begin{tikzpicture}[scale=.85]
  \draw[very thick,black,zigzag] (-4,0) -- (-1,0);
  \draw[ultra thin] (-1,0) -- (4,0);
  \draw (0,-1.6) -- (0,1.6);
  \draw (-1,0.5) node {$-\frac{\md}{\mc}$};
\end{tikzpicture}
\end{center}
\end{defn}
\begin{rmkb}
    Note that for reasons of technical convenience we define $\DD_M$ for an arbitrary matrix $M\in \GLtwo{\mbb{Z}}$, although in its main application, to the definition of the Shintani--Faddeev modular cocycle (see below), we only need it for matrices $M\in \SLtwo{\mbb{Z}}$.
\end{rmkb}
\begin{defn}[\SFKShort{} Jacobi cocycle]
    \label{df:shinfadjacocycle}
    For $M\in \SLtwo{\mbb{Z}}$ the \emph{\SFKFull{} (\SFKShort{}) Jacobi cocycle}   is a meromorphic function $\sigma_M$ on $\mbb{C}\times \DD_M$
whose restriction to $\mbb{C}\times \mbb{H}$ is given by
    \eag{
   \sfjM{M}{z}{\tau} &= \frac{\qp\!\left(\frac{z}{j_M(\tau)},M\cdot\tau\right)}{\qp(z,\tau)}.
   \label{eq:sfjfrm}
    }
\end{defn}  

\begin{rmkb}
See \cite{Kopp2020d} and \Cref{sec:SFJCocycleAppendix} for the continuation to $\mbb{C}\times \DD_M$.
\end{rmkb} 
It is shown in \Cref{sec:SFJCocycleAppendix} that   $\sigma_M$ satisfies the cocycle condition 
\eag{
\sfjM{MM'}{z}{\tau} &= \sfjM{M}{\frac{z}{j_{M'}(\tau)}}{M'\cdot\tau} \sfjM{M'}{z}{\tau}.
\label{eq:sfjcocyclerelInt}
}
for all values of $z,\tau$ such that both sides of the equation are defined.

Up to a scale factor, $\sfjM{S}{z}{\tau}$ with $S=\smt{0&-1\\1&0}$ is the \textit{double sine function} or \textit{noncompact quantum dilogarithm}. 
The name \SFKFull{} acknowledges Shintani's original introduction of the double sine function in connection with his work on Kronecker-type limit formulas and the Stark conjectures~\cite{Shintani1976,Shintani1977,Shintani1977b,Shintani1980}, and its subsequent rediscovery under the name quantum dilogarithm by Faddeev in connection with his work on discrete Liouville theory~\cite{Faddeev:1994,Faddeev:1995}. 
Subsequently it has also featured in quantum Teichm{\"{u}}ller theory, three-dimensional supersymmetric gauge theory, complex Chern--Simons theory, quantum group theory,  and quantum knot theory (see ~\cite{Woronowicz:2000,Faddeev:2001,Dimofte:2010,Dimofte:2015,Closset:2019,Garoufalidis:2021} and references cited therein). 

For $\mbf{p}=\smt{p_1 \\ p_2}$ and $\mbf{q}=\smt{q_1 \\ q_2} \in \mbb{C}^2$, define the nondegenerate \textit{symplectic form}
\eag{\label{eq:SymplecticForm}
    \symp{\mbf{p}}{\mbf{q}} &= -\det{\!\smmattwo{p_1}{q_1}{p_2}{q_2}} = p_2 q_1-p_1 q_2\,.
}
When the second argument is a complex number $\tau \in \mbb{C}$, we also use the notation
\eag{\label{eq:FractionalSymplecticForm}
    \pt{\mbf{p}}{\tau} &= \symp{\p}{\smcoltwo{\tau}{1}} = p_2\tau-p_1\,.
}
It is easily verified that
\eag{
\la M\mbf{p},M\mbf{q}\ra &= (\det M) \la \mbf{p},\mbf{q}\ra, & 
\pt{M\mbf{p}}{M\cdot\tau} &= \frac{(\det M) \pt{\mbf{p}}{\tau}}{j_{\!M}(\tau)}
\label{eq:LActOnSymplecticInnerProduct}
}
for all $L\in \GLtwo{\mbb{Z}}$, $\mbf{p},\mbf{q}\in \mbb{Z}^2$, and $\tau \in \mbb{C}$.

For $d\in \mbb{N}$ define $\Gamma(d)$ to be the \textit{principal congruence subgroup of level $d$} consisting of matrices $A \in \SLtwo{\mbb{Z}}$ such that $A\equiv I \Mod{d}$. 
We will also need a particular family of non-principal congruence subgroups, which we define now.

\begin{defn}[$\Gamma_\r$]\label{dfn:gammarDef}
    For $\r \in \Q^2$, let $\Gamma_\r$ be the subgroup of $\SLtwo{\mbb{Z}}$ consisting of matrices $A$ such that $(A-I)\r \in \mbb{Z}^2$. 
\end{defn}
\begin{rmkb}
    Note that if  $\r \in \frac{1}{d}\Z^2$, then $\Gamma(d)$ is a subgroup of $\Gamma_\r$.
\end{rmkb}

\begin{defn}[\SFKShort{} modular cocycle]
    \label{def:shin}
    If $\r\in \Q^2$ and $M \in \Gamma_\r$, then the \emph{\SFKFull{} (\SFKShort{}) modular cocycle} is defined by
    \eag{
    \sfc{\r}{M}{\tau} &=  
    \frac{\sfjM{M}{\pt{\r}{\tau}}{\tau}}{\qp_{((I-M)\mbf{r})_2}\!\left(\frac{\pt{\r}{\tau}}{j_M(\tau)},M\cdot\tau\right)}
    \label{eq:shindf}
    }
    for all $\tau\in \DD_M$. (This is equivalent to the definition given in \cite[Defn.~4.18]{Kopp2020d} by \cite[Prop.~4.19]{Kopp2020d}.)
\end{defn}

\begin{rmkb}
    The symbol $\shin$ is the Hebrew letter ``shin.'' 
    For instructions on how to typeset this character in \LaTeX, see \cite[Sec.\ 9]{Kopp2020d}.

\end{rmkb}

For $\tau \in \HH$ and $\r \nin \Z^2$, a straightforward calculation shows that
\begin{equation}\label{eq:coboundary}
    \sfc{\r}{M}{\tau} = \frac{\varpi(\pt{\r}{M\cdot\tau},M\cdot\tau)}{\varpi(\pt{\r}{\tau},\tau)},
\end{equation}
and thus that the multiplicative group-cohomological cocycle condition 
\begin{equation}
    \sfc{\r}{MN}{\tau} = \sfc{\r}{M}{N\cdot\tau}\sfc{\r}{N}{\tau}
\end{equation}
holds for $M, N\in\Gamma_\r$. By meromorphic continuation (see also~\Cref{sec:SFJCocycleAppendix}), the cocycle condition holds for $\tau \in \mathcal{D}_M$, while the ``coboundary'' 
expression \eqref{eq:coboundary} does not make sense outside the upper half plane (although a similar expression may be given on the lower half plane, but not on the real line). 
The map $M \mapsto \shin^\r_{\!M}$ from $\Gamma_\r$ to the multiplicative group of meromorphic functions defines a cohomologically nontrivial class in a certain cohomology group, which is somewhat tricky to define correctly and is discussed in more detail in \cite[Sec.~4.1 and Sec.~5]{Kopp2020d}. 
In this paper, we will primarily be concerned with $M \in \Gamma(d)$ rather than in the larger group $\Gamma_\r$, because we will be fixing $d$ and varying $\r$. 
We abuse terminology slightly by referring to the meromorphic function $\sfc{\r}{M}{\tau}$ itself (rather than the map from $\Gamma_\r$ or $\Gamma(d)$) as a ``cocycle.''

The importance of the function $\sfc{\r}{M}{\tau}$ for us is that, as we will see, it provides a bridge between the geometric construct of an $r$-SIC with algebraic number theory. 
On the one hand we use special values of the function to construct \ghostOverlapsText{}, while on the other hand these same special values are related to Stark units. In particular, under the assumption of 
the Stark Conjecture, %
(see \Cref{ssec:TheConjectures}), they are algebraic integers, and indeed units.
For abelian extensions of a large set of real quadratic fields, they play an analogous role to the one that roots of unity play in connection with abelian extensions of $\mbb{Q}$, and that elliptic functions and modular forms (more precisely, Siegel units \cite{Kubert:1981}) play in connection with abelian extensions of imaginary quadratic fields. 
These algebraic properties are essential if we wish to take  Galois conjugates of our constructed \ghostOverlapsText{}, as we do to construct live $r$-SIC fiducials.

\subsection{Quadratic fields and quadratic forms}
\label{ssec:quadraticfieldsforms}

We next briefly review some definitions associated to the theory of quadratic forms and quadratic fields, mainly to fix notation. 

Let $D$ be a square-free positive integer, and let $K=\mbb{Q}(\sqrt{D})$ be the corresponding real quadratic field. 
Then the \textit{discriminant} of $K$ is 
\eag{\label{eq:discriminant}
    \Delta_0
    &=\begin{cases}
        D \qquad & \text{if $D\equiv 1 \Mod{4}$}, \\
        4 D \qquad & \text{otherwise}.
    \end{cases}
}
The \textit{ring of integers} in $K$ is denoted by $\mcl{O}_K$ and the \textit{unit group} by $\mcl{O}^{\times}_K$. 

A \textit{binary quadratic form} is a bivariate polynomial $Q(x,y) = a x^2+bxy+cy^2$. 
Unless stated explicitly to the contrary, we will simply say \textit{form} to mean an integral, primitive, irreducible, and indefinite binary quadratic form. 
That is, $a,b,c$ are coprime integers, the roots of $Q(x,1)$ are in $\mbb{R}$ but not in $\mbb{Q}$, and $Q$ takes both positive and negative values. 
We will employ the shorthand $Q=\la a, b , c\ra$, and where there is no risk of confusion we will use the same symbol $Q$ to denote the Hessian matrix of $Q$ scaled by a factor of $\foh$:
\eag{
    Q&= \bmt a & \frac{b}{2} \\ \frac{b}{2} & c\emt.
}
If $\mbf{p}$ is the vector $\smt{p_1\\p_2}$, we will also write $Q(\mbf{p})  = Q(p_1,p_2)$.

Let $Q$ be a form and let $M\in \GLtwo{\mbb{Z}}$. 
Then we denote by $\lOnQ{M}{Q}$ the form 
\eag{
    \lOnQ{M}{Q} &= \Det(M) M^{\rm{T}}Q M.
    \label{eq:afrmtrans}
}
We say two forms $Q, Q'$ are \textit{equivalent} and write $Q\sim Q'$ if
\eag{
    Q'&= \lOnQ{M}{Q}
}
for some $M\in\GLtwo{\mbb{Z}}$.

Let $Q=\la a, b , c\ra$ be a form; then $\Delta = -4\det(Q) = b^2-4ac$ is its \textit{discriminant}. 
Let $D$ be the square-free part of $\Delta$. 
Then the \textit{fundamental discriminant} of $Q$ is
\eag{
    \Delta_0 &=\begin{cases} 
            D \quad & \text{if $D\equiv1 \Mod{4}$}\\
            4D \quad &\text{otherwise}\,,
        \end{cases}
}
and the \textit{conductor} of $Q$ is the integer
\eag{\label{eq:ConductorOfQ}
    f&= \sqrt{\frac{\Delta}{\Delta_0}}\,.
}
We say that the form $Q$ is \emph{associated} to the real quadratic field $K$ if the discriminant of $K$ is the fundamental discriminant of $Q$.
We define the \textit{roots} of $Q$ to be 
\eag{
    \label{dfn:qrtdf}
    \qrt_{Q,\pm}&=\frac{-b \pm \sqrt{\Delta}}{2a}\,.
}

We will find it convenient to introduce a notion of sign to both forms and elements of $\GLtwo{\mbb{Z}}$. 
\begin{defn}[Sign]\label{dfn:sign}
    Let $Q=\la a, b, c\ra$ be a form, and let $M=\smt{\ma & \mb \\ \mc &\md}$ be an element of $\GLtwo{\mbb{Z}}$. 
    We define: 
    \begin{enumerate}
        \item The \textit{sign} of $Q$, denoted $\sgn(Q)$, to be the sign of $a$, $\sgn(a)$;
        \item The \textit{sign} of $M$, denoted $\sgn(M)$, to be the sign of $\mc$, with the convention that if $\mc=0$, then $\sgn(M)=\sgn(\md)$.
    \end{enumerate}
    (Note that $\sgn(Q)$ is not the same as the sign of the scaled Hessian matrix of $Q$, which we also denote by $Q$. 
    However, it will always be clear from context whether $\sgn$ denotes the sign of a form or a matrix.)
\end{defn}

We will also need the usual definition of the stability group of a form as well as one variant. 
\begin{defn}[Stability group of a quadratic form]\label{dfn:stbgpqdf}
    Let $Q$ be a form and $d$ a positive integer. We define:
    \begin{enumerate}
        \item $\mcl{S}(Q)$ to be the set of all $M\in \GLtwo{\mbb{Z}}$ such that $Q_M = Q$;
        \item $\mcl{S}_d(Q) = \mcl{S}(Q)\cap \Gamma(d)$.
    \end{enumerate}
    We refer to $\mcl{S}(Q)$ as the \textit{stability group} of $Q$.
\end{defn}

\subsection{Admissible tuples, the Shintani--Faddeev phase, and normalized ghost overlaps}\label{ssec:admissibletuples}

We now introduce the data necessary to define the set of \normalizedGhostOverlapsText{} corresponding to a ghost fiducial. 
This is encapsulated in the notion of an \textit{admissible tuple}.

It will turn out to be convenient to have two equivalent notions of admissible tuples: the first notion starts with the dimension and the rank and has a more geometric flavor; the second starts with the real quadratic field and has a more number-theoretic flavor. 
We will describe the geometric definition first and then describe the equivalence with the number-theoretic definition. 
The equivalence of these data is stated below in \Cref{thm:bijectionofadmissibletuples}. 

The dimension $d$ and rank $r$ of the $r$-SICs we consider in this paper must satisfy the following conditions. 
\begin{defn}[Admissible pair, associated field]\label{dfn:admissiblePair}
    A pair of integers $(d,r)$ is called \emph{admissible} if there exists an integer $n > 4$ such that
    \eag{
        \label{eq:drnconditions}
        0 &< r < \frac{d-1}{2}\,, & nr(d-r) &=d^2-1\,.
    }
    For each such admissible pair, we define a real quadratic \emph{associated field} $K = \mbb{Q}\!\left(\sqrt{n(n-4)}\right)$. 
\end{defn}

The condition on $r$ immediately implies $d>3$. 
The reason for the requirement $n>4$ is to ensure that $K$ is real quadratic. 
The reasons for the restriction $r < (d-1)/2$ are: Firstly, the transformation $r \to d-r$ can be used to swap between a projector $\Pi$ and its complement $I-\Pi$, and these give essentially equivalent objects; secondly, the cases\footnote{For the Diophantine equation $nr(d-r) = d^2-1$, the case $r = d/2$ reduces to the single solution $(d,r,n)=(2,1,3)$ in positive integers, whereas the case $r=(d-1)/2$ produces the family of solutions $(d,r,n) = (2k+1,k,4)$. The former case leads to the $1$-SIC in dimension $d=2$, whereas the latter case leads to continuous families of $r$-SICs of a very different nature to those described herein, including examples with elementary descriptions.} $r = d/2$ and $r=(d-1)/2$ are inconsistent with the requirement $n>4$, and so with the requirement that $K$ be real quadratic. 

There are infinitely many admissible pairs $(d,r)$. 
For example, for a given dimension $d$ there is always the $r=1$ solution $(d,1)$ leading to a $1$-SIC with associated field $\mbb{Q}\bigl(\sqrt{(d+1)(d-3)}\bigr)$. 
There is also a solution for arbitrary $r>2$ given by $(r^2+r-1,r)$ corresponding to an $r$-SIC with associated field $\mbb{Q}\bigl(\sqrt{r^2-4}\bigr)$. 
However it is easily seen that there is no solution for $r=2$. 
These are only a few of the possible solutions, as we discuss below. 

Rather than starting with an admissible pair $(d,r)$, the number-theoretic approach starts with a given field $K$ and then characterizes the admissible pairs $(d,r)$ associated to that field. 
One finds that they form two-index grids $d_{j,m}$, $r_{j,m}$ called the \textit{dimension grid} and the \textit{rank grid} respectively, defined below in \Cref{dfn:fjrjmdjm}. 
The dimension and rank grids are defined in terms of powers of a unit $\vn$ defined as follows: 
\begin{defn}[Fundamental totally positive unit $\vn$]
\label{dfn:fundamentalTotallyPositiveUnit}
    For $K$ a real quadratic field, define the \textit{fundamental totally positive unit} $\vn$ to be the smallest positive-norm unit greater than 1.
\end{defn}
\begin{rmkb}
    To avoid cluttering the notation, we do not indicate the field $K$ explicitly. 
    The field will always be clear from context.
\end{rmkb}
\begin{rmkb}
As we will see, the properties of $\vn$ are intimately related to the Zauner symmetry of an $r$-SIC.
\end{rmkb}

Before defining the dimension grid, we define the \textit{sequence of conductors}, which is independently important, as it plays a central role in the classification of $r$-SICs. 
\begin{defn}[Sequence of conductors $f_j$]
\label{dfn:sequenceofconductors}
    Let $K$ be a real quadratic field, let $\Delta_0$ be its discriminant, and let $\vn$ be the unit from \Cref{dfn:fundamentalTotallyPositiveUnit}. 
    For positive integers $j$, define 
    \eag{
        f_j&= \frac{\vn^j-\vn^{-j}}{\sqrt{\Delta_0}}
    }
    to be the \emph{sequence of conductors}. 
\end{defn}
\begin{rmkb}
    Note that the $f_j$ are positive integers.
\end{rmkb}
\begin{rmkb}
    As with $\vn$, we do not indicate the field $K$ explicitly. 
    This will always be clear from context.
\end{rmkb}
We are now ready to define the dimension and rank grids.
\begin{defn}[Dimension grid, rank grid, dimension tower, root dimension, admissible triple]
\label{dfn:fjrjmdjm}
    Let $K$ be a real quadratic field and $f_j$ its sequence of conductors. 
    For all positive integers $j, m$, define 
    \eag{
        r_{j,m} & = \frac{f_{jm}}{f_j}\,, &
        d_{j,m} &= r_{j,m+1}+r_{j,m}\,, &
        d_j &= d_{j,1}.
    }
    We define 
    \begin{enumerate}
        \item the two-index sequence $d_{j,m}$ as the \emph{dimension grid} associated to $K$ and $r_{j,m}$ as the \emph{rank grid},
        \item the sequence $d_j$ as the \emph{dimension tower} associated to $K$, and
        \item the integer $d_1$ as the \emph{root dimension} of $K$.
    \end{enumerate}
    For all real quadratic $K$ and positive integers $j,m$ we say $(K,j,m)$ is an \emph{admissible triple}. 
\end{defn}
\begin{rmkb}
    As with $\vn$ and $f_j$, we do not indicate the field $K$ explicitly in the definitions of $r_{j,m}$, $d_{j,m}$, $d_j$. 
    This will always be clear from context.
\end{rmkb}
The next theorem establishes a bijection between the set of admissible pairs and the set of admissible triples.

\begin{thm}
    \label{thm:bijectionofadmissibletuples}
    For each admissible triple $(K,j,m)$, the pair $(d_{j,m},r_{j,m})$ is admissible. 
    Conversely, for each admissible pair, $(d,r)$ there is a unique admissible triple $(K,j,m)$ such that $(d_{j,m},r_{j,m})=(d,r)$.
\end{thm}
\begin{proof}
    This is an immediate consequence of \Cref{thm:nrddjmrjm}. 
\end{proof}

\begin{defn}[Admissible tuple equivalence $\sim$]
\label{def:tupleequiv}
    We write $(d,r)\sim(K,j,m)$ if $(d,r)$ is associated to $(K,j,m)$ under the bijection just described---that is, if $d=d_{j,m}$ and $r=r_{j,m}$. 
\end{defn}

The dimension grid $d_{j,m}$ consists of the dimensions in which, conditional on our conjectures, there exist $r$-SICs generating abelian extensions of the number field $K$, and the corresponding ranks of the $r$-SICs are $r_{j,m}$. 
When $m=1$, the rank $r_{j,1}$ is always $1$, and the sequence of dimensions $d_{j,1}$ is distinguished as the dimensions where there occur $1$-SICs. 
From a physics and geometric point of view the $1$-SICs have a special significance, which motivates picking out the dimensions in which they occur by defining $d_j = d_{j,1}$. 

\begin{example}
    Consider $K=\mbb{Q}(\sqrt{5})$. 
    The fundamental unit greater than $1$ is $\tfrac{1}{2}\bigl(1+\sqrt{5}\bigr)$, but this has norm $-1$. 
    The unit $\vn$ is therefore the square of this unit, $\vn = \tfrac{1}{2}\bigl(3+\sqrt{5}\bigr)$. 
    The sequence of conductors $f_j$ begins $1, 3, 8, 21, 55, 144, 377, 987,\ldots$ and is given by the $(2n)^{\text{th}}$ Fibonacci number. 
    Consequently the dimension grid and rank grid for this case are 
    \begin{align}
        \label{eq:exampledimgrid}
        d_{j,m} = \begin{array}{lllll}
            \vdots & \vdots & \vdots & \vdots & \\
            48 & 2\,255 & 105\,937 & 4\,976\,784 & \dots \\
            19 & 341 & 6\,119 & 109\,801 & \dots \\
            8 & 55 & 377 & 2\,584 & \dots \\
            4 & 11 & 29 & 76 & \dots
        \end{array}\,,
        \quad
        r_{j,m} = \begin{array}{lllll}
            \vdots & \vdots & \vdots & \vdots & \\
            1 & 47 & 2\,208 & 103\,729 & \dots \\
            1 & 18 & 323 & 5\,796 & \dots \\
            1 & 7 & 48 & 329 & \dots \\
            1 & 3 & 8 & 21 & \dots
        \end{array}
    \end{align}
    In these grids the lower-left entries are $d_{1,1}$ and $r_{1,1}$ respectively; $j$ increases from bottom to top and $m$ increases from left to right. 
    The left-hand column of the first grid gives $d_j$, the sequence of dimensions in which one finds $1$-SICs associated to the field $\mbb{Q}(\sqrt{5})$.
\end{example}

We now extend the notion of admissible to include a form as part of an admissible tuple. 
\begin{defn}[Admissible tuple with form]
\label{def:admissibleform}
    Let $Q$ be a form, and let $(d,r) \sim (K,j,m)$ be admissible tuples. 
    We say that  $(d,r,Q) \sim (K,j,m,Q)$ are \textit{admissible} if the fundamental discriminant of $Q$ is the discriminant of $K$ and the conductor of $Q$ is a divisor of $f_j$. 
\end{defn}

\begin{example}
    Consider again $K=\mbb{Q}\bigl(\sqrt{5}\bigr)$. 
    The form $Q = \la 1,-3,1\ra$ has fundamental discriminant $5$ and conductor $1$, so $(K,j,m,Q)$ is always admissible for positive integers $j,m$. 
    When $j=1,m=1$, we see that $d_{1,1} = 4$ and $r_{1,1} = 1$, so $(K,1,1,Q) \sim (4,1,Q)$ and $(4,1,Q)$ is also admissible. 
    
    The form $Q' = \la 5, -20, 4\ra$ also has fundamental discriminant $5$, but it has conductor $8$. 
    From the sequence of conductors $f_j = 1, 3, 8, 21, 55, 144,\ldots$ we see that $(K,j,m,Q')$ is only admissible for $j=3,6,\ldots$ and in fact when $3\div j$. 
    A corresponding admissible tuple for $j=3, m=1$ is given by $(K,3,1,Q') \sim (19,1,Q')$. 
\end{example}

Admissible tuples contain all of the data necessary to define a corresponding set of \normalizedGhostOverlapsText{}. 
However, before giving the explicit formula, it is convenient to introduce a few more definitions. 
Recall from \Cref{dfn:stbgpqdf} that $\mcl{S}(Q)$ is the stability group of $Q$ and $\mcl{S}_d(Q)$ is the intersection of $\mcl{S}(Q)$ with $\Gamma(d)$. 
The following definition relies on \Cref{tm:symgp}, which proves basic properties about the structure of the groups $\mcl{S}(Q)$ and $\mcl{S}_d(Q)$. 
The group $\mcl{S}_d(Q)$ is infinite cyclic, whereas $\mcl{S}(Q)$ has nontrivial $2$-torsion with $\mcl{S}(Q)/\{\pm I\}$ being infinite cyclic.

\begin{defn}[Associated stabilizers, $\stabQGen{t}, A_t$, $\posGen{t}, \zaunerGen{t}$]\label{dfn:AssociatedStabilizers}
    Let $t=(d,r,Q) \sim (K,j,m,Q)$ be an admissible tuple, and let $f$ be the conductor of $Q$. 
    Define the \emph{associated stabilizer} for $\mcl{S}(Q)$, denoted $\stabQGen{t}$, to be the positive-trace element of $\mcl{S}(Q)$ such that $\sgn(L_t) = \sgn(Q)$ and $\stabQGen{t}$ generates $\mcl{S}(Q)/\{\pm I\}$,
    and define the \emph{associated stabilizer} for $\mcl{S}_d(Q)$, denoted
    $A_t$, to be the generator of $\mcl{S}_d(Q)$ 
    such that $\sgn(A_t) = \sgn(Q)$.
    
    Also define
    \eag{
    \posGen{t}&= 
    \begin{cases}
        \stabQGen{t} \qquad & \Det(\stabQGen{t}) = +1,
        \\
        \stabQGen{t}^2 \qquad & \Det(\stabQGen{t}) = -1,
    \end{cases}
    \\
    \zaunerGen{t}&= \frac{d_j-1}{2}I + \frac{f_j}{f}SQ
    }
    where $S=\smt{0 & -1 \\ 1 & 0}$.
\end{defn}
\begin{rmkb}
    We require that $\stabQGen{t}$ is positive trace, and that $\stabQGen{t}$, $A_t$ have the same sign as $Q$ mainly for the sake of definiteness. 
    Note, however, that other choices, though possible, might complicate the statements of some of our results.

    The significance of the matrix $\posGen{t}$ is that $\mcl{S}(Q)\cap \SLtwo{\mbb{Z}}$ is generated by $\posGen{t}$ and $-I$ (see \Cref{tm:symgp}). 
    The significance of the matrix $\zaunerGen{t}$ is that it is the unique element of $\mcl{S}(Q)$ such that $\zaunerGen{t}^{2m+1} = A_t$ (see \Cref{tm:symgp}). 
    In the case of a $1$-SIC, it is related to the canonical order 3 symmetry which is a prominent feature of SIC phenomenology (see \Cref{thm:FzAndFaSymmetries}).
\end{rmkb}
\begin{example}
    Consider again $K=\mbb{Q}\bigl(\sqrt{5}\bigr)$. 
    We have seen that the tuple $t=(d,r,Q)=(4,1,\la 1, -3, 1\ra)$ is admissible. 
    It is easily checked that the matrices $\pm\smt{3&-1\\1&0}$ and $\pm\smt{0&1\\-1&3}$ are in $\mcl{S}(Q)$. 
    In fact they are the generators, and accordingly we choose $\stabQGen{t} = +\smt{3&-1\\1&0}$ as the associated stabilizer since this has the same sign as $Q$ and positive trace. One then finds $\posGen{t} = \zaunerGen{t} = \stabQGen{t}$.
    The matrix $A_t = \stabQGen{t}^3$ is easily seen to be an element of $\mcl{S}_d(Q)$ and is in fact a generator with the same sign as $Q$. 

    We have seen that the tuple $t=(19,1,\la 5,-20,4\ra)$ is another admissible tuple corresponding to the same field.
    The associated stabilizers are $\stabQGen{t} = \posGen{t}=\zaunerGen{t} = \smt{19&-4\\5&-1}$ and $A_t = \stabQGen{t}^3$. 
\end{example}

There are two more functions that we need to define. 
While the definitions are rather technical, the role these functions play is easy to motivate. 

In general, the \SFKShort{} modular cocycle is complex, but the \normalizedGhostOverlapsText{} $\normalizedGhostOverlap_{\mbf{p}}$ are by definition real. 
The \textit{\SFKFull{} (\SFKShort{}) phase}, defined below, is a complex unit which multiplies the \SFKShort{} modular cocycle so that the product is always a real number. 
This requirement alone could of course be achieved by simply multiplying by the complex unit having the conjugate argument, but for the result to be a \ghostOverlapSingularText{} requires certain additional structure in the pattern of signs as $\mbf{p}$ varies. 
The \SFKShort{} phase achieves the desired sign structure. 
Moreover, it is simply a root of unity with a quadratic dependence on $\mbf{p}$ in the exponent. 

It is worth noting that the approach here, using the \SFKShort{} modular cocycle, has an important advantage over the $L$-function approach in ~\cite{Kopp2019,Appleby:2022,Bengtsson:2024} in that it provides a simple way to resolve the sign ambiguity in the definition of the \normalizedGhostOverlapsText{}. 

The definition of the \SFKShort{} phase depends on a certain integral class function on $\SLtwo{\mbb{Z}}$ known as the \textit{Rademacher class invariant}~\cite{Rademacher:1972}. 

\begin{defn}[Rademacher class invariant]
\label{df:meyinv} 
    For all $M \in \smt{\ma & \mb \\ \mc & \md} \in \SLtwo{\mbb{Z}}$
    the \emph{Rademacher class invariant} is given by
    \eag{
    \rade(M) &= 
      \begin{cases}
        \tfrac{\Tr(M)}{\mc} -3\sgn\bigl(\mc \Tr(M)\bigr)-12 \sgn(\mc)\mathfrak{s}(\ma,\mc) \qquad & \mc\neq 0\,,
        \\
        \frac{\mb}{\md} \qquad & \mc = 0\,,
      \end{cases}
    \label{eq:phiMdf}
    }
    where $\mathfrak{s}(a,b)$ with $a,b\in\Z$ and $b\not=0$ is the Dedekind sum
    \eag{
        \mathfrak{s}(a,b) &= \sum_{n=1}^{|b|-1} \left(\!\left(\frac{n}{b}\right)\!\right) \left(\!\left(\frac{n a}{b}\right)\!\right)\,, 
        \quad & \text{with\ \ } 
        (\!(x)\!) &=\begin{cases} 0 \qquad & x\in \mbb{Z}
        \\
        x-\lfloor x \rfloor -\frac{1}{2} \qquad & x\notin \mbb{Z}
        \end{cases}\,,
        \label{eq:dedekindsum}
    }
    and where we adopt the convention $\sgn(0)=0$.
\end{defn}

\begin{defn}[\SFKShort{} phase]\label{dfn:SFKPhase}
    Let $t=(d,r,Q) \sim (K,j,m,Q)$ be an admissible tuple, and let $\mbf{p} \in \mbb{Z}^2$. 
    The \emph{\SFKShort{} phase}, denoted
    $\SFPhase{t}{\mbf{p}}$,
    is 
    \eag{
        \SFPhase{t}{\mbf{p}}= (-1)^{s_d(\mbf{p})} e^{-\tfrac{\pi i}{12} \rade(A_t)} \rtu_d^{- \tfrac{f_{jm}}{f} Q(\mbf{p})}  \,,
    }
    where $s_d(\mbf{p}) =d+ (1+d)(1+p_1)(1+p_2)$ and $f$ is the conductor of $Q$.
\end{defn}
\begin{rmkb}
    In the even dimensional case other choices for the sign $(-1)^{s_d(\mbf{p})}$ are possible. 
    However, it is shown in \Cref{ap:alternativeFiducialData} that they do not lead to new equivalence classes of $r$-SICs.
\end{rmkb}
\begin{thm}\label{thm:ghostWellDefinedCondition}
    Let $t=(d,r,Q)$ be an admissible tuple. 
    Then
    \eag{
     \qrt_{Q,\pm} \in \DD_{A\vpu{-1}_t}\cap\DD_{A^{-1}_t}.
    }
\end{thm}
\begin{proof}
    See \cpageref{ssc:ghostWellDefinedProof}
    in \Cref{ssc:ghostWellDefinedProof}.
\end{proof}
We now have all of the ingredients to state our formula for the \normalizedGhostOverlapsText{}. 
\begin{defn}[\candidateNormalizedGhostOverlapsText{} corresponding to an admissible tuple]\label{dfn:GhostOverlaps}
 Let $t=(d,r,Q)\sim(K,j,m,Q)$ be an admissible tuple, and let $\mbf{p} \in \mbb{Z}^2$. 
    The corresponding \textit{\candidateNormalizedGhostOverlapsText{}} are defined by
 \eag{
    \label{eq:ghostoverlapformula}
    \normalizedGhostOverlapC{t}{\mbf{p}} = 
    \SFPhase{t}{\mbf{p}}
    \,\sfc{d^{-1} \mbf{p}}{A_t}{\qrt_{t}}.
}
for all $\mbf{p}$, 
where the $\SFPhase{t}{\mbf{p}}$ are defined by \Cref{dfn:SFKPhase} and 
$
\qrt_{t} = \qrt_{Q,+}.
$
The corresponding \textit{\candidateGhostOverlapsText{}} are defined by
\eag{\label{eq:CandidateGhostOverlapDefinition}
\ghostOverlapC{t}{\mbf{p}}
&= \begin{cases}
    r \qquad & \mbf{p}\equiv \zero\Mod{d},
    \\
    \sqrt{\frac{r(d-r)}{d^2-1}} \, \normalizedGhostOverlapC{t}{\mbf{p}}\qquad & \text{otherwise.}
\end{cases}
}
\end{defn}
\begin{rmkb}
Note that this definition relies on \Cref{thm:ghostWellDefinedCondition}, since otherwise the RHS would not be well-defined. 
Note also that instead of defining $\qrt_t = \qrt_{Q,+}$, we could equally well define $\qrt_t =\qrt_{Q,-}$. 
Specifically, it is shown in \Cref{ap:alternativeFiducialData} that \ghostOverlapsText{} calculated using $\qrt_{Q,-}$ with the form $Q$ coincide with those calculated using $\qrt_{Q',-}$ with a different form $Q'$.
Finally, note that in \Cref{cor:AlternativeDefinitionOfCandidateGhostOverlaps} we derive a simpler expression for the scaling factor $\sqrt{\frac{r(d-r)}{d^2-1}}$ in \eqref{eq:CandidateGhostOverlapDefinition}.
\end{rmkb}

\subsection{The main conjectures}\label{ssec:TheConjectures}

Our goal is to show that the \candidateNormalizedGhostOverlapsText{} defined by \eqref{eq:ghostoverlapformula}, when inserted on the right hand side of \eqref{eq:ghstSIC} give, with a suitable choice of twist $G$, ghost fiducials from which live fiducials can then be constructed by applying a suitable Galois conjugation. 
The validity of the construction depends on the Twisted Convolution Conjecture and Stark Conjecture.

To motivate the Twisted Convolution Conjecture, observe that if the \candidateNormalizedGhostOverlapsText{} $\normalizedGhostOverlapC{t}{\mbf{p}}$ specified by \eqref{eq:ghostoverlapformula} are to give rise to a ghost fiducial when substituted into \eqref{eq:ghstSIC}, then we must have
\begin{enumerate}
    \item The numbers $\normalizedGhostOverlapC{t}{\mbf{p}}$ are real for all $\mbf{p}\not\equiv \zero \Mod{d}$,
    \item 
    $\normalizedGhostOverlapC{t}{\mbf{p}} \normalizedGhostOverlapC{t}{-\mbf{p}} = 1$ for all $\mbf{p}\not\equiv \zero \Mod{d}$),
    \item $\tilde{\Pi}^2 = \tilde{\Pi}$.
\end{enumerate}
The first two conditions are proved in \Cref{thm:nupnumpeq1}. 
However, we have so far been unable to prove the last condition, which must therefore be posited as an additional conjecture. 
Before stating it we need some definitions.
\begin{defn}\label{dfn:modulardeltafunction}
    Let $\mbf{p},\mbf{q}\in \Q^2$ and let $n$ be a positive integer. 
    Then we define 
    \eag{
    \delta^{(n)}_{\mbf{p},\mbf{q}} &= 
    \begin{cases}
        1 \qquad & \mbf{p} - \mbf{q} \in n\Z^2,
        \\
        0 \qquad & \text{otherwise}.
    \end{cases}
    }
\end{defn}
\begin{defn}[Shift]\label{dfn:shift}
Let $t=(d,r,Q)\sim(K,j,m,Q)$ be an admissible tuple. 
We say that  $\shift\in \mbb{Z}/d\mbb{Z}$ is a \emph{shift} for $t$ if
\begin{enumerate}
    \item $2\shift+d_j-1$ is coprime to $d$,
    \item $\shift$ satisfies 
        \eag{
        \sum_{\mbf{q} \in \mcl{I}_{\mbf{p}}} \rtus_d^{ r\la \mbf{p},(\shift I + \zaunerGen{t})\mbf{q}\ra} \sfc{d^{-1}\q}{A\vpu{-1}_t}{\qrt_t }\sfc{d^{-1}(\mbf{q}-\mbf{p})}{A^{-1}_t}{\qrt_t} &= d^2\delta^{(d)}_{\mbf{p},\zero}
        \label{eq:tcc}
    }
 for all $\mbf{p}\in \mbb{Z}^2$, where the index set $\mcl{I}_{\mbf{p}}$ is any complete set of coset representatives for $\mbb{Z}^2/d\mbb{Z}^2$ containing $\zero$ and $\mbf{p}$.
 The set of all shifts for $t$ is denoted $\mcl{Z}_t$.
\end{enumerate} 
\end{defn}
We are now ready to state our additional conjecture:
\begin{conj}[Twisted Convolution Conjecture]\label{cnj:tci}
For every admissible tuple the set of shifts $\mcl{Z}_t$ includes the values $\lambda = 0, 1$.  Moreover, if $t=(d,r,Q)$ and $t'=(d,r,Q')$ are admissible tuples such that $Q$ and $Q'$ have the same discriminant, then $\mcl{Z}_t = \mcl{Z}_{t'}$.
\end{conj} 
As a matter of empirical observation it appears that  $0$, $1$ are the only shifts when $r=1$, but that when $r>1$ there are others.

The set of shifts for a given tuple $t= (d,r,Q)\sim(K,j,m,Q)$ also determines the possible choices of the twist  in \Cref{dfn:ghostFiducial}. 
Specifically, it can be seen from the proof of \Cref{thm:ghstExist} that if the Twisted Convolution Conjecture is valid, then a matrix $G$ is a possible choice of twist if and only if
\eag{
    \Det(G) r(2\shift+d_j-1+d) &\equiv 1 \Mod{\db}
}
for some $\shift \in \mcl{Z}_t$.

The Twisted Convolution Conjecture guarantees the existence of ghost fiducials in every dimension. 
To get from there to the existence of live fiducials in every dimension we need a guarantee that
\begin{enumerate}
    \item the matrix entries of the ghost fiducial are algebraic numbers,
    \item there exists a Galois automorphism with the properties needed to convert the ghost fiducial into a live fiducial.
\end{enumerate}

These guarantees are provided directly by a conjectures about special values of the Shintani--Faddeev modular cocycle (called \textit{real multiplication (RM) values} in \cite{Kopp2020d}). We first state a ``minimalist''
such conjecture, which will be sufficient (together with the Twisted Convolution Conjecture) to prove SIC existence.

\begin{conj}[Minimalist\footnote{Strictly, we could prove SIC existence from an even more ``minimalist'' conjecture by only assuming the existence of one Galois automorphism satisfying property (2).} Real Multiplication Values Conjecture]\label{conj:mrmvc}
    Let $\qrt \in \R$ such that $a\qrt^2 + b\qrt + c = 0$ with $a,b,c \in \Z$ and $\Delta = b^2-4ac$ is not a square. 
    Let $\r \in \Q^2 \setminus \Z^2$ and $A \in \Gamma_\r$ such that $A \cdot \qrt = \qrt$.
    Then:
    \begin{itemize}
        \item[(1)] $\sfc{\r}{\!A}{\qrt}$ is an algebraic number.
        \item[(2)] If $g \in \Gal(\ol{\Q}/\Q)$ such that $g(\sqrt{\Delta}) = -\sqrt{\Delta}$, then $\abs{g(\sfc{\r}{\!A}{\qrt})}=1$.
    \end{itemize}
\end{conj}

We also state here two stronger conjectures. These are identical except that the former is restricted to fundamental discriminants, while the latter is for non-square discriminants.%

\begin{conj}[Fundamental Real Multiplication Values Conjecture]\label{conj:frmvc}
    Let $\qrt \in \R$ such that $a\qrt^2 + b\qrt + c = 0$ with $a,b,c \in \Z$ and $\Delta = b^2-4ac$ is a fundamental discriminant.
    Let $\r \in \Q^2 \setminus \Z^2$ and $A \in \Gamma_\r$ such that $A \cdot \qrt = \qrt$.
    Then:
    \begin{itemize}
        \item[(1)] $\sfc{\r}{\!A}{\qrt}$ is an algebraic unit in an abelian Galois extension of $\Q(\sqrt{\Delta})$.
        \item[(2)] If $g \in \Gal(\ol{\Q}/\Q)$ such that $g(\sqrt{\Delta}) = -\sqrt{\Delta}$, then $\abs{g(\sfc{\r}{\!A}{\qrt})}=1$.
    \end{itemize}
\end{conj}

\begin{conj}[General Real Multiplication Values Conjecture]\label{conj:grmvc}
    Let $\qrt \in \R$ such that $a\qrt^2 + b\qrt + c = 0$ with $a,b,c \in \Z$ and $\Delta = b^2-4ac$ is not a square.
    Let $\r \in \Q^2 \setminus \Z^2$ and $A \in \Gamma_\r$ such that $A \cdot \qrt = \qrt$.
    Then:
    \begin{itemize}
        \item[(1)] $\sfc{\r}{\!A}{\qrt}$ is an algebraic unit in an abelian Galois extension of $\Q(\sqrt{\Delta})$.
        \item[(2)] If $g \in \Gal(\ol{\Q}/\Q)$ such that $g(\sqrt{\Delta}) = -\sqrt{\Delta}$, then $\abs{g(\sfc{\r}{\!A}{\qrt})}=1$.
    \end{itemize}
\end{conj}

\Cref{conj:mrmvc} is implied by the Stark Conjecture (\Cref{conj:stark}), that is, the version of Stark's conjecture on special values of derivatives of partial zeta functions attached to real quadratic fields that is conjectured in Stark's original work. \Cref{conj:mrmvc} is indeed considerably weaker than the Stark Conjecture.

\Cref{conj:frmvc} is implied by the Stark--Tate Conjecture (\Cref{conj:stc}), which includes a small refinement of the Stark Conjecture due to Tate. Of course, \Cref{conj:frmvc} is also implied by \Cref{conj:grmvc}.

\Cref{conj:grmvc} is implied by the Monoid Stark Conjecture (\Cref{conj:msc}), a Stark-type conjecture for special values of derivatives of more general partial zeta functions attached to classes in ray class monoids. The Monoid Stark Conjecture is technically due to the third author (as it is equivalent to \cite[Conj.~1.4]{Kopp2020d}) and is not currently known to follow from the Stark-Tate Conjecture. %
The original form of the Stark Conjecture does imply that some integral power of $\sfc{\r}{\!A}{\qrt}$ is in an abelian extension of $\Q(\sqrt{\Delta})$; see \Cref{thm:field0}.

The conditional implications between the Stark-type conjectures and the RM values conjectures are summarized in the following theorem.
\begin{thm}\label{thm:starkimplications}
The following implications hold.
\begin{itemize}
    \item[(1)] \Cref{conj:stark} (the Stark Conjecture) implies \Cref{conj:mrmvc}.
    \item[(2)] %
    \Cref{conj:stc} (the Stark--Tate Conjecture) implies \Cref{conj:frmvc}.
    \item[(3)] %
    \Cref{conj:msc} (the Monoid Stark Conjecture) implies \Cref{conj:grmvc}.
\end{itemize}
\end{thm}
\begin{proof}
    See \Cref{ssec:conjcond}.
\end{proof}

\subsection{The main theorems: existence}\label{ssc:MainTheorems1}
We now state the main theorems on the existence of ghost $r$-SICs and live $r$-SICs, conditional on the Twisted Convolution Conjecture and the Stark Conjecture. These theorems are proven in \Cref{sec:existence}.

It will be helpful to attach a field $E_t$ to an admissible tuple $t$ in an unconditional manner independent of the connection to SICs. Conditionally, this field will be identical to the (extended projection) SIC field of any $r$-SIC fiducial associated to $t$.
\begin{defn}[Fields associated to an admissible tuple]\label{dfn:SICfield}
Let $t=(d,r,Q)\sim(K,j,m,Q)$ be an admissible tuple. 
\begin{itemize}
    \item[(1)]
    We define the \textit{field associated to $t$}, denoted $\sicField_t$, to be the field generated over $\Q$ by the numbers $\{\ghostOverlapC{t}{\mbf{p}}\colon 0\le p_1, p_2 < d, \, \mbf{p}\neq \zero\}$ together with $\rtu_d$.
    \item[(2)]
    We define the \textit{Galois-closed field associated to $t$}, denoted $\hat\sicField_t$, to be the Galois closure (within $\C$) of the compositum of $K$ and $E_t$.
\end{itemize}
\end{defn}
\begin{rmkb}
    It will be shown in \Cref{thm:RayClassField2} that, under the assumption of the Stark--Tate Conjecture, the field $\sicField_t$ associated to $t$ actually depends only on the pair $(d,r)$. Under the same assumption, this then also holds for $\hat\sicField_t$.
\end{rmkb}

The construction of $r$-SICs from admissible tuples requires some non-canonical choices. 
We bundle two additional pieces of data, a $2 \times 2$ matrix modulo $\db$ and a Galois automorphism, with an admissible tuple to form a \textit{fiducial datum}, from which an $r$-SIC fiducial will be constructed.
\begin{defn}[Fiducial datum]
\label{def:fiducialdata}
A \emph{fiducial datum} is a tuple $(d,r,Q,G,g)\sim (K,j,m,Q,G,g)$ such that $t=(d,r,Q)\sim(K,j,m,Q)$ is an admissible tuple, $G$ is an element of $\GLtwo{\mbb{Z}}$ whose determinant satisfies
\eag{
    \Det(G) r(2\shift+d_j-1+d) &\equiv 1 \Mod{\db}
    \label{eq:TwistCondition}
}
for some $\shift\in \mcl{Z}_t$, and $g$ is any element of $\Gal(\hat\sicField_t/\mbb{Q})$ such that $g(\sqrt{\Delta_0}) = -\sqrt{\Delta_0}$, where $\Delta_0$ is the fundamental discriminant of $Q$. 

We will sometimes write $s=(t,G,g)$, and say that the datum $s$ \emph{contains} or \emph{extends} the tuple $t$.
\end{defn}

\begin{rmkb}
    If $E_t$ contains transcendentals, then we make sense of the above definitions as follows: The field $\hat{E}_t$ is all of $\C$, and $\Gal(\hat{E}_t/\Q)$ is the full automorphism group of $\C$ over $\Q$.
    This will not matter in practice, because the Stark Conjecture will imply that $E_t$ is a finite Galois extension of $\Q$ and $\hat{E}_t = E_t$.
\end{rmkb}

It is not the case that $\normalizedGhostOverlapC{t}{\mbf{p}}$, $\ghostOverlapC{t}{\mbf{p}}$, considered as functions of $\mbf{p}$, have period $d$. 
It is, however, true that the products $\normalizedGhostOverlapC{t}{\mbf{p}}D_{\mbf{p}}$, $\ghostOverlapC{t}{\mbf{p}}D_{\mbf{p}}$ have period $d$ provided one excludes the case $\mbf{p}\equiv 0 \Mod{d}$.  
More generally, we have the following result:
\begin{lem}\label{lem:GhostFiducialIndependenceOfTransversal}
    Let $s=(t,G,g)$ be a fiducial datum, and let $\mbf{p},\mbf{p}'\in \mbb{Z}^2/d\mbb{Z}^2$ be such that $\mbf{p}'\equiv \mbf{p}\Mod{d}$ and $\mbf{p}$, $\mbf{p}'\notin d\mbb{Z}^2$. 
    Then
    \eag{
    \normalizedGhostOverlapC{t}{G\mbf{p}'}D_{\mbf{p}'} &= 
    \normalizedGhostOverlapC{t}{G\mbf{p}}D_{\mbf{p}}, & \ghostOverlapC{t}{G\mbf{p}'}D_{\mbf{p}'} &= 
    \ghostOverlapC{t}{G\mbf{p}}D_{\mbf{p}}.
    }
\end{lem}
 \begin{proof}
     The proof is given in \Cref{sbsc:ghostprops}, following \Cref{lem:nupperiodicity}.
 \end{proof}
\begin{defn}[Candidate ghost $r$-SIC fiducial $\tilde{\Pi}_s$, Candidate $r$-SIC fiducial $\Pi_s$, candidate normalized overlap]\label{dfn:CandidateGhostAndSICFiducials}
    Let $s=(d,r,Q,G,g)\sim(K,j,m,Q,G,g)$ be a fiducial datum, and let $t$ be the corresponding admissible tuple $(d,r,Q)\sim (K,j,m,Q)$. 
    We define the corresponding candidate ghost $r$-SIC fiducial by
    \eag{
    \label{eq:ghostProjectorDef}
\tilde{\Pi}_s &= \frac{1}{d}\sum_{\mbf{p}}
\ghostOverlapC{t}{G\mbf{p}}
D\vpu{t}_{\mbf{p}}
    }
    where the sum is over any complete set of coset representatives for $\mbb{Z}^2/d\mbb{Z}^2$, and 
    where $\ghostOverlapC{t}{\mbf{p}}$ is as defined in \Cref{dfn:GhostOverlaps}.

    We define the corresponding candidate $r$-SIC fiducial by
    \eag{
    \label{eq:sicProjectorDef}
     \Pi_s &= g(\tilde{\Pi}_s),
     \\
     \intertext{the candidate overlaps by}
          \overlapC{s}{\mbf{p}}&= \Tr\!\left(\Pi\vpu{\dagger}_sD^{\dagger}_{G^{-1}\mbf{p}}\right),
          \label{eq:CandidateOverlaps}
    \\
    \intertext{and, for $\mbf{p}\not \equiv \zero \Mod{d}$, the normalized candidate overlaps by}
      \normalizedOverlapC{s}{\mbf{p}} &= \sqrt{\frac{d^2-1}{r(d-r)}} \overlapC{s}{\mbf{p}}.
    }
\end{defn}
\begin{rmkb}
Note that this definition tacitly relies on \Cref{lem:GhostFiducialIndependenceOfTransversal}, which shows that that the summand on the RHS of \eqref{eq:ghostProjectorDef} is independent of the set of coset representatives chosen.
\end{rmkb}
The candidate overlaps can be expressed directly in terms of their ghost counterparts via:
\begin{lem}\label{lm:OverlapTermsGhostOverlap}
     Let $s=(d,r,Q,G,g)\sim(K,j,m,Q,G,g)$ be a fiducial datum, and let $t$ be the corresponding admissible tuple $(d,r,Q)\sim (K,j,m,Q)$.  Then
    \eag{
    \overlapC{s}{\mbf{p}}&=
      g\!\left(
    \ghostOverlapC{t}{GH_g^{-1}G^{-1}\mbf{p}}\right)
    }
    for all $\mbf{p}$,
    where $H_g$ is the matrix specified in \Cref{dfn:HgMatrixDefinition}.
\end{lem}
\begin{rmkb}
    Note that, unlike $\ghostOverlapC{t}{\mbf{p}}$, the candidate  overlaps $\overlapC{s}{\mbf{p}}$ depend on $G$ and $g$ as well as $t$. 
\end{rmkb}
\begin{proof}
    See \Cref{ssc:ecdgp}, following \Cref{thm:GalActOnClifford}.
\end{proof}
\begin{thm}\label{thm:ghstExist}
Assume \Cref{cnj:tci} (the Twisted Convolution Conjecture).  
Then, for every fiducial datum $s$, the corresponding operator $\tilde{\Pi}_s$ given in \Cref{dfn:CandidateGhostAndSICFiducials} is a ghost $r$-SIC fiducial.
\end{thm}
\begin{proof}
    See \Cref{sbsc:proofofghosttheorem}.
\end{proof}
\begin{thm}\label{thm:rayclassfieldrsicgen}
    Assume \Cref{cnj:tci} (the Twisted Convolution Conjecture), and also assume \Cref{conj:mrmvc} (as implied by \Cref{conj:stark}, the Stark Conjecture).
    Let $s=(d,r,Q,G,g)$ be a fiducial datum. 
    Then the operator $\Pi_s$ given in \Cref{dfn:CandidateGhostAndSICFiducials} is an $r$-SIC fiducial. 
\end{thm}
\begin{proof}
    See \Cref{sbsc:proofofghosttorsic}.
\end{proof}

\subsection{The main theorems: class fields attained}\label{ssc:MainTheorems2}

We now state our main conditional results about the abelian extensions generated by SICs. 
These results are based on unconditional results in pure algebraic number theory giving containment of certain class fields. 
They are proven in \Cref{sec:classfieldsattained}.

We use the following notation for orders of real quadratic fields.
\begin{defn}\label{dfn:orderConductorf}
    Given a real quadratic field $K$ and positive integer $f$, we denote the order with conductor $f$ in $K$ by $\mcl{O}_f$.  That is,
    \eag{
           \mcl{O}_f &= \left\{m+n f \left(\frac{\Delta_0+\sqrt{\Delta_0}}{2}\right)\colon m, n\in \mbb{Z}\right\},
     }
     where $\Delta_0$ is the discriminant of $K$.
\end{defn}
\begin{rmkb}
      Note that $K$ will always be clear from context. 
      In particular, the ring of integers $\mcl{O}_K$ may alternatively be written $\mcl{O}_1$. 
\end{rmkb}

The following two results give properties of the field $E_t$ associated to an admissible tuple $t$ within our framework of conjectures. 
Together, they show conditionally that $E_t$ is an abelian extension of the real quadratic field $K$ containing a particular ray class field. 
The latter theorem is restricted to the case when $Q$ has conductor $1$, i.e., $\disc(Q)$ is a fundamental discriminant.
\begin{thm}\label{thm:RayClassField}
    Let $t=(d,r,Q)\sim(K,j,m,Q)$ be an admissible tuple. Make the following conditional assumptions:
    \begin{itemize}
        \item If $\disc(Q)$ is fundamental, assume \Cref{conj:frmvc} (as implied by \Cref{conj:stc}, the Stark--Tate Conjecture).
        \item If $\disc(Q)$ is not fundamental, assume \Cref{conj:grmvc} (as implied by \Cref{conj:msc}, the Monoid Stark Conjecture).
    \end{itemize}
    Then the field $E_t$ is an abelian extension of $K$. 
\end{thm}
\begin{proof}
    See \Cref{sbsc:rayclassproof}.
\end{proof}

\begin{thm}\label{thm:RayClassField2}
    Let $t=(d,r,Q)\sim(K,j,m,Q)$ be an admissible tuple for which $\disc(Q)$ is a fundamental, and let $d=d_{j,m}$.
    Assume \Cref{conj:stc} (the Stark--Tate Conjecture).
    Let $E = H^{\OO_1}_{\db\infty_1\infty_2}$ be the ray class field with level datum $(\mcl{O}_1; \db \mcl{O}_1,\{\infty_1,\infty_2\})$, as defined by \Cref{thm:rayclassfield}.
    Then, $E$ is equal to the field extension of $K$ generated by the numbers $\{\ghostOverlapC{t}{\mbf{p}}^2\colon 0\le p_1, p_2 < d, \, \mbf{p}\neq \zero\}$ together with $\rtu_d$.
    The field $E_t \supseteq E \supseteq K$, the extension $E_t/K$ is ramified at both infinite places of $K$, and field $E_t$ depends only on the pair $(d,r)$.
\end{thm}
\begin{proof}
    See \Cref{sbsc:rayclassproof}.
\end{proof}

Empirically, it seems that $E_t$ is actually equal to the ray class field $E$ in \Cref{thm:RayClassField2}, and indeed a similar statement may be made when $Q$ is not fundamental. 
As we do not know how to prove this from any form of the Stark conjectures in the literature, we state it as a separate conjecture.
\begin{conj}\label{conj:RayClassField3}
    Let $t=(d,r,Q)\sim(K,j,m,Q)$ be an admissible tuple, let $d=d_{j,m}$, and let $f$ be the conductor of $Q$.
    Let $E= H^{\OO_f}_{\db\infty_1\infty_2}$ be the ray class field with level datum $(\mcl{O}_f; \db \mcl{O}_f,(\infty_1,\infty_2))$, as defined by \Cref{thm:rayclassfield}. 
    Then $\hat{E}_t = E_t = E$.
\end{conj}

Our results suggest that $r$-SICs provide a geometric interpretation of class field theory over a real quadratic field $K$. 
Thus, we'd like to realize arbitrary abelian extensions of $K$ using $r$-SICs. 
We show conditionally that this is possible when the trace of the fundamental unit is odd.

\begin{thm}\label{thm:cofinal}
    Assume the \Cref{conj:stc} (the Stark--Tate Conjecture).
    Let $K$ be a real quadratic field of discriminant $\Delta_0$, and let $\vn$ be a fundamental totally positive unit in $K$ (as in \Cref{dfn:fundamentalTotallyPositiveUnit}).
    \begin{itemize}
        \item[(1)]
        If $\Tr(\vn)$ is odd, then every abelian extension of $K$ is contained in $E_t$ for some admissible tuple $t \sim (d,r,Q)$ with $\disc(Q) = \Delta_0$.
        \item[(2)]
        If $\Tr(\vn)$ is even, then every abelian extension of $K$ that is unramified at the primes of $K$ lying over $2$ is contained in $E_t$ for some admissible tuple $t \sim (d,r,Q)$ with $\disc(Q) = \Delta_0$.
    \end{itemize}
\end{thm}
\begin{proof}
    See \Cref{ssec:cofinal}.
\end{proof}

The condition that $\Tr(\e)$ is odd is common among real quadratic fields. When ordered by discriminant, the condition holds for at least $7.4\%$ of real quadratic $K$, in the sense of asymptotic density, by \Cref{thm:oddtracecount}. 
(The true density looks empirically like $22.2\%$.) 
When ordered by root dimension $d_1$, $\Tr(\e)$ is odd if and only if $d_1$ is even, so instead $50\%$ of real quadratic $K$ satisfy the condition. 
For those real quadratic fields for which $\Tr(\e)$ is even, \Cref{thm:cofinal} still says that ``many'' abelian extensions are contained in some $E_t$.

\Cref{thm:cofinal} makes clear the relevance of $r$-SICs to Hilbert's twelfth problem of generating abelian extensions from special values of explicit complex-analytic functions. Specifically, a proof of the Stark--Tate Conjecture and the Twisted Convolution Conjecture would give a solution to Hilbert's twelfth problem that is both complex-analytic and geometric, for a positive proportion of real quadratic fields. Our construction is complex-analytic because the $\shin$-function is a complex analytic function. It is geometric both in the sense that $r$-SICs are described by sets of pairwise equichordal subspaces, and in the sense that the algebraic equations for a Weyl--Heisenberg $r$-SIC projector cut out an algebraic variety.

\clearpage

\subsection{Table of notation}\label{sec:tableofnotation}
For the convenience of the reader we include the following summary of the notation and terminology used in this paper.

\begin{longtable}{p{3.9 cm} p{9.79 cm} p{1.8 cm}}
    \toprule
        Notation & Terminology & Definition\\
    \midrule
    $\mcl{L}(\C^d)$ & space of linear operators on $\C^d$ & --
    \\
    -- & \hprj & \ref{dfn:Hprojector}
    \\
    -- & \pprj{} & \ref{def:pProjector}
    \\
      --  & $r$-SIC &\ref{def:equiangularcond}
    \\
      --   & Zauner's Conjecture  & Conj.~\ref{conj:zauner}
    \\
    $\WH(d)$    & Weyl--Heisenberg group in dimension $d$ & \ref{def:WHGroup}
    \\
        $\bar{d}$  & $d$ (resp.,~$2d$) if $d$ is odd (resp.,~even) & \ref{def:WHGroup}
    \\
        $\rtus_d$ & $e^{2 \pi i/d}$ & \ref{def:WHGroup}
    \\
        $\rtu_d$ & $-e^{\pi i/d}$ & \ref{def:WHGroup}
    \\
       $X$, $Z$, $D_{\mbf{p}}$ & $\WH(d)$ displacement operators & \ref{def:WHGroup}
    \\
        $U_P$ & parity operator & \ref{def:pProjector}, \ref{dfn:parityMatrix}
    \\
      $P$ & parity matrix & \ref{dfn:parityMatrix}
    \\
    --    & WH covariant $r$-SIC & \ref{dfn:whcovrsic}
    \\
        $\Pi$ & fiducial \hprj{}  & \ref{dfn:whcovrsic}
     \\
        $\Pi$ & live fiducial (alternative name for fiducial \hprj{})  & \ref{dfn:liveFiducial}
     \\
        $\overlap_{\mbf{p}}$ & \overlapSingularText{} & \ref{dfn:sicovlp}
    \\
      $\normalizedOverlap_{\mbf{p}}$ &\normalizedOverlapSingularText{} & \ref{dfn:sicovlp}
     \\
        $\tilde{\Pi}$ & ghost fiducial \pprj{} &\ref{dfn:ghostFiducial}
    \\
       $ \ghostOverlap_{\mbf{p}}$ & \ghostOverlapSingularText{} &   \ref{dfn:ghostFiducial}
    \\
       $ \ghostOverlapC{t}{\mbf{p}}$ & \candidateGhostOverlapSingularText{} for admissible tuple $t$&   \ref{dfn:ghostFiducial}
    \\
       $\normalizedGhostOverlap_{\mbf{p}}$ & \normalizedGhostOverlapSingularText{} &   \ref{dfn:ghostFiducial}
       \\
       $\normalizedGhostOverlapC{t}{\mbf{p}}$ & \candidateNormalizedGhostOverlapSingularText{} for admissible tuple $t$ & \ref{dfn:GhostOverlaps}
        \\
        $G$ & twist (of the ghost overlap index $\p$) & \ref{dfn:ghostFiducial}
    \\
        $\mbb{H}$ & upper half-plane & \eqref{eq:upperHalfPlane}
    \\
        $\varpi$, $\varpi_n$ & variant $q$-Pochhammer symbols  & \ref{dfn:variantqPochhammer}
    \\
    $\symp{\mbf{p}}{\mbf{q}}$ & symplectic form & \eqref{eq:SymplecticForm}
    \\
    $\pt{\mbf{p}}{\tau}$ & fractional symplectic form & \eqref{eq:FractionalSymplecticForm} 
    \\
        $\eta(\tau)$ & Dedekind $\eta$-function & \eqref{eq:DedekindEta}
    \\
       $M\cdot\tau$, $j_M(\tau)$ & fractional linear transformation and its denominator & \eqref{eq:fractionalineartransform}
    \\
    $\DD_M$ & domain of $\tau$ in  \SFKShort{} Jacobi and  \SFKShort{}  modular cocycles & \ref{def:sl2ldmndf}
    \\
        $\Gamma(d)$, $\Gamma_\r$ & principal congruence subgroup and a variant & \ref{dfn:gammarDef}
    \\
    $\sfjM{M}{z}{\tau}$ & \SFKShort{} Jacobi cocycle & \ref{df:shinfadjacocycle}
    \\
    $\sfc{\r}{A}{\tau}$ & \SFKShort{} modular cocycle & \ref{def:shin}
    \\
        $K$, $\Delta_0$ & real quadratic field and its discriminant &\eqref{eq:discriminant}
    \\
    $Q$ & integral, primitive, irreducible, indefinite  quadratic form & Sec.~\ref{ssec:quadraticfieldsforms}
    \\
    $\lOnQ{M}{Q}$ & $M$-transform of $Q$ & \eqref{eq:afrmtrans}
    \\
        $\Delta_0$, $f$ & fundamental discriminant and conductor of a form & \eqref{eq:ConductorOfQ}
    \\
    $\qrt_{Q,\pm}$ & roots of $Q$ & \ref{dfn:qrtdf}
    \\
    $\qrt_t$ & root corresponding to admissible tuple $t$ & \ref{dfn:GhostOverlaps}
    \\
        $\sgn(Q)$, $\sgn(M)$ & signs of $Q$, $M$ & \ref{dfn:sign}
    \\
        $\mcl{S}(Q)$, $\mcl{S}_d(Q)$  & stability group of $Q$, and a variant
        & \ref{dfn:stbgpqdf}
    \\
    $\vn$ & fundamental totally positive  unit $> 1$ & \ref{dfn:fundamentalTotallyPositiveUnit}
    \\
    $f_j$ & sequence of conductors of a real quadratic field & \ref{dfn:sequenceofconductors}
    \\
        $d_{j,m}$, $r_{j,m}$ & dimension and rank grids of a real quadratic field & \ref{dfn:fjrjmdjm}
    \\
        $(d,r)$ & admissible pair & \ref{dfn:admissiblePair}
    \\
        $d_j$, $d_1$ & dimension tower, root dimension of a real quadratic field & \ref{dfn:fjrjmdjm}
    \\
        $(K,j,m)$ & admissible triple & \ref{dfn:fjrjmdjm}
    \\
        $(d,r)\sim(K,j,m)\hspace{-1em}$ & admissible tuple equivalence & \ref{def:tupleequiv}
    \\
        $\rade(M)$ & Rademacher class invariant & \ref{df:meyinv}
    \\
        $(d,r,Q)$  & admissible tuple with form & \ref{def:admissibleform}
    \\
       $(K,j,m,Q)$ & admissible tuple with form & \ref{def:admissibleform}
    \\
       $\stabQGen{t}$, $\posGen{t}$, $\zaunerGen{t}$, $A_t$ & stabilizers associated to admissible tuple $t$ & \ref{dfn:AssociatedStabilizers}
    \\
        $\SFPhase{t}{\mbf{p}}$
        & \SFKShort{} phase for admissible tuple $t$ & \ref{dfn:SFKPhase}
    \\
$\delta^{(n)}_{\mbf{p},\mbf{q}}$ & modular $\delta$-function & \ref{dfn:modulardeltafunction}
    \\
        $\shift$, $\mcl{Z}_t$ & shift and set of shifts for admissible tuple $t$ & \ref{dfn:shift}
    \\
--         & Stark--Tate Conjecture & Conj.~\ref{conj:stc} 
    \\
   --      & Twisted Convolution Conjecture & Conj.~\ref{cnj:tci}
    \\
        $\sicField_t$, $\hat{\sicField_t}$ & field \& Galois closed field associated to $t$ & \ref{dfn:SICfield}
    \\
        $\mcl{O}_f$ & order with conductor $f$ & \ref{dfn:orderConductorf}
    \\
        $g$ & standard notation for a Galois automorphism 
    \\
        $s=(d,r,Q,G,g)$ & fiducial datum & \ref{def:fiducialdata}
    \\
        $s=(K,j,m,Q,G,g)$ & fiducial datum & \ref{def:fiducialdata}
     \\
        $s=(t,G,g)$ & fiducial datum extending admissible tuple $t$ & \ref{def:fiducialdata}
     \\
          $\tilde{\Pi}_s$ & ghost  \pprj{} for fiducial datum  $s$ &      \ref{dfn:CandidateGhostAndSICFiducials}
    \\
       $\Pi_s$ & $r$-SIC \hprj{} for  fiducial datum $s$ & \ref{dfn:CandidateGhostAndSICFiducials}
    \\
    $\normalizedOverlapC{s}{\mbf{p}}$ & candidate \normalizedOverlapSingularText{} for fiducial datum $s$ & \ref{dfn:CandidateGhostAndSICFiducials}
    \\
    $\Cl_{\mm,\rS}(\OO)$ & ray class group & \ref{defn:rayclassgroup}
    \\
    $\Clt_{\mm,\rS}(\OO)$ & flat imprimitive ray class monoid & \ref{defn:rayclassmonoid}
    \\
    $\ZClt_{\mm,\rS}(\OO)$ & submonoid of zero classes & \ref{defn:rayclassmonoid}
    \\
    $\A$ & ray class in $\Cl_{\mm,\rS}(\OO)$ or $\Clt_{\mm,\rS}(\OO)$ & \ref{defn:rayclasspartialzeta}
    \\
    $\zeta_{\mm,\rS}(s,\A)$ & ray class partial zeta function & \ref{defn:rayclasspartialzeta}
    \\
    $Z_{\mm,\rS}(s,\A)$ & differenced ray class partial zeta function & \ref{defn:rayclasspartialzeta}
    \\
    $\su_\A$ & Stark unit & Sec.~\ref{ssc:StarkConjectures}
    \\
    $\psi$ & $\eta$-character & ~\eqref{eq:etatrans}
    \\
    $\theta_\r(\tau)$ & Jacobi theta function with characteristics & ~\eqref{eq:jacobitheta}
    \\
    $\chi_\r$ & $\theta_\r$-character & ~\eqref{eq:thetatrans}
    \\
       $\Cliff(d)$, $\EC(d)$ & Clifford and extended Clifford groups  & \ref{dfn:CliffordGroup}
    \\
        $\PC(d)$, $\PEC(d)$ & projective Clifford and extended Clifford groups  & \ref{dfn:CliffordGroup}
    \\
    $\SLtwo{\mbb{Z}/\bar{d}\mbb{Z}}$  & symplectic group & \ref{dfn:symplectic}
    \\
        $\ESLtwo{\mbb{Z}/\bar{d}\mbb{Z}}$ & extended symplectic group & \ref{dfn:symplectic}
    \\
      $F$ & symplectic and anti-symplectic matrices & \ref{dfn:symplectic}
    \\
      $J$ & canonical anti-symplectic matrix & \ref{dfn:symplectic}
    \\
        $U_F$ & symplectic unitary & \eqref{eq:symplecticUnitary}
    \\
    $k_g$  & integer describing action of Galois conjugation $g$ & \ref{dfn:HgMatrixDefinition}
    \\
    $H_g$ & matrix describing  action of Galois conjugation $g$ & \ref{dfn:HgMatrixDefinition}
    \\
        $U_F$ & anti-symplectic anti-unitary & \eqref{eq:antiSymplecticAntiUnitary}
    \\
        $\ECS(d)$ & (anti-)symplectic subgroup of $\EC(d)$ & \ref{df:antisymplecticSubgroup}
    \\
        $\mcl{S}(\Pi)$ & symmetry group of fiducial $\Pi$ & \ref{dfn:symmetryGroupOfFiducial}
    \\
    -- & canonical order 3 unitary & \ref{df:canonicalOrder3}
    \\
     $F\vpu{'}_z$, $F\vpu{'}_a$, $F'_a$ & Zauner and variant Zauner matrices &
    \ref{dfn:FzFaFpaMatrices}
    \\
      --  & type-$z$, type-$a$, type-$a'$ fiducials & \ref{df:typeztypeatypeaprime}
    \\
    -- & centered fiducial & \ref{dfn:centered}
    \\
    $\StabPiESL{\Pi}$ & symplectic symmetry group of fiducial $\Pi$  & \ref{dfn:sogppi}
    \\
        $\StabPiOverlap{\Pi}$ & \overlapSingularText{} symmetry group of fiducial $\Pi$ & \ref{dfn:sogppi}
    \\
        -- & Galois multiplet & \ref{dfn:GaloisMultiplet}
    \\
        -- & strongly centered fiducial & \ref{dfn:stronglycentered}
    \\
        $\un$ & fundamental unit of $K$ & \ref{dfn:fundamentalUnit}
    \\
        $\Delta_j$ & discriminant at level $j$ & \ref{dfn:discriminantLevelj}
    \\
        $T^{*}_j(x)$, $U^{*}_j(x)$ & variant Chebyshev polynomials & \ref{dfn:variantChebyshev}
    \\
        $\mcl{U}_f$ & unit group of $\mcl{O}_f$ & \ref{df:ofufdef}
    \\
        $\mcl{U}^{+}_f$ & positive norm subgroup of $\mcl{U}_f$ & \ref{df:ofufdef}
     \\
        $j_{\rm{min}}(f)$ & minimum level of conductor $f$ & \ref{df:epsilonfDefinition}
    \\
        $\vn_f$ & smallest unit greater than 1 in $\mcl{U}^{+}_f$ & \ref{df:epsilonfDefinition}
    \\
        $\un_f$ & smallest unit greater than 1 in $\mcl{U}_f$ & \ref{df:ufvf}
    \\
        $\mcl{M}(R)$ & ring of $2\times 2$ matrices over $R$ & \ref{dfn:MatrixRings}
    \\
        $\mcl{M}_{\rm{S}}(R)$ & sub ring of symmetric matrices in $\mcl{M}(R)$ & \ref{dfn:MatrixRings}
    \\
        $\mcl{M}_{0}(R)$ & sub ring of trace-zero matrices in $\mcl{M}(R)$ & \ref{dfn:MatrixRings}
    \\
        $R\la M_1, \dots, M_n\ra$
        & matrix sub-algebra generated by $M_1 \dots, M_n$ & \ref{dfn:subAlgebra}
    \\
        $S$, $T$ & generators of $\SLtwo{\mbb{Z}}$ & \ref{dfn:GeneratorsOfSL2Z}
    \\
        $\canrep$ & canonical representation of field $K$ & \ref{dfn:canonicalRepresentation}
    \\
        $\canrep_Q$ & canonical representation of $K$ associated to $Q$ & \ref{df:etaq}
    \\
        $Q(M)$ & form stabilized by $M$ & Thm.~\ref{tm:ltnqdecomp}
    \\
      $\mcl{H}_{+}$ & $M\in \GLtwo{\mbb{Z}} \setminus \{\pm I\} $ such that $\left(\Tr M\right)^2 -4\Det M > 4 $ & \ref{dfn:EllipticHyperbolic}
    \\
      $\mcl{H}_{-}$ & $M\in \GLtwo{\mbb{Z}} \setminus \{\pm I\} $ such that $\left(\Tr M\right)^2 -4\Det M \le 4 $ & \ref{dfn:EllipticHyperbolic}
    \\
        $\mcl{F}_{+}$ & invariant, irreducible forms &  \ref{dfn:EllipticHyperbolic}
    \\
        $\mcl{F}_{-}$ & forms with discriminant $-4,-3,0,1$, or $4$ &  \ref{dfn:EllipticHyperbolic}
    \\
    $n_t$ & level of admissible tuple $t$ & \ref{dfn:level}
    \\
        $\fn_t$ & function from $\mbb{Z}/(d\mbb{Z})$ to $\mbb{Z}/(\bar{d}\mbb{Z})$ & \ref{dfn:functionht}
    \\
     $E^{(1)}_s$, $E^{(2)}_t$ & subfields of $E_t$ generated by (ghost) overlaps &\ref{defn:tuplefields}
    \\
        $t_M$ & $M$-transformed admissible tuple  
        & \ref{dfn:MTransformedt}
    \\
        $t\sim t'$ & equivalence of admissible tuples $t$ and $t'$ & \ref{dfn:MTransformedt}
    \\
        $s_M$ & $M$-transformed fiducial datum 
        & \ref{dfn:MTransformedt}
     \\
        $\GLMorph_s$ & homomorphism of $\GLtwo{\mbb{Z}}$ onto $\ESLtwo{\mbb{Z}/\db\mbb{Z}}$  & \ref{dfn:GLHomomorphism}
    \\
      $[t]$ & equivalence class of tuples specifying an $\EC(d)$ orbit & \ref{df:EquivalenceClassesTuples}
    \\
        $\dBr{t}$ & equivalence class of tuples specifying a Galois multiplet & 
        \ref{df:EquivalenceClassesTuples}
    \\
      $H_t$, $H_{K,f}$ & ring class field for admissible tuple $t$ & \ref{df:RingClassField}
    \\
        $h_t$, $h_{K,f}$ & class number for admissible tuple $t$ & \ref{df:RingClassField}
     \\
     $H_{\dBr{t}}$  & ring class field for equivalence class $\dBr{t}$ & \ref{df:RingClassField}
     \\
     $h_{\dBr{t}}$ & class number for equivalence class $\dBr{t}$ & \ref{df:RingClassField}
    \\
    $\ESLsym{s}$, $\ESLposSym{s}$, $\ESLzaunSym{s}$ & elements of overlap stabilizer group associated to  $s$ & \ref{df:ESLStabilizers}
    \\
    -- & tuples of unitary/anti-unitary type & \ref{dfn:antiUnitaryType}
    \\
    -- & canonical expansion, length of canonical expansion & \ref{dfn:canonicalExpansion}
    \\
    \bottomrule
\end{longtable}

\section{Shintani--Faddeev cocycles and the Stark conjectures}\label{sec:stark}

This section summarizes some of the algebraic properties of the Shintani--Faddeev modular cocycle established in \cite{Kopp2020d} as well as its relationship to the Stark conjectures, after first providing some necessary background. Proofs of theorems not proven here may be found in \cite{Kopp2020b} and \cite{Kopp2020d}. This section assumes some familiarity with algebraic number theory; for a standard text on the subject, see~\cite{Neukirch1999}, or see \cite{Marcus:1977} for a more elementary exposition.

\subsection{Class field theory (for orders of number fields)}\label{ssec:cft}

For a number field $K$, it is natural to ask for a characterization of the set of abelian Galois extensions
\begin{equation}
    \{H/K : H \mbox{ is a number field and } \Gal(H/K) \mbox{ is abelian}\}.
\end{equation}
Such fields are characterized abstractly by class field theory. Class field theory realizes every abelian extension of $K$ as a subfield of a \textit{ray class field}; ray class fields, proven to exist by Takagi, are parameterized by data intrinsic to the base field $K$.

It suits our purposes to give a broader definition of \textit{ray class field} than is typical. Takagi's ray class fields are attached to the data of a \textit{modulus}, which is a pair $(\mm,\rS)$ such that $\mm$ is a nonzero ideal of the ring of integers $\OO_K$ and $\rS$ is a subset of the set of \emph{real embeddings} of $K$ (that is, injective ring homomorphisms $K \to \R$). More general ray class field are used here in the sense defined by Kopp and Lagarias \cite{Kopp2020b}. Each is attached to a \textit{level datum}, which is a triple $(\OO; \mm, \rS)$ such that $\OO$ is an order\footnote{An order $\OO$ in a number field $K$ is a subring of $K$ having rank $[K : \Q]$ as an abelian group.} in the number field $K$, $\mm$ an ideal of $\OO$, and $\rS$ a subset of the set of real embeddings of $K$.

A level datum is used directly to define the \textit{ray class group}, a finite abelian group, which will by the main theorems of class field theory be isomorphic to the Galois group over $K$ of the corresponding ray class field. 
The definition of the ray class group uses \textit{fractional ideals}, which may be defined for an order $\OO$ in a number field $K$ equivalently as either:
\begin{enumerate}
    \item a \textit{fractional ideal} $\mathfrak{a}$ of $\OO$ is a finitely-generated $\OO$-submodule of $K$;
    \item a \textit{fractional ideal} $\mathfrak{a}$ is an additive subgroup of $K$ with the property that there is some $n \in \N$ such that $n\mathfrak{a}$ is an ideal of $\OO$.
\end{enumerate}
Ideals of $\OO$ will be called \textit{integral ideals} to distinguish them from more general fractional ideals. Fractional ideals may be multiplied together to give new fractional ideals, using the multiplication 
\begin{equation}
    \mathfrak{a} \mathfrak{b}
    = \left\{\sum_{i=1}^k a_i b_i : a_i \in \mathfrak{a}, b_i \in \mathfrak{b}, k \in \N\right\}.
\end{equation}
Nonzero fractional ideals form a group $\rJ^\ast(\OO)$.

The ray class group is defined as a quotient of a subgroup of $\rJ^\ast(\OO)$ by a smaller subgroup, so its elements are cosets consisting of fractional ideals.
The following definition, given as \cite[Defn.~5.4]{Kopp2020b}, generalizes the standard one by introducing a dependence on $\OO$.
\begin{defn}[Ray class group]\label{defn:rayclassgroup}
Let $K$ be a number field and $(\OO; \mm, \rS)$ be a level datum for $K$.
The \textit{ray class group of the order $\OO$ modulo $(\mm, \rS)$} is
\begin{equation}
\Cl_{\mm,\rS}(\OO) = \frac{\rJ_{\mm}^\ast(\OO)}{\rP_{\mm,\rS}(\OO)},
\end{equation}
where
\begin{align}
\rJ_{\mm}^\ast(\OO) &= \{\mbox{invertible fractional ideals of $\OO$ coprime to $\mm$}\}, \mbox{ and} \\
\rP_{\mm,\rS}(\OO) &= \{\alpha\OO \mbox{ such that } \alpha \con 1 \Mod{\mm} \mbox{ and } \rem(\alpha)>0 \mbox{ for } \rem \in \rS\}.
\end{align}
If the real embeddings of $K$ are labelled $\rem_1, \ldots, \rem_r$ and $\rS = \{\rem_{j_1}, \ldots, \rem_{j_k}\}$, the pair $(\mm,\rS)$ may be abbreviated as $\mm\infty_{j_1} \cdots \infty_{j_k}$.
\end{defn}

It is worth highlighting the meanings of the terms ``invertible'' and ``coprime'' in the above definition, as they involve features that do not appear in the maximal order case. A fractional $\OO$-ideal $\fa$ is \textit{invertible} if there is some fractional $\OO$-ideal $\bb$ such that $\fa\bb = \OO$. (The order $\OO = \OO_K$ if and only if all nonzero ideals are invertible.) The fractional ideal $\fa$ is \textit{coprime} to the integral ideal $\mm$ if it can be written as $\fa = \fa_1\fa_2^{-1}$ for an integral $\OO$-ideal $\fa_1$ and an invertible integral $\OO$-ideal $\fa_2$ satisfying $\fa_1+\mm = \fa_2+\mm = \OO$. (Ideals of non-maximal orders do not always have prime factorizations.)

The following theorem defines ray class fields uniquely and asserts their existence. It is stated as \cite[Thm.\ 3.4]{Kopp2020d} and is a summary of \cite[Thm.\ 1.1, Thm.\ 1.2, Thm.\ 1.3]{Kopp2020b}.
\begin{thm}\label{thm:rayclassfield} 
Let $K$ be a number field and $(\OO; \mm, \rS)$ be a level datum for $K$.
Then there exists a unique abelian Galois extension 
$H_{\mm,\rS}^{\OO}/K$ with the property that a prime ideal $\pp$ of $\OO_K$ that is 
coprime to the quotient ideal $\colonideal{\mm}{\OO_K}$ 
splits completely in $H_{\mm,\rS}^{\OO}/K$ if and only if $\pp \cap \OO = \pi\OO$, a principal prime $\OO$-ideal having $\pi \in \OO$ with $\pi \equiv 1 \Mod{\mm}$ and $\rem(\pi)>0$ for $\rem \in \rS$.

Additionally, these fields have the following properties:
\begin{itemize}
\item $H_{\mm\OO_K,\rS}^{\OO_K} \subseteq H_{\mm,\rS}^{\OO} \subseteq H_{\colonideal{\mm}{\OO_K},\rS}^{\OO_K}$.
\item There is a canonical isomorphism $\Art_{\OO} : \Cl_{\mm, \rS}(\OO) \to \Gal\!\left(H_{\mm,\rS}^{\OO}/K\right)$.
\end{itemize}
\end{thm}

Another formal structure, the \textit{\rcmia}, will be needed to define generalized zeta values that ultimately give rise to ghost $r$-SIC overlaps (by way of the $\shin$-function). This finite commutative monoid\footnote{A monoid is a semigroup with identity, that is, a set with a binary operation satisfying associativity and having an identity element.} contains the ray class group as a submonoid and is defined by weakening the ``coprime to $\mm$'' condition in \Cref{defn:rayclassgroup} and making other modifications.
Further discussion and properties are given in \cite[Sec.~4]{Kopp2024} and \cite[Sec.~3.2]{Kopp2020d}.
\begin{defn}\label{defn:rayclassmonoid}
The \textit{\rcmia} is
\begin{equation}
\Clt_{\mm,\rS}(\OO) = \frac{\ol\rJ_{\mm}^\flat(\OO)}{\sim_{\mm,\rS}},
\end{equation}
where
\begin{align}
\ol{\rJ}_{\mm}^\flat(\OO) &= \{\aaF \in \rJ_\OO^\ast(\OO) : \aaF\OO[S_\mm^{-1}] \subseteq \OO[S_\mm^{-1}]\} \mbox{ with} \\
S_\mm &= \{\alpha \in \OO : \alpha\OO + \mm = \OO\},
\end{align}
and the equivalence relation $\sim_{\mm,\rS}$ is defined by
\begin{equation}
\aaF \sim_{\mm,\rS} \bb \iff \begin{array}{c}\exists \ccc \in \ol\rJ_{\mm}^{\flat}(\OO) \mbox{ and } \alpha,\beta \in \OO[S_\mm^{-1}] \mbox{ such that } \aaF = \alpha\ccc, \bb = \beta\ccc, \\ 
\alpha - \beta \in \mm\OO[S_\mm^{-1}],
\sgn(\rem(\alpha)) = \sgn(\rem(\beta)) \mbox{ for all } \rem \in \rS.\end{array}
\end{equation}
The \textit{submonoid of zero classes} is
\begin{equation}
    \ZClt_{\mm,\rS}(\OO) = \{[\dd] \in \Clt_{\mm,\rS}(\OO) : \dd \subseteq \mm\}.
\end{equation}
\end{defn}

\subsection{Partial zeta functions} \label{ssc:ZetaFunctions}

The Stark conjectures relate the value at $s=0$ of certain zeta functions to algebraic units in certain number fields. The zeta functions are \textit{partial zeta functions}, meaning that they are defined by an infinite sum corresponding to ``part'' of a Dirichlet series used to define another zeta function. The Dedekind zeta function is written as a finite sum of partial zeta functions. Those partial zeta functions may be indexed either by ray classes in a ray class group (and more generally a ray class monoid) or by field automorphisms in a finite Galois extension.

In the simplest case, the Dedekind zeta function of $\Q$ is the Riemann zeta function
\begin{align}
    \zeta_\Q(s) = \sum_{n=1}^\infty n^{-s}, && \re(s)>1.
\end{align}
For any positive integer $d$, the Riemann zeta function may be written as a finite sum of Hurwitz zeta functions
\begin{align}\label{eq:Rzetadecomp}
    \zeta_\Q(s) = \sum_{k=1}^{d} d^{-s}\zeta(s,\tfrac{k}{d}),
\end{align}
where the Hurwitz zeta function is defined as
\begin{align}
    \zeta(s,a) = \sum_{n=0}^\infty (n+a)^{-s}, && \re(s)>1.
\end{align}
The Hurwitz zeta function may be understood as a partial zeta function associated to the congruence class of $k \Mod{d}$. 
Such $k$ may be thought of as classes in a \rcmia,
\begin{align}\label{eq:monoidzeg}
    \Clt_{d\infty}(\Z) 
    &= \frac{\{r\Z : r = \tfrac{a}{b} \in \Q^\times, \gcd(b,d)=1\}}{\left(r_1\Z \sim r_2\Z \text{ if } r_i = \pm \tfrac{\gamma_i c}{\delta_i d}, \, \tfrac{\gamma_1}{\delta_1} - \tfrac{\gamma_2}{\delta_2} = \tfrac{\gamma_3}{\delta_3}, \, m|\delta_3, \, \sgn(\tfrac{\gamma_1}{\delta_1})=\sgn(\tfrac{\gamma_2}{\delta_2})\right)} \\
    &\isom (\Z/d\Z, \times), \label{eq:zdztimes}
\end{align}
where we stipulate that all fractions in \eqref{eq:monoidzeg} are in simplest form, and the notation in \eqref{eq:zdztimes} indicates the set $\Z/d\Z$ thought of as a monoid with the binary operation of multiplication.
When $k$ is coprime to $d$, it may be thought of as either a class in the ray class group
\begin{align}
    \Cl_{d\infty}(\Q) 
    &= \frac{\{r\Z : r = \tfrac{a}{b} \in \Q^\times, \gcd(a,d)=\gcd(b,d)=1\}}{\{r\Z : r = \tfrac{a}{b} \in \Q^\times, \, \gcd(a,d)=\gcd(b,d)=1, \, a\equiv b \Mod{d}, \, r>0\}} \\
    &\isom \left(\Z/d\Z\right)^\times
\end{align}
or to an element of the Galois group
\begin{align}
    \Gal(\Q(\rtus_d)/\Q) 
    &= \{\mbox{field automorphisms } g : \Q(\rtus_d) \to \Q(\rtus_d)\} \\
    &\isom \left(\Z/d\Z\right)^\times.
\end{align}
The restriction that $k$ is coprime to $d$ is no great obstacle, as \eqref{eq:Rzetadecomp} may be rewritten as
\begin{align}
    \zeta_\Q(s) = \left(\prod_{\substack{p|d \\ p \text{ prime}}} (1-p^{-s})^{-1}\right)\sum_{\substack{1 \leq k \leq d \\ \gcd(k,d)=1}} d^{-s}\zeta(s,\tfrac{k}{d}).
\end{align}

The Dedekind zeta function
\begin{equation}
    \zeta_K(s) = \sum_{\aaF \subseteq \OO_K} \Nm(\aaF)^{-s}
\end{equation}
which generalizes the Riemann zeta function, can likewise be split up as a sum of finitely many ray class partial zeta functions. 
\begin{defn}[Ray class partial zeta function and differenced ray class partial zeta function]
\label{defn:rayclasspartialzeta}
Let $K$ be a number field and $(\OO; \mm,\rS)$ a level datum for $K$.
Let $\A \in \Clt_{\mm,\rS}(\OO)$, and let $\sR$ be the element of $\Cl_{\mm,\rS}(\OO)$ defined by
\begin{equation}
\sR := \{\alpha\OO : \alpha \equiv -1 \Mod{\mm} \mbox{ and } \rem(\alpha)>0 \mbox{ for all } \rem \in \rS\}.
\end{equation} 
For $\re(s)>1$, define the \textit{ray class partial zeta function} and  the \textit{differenced ray class partial zeta function}, respectively, by
\begin{align}
\zeta_{\mm,\rS}(s,\A) &= \sum_{\substack{\aaF \subseteq \OO \\ \aaF \in \A}} \Nm(\aaF)^{-s}, \mbox{ and} \\
Z_{\mm,\rS}(s,\A) &= \zeta_{\mm,\rS}(s,\A) - \zeta_{\mm,\rS}(s,\sR\A).
\end{align}
\end{defn}
Ray class partial zeta functions are closely related (by Artin reciprocity) to partial zeta functions indexed by elements of a Galois group. We introduce different terminology and notation for Galois-thoeretic partial zeta functions, which are not always identical to ray class partial zeta functions, as they impose stricter coprimality conditions on the ideals indexing the summands.
\begin{defn}[Galois-theoretic partial zeta function]
Let $H/K$ be an abelian Galois extension of number fields. Let $S$ be a finite set of places of $K$ containing all the places that ramify in $H$ as well as all the infinite places of $K$, and let $S=S_{\rm fin} \sqcup S_\infty$ for a set of finite places $S_{\rm fin}$ and a set of infinite places $S_\infty$. For any $g \in \Gal(H/K)$ and $\re(s)>1$, define
\begin{equation}\label{eq:galseries}
    \zeta_S^{\Gal}(g,s) = \sum_{\substack{\aaF \subseteq \OO_K \\ (\forall \pp \in S_{\rm fin}) \aaF+\pp=\OO_K \\ \Art([\aaF])=g}} \Nm(\aaF)^{-s},
\end{equation}
where $\Art = \Art_{\OO_K}$ is the Artin map of class field theory.
\end{defn}
In the case when $\OO=\OO_K$ is the maximal order, each ray class partial zeta functions is equal to some Galois-theoretic partial zeta function times a factor of the form $\Nm(\mathfrak{d})^{-s}$, by \cite[Prop.~6.2 and Thm.~6.7]{Kopp2020d}. See \cite[Sec.\ 6]{Kopp2020d} for further results and discussion.
 
Often considered more fundamental that partial zeta functions are finite-order Hecke $L$-functions (associated to characters of a ray class group) and Artin $L$-functions (associated to characters, or more generally representations, of a Galois group). These $L$-functions have Euler products and are expected to satisfy the Riemann hypothesis. For abelian Galois extensions, Hecke and Artin $L$-functions are equal up to a finite number of Euler factors. One can state the Stark conjectures in terms of Hecke or Artin $L$-functions \cite{Tate:1981}, but the formulas are more complicated. We stick to partial zeta functions here, as they are most closely linked to the Stark units.

\subsection{The Stark conjectures}\label{ssc:StarkConjectures}

We will need a special case of Tate's refinement \cite{Tate:1981} of Stark's order 1 abelian $L$-values conjectures \cite{Stark1, Stark2, Stark3, Starkrealquad, Stark4}. We first state Tate's refinement in general. The following statement is part (II)(a) of \cite[Conj.~4.2]{Tate:1981} and is equivalent to the full statement of that conjecture. Tate notates this conjecture ${\mathrm St}(S,K/k)$, with his $k$ taking the role of our $K$, and his $K$ taking the role of our $H$.
\begin{conj}[Stark--Tate Conjecture ${\rm ST}(H/K,S)$, general case]\label{conj:StarkTateFull}
Let $H/K$ be an abelian extension of number fields, and let $W$ be the number of roots of unity in $H$. Let $S$ be a finite set of places of $K$ containing all the places that ramify in $H$ as well as all the infinite places of $K$, satisfying $\abs{S}\geq 2$. Suppose that $S$ contains a place $\pp$ (finite or infinite) that splits completely in $K$, and let $T = S \setminus \{\pp\}$. Let $U_{S,H}^T$ denote the set of elements $\alpha \in H^\times$ such that its $\QQ$-adic valuations at places $\QQ$ of $H$ satisfy
\begin{align}
& \abs{\alpha}_{\QQ} = 1 \mbox{ for } \QQ | \qq \nin S, &&  \\
& \abs{\alpha}_{\QQ} = 1 \mbox{ for } \QQ | \qq \in T, && \mbox{ if } \abs{T} \geq 2, \mbox{ and}  \\
& \abs{\alpha}_{\QQ} = a \mbox{ for } \QQ | \qq \mbox{ and $a$ constant}, && \mbox{ if } T=\{\qq\}.
\end{align}
Then, there is an element $\su \in U_{S,H}^T$ such that
\begin{equation}
\log\abs{g(\su)}_{\mathfrak{P}} = -W\zeta_S'(g,0) \mbox{ for each } g \in \Gal(H/K) \mbox{ and } \mathfrak{P}|\pp
\end{equation}
and such that $H(\su^{1/W})$ is abelian over $K$.
\end{conj}
We now state specialized consequences of the above conjecture in the case of interest to this paper. From now on, the field $K$ will be real quadratic, and it will be considered to be a subfield of $\R$. The two real embeddings are $\sigma_1(x) = x$ and $\sigma_2(x) = x'$, where $x'$ is the nontrivial Galois conjugate of $x$. We will also specialize the extension $H$ to be a ray class field and state the conjecture in terms of ray class partial zeta functions. We state two versions, with the first being provable from conjectures in Stark's 1976 paper \cite[Conj.~1 and Conj.~2]{Stark3}, and the second containing the condition on the square root appearing in Tate's work \cite{Tate:1981}.
\begin{conj}[Stark Conjecture $\Stark(K,\mm)$, real quadratic Archimedean ray class field case]\label{conj:stark}
Let $K \subset \R$ be a real quadratic number field embedded in $\R$, and let $\mm$ be a nonzero integral $\OO_K$-ideal such that $\mm \neq \OO_K$.
Let $H = H_{\mm\infty_2}^{\OO_K} \subset \R$.
Then, for all $\A \in \Cl_{\mm\infty_2}(\OO_K)$, there are elements $\su_{\A} \in \OO_H^\times$ such that 
\begin{equation}
\su_{\A} = \exp\!\left(-2\zeta_{\mm\infty_2}'(0,\A)\right),
\label{eq:stc}
\end{equation}
$(\Art(\B))(\su_{\A}) = \su_{\A\B}$ for $\B \in \Cl_{\mm\infty_2}(\OO_K)$, $\abs{g(\su_{\A})}=1$ for any $g \in \Gal(H/\Q) \setminus \Gal(H/K)$.
\end{conj}
\begin{conj}[Stark--Tate Conjecture $\ST(K,\mm)$, real quadratic Archimedean ray class field case]\label{conj:stc}
Let $K \subset \R$ be a real quadratic number field embedded in $\R$, and let $\mm$ be a nonzero integral $\OO_K$-ideal such that $\mm \neq \OO_K$.
Then, $\Stark(K,\mm)$ holds, and for all $\A \in \Cl_{\mm\infty_2}(\OO_K)$, the field $H(\su_\A^{1/2})$ is abelian over $K$.
\end{conj}
We also state a Stark-type conjecture for differenced ray class partial zeta functions attached to potentially imprimitive ray classes. This ``Monoid Stark Conjecture'' is not known to follow completely from the Stark (or Stark--Tate) conjectures, but it does follow in the case of the maximal order $\OO = \OO_K$.
\begin{conj}[Monoid Stark Conjecture $\MS(\OO,\mm)$]\label{conj:msc}
Let $K \subset \R$ be a real quadratic field, $\OO$ an order in $K$, and $\mm$ a nonzero $\OO$-ideal such that $\mm \neq \OO$.
Let $\{\infty_1, \infty_2\}$ be the two real places of $K$.
Let $H = H_{\mm\infty_2}^{\OO} \subset \R$.
Then, for all $\A \in \Clt_{\mm\infty_2}(\OO)$, 
there are elements $\su_{\A} \in \OO_H^\times$ such that 
\begin{equation}
\su_{\A} = \exp\!\left(-Z_{\mm\infty_2}'(0,\A)\right),
\label{eq:msc}
\end{equation}
$(\Art(\B))(\su_{\A}) = \su_{\A\B}$ for $\B \in \Cl_{\mm\infty_2}(\OO)$, $\abs{g(\su_{\A})}=1$ for any $g \in \Gal(H/\Q) \setminus \Gal(H/K)$, and $H(\su_\A^{1/2})$ is abelian over $K$.
\end{conj}

\begin{prop}\label{prop:starkimp}
    We describe some nontrivial conditional implications between these Stark-type conjectures. Let $K \subset \R$ be a real quadratic field. Let $H = H_{\mm,\rS}^{\OO_K}$ for a modulus $(\mm,\rS)$ for $K$. Let $S = \{\pp \text{ finite prime of $\OO_K$} : \pp|\mm\} \cup \{\infty_1,\infty_2\}$.
    \begin{itemize}
        \item[(1)] $\ST(K,\mm)$ is equivalent to $\ST(H/K,S)$.
        \item[(2)] $\MS(\OO_K,\mm)$ is equivalent to $(\forall \mm'|\mm)\ST(S,\mm')$.
    \end{itemize}
\end{prop}
\begin{proof}
    Claim (1) is shown in \cite[Prop.~6.10]{Kopp2020d}.
    Claim (2) may be seen to follow from the proof of \cite[Prop.~6.11]{Kopp2020d}.
\end{proof}

The units $g(\su)$ in \Cref{conj:StarkTateFull} and the units $\su_\A$ in \Cref{conj:stark}, \Cref{conj:stc}, and \Cref{conj:msc} (at least for $\OO=\OO_K$) are generally called ``Stark units'' and are equal when the conjectures align, except in some trivial cases. Stark's original formulation of his conjecture in the rank 1 totally real case involved differenced ray class partial zeta functions (albeit only for primitive ray classes of the maximal order), denoted as $\zeta(s,\ccc)$ in \cite{Stark3}, whereas most modern references follow Tate and state Stark's conjectures using Galois groups. We call the units $\su_\A$ \textit{Stark units} (for $\OO=\OO_K$) or \textit{generalized Stark units} (for non-maximal orders), without further comment.

\subsection{Eta-multipliers and theta-multipliers}

The relation between zeta functions and the \SFKFull{} modular cocycle involves a nontrivial root of unity factor that is best described as a value of a character $\psi^{-2}\chi_\r^{-1}$ at an element of congruence subgroup of $\SLtwo{\Z}$, with the characters $\psi$ and $\chi$ arising from multipliers of half-integral weight modular forms. We describe these characters here in terms of their relationships to modular forms.

Half-integral weight modular forms are best understood as modular forms for the \textit{metaplectic group}, which is a double cover of $\SL_2$. The real metaplectic group is defined to be
\begin{equation}
    \Mp_2(\R) = \{(M,\ep) : M \in \SL_2(\R), \ep \mbox{ a continuous function on $\HH$ with } \ep(\tau)^2=j_M(\tau)\},
\end{equation}
having multiplication $(M_1,\ep_1)(M_2,\ep_2) = (M_1M_2,\ep_3)$ with $\ep_3(\tau) = \ep_1(M_2\cdot\tau)\ep_2(\tau)$. The integer metaplectic group is defined to be $\Mp_2(\Z) = \{(M,\ep) \in \Mp_2(\R) : M \in \SL_2(\Z)\}$.

The \textit{Dedekind eta function} is the function 
\begin{equation}
\label{eq:DedekindEta}
    \eta(\tau) = \prod_{k=1}^\infty (1-e^{2\pi i k\tau})
\end{equation}
defined for $\tau \in \HH$. For $M \in \SLtwo{\Z}$, it transforms under the fractional linear transformation $\tau \mapsto M\cdot\tau$ according to the equation
\begin{equation}\label{eq:etatrans}
    \eta(M\cdot\tau) = \psi(M,\ep)\ep(\tau)\eta(\tau),
\end{equation}
where $(M,\ep) \in \Mp_2(\Z)$.

An explicit formula for $\psi$ is given by \cite[Thm.~2.4]{Kopp2020d}. Another explicit formula, in terms of the Rademacher function, is given as \Cref{prop:rademacher}.

The \textit{Jacobi theta function with characteristics} $\r = \smcoltwo{r_1}{r_2} \in \Q^2$ is
\begin{equation}\label{eq:jacobitheta}
\theta_\r(\tau) 
= \sum_{n=-\infty}^\infty e^{2\pi i\left(\foh \left(n+r_2+\foh\right)^2\tau + \left(n+r_2+\foh\right)\left(-r_1+\foh\right)\right)}.
\end{equation}
Under the fractional linear transformation action of $(M,\ep) \in \Mp_2(\Z)$ such that $M \in \Gamma_\r$, this theta function transforms by
\begin{equation}\label{eq:thetatrans}
\theta_\r(M\cdot\tau) = \psi\!\left(M,\ep\right)^3\chi_\r(M)\ep(\tau)\theta_\r(\tau).
\end{equation}
The character $\chi_\r : \Gamma_\r \to \C^\times$ is given by the formula
\begin{align}\label{eq:chidefn}
\chi_\r(M) &:= (-1)^{1+\delta_{\!M\r,\r}^{(2)}}e^{-\pi i\symp{M\r}{\r}},
\end{align}
where $\delta_{\!M\r,\r}^{(2)}$ is defined by \Cref{dfn:modulardeltafunction}.
For proofs of \eqref{eq:thetatrans} and \eqref{eq:thetatrans}, see \cite[Thm.~2.14 and Lem.~2.15]{Kopp2020d}.

In the sequel, we will often want to specify a standard choice of square root of $j_M(\tau)$ rather than using the metaplectic group. Define the choice of logarithm $(\log j_{M})(\tau) := \log(j_M(\tau))$ according to the principal branch of the logarithm, with $\log(1)=0$ and a branch cut along the negative real axis, along with the additional values $(\log j_{M})(\tau) = \log{\abs{j_M(\tau)}} + \pi i$ when $j_{M}(\tau)$ is on the negative real axis. Define the principal branch of the square root by $\sqrt{j_{M}(\tau)} := \exp\!\left(\foh (\log j_{M})(\tau)\right)$.

This characters $\psi$ and $\chi_\r$ are closely related to the SF phase $\SFPhase{\p}{t}$, as defined in \Cref{dfn:SFKPhase}. The exact relationship is proven in \Cref{ssec:rademacher} and \Cref{sbsc:phaseprops}.

\subsection{The functional equations of the Shintani--Faddeev modular cocycle}
\label{ssc:ShintaniFaddeevFunctionalEquations}%
We now present several key identities satisfied by the Shintani--Faddeev modular cocycle $\sfc{\r}{M}{\tau}$, as defined in \Cref{def:shin}. Most of these results are proven in \cite{Kopp2020d}.

The function $\sfc{\r}{M}{\tau}$ satisfies a particular symmetry under the involution $\r \mapsto -\r$. This symmetry derives from the modular properties of $\theta_\r(\tau)$ and $\eta(\tau)$ together with the Jacobi triple product identity.
\begin{thm}\label{thm:funchar}
Let $\r \in \Q^2$, $\g \in \Gamma_\r$, and $\tau \in \DD_{\!\g}$.
We have the identity
\begin{align}
\sfc{\r}{\g}{\tau}\sfc{-\r}{\g}{\tau} 
&= \psi^2(M)\chi_\r(M)\,e^{2\pi i\left(\tfrac{r_2^2}{2}+\tfrac{1}{12}\right)(\tau - M\cdot\tau)} \cdot 
\frac{e^{\pi i\left(r_2(M\cdot\tau)-r_1\right)}-e^{\pi i\left(-r_2(M\cdot\tau)+r_1\right)}}{e^{\pi i\left(r_2\tau-r_1\right)}-e^{\pi i\left(-r_2\tau+r_1\right)}}.
\label{eq:sfmcidentity}
\end{align}
\end{thm}
\begin{proof}
    See \cite[Thm.\ 4.32]{Kopp2020d}.
\end{proof}

An important special case is when $\tau$ is a fixed point of $M$ under the fractional linear transformation, in which case the above transformation identity reduces to the following. 
\begin{cor}\label{cor:funchar}
When $\tau = \qrt$ satisfying $M\cdot\qrt = \qrt$, \eqref{eq:sfmcidentity} reduces to
\begin{align}
\sfc{\r}{\g}{\qrt}\sfc{-\r}{\g}{\qrt} 
&= \psi^2(M)\chi_\r(M).
\end{align}
\end{cor}

\begin{lem}\label{lm:sfam1sfaeq1}
    For all $\mbf{r}\in\mbb{Q}^2$, $M\in \Gamma_\r$, and $\tau\in \DD_{M^{-1}}$,
    \eag{\label{eq:cocyclerelinv}
    \sfc{\mbf{r}}{M^{-1}}{\tau} 
    \sfc{\mbf{r}}{M\vpu{-1}}{M^{-1}\tau} &=1.
    }
\end{lem}
\begin{proof}
    Write the identity matrix $I = M^{-1}M$. The cocycle relation implies that
    \begin{equation}
        \sfc{\r}{I}{\tau} = \sfc{\r}{M^{-1}}{\tau} \sfc{\r}{M}{M^{-1}\cdot\tau}.
    \end{equation}
    Moreover, $\sfc{\r}{I}{\tau} = 1$, proving \eqref{eq:cocyclerelinv}.
\end{proof}

We also give some further properties and identities that are useful. Recall that the function $j_M$ was defined by $j_M(\tau) = \mc\tau+\md$ for $M = \smmattwo{\ma}{\mb}{\mc}{\md}$.

\begin{lem}\label{lm:shinperiodicity}
Let $\r \in \Q^2 \setminus \Z^2$, $M \in \Gamma_\r$, and $\qrt \in \DD_M$ either of the fixed points of $M$.
For all $\mbf{s}\in \mbb{Z}^2$,
\eag{
\sfc{\mbf{r}+\mbf{s}}{M}{\qrt} = \sfc{\mbf{r}}{M}{\qrt}.
}
\end{lem}
\begin{proof}
    See \cite[Prop.~4.35]{Kopp2020d}.
\end{proof}

\begin{lem}\label{lm:shinatzero}
    Let $\qrt$ be either of the fixed points of $M\in \Gamma(d)$.
    Then%
\eag{
\sfc{\r}{M}{\qrt} = 
\begin{cases}
\psi(M,\sqrt{j_M})\sqrt{j_M(\qrt)} & \text{ if  $r_2>0$,} \\
\frac{\psi(M,\sqrt{j_M})}{\sqrt{j_M(\qrt)}} & \text{ if  $r_2 \leq 0$,}
\end{cases}
}
with $\sqrt{j_M}$ denoting the standard branch,
for all $\mbf{r} = \smcoltwo{r_1}{r_2} \in \mbb{Z}^2$.
\end{lem}
\begin{proof}
    See \cite[Thm.~4.38]{Kopp2020d}.
\end{proof}

Finally, we give some elementary properties of the function $j_M$ and of the domains $\DD_M$ that we will want to use frequently.
\begin{lem}\label{lm:jMproperties}
    For all $M,N\in \GLtwo{\mbb{Z}}$ and $\tau\in \mbb{C}$,
    \eag{
     j_{MN}(\tau)&= j_M(N\cdot \tau)j_{N}(\tau).
     \label{eq:jMMPrimeCocycle}
    }
    For all $M\in \GLtwo{\mbb{Z}}$ and $\tau\in \mbb{C}$,
    \eag{
     j_M(M^{-1}\cdot \tau) &=\frac{1}{j_{M^{-1}}(\tau)}.
    }
\end{lem}
\begin{proof}
    Straightfoward consequences of the definition.
\end{proof}
\begin{lem}\label{lm:domainproperties}
    For all $M=\smt{\ma & \mb \\ \mc & \md }\in \GLtwo{\mbb{Z}}$
    \eag{
    \DD_{M^{-1}} &= M\cdot \DD_M,
    \\
     \DD_{JM} &= -\DD_{MJ}=\DD_{M}
    \\
    \DD_M\cup\DD_{-M} &= \begin{cases}
        \mbb{C} \qquad & \mc = 0,
        \\
        \mbb{C}\setminus \{-\md/\mc\} \qquad & \mc \neq 0,
    \end{cases}
    \label{eq:dludml}
    \\
    \DD_M\cap\DD_{-M} &= \mbb{C}\setminus \mbb{R}.
    }
    where 
    \eag{\label{eq:JDef}
    J=\smt{1 & 0 \\ 0 & -1}
    }
\end{lem}
\begin{proof}
Straightforward consequences of the definition.
\end{proof}

\begin{lem}\label{lem:fixedinda}
    Let $\qrt\in\mbb{R}$ be a fixed point of $M\in \SLtwo{\mbb{Z}}$.  Then $\qrt\in \DD_M$ if and only if $\Tr(M)>0$.
\end{lem}
\begin{proof}
    Write $M = \smmattwo{\alpha}{\beta}{\gamma}{\delta}$. We have $M \smcoltwo{\qrt}{1} = (\gamma \qrt + \delta)\smcoltwo{M\cdot\qrt}{1} = (\gamma \qrt + \delta)\smcoltwo{\qrt}{1}$ because $\qrt$ is a fixed point of $M$. Thus, $\gamma \qrt + \delta$ is an eigenvalue of $M$. Hence,
    \eag{
        \qrt \in \DD_M \iff \gamma\qrt+\delta>0
        \iff M \mbox{ has a positive eigenvalue}
        \iff \Tr(M)>0,
    }
    because $\det{M}=1$, so the two eigenvalues must have the same sign.
\end{proof}

\subsection{The relation of the Shintani--Faddeev modular cocycle to Stark units}\label{ssec:qpochmain}

We now present the main theorem of \cite{Kopp2020d}. It expresses generalized Stark units, that is, $\su_{\sR\A} = \su_\A^{-1} = \exp(Z_{\mm\infty_2}'(0,\A))$ for ray classes $\A$ in a \rcmia, in terms of special values $\sfc{\r}{A}{\qrt}^2$ of the Shintani--Faddeev modular cocycle. The special values of interest are \textit{real multiplication (RM) values}, that is, they occur at real quadratic $\rho$ such that $A\cdot\rho = \rho$.

\begin{thm}\label{thm:qpochmain}
Let $\OO$ be an order in a real quadratic field $F$, and let $\mm$ be a nonzero $\OO$-ideal.
Let $\A \in \Clt_{\mm\infty_2}(\OO) \setminus \ZClt_{\mm\infty_2}(\OO)$, let $\A_0$ be the class of $\A$ in $\Cl(\OO)$, choose some $\bb \in \A_0^{-1}$ coprime to $\mm$, and write $\bb\mm = \alpha(\qrt\Z+\Z)$ for some $\alpha,\qrt \in K$ such that $\alpha$ is totally positive and $\qrt > \qrt'$. Choose $\r = \smcoltwo{r_1}{r_2} \in \Q^2$ such that $(\alpha(r_2\qrt-r_1))\bb^{-1} \in \A$ and $r_2\qrt'-r_1 > 0$.
Write
\begin{equation}
\{B \in \Gamma_\r : B \cdot \qrt = \qrt\} = \langle A \rangle \mbox{ or } \langle -I, A \rangle
\end{equation}
such that $A \smcoltwo{\lambda}{1} = \lambda \smcoltwo{\qrt}{1}$ for $\lambda > 1$.
Let $n=\frac{2}{\abs{\phi^{-1}(\A)}}$, where $\phi : \Clt_{\mm\infty_1\infty_2}(\OO) \to \Clt_{\mm\infty_2}(\OO)$ is the natural quotient map.
Then
\begin{align}\label{eq:main}
\exp\!\left(n Z_{\mm\infty_2}'(0,\A)\right)
& = (\psi^{-2}\chi_\r^{-1})(A) \ \sfc{\r}{A}{\qrt}^2.
\end{align}
\end{thm}
\begin{proof}
    See \cite[Thm.\ 1.1]{Kopp2020d}.
\end{proof}

One may ask whether all real multiplication values of the Shintani--Faddeev cocycle are captured by \eqref{eq:main}. 
This is answered in the affirmative by \cite[Thm.~3.14]{Kopp2020d}, which proves certain properties of a function
\begin{equation}
    \Clt_{\mm\infty_2}(\OO) \xrightarrow{\Upsilon_\mm} \GL_2(\Z)\backslash(\Q^2/\Z^2 \times K_{\rm quad}),
\end{equation}
sending a ray class $\A$ to $\Upsilon_\mm(\A) = \GL_2(\Z) \cdot (\r + \Z^2,\qrt)$ with $\r$ and $\qrt$ chosen in the manner described in \Cref{thm:qpochmain}. 
Here, $K_{\rm quad} = K \setminus \Q$, and the notation $\GL_2(\Z)\backslash(\Q^2/\Z^2 \times K_{\rm quad})$ denotes the set of orbits by a certain left action of $\GL_2(\Z)$; namely, $M \cdot (\r,\qrt) = (\sgn(j_M(\qrt))M\r, M\cdot\qrt)$. In particular, it is proven that every orbit on the right-hand side is in the image of $\Upsilon_\mm$ for some choice of $\mm$. In the case $\mm = d\OO$ for $d \in \N$, it is shown that
\begin{equation}
    \image(\Upsilon_\mm) = \GL_2(\Z)\backslash\!\left(\tfrac{1}{d}\Z^2/\Z^2 \times K_{\OO}\right),
\end{equation}
where $K_\OO = \left\{\rho \in K \, : \, \lambda\rho\Z+\lambda\Z \subseteq \rho\Z+\Z \iff \lambda \in \OO\right\}$.

\subsection{Conditional results on algebraicity of real multiplication values}\label{ssec:conjcond}

We now prove several results on the implication of the Stark conjectures for real multiplication values of the Shintani--Faddeev cocycle. These results are refined versions of \cite[Thm.~1.3]{Kopp2020d} allowing for additional control on the conjectural assumptions and giving some additional conclusions needed in this paper. We will conclude with a proof of \Cref{thm:starkimplications}.

We first examine the implications of our weakest Stark-type conjecture. This proof and the next are related to the proof of \cite[Thm.~1.3]{Kopp2020d} in \cite[Sec.~8.2]{Kopp2020d}.
\begin{thm}\label{thm:field0}
Assume \Cref{conj:stark} (the Stark Conjecture). 
Let $\qrt \in \R$ such that $a\qrt^2+b\qrt+c=0$ with $a,b,c \in \Z$, $\Delta := b^2-4ac$ not a square, and let $\r \in \Q^2 \setminus \Z^2$. Let $A \in \Gamma_\r$ such that $A \cdot \qrt = \qrt$.
\begin{itemize}
    \item[(1)] There exists some $n \in \N$ such that $\sfc{\r}{\!A}{\qrt}^n$ is an algebraic unit in an abelian extension of $K=\Q(\qrt)$.
    \item[(2)] If $g \in \Gal(\ol{\Q}/\Q)$ such that $g(\sqrt{\Delta}) = -\sqrt{\Delta}$, then $\abs{g(\sfc{\r}{\!A}{\qrt})} = 1$.
\end{itemize}
\end{thm}
\begin{proof}
    Let $f$ be the conductor of $\beta$ (that is, $b^2-4ac = f^2\Delta_0$ for a fundamental discriminant $\Delta_0$ and a positive integer $f$).
    By \cite[Lem.~4.42]{Kopp2020d}, there is some $2 \times 2$ integral matrix $B$ of determinant $f$ and some real quadratic number $\alpha$ of conductor $1$ such that $\beta = B\cdot\alpha$. Choose $n \in \N$ so that 
    \begin{equation}
        C := B^{-1}A^nB \in \bigcap_{\substack{\mbf{s} \in \Q^2/\Z^2 \\ B\mbf{s} - \r \in \Z^2}} \Gamma_{\mbf{s}}.
    \end{equation}
    Then, by \cite[Thm.~4.46]{Kopp2020d}, we have
    \begin{equation}\label{eq:shinC}
        \sfc{\r}{A}{\beta}^n
        = \sfc{\r}{A^n}{\beta} 
        = \sfc{\r}{B C B^{-1}}{B \cdot \alpha}
        = \prod_{\substack{\mbf{s} \in \Q^2/\Z^2 \\ B\mbf{s} - \r \in \Z^2}} \sfc{\mbf{s}}{C}{\alpha}.
    \end{equation}
    Note that, if (1) and (2) hold for the factors in the product \eqref{eq:shinC}, then they hold for $\sfc{\r}{A}{\beta}^n$. For (1), any product of algebraic units in abelian extensions of $K$ is an algebraic unit in an abelian extension of $K$ (namely, the compositum of the fields generated by the factors over $K$). For (2), we would obtain
    \begin{equation}
        \abs{\sfc{\r}{A}{\beta}}^n
        = \abs{\sfc{\r}{A}{\beta}^n}
        = \prod_{\substack{\mbf{s} \in \Q^2/\Z^2 \\ B\mbf{s} - \r \in \Z^2}} \abs{\sfc{\mbf{s}}{C}{\alpha}}
        = 1,
    \end{equation}
    and thus, since the absolute value function returns a nonnegative real number, $\abs{\sfc{\r}{A}{\beta}}=1$.
    It thus suffices to prove the theorem when $f=1$, which we henceforth assume.

    As in the statement of \Cref{thm:qpochmain}, write
    \begin{equation}
        \{B \in \Gamma_\r : B \cdot \beta = \beta\} = \langle A_0 \rangle
    \end{equation}
    such that $A_0 \smcoltwo{\lambda}{1} = \lambda \smcoltwo{\beta}{1}$ for $\lambda > 1$. We have $A = A_0^k$ for some $k \in \Z$. 
    Let $\mm$ be the largest $\OO_K$-ideal such that $(\r,\qrt) \in \M_{\OO,\mm}$ in the notation of \cite[Thm.~3.12]{Kopp2020d}. Since $\mm$ in $\OO_K$-invertible (because $\OO_K$ is the maximal order), there is some $\A \in \Clt_{\mm\infty_2}(\OO)$ such that $\Upsilon_\mm(\A) = (\r,\qrt)$ in the notation of \cite[Thm.~3.12]{Kopp2020d} (as described at the end of \Cref{ssec:qpochmain}), and moreover $\A \in \Cl_{\mm\infty_2}(\OO)$ (because otherwise $\mm$ would not be the largest such $\OO_K$-ideal).
    By \Cref{thm:qpochmain}, for some $n \in \{1,2\}$,
    \begin{align}
        \exp\!\left(n Z_{\mm\infty_2}'(0,\A)\right)
        & = (\psi^{-2}\chi_\r^{-1})(A_0) \ \sfc{\r}{A_0}{\qrt}^2.
    \end{align}
    By the cocycle property, $\sfc{\r}{\!A_0^{t+1}}{\qrt} = \sfc{\r}{\!A_0^{t}}{A_0\cdot\qrt}\sfc{\r}{\!A_0}{\qrt} = \sfc{\r}{\!A_0^{t}}{\qrt}\sfc{\r}{\!A_0}{\qrt}$ for any $t \in \Z$, so by induction $\sfc{\r}{\!A_0^{k}}{\qrt} = \sfc{\r}{\!A_0}{\qrt}^k$.
    Also using the fact that $\psi^2, \chi_\r$ are homomorphisms, we obtain
    \begin{align}
        \exp\!\left(kn Z_{\mm\infty_2}'(0,\A)\right)
        & = (\psi^{-2}\chi_\r^{-1})(A_0^k) \ \sfc{\r}{A_0^k}{\qrt}^2
        = (\psi^{-2}\chi_\r^{-1})(A) \ \sfc{\r}{A}{\qrt}^2.
    \end{align}
    Let $\su_\A = \exp\!\left(-Z_{\mm\infty_2}'(0,\A)\right)$, so
    \begin{equation}\label{eq:field0step}
        \sfc{\r}{A}{\qrt}^2 = (\psi^{2}\chi_\r)(A) \ \su_\A^{-kn}.
    \end{equation}
    We have $Z_{\mm\infty_2}'(0,\A) = \zeta_{\mm\infty_2}'(0,\A) - \zeta_{\mm\infty_2}'(0,\sR\A)$.
    By \cite[Prop.~5]{Tangedal:2007}, either $\sR$ is the identity class, in which case $Z_{\mm\infty_2}'(0,\A) = 0$, or $\zeta_{\mm\infty_2}'(0,\sR\A) = -\zeta_{\mm\infty_2}'(0,\sR\A)$, in which case $Z_{\mm\infty_2}'(0,\A) = 2\zeta_{\mm\infty_2}'(0,\A)$. In the former case, $\su_\A=1$, and in the latter case, $\su_\A$ is the Stark unit from \Cref{conj:stark}. In both cases, that conjecture implies that $\su_\A$ is an algebraic unit in an abelian extension of $K$, and thus so is $(\psi^{2}\chi_\r)(A) \ \su_\A^{-kn}$ (since $(\psi^{2}\chi_\r)(A)$ is a root of unity), proving (1). Additionally, in both cases, the conjecture implies that $\abs{g(\su_\A)} = 1$, and moreover, $g\!\left((\psi^{2}\chi_\r)(A)\right)$ must by a root of unity. Applying the Galois automorphism $g$ followed by the absolute value function to \eqref{eq:field0step} gives $\abs{\sfc{\r}{A}{\qrt}}^2 = 1$, and thus $\abs{g(\sfc{\r}{A}{\qrt})} = 1$.
\end{proof}

We now prove another conditional result with stronger assumptions. This time, we assume an appropriate case of the Monoid Stark Conjecture and include stronger conclusions.
\begin{thm}\label{thm:field2}
    Let $\qrt \in \R$ such that $a\qrt^2+b\qrt+c=0$ with $a,b,c \in \Z$, $\Delta := b^2-4ac$ not a square, and let $\r \in \Q^2 \setminus \Z^2$. 
    Let $\OO = \colonideal{\qrt\Z+\Z}{\qrt{\Z}+\Z} = \Z[\tfrac{-b+\sqrt{\Delta}}{2}]$.
    Let $A \in \Gamma_\r$ such that $A \cdot \qrt = \qrt$.
    Suppose that $(\r,\qrt) \in \M_{\OO,\mm}$ in the notation of \cite[Thm.~3.12]{Kopp2020d}.
    Assume $\MS(\OO,\mm)$ from \Cref{conj:msc}. Then:
    \begin{itemize}
        \item[(1)] $\sfc{\r}{\!A}{\qrt}$ is an algebraic unit in an abelian extension of $K = \Q(\qrt)$.
        \item[(2)] $(\psi^{-2}\chi_\r^{-1})(A)\,\sfc{\r}{\!A}{\qrt} \in \OO_H^\times$ for $H=H_{\mm\infty_2}^{\OO}$.
        \item[(3)] In particular, if $d \in \N$ and $\r \in \tfrac{1}{d}\Z$, then $(\psi^{-2}\chi_\r^{-1})(A)\,\sfc{\r}{\!A}{\qrt} \in \OO_H^\times$ for $H=H_{d\infty_2}^{\OO}$.
        \item[(4)] If $g \in \Gal(\ol{\Q}/\Q)$ such that $g(\sqrt{\Delta}) = -\sqrt{\Delta}$, then $\abs{g(\sfc{\r}{\!A}{\qrt})} = 1$.
    \end{itemize}
\end{thm}
\begin{proof}
    If $\r \in \foh\Z^2 \setminus \Z^2$, then $\sfc{\r}{\!A}{\qrt} = \pm 1$ unconditionally by \cite[Thm.~4.38]{Kopp2020d}.
    Henceforth, we assume $\r \nin \foh\Z^2$, so $-I \nin \Gamma_\r$.
    As in the statement of \Cref{thm:qpochmain}, write
    \begin{equation}
        \{B \in \Gamma_\r : B \cdot \beta = \beta\} = \langle A_0 \rangle
    \end{equation}
    such that $A_0 \smcoltwo{\lambda}{1} = \lambda \smcoltwo{\beta}{1}$ for $\lambda > 1$. We have $A = A_0^k$ for some $k \in \Z$. 
    By \Cref{thm:qpochmain}, for some $n \in \{1,2\}$,
    \begin{align}
        \exp\!\left(n Z_{\mm\infty_2}'(0,\A)\right)
        & = (\psi^{-2}\chi_\r^{-1})(A_0) \ \sfc{\r}{A_0}{\qrt}^2.
    \end{align}
    By the cocycle property, $\sfc{\r}{\!A_0^{t+1}}{\qrt} = \sfc{\r}{\!A_0^{t}}{A_0\cdot\qrt}\sfc{\r}{\!A_0}{\qrt} = \sfc{\r}{\!A_0^{t}}{\qrt}\sfc{\r}{\!A_0}{\qrt}$ for any $t \in \Z$, so by induction $\sfc{\r}{\!A_0^{k}}{\qrt} = \sfc{\r}{\!A_0}{\qrt}^k$.
    Also using the fact that $\psi^2, \chi_\r$ are homomorphisms, we obtain
    \begin{align}
        \exp\!\left(kn Z_{\mm\infty_2}'(0,\A)\right)
        & = (\psi^{-2}\chi_\r^{-1})(A_0^k) \ \sfc{\r}{A_0^k}{\qrt}^2
        = (\psi^{-2}\chi_\r^{-1})(A) \ \sfc{\r}{A}{\qrt}^2.
    \end{align}%
    Let $\su_\A = \exp\!\left(-Z_{\mm\infty_2}'(0,\A)\right)$, so $\su_\A^{-kn} = (\psi^{-2}\chi_\r^{-1})(A) \ \sfc{\r}{A}{\qrt}^2$.%
    By the conjecture $\MS(\OO,\mm)$, we have $\su_\A \in \OO_H^\times$ for $H=H_{\mm\infty_2}^{\OO}$; thus, $(\psi^{-2}\chi_\r^{-1})(A) \ \sfc{\r}{A}{\qrt}^2 \in \OO_H^\times$, giving (2). If $\r \in \frac{1}{d}\Z^2$, then $(\r,\qrt) \in \M_{\OO,d\OO}$ by \cite[Thm.~3.12]{Kopp2020d}, so (3) follows from (2).
    
    Conjecture $\MS(\OO,\mm)$ also says that $\su_\A^{1/2}$ is an algebraic unit in an abelian extension of $K = \Q(\beta)$, and $\sfc{\r}{\!A}{\qrt} = \pm \sqrt{(\psi^2\chi_\r)(A)}\,\su_\A^{-kn/2}$ (with the square root factor being a root of unity), so $\sfc{\r}{\!A}{\qrt}$ is an algebraic unit in an abelian extension of $K$, giving (1).
    
    Finally, $\MS(\OO,\mm)$ says that $\abs{g(\su_\A)} = 1$; since $g\!\left(\su_\A^{1/2}\right)^2 = g(\su_\A)$, it follows that $\abs{g\!\left(\su_\A^{1/2}\right)} = 1$. Since $g$ is a homomorphism,
    \begin{equation}
        g(\sfc{\r}{\!A}{\qrt})
        = \pm g\!\left(\sqrt{(\psi^2\chi_\r)(A)}\right)\,g\!\left(\su_\A^{1/2}\right)^{-kn},
    \end{equation}
    and $g\!\left(\sqrt{(\psi^2\chi_\r)(A)}\right)$ is a root of unity, so $\abs{g(\sfc{\r}{\!A}{\qrt})} = 1$, giving (4).
\end{proof}
\begin{proof}[Proof of \Cref{thm:starkimplications}] %
    \Cref{thm:starkimplications}(1) says that \Cref{conj:stark} implies \Cref{conj:mrmvc}. This follows from \Cref{thm:field0}.
    
    \Cref{thm:starkimplications}(3) says that \Cref{conj:msc} implies \Cref{conj:grmvc}. Assume \Cref{conj:msc}. Then \Cref{conj:grmvc}(1) follows from \Cref{thm:field2}(1), and \Cref{conj:grmvc}(2) follows from \Cref{thm:field2}(4).

    \Cref{thm:starkimplications}(2) says that \Cref{conj:stc} implies \Cref{conj:frmvc}. Assume \Cref{conj:stc}. By \Cref{prop:starkimp}, $\MS(\OO_K,\mm)$ holds for $K = \Q(\qrt)$ and $\mm$ an ideal of the maximal order $\OO_K$. Then \Cref{conj:grmvc}(1) follows from \Cref{thm:field2}(1), and \Cref{conj:grmvc}(2) follows from \Cref{thm:field2}(4).
\end{proof}

\section{Weyl--Heisenberg group, extended Clifford group, and SIC phenomenology}
\label{sec:WHgp}

The purpose of this section is, firstly, to  review some relevant background material concerning the Weyl--Heisenberg group, the extended Clifford  group, and $r$-SICs. 
For more details, see ~\cite{Appleby2005,Appleby:2009}. We then go on to prove Theorems~\ref{thm:rsicbsc} and~\ref{thm:sicfidcond} from the introduction.

\subsection{Weyl--Heisenberg group}

\begin{defn}
  The discrete symplectic form is
\eag{
\la \mbf{p},\mbf{q}\ra &= \p^\top\smmattwo{0}{1}{-1}{0}\q = p_2 q_1-p_1 q_2, & \mbf{p},\mbf{q}\in \mbb{Z}^2.
}
\end{defn}
The WH displacement operators (Definition~\ref{def:WHGroup})  satisfy
\eag{
D_{\mbf{p}}^{\dagger} &= D\vpu{\dagger}_{-\mbf{p}}, & \forall \mbf{p} &\in \mbb{Z}^2
\\
D_{\mbf{p}}D_{\mbf{q}} & = \rtu_d^{\la \mbf{p},\mbf{q}\ra} D_{\mbf{p}+\mbf{q}}, & \forall \mbf{p},\mbf{q} &\in \mbb{Z}^2.
}
The fact $(\rtu_d)^d = (-1)^{d+1}$ means
\eag{
D_{\mbf{p}+d\mbf{q}} &=(-1)^{(d+1)\la \mbf{p},\mbf{q}\ra} D_{\mbf{p}} .
\label{eq:Dpperiodicity}
}
for all $\mbf{p}$, $\mbf{q}$. 
So the displacement operators are $d$-periodic when $d$ is odd, but not when $d$ is even. 
It would be possible to define the displacement operators by $D_{\mbf{p}} = X^{p_1} Z^{p_2}$, so that they were $d$-periodic for all values of $d$. 
Defining them the way we do introduces major simplifications later on, at the cost of some additional complexity at the outset. 
To get an idea of the relative merits of the two definitions, see ~\cite{Bos:2019}.

One has
\eag{
\Tr\!\left(D\vpu{\dagger}_{\mbf{p}}D^{\dagger}_{\mbf{q}}\right) &= d(-1)^{\frac{d+1}{d}\la \mbf{p},\mbf{q}\ra} \delta^{(d)}_{\mbf{p},\mbf{q}}, & \delta^{(d)}_{\mbf{p},\mbf{q}} 
&=\begin{cases}
1 \qquad & \text{if $\mbf{p} \equiv \mbf{q}\Mod{d}$},
\\
0 \qquad & \text{otherwise}.
\end{cases}
}
It follows that the displacement operators are a basis for $\mcl{L}(H_d)$.  In particular, an arbitrary operator $W\in \mcl{L}(H_d)$ can be expanded in terms of the $D_{\mbf{p}}$ using
\eag{
M &= \frac{1}{d} \sum_{\mbf{p}} \Tr\!\left(WD^{\dagger}_{\mbf{p}}\right) D\vpu{\dagger}_{\mbf{p}}
\label{eq:mdopexp}
}
where the summation is over any transversal for the quotient group $\mbb{Z}^2/(d\mbb{Z}^2)$ (note that the product $\Tr\!\left(WD^{\dagger}_{\mbf{p}}\right) D\vpu{\dagger}_{\mbf{p}}$ is $d$-periodic, even though the two factors may not be).

\subsection{Clifford and extended Clifford groups}
\label{ssc:ecdgp}

\begin{defn}[Clifford Group; extended Clifford group; projective Clifford and extended Clifford groups]\label{dfn:CliffordGroup}
    The  Clifford  group in dimension $d$, denoted $\Cliff(d)$, is the set of all unitaries $U$ with the property
    \eag{
        U  D_{\mbf{p}} U^{\dagger} &= e^{i\varphi(\mbf{p})} D_{f(\mbf{p})}
    }
    for all $\mbf{p}$ and some pair of functions $\varphi\colon \mbb{Z}^2 \to \mbb{R}$, $f\colon \mbb{Z}^2\to \mbb{Z}^2$.

    The extended Clifford group in dimension $d$, denoted $\EC(d)$, is the set of all unitaries and anti-unitaries with this property.

    The projective  groups $\PC(d)$ and $\PEC(d)$ are the quotients of   $\Cliff(d)$ and $\EC(d)$ by their centres:
    \eag{
    \PC(d) & = \Cliff(d)/ \la I\ra, & \PEC(d)&= \EC(d)/\la I \ra.
    }
\end{defn}  
The importance of the groups $\Cliff(d)$, $\EC(d)$ for us is that they preserve $r$-SIC fiduciality:  if $\Pi$ is a $r$-SIC fiducial, then so is $U\Pi U^{\dagger}$, for all $U\in \EC(d)$. 
Since the replacement $U\to e^{i\theta} U$ does not change $U\Pi U^{\dagger}$ we only need consider one representative of each coset in $\PEC(d)$.

\begin{defn}[symplectic and extended symplectic groups; symplectic and anti-symplectic matrices]\label{dfn:symplectic}
    We refer to the group $\SLtwo{\mbb{Z}/\bar{d}\mbb{Z}}$ simply as ``the symplectic group.''
    Additionally, the extended symplectic group $\ESLtwo{\mbb{Z}/\bar{d}\mbb{Z}}$ is the set of all matrices in $\GLtwo{\mbb{Z}/\bar{d}\mbb{Z}}$ with determinant equal to $\pm 1$. An element of $\ESLtwo{\mbb{Z}/\bar{d}\mbb{Z}}$ is said to be \textit{symplectic} (respectively \textit{anti-symplectic}) if it has determinant equal to $+1$ (respectively $-1$). The \textit{canonical anti-symplectic} matrix is defined to be
\eag{
J &= \bmt 1 & 0 \\ 0 & -1\emt.
\label{eq:JMatrixDefinition}
}
(Note that the $n \times n$ symplectic group $\Sp_n(R)$ and the $n \times n$ special linear group $\SL_n(R)$ for a ring $R$ are not generally isomorphic, but for $n=2$, they are isomorphic and equal inside $\GL_2(R)$.)
\end{defn}
The significance of the canonical anti-symplectic matrix is that the map $F \mapsto JF$ converts symplectic matrices into anti-symplectic matrices, and conversely.

Trivially, $\Cliff(d)$ contains $\WH(d)$. 
Less trivially, it contains \cite{Appleby2005} a  representation of the symplectic group. 
Specifically, for each $F=\smt{\alpha & \beta \\ \gamma & \delta}\in \SLtwo{\mbb{Z}/\bar{d}\mbb{Z}}$ there exists a unitary $U_F\in\mcl{L}(H_d)$, unique up to multiplication by a number having absolute value equal to $1$, such that
\eag{\label{eq:symplecticUnitary}
U\vpu{\dagger}_FD\vpu{\dagger}_{\mbf{p}}U^{\dagger}_F
&= D\vpu{\dagger}_{F\mbf{p}}
}
for all  $\mbf{p}$. We refer to $U_F$ as a symplectic unitary. One has
\eag{
U\vpu{\dagger}_{F^{-1}} &\dot{=} U^{\dagger}_F & &\forall F \in \SLtwo{\mbb{Z}/\bar{d}\mbb{Z}},
\\
U_{F_1}U_{F_2}&\dot{=} U_{F_1F_2} && \forall F_1,F_2\in \SLtwo{\mbb{Z}/\bar{d}\mbb{Z}},
}
where the symbol $\dot{=}$ signifies ``equal up to multiplication by a number having absolute value equal to 1''.
So the map $F \mapsto U_F$ is a projective representation of $\SLtwo{\mbb{Z}/\bar{d}\mbb{Z}}$.
It can be shown\cite{Appleby2005} that $\Cliff(d)$ consists of all products of the form
\eag{
e^{i\lambda} D_{\mbf{p}} U_F 
\label{eq:eixidpufrepcd}
}
for  $\lambda \in \mbb{R}$, $\mbf{p}\in \mbb{Z}^2$, $F\in \SLtwo{\mbb{Z}/\bar{d}\mbb{Z}}$.

There exist\cite{Appleby2005,Appleby:2009} explicit formulae for the $U_F$. 
We say\cite{Appleby2005} that $F=\smt{\alpha & \beta \\ \gamma & \delta}\in \SLtwo{\mbb{Z}/\bar{d}\mbb{Z}}$ is a \textit{prime matrix} if $\beta$ is coprime to $\bar{d}$. 
It can be shown~\cite{Appleby2005} that every matrix in $\SLtwo{\mbb{Z}/\bar{d}\mbb{Z}}$ is a product of two prime matrices. 
It can also be shown~\cite{Appleby2005} that if $F$ is a prime matrix then
\eag{
U_F &= \frac{e^{i\theta}}{\sqrt{d}}\sum_{j,k=0}^{d-1}\rtu_d^{\beta^{-1}(\delta j^2 -2 jk+\alpha k^2)}|j\ra\la k|
\label{eq:ufppmexpn}
}
where $\beta^{-1}$ is the multiplicative inverse of $\beta$ as an element of $\mbb{Z}/\bar{d}\mbb{Z}$ and $e^{i\theta}$ is an arbitrary phase.  
\begin{defn}[Parity matrix]\label{dfn:parityMatrix}
    We define the parity matrix by 
    \eag{
    P &= \bmt -1 & 0 \\ 0 & -1 \emt.
    }
\end{defn}
Using the decomposition
\eag{
\bmt -1 & 0 \\ 0 & -1 \emt&= \bmt 0 & -1\\ -1 & 0\emt \bmt 0 & 1 \\ 1 & 0\emt
}
together with \eqref{eq:ufppmexpn} one finds, after a certain amount of algebra, that
\eag{\label{eq:ParityMatrixElements}
U_P &= \sum_{j=0}^{d-1}|-j\ra \la j|
}
in agreement with \Cref{def:pProjector}. 

If  we choose the arbitrary phase in \eqref{eq:ufppmexpn} according to
\eag{
e^{i\theta}
&=
\begin{cases}
    1 \qquad & \text{$d\equiv 1\Mod{4}$ }
    \\
    i \qquad & \text{$d\equiv 3\Mod{4}$}
    \\
    e^{\frac{\pi i}{4}} \qquad & \text{$d\equiv 0\Mod{2}$}
\end{cases}
\label{eq:SymplecticUnitaryPhaseChoice}
}
then~\cite{Appleby:2009,Appleby:2018} the components of $U_F$ are all in the cyclotomic field $\mbb{Q}(\rtu_d)$.  In the sequel we will always assume this choice has been made.  On this assumption we have the following description of  the action of a Galois automorphism.

We will need the fact
\begin{thm}\label{thm:symplecticKernel}
    The homomorphism $F\in \SLtwo{\mbb{Z}} \mapsto U_F \la I\ra \in \PC(d)$
    \begin{enumerate}
        \item is injective if $d$ is odd
        \item has kernel $\left< \smt{d+1 & 0 \\ 0 & d+1}\right>$ if $d$ is even.
    \end{enumerate}
\end{thm}
\begin{proof}
    See Theorem~1 of ~\cite{Appleby2005}.
\end{proof}
\begin{defn}\label{dfn:HgMatrixDefinition}
      Let $g$ be any Galois automorphism of $\mbb{Q}(\rtu_d)/\mbb{Q}$.  Define  $k_g$ to  be the unique integer in the range $0\le k_g < \db$ such that $g(\rtu_d)=\rtu_d^{k_g}$, and define $H_g\in \GLtwo{\mbb{Z}/\db\mbb{Z}}$ to be the matrix
  \eag{\label{eq:HgMatrixDefinition}
  H_g=\bmt 1 & 0 \\ 0 & k_g \emt.
  }
\end{defn}
\begin{theorem}\label{thm:GalActOnClifford}
  Let $g$ be any Galois automorphism of $\mbb{Q}(\rtu_d)/\mbb{Q}$.
  Then
  \eag{
g(D_{\mbf{p}})&= D_{H_g\vpu{-1}\mbf{p}}
\label{eq:GaloisActionDisplacement}
\\
g(U_F) &\; \dot{=} \; U_{H\vpu{-1}_gFH^{-1}_g}
\label{eq:GaloisActionSymplecticUnitary}
}
for all $\mbf{p}\in \mbb{Z}^2$, $F\in \SLtwo{\mbb{Z}/\db\mbb{Z}}$.
\end{theorem}
\begin{rmkb}
    If $g$ is complex conjugation, then $H_{g} = J$, where $J$ is the matrix defined in ~\eqref{eq:JMatrixDefinition}.
\end{rmkb}
\begin{proof}
   Theorem~2 in ~\cite{Appleby2005}, with the obvious modification to take account of ~\eqref{eq:SymplecticUnitaryPhaseChoice} above. 
\end{proof}

\begin{proof}[Proof of \Cref{lm:OverlapTermsGhostOverlap}]
    It follows from \Cref{dfn:CandidateGhostAndSICFiducials} and \Cref{thm:GalActOnClifford} that
    \eag{
    \Pi_s &= 
        \frac{1}{d}\sum_{\mbf{p}}g\left(\ghostOverlapC{t}{G\mbf{p}}\right)D_{H_g\mbf{p}}.
    }
    In view of ~\eqref{eq:CandidateOverlaps} this means
    \eag{
    \overlapC{s}{\mbf{p}}&= \Tr\left(\Pi\vpu{\dagger}_sD^{\dagger}_{G^{-1}\mbf{p}}\right)
=
    g\left(\ghostOverlapC{t}{GH^{-1}_g G^{-1}\mbf{p}}\right).
    }
\end{proof}

Just as symplectic matrices $F$ are associated to unitaries $U_F$ satisfying ~\eqref{eq:symplecticUnitary}, so anti-symplectic matrices $F$  are associated to anti-unitaries $U_F$ satisfying the  same equation.  Explicitly, such anti-unitaries act according to
\eag{\label{eq:antiSymplecticAntiUnitary}
U_F |\psi\ra &= \sum_{j=0}^{d-1} \la j |\psi\ra^{*} U_{FJ}|j\ra.
}
In particular  $U_J$ acts by complex conjugation in the standard basis:
\eag{
U_J |\psi\ra &= \sum_{j=0}^{d-1} \la j |\psi\ra^{*} |j\ra.
}
 One then finds that $\EC(d)$ consists of all products of the form~\eqref{eq:eixidpufrepcd}, where now $F$ is an arbitrary element of $\ESLtwo{\mbb{Z}/\bar{d}\mbb{Z}}$.

Finally, we define
\begin{defn}[(anti-)symplectic subgroup of $\EC(d)$]\label{df:antisymplecticSubgroup}We define the (anti-)symplectic subgroup of $\EC(d)$, denoted $\ECS(d)$ to consist of all unitaries and anti-unitaries of the form
\eag{
e^{i\theta}U_F
}
with $F\in \ESLtwo{\mbb{Z}/\db\mbb{Z}}$ and $e^{i\theta}$ a 
phase.
\end{defn}
\subsection{SIC phenomenology}\label{subsc:sicphenomenology}
Over the last 25 years, investigation of the large number of known $1$-SICs has resulted in a large number of empirical observations.  To reflect the fact that a lot of this material is unproven, we refer to it as \textit{SIC phenomenology}. One of the aims of this paper is to show that many of these observations are consequences of the Twisted Convolution Conjecture together with the Stark Conjecture and its refinements.  

The purpose of this subsection is to summarize this pre-existing body of empirical observations.
Since it is concerned with previous results, we only discuss $1$-SICs.

\subsubsection{Number of orbits}
\label{ssc:NumberOfOrbits}
The fact that the elements of $\EC(d)$ preserve $r$-SIC fiduciality suggests we group $r$-SICs into $\EC(d)$ orbits. The question then arises:  how many orbits are there in each dimension?  In the case of $1$-SICs this question has been answered by brute-force numerical calculation in many low-lying dimensions (see ~\cite{Scott2010,Scott:2017,Grassl:2021,Grassl:2021a} and references cited therein). One finds that for $d=3$ there are infinitely many orbits.  However, those are examples of  sporadic $1$-SICs~\cite{Stacey:2021}, which are in various ways exceptional.  In particular, $1$-SICs in dimension $3$  typically generate transcendental number fields.  If we confine ourselves to non-sporadic $1$-SICs (i.e. WH covariant $1$-SICs in dimensions greater than 3) then explicit calculation suggests that the number of orbits is always finite.  For the first $20$ dimensions greater than 3 the number of orbits is given in \Cref{tab:numberOfECdSICOrbitsDims4To20}.
\begin{table}[htb]
    \centering
\begin{tabular}{lllllllllllllllllllll}
$d$ & 4 & 5 & 6 & 7 & 8 & 9 & 10 & 11 & 12 & 13 & 14 & 15
& 16 & 17 & 18 & 19 & 20 & 21 & 22 & 23
\\
\midrule
\# orbits & 1 & 1 & 1 & 2  & 2 & 2 & 1 &  3 & 2 & 2 & 2 & 4 & 2 & 3 &2 & 5 & 2 & 5 & 1 & 6
\end{tabular}
    \caption{Number of $\EC(d)$ orbits of $1$-SICs in dimensions 4--20.}
    \label{tab:numberOfECdSICOrbitsDims4To20}
\end{table}
 A long-standing puzzle has been to understand where these numbers are coming from. As will be shown in \Cref{sec:sicphenexplain}, the construction we describe answers that question.

\subsubsection{Symmetry group}\label{ssc:symgp}
\begin{defn}[Symmetry group of a fiducial]\label{dfn:symmetryGroupOfFiducial}
    Let $\Pi$ be a $r$-SIC fiducial in dimension $d$.  Its symmetry group, denoted $\mcl{S}(\Pi)$, is the set of  cosets  $U\la I \ra \in \PEC(d)$ such that 
\eag{
U\Pi U^{ \dagger} &= \Pi.
}
\end{defn}
Brute-force numerical computation~\cite{Appleby2005,Scott2010,Scott:2017} suggests that in every case $\mcl{S}(\Pi)$ 
\begin{enumerate}
    \item is non-trivial cyclic,
    \item \label{it:canonicalOrder3} contains a coset 
    \eag{
        D_{\mbf{p}} U_F \la I\ra
    }
    for which $F\in\SLtwo{\mbb{Z}/\bar{d}\mbb{Z}}$ is such that $\Tr(F) \equiv -1 \Mod{d}$. 
    It can be shown that if $d>3$ then such cosets are necessarily order 3. 
\end{enumerate} 
\begin{thm}
    Let $F\in\SLtwo{\mbb{Z}/\bar{d}\mbb{Z}}$ be such that $\Tr(F) \equiv -1 \Mod{d}$. Assume $d>3$. 
    Then
    \begin{enumerate}
        \item \label{it:dpufIsOrder3} The coset $D_{\mbf{p}} U_F \la I\ra$ is order $3$ for all $\mbf{p}\in \left(\mbb{Z}/\db\mbb{Z}\right)^2$.
    \item If $d$ is odd then $F$ is order 3.
    \item If $d$ is even then $F$ is order 3 if $\Tr(F) \equiv -1 \Mod{\db}$ and order 6 if $\Tr(F) \equiv d-1 \Mod{\db}$.
    \end{enumerate}
\end{thm}
\begin{proof}
    The first statement is proved in ~\cite{Appleby2005}. 
    The other two are proved by repeated application of the identity 
    \eag{
    L^2 &= \Tr(L) L -I,
    }
    valid for any $L\in \SLtwo{\mbb{Z}/\bar{d}\mbb{Z}}$.
\end{proof}
Note that we can switch between the two cases $\Tr(F) \equiv -1 \Mod{\db}$ and $\Tr(F) \equiv d-1 \Mod{\db}$ by multiplying by $d+1$: 
If $F' = (d+1)F$, then $\Tr(F') \equiv -1 \Mod{\db}$ if and only if $\Tr(F) \equiv d-1 \Mod{\db}$. 
Note also~\cite{Appleby2005} that if $F'=(d+1)F$ then $D_{\mbf{p}}U_{F'} \la I \ra = D_{\mbf{p}}U_{F} \la I \ra$ for all $\mbf{p}$. 
So if one is only interested in the corresponding elements of $\PC(d)$  there is no loss of generality in focusing on just one of the cases. 
Accordingly, in the following, if $D_{\mbf{p}}U_{F} \la I \ra$ is a canonical order 3 unitary, it will always, unless the   contrary is explicitly stated, be assumed that $\Tr(F) \equiv d-1 \Mod{\db}$. 
We thus define
    \begin{defn}[canonical order 3 unitary]\label{df:canonicalOrder3}
        A coset $D_{\mbf{p}} U_F \la I\ra$ in dimension $d>3$ is said to be \emph{canonical order 3} if $\Det(F) = +1$ and $\Tr(F) \equiv d -1 \Mod{\db}$.
    \end{defn}    

It is convenient to define the following standard trace $d-1$ matrices:
\begin{defn}[$F_z$ (Zauner), $F\vpu{'}_a$, $F'_a$ matrices]
\label{dfn:FzFaFpaMatrices}
For all $d$ define the \emph{Zauner} matrix to be 
    \eag{
    F_z &= \bmt 0 & d-1 \\ d+1 & d-1\emt.
    \\
    \intertext{If $d \equiv 3 \Mod{9}$, also define the variant Zauner matrix}
    F_a &= \bmt 1 & d+3 \\ \frac{4d-3}{3} & d-2 \emt,
    \\
    \intertext{while if $d \equiv 6 \Mod{9}$, define the variant Zauner matrix}
    F'_a &= \bmt 1 & d+3 \\ \frac{2d-3}{3} & d-2 \emt.
    }
\end{defn}
It turns out that every determinant $+1$, trace $d-1$ element of $\ESLtwo{\mbb{Z}/\db\mbb{Z}}$ is conjugate to one of these three matrices.  Specifically:
\begin{thm}\label{thm:CanonicalOrder3ConjugacyClasses}
The set of matrices in $\ESLtwo{\mbb{Z}/\db \mbb{Z}}$ having determinant equal to $+1$ and trace equal to $d-1$ consists of
\begin{enumerate}
    \item The single conjugacy class $[F_z]$ if $d \not\equiv 3,6 \Mod{9}$,
    \item The two disjoint conjugacy classes $[F_z]$, $[F_a]$ if $d\equiv 3 \Mod{9}$,
    \item The two disjoint conjugacy classes $[F\vpu{'}_z]$, $[F'_a]$ if $d\equiv 6 \Mod{9}$,
\end{enumerate}
where the notation ``$[F]$'' means ``conjugacy class of $F$ considered as an element of $\ESLtwo{\mbb{Z}/\db\mbb{Z}}$''.
\end{thm}
\begin{rmkb}
    Although the  matrices of interest are all in $\SLtwo{\mbb{Z}/\db \mbb{Z}}$, we consider conjugacy relative to $\ESLtwo{\mbb{Z}/\db\mbb{Z}}$.  This is essential.  If instead we considered conjugacy relative to $\SLtwo{\mbb{Z}/\db \mbb{Z}}$ the conjugacy classes would be smaller, and the description significantly more complicated.
\end{rmkb}
\begin{proof}
    See \Cref{ap:canord3}.
\end{proof}
If $d \equiv 3 \Mod{9}$ (respectively $d \equiv 6 \Mod{9}$) we can use the following convenient criterion to tell if a given matrix is in $[F\vpu{'}_a]$ or $[F\vpu{'}_z]$ (respectively $[F'_a]$ or $[F\vpu{'}_z]$).
\begin{thm}\label{thm:CriterionForFaFaprime}
Let $F\in \ESLtwo{\mbb{Z}/\db\mbb{Z}}$ have determinant equal to $+1$ and trace equal to $d-1$.  
\begin{enumerate}
    \item If $d \equiv 3 \Mod{9}$ then $F\in [F_a]$ if and only if $F \equiv I \Mod{3}$.
    \item If $d \equiv 6 \Mod{9}$ then $F\in [F'_a]$ if and only if $F \equiv I \Mod{3}$.
\end{enumerate}
\end{thm}
\begin{rmkb}
    In particular $[F'_a] = [F\vpu{'}_z]$ if $d \equiv 3 \Mod{9}$, and $[F\vpu{'}_a] =  [F\vpu{'}_z]$ if $d \equiv 6 \Mod{9}$.
\end{rmkb}
\begin{proof}
Immediate consequence of the fact
\eag{
d&\equiv 3 \Mod{9} &&\implies & F_z &= \bmt 0 & 2 \\ 1 & 2\emt, \ F_a \equiv I \Mod{3},
\\
d&\equiv 6 \Mod{9} &&\implies & F_z &= \bmt 0 & 2 \\ 1 & 2\emt, \  F'_a \equiv I \Mod{3}.
}
\end{proof}
It is observed empirically that if $d \equiv 3 \Mod{9}$ (respectively $d \equiv 6 \Mod{9}$)  the symmetry group
$\mcl{S}(\Pi)$ never seems to contain two canonical order 3 unitaries $D_{\mbf{p}}U_F $, $D_{\mbf{q}}U_{F'} $ such that $F\in [F\vpu{'}_z]$ and $F'\in [F_a]$ (respectively $F\in [F\vpu{'}_z]$ and $F'\in [F'_a]$).  This suggests that we may classify fiducials by conjugacy class:
\begin{defn}[type-$z$, type-$a$, type-$a'$ fiducials]\label{df:typeztypeatypeaprime}
    A fiducial $\Pi$ is said to be 
\begin{enumerate}
    \item \emph{type-z}  if $\mcl{S}(\Pi)$ contains a canonical order 3 fiducial $D_{\mbf{p}} U_F$ with $F$ conjugate to $F_{\rm{z}}$,
    \item \emph{type- a}  if $d \equiv 3 \Mod{9}$ and $\mcl{S}(\Pi)$ contains a canonical order 3 fiducial $D_{\mbf{p}} U_F$ with $F$ conjugate to $F_{\rm{a}}$,
    \item \emph{type-$a'$}  if $d \equiv 6 \Mod{9}$ and $\mcl{S}(\Pi)$ contains a canonical order 3 fiducial $D_{\mbf{p}} U_F$ with $F$ conjugate to $F_{\rm{a}'}$.
\end{enumerate} 
\end{defn}
This is the source of four long-standing puzzles. 

\emph{Question 1.}  One would like to understand why $\mcl{S}(\Pi)$ always seems to contain a canonical order 3 unitary.

\emph{Question 2.} One would like to understand why classification by conjugacy class works:  why one seems never  to find orbits which are simultaneously type-$z$ and type-$a$ (if $d \equiv 3 \Mod{9}$), or simultaneously type-$z$ and type-$a'$ (if $d \equiv 6 \Mod{9}$).

\emph{Question 3.} It appears that when $d \equiv 3 \Mod{9}$ there exist both type-$z$ and type-$a$ orbits.  Specifically brute-force numerical computation~\cite{Scott:2017} indicates that in the first nine such dimensions the numbers of type-z and type-a $\EC(d)$ orbits are

\spc

\begin{center}
\begin{tabular}{rccccccccc}
$d:$ & 12 & 21 & 30 & 39 & 48 & 57 & 66 & 75 & 84
\\
\cmidrule{2-10}
\# type-z orbits:  & 1 & 4 & 3 & 6 & 5 & 6 & 6 &12 & 6
\\
\# type-a orbits: & 1 & 1 & 1 & 4& 2 & 2 & 3 & 3 & 4
\end{tabular}
\end{center}

\spc

\noindent One would like to know what determines these numbers. 

\emph{Question 4.} When $d \equiv 6 \Mod{9}$, brute-force numerical computation~\cite{Scott:2017} indicates that the number of type-$\rm{a}'$ orbits is  zero, at least when $d\le 87$.  One would like to know the reason.

As we will see, our conjectures provide answers to all of these questions.

\begin{defn}[centered fiducial]\label{dfn:centered}
    A fiducial $\Pi$ is said to be \textit{centered} if $\mcl{S}(\Pi)$  only contains cosets of the form $U_F \la I \ra$, $F\in \ESLtwo{\mbb{Z}/\bar{d}\mbb{Z}}$.
\end{defn}
Numerical investigations  suggest that every $1$-SIC contains at least one centered fiducial.

\begin{defn}[\overlapSingularText{} symmetry group]\label{dfn:sogppi}
Suppose $\Pi$ is centered, and let $\normalizedOverlap_{\mbf{p}}$ be its \normalizedOverlapsText{}. Define the \emph{symplectic symmetry group} by
\eag{
\StabPiESL{\Pi}&= \{F\in \ESLtwo{\mbb{Z}/\db\mbb{Z}}\colon U_F \in \mcl{S}(\Pi)\},
\\
\intertext{and the \emph{\overlapSingularText{} symmetry group} by}
\StabPiOverlap{\Pi} 
&= \left\{F\in \ESLtwo{\mbb{Z}/\bar{d}\mbb{Z}} \colon \normalizedOverlap_{F\mbf{p}}=\normalizedOverlap_{\mbf{p}}\ \forall \mbf{p}\in \mbb{Z}^2\right\}.
}
\end{defn}
It is easily seen\cite{Appleby:2013} that if $\Pi$ is centered then $\StabPiOverlap{\Pi}=\{\Det(F)F \colon F \in \StabPiESL{\Pi}\}$.   

Inspection of the $1$-SIC symmetry groups in low-lying dimensions~\cite{Scott:2017} raises other questions.  In the first place, one finds that in some cases, but not in others, the symmetry group contains anti-unitaries as well as unitaries.  One would like to know why.  One would also like to know, in dimensions where anti-unitary symmetries occur, on exactly how many $\EC(d)$ orbits this happens.  Furthermore, one would like to know what determines the order of the symmetry group.  
As we will see, the Twisted Convolution Conjecture and the Stark Conjecture confirm and explain all the conjectures about the symmetry group in ~\cite{Scott:2017}.  It also provides answers to the question marks in the Table in ~\cite{Scott:2017}.  

\subsubsection{Fields, multiplets, and ghosts}
\label{sssc:FieldsMultipletsGhosts}
The study~\cite{Appleby:2013,Appleby:2017a,Appleby:2018,Appleby:2020,Kopp2019,Appleby:2022,Bengtsson:2024} of known exact $1$-SICs has led to a number of conjectures regarding the field generated by a $1$-SIC, and the associated Galois group. 

Let $\Pi$ be a $1$-SIC fiducial in dimension $d>3$, and  let $E$ be the field\footnote{If our conjectures are correct, and if $\Pi$ is calculated from an admissible tuple $t$ in the manner prescribed by \Cref{dfn:CandidateGhostAndSICFiducials}, then $E$ is the SIC field $E_t$ specified by \Cref{dfn:SICfield}.} generated by the matrix elements of $\Pi$ together with the root of unity $\rtu_d$.  Then it was observed in ~\cite{Appleby:2013} that, in all the cases considered there, $E$ is a finite degree abelian extension of $K=\mbb{Q}(\sqrt{(d-3)(d+1)})$ which is normal over $\mbb{Q}$.  It was also observed that an important role is played by three subfields $H$, $E_1$, $E_2$ related to $E$ and $K$ as shown in Fig.~\ref{fg:ee1e2k}.
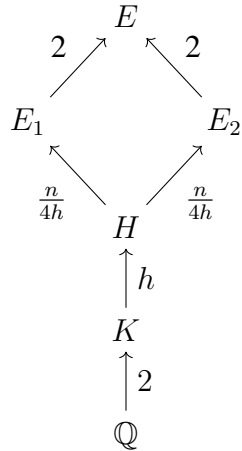
\begin{figure}[htb]
\begin{tikzpicture}
[auto, style={}]
\matrix[row sep=8mm,column sep=6mm] {
 &\node (E)  {$E$};
\\
\node (E1)  {$E_1$}; &  & \node (E2)  {$E_2$};
\\
 &\node (E0)  {$H$};
\\
 &\node (K)  {$K$};
\\
 &\node (Q)  {$\mathbb{Q}$};
\\
};
\draw[->] (Q) to node[swap] {$2$} (K);
\draw[->] (K) to node[swap] {$h$} (E0);
\draw[->] (E0) to node[swap] {$\frac{n}{4h}$} (E2);
\draw[->] (E0) to node{$\frac{n}{4h}$} (E1);
\draw[->] (E1) to node {2} (E);
\draw[->] (E2) to node[swap] {2} (E);
\end{tikzpicture}
\caption{\label{fg:ee1e2k}
Structure of the field $E$.  Arrows show field inclusions, and run from the smaller field to the larger. Numbers besides the arrows are the extension degrees, and $n=[E\colon \mbb{Q}]$.}
\end{figure}

Specifically, let $\bar{g}_1$, $\bar{g}_2$ be the non-trivial automorphisms of the order $2$ groups $\Gal(E/E_1)$, $\Gal(E/E_2)$.  Then it was observed, in every case examined,
\begin{enumerate}
    \item $E_2 \subseteq \mbb{R}$, and $\bar{g}_2$
is complex conjugation
\item Let $\bar{g}_1$, $\bar{g}_2$ be the non-
    \item For all $g\in \Gal(E/\mbb{Q})$,
    \eag{
        g\bar{g}_1 &= \bar{g}_1 g &&\iff & g\bar{g}_2 &= \bar{g}_2 g   &  &\iff & g&\in \Gal(E/K).
    }
    \item\label{en:ge12} For all $g\in \Gal(E/\mbb{Q})$,  
    \eag{
    g(E_1) &= E_1, & g(E_2 &= E_2, &&\text{if $g\in \Gal(E/K$,}
    \\
    g(E_1) &= E_2, & g(E_2) &= E_1,  & &\text{if $g\notin \Gal(E/K)$.}
    }
    \item Let $g\in \Gal(E/\mbb{Q})$ be arbitrary.  Then 
    \begin{enumerate}
        \item $g(\Pi)$ is a $1$-SIC fiducial if and only if $g\in\Gal(E/K)$,
        \item \label{en:galeh} $g(\Pi)$ is a $1$-SIC fiducial on the same $\EC(d)$ orbit as $\Pi$ if and only if $g\in\Gal(E/H)$.
    \end{enumerate}
\end{enumerate}
\begin{defn}[Galois multiplet]\label{dfn:GaloisMultiplet}
    We refer to the set of $\EC(d)$ orbits obtained by acting on a fiducial $\Pi$ by elements of $\Gal(E/K)$ and/or conjugating with an element of $\EC(d)$ as a \textit{Galois multiplet}.
\end{defn}
It follows from \cref{en:galeh} in the above list that the Galois multiplet associated to $\Pi$ has cardinality $h$.

In every case examined, the fields associated to the multiplets in a given dimension appear to form a bounded lattice under set inclusion.  In particular, there is, in every case examined, a unique minimal field, and a unique maximal field.  The dimension 35 fields and associated multiplets\footnote{We are grateful to Markus Grassl for pointing out an error in the version of this diagram which appeared in ~\cite{ Appleby:2018}.} are illustrated in \Cref{fg:multiplet}.
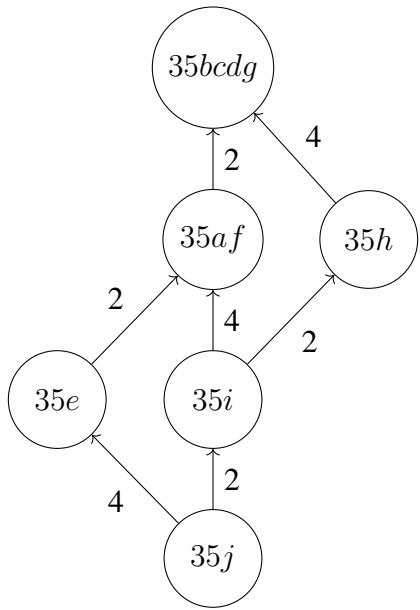
\begin{figure}[htb]
\begin{tikzpicture}
[auto, nd/.style={circle,minimum size = 13 mm,draw}]
\matrix[row sep=8mm,column sep=6mm] {
 & \node (35bcdg) [nd] {$35bcdg$};
\\
 & \node (35af) [nd] {$35af$}; & \node (35h) [nd] {$35h$};
\\
\node  (35e) [nd] {$35e$}; & \node (35i) [nd] {$35i$};
\\
 & \node (35j) [nd] {$35j$};
\\
};
\draw[->] (35j) to node[swap] {2} (35i);
\draw[->] (35i) to node[swap] {2} (35h);
\draw[->] (35j) to node{4} (35e);
\draw[->] (35i) to node[swap] {4} (35af);
\draw[->] (35e) to node{2} (35af);
\draw[->] (35af) to node[swap] {2} (35bcdg);
\draw[->] (35h) to node[swap] {4} (35bcdg);
\end{tikzpicture}
\caption{\label{fg:multiplet}Fields and multiplets in dimension 35.  The arrows indicate the field inclusions, and run from the smaller field to the larger.  So $35j$ is the minimal multiplet and $35bcdg$ is the maximal multiplet. Numbers beside the arrows are the degrees of the extensions. In this diagram we use the Scott--Grassl  convention~\cite{Scott2010,Scott:2017}, in which the $\EC(d)$ orbits for a given dimension are labelled by letters.  For example $35bcdg$ denotes the Galois multiplet consisting of the 4 Scott--Grassl orbits $35b$, $35c$, $35d$, $35g$.}
\end{figure}

Empirical investigation~\cite{Appleby:2020,Appleby:2017a,Appleby:2018,Kopp2019,Appleby:2022,Bengtsson:2024} indicates that the minimal multiplet in a given dimension $d$ generates the ray class field over $\mbb{Q}\bigl(\sqrt{D}\bigr)$ with modulus $\bar{d}$ and ramification at both infinite places. 
Empirical investigation also indicates that in the case of the minimal multiplet the field $H$ is the Hilbert class field, meaning that the multiplicity of the minimal multiplet is equal to the class number of $\mbb{Q}\bigl( \sqrt{D}\bigr)$. 
As we will see, it follows from our conjectures that these statements generalize to arbitrary multiplets of arbitrary rank $r$-SICs.  Specifically, the fields $E$, $E_1$, $E_2$, $H$ for an arbitrary multiplet of arbitrary rank are ray class fields of non-maximal orders, as defined in ~\cite{Kopp2020b,Kopp2020c}.

It follows from \cref{en:galeh} in the above list of properties that, for each  $g\in \Gal(E/H)$ there is an associated $U^{(g)}\in \EC(d)$ such that 
\eag{
g(\Pi) &= U^{(g)} \Pi U^{(g) \dagger}
}
To understand this action in more detail it helps to focus on the \overlapsText{} of strongly-centered fiducials:
\begin{defn}[strongly centered fiducial]\label{dfn:stronglycentered}
    A $r$-SIC fiducial $\Pi$ is said to be strongly centered if it is centered, and if all its \overlapsText{} are in the field $E_1$.
\end{defn}
Empirical investigation indicates that every $1$-SIC contains at least one strongly centered fiducial.

Now let $\Pi$ be a strongly centered fiducial, and let $\overlap_{\mbf{p}}$ be its \overlapsText{}. 
Then $\bar{g}_1(\overlap_{\mbf{p}}) = \overlap_{\mbf{p}}$ for all $\mbf{p}$, so it is enough to consider the action of the subgroup $\Gal(E_1/H)$. 

The empirical studies of $1$-SIC fiducials reported in ~\cite{Appleby:2020,Appleby:2017a,Appleby:2018,Kopp2019,Appleby:2022,Bengtsson:2024} indicate that in the case of a type-z orbit there is a natural isomorphism
\eag{
\Gal(E_1/H) \cong C/\StabPiOverlap{\Pi}
}
where $\StabPiOverlap{\Pi}$ is as defined in Definition~\ref{dfn:sogppi} and $C$ is the centralizer of $\StabPiOverlap{\Pi}$ considered as a subgroup of $\GLtwo{\mbb{Z}/\bar{d}\mbb{Z}}$. 
The isomorphism associates to each $g\in \Gal(E_1/H)$ a coset $F \StabPiOverlap{\Pi}$ with the property 
\eag{
g(\overlap_{\mbf{p}}) &= \overlap_{F'\mbf{p}}
}
for all $\mbf{p}$ and any $F'\in F \StabPiOverlap{\Pi}$. 

A similar statement holds for type a fiducials, except that $C$ has to be replaced by one of three maximal abelian subgroups of $\GLtwo{\mbb{Z}/\bar{d}\mbb{Z}}$ containing  
$\StabPiOverlap{\Pi}$. 
For more details see ~\cite{Appleby:2018}.

It is worth noting that this isomorphism gives us a geometrical interpretation of the Galois group which is very much in the spirit of Hilbert's original formulation of his $12^{\rm{th}}$ problem.

In every case examined, the \normalizedOverlapsText{} $\normalizedOverlap_{\mbf{p}}$ (as defined in~\Cref{dfn:sicovlp}) are units. 
The subgroup of the full unit group which they generate has some interesting properties which are described in ~\cite{Appleby:2020}.

\subsubsection{Dimension towers and \texorpdfstring{$1$}{1}-SIC alignment}
\label{ssec:dimtowsicalign}
As we saw in \Cref{sssc:FieldsMultipletsGhosts}, empirical investigations suggest that a $1$-SIC in dimension $d$ gives rise to an abelian extension of the real quadratic field $\mbb{Q}(\sqrt{(d-3)(d+1)})$. 
It is natural to ask how many dimensions are associated in this way to a given real quadratic field.  The following theorem answers that question.

\begin{thm}\label{thm:dimtowunique}
    Let  $K$ be a real quadratic field. 
    Then $K = \mbb{Q}(\sqrt{(d-3)(d+1)})$ if and only if $d$ is in the dimension tower associated to $K$ (see \Cref{dfn:fjrjmdjm}). 
\end{thm}
\begin{proof}
    See ~\cite[Lemma 4]{Appleby:2020}.
\end{proof}
\begin{thm}
    Let $j_1, j_2, \dots$ be an increasing sequence of natural numbers such that, for all $n\in \mbb{N}$,
\begin{enumerate}
    \item $j_n$ divides $j_{n+1}$, 
    \item $j_{n+1}/j_n$ is coprime to 3.
\end{enumerate}
Then $d_{j_n}$ divides $d_{j_{n+1}}$ for all $n\in \mbb{N}$.
\end{thm}
\begin{proof}
    See the proof of ~\cite[Prop.~7]{Appleby:2020}. 
    In ~\cite{Appleby:2020} this result is stated with the unnecessarily strong condition that the individual terms of the sequence $j_1, j_2, \dots$ are coprime to 3. 
    However, the proof is easily seen to imply that the result continues to hold with the weaker condition we have stated here. 
\end{proof}   
It follows from this theorem that, if the various properties of the known $1$-SICs described above generalize  to every $1$-SIC in  every dimension, then for each subsequence $j_1, j_2, \dots$ satisfying the conditions of the theorem, one has the field inclusions
\eag{
E(d_{j_1}) \subseteq E(d_{j_2}) \subseteq \dots
}
where $E(d)$ denotes the field associated to the minimal multiplet in dimension $d$. 
Given this relationship of the fields, it is natural to ask if there is a corresponding relationship of the $1$-SICs themselves. 
This question was addressed in ~\cite{Appleby:2017} for a number of dimension pairs $d_j$, $d_{2j}$ (see also ~\cite{Andersson:2019}). 
Denoting the \normalizedOverlapsText{} in dimensions $d_j$, $d_{2j}$ by $e^{i\theta_{\mbf{p}}}$, $e^{i\Theta_{\mbf{p}}}$ respectively, it was found, in the cases examined, that
\eag{
e^{i\Theta_{d_j\mbf{p}}} &=
\begin{cases}
    1 \qquad &\text{$d_j$ odd},
    \\
    -(-1)^{(p_1+1)(p_2+1)} & \text{$d_j$ even},
\end{cases}
\label{eq:sicalign1}
\\
e^{i\Theta_{(d_j-2)\mbf{p}}} &=
\begin{cases}
    -e^{2i\theta_{F\mbf{p}}} \qquad &\text{$d_j$ odd},
    \\
    (-1)^{(p_1+1)(p_2+1)}e^{2i\theta_{F\mbf{p}}} & \text{$d_j$ even}.
\end{cases}
\label{eq:sicalign2}
}
for some $F\in \ESLtwo{\mbb{Z}/\bar{d}\mbb{Z}}$.
Moreover, it was found that these properties generalize to other $1$-SIC multiplets.  
\begin{defn}\label{dfn:sicalign}
    A pair of $1$-SICs in dimensions $d_j$, $d_{2j}$ whose \normalizedOverlapsText{} satisfy ~\eqref{eq:sicalign1}, \eqref{eq:sicalign2} are said to be \textit{aligned}.
\end{defn}
In \Cref{sec:sicphenexplain} a generalized property of $1$-SIC-alignment will be defined and shown to be a consequence of our conjectures.

\subsection{Proofs of theorems on SIC projectors and overlaps}

We will now prove Theorems~\ref{thm:rsicbsc} and~\ref{thm:sicfidcond} from the introduction.

\begin{proof}[Proof of Theorem~\ref{thm:rsicbsc}]
\label{sssec:proofofrsicoverlap}

The result is well-known.
However, existing discussions~\cite{Appleby:2007a,Casazza:2011,Fickus:2017,King:2021} locate $r$-SICs in a larger context (symmetric POVMs, or fusion frames not assumed to be maximal), which obscures the fact that everything follows from maximality plus the equiangularity condition in \Cref{def:equiangularcond}. 
We therefore give an independent proof.

Let $\mcl{T}_0$ be the $d^2-1$ dimensional subspace of $\mcl{L}(H_d)$ consisting of all operators $A$ such that $\Tr(A) =0$. 
Define operators $B_j\in \mcl{T}_0$ by
\eag{
B_{j}&=\sqrt{\frac{d}{r(d-r)}} \Pi_j -\sqrt{\frac{r}{d(d-r)}} I.
\label{eq:rsicbscbjdf}
}
There must exist at least one set of numbers $c_j$, not all $0$, such that
\eag{
\sum_{j=1}^{d^2} c_j B_j &=0.
}
The $c_j$ must satisfy, for all $k$
\eag{
0 &= \Tr\left(\left( \sum_{j=1}^{d^2} c_j B_j \right) B_k \right)
= -\left(\frac{d(\alpha-r)}{r(d-r)}\right) c_k +
\frac{\alpha d-r^2}{r(d-r)} \sum_{j=1}^{d^2} c_j
\label{eq:rsicbsclicnd}
}
implying $c_1 = \dots = c_{d^2} = \mu$ for some fixed, non-zero constant $\mu$. 
It follows that the orthogonal complement of the $B_j$ is $1$-dimensional, implying that the $B_j$ are a spanning set for $\mcl{T}_0$. 
Substituting $c_j = \mu$ into ~\eqref{eq:rsicbsclicnd} we deduce 
\eag{
\alpha &= \frac{r(rd-1)}{d^2-1},
}
from which ~\eqref{eq:rsicbscgenolp} follows. 
We have incidentally shown that
\eag{
\sum_{j=1}^{d^2} B_j &=0
}
from which ~\eqref{eq:rsicbscresid} follows. 

Finally, let $\mcl{S}$ be the span of the $\Pi_j$. 
It follows from ~\eqref{eq:rsicbscresid} that $I\in \mcl{S}$, which in turn implies the $B_j$ are all in $\mcl{S}$, and consequently that $\mcl{T}_0\subseteq \mcl{S}$. 
Since every element of $\mcl{L}(H_d)$ can be written as a linear combination of $I$ and an element of $\mcl{T}_0$, it follows that the $\Pi_j$ are a basis for $\mcl{L}(H_d)$.
\end{proof}

\begin{proof}[Proof of Theorem~\ref{thm:sicfidcond}]
\label{sssec:proofofmodulusofoverlapphase}

Let $\Pi$ be an \hprj{} in dimension $d$. 
It follows from ~\eqref{eq:mdopexp} that
\eag{
\Pi_{\mbf{p}}&= \frac{1}{d}\sum_{\mbf{k}}\Tr\left(\Pi D^{\dagger}_{\mbf{k}}\right)D\vpu{\dagger}_{\mbf{p}} D\vpu{\dagger}_{\mbf{k}}D^{\dagger}_{\mbf{p}}
=\frac{1}{d}\sum_{\mbf{k}} 
\Tr\left(\Pi D^{\dagger}_{\mbf{k}}\right)\rtus_d^{\la\mbf{p},\mbf{k}\ra} D_{\mbf{k}}.
}
Hence
\eag{
\Tr\left(\Pi_{\mbf{p}}\Pi_{\mbf{q}}\right)
&= \frac{1}{d}\sum_{\mbf{k},\mbf{k}'}
\Tr\left(\Pi D^{\dagger}_{\mbf{k}}\right)\Tr\left(\Pi D^{\dagger}_{\mbf{k}'}\right)\rtus_d^{\la \mbf{p},\mbf{k}\ra+\la \mbf{q},\mbf{k}\ra}(-1)^{\frac{d+1}{d}\la \mbf{k},\mbf{k}'\ra}\delta^{(d)}_{\mbf{k},-\mbf{k}'}
\nn 
&= \frac{1}{d}\sum_{\mbf{k}}
\left|\Tr\left(\Pi D^{\dagger}_{\mbf{k}}\right)\right|^2\rtus_d^{\la \mbf{p}-\mbf{q},\mbf{k}\ra}.
}
So $\Pi$ is an $r$-SIC fiducial if and only if
\begin{alignat}{3}
&& \qquad \frac{1}{d}\sum_{\mbf{k}}
\left|\Tr\left(\Pi D^{\dagger}_{\mbf{k}}\right)\right|^2\rtus_d^{\la \mbf{p},\mbf{k}\ra}
&=
\left(\frac{rd(d-r)}{d^2-1}\right)\delta^{(d)}_{\mbf{p},\zero} +\frac{r(rd-1)}{d^2-1} & \qquad &\forall \mbf{p}
\nn
&\iff& \qquad 
\left|\Tr\left(\Pi D^{\dagger}_{\mbf{k}}\right)\right|^2
&= \frac{r(d-r)}{d^2-1} + \left(\frac{rd(rd-1)}{d^2-1}\right)\delta^{(d)}_{\mbf{k},\zero}
&\qquad &\forall \mbf{k}
\nn
&\iff& \qquad 
\left|\Tr\left(\Pi D^{\dagger}_{\mbf{k}}\right)\right|^2
&= \frac{r(d-r)}{d^2-1} 
&\qquad &\forall \mbf{k}\neq \zero.
\end{alignat}
\end{proof}

\section{Units, dimensions, and binary quadratic forms}\label{sec:units}

This section examines the unit group of a real quadratic field and its representations by matrices in $\GLtwo{\mbb{Z}}$. 
We prove a number of results on these topics that will be needed in the sequel. 
We also prove some facts about the dimension towers and grids defined in \Cref{ssec:admissibletuples}.

In this section $D$ will always be a fixed square free integer greater than 1, $K=\mbb{Q}(\sqrt{D})$  the associated real quadratic field, $\Delta_0$ the  discriminant of $K$, and $\vn$,  $d_j$, $f_j$, $d_{j,m}$,  $r_{j,m}$ the quantities specified in Definitions~\ref{dfn:fundamentalTotallyPositiveUnit}, \ref{dfn:sequenceofconductors}, and \ref{dfn:fjrjmdjm}. We will also use the following definitions.
\begin{defn}[fundamental unit]\label{dfn:fundamentalUnit}
    Define $\un$ to be the smallest unit of $K$ which is greater than 1.
\end{defn}  
\begin{defn}[discriminant at level $j$]\label{dfn:discriminantLevelj}
    The discriminant at level $j$ of the tower is 
    \eag{
        \Delta_j &= (d_j-3)(d_j+1).
    }
\end{defn} 
In terms of $\un$ one has
\eag{
\vn &= 
\begin{cases}
\un^2 \quad & \mbox{if } \Nm(\un) = -1,
\\
\un \quad & \mbox{if } \Nm(\un) = 1,
\end{cases}
}
where $\Nm(\un)$ denotes the norm of $\un$.

\subsection{Dimension towers}
We now derive various useful relations between the between the quantities $d_j$ and $f_j$ as $j$ varies. In particular, for every integer $n$, we will give (in \Cref{thm:dnjfnjchebyexpns}) an expression for the pair $(d_{nj},r_{nj})$ in terms of the pair $(d_j,r_j)$. Our first lemma concerns that case $n=2$. 
\begin{lem}\label{lem:towerbasic}
For every positive integer $j$, the following relations hold.
\eag{
d_j, f_j &\in \mbb{Z}
\\
d_j&= \vn^j+\vn^{-j}+1
\label{eq:towerbasic1}
\\
\Delta_j &=f_j^2\Delta_0
\\
d_{2j}+1 &=  (d_j-1)^2
\label{eq:dSub2jInTermsOfdSubj}
\\
d_{2j}-3 &= f_j^2\Delta_0
\\
\vn^j &= \frac{(d_j-1)+f_j\sqrt{\Delta_0}}{2} 
=
\frac{\sqrt{d_{2j}+1}+\sqrt{d_{2j}-3}}{2}
}
\end{lem}
\begin{proof}
 The fact that $\vn$ is an algebraic integer means 
 \eag{
 \vn^j &= n_1 + n_2\left(\frac{\Delta_0+\sqrt{\Delta_0}}{2}\right)
 }
 for some $n_1$, $n_2\in \mbb{Z}$.  Hence $f_j = n_2 \in \mbb{Z}$.
 
    To prove~\eqref{eq:towerbasic1}, observe that it follows from \Cref{dfn:fjrjmdjm} that
\eag{
d_j 
&= \frac{\vn^{2j}-\vn^{-2j}}{\vn^{j}-\vn^{-j}}+1
= \vn^j + \vn^{-j}+1,
}
from which it also follows that $d_j\in\mbb{Z}$.  We then have
\eag{
\frac{\Delta_j}{\Delta_0} &= \frac{(\vn^j+\vn^{-j}-2)(\vn^j+\vn^{-j}+2)}{\Delta_0} = \frac{(\vn^j-\vn^{-j})^2}{\Delta_0} = f_j^2, \\
d_{2j}+1 &= \vn^{2j}+\vn^{-2j}+2 = (\vn^j+\vn^{-j})^2 = (d_j-1)^2, \\
d_{2j}-3 &= (d_j-1)^2 -4 =\Delta_j, \\
\vn^j &= \frac{(\vn^j + \vn^{-j})+(\vn^j - \vn^{-j})}{2}
= \frac{d_j-1+f_j\sqrt{\Delta_0}}{2}
= \frac{\sqrt{d_{2j}+1}+\sqrt{d_{2j}-3}}{2},
}
completing the proof.
\end{proof}
The next result is needed for the proof of \Cref{thm:alignment}.
\begin{lem}\label{lm:dnjPlus1Expression}
    For any pair of integers $j\ge 1$ $n>1$,
    \eag{\label{eq:dnjPlus1Expression}
    d_{nj}+1 &=
    \begin{cases}
        (d_{\frac{nj}{2}}-1)^2 \qquad & n\equiv 0 \Mod{2},
        \\
(d_{j}+1)\left(1+\sum_{r=1}^{\frac{n-1}{2}}(-1)^{r} d_{rj}\right)^2 \qquad & n\equiv 1 \Mod{4},
        \\
(d_{j}+1)\left(2+\sum_{r=1}^{\frac{n-1}{2}}(-1)^{r} d_{rj}\right)^2 \qquad & n\equiv 3 \Mod{4}.
    \end{cases}
    }
\end{lem}
\begin{proof}
    If $n$ is even, the result follows from \eqref{eq:dSub2jInTermsOfdSubj}.  
    If $n=2m+1$, then \eqref{eq:towerbasic1} implies
    \eag{
    \frac{d_{nj}+1}{d_j+1} &= \left(\frac{\vn^{\left(m+\frac{1}{2}\right)j} + \vn^{-\left(m+\frac{1}{2}\right)j}}{\vn^{\frac{j}{2}} + \vn^{-\frac{j}{2}}}\right)^2
    \nn
    &=
    \left(
    \vn^{mj}-\vn^{(m-1)j}+ \dots -\vn^{-(m-1)j} +\vn^{-mj}
    \right)^2
    \nn
    &=
  \left( (-1)^m + \sum_{r=0}^{m-1} (-1)^{r}(d_{(m-r)j}-1)\right)^2
  \nn
  &= 
  \begin{cases}
      \left(1+\sum_{r=0}^{m-1}(-1)^rd_{(m-r)j}\right)^2 \quad & \text{$m$ even,}
      \\
      \left(-2+\sum_{r=0}^{m-1}(-1)^rd_{(m-r)j}\right)^2 \quad & \text{$m$ odd,}
  \end{cases}
  \nn
  &= 
  \begin{cases}
      \left(1+\sum_{r=1}^{m}(-1)^rd_{rj}\right)^2 \quad & \text{$m$ even,}
      \\
      \left(2+\sum_{r=1}^{m}(-1)^rd_{rj}\right)^2 \quad & \text{$m$ odd,}
  \end{cases}
    }
    giving the last two cases of \eqref{eq:dnjPlus1Expression}.
\end{proof}

The next result proves the monotoncity of the sequences of $d_j$ and $f_j$.
\begin{thm}\label{thm:djandfjmonotonic}
The sequences of $d_j$ and $f_j$ satisfy
    \eag{
4 &\le d_1 < d_2 < \cdots,
\\
1 & \le f_1 < f_2 < \cdots.
    }
\end{thm} 
\begin{proof}
    The fact that $\vn>1$ means that, if $x>0$, then
    \eag{
\frac{d}{dx}(\vn^x + \vn^{-x}+1) &= \log \vn (\vn^x-\vn^{-x}) > 0,
}
and
\eag{
\frac{d}{dx}\left(\frac{\vn^x - \vn^{-x}}{\sqrt{\Delta_0}}\right) &= \log \vn 
\left(\frac{\vn^x + \vn^{-x}}{\sqrt{\Delta_0}}\right)>0.
    }
    Thus, the sequences $d_j = \e^j+\e^{-j}+1$ and $f_j = \frac{\e^j-\e^{-j}}{\sqrt{\Delta_0}}$ are monotonically increasing. It is also clear that $d_1 > 3$ and $f_1 > 0$, so since they are integers, $d_1 \geq 4$ and $f_1 \geq 1$.
\end{proof}

We now give some special properties of the sequence of dimensions $d_j$ associated to real quadratic fields with a unit of negative norm.
\begin{thm}\label{thm:negunitcondA}
   The following statements are equivalent:
\begin{enumerate}
    \item \label{it:nufem1} $\Nm(\un)=-1$,
    \item \label{it:drrm3psall}$d_j-3$ is a perfect square for all odd values of $j$,
    \item \label{it:drrm3psone}$d_j-3$ is a perfect square for one odd value of $j$.
\end{enumerate}
In that case 
\eag{
\un^j &= \frac{\sqrt{d_j-3}+\sqrt{d_j+1}}{2}
\label{eq:vpowrnegnrm}
}
for all $j\ge 1$.  One has,
\begin{enumerate}[label = (\alph*)]
    \item If $j$ is odd, then $d_j-3$ and $\frac{d_j+1}{\Delta_0}$ are perfect squares;
    \item If $j$ is even, then $\frac{d_j-3}{\Delta_0}$ and $d_j+1 $ are perfect squares.
\end{enumerate}
\end{thm}
\begin{rmkb}
    As we will see, on the assumption that the Stark--Tate Conjecture and the Twisted Convolution Conjecture are both true, \Cref{thm:negunitcondA} explains the empirical observation~\cite{Scott:2017} that, in every case calculated, the minimal multiplet in dimension $d$ has an anti-unitary symmetry if and only if $d-3$ is a perfect square, and that it contains a real $r$-SIC fiducial if in addition $d-3$ is even.
\end{rmkb}
\begin{proof}
    For the equivalence of statements~\eqref{it:nufem1},~\eqref{it:drrm3psall},~\eqref{it:drrm3psone}, see \cite[Thm.~1]{Yokoi:1968}.  If $j$ is even, \eqref{eq:vpowrnegnrm} is proved in \Cref{lem:towerbasic}.  Suppose $j$ is odd.  Then
    \eag{
        \un^j &= m_1+m_2\left(\frac{\Delta_0+\sqrt{\Delta_0}}{2}\right)
    }
    for some pair of integers $m_1,m_2$ such that $2m_1+m_2\Delta_0$, $m_2$ are both positive.  The fact that $\Nm(\un^j) = -1$ means
    \eag{
        \left(\frac{2m_1+m_2\Delta_0}{2}\right)^2 - \left(\frac{m_2\sqrt{\Delta_0}}{2}\right)^2 = -1
    }
    while the fact that $\un^{2j}+\un^{-2j} +1 = d_j$ means
    \eag{
        2\left(\frac{2m_1+m_2\Delta_0}{2}\right)^2 +2\left(\frac{m_2\sqrt{\Delta_0}}{2}\right)^2 = d_j-1.
    }
    Hence
    \eag{
    \left(2m_1+m_2\Delta_0\right)^2 &=d_j-3;
    \\
    \left(m_2\sqrt{\Delta_0}\right)^2 &= d_j+1.
    }
    Hence
    \eag{
        \un^j &=\frac{\sqrt{d_j-3}+\sqrt{d_j+1}}{2}.
    }
    This also shows that, if $j$ is odd, then $d_j-3$ and $\frac{d_j+1}{\Delta_0}$ are perfect squares.  For the proof that   $\frac{d_j-3}{\Delta_0}$ and $d_j+1$ are perfect squares when $j$ is even, see \Cref{lem:towerbasic}.
\end{proof}

We will now define \textit{variant Chebyshev polynomials} and use them to express the pair $(d_{nj},f_{nj})$ as a function of the pair $(d_j,f_j)$ for any natural number $n$, as promised.
\begin{defn}[variant Chebyshev polynomials]\label{dfn:variantChebyshev}
For all $j\in \mbb{N}$ define
    \eag{
T^{*}_j(x) &= 1 + 2T_j\left(\frac{x-1}{2}\right),
\label{eq:tstardef}
\\
 U^{*}_j(x) &= U_{j-1} \left(\frac{x-1}{2}\right)
\label{eq:ustardef}
}
where the $T_j$ and $U_j$ are respectively Chebyshev polynomials of the first and second kind. %
\end{defn}

\begin{thm}\label{thm:dnjfnjchebyexpns}
For $n$, $j\in \mbb{N}$
    \eag{
d_{nj} &=T^{*}_n(d_j),
\label{eq:dkrCheb}
\\
f_{nj} & = f_j U^{*}_n(d_j) 
\label{eq:fkrCheb}
}
\end{thm}
\begin{proof}
    For the proof of ~\eqref{eq:dkrCheb}, see ~\cite[eq.~(13)]{Appleby:2020}.
    To prove ~\eqref{eq:fkrCheb}, let $\theta = \log \vn$.  Then
    \eag{
        f_{nj}&= \frac{2\sinh(nj \theta)}{\sqrt{\Delta_0}}
        =\frac{2U_{n-1}(\cosh j\theta)\sinh j\theta}{\sqrt{\Delta_0}}
        = f_j U^{*}_n(d_j),
    }
    which is \eqref{eq:fkrCheb}.
\end{proof}

\begin{lem}\label{lm:tustarprops}
 $T^{*}_j(x)$, $U^{*}_j(x)$ satisfy the recursion relations
\eag{
T^{*}_1(x) &= x, & T^{*}_2(x) &= x(x-2) ,& T^{*}_j (x) &= 3-x +(x-1)T^{*}_{j-1}(x) -T^{*}_{j-2}(x),
\label{eq:tstrec}
\\
U^{*}_{1}(x) &= 1, & U^{*}_2(x) &=x-1, & U^{*}_j(x) &= (x-1)U^{*}_{j-1}(x) -U^{*}_{j-2}(x).
\label{eq:ustrec}
}
For all $j\in \mbb{N}$,
\eag{
T^{*}_j(x) &= 
\begin{cases}3 +x^2\left(- \frac{j^2}{3}  + \Ord(x)\right) \quad & \text{$j \equiv 0 \Mod{3}$,}
\\
x\left(j  + \frac{j(j-1)}{6}x + \Ord(x^2)\right) \quad & \text{$j \equiv 1 \Mod{3}$,}
\\
x\left(-j  + \frac{j(j+1)}{6}x + \Ord(x^2)\right) \quad & \text{$j \equiv 2 \Mod{3}$,}
\end{cases}
\label{eq:tcheb}
\\
U^{*}_j(x)
&=
\begin{cases}
x\left(-\frac{2j}{3} +\frac{j}{3} x +\Ord(x^2)\right) \quad & \text{$j \equiv 0 \Mod{3}$,}
\\
1+x\left(\frac{j-1}{3} -\frac{(j-1)(j+2)}{6}x + \Ord(x^2)\right) \quad & \text{$j\equiv 1 \Mod 3$,}
\\
-1+x\left(\frac{j+1}{3} +\frac{(j+1)(j-2)}{6}x+\Ord(x^2)\right) \quad & \text{$j \equiv 2 \Mod 3$.}
\end{cases}
\label{eq:ucheb}
}
\end{lem}
\begin{proof}
    Straightfoward consequence of the recursion relations for the Chebyshev polynomials.
\end{proof}

We are now in a position to prove some congruence and divisibility properties of the $d_j$ and $f_j$.
\begin{lem}\label{lem:dfdelprops}
Let $j,n\in \mbb{N}$.
\begin{enumerate}
\item \label{it:dfdelprop1}  $d_j$ is a divisor of $d_{nj}$ if and only if $n\not\equiv 0 \Mod{3}$.

\item  If $d_j$ is odd, then 
\begin{enumerate}
    \item \label{it:dfdelprop2e}   $\Delta_0 \equiv 0 \Mod{4}$ if $f_j$ is odd,
    \item \label{it:dfdelprop2d} $\Delta_{nj} \equiv 0 \Mod{4}$ for all $n$,
    \item \label{it:dfdelprop2a} $d_{nj}$ is odd for all $n$,
    \item \label{it:dfdelprop2b} if $f_j$ is even, then $f_{nj}$ is even for all $n$,
    \item \label{it:dfdelprop2c} if $f_j$ is odd, then $f_{nj}$ is odd if and only if $n$ is odd.
\end{enumerate}

\item \label{it:dfdelprop3} If $d_j$ is even, then
\begin{enumerate}
    \item \label{it:dfdelprop3a} $\Delta_0 \equiv 1 \Mod{4}$,
    \item \label{it:dfdelprop3b} $\Delta_{nj} \equiv 1 \Mod{4}$ if and only if $n\not\equiv 0 \Mod{3}$,
    \item \label{it:dfdelprop3c} $d_{nj}$ is even if and only if $n\not\equiv 0 \Mod{3}$,
    \item \label{it:dfdelprop3d} $f_{nj}$ is odd if and only if $n\not\equiv 0 \Mod{3}$.
\end{enumerate}
\end{enumerate}
\end{lem}
\begin{proof}
\Cref{it:dfdelprop1} is a consequence of 
~\eqref{eq:tcheb} and the fact that $d_j>3$. 

To prove \cref{it:dfdelprop2e}, observe that the fact that $\Delta_0$ is a discriminant means it is congruent to $0$ or $1$ modulo $4$.  The fact that
\eag{
f_j^2\Delta_0 &= (d_j-3)(d_j+1)
\label{eq:fjSquaredDelta0Expression}
}
is even, together with the fact that $f_j$ is odd, then implies that $\Delta_0 \equiv 0 \Mod{4}$.

To prove \cref{it:dfdelprop2d} and \cref{it:dfdelprop2a}, observe that it follows from \eqref{eq:tstrec} that, if $d_j$ is odd, then
\eag{
T^{*}_1(d_j) &\equiv T^{*}_2(d_j) \equiv 1 \ \Mod{2},
\\
\text{and } \
T^{*}_n(d_j) &\equiv T^{*}_{n-2}(d_j) \ \Mod{2}.
}
Consequently $d_{nj} = T^{*}_n(d_j)$ is odd for all $n$.  It then follows that $\Delta_{nj} = (d_{nj}-3)(d_{nj}+1)$ is even for all $n$.  Since $\Delta_{nj}$ is a discriminant, it follows that we must in fact have $\Delta_{nj} \equiv 0 \Mod{4}$.

\Cref{it:dfdelprop2b} is an immediate consequence of 
~\eqref{eq:fkrCheb}. 

To prove \cref{it:dfdelprop2c} observe that it follows from 
~\eqref{eq:ustrec} that $U^{*}_n(d_j)$ is odd if and only if $n$ is odd.  Since $f_j$ is odd, it follows from 
~\eqref{eq:fkrCheb} that $f_{nj}$ is odd if and only if $n$ is odd.

To prove \cref{it:dfdelprop3a}, observe that if $d_j$ is even, then
 \eag{
 f_j^2 \Delta\vpu{2}_0 &= \Delta\vpu{2}_j = (d_j-3)(d_j+1) \equiv 1 \ \Mod{2},
 }
implying that $f_j$, $\Delta_0$, $\Delta_j$ are all odd. The statement follows from this and the fact that $\Delta_0$ is a discriminant, which means $\Delta_0$ is congruent to $0$ or $1$ modulo $4$.

\Cref{it:dfdelprop3c} is an immediate consequence of \eqref{eq:tcheb}.  It then follows that
\eag{
\Delta_{nj} &= (d_{nj}-3)(d_{nj}+1)
}
is odd if and only if $n \not\equiv 0 \Mod{3}$. 
\Cref{it:dfdelprop3b} is a consequence of this and the fact that $\Delta_{nj}$ is a discriminant (and thus congruent to $0$ or $1$ modulo $4$).
 
Finally, \eqref{eq:fjSquaredDelta0Expression} together with  the fact that $d_j$ is even means $f_j$ is odd. In view of 
~\eqref{eq:ucheb},
\eag{
f_{nj}&= f_j U^{*}_n(d_j)
\equiv
\begin{cases}
    0 \ \Mod{2} \quad & \text{if $n \equiv 0 \Mod{3}$},
    \\
    1 \ \Mod{2} \quad & \text{if $n \not\equiv 0 \Mod{3}$}.
\end{cases}
}
This establishes \cref{it:dfdelprop3d} and completes the proof. 
\end{proof}
\subsection{Unit group of an order}
As we will see, 
under our conjectures,
there is a natural correspondence between $r$-SIC multiplets and orders of the real quadratic field $K$.  
In this subsection we prove some facts about the unit group of an order which will be needed in the sequel.
\begin{defn}[unit group, and positive norm unit group of conductor $f$]\label{df:ofufdef}
Let $\mcl{O}_f$ be the order of conductor $f$ in the real quadratic field $K$ (see \Cref{dfn:orderConductorf}).   
    Define 
    \eag{
        \mcl{U}_f &= \{w\in \mcl{O}_f \colon \Nm(w)=\pm 1\} = \OO_f^\times
    }
    to be the unit group of $\mcl{O}_f$, and 
    \eag{
         \mcl{U}^{+}_f &= \{w\in \mcl{O}_f \colon \Nm(w)= 1\}
    }
   to  be the subgroup of $ \mcl{U}_f $ consisting of all positive norm units.
\end{defn}
\begin{rmkb}
    To avoid cluttering the notation, we do not indicate the field $K$ explicitly. It will always be clear from context.
\end{rmkb}
\begin{thm}\label{thm:Leastj}
    Let $K$ be a real quadratic field ,and let $f_1, f_2, \dots$ be its associated sequence of conductors. Then, for  each  positive integer $f$, there exists a positive integer $j$ such that $f\div f_j$.
\end{thm}
\begin{proof}
    It follows from the generalization of Dirichlet's unit theorem to an arbitrary order (see, for example, \cite{Koch2000}) that
    \eag{
    \mcl{U}^{+}_f=\{\pm w^n \colon n\in \mbb{Z}\}
    \label{eq:upftrmsw}
    }
    for some unit $w\neq \pm 1$.  It can be assumed, without loss of generality, that $w>1$.  We then have
    \eag{
    w&=\vn^{j} = \frac{d_j-1-f_j\Delta_0}{2}+f_j\left(\frac{\Delta_0+\sqrt{\Delta_0}}{2}\right)
    \label{eq:wtrmsej}
    }
    for some $j\in \mbb{N}$.  The fact that $w\in \mcl{O}_f$  implies $f\div f_j$.
\end{proof}
This result motivates the following definition.
\begin{defn}[minimum level; fundamental positive norm unit of an order]\label{df:epsilonfDefinition}
    Let $K$ be a real quadratic field, and let $f_1, f_2, \dots$ be its associated sequence of conductors. For each  positive integer $f$, define
    \eag{
    j_{\rm{min}}(f) &= \min \{j\in \mbb{N}\colon f|f_j\},
    \\
    \intertext{and}
    \vn_f &= \vn^{j_{\rm{min}}(f)}.
    }
\end{defn}
\begin{rmkb}
    This definition depends on \Cref{thm:Leastj} for its validity.  Note that $j_{\rm{min}}(f)$ and $\vn_f$ both depend on $K$ as well as $f$, but to avoid cluttering the notation the $K$-dependence is not explicitly indicated.
   The identity of $K$  will always be clear from context. It follows from the next theorem that $\vn_f$ is the smallest unit greater than $1$ in $\mcl{U}^{+}_f$.
\end{rmkb}
\begin{thm}\label{thm:UnitsConductorsIndicesProperties}
Let $K$ be a real quadratic field.  Then:
\begin{enumerate}
    \item \label{it:upftrmsef} For all $f\in \mbb{N}$,
    \eag{
    \mcl{U}^{+}_f &= \{\pm \vn_f^{n}\colon n\in \mbb{Z}\}.
    }
    \item \label{it:jminfkexpn} For all $k\in \mbb{N}$,
        \eag{
    j_{\rm{min}}(f_k) &= k.
    }
    \item \label{it:fdivfjeqjmindivj}  For all $f, j\in \mbb{N}$,
    \eag{
    f
    \div f_j 
    \iff 
    j_{\rm{min}}(f) 
    \div j.
    }
    \item \label{it:jdivkfjdivfk} For all $j, k\in \mbb{N}$,
    \eag{\label{eq:jdivkfjdivfk}
    j
    \div k 
    \iff 
    f_j 
    \div f_k
    }
\end{enumerate}
\end{thm}
\begin{proof}
    \textit{Statement \eqref{it:upftrmsef}.} Let $w$, $j$ be as in \eqref{eq:upftrmsw} and \eqref{eq:wtrmsej}.  The fact that $f\div f_j$ implies $j_{\rm{min}}(f)\le j$.  On the other hand, the fact that $f\div f_{j_{\rm{min}}(f)}$ implies 
    \eag{
    \vn^{j_{\rm{min}}(f)} &= \frac{d_{j_{\rm{min}}(f)}-1-f_{j_{\rm{min}}(f)}\Delta_0}{2}+f_{j_{\rm{min}}(f)}\left(\frac{\Delta_0+\sqrt{\Delta_0}}{2}\right)
    }
    is in $\mcl{O}_f$ and consequently $\mcl{U}^{+}_f$.  Since $\vn^{j_{\rm{min}}(f)} >1$ we must have $\vn^{j_{\rm{min}}(f)} = w^t = \vn^{t j}$ for some $t\in \mbb{N}$, implying $j_{\rm{min}}(f) \ge j$.  We conclude $j_{\rm{min}}(f) = j$. 

     \textit{Statement \eqref{it:jminfkexpn}.} It follows from \Cref{thm:djandfjmonotonic} that $j\ge k$ for all $j \in \{j\in \mbb{N}\colon f_k |f_j\} $, implying $j_{\rm{min}}(f_k) \ge k$.  Since $k$ itself is in $\{j\in \mbb{N}\colon f_k |f_j\}$, we must in fact have $j_{\rm{min}}(f_k) = k$.

     \textit{Statement \eqref{it:fdivfjeqjmindivj}.} Observe that $f\div f_j$ if and only if 
     \eag{
     \vn^{j} = \frac{d_j-1-f_j\Delta_0}{2}+f_j\left(\frac{\Delta_0+\sqrt{\Delta_0}}{2}\right)
     }
   is in $\mcl{O}_f$, and consequently in $\mcl{U}^{+}_f$.  In view of Statement~\eqref{it:upftrmsef} this means $f\div f_j$ if and only if $\vn^j = \vn_f^n$ for some $n\in \mbb{N}$, which in turn is true if and only if $j_{\rm{min}}(f)\div j$.

    \textit{Statement \eqref{it:jdivkfjdivfk}.}  It follows from statements \eqref{it:jminfkexpn} and \eqref{it:fdivfjeqjmindivj} that
    \eag{
    j
    \div k 
    \iff 
    j_{\rm{min}}(f_j)
    \div k 
    \iff 
    f_j \div f_k,
    }
    giving \eqref{eq:jdivkfjdivfk}.
\end{proof}

We see from this that the conductor $f$ is naturally associated to the infinite sequence of dimensions $d_{j_{\rm{min}}(f)}, d_{2j_{\rm{min}}(f)}, d_{3j_{\rm{min}}(f)},\dots$.  As we will see, if the Stark Conjecture and Twisted Convolution Conjecture are correct, then the members of the corresponding sequence of $1$-SIC multiplets are  related.  This is a generalization of the phenomenon of $1$-SIC alignment (see \cite{Appleby:2017,Andersson:2019,Bengtsson:2022} and \Cref{dfn:sicalign} above).
As we will see,
\Cref{thm:negunitcondA}, combined with the our conjectures, tells us for which minimal multiplets (multiplets corresponding to $f=1$)  there is an anti-unitary symmetry.  In the sequel we will also address the question of which non-minimal multiplets have such a symmetry.
If $d_1-3$ is not a perfect square the question is easily answered, as it is then automatic that $ \mcl{U}\vpu{+}_f=\mcl{U}^{+}_f$ for all $f$.  The next theorem determines for which values of $f$ the unit group $\mcl{U}_f$ contains a negative norm unit if $d_1-3$ is a perfect square.  
\begin{thm} \label{thm:negunitcondb}
Let $f$ be any positive integer.  Then the following statements are equivalent:
\begin{enumerate}
    \item \label{it:negunitcondb1} $\mcl{U}\vpu{+}_f$ contains a negative norm unit,
    \item \label{it:negunitcondb2} $d_1-3$ is a perfect square,   $j_{\rm{min}}(f)$ is odd, and the integer $\sqrt{\frac{d_{j_{\rm{min}}(f)}+1}{\Delta_0}} = \frac{f_{j_{\rm{min}}(f)}}{\sqrt{d_{j_{\rm{min}}(f)}-3}}$ is divisible by $f$.
\end{enumerate}
In that case,
\eag{
\mcl{U}_{f} &= \{\pm \un^{nj_{\rm{min}}(f)}\colon n\in \mbb{Z}\}.
}
\end{thm}
\begin{proof}
    \eqref{it:negunitcondb1} $\implies$ \eqref{it:negunitcondb2}. The fact that $\mcl{U}_f$ contains a negative norm unit means $\un$ must be negative norm.   Theorem~\ref{thm:negunitcondA} then implies that  $d_1-3$ is a perfect square. We may assume, without loss of generality, that the negative norm unit is greater than 1, and therefore equal to $\un^j$ for some odd positive integer $j$. Since $\vn^j= \un^{2j} \in \mcl{U}^{+}_f$ we must have $j=\ell j_{\rm{min}}(f)$ for some $\ell \in \mbb{N}$. The fact that $j$ is odd means $\ell$ and $j_{\rm{min}}(f)$ are both odd. In particular $\ell = 2p+1$ for some non-negative integer $p$.   Since 
    \eag{
    \un^{2pj_{\rm{min}}(f)} = \vn^{pj_{\rm{min}}(f)}\in \mcl{U}_f
    }
   this means $\un^{j_{\rm{min}}(f)}\in \mcl{U}_f$.
It follows from Theorem~\ref{thm:negunitcondA} that
\eag{
d_{j_{\rm{min}}(f)}-3 &= m_1^2,
\label{eq:djminfMinus3}
\\
d_{j_{\rm{min}}(f)}+1 &= m_2^2 \Delta_0,
\label{eq:djminfPlus3}
}
\eag{
f_{j_{\rm{min}}(f)}^2 &= \frac{(d_{j_{\rm{min}}(f)}-3)(d_{j_{\rm{min}}(f)}+1)}{\Delta_0} = m_1^2m_2^2,
}
and
\eag{
\un^{j_{\rm{min}}(f)}&= \frac{m_1+m_2\Delta_0}{2} = \frac{m_1-m_2\Delta_0}{2}+m_2\left(\frac{\Delta_0+\sqrt{\Delta_0}}{2}\right),
\label{eq:uPowerjminfExpression}
}
for some pair of positive integers $m_1$, $m_2$.  The fact that $\un^{j_{\rm{min}}(f)}\in \mcl{U}_f\subseteq \mcl{O}_f$ then implies that $f$ is a divisor of $m_2 =f_{j_{\rm{min}}(f)}/m_1$.

\eqref{it:negunitcondb2} $\implies$ \eqref{it:negunitcondb1}.  
The fact that $d_1-3$ is a perfect square means, in view of \Cref{thm:negunitcondA}, that $\un$ is negative norm.  
Since $j_{\rm{min}}(f)$ is odd  $\un^{j_{\rm{min}}(f)}$ is also negative norm.  
It follows from Theorem~\ref{thm:negunitcondA} that $\un^{j_{\rm{min}}(f)}$ can be written in the form of \eqref{eq:uPowerjminfExpression}, with $m_1$, $m_2$ as given by \eqref{eq:djminfMinus3} and \eqref{eq:djminfPlus3}. 
Since $f\div m_2$, it follows that $\un^{j_{\rm{min}}(f)}$ is in $\mcl{O}_f$, and therefore in $\mcl{U}_f$.

    To prove the last statement, suppose $\mcl{U}_f$ contains a negative norm unit.  Then it follows from the argument  above that $\un^{j_{\rm{min}}(f)}\in \mcl{U}_f$.  So
    \eag{
        \{\pm \un^{nj_{\rm{min}}(f)}\colon n\in \mbb{Z}\} \subseteq \mcl{U}_f.
    }
   Conversely, let $w$ be any element of $\mcl{U}_f$.   Without loss of generality we may assume  $w>1$.  Then $w=\un^\ell $ for some $\ell \in \mbb{N}$.  Since $\vn^\ell  = w^2\in \mcl{U}^{+}_f$, we must have $\ell =nj_{\rm{min}}(f)$ for some $n\in \mbb{N}$.  So $w\in \{\pm \un^{nj_{\rm{min}}(f)}\colon n\in \mbb{Z}\}$.
   \end{proof}
   \begin{defn}[fundamental unit of an order]\label{df:ufvf}
        Given a real quadratic field $K$ and positive integer $f$, define 
        \eag{
         \un_f &=\begin{cases}
                    \vn^{j_{\rm{min}}(f)} \quad & \text{if } \mcl{U}_f = \mcl{U}^{+}_f,
                    \\
                    \un^{j_{\rm{min}}(f)} \quad & \text{if } \mcl{U}_f \supsetneq \mcl{U}^{+}_f.
         \end{cases}
        }
   \end{defn}
   \begin{cor}\label{cr:UfGroupGenerator}
       For all $f\in \mbb{N}$
       \eag{
       \mcl{U}_f &= \{\pm \un_f^n\colon n\in \mbb{N}\}.
       }
   \end{cor}
   \begin{proof}
       Immediate consequence of Theorem~\ref{thm:negunitcondb}.
   \end{proof}
   
\subsection{Dimension grid}
So far we have been focusing on the dimension tower $d_1, d_2, \dots$. 
In the sequel we will show that, if the Stark Conjecture and Twisted Convolution Conjecture are correct, these are the dimensions in which there exist $1$-SICs with base field $K$. 
By contrast, it will turn out that $r$-SICs with base field $K$ and $r>1$ are found in dimensions $d_{j,m}$, with $m>1$. 
The purpose of this subsection is to establish some properties of these dimensions which will be needed in the sequel.

\begin{prop}\label{prop:rdtrmstheta}
    Let $\theta=\log \vn$.  Then the following equalities hold for all $j,m\in \mbb{N}$.
    \eag{
        r_{j,m}&=\frac{\sinh mj\theta}{\sinh j\theta}
        \\
        d_{j,m} &= \frac{\sinh \frac{(2m+1)j\theta}{2}}{\sinh\frac{j\theta}{2}}
        \\
        d_j &= 1+2\cosh j \theta
    }
\end{prop}   
\begin{proof}
    Straightforward consequence of the definitions.
\end{proof}

\begin{prop}\label{prop:rjmdjmord}
    Let $K$ be a real quadratic field. 
    Then the following holds for all $j\in \mbb{N}$.
    \eag{
         1&= r_{j,1}< r_{j,2} < \cdots  
         \\
         3&< d_{j,1} < d_{j,2} < \cdots
         \label{eq:djmincreasing}
    }
\end{prop}
\begin{proof}
    Straightforward consequence of Proposition~\ref{prop:rdtrmstheta}.
\end{proof}

We now characterize admissible pairs $(d,r)$ (see \Cref{dfn:admissiblePair}), which come from integer solutions $(d,n,r)$ to the three-variable Diophantine equation $nr(d-r) = d^2-1$, in terms of rank and dimension grids.
\begin{thm}\label{thm:nrddjmrjm}
This theorem has two parts.
\begin{itemize}
    \item[(A)]
    Let $K$ be a real quadratic field, and let  $r_{j,m}$, $d_{j,m}$ be the associated rank and dimension grids. 
Then for all $j,m \in \N$,
\begin{enumerate}
    \item \label{it:rjminequality} $0<r_{j,m}<\frac{d_{j,m}-1}{2}$, and
    \item \label{it:rjmdjmequality} $(d_j+1)r_{j,m}(d_{j,m}-r_{j,m}) = d^2_{j,m}-1$.
\end{enumerate}
    \item[(B)]
    Conversely, let $r$, $d$ be any pair of integers such that
\begin{enumerate}
    \item $0<r<\frac{d-1}{2}$, and
    \item $nr(d-r) = d^2-1$ for some integer $n>4$.
\end{enumerate}
Then there exists a unique real quadratic field $K$ and unique natural numbers $j,m$ such that $r=r_{j,m}$, $d=d_{j,m}$, and $n=d_j+1$ where $r_{j,m}$, $d_{j,m}$ are the rank and dimension grids associated to $K$.
\end{itemize}
\end{thm}
\begin{proof}
    \emph{Part (A).} To prove \cref{it:rjminequality}, let $\theta = \log \vn$. Then it follows from \Cref{prop:rdtrmstheta} that 
    \eag{
    \frac{d_{j,m}-1}{2}-r_{j,m}&=\frac{\sinh \frac{(2m+1)j\theta}{2}}{2\sinh\frac{j\theta}{2}}-\frac{1}{2} - \frac{\sinh m j\theta}{\sinh j\theta}
    \nn
    &=\frac{ \sinh mj\theta \cosh\frac{j\theta}{2} +\cosh mj\theta \sinh\frac{j\theta}{2}}{2\sinh \frac{j\theta}{2}}-\frac{1}{2}- \frac{\sinh m j\theta}{\sinh j\theta}
    \nn
    &= \frac{\sinh m j\theta}{\sinh j\theta}\left(\cosh^2\frac{j\theta}{2}-1\right) + \frac{1}{2}\left(\cosh mj\theta-1\right)
    \nn
    &>0.
    }
    The fact that $0<r_{j,m}$ is immediate.  To prove \cref{it:rjmdjmequality}, observe \Cref{prop:rdtrmstheta} also implies
    \eag{
    (d_j+1)r_{j,m}(d_{j,m}-1) 
&=(d_j+1)r_{j,m}r_{j,m+1}
\nn
&= \frac{2(1+\cosh j\theta)\sinh m j\theta \sinh (m+1)j\theta}{\sinh^2j\theta}
\nn
&= \frac{\sinh m j\theta \sinh (m+1)j\theta}{\sinh^2\frac{j\theta}{2}}
\nn
&= \frac{\cosh(2m+1)j\theta - \cosh j\theta}{2\sinh^2\frac{j\theta}{2}}
\nn
&=\frac{\sinh^2\frac{(2m+1)j\theta}{2} - \sinh^2\frac{j\theta}{2}}{\sinh^2\frac{j\theta}{2}}
\nn
&=d^2_{j,m}-1.
    }

\emph{Part (B).} The fact that $nr(d-r) = d^2-1$ implies
\eag{
d &= \frac{nr\pm \sqrt{r^2n(n-4)+4}}{2}
}
The fact that $n>4$ means $K = \mbb{Q}\!\left(\sqrt{n(n-4)}\right)$ is a real quadratic field. Let $\Delta_0$ be its discriminant, and $f_j$, $d_j$, $d_{j,m}$,  $r_{j,m}$ the associated conductors, dimensions, and ranks. 
 It then follows from Theorem~\ref{thm:dimtowunique} that $n=d_j+1$ for some $j$.  So
\eag{
d&=\frac{r(d_j+1)\pm \sqrt{r^2 f_j^2\Delta_0 +4}}{2}
}
Since $d$ is an integer we must have
\eag{
r^2 f_j^2\Delta_0 +4 &= p^2
}
for some $p\in \mbb{N}$.  Let 
\eag{
w_{\pm} &= \frac{p\pm r f_j\sqrt{\Delta_0}}{2}.
}
Then $w_{\pm}$ are  positive norm units in $K$, with $0< w_{-} < 1 < w_{+}$. So $w_{\pm} = \vn^{\pm k}$ for some positive integer $k$.  We then have
\eag{
p &= \vn^k + \vn^{-k} = d_k -1
}
and
\begin{alignat}{2}
&& \qquad rf_j\sqrt{\Delta_0} &= \vn^k-\vn^{-k}
\nn
&\implies & \qquad rf_j &= f_k
\end{alignat}
In view of \Cref{thm:UnitsConductorsIndicesProperties}, this means $k = jm$ for some $m\in \mbb{N}$.  Consequently,
\eag{
r&= r_{j,m}
}
and
\eag{
d &= \frac{r_{j,m}(d_j+1)\pm (d_{jm}-1)}{2}.
}
Setting $\theta = \log \vn$ and using Proposition~\ref{prop:rdtrmstheta} the second expression becomes
\eag{
d&=\frac{\sinh mj\theta\left(1+\cosh j\theta\right)}{\sinh j\theta} \pm \cosh mj\theta
\nn
&= \frac{\sinh mj\theta + \sinh (m\pm 1) j\theta}{\sinh j\theta} 
\nn
&=r_{j,m} + r_{j,m\pm 1}.
}
The fact that $d\ge 2r$, together with Lemma~\ref{prop:rjmdjmord}, means we must in fact have
\eag{
d&=r_{j,m}+r_{j,m+1} = d_{j,m}.
}

It remains to prove uniqueness.  Suppose $r=r'_{j',m'}$ and $d=d'_{j',m'}$, where $r'_{j',m'}$ and $d'_{j',m'}$ are in the rank and dimension towers associated to some other real quadratic field $K'$. Then it follows from Part (A) that
\eag{
(d'_{j'}+1)r (d-r) &=(d'_j+1)r'_{j',m'}(d'_{j',m'}-r'_{j',m'}) 
\nn
&=d^{'2}_{j',m'}-1
\nn
&=d^2_{j,m}-1
\nn
&= (d_j+1)r_{j,m}(d_{j,m}-r_{j,m}) 
\nn
&= (d_j+1)r(d-r),
}
implying $d'_{j'}=d_j$.  Hence
\eag{
K'&=\mbb{Q}\left(\sqrt{(d'_{j'}+1)(d'_{j'}-3)}\right)=\mbb{Q}\left(\sqrt{(d_j+1)(d_j-3)}\right) = K
}
and $j'=j$.  Consequently $d_{j,m'} = d_{j,m}$ which, together with the fact that $d_{j,1},d_{j,2},\dots$ is a monotonically increasing sequence (see \Cref{prop:rjmdjmord}), implies $m'=m$.
\end{proof}
This result implies the following expression for $\ghostOverlapC{t}{\mbf{p}}$, as an alternative to \eqref{eq:CandidateGhostOverlapDefinition}.
\begin{cor}\label{cor:AlternativeDefinitionOfCandidateGhostOverlaps}
    Let $t=(d,r,Q)\sim(K,j,m,Q)$ be an admissible tuple.  Then
    \eag{
\ghostOverlapC{t}{\mbf{p}}
&= \frac{1}{\sqrt{d_j+1}} \, \normalizedGhostOverlapC{t}{\mbf{p}},
&
  \overlapC{s}{\mbf{p}}&= \frac{1}{\sqrt{d_j+1}}\normalizedOverlapC{s}{\mbf{p}}
    }
for all $\mbf{p}\not\equiv \zero\Mod{d}$.
\end{cor}
\begin{proof}
    Immediate consequence of
        \eag{
      (d_j+1)r_{j,m}(d_{j,m}-r_{j,m})& = d^2_{j,m}-1
      \label{eq:djrjmdjmid}
    }
    from \Cref{thm:nrddjmrjm}(A)(2).
\end{proof}
\begin{cor}\label{cor:djmcprjm}
    $d_{j,m}$ is coprime to $r_{j,m}$ for  all $j,m\in\mbb{N}$.
\end{cor}
\begin{proof}
    Immediate consequence of \eqref{eq:djrjmdjmid}.
\end{proof}
The next lemma collects together various technical results which will be needed in the sequel.
\begin{lem} \label{lem:dimgridtechres}
    The following inequalities and equations hold for all $j, m \in \mbb{N}$.
    \eag{
       r_{j,m} & < r_{j,m+1}-1
    \\
       r_{j,m}&< \frac{d_{j,m}-1}{2}
    \\
       r_{j,m+1}-r_{j,m}&= \sqrt{\frac{d_{(2m+1)j}+1}{d_j+1}}
    \\
        r^2_{j,m+1}-r^2_{j,m}&=r\vpu{2}_{j,2m+1}
    \\
     \vn^{(2m+1)j}-1 &= d_{j,m} \vn^{mj} (\vn^j-1) \label{eq:epowerminusone}
     \\
     d_{j,m+2}&= (d_j-1)d_{j,m+1}-d_{j,m}
    }
    If $d_j$ is even, then
    \begin{enumerate}
        \item $r_{j,m}$ is even if and only if $m \equiv 0 \Mod{3}$;
        \item $d_{j,m}$ is even if and only if $m \equiv 1 \Mod{3}$.
    \end{enumerate}
    If $d_j$ is odd, then
    \begin{enumerate}
        \item $r_{j,m}$ is odd if and only if $m$ is odd;
        \item $d_{j,m}$ is odd for all $m$.
    \end{enumerate}
    If $d_{j,m}$ is even, then the following congruences hold.
    \eag{
      d_{(m+1)j}-d_{mj} &\equiv d_j \ \Mod{4}
      \\
      d_{j,m}&\equiv d_j \ \Mod{4}
      \\
      r_{j,m+1}-r_{j,m}&\equiv d_j+2 \ \Mod{4}
      \\
      f_{j(m+1)}-f_{jm}&\equiv d_j+2 \ \Mod{4}
    }
\end{lem}
\begin{proof}
    It follows from Proposition~\ref{prop:rdtrmstheta} that
    \eag{
      r_{j,m+1} &=\frac{\sinh j(m+1)\theta}{\sinh j\theta} 
      =\frac{\sinh jm\theta\cosh j\theta}{\sinh j\theta} + \cosh jm\theta >r_{j,m}+1,
      }
      \eag{
      \frac{d_{j,m}-1}{2} & = \frac{r_{j,m+1}+r_{j,m}-1}{2}>r_{j,m},
\\
      r_{j,m+1}-r_{j,m} &= \frac{\sinh j(m+1)\theta - \sinh j m\theta}{\sinh j\theta} 
      \nn
      & =
      \frac{\cosh \frac{(2m+1)j\theta}{2} }{\cosh \frac{j\theta}{2}}
      \nn
      & =\sqrt{\frac{\cosh (2m+1)j\theta+1}{\cosh j\theta +1}}
      \nn
      &=\sqrt{\frac{d_{(2m+1)j}+1}{d_j+1}},
\\
        r^2_{j,m+1}-r^2_{j,m}&=
        \frac{\sinh^2(m+1)j\theta-\sinh^2mj\theta}{\sinh^2j\theta}
        \nn
        &=\frac{\cosh2(m+1)j\theta-\cosh 2mj\theta}{2\sinh^2j\theta}
        \nn
        &=\frac{\cosh2mj\theta\cosh 2j\theta+\sinh 2mj\theta \sinh 2j\theta-\cosh 2mj\theta}{2\sinh^2j\theta}
        \nn
        &=\frac{\cosh2mj\theta\sinh^2 j\theta+\sinh2mj\theta\sinh j\theta \cosh j\theta}{\sinh^2j\theta}
        \nn
        &= \frac{\sinh(2m+1)j\theta}{\sinh j\theta}
        \nn
        &=r_{j,2m+1},
\\
        d_{j,m} \vn^{mj}(\vn^j-1) &= \frac{\sinh\frac{(2m+1)j\theta}{2} e^{mj\theta} \left(e^{j\theta}-1\right)}{\sinh\frac{j\theta}{2}}
        \nn
        &= e^{(2m+1)j\theta}-1
        \nn
        &=  \vn^{(2m+1)j}-1.
    }

It follows from Theorem~\ref{thm:dnjfnjchebyexpns} and Lemma~\ref{lm:tustarprops} that
\eag{
d_{j,m+2}&= r_{j,m+3}+r_{j,m+2}
\nn
&= U^{*}_{m+3}(d_j)+U^{*}_{m+2}(d_j)
\nn
&=(d_j-1)\left(U^{*}_{m+2}(d_j)+U^{*}_{m+1}(d_j)\right) - \left(U^{*}_{m+1}(d_j)+U^{*}_m(d_j)\right)
\nn
&= (d_j-1)d_{j,m+1}-d_{j,m}.
}
Suppose $d_j$ is even. Then ~\eqref{eq:fkrCheb} and~\eqref{eq:ucheb} imply 
\eag{
r_{j,m}&= 
U^{*}_m(d_j) = \begin{cases} 
                0 \ & \text{if } m \equiv 0 \Mod{3},
                \\
                1 \ & \text{otherwise},
                \end{cases}
}
from which it follows that $d_{j,m}=r_{j,m}+r_{j,m+1}$ is even if and only if $m \equiv 1 \Mod{3}$.

Suppose $d_j$ is odd.  Then 
~\eqref{eq:ustrec} implies $U^{*}_{j,1}$ is odd, $U^{*}_{j,2}$ is even, and $U^{*}_{j,m} \equiv U^{*}_{j,m-2} \Mod{2}$ for all $m>2$.  Consequently $r_{j,m}=U^{*}_{m}(d_j)$ is odd if and only $m$ is odd, and $d_{j,m}=U^{*}_{m+1}(d_j)+U^{*}_{m}(d_j)$ is odd for all $m$.

Suppose $d_{j,m}$ is even.  Then it follows from results just proved that $d_j$ is even and $m \equiv 1 \Mod{3}$. Equation ~\eqref{eq:tcheb} then implies
\eag{
d_{(m+1)j}-d_{mj} = T^{*}_{m+1}(d_j) - T^{*}_{m}(d_j)
&\equiv 
-(2m+1)d_j \ \Mod{d_j^2} \\
&\equiv 
d_j \ \Mod{4},
}
while 
~\eqref{eq:ucheb} implies
\eag{
d_{j,m}&=r_{j,m+1}+r_{j,m}
\nn
&= 
U^{*}_{m+1}(d_j)+U^{*}_{m}(d_j)
\nn
&\equiv \frac{(2m+1)d_j}{3} \Mod{d_j^2}
\nn
&\equiv d_j \ \Mod{4},
\\
r_{j,m+1}-r_{j,m} &= U^{*}_{m+1}(d_j)-U^{*}_{m}(d_j) 
\nn
&\equiv -2 + d_j \Mod{d_j^2}
\nn
&\equiv
d_j+2 \ \Mod{4}.
}
Finally, it follows from  Lemma~\ref{lem:dfdelprops} that $f_j$ is odd, and thus
\eag{
f_{(m+1)j}-f_{mj} &= f_j \left(r_{j,m+1}-r_{j,m}\right) \equiv d_j+2 \Mod{4},
}
completing the proof.
\end{proof}
\begin{lem}\label{lem:rjmm1expn}
    Let $r^{-1}_{j,m}$ be the multiplicative inverse of $r_{j,m}$ modulo $\db_{j,m}$.  Then
    \eag{
       r^{-1}_{j,m}
       &\equiv 
       r\vpu{-1}_{j,m}\left(1+d_j+d_{j,m}\right) \Mod{\db_{j,m}}.
    }
\end{lem}
\begin{proof}
It follows from Theorem~\ref{thm:nrddjmrjm} that
\eag{
(d_{j}+1)r_{j,m}(r_{j,m}-d_{j,m}) \equiv 1-d_{j,m}^2 &\equiv 1 \ \Mod{\db_{j,m}}.
\label{eq:rjmm1inter1}
}
    If $d_{j,m}$ is even, then it follows from Corollary~\ref{cor:djmcprjm} and Lemma~\ref{lem:dimgridtechres} that $d_j+1$ and $r_{j,m}$ are both odd.  So \eqref{eq:rjmm1inter1} implies
    \eag{\label{eq:rSquaredIdentity}
    (d_j+1)r_{j,m}^2+d_{j,m} \equiv 1 \ \Mod{\db_{j,m}}.
    }
    Consequently
    \eag{
    r_{j,m}^{-1}&\equiv (d_j+1)r_{j,m}+ d_{j,m} r^{-1}_{j,m}= (d_j+1)r_{j,m}+ d_{j,m} \ \Mod{\db_{j,m}}.
    }
    In view of \Cref{cor:djmcprjm} we can alternatively write
    \eag{
     r_{j,m}^{-1}&\equiv r_{j,m}(1+d_j+ d_{j,m} ) \ \Mod{\db_{j,m}}.
    }
\end{proof}
\begin{lem}\label{lem:djpm1cprimedjm}
    $d_j\pm 1$ are coprime to $d_{j,m}$ for  all $j,m\in\mbb{N}$.
\end{lem}
\begin{proof}
    The fact that $d_j+1$ is coprime to $d_{j,m}$ is an immediate consequence of the relation
    \eag{
     (d_j+1)r_{j,m}(d_{j,m}-r_{j,m}) &= d^2_{j,m}-1
    }
    proved in Theorem~\ref{thm:nrddjmrjm}.

    To prove that $d_j-1$ is coprime to $d_{j,m}$ observe that it follows from  Lemma~\ref{lem:dimgridtechres} that 
    \eag{
    d_{j,m} \equiv -d_{j,m-2} \ \Mod{d_j-1}
    }
    for all $m>2$.  Consequently
    \eag{
     d_{j,m} \equiv \begin{cases} 
     \pm d_{j,1} \ \Mod{(d_j-1)} \quad & \text{if $m$ odd}
     \\
     \pm d_{j,2} \ \Mod{(d_j-1)} \quad & \text{if $m$ even}
     \end{cases}
    }
    Since
    \eag{
     d_{j,1}&\equiv d_j \equiv 1 \ \Mod{d_j-1}
    }
    and
    \eag{
     d_{j,2} &= (d_j-1)^2+(d_j-1)-1 \equiv -1 \ \Mod{d_j-1},
    }
    it follows that $d_{j,m} \equiv \pm 1 \Mod{d_j-1}$, and $d_{j,m}$ is thus coprime to $d_j-1$, for all $j,m \in \N$.
\end{proof}

\subsection{Representations}

An important role in the following is played by what we refer to as \textit{canonical representations} of the real quadratic base field $K$. 
These are faithful $\Q$-algebra representations of $K$ by $2 \times 2$ matrices. All such representations are isomorphic, justifying the term \textit{canonical}. However, up to equality, there are multiple canonical representations, and they are in natural bijective correspondence with the set of forms associated to $K$ (recall the definition of \textit{form} in \Cref{ssec:quadraticfieldsforms}). 
This correspondence in turn leads to a natural correspondence between extended Clifford group orbits of $r$-SICs and equivalence classes of forms.

The purpose of this subsection is to derive the results on canonical representations which will be needed in the sequel. 
We begin with some preliminary definitions and lemmas.
\begin{defn}[matrices; symmetric matrices; trace-zero matrices]\label{dfn:MatrixRings}
    Let $R$ be a commutative ring with identity.  Then
    \begin{enumerate}
        \item $\mcl{M}(R)$ is the ring of $2\times 2$ matrices over $R$,
        \item $\mcl{M}_{\rm{S}}(R)\subseteq \mcl{M}(R)$ is the subring of symmetric matrices,
        \item $\mcl{M}_0(R)\subseteq \mcl{M}(R)$ is the additive subgroup of trace-zero matrices.
    \end{enumerate} 
\end{defn}
\begin{defn}[matrix sub-algebra]\label{dfn:subAlgebra}
    If $R$ is a field and $M_1, \dots, M_n\in \mcl{M}(R)$, then $R\la M_1, \dots, M_n\ra$ denotes the $R$-subalgebra of $\mcl{M}(R)$ generated by the matrices $M_j$.
\end{defn}
\begin{defn}[generators of $\SLtwo{\mbb{Z}}$]\label{dfn:GeneratorsOfSL2Z}
    Let
    \eag{
        S &= \bmt 0 &-1\\ 1 & 0 \emt, & T&= \bmt 1 & 1 \\ 0 & 1\emt
    }
    be the usual generators of $\SLtwo{\mbb{Z}}$.
\end{defn}
\begin{lem}\label{lm:mttchlm}
    Let $R$ be a commutative ring with identity.
    \begin{enumerate}
        \item \label{it:mttchlma} If $M\in \mcl{M}_0(R)$, then
            $M^2 = - \Det(M) I$.
        \item \label{it:mttchlmb} The map $M\mapsto S M$ is a bijection of $\mcl{M}_{\rm{S}}(R)$ onto $\mcl{M}_{0}(R)$.
    \end{enumerate}
    Suppose further that $R$ is a field.
    \begin{enumerate}
        \setcounter{enumi}{2}
        \item \label{it:mttchlmc} If $M,M'\in \mcl{M}_0(R)$ are both non-zero, then $MM' = M'M$ if and only if $M' =\lambda M$ for some non-zero $\lambda \in R$.
        \item \label{it:mttchlmd} If $M\in \mcl{M}_0(R)$, then
        \eag{
            R\la I, M\ra & = \{xI + yM\colon x, y\in R\}.
        }
    \end{enumerate}
\end{lem}
\begin{proof}
Statements~\eqref{it:mttchlma}, \eqref{it:mttchlmb} are immediate consequences of the definitions.  To prove~\eqref{it:mttchlmc} let $M=\smt{\ma & \mb \\ \mc & -\ma}$, $M'=\smt{\ma' & \mb' \\ \mc' & -\ma'}$.  Then
\eag{
MM'-M'M &= \bmt (\mb \mc' - \mc \mb') & 2(\ma \mb'-\mb \ma') \\ 2(\mc \ma' -\ma \mc') & -(\mb\mc'-\mc \mb')\emt,
}
from which the statement  follows.  To prove~\eqref{it:mttchlmd}, let $\mcl{A}= \{ xI + yM\colon x, y\in R\}$.  In view of~\eqref{it:mttchlma}, $\mcl{A}$ is closed under addition and multiplication, and is therefore a sub-algebra which contains $I,M$, and is contained in $R\langle I, M\rangle$.
\end{proof}

We now define canonical representations of $K$ and prove a proposition giving their basic properties. The following definition and proposition would be exactly the same (with minor modifications to the proof) if $K$ were replaced by an arbitrary degree $n$ number field and $\mcl{M}(\Q)$ were replaced by the $\Q$-algebra of $n \times n$ matrices (where the $n$ is the same on both sides).
\begin{defn}[canonical representation]\label{dfn:canonicalRepresentation}
    A \textit{canonical representation} of $K$ is a map $\canrep\colon K \to \mcl{M}(\mbb{Q})$ that is a $\Q$-algebra isomorphism of $K$ onto a sub-algebra of $\mcl{M}(\mbb{Q})$.
\end{defn}

\begin{prop}
    Let $\chi : K \to \mcl{M}(\Q)$ be a canonical representation. Then $\chi$ enjoys the following properties.
    \begin{enumerate}
        \item The map $\chi$ is isomorphic to the ``multiplication map'' representation $\chi_0 : K \to \mcl{L}_\Q(K)$ given by $(\chi_0(\kappa))(\kappa') = \kappa\kappa'$, where $\mcl{L}_\Q(K)$ is the $\Q$-algebra of $\Q$-linear maps from $K$ to $K$.
        \item The map $\chi$ turns norms into determinants, so that $\Nm(\kappa) = \Det(\canrep(\kappa))$ for all $\kappa \in K$.
        \item The map $\chi$ preserves traces in the sense that $\Tr(\kappa)$, the number field trace of $\kappa$, is equal to $\Tr(\canrep(\kappa))$, the matrix trace of $\canrep(\kappa)$, for all $\kappa \in K$.
    \end{enumerate}
\end{prop}
\begin{proof}
    Write $K = \kappa_1\Q + \kappa_2\Q$. Identify $\mcl{L}_\Q(K)$ with $\MM(\Q)$, so that for any $\kappa \in K$, the linear transformation $\chi_0(\kappa)$ is represented by a matrix $M = \smmattwo{m_{11}}{m_{12}}{m_{21}}{m_{22}} \in \MM(\Q)$ such that $\kappa\kappa_i = m_{i1}\kappa_1 + m_{i2}\kappa_2$ for $i \in \{1,2\}$.
    
    Fix some $v_0 \in \MM(\Q)$ such that $v_1 := \chi(\kappa_1)v_0$ and $v_2 := \chi(\kappa_2)v_0$ are linearly independent.
    Then, for any $\kappa \in K$,
    \begin{equation}
        \chi(\kappa)v_i 
        = \chi(\kappa)\chi(\kappa_i)v_0
        = \chi(\kappa\kappa_i)v_0
        = \chi(m_{i1}\kappa_1 + m_{i2}\kappa_2)v_0
        = m_{i1}v_1 + m_{i2}v_2.
    \end{equation}
    Thus, $\chi(\kappa)$ is represented by $M$ in the basis $\{v_1,v_2\}$ of $\Q^2$. So $\chi$ is isomorphic to $\chi_0$ as a $\Q$-algebra representation of $K$; that is, property (1) holds.

    Property (2) and (3) follow, because the norm $\Nm(\kappa)$ is defined to be the determinant of $\chi_0(\kappa)$ (or, equivalently, the determinant of $M$), and the number field trace $\Tr(\kappa)$ is defined to be the trace of $\chi_0(\kappa)$ (or, equivalently, the trace of $M$).
\end{proof}

We will now use the fact that $K$ is a quadratic field to parametrize all canonical representations up to equality by (binary quadratic) forms $Q$.
\begin{defn}[canonical representation associated to a form]\label{df:etaq}
    Given a form $Q$ associated to $K$ with conductor $f$, define $\canrep_Q\colon K\to \mcl{M}(\mbb{Q})$ to be the map specified by
    \eag{
        \canrep_{Q}\!\left(x+y\sqrt{\Delta_0}\right) &=x I + \frac{2y}{f} SQ
    }
    for all $x, y\in \mbb{Q}$.
\end{defn}  
\begin{thm}\label{thm:canrep}
    For every form $Q$ associated to $K$, the map $\canrep_Q$ is a canonical representation of $K$.  Conversely, if $\canrep$ is a canonical representation of $K$, then there exists a form $Q$ associated to $K$ such that $\canrep=\canrep_Q$.
\end{thm}
\begin{proof}
   Let $Q$ be a form associated to $K$ with conductor $f$.  It is immediate that $\canrep_Q$ is a $\mbb{Q}$-linear monomorphism of $K$ into $\mcl{M}(\mbb{Q})$.  To show that  it is a $\mathbb{Q}$-algebra monomorphism, observe that it follows from Lemma~\ref{lm:mttchlm} that
   \eag{
        (SQ)^2 &= -\Det(SQ) I = \frac{1}{4} f^2\Delta_0 I.
   }
   Hence
   \eag{
        \canrep_Q\!\left(x+y\sqrt{\Delta_0}\right)
        \canrep_Q\!\left(x'+y'\sqrt{\Delta_0}\right)
        &= x x'I  + \frac{4yy'}{f^2}(SQ)^2 + \frac{2(x y'+x'y)}{f} SQ
        \nn
        &= \canrep_Q\!\left( \left(x+y\sqrt{\Delta_0}\right)\left(x'+y'\sqrt{\Delta_0}\right)\right)
   }
   for all $x,y,x',y'\in\mbb{Q}$.
So $\canrep_Q$ is a canonical representation of $K$.

Conversely, suppose $\canrep$ is an arbitrary canonical representation of $K$.
Define
\eag{
L_1 &= \canrep(1), & L_2 &= \canrep(\sqrt{\Delta_0}).
}
We have $\Tr(L_1) = 2$ and $L_1^2=L_1$.  Consequently $I-L_1 \in \mcl{M}_0$ and $(I-L_1)^2 = I-L_1$.  In view of Lemma~\ref{lm:mttchlm} this means
\eag{
I-L_1 &= -\Det(I-L_1) I.
}
So $L_1 = k I$ for some $k$.  The fact that $\Tr(L_1) =2$ means $k=1$, implying $L_1 = I$.

Turning to $L_2$, the fact that $\Tr(L_2) =\Tr(\Delta_0)$ implies, by another application of Lemma~\ref{lm:mttchlm}, that
\eag{
L_2 &= SM
}
for some  $M\in \mcl{M}_{\rm{S}}(\mbb{Q})$.  We can write $M$ in the form 
\eag{
M &= q Q, & Q&= \bmt a & \frac{b}{2} \\ \frac{b}{2} & c\emt
}
with $q\in \mbb{Q}$ and $a, b, c$ coprime integers. The fact that $\Det(L_2) = \Nm(\sqrt{\Delta_0})$ means
\eag{
\Delta_0 &= \frac{q^2 \Delta}{4}
}
where $\Delta=b^2-4ac$ is the discriminant of $Q$.
Let $D, D'$ be the square-free parts of $\Delta_0, \Delta$ respectively.  Then $n^2 D = m^2 D'$ for some $n,m\in \mbb{Z}$, implying $D' = D$.  So $Q$ is associated to $K$, and  $\Delta=f^2\Delta_0$ where $f$ is the conductor of $Q$. It follows that $q = 2/f$ implying
\eag{
L_2 &= \frac{2}{f} SQ.
}
Hence $\canrep=\canrep_Q$.
\end{proof}
\begin{thm}\label{tm:canisointun}
    Let $Q$ be a form associated to $K$ and  let $f$ be its conductor.  Let $\mcl{O}_f$, $\mcl{U}_f$ be the  order and unit group with conductor $f$  (see  Definition~\ref{df:ofufdef}).  Then $\canrep_Q$ restricts to
    \begin{enumerate}
        \item \label{it:canisointuna} a ring isomorphism of $\mcl{O}_f$ onto $\Z\la I, SQ\ra \cap \mcl{M}(\mbb{Z})$,
        \item \label{it:canisointunb} a group isomorphism of $\mcl{U}_f$ onto 
        $\Z\la I, SQ\ra \cap \GLtwo{\mbb{Z}}$.
    \end{enumerate}
\end{thm}
\begin{proof}
    \emph{Statement~\eqref{it:canisointuna}.}  Let $\kappa \in \mcl{O}_f$ be arbitrary.  Then
    \eag{
        \kappa &= x + yf\left(\frac{\Delta_0+\sqrt{\Delta_0}}{2}\right)
    }
    for some $x,y\in \mbb{Z}$. Setting $Q=\smt{a & \frac{b}{2} \\ \frac{b}{2} & c}$, we have
    \eag{
        \canrep_Q(\kappa) &= \left(x+\frac{yf\Delta_0}{2}\right)I + y SQ = \bmt x+\frac{y\left(f\Delta_0 - b\right)}{2} & -yc \\ y a & x+\frac{y\left(f\Delta_0 +b\right)}{2} \emt.
    }
    The fact that $f^2\Delta_0 = b^2-4ac$ means $f\Delta_0 \equiv b \Mod{2}$.  So $\canrep_Q(\kappa) \in \la I, SQ\ra \cap \mcl{M}(\mbb{Z})$.  

    Conversely, suppose $L \in \Z\la I, SQ\ra\cap \mcl{M}(\mbb{Z}) $.  Then it follows from Lemma~\ref{lm:mttchlm} that
\eag{
L &= x I + y SQ 
}
for some $x$, $y\in \mbb{Q}$. So
\eag{
L&= \canrep_Q(\kappa), & \kappa &= \left(x-\frac{yf\Delta_0}{2}\right) + yf\left(\frac{\Delta_0+\sqrt{\Delta_0}}{2}\right).
}
Moreover, the fact that
\eag{
L = \bmt x-\frac{yb}{2} & -yc \\ ya & x+\frac{yb}{2} \emt \in \mcl{M}(\mbb{Z})
}
means $ ya, yc$ and $yb = (x+yb/2)-(x-yb/2)$ are all in $\mbb{Z}$.  Since $a, b, c$ are coprime, it follows that $y\in \mbb{Z}$.  Also the fact that $x-yb/2 \in \mbb{Z}$, together with the fact that $f\Delta_0 \equiv b \Mod{2}$, means
\eag{
x-\frac{yf\Delta_0}{2} &= x-\frac{yb}{2} -\frac{y\left(f\Delta_0-b\right)}{2} \in \mbb{Z}.
}
So $\kappa \in \mcl{O}_f$.

\emph{Statement~\eqref{it:canisointunb}.}  It follows from Statement~\eqref{it:canisointuna} that $\canrep_Q$ maps $\mcl{U}_f$ into a multiplicative subgroup of $\Z\la I , SQ\ra \cap \mcl{M}(\mbb{Z})$.  The fact that $\Det(\canrep_Q(\kappa)) = \Nm(\kappa)$ for all $\kappa$ means that, if $\kappa \in \mcl{U}_f$, then $\Det(\canrep_Q(\kappa)) = \pm 1$ implying $\canrep_Q(\kappa) \in \Z\la I , SQ\ra \cap \GLtwo{\mbb{Z}}$.  It also means that, if $\canrep_Q(\kappa) \in \Z\la I, SQ\ra \cap \GLtwo{\mbb{Z}}$, then $\Nm(\kappa) = \pm 1$.
Since Statement~\eqref{it:canisointuna} implies $\kappa \in \mcl{O}_f$ we must have $\kappa \in \mcl{U}_f$. 
\end{proof}

\begin{cor}\label{cr:etaqmodn}
    Let $Q$ be a form associated to $K$, and let $f$ be its conductor.  Let $\kappa \in \mcl{O}_f$ and $n\in \mbb{N}$ be arbitrary.  Then $\kappa \in n\mcl{O}_f$ if and only if $\canrep_Q(\kappa) \in n\mcl{M}(\mbb{Z})$.
\end{cor}
\begin{proof}
    Necessity is immediate. Suppose, on the other hand, that $\canrep_Q(\kappa) = nL$ for $L\in \mcl{M}(\mbb{Z})$.  It follows from Theorem~\ref{tm:canisointun} that $nL \in \Z\la I, SQ\ra$.  So $L$ is also in $\Z\la I, Q\ra$.  By another application of Theorem~\ref{tm:canisointun}, $L=\canrep_Q(a)$ for some $a \in \mcl{O}_f$.  Hence $\kappa  = n a \in n\mcl{O}_f$.
\end{proof}

\subsection{\texorpdfstring{Stability groups and maximal abelian subgroups of $\GLtwo{\mbb{Z}}$}{Stability groups and maximal abelian subgroups of GL(2,Z)}}
The purpose of this subsection is to prove some results concerning the stability group of a form, which, as we will see, is  the same thing as a  maximal abelian subgroup of $\GLtwo{\mbb{Z}}$. 

In order to treat arbitrary maximal abelian subgroups of $\GLtwo{\mbb{Z}}$ we need to temporarily relax the restrictions in \Cref{ssec:quadraticfieldsforms} on the definition of \textit{form}. 
We continue to assume without comment that the forms with which we deal are binary, quadratic, integral, and primitive. 
However, in this subsection (and nowhere else) we drop the assumption that they are irreducible and indefinite.

We now investigate the properties of $\mcl{S}(Q)$, the stability group of $Q$ as specified in Definition~\ref{dfn:stbgpqdf}.
\begin{lem}\label{lm:slm1sdllt}
    For all $M\in\GLtwo{\mbb{Z}}$,
    \eag{
       \Det(M) M^{\rm{T}} &= - SM^{-1}S.
    }
\end{lem}  
\begin{proof}
    Write $M=\smt{\ma &\mb \\ \mc & \md}$. Then
    \eag{
        SM^{-1}S &=\frac{1}{\Det M} \bmt 0 & -1 \\ 1 & 0\emt \bmt \md &-\mb \\ -\mc & \ma \emt 
        \bmt 0 & -1 \\ 1 & 0\emt 
        =
       -\frac{1}{\Det M} \bmt \ma & \mc \\ \mb & \md \emt,
    }
    as desired.
\end{proof}
\begin{thm}\label{tm:sqchar}
   Let $Q$ be a form and $\mcl{S}(Q)$ its stability group (as in \Cref{dfn:stbgpqdf}).  Then:
    \begin{enumerate}
        \item \label{it:sqchara} $\mcl{S}(Q)$ is the centralizer of $SQ$ in $\GLtwo{\mbb{Z}}$.
        \item \label{it:sqcharb} $\mcl{S}(Q)= \Z\la I , SQ\ra \cap \GLtwo{\mbb{Z}}$.
        \item \label{it:sqchard} For all $M\in \mcl{S}(Q)$, there exist unique integers $t, n$ such that
        \eag{
                M &= \frac{t}{2}I + n SQ.
        }
        \item \label{it:sqcharc} $\mcl{S}(Q)$ is abelian.
    \end{enumerate}
\end{thm}
\begin{proof}
    \emph{Statement~\eqref{it:sqchara}.} 
    For all $M\in \GLtwo{\mbb{Z}}$,
    \begin{alignat}{2}
     & & \qquad \lOnQ{M}{Q} &= Q
     \\
       &\iff& \qquad \Det(M) M^{\rm{T}} Q M &= Q
       \\
       &\iff& \qquad -SM^{-1}SQM &=Q
       \\
       &\iff & \qquad  SQM&= MSQ,
    \end{alignat}
    where $Q_M$ is as defined in \eqref{eq:afrmtrans} and where we used Lemma~\ref{lm:slm1sdllt} in the second step.
    
    \emph{Statement~\eqref{it:sqcharb}.} Let $M$ be any element of the centralizer of $SQ$, and let
    \eag{
        M_0&=M -  \frac{1}{2}\Tr(M) I.
    }
    Then $M_0$ commutes with $SQ$. Since $M_0$ and $SQ$ are both trace-zero, it follows from Lemma~\ref{lm:mttchlm} that 
    \eag{
        M_0 &= \lambda S Q
    }
    for some non-zero $\lambda \in \mbb{Q}$.  So $M$ belongs to $\Z\la I, SQ\ra$ and therefore to $\Z\la I, SQ\ra \cap \GLtwo{\mbb{Z}}$.  

    Conversely, if $M\in \Z\la I, SQ\ra \cap \GLtwo{\mbb{Z}}$ then it follows from Lemma~\ref{lm:mttchlm} that $M$ is a linear combination of $I$ and $SQ$ and is therefore in the centralizer of $SQ$.

    \emph{Statement~\eqref{it:sqchard}.} Let $Q=\la a, b, c\ra$, and let $M\in \mcl{S}(Q)$. It follows from Statement~\eqref{it:sqcharb} that
    \eag{
      M &= \bmt \frac{t-nb}{2} & -n c \\ n a & \frac{t+nb}{2} \emt
    }
    for some $t, n \in \mbb{Q}$.  We must have $t, na, nb, nc\in\mbb{Z}$.  Since $a, b, c$ are coprime, it follows that $n\in \mbb{Z}$.  Suppose that $t', n'$ are any other pair of integers such that $M = (t'/2)I +n' SQ$.  Then
    \eag{
        t' &= \Tr(M) = t
    }
   and, consequently, $n' Q = n Q$, implying $n' = n$.
   
   \emph{Statement~\eqref{it:sqcharc}.}  Immediate consequence of Statement~\eqref{it:sqcharb}.
\end{proof}

\begin{thm}\label{tm:ltnqdecomp}
For each $M=\smt{\ma & \mb \\ \mc & \md} \in \GLtwo{\mbb{Z}}\setminus \{\pm I\}$ there exists a  unique   form $\qOfL{Q}{M}$ such that $M\in\mcl{S}(\qOfL{Q}{M})$ and $\sgn\left(M\right) = \sgn\left(\qOfL{Q}{M}\right)$. Explicitly,
\eag{
\qOfL{Q}{M} &= 
\frac{1}{n_M}
\bmt \mc &  \frac{\md-\ma}{2} \\ \frac{\md-\ma}{2} & -\mb\emt
\label{eq:QLdf}
}
where
\eag{
n_M &= 
\gcd(\mb,\mc,\ma-\md). %
\label{eq:nLdf}
}
Setting
\eag{
t_M &= \Tr(M),
\label{eq:tLdf}
}
one then has
\eag{
M &= \frac{t_M}{2}I + n_M S\qOfL{Q}{M}.
}
\end{thm}
\begin{rmkb}
    When reading \Cref{tm:ltnqdecomp}, one should keep in mind that, in this subsection, forms are not assumed to be irreducible or indefinite.
\end{rmkb}
\begin{proof}
Let $\qOfL{Q}{M}, n_M,t_M$ be as defined in ~\eqref{eq:QLdf}--\eqref{eq:tLdf}. The fact that $M\neq \pm I$ means $n_M$ is non-zero, and $\qOfL{Q}{M}$ is a well-defined  form. We have $\sgn(M) = \sgn(\gamma) = \sgn\left(\qOfL{Q}{M}\right)$. Moreover
\eag{
\frac{t_M}{2} I + n_M S\qOfL{Q}{M}
&=
\frac{\ma+\md}{2}I + \bmt 0 & -1 \\ 1 & 0 \emt \bmt \mc & \frac{\md-\ma}{2} \\ \frac{\md-\ma}{2} & -\mb \emt
=
\bmt \ma & \mb \\ \mc & \md \emt.
}
In view of Theorem~\ref{tm:sqchar}, this means $M\in\mcl{S}\left(\qOfL{Q}{M}\right)$.

Suppose $Q$ is any other form such that $M\in \mcl{S}(Q)$ and $\sgn(Q) = \sgn(M)$. In view of Theorem~\ref{tm:sqchar}, this means
\eag{
M &= \frac{t}{2}I+n SQ
}
for some pair of integers $t,n$ such that $n>0$.
It is immediate that $t = \Tr(M) = t_M$.  Hence  
\eag{
nQ &= n_M \qOfL{Q}{M}.
}
Consequently
\eag{
n a &= n_Ma_M & n b &= n_M b_M & n c &= n_M c_M
}  
where we have set $\qOfL{Q}{M}=\la a_M, b_M, c_M\ra$, $Q=\la a, b, c\ra$.
Since $Q$, $\qOfL{Q}{M}$ are primitive and $n$, $n_M$ are both positive, it follows that $Q=\qOfL{Q}{M}$.
\end{proof}
\begin{cor}
For any pair of distinct forms $Q_1, Q_2$,
\eag{
\mcl{S}(Q_1)\cap \mcl{S}(Q_2) &= \{\pm I\}.
}
\end{cor}
\begin{proof}
Immediate consequence of Theorem~\ref{tm:ltnqdecomp} and the fact that $\pm I\in \mcl{S}(Q)$ for all $Q$.
\end{proof}
\begin{lem}\label{lm:llpcommqcond}
Let $M$, $M'$ be any pair of elements of $\GLtwo{\mbb{Z}}\setminus \{\pm I\}$.  Then $M$ commutes with $M'$ if and only if $\qOfL{Q}{M}=\qOfL{Q}{M'}$, where $\qOfL{Q}{M}, \qOfL{Q}{M'}$ are as defined in \eqref{eq:QLdf}.
\end{lem}
\begin{proof}
Suppose $M$, $M'$ commute.  Then $S\qOfL{Q}{M'} = n^{-1}_{M'} \left(M'-(t_{M'}/2)I\right)$ commutes with $S\qOfL{Q}{M} =n^{-1}_M\left( M-(t_M/2)I\right)$.  In view of Lemma~\ref{lm:mttchlm}, and the fact that $S\qOfL{Q}{M'}$, $S\qOfL{Q}{M}$ are both trace-zero, this means
\eag{
S\qOfL{Q}{M'} &= \lambda S \qOfL{Q}{M}
}
for some non-zero $\lambda \in \mbb{Q}$, implying
\eag{
m' \qOfL{Q}{M'} &= m \qOfL{Q}{M}
}
for some non-zero $m, m' \in \mbb{Z}$. Because the forms are primitive, it follows that $\qOfL{Q}{M'} = \qOfL{Q}{M}$.

The converse  is immediate.
\end{proof}
\begin{defn}\label{dfn:EllipticHyperbolic}Let
\begin{enumerate}
\item $\mcl{H}_{+}$  be the set of all $M\in \GLtwo{\mbb{Z}} \setminus \{\pm I\} $ such that $\left(\Tr M\right)^2 -4\Det M > 4 $,
\item $\mcl{H}_{-}$  be the set of all $M\in \GLtwo{\mbb{Z}} $ such that $\left(\Tr M\right)^2 -4\Det M \le 4 $,
\end{enumerate}
and let
\begin{enumerate}
    \item $\mcl{F}_{+}$  be the set of all  forms which are indefinite and irreducible,
        \item $\mcl{F}_{-}$  be the set of all  forms  whose discriminant equals $-4, -3,  0, 1, $ or $4$.
\end{enumerate}
\end{defn}
\begin{thm}\label{tm:hpmfpmrel}
$M\in \mcl{H}_{\pm}$  if and only if $\qOfL{Q}{M} \in \mcl{F}_{\pm}$.
\end{thm}
\begin{proof}
Write $M = \smt{\ma & \mb \\ \mc & \md}$. Then
\eag{
\left(\Tr M\right)^2-4 \Det M &= (\ma+\md)^2 - 4\ma \md +4\mb \mc
=(\ma-\md)^2 + 4\mb \mc = n_M^2\Delta\big(\qOfL{Q}{M}\big).
\label{eq:trl4dleqn2disq}
}

Now suppose $M\in \mcl{H}_{+}$.  Then  \eqref{eq:trl4dleqn2disq} means $n_M^2\Delta(\qOfL{Q}{M})>4$.  It follows that $\qOfL{Q}{M}$ is indefinite.  It also follows that $\qOfL{Q}{M}$ is irreducible.  Indeed, suppose that were not the case. Then $\Delta(\qOfL{Q}{M}) = m^2$ for some $m\in \mbb{N}$, implying
\eag{
\bigl| \abs{\Tr M} - m n_M\bigr|\big(\abs{\Tr M} + m n_M\big) &= 4
}
Since $\abs{\Tr M} - m n_M \equiv \abs{\Tr M} +m n_M \Mod{2}$, this is only possible if 
\eag{
\bigl| \abs{\Tr M} - m n_M \bigr| = \abs{\Tr M} + m n_M = 2.
}
Since $m n_M >0$, this in turn implies $\Tr M = 0$, which is inconsistent with the assumption $M\in \mcl{H}_{+}$.  Conversely, if $\qOfL{Q}{M}\in \mcl{F}_{+}$, then $\Delta(\qOfL{Q}{M}) >4$, implying $\left(\Tr M\right)^2-4 \Det M>4$. 

Suppose, on the other hand, that $M\in \mcl{H}_{-}$.  Then it follows from \eqref{eq:trl4dleqn2disq} that  
\eag{
n_M^2\Delta(\qOfL{Q}{M}) &=
\begin{cases}
-4 \quad & \text{if } \Tr M = 0,\, \Det M = +1,
\\
-3 \quad & \text{if } \Tr M =\pm 1,\, \Det M = +1, 
\\
0 \quad & \text{if } \Tr M = \pm 2,\, \Det M = +1,
\\
4 \quad & \text{if } \Tr M = 0,\, \Det M = -1.
\end{cases}
}
Since a discriminant must be $0$ or $1$ modulo $4$, it follows that
 $\Delta(\qOfL{Q}{M}) = -4, -3,  0, 1, $ or $4$.  Conversely, if $\qOfL{Q}{M} \in \mcl{F}_{-}$, then it follows from the first part of the proof that $M$ cannot belong $\mcl{H}_{+}$, and so $M$ must belong to $\mcl{H}_{-}$.
\end{proof}
\begin{thm}\label{tm:qfpmst}
Let $Q$ be any form. 
\begin{enumerate}
    \item If $Q\notin \mcl{F}_{+}\cup\mcl{F}_{-}$, then 
    \eag{
\mcl{S}(Q) &= \{\pm I\}.
    }
\item If $Q\in \mcl{F}_{+}$, then
\eag{
\mcl{S}(Q) &= \canrep_Q\left(\mcl{U}_Q\right) \cong \left(\mbb{Z}/2\mbb{Z}\right)\times \mbb{Z}.
}
\item If $Q\in \mcl{F}_{-}$, then 
\eag{
\mcl{S}(Q)
 &= 
\begin{cases}
\la SQ \ra \cong \mbb{Z}/4\mbb{Z} \qquad & \text{if $\Delta(Q) =-4$},
\\
\left< \frac{1}{2}I + SQ \right> \cong\mbb{Z}/6\mbb{Z} \qquad & \text{if $\Delta(Q) = -3$},
\\
\la -I, I+SQ\ra \cong
\left(\mbb{Z}/2\mbb{Z}\right)\times \mbb{Z} \qquad & 
 \text{if $\Delta(Q) = 0$},
\\
\la -I, 2SQ \ra \cong \left(\mbb{Z}/2\mbb{Z}\right)\times \left(\mbb{Z}/2\mbb{Z}\right)  \qquad & \text{if $\Delta(Q) =1$},
\\
\la -I, SQ \ra \cong \left(\mbb{Z}/2\mbb{Z}\right)\times \left(\mbb{Z}/2\mbb{Z}\right)  \qquad & \text{if $\Delta(Q) =4$}.
\end{cases}
}
\end{enumerate}
\end{thm}
\begin{proof}
Suppose $Q\notin \mcl{F}_{+}\cup \mcl{F}_{-}$, and let $M$ be any element of $\mcl{S}(Q)$.  To show that $M=\pm I$, assume the contrary.  It would then follow from Theorem~\ref{tm:ltnqdecomp} that $Q=\qOfL{Q}{M}$,  which is a contradiction since we know from Theorem~\ref{tm:hpmfpmrel} that $\qOfL{Q}{M} \in \mcl{F}_{+}\cup \mcl{F}_{-}$.

 Suppose  $Q\in \mcl{F}_{+}$.  Then it follows from Theorems~\ref{tm:canisointun} and~\ref{tm:sqchar} that 
 \eag{
\mcl{S}(Q) &= \Z\la I, SQ\ra \cap \GLtwo{\mbb{Z}} = \canrep_Q\left(\mcl{U}_Q\right)\cong \left(\mbb{Z}/2\mbb{Z}\right)\times \mbb{Z}.
 }

Suppose  $Q\in \mcl{F}_{-}$, and let  $\pm M \in \mcl{S}(Q)$. Write $Q=\la a, b, c\ra$.  It follows from Theorem~\ref{tm:sqchar} that there exist unique integers $t,n$ such that
\eag{
M &= \frac{t}{2}I + n SQ = \bmt \frac{t-nb}{2} & -nc \\ na 
&\frac{t+nb}{2} \emt,
}
implying
\eag{
t^2 -n^2\Delta(Q) &= 4 \Det M.
\label{eq:tndel4dtlcond}
}

If $\Delta(Q)=0$, then it follows from \eqref{eq:tndel4dtlcond} that $\Det M = +1$ and $t=\pm 2$. It follows from Lemma~\ref{lm:mttchlm} that $(SQ)^2=0$.
Hence
\eag{
\left(I + SQ\right)^n &= I + n SQ
}
for all $n\in \mbb{Z}$.  Multiplying by $-I$ gives the elements with negative trace.  We conclude that $I+SQ$ is infinite order and $\mcl{S}(Q) = \la -I, I+SQ\ra$.

If $\Delta(Q) =-3$, then it follows from \eqref{eq:tndel4dtlcond} that $\Det M = +1$ and that either $M = \pm I$ or $\abs{t}=\abs{n}=1$, implying $\mcl{S}(Q)$ is order 6.  It follows from Lemma~\ref{lm:mttchlm} that $(SQ)^2 =-3/4 I$.
Putting these facts together, we deduce that $(1/2)I + SQ$ is order $6$ and consequently that $\mcl{S}(Q) = \la (1/2)I + SQ \ra$.

If $\Delta(Q) =-4$, then it follows from \eqref{eq:tndel4dtlcond} that $\Det M = +1$ and that either $M=\pm I$ or $t=0, n=\pm 1$, implying $\mcl{S}(Q)$ is order 4. It follows from Lemma~\ref{lm:mttchlm} that $(SQ)^2 =- I$.  Putting these facts together, we deduce that $SQ$ is order $4$ and consequently that $\mcl{S}(Q) = \la SQ\ra$.

If $\Delta(Q) =1$ (respectively $\Delta(Q) = 4$), then it follows from \eqref{eq:tndel4dtlcond} that either $M=\pm I$ or $\Det M = -1$, $t=0$ and  $n=\pm 2$ (respectively $n=\pm 1$), implying $\mcl{S}(Q)$ is order $4$.  It follows from Lemma~\ref{lm:mttchlm} that $(SQ)^2 =(1/4) I$ (respectively $(SQ)^2 = I$).  Putting these facts together, we deduce that $2SQ$ (respectively $SQ$) is order $2$ and consequently that $\mcl{S}(Q) = \la -I, 2SQ\ra$ (respectively $\mcl{S}(Q) = \la -I, SQ\ra$).
\end{proof}

\begin{thm}
Let $\mcl{G}$ be any subgroup of $\GLtwo{\mbb{Z}}$.  Then $\mcl{G}$ is maximal abelian if and only if $\mcl{G}=\mcl{S}(Q)$ for some  $Q\in \mcl{F}_{+}\cup \mcl{F}_{-}$. 
\end{thm}
\begin{proof}
Let $\mcl{G}$ be a maximal abelian subgroup of $\GLtwo{\mbb{Z}}$.  The fact that $\{\pm I\}$ is the centre of $\GLtwo{\mbb{Z}}$ means that $\{\pm I\}$ is properly contained in  $\mcl{G}$.  Choose $M \in \mcl{G}\setminus\{\pm I\}$, and set $Q=\qOfL{Q}{M}$.  It follows from Theorem~\ref{tm:hpmfpmrel} that $Q\in \mcl{F}_{+}\cup \mcl{F}_{-}$. If $M'$ is any other element of $\mcl{G}\setminus\{\pm I\}$, then it follows from Lemma~\ref{lm:llpcommqcond} that $\qsb{M'}=Q$.  So $\mcl{G}\subseteq \mcl{S}(Q)$. Since $\mcl{S}(Q)$ is abelian, it follows that $\mcl{G}= \mcl{S}(Q)$.

Conversely, let $Q\in \mcl{F}_{+}\cup\mcl{F}_{-}$. Then $\mcl{S}(Q)$ is an abelian group.  It follows from Theorem~\ref{tm:qfpmst} that  we can choose $M \in \mcl{S}(Q)$ such that $M\neq \pm I$.   It then follows from Lemma~\ref{lm:llpcommqcond} that if $M' \notin \mcl{S}(Q)$, then $M'$ does not commute with $M$.  So $\mcl{S}(Q)$ is maximal abelian.
\end{proof}

\begin{thm}
Let $M\in \mcl{H}_{+}$ and let 
\eag{
w_M &= \frac{\Tr(M) + \sqrt{\left(\Tr M\right)^2 - 4\Det M}}{2}.
}
Then $w_M$ is a unit in $\mcl{U}_{\qOfL{Q}{M}}$ and 
\eag{
M &= \canrep_{\qOfL{Q}{M}}(w_M).
}
\end{thm}
\begin{proof}
Write $M = \smt{\ma & \mb \\ \mc & \md}$.  Then we have from Theorem~\ref{tm:ltnqdecomp} the formula
\eag{
\qOfL{Q}{M} &= \frac{1}{n_M}\bmt \mc & \frac{\md-\ma}{2} \\ \frac{\md - \ma}{2} & -\mb\emt,
}
implying 
\eag{
\Delta(\qOfL{Q}{M}) & = \frac{\left(\Tr M\right)^2 - 4 \Det M}{n_M^2}.
}
Let $\Delta_0$ and $f$ be the fundamental discriminant and conductor of $Q$.  Then it follows that
\eag{
w_M &= \frac{\Tr M - fn_M\Delta_0}{2}+ fn_M \left(\frac{\Delta_0+ \sqrt{\Delta_0}}{2}\right).
}
The fact that $f^2n^2_M \Delta_0 = \left(\Tr M\right)^2 -4 \Det M$ means $\left(\Tr M-fn_M \Delta_0\right)/2 \in \mbb{Z}$.  Since $\Nm(w_M) = \Det M$, it follows that $w_M$ is a unit in $\mcl{U}_{\qOfL{Q}{M}}$. The fact that $\canrep_{\qOfL{Q}{M}}(w_M) = M$ is immediate.
\end{proof}

\subsection{Additional results}
\label{ssec:additionalresults}
We now revert to the convention explained in \Cref{ssec:quadraticfieldsforms}, according to which a form is always understood to be irreducible and indefinite unless the contrary is explicitly stated.

\begin{lem}
    Let $Q$, $Q'$ be any pair of forms.  Then
    \eag{
        \qrt_{Q,\pm} 
        = \qrt_{Q',\pm} 
        \,\iff\, 
        Q 
        = Q'
    }
\end{lem}
\begin{proof}
    Sufficiency is immediate.  To prove necessity
    write $Q=\la a, b, c\ra$, $Q'=\la a', b', c'\ra$.  Then 
    $\qrt_{Q,\pm} = \qrt_{Q',\pm}$ implies
    \eag{
    \frac{b}{a}&= -(\qrt_{Q,+}+\qrt_{Q,-}) = -(\qrt_{Q',+}+\qrt_{Q',-}) = \frac{b'}{a'},
    }
    \eag{
    \frac{c}{a}&= \qrt_{Q,+}\qrt_{Q,-} = \qrt_{Q',+}\qrt_{Q',-} = \frac{c'}{a'}.
    }
    Also
    \eag{
    \frac{\sqrt{b^2-4ac}}{a} &=\qrt_{Q,+}-\qrt_{Q,-} =\qrt_{Q',+}-\qrt_{Q',-}
    = \frac{\sqrt{b^{\prime 2}-4a'c'}}{a'},
    }
    implying $\sgn(a) = \sgn(a')$. It follows that
    \eag{
        n Q &= n' Q'
    }
    for some pair of positive integers $n$, $n'$.  Since $Q$, $Q'$ are primitive, we must in fact have $n=n'=1$ and $Q=Q'$.
\end{proof}
\begin{lem}\label{lm:fxdpt}
Let $Q$ be an arbitrary form and let $M\in \mcl{S}(Q)$ be such that $M\neq \pm I$. Then for all $\tau\in \mbb{C}$
\eag{
M\cdot \tau 
=\tau 
\,\iff\, 
\tau
=\qrt_{Q,\pm}.
}
\end{lem}
\begin{rmkb}
    It can be shown that, if $M$ and $Q$ have the same (resp. opposite) sign, then $\qrt_{Q,+}$ is the attractive (resp. repulsive) fixed point and $\qrt_{Q,-}$ is the repulsive (resp. attractive) fixed point of $M$.  We omit the proof, since the result is not actually needed for the purposes of this paper.
\end{rmkb}
\begin{proof}
Straightforward consequence of the definitions.
\end{proof}
\begin{lem}\label{lm:lactqrt}
For any form $Q$ and matrix $M\in \GLtwo{\mbb{Z}}$
\eag{\label{eq:lactqrt1}
M^{-1}\cdot \qrt_{Q,\pm} &= \qrt_{\lOnQ{M}{Q},\pm}.
}
In particular 
\eag{\label{eq:lactqrt2}
M\cdot \qrt_{Q,\pm} 
= \qrt_{Q,\pm} 
\,\iff\, 
M
\in \mcl{S}(Q).
}
\end{lem}
\begin{proof}
Write $M=\smt{\ma & \mb \\ \mc & \md}$, $Q=\la a, b, c\ra$, $\lOnQ{M}{Q} = \la a', b', c'\ra$.
Then
\eag{
a'&= (\Det M)\left(a \ma^2+b\ma\mc + c\mc^2\right),
\\
b'&= 2(\Det M) \left(a \ma \mb+b\left(\frac{\ma\md+\mb\mc}{2}\right) + c \mc \md\right),
\\
c'&=(\Det M)\left(a\mb^2+b \mb \md +c \md^2\right).
}
Hence
\eag{
b^{\prime 2}-4a'c' &= b^2 - 4 a c
}
and
\eag{
M^{-1}\cdot \qrt_{Q,\pm} &= \frac{-(2a\mb  + b\md )\pm \md\sqrt{b^2-4ac}}{(2a\ma + b \mc)\mp \mc \sqrt{b^2-4ac}}
=\frac{-b'\pm \sqrt{b^2-4ac}}{2a'} = \qrt_{\lOnQ{M}{Q},\pm},
}
which is \eqref{eq:lactqrt1}. Equation ~\eqref{eq:lactqrt2} follows immediately.
\end{proof}

\begin{defn}[Level of an admissible tuple]\label{dfn:level}
    We define the \emph{level} of the admissible tuple $t=(d,r,Q)\sim(K,j,m,Q)$ to be the integer 
    \eag{
    n_t &= \frac{j}{j_{\rm{min}}(f)},
    }
    where $f$ is the conductor of $Q$.
\end{defn}
\begin{thm}\label{tm:symgp}Let $t=(d,r,Q)\sim(K,j,m,Q)$ be an admissible tuple, and let $f$ be the conductor of $Q$. Then the stabilizers $\stabQGen{t}$, $\posGen{t}$, $A_t$ (see \Cref{dfn:AssociatedStabilizers}) are given by
\eag{
    \stabQGen{t} &= \canrep_Q(\un_f),
    \label{eq:LStabilizerTermsUnit}
    \\
    \posGen{t}&= \canrep_Q(\vn_f),
    \\
    \zaunerGen{t}&=\posGen{t}^{n_t} =\canrep_Q(\vn^j),
    \label{eq:LzStabilizerTermsUnit}
    \\
    A_t &= \posGen{t}^{n_t(2m+1)}=\canrep_Q(\vn^{j(2m+1)}).
    \label{eq:AStabilizerTermsUnit}
}

The stability groups $\mcl{S}(Q)$, $\mcl{S}_d(Q)$ (see \Cref{dfn:stbgpqdf}) are then given by
\eag{
\mcl{S}(Q) &= \la - I, \stabQGen{t}\rangle,
\label{eq:QStabilityGroupGenerators1}
\\
\mcl{S}_d(Q) &= \la A_t\ra.
\label{eq:QStabilityGroupGenerators2}
}
In particular, $\mcl{S}_{d}(Q)$ is a cyclic group, and every element has positive determinant and trace. Moreover, if $d$ is even, then
    \eag{
      A_t &\equiv \left(d+1\right)I \ \Mod{2d}.
      \label{eq:AtMod2d}
    }
\end{thm}
\begin{rmkb}
    The fact that $\mcl{S}_{d}(Q)$ contains no negative determinant matrices might at first sight seem to be an artifact of our decision to define $\Gamma(d)$ to be a subgroup of $\SLtwo{\mbb{Z}}$.  However, it is easily seen that, if $d>2$, then there are no matrices in $\GL_2(\Z)$ but outside $\SLtwo{\mbb{Z}}$ that are congruent to $I$ modulo $d$.  Indeed, suppose $L=I+d\smt{\ma & \mb \\ \mc & \md} $ were such a matrix.  Then we would have $-1 = \det L = 1+d(\ma+\md + d(\ma\md-\mb\mc))$ implying $d\div 2$).
\end{rmkb}
\begin{proof}
To prove ~\eqref{eq:LStabilizerTermsUnit} and~\eqref{eq:QStabilityGroupGenerators1}, observe that it follows from \Cref{tm:qfpmst} and \Cref{cr:UfGroupGenerator} that
\eag{
\mcl{S}(Q) &= \canrep_Q\!\left(\mcl{U}_Q\right) 
= \la -I,\canrep_Q(\un_f)\ra.
}
Consequently, $\stabQGen{t}$ must be one of the four matrices $\pm \canrep_Q(\un_f)$, $\pm \canrep_Q(\un^{-1}_f)$.  The requirement that $\Tr(\stabQGen{t})$ is positive and $\sgn(\stabQGen{t}) = \sgn(Q)$ means we must in fact have $\stabQGen{t} = \canrep_Q(\un_f)$.

If $\Nm(\un_f) = 1$, then $\vn_f = \un_f$ and $\Det(\stabQGen{t}) = +1$, implying  $\posGen{t} = \stabQGen{t} = \canrep_Q(\vn_f)$. If, on the other hand, $\Nm(\un_f) =-1$, then $\vn_f = \un_f^2$ and $\Det(\stabQGen{t}) = -1$, implying $\posGen{t} = L^2_t = \canrep_Q(\un^2_f)$, so that we again have $\posGen{t} = \canrep(\vn_f)$.

\Cref{eq:LzStabilizerTermsUnit} is an immediate consequence of the definition of $\zaunerGen{t}$.

Now consider ~\eqref{eq:AStabilizerTermsUnit} and~\eqref{eq:QStabilityGroupGenerators2}.  The fact that $\mcl{S}_{d}(Q)$ is a subgroup of $\SLtwo{\mbb{Z}}$ means every element is of the form $\pm \posGen{t}^{l}$, where $l\in \mbb{Z}$.
Let $q$ be the multiplicative order of $\posGen{t}$ modulo $d$. We claim that $q = (2m+1)n$. Indeed, Lemma~\ref{lem:dimgridtechres} implies     
\eag{
      \vn_f^{(2m+1)n}-1 &= \vn^{(2m+1)j} -1 = z d
      }
      where $n = j/j_{\rm{min}}(f)$  and $z=\vn^{mj}(\vn^j-1) =\vn_f^{mn}(\vn_f^{n}-1)\in \mcl{O}_f$.      It then follows from Corollary~\ref{cr:etaqmodn} that $\posGen{t}^{(2m+1)n} \equiv I \Mod{d}$.  So $q \div (2m+1)n$.  To show that in fact $q=(2m+1)n$ assume the contrary.  Then $q\le (2m+1)n/2$.  The fact that $\posGen{t}^q \equiv I \Mod{d}$ means,  in view of Corollary~\ref{cr:etaqmodn}, that
     \eag{
       \vn_f^{q}-1 &= z' d
     }
     for some  $z'\in \mcl{O}_f$. 
       It follows from Lemma~\ref{lem:dimgridtechres} 
     \eag{
      d&= \frac{\vn_f^{(2m+1)n}-1}{\vn_f^{mn}(\vn_f^{n}-1)}.
     }
     So
     \eag{
     z' &= \frac{\vn_f^{mn}\left(\vn_f^n-1\right)\left(\vn_f^q-1\right)}{\vn_f^{(2m+1)n}-1}.
     }
     Taking norms on both sides, we deduce
     \eag{
     \Nm(z') &= \frac{\left(2-\vn_f^n-\vn_f^{-n}\right)\left(2-\vn_f^{q}-\vn_f^{-q}\right)}{2-\vn_f^{(2m+1)n}-\vn_f^{-(2m+1)n}}.
     }
     Setting $\vn_f^n = e^{\theta}$ this becomes
     \eag{
     \Nm(z') &= 
     -\frac{4\sinh^2\frac{\theta}{2}\sinh^2\frac{q\theta}{2n}}{\sinh^2\frac{(2m+1)\theta}{2}}
     }
     In view of the assumption that $q\le(2m+1)n/2$, this means
     \eag{
     \abs{\Nm(z')} &\le \frac{4\sinh^4\frac{(2m+1)\theta}{4}}{\sinh^2\frac{(2m+1)\theta}{2}} 
     = \tanh^2\frac{(2m+1)\theta}{4} < 1,
     \label{eq:Nzprimebound}
     }
     contradicting the fact that $\abs{\Nm(z')}$ is a positive integer.  

     We have thus shown that $\la \posGen{t}^{(2m+1)n}\ra \subseteq \mcl{S}_d (Q)$. To show that in fact
     $\la \posGen{t}^{(2m+1)n}\ra =\mcl{S}_d(Q)$, assume the contrary.  Then there would exist a positive integer $s$ such that
     \eag{
        \posGen{t}^{s} &\equiv -I \ \Mod{d}
     }
     It would follow that $\posGen{t}^{2s} \equiv I \Mod{d}$, implying $2s$ is a multiple of $(2m+1) n$.  If it were an even multiple, it would follow that $s$ is a multiple of $(2m+1)n$, implying $\posGen{t}^s \equiv I \Mod{d}$, contrary the assumption.  So $2s = (2k+1)(2m+1)n$ for some non-negative integer $k$.  In particular $n$ would be even. We would then have
     \eag{
      \posGen{t}^{\frac{(2m+1)n}{2}} &= \posGen{t}^{k(2m+1)n} \posGen{t}^{\frac{(2m+1)n}{2}} = \posGen{t}^{\frac{(2k+1)(2m+1)n}{2}} = \posGen{t}^{s} = -I \ \Mod{d}.
     }
     In view of Corollary~\ref{cr:etaqmodn}, this would mean
     \eag{
            \vn_f^{\frac{(2m+1)n}{2}}+1 &= z''d
     }
     for some $z''\in \mcl{O}_f$. 
     By a suitably modified version of the argument leading to inequality~\eqref{eq:Nzprimebound}, it would follow that
     \eag{
     \abs{\Nm(z'')} &= \frac{\sinh^2\frac{\theta}{2}}{\sinh^2\frac{(2m+1)\theta}{4}}
     <1,
     }
     contradicting the fact that $\Nm(z'')$ is a non-zero integer.
We have thus shown
$\mcl{S}_d(Q) = \la \posGen{t}^{(2m+1)n}\ra $, implying $A_t = \posGen{t}^{\pm(2m+1)n}$. The requirement that $\sgn(A_t) = \sgn(Q)$ means we must in fact have $A_t = \posGen{t}^{(2m+1)n}$.  ~\eqref{eq:AStabilizerTermsUnit}, \eqref{eq:QStabilityGroupGenerators2} now follow.

     The fact that the elements of $\mcl{S}_{d}(Q)$ all have positive determinant  follows from
     \eag{
     \Det\left(\posGen{t}^k\right) &= \Nm(\vn_f^k) = 1.
     }
     The fact that they all have positive trace  follows from
     \eag{
     \Tr\left(\posGen{t}^{k}\right)&= \Tr(\vn_f^{k}) =
     \begin{cases}
         d_{|k|j_{\rm{min}(f)}}-1 \quad & \text{ if} k\neq 0,
         \\
         2 \quad & \text{ if} k=0.
     \end{cases}
     }

     It remains to prove ~\eqref{eq:AtMod2d}.  The fact that $A_t = \canrep_Q\left(\vn_f^{(2m+1)n}\right)$ together with Corollary~\ref{cr:etaqmodn} means
     \eag{
     A_t-(d+1)I &\equiv 0 \ \Mod{2d}
     }
     if and only if 
     \eag{
     z &= \frac{\vn_f^{(2m+1)n}-d-1}{2d} \in \mcl{O}_f.
     }
     The fact that $d$ is even means, in view of Lemmas~\ref{lem:dfdelprops} and~\ref{lem:dimgridtechres}, that $\Delta_0 \equiv 1 \Mod{4}$.  So we need to show
     \eag{
            z & = n_1+n_2f\left(\frac{1+\sqrt{\Delta_0}}{2}\right)
     }
     for some $n_1,n_2\in \mbb{Z}$.  It follows from Lemma~\ref{lem:dimgridtechres} that
     \eag{
      z &= \frac{\vn^{mj}(\vn^j-1)-1}{2}
      = \frac{d_{(m+1)j}-d_{mj}-2+(f_{(m+1)j}-f_{mj})\sqrt{\Delta_0}}{4}.
     }
     That is,
     \eag{
         z &= \alpha_1+\alpha_2 f\left(\frac{1+\sqrt{\Delta_0}}{2}\right)
     }
     where
     \eag{
        \alpha_1 &= \frac{\left(d_{(m+1)j}-d_{mj}-2\right)-\left(f_{(m+1)j}-f_{mj}\right)}{4}, & \alpha_2&= \frac{f_{(m+1)j}-f_{mj}}{2f}.
     }
     We need to show that $\alpha_1$, $\alpha_2$ are both in $\mbb{Z}$.  The fact that $\alpha_1\in \mbb{Z}$ is an immediate consequence of Lemma~\ref{lem:dimgridtechres}. Moreover,
     \eag{
       \alpha_2 &= \left(\frac{f_j}{f}\right) \left(\frac{r_{j,m+1}-r_{j,m}}{2}\right).
     }
     It follows from Lemma~\ref{lem:dimgridtechres} that $r_{j,m+1}-r_{j,m}$ is even.  Since $f\div f_j$, this means $\alpha_2\in \mbb{Z}$.
\end{proof}

\begin{proof}[Proof of Theorem~\ref{thm:ghostWellDefinedCondition}]\label{ssc:ghostWellDefinedProof}
It follows from \Cref{tm:symgp} that $\Tr(A\vpu{-1}_t)$ and $\Tr(A^{-1}_t)$ are both positive which, in view of Lemmas \ref{lem:fixedinda} and \ref{lm:lactqrt}, means $\qrt_{Q,\pm}\in \DD_{A\vpu{-1}_t}\cap\DD_{A^{-1}_t}$.
\end{proof}

\section{Proof of main theorems (1): Existence}\label{sec:existence}

In this section, we show that the correctness of our construction of $r$-SICs follows from several number-theoretic conjectures. Specifically, we prove \Cref{thm:ghstExist}, which asserts that for every fiducial datum $s$, the matrix $\GP_s$ defined in \Cref{dfn:CandidateGhostAndSICFiducials} is a ghost $r$-SIC fiducial under the assumption of the Twisted Convolution Conjecture (\Cref{cnj:tci}). Furthermore, we show \Cref{thm:rayclassfieldrsicgen}, which asserts that the matrix $\SP_s$ defined in \Cref{dfn:CandidateGhostAndSICFiducials} is an $r$-SICs fiducial under the assumption of the Twisted Convolution Conjecture and the Stark Conjecture (\Cref{conj:stark}). 

This section first establishes the relationship of the Shintani--Faddeev modular cocycle $\shin^{d^{-1}\p}$ to the Shintani--Faddeev phase $\phi_\p(t)$, then uses that relationship in conjunction with the conjectures to prove the theorems on ghost $r$-SIC and ``live'' $r$-SIC existence.
\Cref{ssec:rademacher} and \Cref{sbsc:phaseprops} proves some properties of the $\SFPhase{t}{\p}$, in order to connect it to the eta-multiplier character $\psi$ and the theta-multiplier character $\chi_{d^{-1}\p}$. In \Cref{sbsc:ghostprops}, we use the properties of the SF phase to establish properties of our candidate ghost overlaps, including the crucial property that they are real numbers. The latter relies of results of \cite{Kopp2020d} relating the SF cocycle to zeta values. In \Cref{sbsc:proofofghosttheorem}, we complete the proof of \Cref{thm:ghstExist}. In \Cref{sbsc:ProofMainTheoremsFurtherRemarks}, we give some remarks on the ``shift'' appearing in our construction. Finally, in \Cref{sbsc:proofofghosttorsic}, we complete the proof of \Cref{thm:rayclassfieldrsicgen}.

\subsection{Properties of the Rademacher invariant}\label{ssec:rademacher}

For the convenience of the reader the following proposition pulls together some properties of the Rademacher class invariant that we require. Aside from a couple of minor additions they are all well-known.
\begin{prop}\label{lem:RademacherProperties}
Let $\rade$ be the Rademacher class invariant as defined in \Cref{df:meyinv}, and let $\eta(\tau)$ be the Dedekind $\eta$-function.
Then for all $M = \smt{\ma & \mb \\ \mc & \md}\in\SLtwo{\mbb{Z}}$, $N\in \GLtwo{\mbb{Z}}$ and $\tau\in \mbb{H}$ we have
\eag{
    \rade(-M) &= \rade(M),
    \label{eq:myml}
    \\
    \rade(M^{-1})& = -\rade(M),
    \label{eq:mypm1} 
    \\
    \rade(NMN^{-1}) &= (\Det N) \rade(M)
    \label{eq:meyjlj}
    }
    \eag{
        \eta(M\cdot\tau)
    &=
    \begin{cases}
        e^{\frac{\pi i}{12}\rade(M)} \sqrt{\sgn(\Tr(M))j_M(\tau)} \eta(\tau) \qquad & \mc \neq 0, \, \Tr(M)\neq 0\,,
        \\
        e^{\frac{\pi i}{12}\rade(M)} \sqrt{-i\sgn(\mc) j_M(\tau)} \eta(\tau) \qquad & \mc \neq 0, \, \Tr(M)=0\,,
        \\
        e^{\frac{\pi i}{12}\rade(M)}  \eta(\tau)  \qquad &\gamma = 0\,.
    \end{cases}
    \label{eq:RademacherInvariantInTermseta}
}
where in \eqref{eq:RademacherInvariantInTermseta} the principal branch of the square root is taken.
If $\Tr(M)\neq \pm 1$, then we also have
\eag{
\psi(M^n) &= n\psi(M)
\label{eq:PsiMPower}
}
for all $n\in \mbb{Z}$.
\end{prop}
\begin{proof}
~\eqref{eq:myml},~\eqref{eq:mypm1} are proved in \cite[Satz 7]{Rademacher1955}
 and \cite[Chapter 4, Section C]{Rademacher:1972}, as is ~\eqref{eq:meyjlj} if $\det N = 1$.  To show that ~\eqref{eq:meyjlj} also holds if $\det N = -1$, observe that it follows from ~\eqref{eq:phiMdf} that $\rade(JMJ^{-1}) = -\rade(M)$, where $J=\smt{1 & 0 \\ 0 & -1}$. 
If $N$ is any other element of $\GLtwo{\mbb{Z}}$ such that $\Det(N) =-1$, let $N' = NJ$. Then applying ~\eqref{eq:meyjlj} to $N'$ we find $\rade(NMN^{-1}) = \rade(N'JMJ^{-1}N^{\prime -1}) = \rade(JMJ^{-1}) = -\rade(M)$. \Cref{eq:RademacherInvariantInTermseta} is a straightforward consequence of results proved in 
~\cite[Chapter 9]{Rademacher:1973}.
Finally, if $|\Tr(M)|> 1$ then ~\eqref{eq:PsiMPower} is proved in \cite[Satz 9]{Rademacher1955}.  It remains to show that it also holds if $\Tr(M)=0$.  To see this observe that if $\Tr(M) =0$ then
$M^2=-I$, implying
\eag{
M^n&= 
\begin{cases}
(-1)^{\frac{n}{2}}I \qquad & \text{$n$ even,}
\\
(-1)^{\frac{n-1}{2}}M \qquad & \text{$n$ odd.}
\end{cases}
}
In particular $M^{-1}=-M$.  In view of  ~\eqref{eq:myml}~\eqref{eq:mypm1} this means $\rade(M) = -\rade(-M^{-1}) = -\rade(M)$, implying $\rade(M) = 0$.  It is immediate that $\rade(I) = 0$. The result now follows.
\end{proof}

We now explicitly relate the metaplectic character $\psi$ to the Rademacher invariant $\Psi$.
\begin{prop}\label{prop:rademacher}
    Let $M 
    = \smmattwo{\ma}{\mb}{\mc}{\md} 
    \in \SL_2(\Z)$ such that $\mc \neq 0$ and $\Tr(M)>0$.
    Then, taking $\sqrt{j_M(\tau)}$ to be the principal branch,
    \begin{equation}
        \psi(M,\sqrt{j_M}) = e^{\frac{\pi i}{12}\Psi(M)}.
    \end{equation}
\end{prop}
\begin{proof}
    By \eqref{eq:etatrans}, for any choice of $\tau \in \HH$, we have
    \begin{equation}
        \eta(M\cdot\tau) = \psi(M,\sqrt{j_M}) \sqrt{j_M(\tau)} \eta(\tau).
    \end{equation}
    By \eqref{eq:RademacherInvariantInTermseta} in \Cref{lem:RademacherProperties}, we also have
    \begin{equation}
        \eta(M\cdot\tau) = e^{\frac{\pi i}{12}\Psi(M)} \sqrt{j_M(\tau)} \eta(\tau).
    \end{equation}
    Therefore,
    \begin{equation}
        \psi(M,\sqrt{j_M})
        = \frac{\eta(M\cdot\tau)}{\sqrt{j_M(\tau)}\eta(\tau)}
        = e^{\frac{\pi i}{12}\Psi(M)},
    \end{equation}
    completing the proof.
\end{proof}

\subsection{Properties of the Shintani--Faddeev phase}\label{sbsc:phaseprops}

The phase $(\psi^{-2}\chi_\r^{-1})(A)$ involves only $A$, whereas the phase $\SFPhase{t}{\p}$ involves both $A_t$ and $Q$. In order the establish a relation, we require some technical lemmas about the quadratic form $Q$.
\begin{lem}\label{lem:fjfqcmpsoddeven}
Let $(K,j,m,Q)$ be an admissible tuple, let $f$ be the conductor of $Q$, and define
    \eag{
     \la a, b, c\ra &= \frac{f_{j}}{f} Q.
    }
\begin{enumerate}
\item If $d_j$ is even, then $a$, $b$, $c$ are all odd.
\item If $d_j\equiv 1 \Mod{4}$, then 
\begin{alignat}{2}
&\text{either} \qquad & b&\equiv 0 \Mod{4}, \ ac \equiv 1  \Mod{2}, 
\nn
&\text{or} \qquad &
b& \equiv 2 \Mod{4}, \ ac \equiv 0 \Mod{2}.
\end{alignat}
\item If $d_j \equiv 3 \Mod{4}$, then 
\begin{alignat}{2}
&\text{either} \qquad & b &\equiv 0 \Mod{4}, \ ac \equiv 0  \Mod{2}, 
\nn
&\text{or} \qquad &
b& \equiv 2  \Mod{4}, \ ac \equiv 1  \Mod{2}.
\end{alignat}
\end{enumerate}
\end{lem}
\begin{proof}
We have
\eag{
  b^2-4ac &= \Delta_j = (d_j-1)^2 -4.
}
It follows that if $d_j$ is even, then
$b$ is odd.  Suppose $ac$ is not also odd.  Then
\eag{
b^2 &\equiv (d_j-1)^2 + 4 \Mod{8},
}
which is impossible, because $n^2 \equiv 1 \Mod{8}$ for every odd integer $n$.

Next, suppose $d_j \equiv 1 \Mod{4}$. 
Then $d_j =4 n+1$ for some integer $n$. 
Consequently
\eag{
b^2-4ac &= 16n^2-4.
}
It follows that $b$ is even, and
\eag{
\left(b/2\right)^2 - a c &= 4n^2 -1.
}
implying one of the pair $(b/2, ac)$ is even and the other odd.

Finally, suppose $d_j \equiv 3 \Mod{4}$. 
Then $d_j = 4n+3$ for some integer $n$. 
Consequently
\eag{
b^2-4ac &= 16n(n+1).
}
It follows that $b$ is even and
\eag{
\left(b/2\right)^2 -ac &= 4n(n+1),
}
implying that the numbers $b/2$, $ac$ are either both even or both odd.
\end{proof}
\begin{cor}\label{cor:fjmfqcmpsoddeven}
    Let $(K,j,m,Q)$ be an admissible tuple, let $f$ be the conductor of $Q$, and let
    \eag{
     \la a, b, c\ra &= \frac{f_{jm}}{f} Q.
    }
    Suppose $d_{j,m}$ is even.  
    Then $a, b, c$ are all odd.
\end{cor}
\begin{proof}
    It follows from \Cref{lem:dimgridtechres} that $d_j$ is even, $m \equiv 1 \Mod{3}$, and $r_{j,m}$ is odd. We have
    \eag{
        \tfrac{f_{jm}}{f} Q &= r_{j,m} \tfrac{f_j}{f}Q
    }
    It follows from
    \Cref{lem:fjfqcmpsoddeven} that the coefficients of $\frac{f_j}{f} Q$ are all odd.  Since $r_{j,m}$ is odd, the same must be true of the coefficients of $\frac{f_{jm}}{f} Q$.
\end{proof}

The next result is the main technical lemma needed to relate the two phases by showing an agreement of signs.
\begin{lem}\label{lem:drafjfqp}
    Let $t=(d,r,Q)\sim (K,j,m,Q)$ be an admissible tuple and let $f$ be the conductor of $Q$.  Then
    \eag{
    (-1)^{\frac{f_j}{f}Q(\mbf{p})}&= (-1)^{1+\delta^{(2d)}_{A_t\mbf{p},\mbf{p}}}%
    }
    for all $\mbf{p}\in \mbb{Z}^2$. (See~\Cref{dfn:AssociatedStabilizers} for the definition of $A_t$).%
\end{lem}
\begin{proof}
    We begin by finding an expression for $A_t-I$. 
 Using \Cref{tm:symgp} together with \Cref{df:etaq},
 \Cref{lem:towerbasic}, and \Cref{dfn:fjrjmdjm},
 we have
    \eag{
     A_t -I &= \canrep_Q(\vn^{j(2m+1)})-I
     \nn
     &=\left(\frac{d_{j(2m+1)}-3}{2}\right)I + \frac{f_{j(2m+1)}}{f} SQ
     \nn
     &=\left(\frac{d_{j(2m+1)}-3}{2}\right)I + r_{j,2m+1} S\bar{Q}
    }
    where
    \eag{\label{eq:barQDefinition}
    \bar{Q}&= \tfrac{f_j}{f}Q.
    }
    It follows from \Cref{prop:rdtrmstheta} that
    \eag{
        \frac{d_{j(2m+1)}-3}{2}&= \cosh(2m+1)j\theta-1
        \nn
        &= 2\sinh^2\frac{(2m+1)j\theta}{2}
        \nn
        &= 2d^2\sinh^2 \frac{j\theta}{2}
        \nn
        &= d^2(\cosh j\theta-1)
        \nn
        &= d^2\left(\frac{d_j-3}{2}\right),
    }
    while Lemma~\ref{lem:dimgridtechres} implies $r_{j,2m+1} = r^2_{j,m+1}-r^2_{j,m}=(r_{j,m+1}-r_{j,m})d$.  Hence
    \eag{
        A_t-I &= dH,
        \label{eq:amidjmh}
    }
    where
    \eag{
    H &= d\left(\frac{d_j-3}{2}\right)I + (r_{j,m+1}-r_{j,m})S\bar{Q}.
    \label{eq:hdjmrjmrjmsbq}
    }
    So the problem is to show that, for all $\mbf{p}$, \eag{
    \bar{Q}(\mbf{p})&\equiv 0 \Mod{2} &&\iff & H\mbf{p} &\equiv \zero \Mod{2}
    }
    Set $\bar{Q}=\la a, b, c\ra$.
    There are four cases to consider.

    \sbhd{Case 1.} $d_{j}$ even, $m \equiv 1 \Mod{3}$. 
    It follows from Lemma~\ref{lem:fjfqcmpsoddeven} that $a$, $b$, $c$ are all odd. 
    So 
    \eag{
     \bar{Q}(\mbf{p}) &\equiv p_1+p_1p_2 + p_2 \Mod{2},
    }
    implying $Q(\mbf{p})  \equiv 0\Mod{2}$ if and only if $p_1\equiv p_2\equiv 0  \Mod{2}$. 
    On the other hand it follows from Lemma~\ref{lem:dimgridtechres} that $d$ is even, which in view of Theorem~\ref{tm:symgp}, means $H \equiv I \Mod{2}$. 
    So $H \mbf{p} \equiv \zero \Mod{2}$ if and only if $p_1=p_2 \equiv 0 \Mod{2}$.

    \sbhd{Case 2.} $d_{j}$ even, $m \not\equiv 1 \Mod{3}$. 
    As before $a, b, c$ all odd, implying $\bar{Q}(\mbf{p}) \equiv 0 \Mod{2}$ if and only if $p_1 \equiv p_2 \equiv 0 \Mod{2}$. 
  On the other hand it follows from Lemma~\ref{lem:dimgridtechres} that $d$ and $r_{j,m+1}-r_{j,m}=d-2r_{j,m}$ are odd. 
  So
  \eag{
  H &\equiv \bmt \frac{d(d_j-3)-(r_{j,m+1}-r_{j,m})b}{2} & 1
  \\
  1 & 
  \frac{d(d_j-3)+(r_{j,m+1}-r_{j,m})b}{2}
  \emt
  \quad \Mod{2}.
  }
  Since 
  \eag{
  \frac{d(d_j-3)-(r_{j,m+1}-r_{j,m})b}{2} + \frac{d(d_j-3)+(r_{j,m+1}-r_{j,m})b}{2} \equiv 1 \quad \Mod{2},
  }
  $H \equiv \smt{1 & 1 \\ 1 & 0}$ or $\smt{0&1 \\ 1 & 1} \Mod{2}$. Consequently $H\mbf{p} \equiv \zero \Mod{2}$ if and only if $p_1 \equiv p_2 \equiv 0\Mod{2}$.

  \sbhd{Case 3.} $d_j \equiv 1 \Mod{4}$. Then $(d_j-3)/2$ is odd.  It follows from Lemma~\ref{lem:dimgridtechres} that $d$ and  $r_{j,m+1}-r_{j,m}=d-2r_{j,m}$ are  also odd.
 So
  \eag{
  H &\equiv I + S\bar{Q} \quad \Mod{2}.
  }
In view of Lemma~\ref{lem:fjfqcmpsoddeven} there are four possibilities:
\begin{enumerate}
    \item $b \equiv 0 \Mod{4}$ and $a \equiv c \equiv 1 \Mod{2}$.  We have
    \eag{
     Q(\mbf{p}) &\equiv p_1 + p_2 \quad \Mod{2}, & H &\equiv \bmt 1 & 1 \\ 1 & 1
     \emt \quad \Mod{2}\,.
    }
    So $\bar{Q}(\mbf{p}) \equiv 0 \Mod{2} \iff p_1 + p_2 \equiv 0 \Mod{2}\iff H\mbf{p} \equiv \zero\Mod{2}$.
   \item $b \equiv 2 \Mod{4}$,  $a \equiv 1 \Mod{2}$, $c \equiv 0\Mod{2}$. 
   \eag{
        \bar{Q}(\mbf{p}) &\equiv p_1 \quad \Mod{2}, & H &= \bmt 0 & 0 \\ 1 & 0
     \emt \quad \Mod{2}.
   }
   So $\bar{Q}(\mbf{p}) \equiv 0 \Mod{2} \iff p_1 \equiv 0 \Mod{2} \iff H\mbf{p} \equiv \zero \Mod{2}$.
   \item $b \equiv 2 \Mod{4}$, $a \equiv 0 \Mod{2}$, $c \equiv 1\Mod{2}$.
   \eag{
        \bar{Q}(\mbf{p}) &\equiv p_2 \quad\Mod{2}, & H &\equiv \bmt 0 & 1 \\ 0 & 0
     \emt \quad \Mod{2}\,.
   }
   So $\bar{Q}(\mbf{p}) \equiv 0 \Mod{2}\iff p_2 \equiv 0 \Mod{2} \iff H\mbf{p} \equiv \zero \Mod{2}$.
   \item $b \equiv 2 \Mod{4}$ and $a \equiv c \equiv 0\Mod{2}$.  We have
       \eag{
     \bar{Q}(\mbf{p}) &\equiv 0 \quad \Mod{2}, & H &\equiv \bmt 0 & 0 \\ 0 & 0
     \emt \quad \Mod{2}.
    }
    So $\bar{Q}(\mbf{p}) \equiv 0 \Mod{2}$ and $H\mbf{p} \equiv \zero\Mod{2}$ for all $\mbf{p}$.
\end{enumerate}

  \sbhd{Case 4.} $d_j \equiv 3 \Mod{4}$. Then $(d_j-3)/2$ is even.  It follows from Lemma~\ref{lem:dimgridtechres} that $d$ and $r_{j,m+1}-r_{j,m}=d-2r_{j,m}$ are odd. 
  So
  \eag{
  H &\equiv S\bar{Q} \quad \Mod{2}\,.
  }
In view of Lemma~\ref{lem:fjfqcmpsoddeven} there are four possibilities:
\begin{enumerate}
    \item $b \equiv 0 \Mod{4}$, $a \equiv 1 \Mod{2}$, $c \equiv 0 \Mod{2}$.  We have
    \eag{
     \bar{Q}(\mbf{p}) &\equiv p_1 \quad \Mod{2}, & H &\equiv \bmt 0 & 0 \\ 1 & 0
     \emt \quad \Mod{2}.
    }
    So $ \bar{Q}(\mbf{p}) \equiv 0 \Mod{2} \iff p_1  = 0 \Mod{2} \iff H\mbf{p} = \zero \Mod{2}$.
   \item $b \equiv 0 \Mod{4}$, $a \equiv 0 \Mod{2}$, $c \equiv 1 \Mod{2}$.
   \eag{
        \bar{Q}(\mbf{p}) &\equiv p_2 \quad \Mod{2}, & H &\equiv \bmt 0 & 1 \\ 0 & 0
     \emt \quad \Mod{2}.
   }
   So $ \bar{Q}(\mbf{p}) \equiv 0 \Mod{2}\iff p_2 \equiv 0 \Mod{2}\iff H\mbf{p} \equiv \zero \Mod{2}$.
   \item $b \equiv 0 \Mod{4}$, $a \equiv c \equiv 0\Mod{2}$.
   \eag{
        \bar{Q}(\mbf{p}) &\equiv 0 \quad \Mod{2}, & H &\equiv \bmt 0 & 0 \\ 0 & 0
     \emt \quad \Mod{2}.
   }
   So $\bar{Q}(\mbf{p}) \equiv 0 \Mod{2}$ and $ H\mbf{p} \equiv \zero \Mod{2}$ for all $\mbf{p}$.
   \item $b \equiv 2 \Mod{4}$ and $a \equiv c \equiv 1 \Mod{2}$. We have
   \eag{
     \bar{Q}(\mbf{p}) &\equiv p_1+p_2  \quad \Mod{2}, & H &\equiv \bmt 1 & 1 \\ 1 & 1
     \emt \quad \Mod{2}.
    }
 So $\bar{Q}(\mbf{p}) \equiv 0 \Mod{2} \iff p_1 + p_2  \equiv 0 \Mod{2}\iff H\mbf{p} \equiv \zero \Mod{2}$.
\end{enumerate}
This completes the proof of \Cref{eq:hdjmrjmrjmsbq} in all four cases and thus proves the lemma.
\end{proof}

The sign calculation we have just finished allows us to establish an equality between the square of the SF phase and a product of eta-multiplier and theta-multiplier values.
\begin{thm}[Phase relation]\label{thm:phaserelation}
    Let $t = (d,r,Q)$ be an admissible tuple, and let $\p \in \Z^2/d\Z^2$. Then,
    \begin{equation}
        \SFPhase{t}{\p}^2 = (\psi^{-2}\chi_{d^{-1}\p}^{-1})(A_t). 
    \end{equation}
\end{thm}
\begin{proof}
    Let $f$ be the conductor of $Q$.
    By \Cref{dfn:SFKPhase}, the square of the SF-phase is
    \begin{equation}\label{eq:sfphasesquare}
        \SFPhase{t}{\p}^2
        = \left((-1)^{s_d(\p)}e^{-\frac{\pi i}{12}\Psi(A_t)}\rtu_d^{-\frac{f_{jm}}{f}Q(\p)}\right)^2
        = e^{-\frac{\pi i}{6}\Psi(A_t)}\rtus_d^{-\frac{f_{jm}}{f}Q(\p)}.
    \end{equation}
    The matrix $A_t = \smmattwo{\ma}{\mb}{\mc}{\md}$ satisfies the conditions
    $\gamma \neq 0$ and $\Tr(A_t)>0$.
    By \Cref{prop:rademacher}, we have
    \begin{equation}\label{eq:gphasesubcalc}
        \psi^{-2}(A_t)
        = \left(e^{\frac{\pi i}{12}\Psi(A_t)}\right)^{-2}
        = e^{-\frac{\pi i}{6}\Psi(A_t)}
    \end{equation}
    It follows from \eqref{eq:amidjmh}, \eqref{eq:hdjmrjmrjmsbq}, and \eqref{eq:barQDefinition} that
    \begin{align}
        \la A_t\mbf{p}, \mbf{p}\ra 
        = -\la \mbf{p}, A_t\mbf{p}\ra 
        &= -\frac{df_j(r_{j,m+1}-r_{j,m})}{f} \la \mbf{p},SQ\mbf{p}\ra 
        \nn
        &= \frac{df_j(d-2r_{j,m})}{f}Q(\mbf{p})
        \nn
        &= \left(\frac{d^2f_j}{f}-\frac{2df_{jm}}{f}\right) Q(\mbf{p}).
    \end{align}
    Hence the character value $\chi_{d^{-1}\p}(A_t)$ can therefore be written as
    \begin{align}
        \chi_{d^{-1}\p}^{-1}(A_t)
        &= (-1)^{1+\delta_{A_t(d^{-1}\p),d^{-1}\p}^{(2)}} e^{\frac{\pi i}{d^2}\langle A_t\p, \p \rangle} \nn
        &= (-1)^{1+\delta_{A_t\p,\p}^{(2d)}} e^{\frac{\pi i}{d^2}\left(\frac{d^2f_j}{f}-\frac{2df_{jm}}{f}\right) Q(\mbf{p})} \nn
        &= (-1)^{1+\delta_{A_t\p,\p}^{(2d)}+\frac{f_j}{f}}\rtus_d^{-\frac{f_{jm}}{f}Q(\p)} \nn
        &= \rtus_d^{-\frac{f_{jm}}{f}Q(\p)}, \label{eq:lphasesubcalc}
    \end{align}
    where we have used \Cref{lem:drafjfqp} in the last step.
    Thus, plugging \eqref{eq:gphasesubcalc} and \eqref{eq:lphasesubcalc} into \eqref{eq:sfphasesquare},
    \begin{equation}
        \phi_\p(t)^2
        = \psi^{-2}(A_t) \chi_{d^{-1}\p}^{-1}(A_t)
        = (\psi^{-2}\chi_{d^{-1}\p}^{-1})(A_t),
    \end{equation}
    completing the proof.
\end{proof}

\subsection{Properties of the ghost overlaps}\label{sbsc:ghostprops}

We now apply the results on the SF phase, which appears in the definition of the candidate ghost overlaps $\normalizedGhostOverlapC{t}{\mbf{p}}$, to establish some relations satisfied by the $\normalizedGhostOverlapC{t}{\mbf{p}}$. 

\begin{lem}\label{lem:nupperiodicity}
Let $t=(d,r,Q)$ be an admissible tuple. Then
  for all $\mbf{p},\mbf{p}'\in \mbb{Z}^2$ such that $\mbf{p}' \equiv \mbf{p} \Mod{d}$ and $\mbf{p}',
  \mbf{p} \not\equiv \zero \Mod{d}$,
    \eag{
    \normalizedGhostOverlapC{t}{\mbf{p}'}
     &= \rtu_d^{\la \mbf{p}',\mbf{p}\ra} \normalizedGhostOverlapC{t}{\mbf{p}}.
    }
\end{lem}
\begin{proof}
We have
\eag{
\SFPhase{t}{\mbf{p}'}&= (-1)^{s_d(\mbf{p}')-s_d(\mbf{p})}\rtu_d^{-\frac{f_{jm}}{f}(Q(\mbf{p}')-Q(\mbf{p}))} 
\SFPhase{t}{\mbf{p}}
\nn
&= (-1)^{s_d(\mbf{p}')-s_d(\mbf{p})} \rtu_d^{-a(p^{\prime 2}_1-p_1^2)-b(p'_1p'_2-p_1p_2)-c(p^{\prime 2}_2-p_2^2)}
\SFPhase{t}{\mbf{p}}
}
where we have set $\frac{f_{jm}}{f}Q=\la a, b, c\ra$. If $d$ is odd, then
\eag{
(-1)^{s_d(\mbf{p}')-s_d(\mbf{p})} \rtu_d^{-a(p^{\prime 2}_1-p_1^2)-b(p'_1p'_2-p_1p_2)-c(p^{\prime 2}_2-p_2^2)} = 1 = \rtu_d^{\la \mbf{p}',\mbf{p}\ra}.
}
Suppose, on the other hand, that $d$ is even.  Then $\mbf{p}' \equiv \mbf{p} \Mod{2}$, implying
\eag{
(-1)^{s_d(\mbf{p}')-s_d(\mbf{p})}  &= 1
}
Also, it follows from Lemma~\ref{cor:fjmfqcmpsoddeven} that $b$ is odd.  So, setting $\mbf{p}' = \mbf{p} + d\mbf{q}$,
\eag{
\rtu_d^{-a(p^{\prime 2}_1-p_1^2)-b(p'_1p'_2-p_1p_2)-c(p^{\prime 2}_2-p_2^2)}
&= \rtu_d^{-ad(p'_1+p_1)q_1-((p_1+dq_1)(p_2+dq_2)-p_1p_2)-cd(p'_2+p_2)q_2}
\nn
&= (-1)^{q_1p_2+q_2p_1}
\nn
&= \rtu_d^{\la \mbf{p}',\mbf{p}\ra}.
}
We conclude that
\eag{
\SFPhase{t}{\mbf{p}'}&= \rtu_d^{\la \mbf{p}',\mbf{p}\ra} 
\SFPhase{t}{\mbf{p}}
}
irrespective of the value of $d$.  Lastly, it follows from \Cref{lm:shinperiodicity} that
\eag{
\sfc{d^{-1} \mbf{p}'}{A_t}{\qrt_{Q,+}} &= 
\sfc{d^{-1} \mbf{p}}{A_t}{\qrt_{Q,+}}.
}
Hence
$
 \normalizedGhostOverlapC{t}{\mbf{p}'} = \rtu_d^{\la \mbf{p}',\mbf{p}\ra} \normalizedGhostOverlapC{t}{\mbf{p}}.
$
\end{proof}
\begin{proof}[Proof of \Cref{lem:GhostFiducialIndependenceOfTransversal}]
    Let 
    \eag{
    f(\mbf{p}) &= 
    \normalizedGhostOverlapC{t}{G\mbf{p}}
    D\vpu{t}_{\mbf{p}},
    }
    and let $\mbf{p},\mbf{p}'\in \mbb{Z}^2$ be  such that $\mbf{p}'\equiv\mbf{p}\Mod{d}$  and $\mbf{p}',
  \mbf{p}\not\equiv \zero\Mod{d}$.  Setting $\mbf{p}'\equiv\mbf{p} + d\mbf{q}$, it follows from \eqref{eq:Dpperiodicity} and \Cref{lem:nupperiodicity} that
  \eag{
  f(\mbf{p}') &= (-1)^{(d+1)\la \mbf{p},\mbf{q}\ra} \rtu_d^{\la G\mbf{p}',G\mbf{p}\ra}f(\mbf{p})
  \nn
    &= (-1)^{(d+1)\la \mbf{p},\mbf{q}\ra}\rtu_d^{d\det G\la \mbf{p}, \mbf{q}\ra}f(\mbf{p})
  \nn
    &= (-1)^{(d+1)(1+\det G)\la \mbf{p},\mbf{q}\ra}f(\mbf{p})
  \nn
  &= f(\mbf{p}),
  }
  where in the last step we used the fact that $\det G$ is coprime to $d$, as follows from ~\eqref{eq:TwistCondition}.  The second statement is an immediate consequence of this.
\end{proof}

\begin{thm}
\label{thm:nupnumpeq1}
Let $t=(d,r,Q)\sim (K,j,m,Q)$ be an admissible tuple, and let $\mbf{p}\in \mbb{Z}^2$.  
  If $\mbf{p} \not\equiv \zero \Mod{d}$, then the numbers $\normalizedGhostOverlapC{t}{\mbf{p}}$ are real and satisfy
    \eag{ \normalizedGhostOverlapC{t}{\mbf{p}}\normalizedGhostOverlapC{t}{-\mbf{p}} = 1.
    }
\end{thm}
\begin{rmkb}
    This theorem establishes two of the requirements which must be satisfied if the expression on the right-hand side of \eqref{eq:ghostProjectorDef} in \Cref{dfn:CandidateGhostAndSICFiducials} is to be a ghost fiducial.
\end{rmkb} 
\begin{proof}
Let $\A$ be the unique class in $\Clt_{d\infty_2}(\OO_f)$ that maps to the $\SL_2(\Z)$-orbit of $(d^{-1}\p,\qrt_t)$ under the map $\Upsilon_{d\OO_f}$ described in \cite[Thm.\ 3.12]{Kopp2020d} and at the end of \Cref{ssec:qpochmain}.
\Cref{thm:qpochmain} then gives the formula
\begin{equation}
    (\psi^{-2}\chi_{d^{-1}\p}^{-1})(A_t) \ \sfc{d^{-1}\p}{A}{\qrt_t}^2 = \exp\!\left(n Z_{d\infty_2}'(0,\A)\right),
\end{equation}
where $Z_{d\infty_2}(s,\A)$ is the differenced ray class partial zeta function defined in \Cref{defn:rayclasspartialzeta}.
By \Cref{dfn:GhostOverlaps} and \Cref{thm:phaserelation}, we have 
\begin{equation}\label{eq:olsquares}
    (\normalizedGhostOverlapC{t}{\p})^2
    = \SFPhase{t}{\p}^2 \ \sfc{d^{-1}\p}{A}{\qrt_t}^2
    = (\psi^{-2}\chi_{d^{-1}\p}^{-1})(A_t) \ \sfc{d^{-1}\p}{A}{\qrt_t}^2,
\end{equation}
and thus
\begin{equation}\label{eq:nusquared}
    (\normalizedGhostOverlapC{t}{\p})^2 = \exp\!\left(n Z_{d\infty_2}'(0,\A)\right).
\end{equation}
The partial zeta function $Z_{d\infty_2}(s,\A)$ is a complex analytic function defined by a Dirichet series with real coefficients for $\Re(s)>1$ and by analytic continuation to all $s \in \C$. The equation $(Z_{d\infty_2}(s,\A))^\ast = Z_{d\infty_2}(s^\ast,\A)$ holds for $\Re(s)>1$ and thus for all $s \in \C$ (where ${}^\ast$ denotes complex conjugation). Therefore, $Z_{d\infty_2}'(0,\A)$ is real, so $\exp\!\left(n Z_{d\infty_2}'(0,\A)\right)$ is real and positive, and thus 
$\normalizedGhostOverlapC{t}{\mbf{p}}$ is real.

We now compute
$\normalizedGhostOverlapC{t}{\p}\normalizedGhostOverlapC{t}{-\p}$. Using the definition of $\normalizedGhostOverlapC{t}{\pm\p}$ (\Cref{dfn:GhostOverlaps}), we have
\begin{align}
    \normalizedGhostOverlapC{t}{\p}\normalizedGhostOverlapC{t}{-\p}
    = \left(\phi_\p(t)\phi_{-\p}(t)\right) \left(\sfc{d^{-1}\p}{A_t}{\qrt_t}\sfc{-d^{-1}\p}{A_t}{\qrt_t}\right).
\end{align}
Using \Cref{dfn:SFKPhase}, the SF phase satisfies
    \begin{equation}\label{eq:sfphasenegative}
        \SFPhase{t}{-\p}
        = (-1)^{s_d(-\p)}e^{-\frac{\pi i}{12}\Psi(A_t)}\rtu_d^{-\frac{f_{jm}}{f}Q(-\p)}
        = (-1)^{s_d(\p)}e^{-\frac{\pi i}{12}\Psi(A_t)}\rtu_d^{-\frac{f_{jm}}{f}Q(\p)}
        = \SFPhase{t}{\p}.
    \end{equation}
\Cref{cor:funchar} and \Cref{thm:phaserelation} imply
\begin{align}
    \sfc{d^{-1}\mbf{p}}{A_t}{\qrt_t}\sfc{-d^{-1}\mbf{p}}{A_t}{\qrt_t} 
    &= (\psi^2\chi_\r)(A_t) = \phi_\p(t)^{-2}.
\end{align}
Consequently
$\normalizedGhostOverlapC{t}{\p}\normalizedGhostOverlapC{t}{-\p} 
= \phi_\p(t)^{2} \phi_\p(t)^{-2} = 1$.
\end{proof}

\begin{lem}\label{lem:nu01overnu0val}
   Let $t=(d,r,Q)\sim(K,j,m,Q)$ be an admissible tuple.  Then
    \eag{
    \normalizedGhostOverlapC{t}{\zero}
+\frac{1}
{\normalizedGhostOverlapC{t}{\zero}}&=-(d-2r)\sqrt{d_j+1}.
    }
\end{lem}
\begin{proof}
    It follows from Lemma~\ref{lm:shinatzero} that
    \eag{\label{eq:nu01overnu0val}
      \normalizedGhostOverlapC{t}{\zero}&=
      \SFPhase{t}{\zero}
      \sfc{\zero}{A_t}{\qrt_{Q,+}}=-\frac{1}{\sqrt{j_{A_t}(\qrt_{Q,+})}}
    }
    Write $A_t = \smt{\ma &\mb \\ \mc & \md}$. It follows from \Cref{lm:fxdpt} that $A_t\qrt_{Q,+} = \qrt_{Q,+}$.  Hence
    \eag{
         \mc \qrt_{Q,+}^2+(\md - \ma)\qrt_{Q,+} - \mb &= 0.
        }
        Consequently
        \eag{
         \qrt_{Q,+} &= \frac{\ma - \md \pm \sqrt{(\ma+\md)^2-4}}{2\mc}
        &&\implies&j_A(\qrt_{Q,+}) &= \frac{\Tr(A)\pm \sqrt{\Tr(A)^2-4}}{2}.
    }
    It follows from Theorem~\ref{tm:symgp} that $\Tr(A_t)$ and consequently $j_{A_t}(\qrt_{Q,+})$ are positive.  So%
    \eag{
    \left(\sqrt{j_{A_t}(\qrt_{Q,+})}+\frac{1}{\sqrt{j_{A_t}(\qrt_{Q,+})}}\right)^2 
    &= \Tr(A_t) +2
    }
    implying
    \eag{\label{eq:nu01overnu0val1}
    \normalizedGhostOverlapC{t}{\zero}
+\frac{1}{\normalizedGhostOverlapC{t}{\zero}} &= -\sqrt{\Tr(A_t)+2}.
    }
    It follows from Theorem~\ref{tm:symgp} that $A_t=\canrep_Q(\vn^{j(2m+1)})$, which in view of \Cref{lem:towerbasic}  means
    $\Tr(A) = \Tr(\vn^{j(2m+1)}) = d_{j(2m+1)}-1$.  Lemma~\ref{lem:dimgridtechres} then implies
    \eag{\label{eq:nu01overnu0val2}
        \sqrt{\Tr(A)+2} &= (r_{j,m+1}-r_{j,m}) \sqrt{d_j+1} = (d-2r_{j,m}) \sqrt{d_j+1}.
    }
    Together, \eqref{eq:nu01overnu0val1} and \eqref{eq:nu01overnu0val2} imply \eqref{eq:nu01overnu0val}.
\end{proof}
\begin{defn}[function $\fn_t$]\label{dfn:functionht}
    Given an admissible tuple $t=(d,r,Q)\sim(K,j,m,Q)$, define $\fn_t\colon \mbb{Z}/d\mbb{Z} \to \mbb{Z}/\db\mbb{Z}$ by 
    \eag{
    \fn_t(x) &= r(2x+d+d_j-1).
      \label{eq:fexpn}
    }
\end{defn}
\begin{lem}\label{lem:fdf} 
    Let $t=(d,r,Q)\sim(K,j,m,Q)$ be an admissible tuple. 
    \begin{enumerate}
        \item If $d$ is odd, then $\fn_t$ is a bijection of $\mbb{Z}/d\mbb{Z}$ onto itself.
        \item If $d$ is even, then $\fn_t$ is a bijection of $\mbb{Z}/d\mbb{Z}$ onto the set of odd elements in $\mbb{Z}/\db\mbb{Z}$.
    \end{enumerate}
    
    The inverse function is given by
    \eag{\label{eq:htInverse}
      \fn_t^{-1}(x) &=\frac{(d+1)\big((d_j+1)r x-(d_j-1)\big)}{2}.
    }

    For all $x\in \mbb{Z}$, $\fn_t(x)$ is coprime to $\db$ if and only if $2x+d_j-1$ is coprime to $d$.
\end{lem}
\begin{proof}
    Suppose $d$ is odd.  If $\fn_t(x)=\fn_t(x')$, then $2r x \equiv 2 r x' \Mod{d}$, which in view of \Cref{cor:djmcprjm} means $x \equiv x' \Mod{\db}$.  So $\fn_t$ is injective.  Since the domain and codomain have the same cardinality, it follows that it is also surjective.  To calculate the inverse, observe that it follows from Lemma~\ref{lem:rjmm1expn} that, if $y = \fn_t(x)$, then
    \eag{
    x &= 2^{-1}\left(r^{-1}y-(d_j-1)\right)
    = \tfrac{d+1}{2}\left((d_j+1)ry-(d_j-1)\right),
    }
    where $2^{-1}$ is the multiplicative inverse of $2 \Mod{d}$.

    Suppose, on the other hand, that $d$ is even. Then it follows from \Cref{cor:djmcprjm}  that $r$ is odd  and from \Cref{lem:dimgridtechres}   that $d_j$ is even.  So $\fn_t(x)$ is odd for all $x\in \mbb{Z}/d\mbb{Z}$. If $\fn_t(x) \equiv \fn_t(x') \Mod{\db}$, then $2r x \equiv 2r x' \Mod{2d}$.  Since $r$ is odd, it follows that $x \equiv x' \Mod{d}$.  So $\fn_t$ is injective.  Since the domain and the set of odd elements in $\mbb{Z}/\db \mbb{Z}$ have the same cardinality, $\fn_t$ must in fact be a bijection onto the set of odd elements in $\mbb{Z}/\db\mbb{Z}$. To calculate the inverse, suppose $y=\fn_t(x)$ for some $x\in \mbb{Z}/d\mbb{Z}$.  Then it follows from Lemma~\ref{lem:rjmm1expn} that
    \eag{
    2 x &\equiv yr^{-1} -(d_j-1)-d \Mod{2d}
    \nn
    & \equiv
    yr(1+d_j+d)-(d_j-1)-d \Mod{2d}
    \nn
    &\equiv (d_j+1)ry-(d_j-1) + d(yr-1) \Mod{2d}.
    }
    Since $yr-1$ is even, this means
    \eag{
    2x &=(d_j+1)ry + 1-d_j \Mod{2d}.
    }
    The fact that $yr$ and $d_{j}\pm 1$ are odd means the right-hand side is even. So
    \eag{
    x &\equiv \frac{(d_j+1)r_{j,m}y -(d_j-1)}{2} \Mod{d}
    }
    or, equivalently,
    \eag{
    x &\equiv  (d+1)\left(\frac{(d_j+1)ry -(d_j-1)}{2}\right) \Mod{d}.
    }
    
    Finally, $\fn_t(x)$ is coprime to $\db$ if and only if it is coprime to $d$. 
    Since $r$ is coprime to $d$, this is true if and only if $2x-1+d_j$ is coprime to $d$.
\end{proof}

\subsection{Ghost existence under the Twisted Convolution Conjecture}
\label{sbsc:proofofghosttheorem}
We now turn to the proof of Theorem~\ref{thm:ghstExist}, asserting that, under the assumption of \Cref{cnj:tci}, the $d \times d$ matrix $\tilde{\Pi}_s$ constructed from a fiducial datum is a ghost $r$-SIC fiducial projector. 
\begin{proof}[Proof of Theorem~\ref{thm:ghstExist}]
To prove that $\tilde{\Pi}_s$ is a ghost projector, we need to show:
\begin{enumerate}
    \item The expression on the RHS of \eqref{eq:ghostProjectorDef} is well-defined, in that the sum is independent of the set of coset representatives chosen,
    \item $\normalizedGhostOverlapC{t}{\mbf{p}}$ is real for all $\mbf{p}\not\equiv \zero \Mod{d}$,
    \item $\normalizedGhostOverlapC{t}{\mbf{p}}\normalizedGhostOverlapC{t}{-\mbf{p}} = 1$ for all $\mbf{p}\not\equiv \zero \Mod{d}$,
    \item $\tilde{\Pi}_s^2 = \tilde{\Pi}\vpu{2}_s$.
\end{enumerate}
The first  proposition is proved in \Cref{lem:GhostFiducialIndependenceOfTransversal},
and the second and third in \Cref{thm:nupnumpeq1}.  It remains to prove the last.
We have
\eag{
\tilde{\Pi}_s &= \frac{r}{d} I + \frac{1}{d\sqrt{d_{j}+1}}\sum_{\mbf{p}\notin d\mbb{Z}^2}\normalizedGhostOverlapC{t}{\mbf{p}} D\vpu{t}_{G^{-1}\mbf{p}}
}
where the sum is over any set of coset representatives of $\mbb{Z}^2/d\mbb{Z}^2$ with the representative of $d\mbb{Z}^2$  excluded.
Hence
\eag{
\tilde{\Pi}_s^2 - \tilde{\Pi}\vpu{2}_s &=
\frac{r(r - d)}{d^2}I
+\frac{2r-d}{d^2\sqrt{d_j+1}}\sum_{\mbf{p}\notin d\mbb{Z}^2}\normalizedGhostOverlapC{t}{\mbf{p}} D\vpu{t}_{G^{-1}\mbf{p}}
\nn
&\hspace{2.3 cm}
+\frac{1}{d^2(d_j+1)}
\sum_{\mbf{p},\mbf{q}\notin d\mbb{Z}^2}\rtu_d^{\la G^{-1}\mbf{p},G^{-1}\mbf{q}\ra}\normalizedGhostOverlapC{t}{\mbf{p}}\normalizedGhostOverlapC{t}{\mbf{q}}D\vpu{t}_{G^{-1}(\mbf{p}+\mbf{q})}
\nn
&=\frac{r(r - d)}{d^2}I
+\frac{2r-d}{d^2\sqrt{d_j+1}}\sum_{\mbf{p}\notin d\mbb{Z}^2}\normalizedGhostOverlapC{t}{\mbf{p}} D\vpu{t}_{G^{-1}\mbf{p}}
\nn
&\hspace{2.3 cm}
+\frac{1}{d^2(d_j+1)}
\sum_{\mbf{p}}\left(\sum_{\substack{\mbf{q}\notin d\mbb{Z}^2, \\ \mbf{q}\notin \mbf{p}+d\mbb{Z}^2}}\rtu_d^{\Det(G^{-1})\la \mbf{p},\mbf{q}\ra}
\normalizedGhostOverlapC{t}{\mbf{p}-\mbf{q}}
\normalizedGhostOverlapC{t}{\mbf{q}}
\right) D\vpu{t}_{G^{-1}\mbf{p}},
\label{eq:Pi2MinusPia}
}
where $\mbf{p}$ is summed over a complete set of cosets for $\mbb{Z}^2/d\mbb{Z}^2$, and $\mbf{q}$ is summed over such a set with $d\mbb{Z}^2$ and $\mbf{p}+d\mbb{Z}^2$ removed. Comparing 
~\eqref{eq:TwistCondition} to ~\eqref{eq:fexpn}, we see that
\eag{
\Det(G^{-1}) &= \left(\Det(G)\right)^{-1} = \fn_t(\shift)
}
for some $\shift\in\mcl{Z}_t$, 
while it follows from Theorem~\ref{thm:nrddjmrjm} that
\eag{
\frac{r(r-d)}{d^2}&= -\frac{(d^2-1)}{d^2(d_j+1)}.
}
Inserting these expressions into \Cref{eq:Pi2MinusPia} gives
\eag{
\tilde{\Pi}_s^2 - \tilde{\Pi}\vpu{2}_s
&=\frac{1}{d^2(d_j+1)}\left(-(d^2-1)I
-(d-2r)\sqrt{d_j+1}\sum_{\mbf{p}\notin d\mbb{Z}^2}\normalizedGhostOverlapC{t}{\mbf{p}}
D\vpu{t}_{G^{-1}\mbf{p}}\right.
\nn
&\hspace{3.2 cm}
+\left. 
\sum_{\mbf{p}}\left(\sum_{\substack{\mbf{q}\notin d\mbb{Z}^2, \\ \mbf{q}\notin \mbf{p}+d\mbb{Z}^2}}\rtu_d^{\fn_t(\shift)\la \mbf{p},\mbf{q}\ra}
\normalizedGhostOverlapC{t}{\mbf{p}-\mbf{q}}
\normalizedGhostOverlapC{t}{\mbf{q}}
\right) D\vpu{t}_{G^{-1}\mbf{p}}\right).
}
\Cref{thm:nupnumpeq1} and~\Cref{lem:nu01overnu0val} imply
\eag{
&\sum_{\mbf{p}}\left(\sum_{\substack{\mbf{q}\notin d\mbb{Z}^2, \\ \mbf{q}\notin \mbf{p}+d\mbb{Z}^2}}\rtu_d^{\fn_t(\shift)\la \mbf{p},\mbf{q}\ra}
\normalizedGhostOverlapC{t}{\mbf{p}-\mbf{q}}
\normalizedGhostOverlapC{t}{\mbf{q}}
\right) D\vpu{t}_{G^{-1}\mbf{p}}
\nn
&\hspace{1 cm} = (d^2-1)I +
\sum_{\mbf{p}\notin d\mbb{Z}^2}\left(\sum_{\substack{\mbf{q}\notin d\mbb{Z}^2, \\ \mbf{q}\notin \mbf{p}+d\mbb{Z}^2}}\rtu_d^{\fn_t(\shift)\la \mbf{p},\mbf{q}\ra}
\normalizedGhostOverlapC{t}{\mbf{p}-\mbf{q}}
\normalizedGhostOverlapC{t}{\mbf{q}}
\right) D\vpu{t}_{G^{-1}\mbf{p}}
\nn
&\hspace{1 cm}
= (d^2-1)I +
\sum_{\mbf{p}\notin d\mbb{Z}^2}\left(\sum_{\substack{\mbf{q}\notin d\mbb{Z}^2, \\ \mbf{q}\notin \mbf{p}+d\mbb{Z}^2}}\rtu_d^{\fn_t(\shift)\la \mbf{p},\mbf{q}\ra}
\frac{\normalizedGhostOverlapC{t}{\mbf{q}}}{\normalizedGhostOverlapC{t}{\mbf{q}-\mbf{p}}}
\right) D\vpu{t}_{G^{-1}\mbf{p}}
\nn &\hspace{1 cm}
= (d^2-1)I +(d-2r)\sqrt{d_j+1}\sum_{\mbf{p}\notin d_{j,m}\mbb{Z}^2}\normalizedGhostOverlapC{t}{\mbf{p}} D\vpu{t}_{G^{-1}\mbf{p}}
\nn
&\hspace{4 cm}
+\sum_{\mbf{p}\notin d\mbb{Z}^2}\left(\sum_{\mbf{q}\in \mcl{I}_{\mbf{p}}}\rtu_d^{\fn_t(\shift)\la \mbf{p},\mbf{q}\ra}
\frac{
\normalizedGhostOverlapC{t}{\mbf{q}}
}{
\normalizedGhostOverlapC{t}{\mbf{q}-\mbf{p}}
}
\right) D\vpu{t}_{G^{-1}\mbf{p}}
}
where $\mcl{I}_{\mbf{p}}$ is a complete set of coset representatives for $\mbb{Z}^2/d_{j,m}\mbb{Z}^2$ which includes $\zero$ and $\mbf{p}$.  Hence
\eag{
\tilde{\Pi}_s^2 - \tilde{\Pi}\vpu{2}_s 
&=
\frac{1}{d^2(d_j+1)}
\sum_{\mbf{p}\notin d\mbb{Z}^2}\left(\sum_{\mbf{q}\in \mcl{I}_{\mbf{p}}}\rtu_d^{\fn_t(\shift)\la \mbf{p},\mbf{q}\ra}
\frac{
\normalizedGhostOverlapC{t}{\mbf{q}}
}{
\normalizedGhostOverlapC{t}{\mbf{q}-\mbf{p}}
}
\right) D\vpu{t}_{G^{-1}\mbf{p}}
\nn
&= \frac{1}{d^2(d_j+1)}
\sum_{\mbf{p}\notin d\mbb{Z}^2}\left(\sum_{\mbf{q}\in \mcl{I}_{\mbf{p}}}\rtu_d^{\fn_t(\shift)\la \mbf{p},\mbf{q}\ra}\left(
\frac{
\SFPhase{t}{\mbf{q}}
}{
\SFPhase{t}{\mbf{q}-\mbf{p}}
}\right)\left(\frac{\sfc{d^{-1}\mbf{q}}{A_t}{\qrt_{Q,+}}}{\sfc{d^{-1}(\mbf{q}-\mbf{p})}{A_t}{\qrt_{Q,+}}}\right)\right) D\vpu{t}_{G^{-1}\mbf{p}}
\label{eq:GhostProjSquaredMinusGhostProj}
}
where we used \Cref{dfn:GhostOverlaps} in the second line.  It follows from \Cref{dfn:SFKPhase} that
\eag{
\frac{
\SFPhase{t}{\mbf{q}}
}{
\SFPhase{t}{\mbf{q}-\mbf{p}}
}
&=(-1)^{s_d(\mbf{q})+s_d(\mbf{q}-\mbf{p})}\rtu_d^{\frac{f_{jm}}{f}(Q(\mbf{q}-\mbf{p})-Q(\mbf{q}))}.
}
We have
\eag{
(-1)^{s_d(\mbf{q})+s_d(\mbf{q}-\mbf{p})} 
&=
(-1)^{(1+d)((1+q_1)(1+q_2) + (1+q_1+p_1)(1+q_2+p_2)}
\nn
&=(-1)^{(1+d)(p_1(1+q_2)+p_2(1+q_1)+p_1p_2}
\nn
&= (-1)^{(1+d)(p_1+p_2+p_1p_2) +(1+d)\la \mbf{p},\mbf{q}\ra}
\nn
&= (-1)^{1+d+(1+d)(1+p_1)(1+p_2)}\rtu_d^{d\la \mbf{p},\mbf{q}\ra}
\nn
&= (-1)^{1+s_d(\mbf{p})}\rtu_d^{d\la \mbf{p},\mbf{q}\ra}.
}
Also, referring to \Cref{dfn:AssociatedStabilizers}, we see that
\eag{
\frac{2f_j}{f}\mbf{p}^{\rm{T}}Q\mbf{q} &=- \frac{2f_j}{f}\la \mbf{p}, SQ\mbf{q}\ra 
=\la \mbf{p}, (d_j-1-2\zaunerGen{t})\mbf{q}\ra,
}
implying
\eag{
\frac{f_{jm}}{f}\left(Q(\mbf{q}-\mbf{p})-Q(\mbf{q})\right) 
&=\frac{rf_j}{f}\left(Q(\mbf{p})-2\mbf{p}^{\rm{T}}Q\mbf{q}\right)
\nn
&=
\frac{rf_j}{f}Q(\mbf{p})-r(d_j-1)\la\mbf{p},\mbf{q}\ra +2r\la\mbf{p},\zaunerGen{t}\mbf{q}\ra.
}
Hence
\eag{
\frac{
\SFPhase{t}{\mbf{q}}
}{
\SFPhase{t}{\mbf{q}-\mbf{p}}
}
&=(-1)^{1+s_d(\mbf{p})}\rtu_d^{\frac{rf_j}{f}Q(\mbf{p})}\rtu_d^{(d-r(d_j-1))\la\mbf{p},\mbf{q}\ra}\rtus_d^{r\la\mbf{p},\zaunerGen{t}\mbf{q}\ra}.
\label{eq:RatioSFPhases}
}
It follows from \Cref{lm:sfam1sfaeq1} that
\eag{
\frac{\sfc{d^{-1}\mbf{q}}{A_t}{\qrt_{Q,+}}}{\sfc{d^{-1}(\mbf{q}-\mbf{p})}{A_t}{\qrt_{Q,+}}}
&=
\sfc{d^{-1}\mbf{q}}{A\vpu{-1}_t}{\qrt_{Q,+}}
\sfc{d^{-1}(\mbf{q}-\mbf{p})}{A^{-1}_t}{\qrt_{Q,+}}.
\label{eq:RatioSFModularCocycles}
}
Substituting 
~\eqref{eq:RatioSFPhases} and~\eqref{eq:RatioSFModularCocycles} into \eqref{eq:GhostProjSquaredMinusGhostProj} gives
\eag{
\tilde{\Pi}_s^2-\tilde{\Pi}\vpu{2}_s &= \frac{1}{d^2(d_j+1)}\sum_{\mbf{p}\notin d\mbb{Z}^2}
(-1)^{1+s_d(\mbf{p})}\rtu_d^{\frac{rf_j}{f}Q(\mbf{p})}\left( 
\sum_{\mbf{q}\in \mcl{I}_{\mbf{p}}}\rtu_d^{\left(\fn_t(\shift) + d-r(d_j-1)\right)\la \mbf{p},\mbf{q}\ra}\right.
\nn
&\hspace{4 cm} \times \left.
\rtus_d^{r\la\mbf{p},\zaunerGen{t}\mbf{q}\ra} 
\sfc{d^{-1}\mbf{q}}{A\vpu{-1}_t}{\qrt_{Q,+}}
\sfc{d^{-1}(\mbf{q}-\mbf{p})}{A^{-1}_t}{\qrt_{Q,+}}
\right)D_{G^{-1}\mbf{p}}.
\label{eq:GhostProjSquaredMinusGhostProjB}
}
It follows from \Cref{dfn:shift} that $2\shift+d_j-1$ is coprime to $d$.  In view of \Cref{lem:fdf} this means $\fn_t(\shift)$ is coprime to $\db$.  Consequently
\eag{
\fn_t(\shift)+d &\equiv (d+1)\fn_t(\shift) \quad \Mod{\db}.
}
Also, it follows from \Cref{lem:dimgridtechres} that $d_j$ is even if $d$ is even.  So
\eag{
dd_j &\equiv 0 \quad \Mod{\db}
}
irrespective of whether $d$ is odd or even.  Putting these two facts together we conclude 
\eag{
\rtu_d^{\left(\fn_t(\shift) + d-r(d_j-1)\right)\la \mbf{p},\mbf{q}\ra}
&=
\rtu_d^{\left((d+1)\fn_t(\shift)-r(d_j-1)\right)\la \mbf{p},\mbf{q}\ra}
\nn
&=\rtu_d^{\left((d+1)r(2z+d+d_j-1)-r(d_j-1)\right)\la \mbf{p},\mbf{q}\ra}
\nn
&=\rtu_d^{r\left(2(d+1)\shift+d(d+1) +d(d_j-1))\right)\la\mbf{p},\mbf{q}\ra}
\nn
&=\rtus_d^{rz\la\mbf{p},\mbf{q}\ra}.
}
Substituting back into \eqref{eq:GhostProjSquaredMinusGhostProjB} we deduce
\eag{
&\tilde{\Pi}_s^2 - \tilde{\Pi}_s 
\nn
& \hspace{0.2 cm}=
\sum_{\mbf{p}\notin d\mbb{Z}^2}\frac{(-1)^{1+s_d(\mbf{p})}\rtu_d^{\frac{rf_j}{f}Q(\mbf{p})}}{d^2(d_j+1)}
\left( \sum_{\mbf{q}\in \mcl{I}_{\mbf{p}}}
\rtus_d^{r\la \mbf{p},(zI +\zaunerGen{t})\mbf{q}\ra}
\sfc{d^{-1}\mbf{q}}{A\vpu{-1}_t}{\qrt_{Q,+}}
\sfc{d^{-1}(\mbf{q}-\mbf{p})}{A^{-1}_t}{\qrt_{Q,+}}
\right)D_{G^{-1}\mbf{p}}.
}
It follows that $\tilde{\Pi}_s^2 = \tilde{\Pi}_s$ if and only if 
\eag{
\sum_{\mbf{q}\in \mcl{I}_{\mbf{p}}}
\rtus_d^{r\la \mbf{p},(zI +\zaunerGen{t})\mbf{q}\ra}
\sfc{d^{-1}\mbf{q}}{A\vpu{-1}_t}{\qrt_{Q,+}}
\sfc{d^{-1}(\mbf{q}-\mbf{p})}{A^{-1}_t}{\qrt_{Q,+}} = 0
}
for all $\mbf{p} \not\equiv \zero \Mod{d}$.  If $\mbf{p} \equiv \zero \Mod{d}$, it follows from \Cref{lm:sfam1sfaeq1} that  
\eag{
\sum_{\mbf{q}\in \mcl{I}_{\mbf{p}}}
\rtus_d^{r\la \mbf{p},(zI +\zaunerGen{t})\mbf{q}\ra}
\sfc{d^{-1}\mbf{q}}{A\vpu{-1}_t}{\qrt_{Q,+}}
\sfc{d^{-1}(\mbf{q}-\mbf{p})}{A^{-1}_t}{\qrt_{Q,+}} = 
\sum_{\mbf{q}\in \mcl{I}_{\zero}}
\sfc{d^{-1}\mbf{q}}{A\vpu{-1}_t}{\qrt_{Q,+}}
\sfc{d^{-1}(\mbf{q})}{A^{-1}_t}{\qrt_{Q,+}} = d^2
}
irrespective of whether $\tilde{\Pi}_s$ is a ghost fiducial.
So it is in fact the case that $\tilde{\Pi}_s^2 = \tilde{\Pi}_s$ if and only if 
\eag{
\sum_{\mbf{q}\in \mcl{I}_{\mbf{p}}}
\rtus_d^{r\la \mbf{p},(zI +\zaunerGen{t})\mbf{q}\ra}
\sfc{d^{-1}\mbf{q}}{A\vpu{-1}_t}{\qrt_{Q,+}}
\sfc{d^{-1}(\mbf{q}-\mbf{p})}{A^{-1}_t}{\qrt_{Q,+}} = d^2\delta^{(d)}_{\mbf{p},\zero}
}
for all $\mbf{p}$.
\end{proof}

\subsection{Remarks concerning the set of shifts}\label{sbsc:ProofMainTheoremsFurtherRemarks}
In this subsection we make some observations concerning the set of shifts $\mcl{Z}_t$ for a given admissible tuple $t$.

The argument in \Cref{sbsc:proofofghosttheorem} actually shows a little more than was required  for the proof of \Cref{thm:ghstExist}.  Specifically, let $t=(d,r,Q)\sim(K,j,m,Q)$ be an admissible tuple, and let $\shift$ be an element of $\mcl{Z}_t$.  Then $2z+d_j-1$ is coprime to $d$, which in view of \Cref{lem:fdf} means $\fn_t(\shift)$ is coprime to $\db$.  So if we define $G = \smt{1 & 0 \\ 0 & \fn_t(\shift)^{-1}}$ (where $\fn_t(\shift)^{-1}$ is the inverse of $\fn_t(\shift)$ modulo $\db$), then $G\in \GLtwo{\mbb{Z}/\db\mbb{Z}}$, and $s=(d,r,Q,G,g)\sim(K,j,m,Q,G,g)$ is a fiducial datum for any appropriate choice of $g$.  Consequently
\eag{
\tilde{\Pi}_s &= \frac{r}{d}I + \frac{1}{d\sqrt{d_j+1}}\sum_{\mbf{p}\notin d\mbb{Z}^2} 
\normalizedGhostOverlapC{t}{G\mbf{p}} 
D_{\mbf{p}}
}
is a ghost fiducial.  Consider its complex conjugate 
\eag{
\tilde{\Pi}_s^{*} & = \frac{r}{d}I + \frac{1}{d\sqrt{d_j+1}}\sum_{\mbf{p}\notin d\mbb{Z}^2} 
\normalizedGhostOverlapC{t}{G\mbf{p}}
D_{J\mbf{p}}
\nn
&= \frac{r}{d}I + \frac{1}{d\sqrt{d_j+1}}\sum_{\mbf{p}\notin d\mbb{Z}^2} 
\normalizedGhostOverlapC{t}{JG\mbf{p}}
D_{\mbf{p}}
}
(where $J$ is defined in \Cref{dfn:symplectic}).  It is easily seen that $\tilde{\Pi}_s^{*}$ is also a ghost fiducial.   The fact that $\Det(JG) = -\Det(G)$ means that $\Det(JG)$ is coprime to $\db$, implying that $\bar{\shift} = \fn_t^{-1}(\Det (JG)^{-1})$ is well-defined. Since $\fn_t(\bar{\shift})  = \Det(JG)^{-1}$ is coprime to $\db$ it follows from \Cref{lem:fdf} that $2\bar{\shift}+d_j-1$. The argument in \Cref{sbsc:proofofghosttheorem} shows that the fact that $\tilde{\Pi}_s^{*2}=\tilde{\Pi}_s^{*}$ implies that $\bar{\shift}$ satisfies \eqref{eq:tcc}.  It follows that $\bar{\shift}\in \mcl{Z}_t$.

We have
\eag{
\bar{\shift} &= \fn_t^{-1}(-\Det G^{-1}) = \fn_t^{-1}(-\fn_t(\shift)).
}
This expression leads to a simple, fully explicit expression for $\bar{\shift}$ in terms of $\shift$:
\begin{lem}\label{lem:shiftslambdabar}
    Let $t=(d,r,Q)\sim(K,j,m,Q)$ be an admissible tuple, let  $\shift\in \mcl{Z}_t$, and let $\bar{\shift} = \fn_t^{-1}(-\fn_t(\shift))$.  Then
    \eag{
    \bar{\shift} &= -(\shift+d_j-1)
    } 
\end{lem}
\begin{proof}
\Cref{dfn:functionht} and \Cref{lem:fdf} imply
\eag{
\bar{\shift}&= \fn_t^{-1}(-\Det G^{-1})
\nn
&= \fn_t^{-1}(-\fn_t(\shift))
\nn
&=\frac{(d+1)\left(-(d_j+1)r^2(2z+d+d_j-1)-(d_j-1)  \right)}{2}.
}
It follows from \eqref{eq:rSquaredIdentity} that 
\eag{
-(d_j+1)r^2 =d-1+k\db
}
for some integer $k$.  So
\eag{
\bar{\shift}&=\frac{(d+1)\left((d-1+k\db)(2z+d+d_j-1)-(d_j-1) \right)}{2}
\nn
&= (d+1)(d-1+k\bar{d})\shift+ \frac{(d+1)\left((d-1+k\db)(d+d_j-1)-(d_j-1)\right)}{2}
\nn
&= -\shift+ \frac{(d+1)\left((d-1+k\db)(d+d_j-1)-(d_j-1)\right)}{2}.
}
Suppose $d$ is odd.  Then $d+1=2m$ for some integer $m$, and
\eag{
\bar{\shift}&= -\shift+ m\left((d-1+k\db)(d+d_j-1) - (d_j-1) \right) 
\nn
&= -\shift-2m(d_j-1)
\nn
&= -\shift -(d+1)(d_j-1) 
\nn
&= -\shift - (d_j-1).
}
Suppose, on the other hand, $d$ is even.  Then $d=2m$ for some integer $m$, and
\eag{
\bar{\shift} &=-\shift  + \frac{(d+1)\left((2m+4km-1)(2m+d_j-1)-(d_j-1)\right)}{2}
\nn
&= -\shift + \frac{(d+1)\left(2m(1+2k)(2m+d_j-1)-2m-2(d_j-1)  \right)}{2}  
\nn
&= -\shift+(d+1)\left( m(1+2k)(2m+d_j-1) - m -(d_j-1)\right) 
\nn
&= -\shift+m(1+2k)(d_j-1) -m -(d_j-1)
\nn
&= -\shift + md_j -(d_j-1).
}
The fact that $d$ is even means, in view of \Cref{lem:dimgridtechres}, that $d_j$ is even.  Consequently $md_j \equiv 0 \Mod{d}$, and
$
\bar{\shift} = -\shift-(d_j-1)
$
in this case too.
\end{proof}

It follows from \Cref{lem:shiftslambdabar} that possible shifts come in pairs $(\shift,\bar{\shift})$.  The question then arises how many such pairs there are.  Observe that it follows from \Cref{lem:djpm1cprimedjm} that $\shift=0$ and $\shift=1$ both satisfy the requirement that $2z+d_j-1$ be coprime to $d$.  In every case examined it appears, as a matter of empirical observation, that they also satisfy \eqref{eq:tcc}.  The corresponding values of $\bar{\shift}$ are $\bar{\shift} = d-d_j+1$ and $\bar{\shift} = d-d_j$, respectively.  The values of $\Det G $ are as given in the table
\begin{center}
    \begin{tabular}{c|c}
         \toprule
         $\shift$ & $\Det G$
         \\
         \midrule
         $0$ & $r^{-1}(d+d_j-1)^{-1}$ 
         \\
         $1$ & $r^{-1}(d+d_j+1)^{-1} = r$
         \\
         $d-d_j$ & $-r^{-1}(d+d_j+1)^{-1} = -r$
         \\
         $d-d_j + 1$ & $-r^{-1}(d+d_j-1)^{-1}$
         \\
         \bottomrule
    \end{tabular}
\end{center}
where the symbols $r^{-1}$, $(d+d_j\pm 1)^{-1}$ denote multiplicative inverses modulo $\db$, and where the alternative expressions for $\Det G$ in the second and third rows are obtained using \eqref{eq:rSquaredIdentity}.
Note that if $m=1$ these four values for $\shift$ reduce to the two values $0$, $1$ with the corresponding values of $\Det G$ being $-1$, $+1$ respectively.  

It appears (again as a matter of empirical observation) that, if $m\le 2$,  these are the only elements of $\mcl{Z}_t$.  However, if $m>2$, there seem to be others.  Specifically, in the cases we examined we found  three pairs for $m=3$ and $m=4$, and five pairs for $m=5$. However, we were not able to discern any obvious pattern to the additional $\shift$-values. This is a question requiring further investigation.

\subsection{Conditional SIC existence}\label{sbsc:proofofghosttorsic}

We now turn to the proof of \Cref{thm:rayclassfieldrsicgen}, which asserts that Zauner's conjecture follows from the conditional assumptions of the Twisted Convolution Conjectural and the Stark Conjecture, and more precisely, that those conditional assumptions imply that the $d \times d$ matrix $\Pi_s$ constructed from a fiducial datum $s$ is an $r$-SIC fiducial projector.

\begin{proof}[Proof of \Cref{thm:rayclassfieldrsicgen}]
We have
\begin{align}
    \tilde{\Pi}_s &= \frac{r}{d} I + \frac{1}{d}\sum_{\mbf{p}\notin d\mbb{Z}^2}\tilde\mu_\p(t) D\vpu{t}_{G^{-1}\mbf{p}}.
\end{align}
Set $\mu_\p(t) = g(\tilde\mu_p(t))$, and apply the Galois automorphism $g$ to obtain
\begin{align}
    \Pi_s
    = g(\tilde{\Pi}_s) 
    = \frac{r}{d} I + \frac{1}{d}\sum_{\mbf{p}\notin d\mbb{Z}^2} \mu_\p(t) g(D\vpu{t}_{G^{-1}\mbf{p}})
    = \frac{r}{d} I + \frac{1}{d}\sum_{\mbf{p}\notin d\mbb{Z}^2} \mu_\p(t) D_{H_g\p},
\end{align}
where $H_g = \smmattwo{1}{0}{0}{k_g}$ and we have used \Cref{thm:GalActOnClifford} in the last line. To show that $\Pi_s$ is a Weyl--Heisenberg $r$-SIC fiducial projector, we must show that:
\begin{itemize}
    \item[(1)] $\Pi_s^2 = \Pi_s$,
    \item[(2)] $\mu_p(t)\mu_{-p}(t) = \frac{1}{d_j+1}$ for $\p \nin d\Z^2$, and
    \item[(3)] $\abs{\mu_p(t)} = \frac{1}{\sqrt{d_j+1}}$ for $\p \nin d\Z^2$.
\end{itemize}
    By \Cref{thm:ghstExist} (using the assumption of \Cref{cnj:tci}), $\tilde{\Pi}_s$ is a ghost $r$-SIC fiducial, so $\tilde\Pi_s^2 = \tilde\Pi_s$. Applying the Galois automorphism $g$ (which commutes with matrix multiplication) shows that $\Pi_s^2 = \Pi_s$, giving (1).
    Now suppose $\p \nin d\Z^2$, and rewrite \Cref{thm:nupnumpeq1} in terms of the unnormalized overlaps:
    \begin{equation}
        \tilde\mu_p(t)\tilde\mu_{-p}(t) = \tfrac{1}{d_j+1}.
    \end{equation}
    Applying $g$ gives $\mu_p(t)\mu_{-p}(t) = \frac{1}{d_j+1}$, which is (2).
    As in \eqref{eq:mushin}, we have
    \begin{equation}
        \tilde\mu_\p(t)
        = \tfrac{1}{\sqrt{d_j+1}}\, \phi_\p(t) \sfc{d^{-1}\p}{A}{\qrt_t} \text{ for } \p \not\equiv \zero \Mod{d}.
    \end{equation}
    By the assumption of \Cref{conj:mrmvc}, we have
    \begin{equation}
        \abs{g(\sfc{d^{-1}\p}{\!A_t}{\qrt_t})}=1.
    \end{equation}
    Therefore, writing $g(\sqrt{d_j+1}) = \pm \sqrt{d_j+1}$ and using the fact that $\phi_\p(t)$ is a root of unity and indeed a power of $\xi_d$, we have
    \begin{equation}
        \mu_\p(t)
        = \pm \tfrac{1}{\sqrt{d_j+1}}\, \phi_\p(t)^{k_g} g(\sfc{d^{-1}\p}{A}{\qrt_t}).
    \end{equation}
    Taking absolute values, we obtain $\abs{\mu_\p(t)} = \frac{1}{\sqrt{d_j+1}}$, which is (3).
\end{proof}

\section{Proof of main theorems (2): Class fields attained}
\label{sec:classfieldsattained}

We now prove results about which abelian extensions are generated by $r$-SICs according to our conjectural framework.

In \Cref{ssec:sicasclass}, we prove results about realizing the field $E_t$ associated to an admissible tuple $t$ as a particular abelian extension of the associated field $K$. We prove \Cref{thm:RayClassField}, showing that $E_t$ is an abelian extension of $K$. We also prove \Cref{thm:RayClassField2}, relating $E_t$ to a particular ray class field. \Cref{thm:RayClassField2plus} provides a strengthening of \Cref{thm:RayClassField2}, and \Cref{conj:RayClassField3plus} provides a strengthening of \Cref{conj:RayClassField3}.

In \Cref{ssec:cofinal}, we prove results on obtaining arbitrary abelian extensions of a real quadratic field $K$ from $r$-SICs. \Cref{thm:cofinalplus} is an unconditional number-theoretic result on the containment of certain abelian extensions in certain ray class fields, based primarily on a technical elementary number theory result given as \Cref{thm:cofinprelim}. As a corollary, we obtain \Cref{thm:cofinal}, asserting under Stark--Tate that, when the trace of the fundamental unit of $K$ is odd, the fields $E_t$ are cofinal in the set of all abelian extensions.

We need some preliminary material to set the stage for the proofs of our main theorems on SIC-generated class fields. \Cref{ssec:fieldsdiscussion} defines three important fields, $E_s^{(1)}$, $E_t^{(2)}$, and $E_t$, attached to a fiducial datum $s = (t,g,G)$. \Cref{ssec:cflemmas} proves some key lemmas showing nontriviality of certain extensions of ray class fields.

\subsection{Discussion of SIC fields}\label{ssec:fieldsdiscussion}

Recall that for an $r$-SIC projector $\Pi$, the \textit{SIC field} $E_\Pi$ was defined in \Cref{defn:sicfield} to be the field generated over $\Q$ by the overlaps $\Tr(\Pi D_\p)$ and the $\db$-th root of unity $\xi_d$. This field is called the \textit{extended projector SIC field} in \cite[Defn.~3.1]{Kopp2020c}; in that paper, it is compared to several other fields that can be associated to a SIC. The literature on SICs contains multiple definitions of fields associated to a SIC, many of which are conjecturally equal but not proven to be so; see \cite[Sec.~3]{Kopp2020c} for further discussion. In this paper, the \textit{SIC field} is always $E_\Pi$ as defined in \Cref{defn:sicfield}.

We also describe fields associated to an admissible tuple $t$ or a fiducial datum $s$. Unlike the definition of the SIC field, the formal definition of these fields is in terms of (quantities defined from) special values of the Shintani--Faddeev modular cocycle, not in terms of a SIC.

For an admissible tuple $t$, an associated field $E_t$ was defined in \Cref{dfn:SICfield}. We expand that definition to give three fields associated to a fiducial datum $s$ or an admissible tuple $t$.
\begin{defn}\label{defn:tuplefields}
    Let $s=(d,r,Q,g,G)\sim(K,j,m,Q,g,G)$ be a fiducial datum, and let $t=(d,r,Q)\sim(K,j,m,Q)$ be the corresponding admissible tuple. Define the following fields.
    \begin{itemize}
        \item[(1)]
        $E_s^{(1)}$ is the field generated over $\Q$ by the numbers $\{\overlapC{t}{\mbf{p}}\colon 0\le p_1, p_2 < d, \, \mbf{p}\neq \zero\}$.
        \item[(2)]
        $E_t^{(2)}$ is the field generated over $\Q$ by the numbers $\{\ghostOverlapC{t}{\mbf{p}}\colon 0\le p_1, p_2 < d, \, \mbf{p}\neq \zero\}$.
        \item[(3)]
        $E_t$ is the field generated over $\Q$ by the numbers $\{\ghostOverlapC{t}{\mbf{p}}\colon 0\le p_1, p_2 < d, \, \mbf{p}\neq \zero\}$ together with $\rtu_d$.
        \item[(4)]
        $\hat\sicField_t$ is the Galois closure (within $\C$) of the compositum of $K$ and $E_t$.
    \end{itemize}
\end{defn}
Recall that for $\p \not\equiv \zero \Mod{d}$ we have by \Cref{lm:OverlapTermsGhostOverlap} the formula
\begin{equation}
    \overlapC{s}{\mbf{p}}
    = g\!\left(\ghostOverlapC{t}{GH_g^{-1}G^{-1}\mbf{p}}\right)
\end{equation}
relating the live and ghost overlaps, with $GH_g^{-1}G^{-1} \in \GLtwo{\Z/\db\Z}$. It follows that $E_s^{(1)} = g(E_t^{(2)})$.

We have defined three types of fields---SIC fields, fields associated to fiducial data, and ray class fields of orders (the latter in \Cref{ssec:cft})---with no unconditional relationship between them. The remainder of this section concerns the strong conditional relationships and the strong empirical relationships between these three types of fields.

\subsection{Lemmas about class fields}\label{ssec:cflemmas}

We now turn our attention to ray class fields of orders.
For comparable level data $\mcl{L} = (\OO; \mm, \rS)$ and $\mcl{L}' = (\OO'; \mm', \rS')$, with $\OO \subseteq \OO'$, $\mm\OO' \subseteq \mm'$, and $\rS \supseteq \rS'$, there is a quotient map $\phi : \Cl_{\mm,\rS}(\OO) \surj \Cl_{\mm',\rS'}(\OO')$ and a corresponding field extension $H_{\mm,\rS}^\OO/H_{\mm',\rS'}^{\OO'}$ with $\Gal(H_{\mm,\rS}^\OO/H_{\mm',\rS'}^{\OO'}) \isom \ker(\phi)$. In order to describe the structure of this kernel, we introduce the $\U$-group notation from \cite{Kopp2020b} for unit groups with congruence and ``ray'' restrictions.

\begin{defn}\label{defn:ugroups}
For a commutative ring with unity $R$ and an ideal $I$ of $R$, define the group
\begin{equation}
\U_I(R) := \{\alpha \in R^\times : \alpha \equiv 1 \Mod{I}\} = (1+I) \cap R^\times.
\end{equation}

If $R$ has real embeddings, and $\rS$ is a subset of the real embeddings of $R$, define
\begin{equation}
\U_{I,\rS}(R) := \{\alpha \in R^\times : \alpha \equiv 1 \Mod{I} \mbox{ and } \rho(\alpha)>0 \mbox{ for } \rho \in \rS\}.
\end{equation}
\end{defn}

The following exact sequence resolves the map $\phi : \Cl_{\mm,\rS}(\OO) \surj \Cl_{\mm',\rS'}(\OO')$, describing $\ker(\phi) \isom \coker(\lambda)$ for a reduction map $\lambda$ from a certain global unit group to a certain $\Mod{\mm'}$ unit group.
\begin{thm}\label{thm:exseq}
Let $K$ be a number field, and consider level data $\mcl{L} = (\OO; \mm, \rS)$ and $\mcl{L}' = (\OO'; \mm', \rS')$ for $K$ such that $\OO \subseteq \OO'$, $\mm\OO' \subseteq \mm'$, and $\rS \supseteq \rS'$.
Let $\dd$ be any $\OO'$-ideal such that $\dd \subseteq \colonideal{\mm}{\OO'}$.
We have the following exact sequence.
\begin{equation}\label{eq:exseq}
1 \to \frac{\U_{\mm',\rS'}\!\left(\OO'\right)}{\U_{\mm,\rS}\!\left(\OO\right)} \xrightarrow{\lambda} \frac{\U_{\mm'}\!\left(\OO'/\dd\right)}{\U_{\mm}\!\left(\OO/\dd\right)} \times \{\pm 1\}^{|\rS \setminus \rS'|} \xrightarrow{\psi} \Cl_{\mm,\rS}(\OO) \xrightarrow{\phi} \Cl_{\mm',\rS'}(\OO') \to 1.
\end{equation}
\end{thm}
\begin{proof}
    See \cite[Thm.~6.5]{Kopp2020b}.
\end{proof}

The exact sequence \eqref{eq:exseq} gives us the most concrete description of $\Gal(H_{\mm,\rS}^\OO/H_{\mm',\rS'}^{\OO'})$ when the ``global units'' term $\frac{\U_{\mm',\rS'}\!\left(\OO'\right)}{\U_{\mm,\rS}\!\left(\OO\right)}$ is the trivial group. More generally, we need to know something about the global units to use \eqref{eq:exseq} effectively.

We now turn our attention to real quadratic fields and connect the $\U$-groups to stability groups of $2 \times 2$ matrices by way of the canonical representations $\chi_Q$ from \Cref{sec:units}.
\begin{lem}\label{lem:ueqstab}
    Fix a real quadratic field $K$ of discriminant $\Delta_0$, and consider any integers $d \geq 3$ and $f \geq 1$.
    Let $Q$ be any primitive binary quadratic form of discriminant $f^2\Delta_0$, and consider the canonical ring homomorphism $\chi_Q$ defined by \Cref{df:etaq}. 
    Then
    \begin{equation}\label{eq:ueqstab0}
        \chi_Q(d\OO_f) = \chi_Q(\OO_f) \cap d\MM(\Z), %
    \end{equation}
    where, following \Cref{dfn:MatrixRings}, $\MM(\Z)$ is the ring of $2 \times 2$ matrices with integer entries.
    Moreover, we have the following equalities of subgroups of $\GL_2(\Z)$.
    \begin{align}
        \chi_Q(\OO_f^\times) &= \mcl{S}(Q). \label{eq:ueqstab1} 
        \\
        \chi_Q(\U_{d\OO_f,\emptyset}(\OO_f)) &= \mcl{S}_d(Q). \label{eq:ueqstab2}
    \end{align}
\end{lem}
\begin{proof}
    We first prove \eqref{eq:ueqstab0}.
    The inclusion $\chi_Q(\OO_f) \subseteq \MM(\Z)$ holds by \Cref{tm:canisointun}; thus, $\chi_Q(d\OO_f) \subseteq \chi_Q(\OO_f) \cap d\MM(\Z)$. To prove the reverse inclusion, suppose $M \in \chi_Q(\OO_f) \cap d\MM(\Z)$, write $Q = \la a,b,c \ra$, and write 
    \begin{equation}\label{eq:Mdivbymeq}
        M 
        = \chi_Q(x+y\sqrt{\Delta_0})
        = xI + \frac{2y}{f}SQ
        = \mattwo{x-\frac{by}{f}}{-\frac{2cy}{f}}{\frac{2ay}{f}}{x+\frac{by}{f}}
    \end{equation}
    for $x,y \in \Q$.
    It follows that $\frac{2ay}{f} \in d\Z$, $\frac{2cy}{f} \in d\Z$, and $\frac{2by}{f} = (x+\frac{by}{f}) - (x-\frac{by}{f}) \in d\Z$. Because $Q$ is primitive, $\gcd(a,b,c)=1$, so $\frac{2y}{f} \in d\Z$; that is, $y \in \frac{df}{2}\Z$.
    We complete the proof of \eqref{eq:ueqstab0} by cases.

    \textit{Case 1:} Suppose $2 \div f^2\Delta_0$.
    Then, since $b^2 - 4ac = f^2\Delta_0$, it follows that $b$ is even. Thus, $\frac{by}{f} \in d\Z$. It follows from \eqref{eq:Mdivbymeq} that $x-\frac{by}{f} \in d\Z$, and thus, $x = (x-\frac{by}{f}) + b\frac{y}{f} \in d\Z$. In this case,
    \begin{equation}
        d\OO_f = \left\{x+y\sqrt{\Delta_0} : x \in d\Z, y \in \tfrac{df}{2}\Z\right\},
    \end{equation}
    and so $x+y\sqrt{\Delta_0} \in d\OO_f$.

    \textit{Case 2:} Suppose $2 \ndiv f^2\Delta_0$.
    Since $b^2 - 4ac = f^2\Delta_0$, it follows that $b$ is odd; also, $f$ is odd. It follows from \eqref{eq:Mdivbymeq} that $x-\frac{by}{f} \in d\Z$.
    Thus, 
    $x = (x-\frac{by}{f}) + b\frac{y}{f} \in \frac{d}{2}\Z$, 
    and $x-y = (x-\frac{by}{f}) + (f-b)\frac{y}{f} \in d\Z$. In this case,
    \begin{equation}
        d\OO_f = \left\{x+y\sqrt{\Delta_0} : x \in \tfrac{d}{2}\Z, y \in \tfrac{df}{2}\Z, x-y \in d\Z\right\},
    \end{equation}
    and so $x+y\sqrt{\Delta_0} \in d\OO_f$.

    Both case being proven, the proof of \eqref{eq:ueqstab0} is complete.
    Equation ~\eqref{eq:ueqstab1} is a restatement of \Cref{tm:qfpmst}(2), which we have already proven. It remains to prove \eqref{eq:ueqstab2}.%

    We first show that $\chi_Q(\U_{d\OO_f,\emptyset}(\OO_f))$ is a subgroup of $\SL_2(\Z)$.
    Consider any $\eta \in \U_{d\OO_{f},\emptyset}(\OO_{f})$.
    If $\eta'$ is the nontrivial Galois conjugate of $\eta$, then $\eta \equiv 1 \Mod{d\OO_1}$ and $\eta' \equiv 1 \Mod{d\OO_1}$, so $\Nm(\eta) = \eta\eta' \equiv 1 \Mod{d\OO_1}$. But $\Nm(\eta) = \pm 1$ because $\eta$ is a unit, and $d \geq 3$, so $\Nm(\eta) = 1$.
    Thus, $\det(\chi_Q(\eta)) = \Nm(\eta) = 1$. So $\chi_Q(\U_{d\OO_f,\emptyset}(\OO_f))$ is a subgroup of $\SL_2(\Z)$.

    We have $\U_{d\OO_f,\emptyset}(\OO_f) = \OO_f^\times \cap (1 + d\OO_f)$.
    Because $\chi_Q$ is an injective homomorphism, it follows that $\chi_Q(\U_{d\OO_f,\emptyset}(\OO_f)) = \chi_Q(\OO_f^\times) \cap (I + \chi_Q(d\OO_f))$. Thus, by \eqref{eq:ueqstab0} and \eqref{eq:ueqstab1},  $\chi_Q(\U_{d\OO_f,\emptyset}(\OO_f)) = \mcl{S}(Q) \cap d\MM(\Z)$. But since $\chi_Q(\U_{d\OO_f,\emptyset}(\OO_f))$ is a subgroup of $\SL_2(\Z)$, in fact $\chi_Q(\U_{d\OO_f,\emptyset}(\OO_f)) = \mcl{S}(Q) \cap \Gamma(d) = \mcl{S}_d(Q)$.
\end{proof}

The focus now narrows to the ray class fields of interest for the our construction of $r$-SICs. We will fix a real quadratic field $K$ and associate to it the sequence of conductors $f_j$ from \Cref{dfn:sequenceofconductors} and the dimension grid $d_{j,m}$ and rank grid $r_{j,m}$ from \Cref{dfn:fjrjmdjm}.

The following lemma is the key technical result on global units that will allow us some control over the behavior of the ray class groups and ray class fields of interest for $r$-SICs. It generalizes \cite[Lem.~5.3]{Kopp2019} and \cite[Lem.~B.2]{Kopp2020c} by using \Cref{lem:ueqstab} together with results proven in \Cref{sec:units}.

\begin{lem}\label{lem:ugroupgen}
    Fix a real quadratic field $K$ of discriminant $\Delta_0$.
    Let $j,m \in \N$.
    Let $d = d_{j,m}$ and $f \div f_{j}$.
    Let $\rS$ be a subset of the real embeddings of $K$.
    Then
    \begin{equation}
        \U_{d\OO_f,\rS}(\OO_f) = \langle \vn_{d_j}^{2m+1} \rangle.
    \end{equation}
\end{lem}
\begin{proof}
    Write $\vn_{d_j} = \vn^{j}$, where $\vn$ is the smallest totally positive unit of $K$ with $\vn > 1$.
    By \eqref{eq:epowerminusone} of \Cref{lem:dimgridtechres},
    \begin{equation}
        \vn_{d_j}^{2m+1} - 1 = \vn^{(2m+1)j}-1 = d_{j,m} \vn^{mj} (\vn^j-1).
    \end{equation}
    Also, $\vn_{d_j} \in \OO_{f_j}$. Thus, $\vn_{d_j}^{2m+1} \equiv 1 \Mod{d_{j,m}\OO_{f_j}}$, so
    \begin{equation}
        \vn_{d_j}^{2m+1} \in \U_{d\OO_f,\{\infty_1,\infty_2\}}(\OO_{f}).
    \end{equation}
    It follows that $\vn_{d_j}^{2m+1} \in \U_{d\OO_{f_j},\rS}(\OO_{f_j})$, because $\U_{d\OO_f,\{\infty_1,\infty_2\}}(\OO_{f}) \subseteq \U_{d\OO_{f_j},\rS}(\OO_{f_j})$.
    
    Since $\U_{d\OO_{f_j},\rS}(\OO_{f_j}) \subseteq \U_{d\OO_{1},\emptyset}(\OO_{1})$, It suffices to show that $\vn_{d_j}$ generates $\U_{d\OO_{1},\emptyset}(\OO_{1})$, which we will now do. 

    Let $Q$ be any primitive binary quadratic form of discriminant $\Delta_0$, and consider the canonical ring homomorphism $\chi_Q$ defined by \Cref{df:etaq}. By \eqref{eq:ueqstab2} of \Cref{lem:ueqstab},
    \begin{equation}\label{eq:chiQSd}
        \chi_Q(\U_{d\OO_1,\emptyset}(\OO_1)) = \mcl{S}_d(Q).
    \end{equation}
    By \eqref{eq:QStabilityGroupGenerators2} of \Cref{tm:symgp}, 
    \begin{equation}\label{eq:SdchiQ}
        \mcl{S}_d(Q) 
        = \la A_t \ra 
        = \left\la \chi_Q(\vn_{d_j}^{2m+1}) \right\ra
        = \chi_Q\!\left(\la\vn_{d_j}^{2m+1}\ra\right).
    \end{equation}
    It follows from \eqref{eq:chiQSd} and \eqref{eq:SdchiQ}, and the fact that $\chi_Q$ is an injective homomorphism, that $\U_{d\OO_1,\emptyset}(\OO_1) = \la\vn_{d_j}^{2m+1}\ra$.
\end{proof}

As a consequence, we deduce the following lemma showing that varying the set of infinite primes in the ray class modulus produces distinct ray class groups and ray class fields.

\begin{lem}\label{lem:diamond}
    Fix a real quadratic field $K$.
    Let $j,m \in \N$.
    Let $d = d_{j,m}$, $f \div f_{j}$, and $d' \in \{d,\db\}$. 
    The all the group homomorphism shown in the following diagram are $2$-to-$1$ surjections.
    \begin{equation}\label{eq:classdiamond}
    \begin{tikzcd}[column sep=tiny]
    & \Cl_{d'\infty_1\infty_2}\!(\OO_f)\ar[dr] \ar[dl] & \\
    \Cl_{d'\infty_1}\!(\OO_f) \ar[dr] & & \Cl_{d'\infty_2}\!(\OO_f) \ar[dl] \\
    & \Cl_{d'}\!(\OO_f) &
    \end{tikzcd}
    \end{equation}
    In the following field diagram, all of the extensions have degree $2$.
    \begin{equation}\label{eq:fielddiamond}
    \begin{tikzcd}[column sep=tiny]
    & H_{d'\infty_1,\infty_2}^{\OO_f}\ar[dr,dash] \ar[dl,dash] & \\
    H_{d'\infty_1}^{\OO_f} \ar[dr,dash] & & H_{d'\infty_2}^{\OO_f} \ar[dl,dash] \\
    & H_{d'\OO}^{\OO_f} &
    \end{tikzcd}
    \end{equation}
\end{lem}
\begin{proof}
    We first prove the claim in the case $d'=d$.
    Let $\OO = \OO_f$.
    Let $\rS', \rS$ be subsets of the set of real embeddings of $K$.
    We apply the exact sequence of \cite[Thm.~6.5]{Kopp2020b} with level data $\mcl{L} = (\OO; d'\OO,\rS)$ and $\mcl{L'} = (\OO; d'\OO,\emptyset)$ and with $\dd = d\OO$:
    \begin{equation}\label{eq:sesdiam}
        1 \to
        \frac{\U_{d\OO,\rS'}(\OO)}{\U_{d\OO,\rS}(\OO)} \to
        \frac{\U_{d\OO}(\OO/d\OO)}{\U_{d\OO}(\OO/d\OO)} \times \{\pm 1\}^{\abs{\rS \setminus \rS'}} \to
        \Cl_{d\OO,\rS}(\OO) \to
        \Cl_{d\OO,\rS'}(\OO) \to
        1.
    \end{equation}
    By \Cref{lem:ugroupgen}, the ``global units'' term is
    \begin{equation}
        \frac{\U_{d\OO,\rS'}(\OO)}{\U_{d\OO,\rS}(\OO)} 
        = \frac{\langle \e_{d_j}^{2m+1} \rangle}{\langle \e_{d_j}^{2m+1} \rangle}
        = 1.
    \end{equation}
    We may thus rewrite \Cref{eq:sesdiam} as the short exact sequence
    \begin{equation}
        1 \to
        \{\pm 1\}^{\abs{\rS \setminus \rS'}} \to
        \Cl_{d\OO,\rS}(\OO) \to
        \Cl_{d\OO,\rS'}(\OO) \to
        1.
    \end{equation}
    It follows that the ray class groups 
    $\Cl_{d}(\OO) = \Cl_{d\OO,\emptyset}(\OO)$,
    $\Cl_{d\infty_1}(\OO) = \Cl_{d\OO,\{\infty_1\}}(\OO)$,
    $\Cl_{d\infty_2}(\OO) = \Cl_{d\OO,\{\infty_2\}}(\OO)$,
    $\Cl_{d\infty_1\infty_2}(\OO) = \Cl_{d\OO,\{\infty_1,\infty_2\}}(\OO)$
    are all distinct and have the $2$-to-$1$ surjections shown in \eqref{eq:classdiamond}.
    The fact that the field extensions in the ray class field diamond \eqref{eq:fielddiamond} have degree $2$ follows by \Cref{thm:rayclassfield}.

    In the case $d' = \db$, note that since $\U_{d\OO,\rS'}(\OO) = \U_{d\OO,\rS}(\OO)$, it follows that $\U_{\db\OO,\rS'}(\OO) = \U_{\db\OO,\rS}(\OO)$, because the latter two groups are subgroups of the former obtained by imposing the condition $u \equiv 1 \Mod{\db}$. Therefore, the same proof applies.
\end{proof}

\subsection{SIC fields as class fields}\label{ssec:sicasclass}
\label{sbsc:rayclassproof}

\Cref{thm:RayClassField} is a straightforward consequence of our conjectures about the RM values of the Shintani--Faddeev cocycle. 

\begin{proof}[Proof of \Cref{thm:RayClassField}]
    Let $t=(d,r,Q)\sim(K,j,m,Q)$ be an admissible tuple. We wish to show that $E_t$ is an abelian extension of $K$.

    The field $E_t$ is defined to be the field generated over $\Q$ by the candidate ghost overlaps $\tilde\mu_\p(t)$ together with $\xi_d$. The candidate ghost overlaps are $\tilde\mu_\zero(t) = 1$ and
    \begin{equation}\label{eq:mushin}
        \tilde\mu_\p(t)
        = \frac{1}{\sqrt{d_j+1}}\, \tilde\nu_\p(t)
        = \frac{1}{\sqrt{d_j+1}}\, \phi_\p(t) \sfc{d^{-1}\p}{A}{\qrt_t} \text{ for } \p \not\equiv \zero \Mod{d}.
    \end{equation}
    Note that $\sqrt{d_j+1}$ and $\phi_\p(t)$ (a root of unity) are both contained in abelian extensions of $\Q$, so in particular they are both are contained in an abelian extension of $K$. By the conditional assumption (of \Cref{conj:frmvc} when $Q$ is fundamental and \Cref{conj:grmvc} when $Q$ is not fundamental), $\sfc{d^{-1}\p}{A_t}{\qrt_t}$ is also in an abelian extension of $K$. We have
    \begin{equation}
        E_t \subseteq K\!\left(\sqrt{d_j+1}, \xi_d, \phi_\p(t), \sfc{d^{-1}\p}{A_t}{\qrt_t} : 0 \leq p_1,p_2 < d\right),
    \end{equation}
    and the right-hand field is abelian over $K$, so $E_t$ is abelian over $K$.
\end{proof}

We prove the following theorem, which implies \Cref{thm:RayClassField2}.
\begin{thm}\label{thm:RayClassField2plus}
    Assume \Cref{conj:stc} (the Stark--Tate Conjecture).
    Let $s=(d,r,Q,G,g)\sim(K,j,m,Q,G,g)$ be a fiducial datum, 
    let $t=(d,r,Q)\sim(K,j,m,Q)$,
    and let $d=d_{j,m}$.
    Suppose that $\disc(Q)$ is fundamental.
    Choose $\p_1 = \smcoltwo{p_{11}}{p_{12}}$ such that $(p_{12}\qrt_t-p_{11})\OO_1$ is coprime to $d\OO_1$ as $\OO_1$-ideals. (For example, one may take $\p_1 = \smcoltwo{-1}{0}$.)
    \begin{itemize}
        \item[(1)]
        The ray class field $H^{\OO_1}_{d\infty_1}$ 
        is equal to the field extension of $K$ generated by the numbers $\{\overlapC{t}{\mbf{p}}\colon 0\le p_1, p_2 < d, \, \mbf{p}\neq \zero\}$ and is also equal to $\Q(\overlapC{t}{\mbf{p_1}}^2)$.
        The field $E_s^{(1)} \supseteq H^{\OO_1}_{d\infty_1} \supseteq K$, the extension $E_s^{(1)}/K$ is ramified at $\infty_1$ and unramified $\infty_2$, and field $E_s^{(1)}$ depends only on the pair $(d,r)$.
        \item[(2)]
        The ray class field $H^{\OO_1}_{d\infty_2}$ 
        is equal to the field extension of $K$ generated by the numbers $\{\ghostOverlapC{t}{\mbf{p}}^2\colon 0\le p_1, p_2 < d, \, \mbf{p}\neq \zero\}$ and is also equal to $\Q(\ghostOverlapC{t}{\mbf{p_1}}^2)$.
        The field $E_t^{(2)} \supseteq H^{\OO_1}_{d\infty_2} \supseteq K$, the extension $E_t^{(2)}/K$ is unramified at $\infty_1$ and ramified $\infty_2$, and field $E_t^{(2)}$ depends only on the pair $(d,r)$.
        \item[(3)]
        The ray class field $H^{\OO_1}_{\db\infty_1\infty_2}$ 
        is equal to the field extension of $K$ generated by the numbers $\{\ghostOverlapC{t}{\mbf{p}}^2\colon 0\le p_1, p_2 < d, \, \mbf{p}\neq \zero\}$ together with $\rtu_d$, and it is also equal to $K(\ghostOverlapC{t}{\mbf{p_1}}^2, \rtu_d)$.
        The field $E_t \supseteq H^{\OO_1}_{\db\infty_1\infty_2} \supseteq K$, the extension $E_t/K$ is ramified at both infinite places of $K$, and field $E_t$ depends only on the pair $(d,r)$.
        \item[(4)]
        Assume \Cref{cnj:tci}. Then the SIC field $E_{\Pi_s} = E_s^{(1)}(\xi_d) \supseteq H^{\OO_1}_{\db\infty_1\infty_2} \supseteq K$.
    \end{itemize}
\end{thm}
\begin{proof}
    We will prove (2) first, and the others will follow.

    Consider $\p \in \left(\Z/d\Z\right)^2$. As in the proof of \Cref{thm:nupnumpeq1}, let $\A_\p$ be the unique class in $\Clt_{d\infty_2}(\OO_1)$ that maps to the $\SL_2(\Z)$-orbit of $(d^{-1}\p,\qrt_t)$ under the map $\Upsilon_{d\OO_1}$ described in \cite[Thm.\ 3.12]{Kopp2020d} and at the end of \Cref{ssec:qpochmain}.
    By \Cref{eq:nusquared}, we have for
    \begin{equation}\label{eq:nuepplus}
    (d_j+1) \ (\ghostOverlapC{t}{\p})^2  
    = (\normalizedGhostOverlapC{t}{\p})^2 
    = \exp\!\left(n Z_{d\infty_2}'(0,\A_\p)\right) 
    = \su_{\A_\p}^{-n_\p}
    \end{equation}
    for some $n_\p \in \{1,2\}$ and $\su_{\A_\p} = \exp\!\left(- Z_{d\infty_2}'(0,\A_\p)\right)$.
    By \Cref{prop:starkimp}, our conditional assumptions imply $\MS(\OO_1,\mm)$ (\Cref{conj:msc}) for all nonzero $\OO_1$-ideals $\mm$ such that $\mm \neq \OO_1$. In particular, they imply that $\su_{\A_\p} \in H^{\OO_1}_{d\infty_2}$. It follows that $(\normalizedGhostOverlapC{t}{\p})^2 \in H^{\OO_1}_{d\infty_2}$, or equivalently $(\ghostOverlapC{t}{\p})^2 \in H^{\OO_1}_{d\infty_2}$.

    Now restrict to the special case of $\p = \p_1 := \smcoltwo{-1}{0}$. Let $\su := \su_{\A_{\p_1}}$.
    By definition, $n_{\p_1} = \frac{2}{\abs{\phi^{-1}(\A_{\p_1})}}$, where $\phi : \Clt_{\mm\infty_1\infty_2}(\OO_1) \to \Clt_{\mm\infty_2}(\OO_1)$ is the natural quotient map.
    The class $\A_{\p_1}$ is primitive, so $\abs{\phi^{-1}(\A_{\p_1})}$ is equal to the cardinality of the kernel of the map $\Cl_{\mm\infty_1\infty_2}(\OO_1) \to \Cl_{\mm\infty_2}(\OO_1)$, which by \Cref{lem:diamond} is $2$. Thus $n_{\p_1} = \frac{2}{2} = 1$, and \eqref{eq:nuepplus} becomes $(\normalizedGhostOverlapC{t}{\p_1})^2 = \e^{-1}$.
    Because $f=1$ and $\A_{\p_1}$ is primitive, the number $\e$ is a Stark unit in the original sense of \cite{Stark3}. The nontriviality of the maps between the ray class groups with different ramification at infinite places shown in \Cref{lem:diamond} implies the non-vanishing condition in the hypotheses of \cite[Thm.~1]{Stark3}. Applying that theorem, which says that the Stark unit generates the ray class field over the rational numbers, we obtain $\Q((\normalizedGhostOverlapC{t}{\p_1})^2) = \Q(\e) = H^{\OO_1}_{d\infty_2}$.

    The claim that $E_t^{(2)} \supseteq H^{\OO_1}_{d\infty_2}$ follows by the definition of $E_t^{(2)}$. The field $H^{\OO_1}_{d\infty_2}$ is always unramified at $\infty_1$, and it is ramified at $\infty_2$ if and only if it is a nontrivial extension of $H^{\OO_1}_{d}$, which is true by \Cref{lem:diamond}. The field $E_t^{(2)}$ is obtained from $H^{\OO_1}_{d\infty_2}$ by adjoining square roots of numbers $(\normalizedGhostOverlapC{t}{\p})^2$ that are positive in the first real embedding, so it remains unramified at $\infty_1$. It also follows from the ``Artin map action'' part of \Cref{conj:msc}, together with \eqref{eq:nuepplus}, that $\Gal(H^{\OO_1}_{d\infty_2}/K)$ permutes the $(\ghostOverlapC{t}{\p})^2$ with the ghost overlaps $(\ghostOverlapC{t'}{\p})^2$ of all $t' = (d,r,Q')$ such that $\disc(Q')$ is fundamental. Thus $E_t^{(2)}$ depends only on the pair $(d,r)$.

    Claim (1) follows from (2), because $H_{d\infty_1}^{\OO_1} = g(H_{d\infty_2}^{\OO_1})$ and $E_s^{(1)} = g(E_t^{(2)})$, and using \Cref{lm:OverlapTermsGhostOverlap}; independence from $g$ follows from the fact (given by \Cref{thm:RayClassField}) that $E_t^{(2)}/K$ is abelian, so $E_t^{(2)}$ can have no more that two conjugate fields over $\Q$, those being $E_s^{(1)}$ and $E_t^{(2))}$. Claim (3) also follows, because $H_{\db\infty_1\infty_2}^{\OO_1} = H_{d\infty_1}^{\OO_1}\!(\xi_d) = H_{d\infty_2}^{\OO_1}\!(\xi_d)$ and $E_t = E_t^{(2)}(\xi_d)$.

     Additionally, if one assumes \Cref{cnj:tci}, then $\Pi_s$ is a SIC fiducial projector by \Cref{thm:rayclassfieldrsicgen}, so $E_{\Pi_s}$ is well-defined. Then it follows from the definition of each that $E_{\Pi_s} = E_s^{(1)}(\xi_d)$, and so claim (4) follows from (1).
\end{proof}

\begin{proof}[Proof of \Cref{thm:RayClassField2}]
    This is \Cref{thm:RayClassField2plus}(3).
\end{proof}

\Cref{thm:RayClassField2plus} does not pin down the fields $E_s^{(1)}$, $E_t^{(2)}$, and $E_t$ precisely but only says that they are abelian extensions containing certain ray class fields. This is because the Stark--Tate Conjecture does not provide a precise description of the field generated by the square root of a Stark unit as a particular subfield of a ray class field. \Cref{thm:RayClassField2plus} also applies only in the case of fundamental discriminants, because the theory of ray class partial zeta functions for non-maximal orders is insufficiently developed at present to provide the analogue of \cite[Thm.~1]{Stark3}.

Nevertheless, the numerical data (discussed further in \Cref{sec:sicphenexplain,sec:nec} as well as in \cite{Kopp2020c}) supports a precise prediction about the structure of the fields $E_s^{(1)}$, $E_t^{(2)}$, and $E_t$; this prediction includes the forms of non-fundamental discriminant. We presented the prediction as a conjecture. We note that this conjecture implies \Cref{conj:RayClassField3} and also predicts the shape of the fields $\hat{E}_t$ and $E_{\Pi_s}$.

\begin{conj}\label{conj:RayClassField3plus}
    Let $s=(d,r,Q,G,g)\sim(K,j,m,Q,G,g)$ be a fiducial datum, 
    let $t=(d,r,Q)\sim(K,j,m,Q)$,
    let $d=d_{j,m}$, 
    and let $f$ be the conductor of $Q$. Then:
    \begin{itemize}
        \item[(1)]
        $E_s^{(1)} = H^{\OO_f}_{\db\infty_1}$.
        \item[(2)]
        $E_t^{(2)} = H^{\OO_f}_{\db\infty_2}$.
        \item[(3)]
        $E_t = H^{\OO_f}_{\db\infty_1\infty_2}$.
    \end{itemize}
\end{conj}

\begin{prop}\label{prop:RayClassField3plus}
Let $s=(d,r,Q,G,g)\sim(K,j,m,Q,G,g)$ be a fiducial datum, and let $t=(d,r,Q)\sim(K,j,m,Q)$.
\begin{itemize}
    \item[(1)]
    Assuming \Cref{conj:RayClassField3plus}, one has $\hat{E}_t = E_t$, and \Cref{conj:RayClassField3} follows.
    \item[(2)]
    Assuming \Cref{conj:RayClassField3plus} and \Cref{cnj:tci}, one has $E_{\Pi_s} = E_t$.
\end{itemize}
\end{prop}
\begin{proof}
    The field $H^{\OO_f}_{\db\infty_1\infty_2}$ contains $K$ and is Galois over $\Q$; thus, by \Cref{conj:RayClassField3plus}, $\hat{E}_t = E_t$.
    It is then clear that \Cref{conj:RayClassField3} follows by \Cref{conj:RayClassField3plus}(3).
    Additionally, if one assumes \Cref{cnj:tci}, then $\Pi_s$ is a SIC fiducial projector by \Cref{thm:rayclassfieldrsicgen}, so $E_{\Pi_s}$ is well-defined. Then it follows from the definition of each that $E_{\Pi_s} = E_s^{(1)}(\xi_d)$, and so by \Cref{conj:RayClassField3plus}, the SIC field $E_{\Pi_s} = H^{\OO_f}_{\db\infty_1}(\xi_d) = H^{\OO_f}_{\db\infty_1\infty_2} = E_t$.
\end{proof}

\subsection{The set of SIC-generated abelian extensions}\label{ssec:cofinal}

We now examine the implications of our conjectural framework for the generation of arbitrary abelian extensions of real quadratic fields.
We begin with two preliminary results.

\begin{defn}
For any prime number $p$ and any $r \in \Q^\times$, denote by $v_p(r)$ the \textit{$p$-adic valuation} of $r$; that is, if $r=p^e\frac{a}{b}$ with $p\ndiv a$ and $p\ndiv b$, then $v_p(r)=e$.
\end{defn}

\begin{lem}\label{lm:binpdiv}
    Let $p,r,e\in \mbb{N}$ be such that $p$ is prime and $1\le r \le p^e$. Then the following statements of about the $p$-adic valuations of binomial coefficients hold:
    \eag{
      v_p\!\left(\binom{p^e}{r}\right) &= e-v_p(r),
      \label{eq:binpdiva}
      \\
      v_p\!\left(p^r \binom{p^e}{r}\right) &\ge e+1.
      \label{eq:binpdivb}
    }
\end{lem}
\begin{proof}
    We prove \eqref{eq:binpdiva} by induction.  The statement is immediate if $r=1$.  Suppose it is true for arbitrary $1\le r < p^e$.  Then
    \eag{
     v_p\!\left(\binom{p^e}{r+1}\right) &=
     v_p\!\left( \frac{p^e-r}{r+1}\binom{p^e}{r}\right)
     = e-v_p(r) + v_p(p^e-r)-v_p(r+1).
    }
    Since $v_p(p^e-r) = v_p(r)$, it follows that
    \eag{
      v_p\!\left(\binom{p^e}{r+1}\right) &=
      e-v_p(r+1).
    }
    Equation ~\eqref{eq:binpdivb} follows from \eqref{eq:binpdiva} and the fact that $r-v_p(r) \ge 1$.
\end{proof}
\begin{lem}\label{lm:ordv2zk}
    Let $\vn$ be the totally positive fundamental unit of the quadratic field $K$ and $d_1 = \frac{\vn+\vn^{-1}}{2}+1$ the associated root dimension.
    The order of $\vn+2\mcl{O}_K$ as an element of $\left(\mcl{O}_K/2\mcl{O}_K\right)^{\times}$ is
    \eag{
     \#\!\left\langle \vn+2\mcl{O}_K \right\rangle
     &=
     \begin{cases}
         1, \quad & \text{if $d_1 \equiv 3 \Mod{4}$},
         \\
         2, \quad & \text{if $d_1 \equiv 1 \Mod{4}$},
         \\
         3, \quad & \text{if $d_1$ is even}.
     \end{cases}
    }
\end{lem}
\begin{proof}
    Suppose $d_1 = 4n+3$ for $n\in \mbb{N}$.  Then
    \eag{
     \vn+2\mcl{O}_K&= \frac{4n+2+4\sqrt{n(n+1)}}{2}+2\mcl{O}_K= 
    1+ 2\mcl{O}_K.
    }

    Suppose $d_1 = 4n+1$ for $n\in \mbb{N}$. Then
    \begin{alignat}{2}
    && \qquad \vn+2\mcl{O}_K&= \frac{4n+\sqrt{(4n-2)(4n+2)}}{2}+2\mcl{O}_K
    \nn
    &&\qquad & = \sqrt{4n^2-1} +2\mcl{O}_K
    \nn
    &\implies & \qquad \vn^2 + 2\mcl{O}_K &= 4n^2-1+2\mcl{O}_K
    \nn
    &&\qquad & = 1 + 2\mcl{O}_K.
    \end{alignat}

    Finally, suppose $d_1$ is even.  Then
    \eag{
        \vn^j -1 &= \frac{d_j-3-f_j\Delta_0}{2} + 
        f_j \left(\frac{\Delta_0+\sqrt{\Delta_0}}{2}\right).
    }
    So $\vn^j \equiv 1 \Mod{2\mcl{O}_K}$ if and only if $(d_j-3-f_j\Delta_0)/2$ and $f_j$ are both even.  It follows from Lemma~\ref{lem:dfdelprops} that $f_1$, $f_2$ are both odd, implying that the order is greater than 2.  It also follows that $f_3$ is even and $\Delta_0 \equiv  1 \Mod{4}$.  Using Lemma~\ref{lm:tustarprops} we  have
    \eag{
        d_3-3-f_3\Delta_0 &= d_1^2(d_1-3)+f_1d_1(d_1-2)\Delta_0 \equiv 0 \Mod{4}.
    }
    Hence $\vn^3 \in 1+2\mcl{O}_K$.
\end{proof}
\begin{thm}\label{thm:cofinprelim}
Let $d$ be any positive integer, and let $r$ be the order of $\vn+d\mcl{O}_K$ as an element of $\left(\mcl{O}_K/d\mcl{O}_K\right)^{\times}$, where $\vn$ is the totally positive fundamental unit of the quadratic field $K$.  Write $d$, $r$ in the form  
\eag{
d&= 2^{\ell_1}(2a+1), & r &= 2^{\ell_2}(2b+1),
}
for suitable non-negative integers $\ell_1$, $\ell_2$, $a$, $b$.  Let $\ell = \max\{\ell_1,\ell_2\}$, and let $j, m\in \mbb{N}$ be such that
\eag{
2^{\ell} & \div j, & (2a+1)(2b+1)&\div (2m+1).
}
If $d$ is even assume in addition that $j$ is coprime to 3 and $2m+1$ is a multiple of 3.  Then
\begin{enumerate}
    \item \label{it:cofinprelima} If $d_1$ is even, then $d\div d_{j,m}$,
    \item \label{it:cofinprelimb} If $d_1$ is odd, then $d\div d_{j,m}$ if and only if $d$ is odd.
\end{enumerate}
\end{thm}
\begin{proof}
   Let $p^e$ be any element in the prime decomposition of $2a+1$.  Since $j(2b+1) \div r$
   \eag{
        \vn^{j(2b+1)} &= 1 + d z
   }
   for some $z\in \mcl{O}_K$.  It then follows from Lemma~\ref{lm:binpdiv} that 
   \eag{
        \vn^{jp^e(2b+1)}&= 1+\sum_{t=1}^{p^e}(dz)^t\binom{p^e}{t}
        =1+p^e dz z'
        =1+p^e(\vn^{j(2b+1)}-1) z'
   }
   for some $z'\in \mcl{O}_K$.  In view of Lemma~\ref{lem:dimgridtechres} this means
   \eag{
         \vn^{jp^e(2b+1)}&= 1 + p^e d_{b,j} \vn^{bj}(\vn^j-1) z'
         =1+p^e (\vn^j-1) z''
   }
    for some $z''\in \mcl{O}_K$.  Since $p^e(2b+1)\div (2m+1)$ it follows that
    \eag{
         \vn^{j(2m+1)}&=1 + p^e(\vn^j-1) z'''
    }
    for some $z'''\in \mcl{O}_K$.  By another application of Lemma~\ref{lem:dimgridtechres} we deduce
    \begin{alignat}{2}
     && \qquad   d_{j,m} \vn^{mj}(\vn^j-1)&= p^e(\vn^j-1) z'''
     \nn
     &\implies & \qquad d_{j,m} &= p^e w
    \end{alignat}
    for some $w\in \mcl{O}_K$. We conclude that $p^e$, and consequently $(2a+1)$ divides $d_{j,m}$.  If $d$ is odd this proves $d\div d_{j,m}$.  
    
    Suppose, on the other hand, that $d$ is even. Suppose, first, that $d_1$ is also even.  By assumption $3\div (2m+1)$, so it follows from Lemma~\ref{lm:ordv2zk} that
    \eag{
        \vn^{2m+1}&= 1+ 2 z
    }
    for some $z\in \mcl{O}_K$.  In view of Lemma~\ref{lm:binpdiv} this means 
    \eag{
        \vn^{(2m+1)2^{\ell}} &= 
        1+\sum_{t=1}^{2^{\ell}} (2z)^t\binom{2^{\ell}}{t} = 1+2^{\ell+1}z z'
        = 1+2^{\ell}(\vn^{2m+1}-1)z'
    }
    for some $z'\in \mcl{O}_K$. Since $2^{\ell} \div j$ it follows that
    \eag{
        \vn^{(2m+1)j}&= 1+2^{\ell}(\vn^{2m+1}-1)z''
    }
    for some $z''\in \mcl{O}_K$.
    Using Lemma~\ref{lem:dimgridtechres} we deduce
    \begin{alignat}{2}
      &&  \qquad  d_{j,m} \vn^{mj} (\vn^j-1) &= 2^{\ell}d_{1,m}\vn^m (\vn-1) z''
      \nn
      &\implies &\qquad d_{j,m} (\vn^j-1) &= 2^{\ell} w
    \end{alignat}
      for some $w\in \mcl{O}_K$.  Since $j$ is coprime to $3$,
      \eag{
        \vn^j -1 &= c_1 + c_2\left(\frac{\Delta_0 + \sqrt{\Delta_0}}{2}\right)
      }
      for $c_1, c_2 \in \mbb{Z}$ not both even.  Since $2^{\ell}$ divides both $c_1 d_{j,m}$ and $c_2 d_{j,m}$ it follows that it must divide $d_{j,m}$, which proves statement~\eqref{it:cofinprelima}.

      Suppose, on the other hand, that $d_1$ is odd. Then it follows from Lemmas~\ref{lem:dfdelprops} and~\ref{lem:dimgridtechres} that $d_{j,m}$ is odd.  So $d\ndiv d_{j,m}$, which proves statement~\eqref{it:cofinprelimb}.
\end{proof}

\begin{thm}\label{thm:cofinalplus}
    Let $K$ be a real quadratic field of discriminant $\Delta_0$, let $\vn$ be a fundamental totally positive unit in $K$ (as in \Cref{dfn:fundamentalTotallyPositiveUnit}), and let $f_j$ be as defined in \Cref{dfn:sequenceofconductors}. Then
    \begin{enumerate}
        \item[(1)] If $\Tr(\vn)$ is odd, then every abelian extension of $K$ is contained in $H_{d_{j,m}\infty_1\infty_2}^{\OO_1}$ for some $j, m \in \N$.
        \item[(2)] If $\Tr(\vn)$ is even, $d$ is a positive odd integer, and $f|f_{j_0}$ for some positive integer $j_0$ such that $3 \ndiv j_0$, then there exists some $j, m \in \N$ such that $H_{d\infty_1\infty_2}^{\OO_f} \subseteq H_{d_{j,m}\infty_1\infty_2}^{\OO_f}$.
        \item[(2')] If $\Tr(\vn)$ is even, then every abelian extension of $K$ that is unramified at the primes of $K$ lying over $2$ is contained in $H_{d_{j,m}\infty_1\infty_2}^{\OO_1}$ for some $j, m \in \N$.
    \end{enumerate}
\end{thm}
\begin{proof}
    We prove (1) first. If $\Tr(\vn)$ is odd, then the root dimension $d_1$ is even. Let $E$ be any abelian extension of $K$. By the Takagi Existence Theorem, there is some $d \in \N$ such that $E \subseteq H^{\OO_1}_{d\infty_1\infty_2}$. It follows from \Cref{thm:cofinprelim}(1) that there are some $j,m \in \N$ such that $d \div d_{j,m}$. Thus, $E \subseteq H^{\OO_1}_{d_{j,m}\infty_1\infty_2}$.

    Now we prove (2). If $\Tr(\vn)$ is even, then the root dimension $d_1$ is odd. Let $r, \ell_1,\ell_2,\ell,a,b\in \mbb{N}$ be as in the statement of \Cref{thm:cofinprelim}. Set $j := 2^\ell j_0$ and $2m+1 := (2a+1)(2b+1)$. Then by \Cref{thm:cofinprelim}(2) we have $d \div d_{j,m}$. Thus, $H_{d\infty_1\infty_2}^{\OO_f} \subseteq H_{d_{j,m}\infty_1\infty_2}^{\OO_f}$.

    Finally, we prove (2'). Let $E$ be any abelian extension of $K$ that is unramified at the primes of $K$ lying over $2$. By the Takagi Existence Theorem, there is some odd $d \in \N$ such that $E \subseteq H^{\OO_1}_{d\infty_1\infty_2}$. Choosing $f=1$ in (2), the condition $f \div f_{j_0}$ holds for any choice of $j_0$, so we obtain $E \subseteq H^{\OO_1}_{d\infty_1\infty_2} \subseteq H_{d_{j,m}\infty_1\infty_2}^{\OO_1}$ for some $j,m \in \N$.
\end{proof}

\begin{proof}[Proof of \Cref{thm:cofinal}]
By \Cref{thm:RayClassField2}, our conditional assumptions imply that, for any admissible tuple $t = (d,r,Q) \sim (K,j,m,Q)$, one has the field containments
\begin{equation}
    E_t \supseteq H_{\db_{j,m}\infty_1\infty_2}^{\OO_1} \supseteq H_{d_{j,m}\infty_1\infty_2}^{\OO_1}.
\end{equation}
\Cref{thm:cofinal} thus follows from \Cref{thm:cofinalplus}.
\end{proof}

For real quadratic fields $K$ satisfying the hypothesis of \Cref{thm:cofinalplus}(1), namely that $\Tr(\e)$ is odd, $r$-SICs have the potential to provide a full solution to Hilbert's twelfth problem. The fields $H^{\OO_1}_{d_{j,m}\infty_1\infty_2}$ contain every abelian extension of $K$. By \Cref{thm:RayClassField2plus}(4), these fields ``SIC-generated'' in the sense that they are contained in the SIC field $E_{\Pi_s}$ of a SIC fiducial $\Pi_s$, under the assumptions of \Cref{conj:stc} and \Cref{cnj:tci}. Proofs of \Cref{conj:stc} and \Cref{cnj:tci} could thus be considered a geometric, complex-analytic solution to Hilbert's twelfth problem for real quadratic fields with a unit of odd trace.

We provide a result of the natural density of such fields in the family of all real quadratic fields ordered by discriminant.

\begin{thm}\label{thm:oddtracecount}
When ordered by discriminant, at least 7.4\% and at most 33.4\% of real quadratic fields $K$ have a fundamental unit $\e_K$ of odd trace (or equivalently, have a totally positive fundamental unit of odd trace, or have any unit of odd trace). Specifically, as $X \to \infty$, there is an asymptotic inequality
\begin{equation}
    \frac{2}{27} + o(1) \leq \frac{\#\{K : [K:\Q]=2, 0<\Delta_K<X, 2\ndiv\Tr(\e_K)\}}{\#\{K : [K:\Q]=2, 0<\Delta_K<X\}} \leq \frac{1}{3} + O(X^{-1/2}),
\end{equation}
where $\Delta_K := \disc K$.
\end{thm}
\begin{proof}
    See \Cref{ap:oddtrace}.
\end{proof}

\section{SIC phenomenology}
\label{sec:sicphenexplain}
The purpose of this section is to show how the conjectures and results presented in previous sections explain many of the features of the calculated SIC fiducials that were described in \Cref{subsc:sicphenomenology}. 

We begin, in \Cref{ssc:TransformFormsFiducials}, by showing how  the action of $\GLtwo{\mbb{Z}}$ on quadratic forms translates into an action of $\ECS(d)$ on the corresponding fiducials.

In \Cref{ssc:Classification} we discuss the classification of $r$-SIC fiducials.  We restrict our attention to what we refer to as \emph{standard fiducials}, by which we mean fiducials corresponding to fiducial data sets $(d,r,Q,G,g)$ for which $\Det (G) = \pm 1$. We show that if Conjectures~\ref{cnj:tci},~\ref{conj:mrmvc}, and~\ref{conj:RayClassField3} are all true,  then one gets the complete set of standard fiducials if one restricts to data sets having some fixed, but arbitrary choice of $G$ and $g$.  Given two such data sets $(d,r,Q,G,g)$, $(d,r,Q',G,g)$ we show that the corresponding fiducials are
\begin{enumerate}
    \item in the same $\EC(d)$ orbit if $Q$ and $Q'$ are equivalent,
    \item in the same Galois multiplet if $Q$ and $Q'$ have the same discriminant.
\end{enumerate}

In \Cref{sc:ClassificationECdOrbits} we describe some illustrative examples.  We also describe how, on the assumption that Conjectures~\ref{cnj:tci},~\ref{conj:mrmvc}, and~\ref{conj:RayClassField3} are all true, the number of $\EC(d)$  orbits, and the number of Galois multiplets of standard $1$-SIC fiducials varies with dimension $d$.  Finally, we describe how the number of standard $r$-SICs with $r>1$ varies with dimension.

In \Cref{sc:SICSymmetries} we investigate the symmetry group for an $r$-SIC fiducial.  We show that if $\Pi_s$ is the fiducial corresponding to the datum $s=(d,r,Q,G,g)$,  then $\mcl{S}(Q)$, the stabilizer group for  $Q$ (see \Cref{dfn:stbgpqdf}), gives rise to a cyclic subgroup of $\StabPiESL{\Pi_s}$.  In every case where it has been checked, one finds in fact that the cyclic group corresponding to $\mcl{S}(Q)$ coincides with $\StabPiESL{\Pi_s}$.  If that is generally true then, given a fiducial data set $s=(d,r,Q,G,g)$,
\begin{enumerate}
    \item we have a criterion for when the symmetry group $\mcl{S}(\Pi_s)$ has an anti-unitary symmetry;
    \item we have an expression for the order of  $\mcl{S}(\Pi_s)$.
\end{enumerate}
In the rank-1 case we can also 
\begin{enumerate}
    \item explain why every fiducial has a canonical order 3 symmetry (see \Cref{df:canonicalOrder3}),
    \item explain why one only gets type $z$ $\EC(d)$ orbits when $d\not\equiv 3\Mod{9}$ (see \Cref{df:typeztypeatypeaprime}),
    \item explain why one gets both type $z$ and type $a$ orbits when $d\equiv 3\Mod{9}$,
    \item give a criterion for identifying the type $a$ orbits when $d\equiv 3\Mod{9}$.
\end{enumerate}

Finally, in \Cref{ssc:alignment}, we consider the phenomenon of SIC alignment.  As discussed in \Cref{ssec:dimtowsicalign}, it is observed in the empirically calculated solutions\cite{Appleby:2017,Andersson:2019} that, up to a sign, the squares of the \normalizedOverlapsText{} for a $1$-SIC at position $d_j$ in a dimension tower reappear among the \normalizedOverlapsText{} at position $d_{2j}$.  We show that this phenomenon is a consequence of our results.  We also show that the phenomenon generalizes to a relation between the  \normalizedOverlapsText{} for a $1$-SIC at positions $d_j$ and $d_{nj}$ in a tower, for any positive integer $n$ coprime to 3.

\subsection{Transformations of forms and fiducials}
\label{ssc:TransformFormsFiducials}
Consider the map $s \mapsto \Pi_s$, where $s$ is a fiducial datum and $\Pi_s$ is an $r$-SIC fiducial. 
We have
\begin{enumerate}
    \item a natural action of $\GLtwo{\mbb{Z}}$ on the domain of this map, in which $M\in \GLtwo{\mbb{Z}}$ takes $s=(d,r,Q,G,g)$ to $s_M=(d,r,Q_M,G,g)$,
    \item a natural action
    of $\ECS(d)$ on the image of the map, in which $U\in \ECS(d)$ takes $\Pi_s$ to $U\Pi_sU^{\dagger}$.
\end{enumerate}
The purpose of this subsection is to show how these two actions are related.  Its importance, among other things, is that it leads to the classification of $r$-SICs described in \Cref{sc:ClassificationECdOrbits}. 
In particular, it explains the numbers in \Cref{tab:numberOfECdSICOrbitsDims4To20} of \Cref{ssc:NumberOfOrbits}.

Although we do not prove it in this paper,
it appears that live fiducials of the form $\Pi_s$, with $s$ a fiducial datum, are always strongly centered (see \Cref{dfn:stronglycentered}). 
It also appears that every $r$-SIC contains at least one strongly centered fiducial. 
Conjugating with elements of $\ECS(d)$ takes strongly centered fiducials to strongly centered fiducials. 
To obtain the full set of $r$-SICs we then conjugate the strongly centered fiducials with elements of $\WH(d)$.  

We will now state the main results of this subsection. 

It is an immediate consequence of the definition that if $(d,r,Q)$ is an admissible tuple, then $(d,r,Q_M)$ is another  admissible tuple, for all $M\in \GLtwo{\mbb{Z}}$.  Our first result says that the corresponding fields are the same, and that a similar statement holds for a fiducial datum.
\begin{thm}  \label{thm:MtransformedTuples}
  Assume \Cref{cnj:tci}.  Let  $(d,r,Q,G,g)$ be a fiducial datum and  $M$  any element of $\GLtwo{\mbb{Z}}$. 
 Define $t=(d,r,Q)$, $t_M=(d,r,Q_M)$.  Then
 \begin{enumerate}
     \item $E_{t_M} = E_{t}$
     \item $(d,r,Q_M,G,g)$ is another fiducial datum.
 \end{enumerate}
\end{thm}
\Cref{thm:MtransformedTuples} motivates the following definition.
\begin{defn}[$M$-transformed tuple and fiducial datum, equivalent tuples and data sets]\label{dfn:MTransformedt}
    Let $t=(d,r,Q)$ be an admissible tuple, $s=(d,r,Q,G,g)$ a fiducial datum, and $M$ an element of $\GLtwo{\mbb{Z}}$.  We define $t_M$ to be the \textit{$M$-transformed admissible tuple} $(d,r,Q_M)$, and we define $s_M$ to be the \textit{$M$-transformed fiducial datum} $(d,r,Q_M,G,g)$.
    
    We say that two admissible tuples $t=(d,r,Q)$, $t'=(d,r,Q')$ are \emph{equivalent}, and write $t\sim t'$, if and only if the forms $Q$, $Q'$ are equivalent.

 Similarly, we say two fiducial datums $s=(d,r,Q,G,g)$, $s'=(d,r,Q',G,g)$ are equivalent, and write $s\sim s'$, if and only if the forms $Q$, $Q'$ are equivalent.
\end{defn}
The following homomorphism describes the relation between transformations of $s$, and transformations of $\Pi_s$ which is the focus of this subsection:
\begin{defn}\label{dfn:GLHomomorphism}
    Let  $s=(d,r,Q,G,g)$ a fiducial datum.  We define $\GLMorph_s$ to be the map of $\GLtwo{\mbb{Z}}$ to $\ESLtwo{\mbb{Z}/\db \mbb{Z}}$ defined by
    \eag{
     \GLMorph_s\colon M \mapsto \sgn\!\left(j_{M^{-1}}(\qrt_t)\right) H\vpu{-1}_gG^{-1}[M]_{\db} G H_g^{-1}
    }
    where $t$ is the admissible tuple $(d,r,Q)$ and $[M]_{\db}$ is the image of $M$ under the canonical homomorphism of $\GLtwo{\mbb{Z}}$ to $\ESLtwo{\mbb{Z}/\db\mbb{Z}}$.
\end{defn}
\begin{rmkb}
    The factor of $\sgn\!\left(j_{M^{-1}}(\qrt_t)\right)$ means $\GLMorph_s$ is not a homomorphism.  It is, however, ``almost'' a homomorphism, in the sense that $\GLMorph_s(M_1M_2) = \pm\GLMorph_s (M_1) \GLMorph_s (M_2) $ for all $M_1, M_2\in \GLtwo{\mbb{Z}}$.
\end{rmkb}
We need the following fact.
\begin{thm}\label{thm:GLHomomorphismSurjective}
    The map $\GLMorph_s$ is  surjective  for every fiducial datum $s$.
\end{thm}
\begin{proof}
    We defer the proof to page~\pageref{proof:GLHomomorphismSurjective}, at the end of this subsection.
\end{proof}
We are now able to state the central result of this subsection:
\begin{thm} \label{thm:MTransformedFiducials}
    Let $s=(d,r,Q,G,g)$ be a fiducial datum,  
and let  $M\in \GLtwo{\mbb{Z}}$ be arbitrary.  Then
        \eag{
         \Pi\vpu{\dagger}_{s_M} &= U^{\dagger}_F\Pi\vpu{\dagger}_s U\vpu{\dagger}_F,
         \label{eq:MtransformToFTransform}
        }
        where $F=\GLMorph_s(M)$.
\end{thm}
\begin{rmkb}
The fact that $\GLMorph_s$ is surjective means that the converse is also true: Given arbitrary $F\in \ESLtwo{\mbb{Z}/\db\mbb{Z}}$, there exists $M\in \GLtwo{\mbb{Z}}$ for which \eqref{eq:MtransformToFTransform} holds.

    Note, however, that this correspondence between $\GLtwo{\mbb{Z}}$-transformations of $s$ and (anti-)symplectic transformations of $\Pi_s$ cannot be a function in either direction. 
    Indeed, $\GLtwo{\mbb{Z}}$ is an infinite group, whereas $\ESLtwo{\mbb{Z}/\db\mbb{Z}}$ is finite. So there must be infinitely many matrices $M$ corresponding to a given matrix $F$.  Conversely, the existence of the symmetries described in \Cref{ssc:symgp} and \Cref{sc:SICSymmetries} below means that there is more than one matrix $F$ corresponding to a given matrix $M$.  This incidentally means that infinitely many different sets of fiducial data $s$ must give rise to the same fiducial $\Pi_s$.
\end{rmkb}
\begin{proof}
    We defer the proof to page~\pageref{proof:MTransformedFiducials}, at the end of this subsection.
\end{proof}
Before proving Theorems
~\ref{thm:GLHomomorphismSurjective}, and \ref{thm:MTransformedFiducials} we need to establish some preliminary results.
\begin{thm}\label{thm:AtmrhotmExpressions}
    Let $t$ be an admissible tuple and $M$ an element of $\GLtwo{\mbb{Z}}$.  Then
    \eag{
    \stabQGen{t_M} &= M^{-1} \stabQGen{t} M,
    \label{eq:LmatrixMTransform}
    \\
    \posGen{t_M} &= M^{-1} \posGen{t} M,
    \\
    \zaunerGen{t_M}&= M^{-1} \zaunerGen{t} M,
    \label{eq:zaunerGenMTransform}
    \\
    A_{t_M} &= M^{-1} A_t M,
    \label{eq:AmatrixMTransform}
    \\
    \intertext{and}
    \qrt_{t_M} &= M^{-1} \qrt_t
    }
    (see Definitions \ref{dfn:AssociatedStabilizers}, \ref{dfn:GhostOverlaps} for definitions of $\stabQGen{t}, \posGen{t},\zaunerGen{t},A_t, \qrt_t$).
\end{thm}
\begin{proof}
    Let $t=(d,r,Q) \sim (K,j,m,Q)$.  It follows from \Cref{tm:symgp} that
    \eag{
    \posGen{t}&= \canrep_Q(\vn^{j_{\rm{min}}(f)}) = \left( \frac{d_{j_{\rm{min}}(f)}-1}{2}\right)I+\frac{f_{j_{\rm{min}}}(f)}{f} SQ
    \\
    \intertext{and}
        \posGen{t_M}&= \canrep_Q(\vn^{j_{\rm{min}}(f)}) = \left( \frac{d_{j_{\rm{min}}(f)}-1}{2}\right)I+\frac{f_{j_{\rm{min}}}(f)}{f} SQ_M.
    }
    It follows from \eqref{eq:afrmtrans} and \Cref{lm:slm1sdllt} that
    \eag{\label{eq:SQMTermsQandM}
        SQ_M &= \Det(M)SM^{\rm{T}}QM = M^{-1} S Q M,
    }
    implying
    \eag{
    \posGen{t_M}&= M^{-1} \posGen{t} M
    }
    \Cref{eq:zaunerGenMTransform} and \eqref{eq:AmatrixMTransform} follow from this and the fact that 
    \eag{
    \zaunerGen{t}&=\posGen{t}^{n}, & \zaunerGen{t_M}&=\posGen{t_M}^{n},
    \\
    A_t&=\posGen{t}^{n(2m+1)} , & A_{t_M}&=\posGen{t_M}^{n(2m+1)},
    }
    where $n = j/j_{\rm{min}}(f)$. 
    If $\un_f = \vn_f$ then $\stabQGen{t} = \posGen{t}$, $L_{t_M} = \posGen{t_M}$, from which \eqref{eq:LmatrixMTransform} follows. 
    Otherwise it follows from Theorems~\ref{thm:negunitcondA} and~\ref{thm:negunitcondb} that $j_{\rm{min}}(f)$ is odd, $d_{j_{\rm{min}}(f)}-3$ is a perfect square, $f_{j_{\rm{min}}(f)}$ is divisible by $f\sqrt{d_{j_{\rm{min}}(f)}-3}$, and
    \eag{
    \stabQGen{t} &= \canrep_Q(\un^{j_{\rm{min}}(f)}) = 
    \frac{\sqrt{d_{j_{\rm{min}}(f)}-3}}{2} I+ \frac{f_{j_{\rm{min}}(f)}}{f\sqrt{d_{j_{\rm{min}}(f)}-3}}SQ,
    \\
    L_{t_M} &= \canrep_Q(\un^{j_{\rm{min}}(f)}) = 
    \frac{\sqrt{d_{j_{\rm{min}}(f)}-3}}{2} I+ \frac{f_{j_{\rm{min}}(f)}}{f\sqrt{d_{j_{\rm{min}}(f)}-3}}SQ_M.
    }
    \Cref{eq:LmatrixMTransform} is then a consequence of this together with \eqref{eq:SQMTermsQandM}.
    
    Finally, it follows from \Cref{dfn:GhostOverlaps} and \Cref{lm:lactqrt} that 
    \eag{
     \qrt_{t_M} &= \qrt_{Q_M,+} = M^{-1}\qrt_{Q,+} = M^{-1}\qrt_t.
    }
\end{proof}

\begin{thm}\label{lem:MTransformOfSFmodular}
    Let $t=(d,r,Q)$ be an admissible tuple,  and let $M\in \GLtwo{\mbb{Z}}$.  Then  
    \eag{
\sfc{d^{-1}\mbf{p}}{B_{t_M}}{\qrt_{t_M}} 
&=
\begin{cases}
    \sfc{d^{-1}l M\mbf{p}}{B_t}{\qrt_t} \qquad & \Det M=1,
    \\
    \left(\sfc{d^{-1}l M\mbf{p}}{B_t}{\qrt_t} \right)^{*} \qquad & \Det M=-1,
\end{cases}
}
   where either $\mbf{p}\in \mbb{Z}^2\setminus d\mbb{Z}^2$ or $\mbf{p} = \zero$, where either $B_t = A_t$, $B_{t_M} = A_{t_M}$ or $B_t=A^{-1}_t$, $B_{t_M} = A^{-1}_{t_M}$, and where $l = \sgn\left(j_{M^{-1}}(\qrt_t)\right)$.
\end{thm}
\begin{proof}
Proved in \cite[Thm.~4.37]{Kopp2020d}.

\end{proof}

\Cref{thm:MTransformNormalizedGhostOverlap} tells us how the \candidateNormalizedGhostOverlapsText{} transform under the action of an element of $\GLtwo{\mbb{Z}}$:

\begin{thm}\label{thm:MTransformNormalizedGhostOverlap}
     Let $t=(d,r,Q)\sim(K,j,m,Q)$ be an admissible tuple,  and let $M\in \GLtwo{\mbb{Z}}$. Then
     \eag{
      \normalizedGhostOverlapC{t_M}{\mbf{p}} &= \normalizedGhostOverlapC{t}{l M\mbf{p}}
     }
     for all $\mbf{p}\notin d\mbb{Z}^2$, where  $l = \sgn\left(j_{M^{-1}}(\qrt_t)\right)$
\end{thm}
\begin{proof}
   Let $f$ be the  conductor of  $Q$.  Then $f$ is also the conductor of $Q_M$.  It follows from \Cref{dfn:SFKPhase} and \Cref{lem:RademacherProperties} that
    \eag{
     \SFPhase{t_M}{\mbf{p}}&=(-1)^{s_d(\mbf{p})}e^{-\frac{\pi i}{12}\rade(A_{t_M})}\rtu_d^{-\frac{f_{jm}}{f}Q_M(\mbf{p})}
     \nn
     &= (-1)^{s_d(\mbf{p})}e^{-\frac{\pi i}{12}(\Det M)\rade(A_{t})}\rtu_d^{-\frac{f_{jm}}{f}(\Det M)Q(M\mbf{p})}
    }
    Let $M=\smt{\ma & \mb \\ \mc & \md}$.  Then 
    \eag{
    &s_d(l M\mbf{p})
    \nn
    &\hspace{0.2 cm} = d + (1+d) (1+l(\ma p_1 + \mb p_2))(1+l(\mc p_1 + \md p_2)) &&
    \nn
    & \hspace{0.2 cm}  \equiv d + (1+d)(1+(\ma + \mc + \ma \mc) p_1 + p_1p_2 +(\mb + \md +  \mb\md) p_2) & &\Mod 2
    \nn
    & \hspace{0.2 cm}  \equiv d+(1+d)\left(((1+p_1)(1+p_2) + (1+\ma)(1+\mc)p_1 +(1+\mb)(1+\md)p_2\right) &&  \Mod 2
    }
    The fact that $\ma$ is coprime to $\mc$ and $\mb$ is coprime to $\md$ means $(1+\ma)(1+\mc)$ and $(1+\mb)(1+\md)$ are both even. So
    \eag{
    s_d(l M\mbf{p}) &\equiv s_d(\mbf{p}) \,\Mod{2}.
    }
    Also
    \eag{
    Q(l M\mbf{p}) &= Q(M\mbf{p}).
    }
    Hence
    \eag{
    \SFPhase{t_M}{\mbf{p}}
    &=
    \begin{cases}
        \SFPhase{t}{l M\mbf{p}} \qquad & \Det M = +1,
        \\
        \left( \SFPhase{t}{l M\mbf{p}} \right)^{*} & \Det M = -1.
    \end{cases}
    }
It follows from \Cref{lem:MTransformOfSFmodular} that 
\eag{
\sfc{d^{-1}\mbf{p}}{A_{t_M}}{\qrt_{t_M}}&=
\begin{cases}
    \sfc{d^{-1}l M\mbf{p}}{A_t}{\qrt_t} \qquad &\Det M = +1,
    \\
    \left(\sfc{d^{-1}l M\mbf{p}}{A_t}{\qrt_t}\right)^{*} \qquad &\Det M = -1.
\end{cases}
}
Combining these results gives
\eag{
\normalizedGhostOverlapC{t_M}{\mbf{p}}
&= \SFPhase{t_M}{\mbf{p}}\sfc{d^{-1}\mbf{p}}{A_{t_M}}{\qrt_{t_M}}
= \begin{cases}
    \normalizedGhostOverlapC{t}{l M \mbf{p}} \qquad & \Det M = +1,
    \\
    \left(\normalizedGhostOverlapC{t}{l M \mbf{p}}\right)^{*}\qquad & \Det M = -1.
\end{cases}
}
Taking account of the fact that $\normalizedGhostOverlapC{t}{l M \mbf{p}}$ is real we conclude
\eag{
\normalizedGhostOverlapC{t_M}{\mbf{p}} &= \normalizedGhostOverlapC{t}{l M \mbf{p}}
}
irrespective of the sign of $\Det M$.
\end{proof}
We are now able to prove the first of our main results:
\begin{proof}[Proof of \Cref{thm:MtransformedTuples}]
It follows from \Cref{thm:MTransformNormalizedGhostOverlap} that the elements of the sets $\{\ghostOverlapC{t}{\mbf{p}}\colon 0\le p_1, p_2 < d, \, \mbf{p}\neq \zero\}$ and $\{\ghostOverlapC{t_M}{\mbf{p}}\colon 0\le p_1, p_2 < d, \, \mbf{p}\neq \zero\}$ are equal up to a sign, implying $E_{t} = E_{t_M}$.

Turning to the second statement, by assumption
\eag{
    \Det(G) r(2\shift+d_j-1+d) &\equiv 1 \Mod{\db}
}
for some for some $\shift\in \mcl{Z}_t$.  To show $\shift\in \mcl{Z}_{t_M}$, let $\mbf{p}\in \mbb{Z}^2$ be arbitrary, and let $\mcl{I}_{\mbf{p}}$ be any complete set of coset representatives for $\mbb{Z}^2/d\mbb{Z}^2$ containing  $\zero$ and $\mbf{p}$.   Suppose, first of all, that $\det M = +1$.  Then it follows from \Cref{thm:AtmrhotmExpressions} and \Cref{lem:MTransformOfSFmodular} that
\eag{
 &       \sum_{\mbf{q} \in \mcl{I}_{\mbf{p}}} \rtus_d^{ r\la \mbf{p},(\shift I + \zaunerGen{t_M})\mbf{q}\ra} \sfc{d^{-1}\q}{A\vpu{-1}_{t_M}}{\qrt_{t_M} }\sfc{d^{-1}(\mbf{q}-\mbf{p})}{A^{-1}_{t_M}}{\qrt_{t_M}} 
\nn
&\hspace{2 cm} = 
                \sum_{\mbf{q} \in \mcl{I}_{\mbf{p}}} \rtus_d^{ r\la \mbf{p},M^{-1}(\shift I + \zaunerGen{t})M\mbf{q}\ra} \sfc{d^{-1}lM\q}{A\vpu{-1}_{t}}{\qrt_{t} }\sfc{d^{-1}lM(\mbf{q}-\mbf{p})}{A^{-1}_{t}}{\qrt_{t}}
\nn
&\hspace{2 cm} = 
                \sum_{\mbf{q} \in \mcl{I}'_{lM\mbf{p}}} \rtus_d^{ r\la M\mbf{p},(\shift I + \zaunerGen{t})\mbf{q}\ra} \sfc{d^{-1}\q}{A\vpu{-1}_{t}}{\qrt_{t} }\sfc{d^{-1}(\mbf{q}-lM\mbf{p})}{A^{-1}_{t}}{\qrt_{t}}
\nn
&\hspace{2 cm} = d^2\delta^{(d)}_{\mbf{p},\zero}
}
where $l = \sgn\left(j_{M^{-1}}(\qrt_t)\right)$ and $\mcl{I}'_{lM\mbf{p}} = lM\mcl{I}_{\mbf{p}}$.  It follows that $\shift\in \mcl{Z}_{t_M}$ if $\Det M = +1$. The fact that this is also true of $\Det M = -1$ is proved similarly.  It follows from the first statement that $\hat\sicField_{t_M} = \hat\sicField_{t}$.  So $(d,r,Q_M,G,g)$ is a fiducial datum.
\end{proof}
We next  prove the following analogue of \Cref{thm:MTransformedFiducials}, applying to ghost fiducials.
\begin{lem} \label{thm:MTransformedGhostFiducials}
    Let $s=(t,G,g)$ be a fiducial datum containing the admissible tuple $t=(d,r,Q)$, and let $M$ be any element of $\GLtwo{\mbb{Z}}$.  Let
    \eag{
     F &= \sgn(j_{M^{-1}}(\qrt_t) G^{-1}[M]_{\db} G
    }
    where $[M]_{\db}$ is the image of  of $M$  under the canonical homomorphism of $\GLtwo{\mbb{Z}}$ into $\ESLtwo{\mbb{Z}/\db\mbb{Z}}$.  Then $F\in \ESLtwo{\mbb{Z}/\db\mbb{Z}}$ and 
        \eag{
         \tilde{\Pi}\vpu{\dagger}_{s_M} &= U^{\dagger}_F\tilde{\Pi}\vpu{\dagger}_s U\vpu{\dagger}_F.
        }
\end{lem}
\begin{proof}
  The fact that $F\in \ESLtwo{\mbb{Z}/\bar{d}\mbb{Z}}$ is immediate. It follows from \Cref{dfn:CandidateGhostAndSICFiducials}, \Cref{cor:AlternativeDefinitionOfCandidateGhostOverlaps} and \Cref{thm:MTransformNormalizedGhostOverlap} that
  \eag{
  \tilde{\Pi}_{s_M} &= \frac{r}{d}I +\frac{1}{d\sqrt{d_j+1}}\sum_{\mbf{p}\notin d\mbb{Z}^2} \normalizedGhostOverlapC{t_M}{G\mbf{p}}D_{\mbf{p}}
\nn
&= \frac{r}{d}I +\frac{1}{d\sqrt{d_j+1}}\sum_{\mbf{p}\notin d\mbb{Z}^2} \normalizedGhostOverlapC{t}{\ell M G\mbf{p}}D_{\mbf{p}}
\nn
&= \frac{r}{d}I +\frac{1}{d\sqrt{d_j+1}}\sum_{\mbf{p}\notin d\mbb{Z}^2} \normalizedGhostOverlapC{t}{\mbf{p}}D_{\ell G^{-1} [M]_{\db}^{-1}\mbf{p}}
\nn
&= \frac{r}{d}I +\frac{1}{d\sqrt{d_j+1}}\sum_{\mbf{p}\notin d\mbb{Z}^2} \normalizedGhostOverlapC{t}{\mbf{p}}U^{\dagger}_FD\vpu{\dagger}_{ G^{-1}\mbf{p}}U\vpu{\dagger}_F
\nn
&= U^{\dagger}_F\tilde{\Pi}_sU\vpu{\dagger}_F,
  }
where $\ell=\sgn(j_{M^{-1}}(\qrt_t)$.
\end{proof}
Before applying this result to the proof of the main theorem of this subsection, we need to establish two technical results.
\begin{lem}\label{lm:MReducedModdbarsuchthatjMxpositive}
    Let $M$ be any element of $\GLtwo{\mbb{Z}}$, $x$ any irrational element of $\mbb{R}$ and $d$ any dimension greater than 1.  Then there exists a matrix $M'\in \GLtwo{\mbb{Z}}$ such that
    \begin{enumerate}
        \item $M' \equiv M \Mod{\db}$,
        \item $j_M(x)j_{M'}(x) < 0$.
    \end{enumerate}
\end{lem}
\begin{proof}
    The fact that $x\notin \mbb{Q}$ means $j_M(x) \neq 0$.  Define
    \eag{
      F_{\pm 1} &= 
      \bmt 
         -1 - \bar{d} & \pm \bar{d}
         \\
         \mp \bar{d} & -1 +\bar{d}
    \emt  \in \SLtwo{\mbb{Z}}.
    }
    We have 
    \eag{
      j_{F_{+}}(-Lx) + j_{F_{-}}(-Lx) &=2(\bar{d}-1) >0.
    }
    Consequently we can choose $\theta=\pm 1$ such that $j_{F_{\theta}}(-Lx)$ is positive.  Define $M' = -F_{\theta}L$.  Then $M' \equiv  M \Mod{\db}$.   Moreover, it follows from \Cref{lm:jMproperties} that
    \eag{
    j_{M'}(x) &= j_{F_{\theta}}(-Lx)j_{-L}(x)=-j_{F_{\theta}}(-Lx) j_L(x),
    }
    implying that
    $
    \sgn(j_{M'}(x)) = -\sgn(j_{M}(x)).
    $ 
\end{proof}
We next use this result to show that the map $\GLMorph_s$ is surjective: 
\begin{proof}[Proof of \Cref{thm:GLHomomorphismSurjective}]
\label{proof:GLHomomorphismSurjective}
Let $F \in \ESLtwo{\mbb{Z}/\db/\mbb{Z}}$ be arbitrary. The fact that the canonical homomorphism $\GLtwo{\mbb{Z}}\to \ESLtwo{\mbb{Z}/\db\mbb{Z}}$ is surjective (see, e.g., \cite[Exer.~1.2.2]{Diamond:2005} for the surjectivity of $\SLtwo{\mbb{Z}} \to \SLtwo{\Z/\db\Z}$, from which this follows easily by consideration of the image of $\smmattwo{-1}{0}{0}{1}$) means that there exists $M\in \GLtwo{\mbb{Z}}$ such that
\eag{
[M]_{\db} &= GH^{-1}_g F H\vpu{-1}_g G^{-1}
}
In view of \Cref{lm:MReducedModdbarsuchthatjMxpositive} we may assume, without loss of generality, that $j_{M^{-1}}(\qrt_t)$ is positive.  We then have $\GLMorph_s(M) = F$.
\end{proof}
We are now ready to prove the main result of this subsection.
\begin{proof}[Proof of \Cref{thm:MTransformedFiducials}.]
\label{proof:MTransformedFiducials}
Let $M\in \GLtwo{\mbb{Z}}$.  It follows from  \Cref{thm:MTransformedGhostFiducials} that 
\eag{
\tilde{\Pi}_{s_M} &= U^{\dagger}_{\tilde{F}}\tilde{\Pi}_sU\vpu{\dagger}_{\tilde{F}}
}
where
\eag{
\tilde{F} &= \sgn\left(j_{M^{-1}}(\qrt_t)\right) G^{-1}[M]_{\db}G.
}
Applying $g$ to both sides and using Definitions~\ref{dfn:CandidateGhostAndSICFiducials},~\ref{dfn:HgMatrixDefinition} and \Cref{thm:GalActOnClifford} we find
\eag{
\Pi_{s_M} &= U^{\dagger}_F \Pi\vpu{\dagger}_s U\vpu{\dagger}_F
\\
\intertext{where}
F &= H\vpu{-1}_g \tilde{F} H^{-1}_g =\GLMorph_s(M),
}
completing the proof.
\end{proof}

\subsection{Classification}
\label{ssc:Classification}
In this subsection we show that, subject to certain assumptions, $1$-SICs can be classified in terms of equivalence classes of quadratic forms. We will need the following definition.
\begin{defn}[equivalence classes of admissible tuples]\label{df:EquivalenceClassesTuples} For each admissible tuple $t=(d,r,Q)$ define
    \begin{enumerate}
        \item $[t]=[d,r,Q]$ to be set of all admissible tuples $(d,r,Q')$, where $Q'$ is equivalent to $Q$,
        \item $\dBr{t}=\dBr{d,r,Q}$ to be the set of all tuples $(d,r,Q')$, where $Q'$ has the same conductor as $Q$.
    \end{enumerate}
   We will sometimes write $\dBr{d,r,f}$ in place of $\dBr{d,r,Q}$, where $f$ is the conductor of $Q$.
\end{defn}
We will show that in the rank 1 case, subject to Conjectures~\ref{cnj:tci},~\ref{conj:RayClassField3} and~\ref{conj:msc}, and the three assumptions listed below, that the equivalence classes $[t]$ are in bijective correspondence with the set of all $\EC(d)$ orbits of $1$-SICs, and that the equivalence classes $\dBr{t}$ are in bijective correspondence with the set of all Galois multiplets of $1$-SICs.  In \Cref{ap:sicdata} we illustrate this statement by listing the classes $[t]$, $\dBr{t}$ along with salient details for the corresponding $1$-SICs for dimensions $4$--$100$.

\begin{assum}\label{assum:1}
    Let $\Pi$, $\Pi'$ be any pair of  $1$-SIC fiducials in dimension $d$.  Then 
    \eag{
    \Tr(\Pi D_{\mbf{p}}) &= \pm \Tr(\Pi' D_{\mbf{p}}) \qquad \forall \mbf{p}
    }
     if and only if $\Pi=\Pi'$.
\end{assum}   

\begin{assum}\label{assum:2}
    Let $s=(d,1,Q,G,g)$, $s'=(d,1,Q',G',g')$ be admissible data such that $Q$ and $Q'$ have different discriminants. 
    Then $\Pi_s$ and $\Pi_{s'}$ are $\EC(d)$-inequivalent.
\end{assum}

\begin{assum}\label{assum:3}
    In the rank 1 case, the only shifts are $0$ and $1$.
\end{assum}

These statements have a different status from the ones we label  ``conjectures'' in that they are only needed for the classification problem.  Moreover, even if one or more of them were to fail, it would not necessarily mean that $1$-SICs could not be classified using quadratic forms, only that the classification would be more complicated.

Assumption~\ref{assum:2} is worth singling out for special mention, in that it is shown in \cite[Thm.~8.2]{Kopp2020c} that equivalence classes $\dBr{d,1,f}$, $\dBr{d,1,f'}$ can give rise to the same field, even though $f\neq f'$. This happens, for instance, with the pairs $\dBr{47,1,1}$, $\dBr{47,1,2}$; $\dBr{67,1,1}$, $\dBr{67,1,2}$; and $\dBr{83,1,1}$, $\dBr{83,1,2}$.  Although there is no known instance, it is natural to wonder if there are cases where the corresponding SICs are $\EC(d)$-equivalent. 

We confine our analysis to the rank 1 case due to  uncertainties concerning the set of shifts when $r>1$, as discussed in \Cref{sbsc:ProofMainTheoremsFurtherRemarks}.  In the following it will accordingly be assumed, without comment, that $r=1$.  In accordance with assumption~3 it will also be assumed without comment that $\det G = \pm 1$ for every twist $G$.

It is immediate that $\Pi_s$, $\Pi_{s'}$ cannot be be equal unless the dimensions   are the same.  So the classification problem reduces to the question of what can be said of the $1$-SICs corresponding to two sets of fiducial data $s=(d,1,Q,G,g)$, $s'=(d,1,Q',G',g')$ when $(Q,G,g)$ and $(Q',G',g')$ are not assumed to be the same.  We begin with the following result.
\begin{thm}\label{thm:QGChangeInFiducialDatum}
    Let $s=(d,1,Q,G,g)$, $s'=(d,1,Q',G',g)$ be a pair of fiducial datums for which  $Q'\sim Q$.  If Assumption~3 is true there exists $F\in \ESLtwo{\mbb{Z}/\bar{d}\mbb{Z}}$ such that 
    \eag{
      \Pi_{s'} &= U\vpu{\dagger}_F \Pi_s U^{\dagger}_F
    }
\end{thm}
\begin{proof}
    It follows from \Cref{thm:MTransformedGhostFiducials} that there exists $F'\in \ESLtwo{\mbb{Z}/\bar{d}\mbb{Z}}$ such that 
    \eag{
      \tilde{\Pi}_{s''} &= U\vpu{\dagger}_{F'}\tilde{\Pi}_sU^{\dagger}_{F'}
    }
    where $s'' = (d,1,Q',G,g)$.  Let  $F''$ be the image of $G^{-1} G'$ under the canonical homomorphism $\GLtwo{\mbb{Z}} \to \ESLtwo{\mbb{Z}/\bar{d}\mbb{Z}}$.  Writing $t'=(d,1,Q')$ and referring to \Cref{dfn:CandidateGhostAndSICFiducials} we see that 
    \eag{
    \tilde{\Pi}_{s''}&= \frac{1}{d}I + \frac{1}{d\sqrt{d+1}}\sum_{\mbf{p}\notin d\mbb{Z}^2}\normalizedGhostOverlapC{t'}{\mbf{p}}D_{G^{-1}\mbf{p}}
    \nn
    &= \frac{1}{d}I + \frac{1}{d\sqrt{d+1}}\sum_{\mbf{p}\notin d\mbb{Z}^2}\normalizedGhostOverlapC{t'}{\mbf{p}}U\vpu{\dagger}_{F''}D\vpu{\dagger}_{G^{\prime -1}\mbf{p}}U^{\dagger}_{F''}
    \nn
    &= U\vpu{\dagger}_{F''}\tilde{\Pi}_{s'}U^{\dagger}_{F''}.
    }
    Hence
    \eag{
    \tilde{\Pi}_{s'} &= U^{\dagger}_{F''}U\vpu{\dagger}_{F'}\tilde{\Pi}_s U^{\dagger}_{F'}U\vpu{\dagger}_{F''}.
    }
    Applying $g$ to both sides we find, in view of Definitions~\ref{dfn:CandidateGhostAndSICFiducials},~\ref{dfn:HgMatrixDefinition} and Theorem~\ref{thm:GalActOnClifford}, 
    \eag{
     \Pi_{s'} &= U\vpu{\dagger}_F \Pi_sU^{\dagger}_F
    }
    with $F = H_g F^{''-1} F' H^{-1}_g$.
\end{proof}
We also need to consider the effect of varying the Galois conjugation $g$. Let $t=(d,r,Q)$ be an admissible tuple,  let $f$ be the conductor of $Q$, and let $E_t$ be the field associated to $t$ (as specified in \Cref{dfn:SICfield}). Then it is easily seen that for every fiducial datum  $s=(t,G,g)$ extending $t$ the matrix elements of $\tilde{\Pi}_s$, $\Pi_s$ are in $E_t$.
Also, if \Cref{conj:RayClassField3} is true, then $E_t$  is the ray class field (in the generalized sense of Kopp and Lagarias~\cite{Kopp2020b}) with datum $(\mcl{O}_f, \bar{d}\mcl{O}_f, (\infty_1,\infty_2))$.  In particular, if \Cref{conj:RayClassField3} is true, and  if  $t'=(d,r,Q')$ is any other element of $\dBr{t}$, then $E_{t'} = E_{t}$.  We accordingly define $E_{\dBr{t}} = E_t$.  

We also need the following subfield of $E_t$:
\begin{defn}[ring class field, class number]\label{df:RingClassField}
    Let $t=(d,r,Q)\sim(K,j,m,Q)$ be an admissible tuple, and let $f$ be the conductor of $Q$.  We define the \emph{ring class field} for $t$, denoted $H_t$, to be the ray class field (in the generalized sense of Kopp and Lagarias~\cite{Kopp2020b}) with datum $(\mcl{O}_f,\mcl{O}_f,\emptyset)$.  We define the \emph{class number} for $t$, denoted $h_t$, to be the class number of $\mcl{O}_f$ (or, equivalently, the degree of the extension $H_t/K$).

    Since $H_t$ and $h_t$ only depend on the equivalence class $\dBr{t}$, we may  define $H_{\dBr{t}} = H_t$, $h_{\dBr{t}} = h_t$.
\end{defn}
\begin{rmkb}
    Note that if $f=1$, so that $\mcl{O}_f$ is the maximal order, then $H_t$ is the Hilbert class field.
\end{rmkb}
Also define a fixed, non-canonical choice of a sign-switching automorphism.
\begin{defn}
    On the assumption that \Cref{conj:RayClassField3} is true, for each equivalence class $\dBr{d,1,Q}$ make a once-and-for-all  choice of automorphism $g_{\dBr{d,1,Q}}\in \Gal(E_{\dBr{d,1,Q}}/\mbb{Q})$ which does not fix $K$.
\end{defn}
We will also need the following result.
\begin{thm}\label{thm:GaloisActionOnFiducial}
    Assume \Cref{conj:RayClassField3}.
    Let  $s=(t,G,g)$ be a fiducial datum containing the admissible tuple $t=(d,1,Q)$, and let $f$ be the conductor of $Q$. Let $Q_1, \dots, Q_{h_t}$ be a set of representatives of the $h_t$ distinct equivalence classes of forms having the same discriminant as $Q$.  Then for each $g'\in \Gal(E_t/K)$ there exists a form $Q(g')$ having the same discriminant as $Q$ and such that
    \eag{
      g'(\Pi_s) &= \Pi_{s(g')}, & s(g')& = (d,1,Q(g'),G,g).
    }
    Moreover
    \begin{enumerate}
        \item For each integer $n=1,\dots, h_t$, there exists  $g'\in \Gal(E_t/K)$ such that $Q(g')\sim Q_n$.  
        \item $Q(g')\sim Q$ if and only if $g'\in \Gal(E_t/H_t)$.
    \end{enumerate}
\end{thm}
\begin{proof}
    To appear in a subsequent publication~\cite{Appleby2022c}.
\end{proof}

We can then prove the following theorem.
\begin{thm}\label{thm:ReducedNotation}
Assume Conjectures~\ref{cnj:tci},~\ref{conj:RayClassField3} and Assumptions~1, 2, 3 are true.
    Let $s=(d,1,Q,G,g)$ be a fiducial datum.  There exists a form $Q'$, having the same discriminant as $Q$,  such that 
    \eag{
    \Pi_s& = \Pi_{s'}, & s' =(d,1,Q',I,g_{\dBr{d,1,Q}}).
    }
\end{thm}
\begin{proof}
Let $s''=(d,1,Q,G,g_{\dBr{d,1,Q}})$.  
  It follows from \Cref{thm:GaloisActionOnFiducial}  that
\eag{
  \Pi_s &= gg^{-1}_{\dBr{d,1,Q}}\left(\Pi_{s''}\right) = \Pi_{s'''}, & s'''&=(d,1,Q'',G,g_{\dBr{d,1,Q}})
\\
\intertext{for some $Q''$ having the same discriminant as $Q$.  It then follows from \Cref{thm:QGChangeInFiducialDatum} that}
\Pi_{s'''}&=U\vpu{\dagger}_F \Pi_{s''''}U^{\dagger}_F, & s''''&=(d,1,Q'',I,g_{\dBr{d,1,Q}})
\\
\intertext{for some $F\in \ESLtwo{\mbb{Z}/\bar{d}\mbb{Z}}$.  In view of Theorems~\ref{thm:GLHomomorphismSurjective} and~\ref{thm:MTransformedFiducials} this means}
\Pi_{s'''} &= \Pi_{s''''_M}, & s''''_M &= (d,1,Q''_M,I,g_{\dBr{d,1,Q}})
}
for some $M\in \GLtwo{\mbb{Z}}$.  Setting $Q''_M=Q'$, $s''''_M=s'$ the claim now follows.
\end{proof}
It follows from this result together with  \Cref{dfn:GaloisMultiplet} and Theorems~\ref{thm:MTransformedFiducials},~\ref{thm:GaloisActionOnFiducial}  that,  if we assume Conjectures~\ref{cnj:tci},~\ref{conj:RayClassField3} and~\ref{conj:msc} and Assumptions~1, 2, 3, , and if we confine ourselves to $1$-SICs of the kind specified by \Cref{dfn:CandidateGhostAndSICFiducials}, then there are bijective maps associating 
\begin{enumerate}
    \item to each equivalence class $[t]$, a corresponding $\EC(d)$ orbit of $1$-SICs,
    \item to each equivalence class $\dBr{t}$, a corresponding Galois multiplet of $1$-SICs.
\end{enumerate}
The procedure for finding all the $1$-SICs corresponding to a given admissible $(d,1)\sim(K,j,m)$ is then as follows:
\begin{enumerate}
    \item[Step 1] Find the set of  divisors of $f_j$.  Each such divisor corresponds to a distinct Galois multiplet of the specified rank and dimension.
    \item[Step 2] For each $f\div f_j$ calculate the corresponding class number $h_{\dBr{d,1,f}}$ (see Definitions~\ref{df:EquivalenceClassesTuples},~\ref{df:RingClassField}).
    This is the number of distinct $\EC(d)$ orbits in the multiplet $\dBr{d,1,f}$.
    \item[Step 3] Find $h_{\dBr{d,1,f}}$ inequivalent quadratic forms $Q_j$ having discriminant $f^2 \Delta_0$, where $\Delta_0$ is the discriminant of $K$.  The $1$-SICs corresponding to the admissible tuples $t_j = (d,1,Q_j)$ give us a full set of representatives for the  distinct $\EC(d)$ orbits in $\dBr{d,1,f}$.
    \item[Step 4] Conjugate the  fiducials corresponding to the  tuples $t_1, \dots, t_{h_{\dBr{d,1,f}}}$ with the elements of $\EC(d)$ to obtain the full set of $1$-SICs in the Galois multiplet $\dBr{d,1,f}$.
\end{enumerate}
Notice that equivalence of forms is defined relative to the group $\GLtwo{\mbb{Z}}$ rather than  $\SLtwo{\mbb{Z}}$.  So it is the wide class number that is relevant here.

This construction is illustrated in the  data tables in \Cref{ap:sicdata}, which gives a complete listing of $1$-SICs and corresponding equivalence classes $[t]$, $\dBr{t}$ for $d\le 100$.

\subsection{Illustrative examples}
\label{sc:ClassificationECdOrbits}
In this subsection we illustrate the discussion in \Cref{ssc:Classification} with some examples, on the assumption that Conjectures~\ref{cnj:tci},~\ref{conj:mrmvc}, and~\ref{conj:RayClassField3} and Assumptions~\ref{assum:1}, \ref{assum:2}, and \ref{assum:3} are all true.

Consider the dimension grid corresponding to the field $K=\mbb{Q}(\sqrt{5})$, illustrated in \eqref{eq:exampledimgrid}.  Consider first the admissible tuple $(4,1)\sim(K,1,1)$.  We have $f_1 =1$, so there is only one Galois multiplet.  The class number is 1, so the multiplet consists of a single $\EC(4)$ orbit.  A choice of form which is calculationally optimal is $Q=\la 1,-3,1\ra$. Comparing with the tables in \cite{Scott2010,Scott:2017}, we see that this agrees with what was previously found by a brute force computational approach, and that Scott--Grassl orbit $4a$ is, in our notation, the orbit $[4,1,\la 1, -3, 1\ra]$.

Moving up the left-hand column of the grid, we come to the admissible tuple $(8,1)\sim(K,2,1)$.  We have $f_2 = 3$, so there are two Galois multiplets corresponding to the choices $f=1,3$.  The class numbers are both 1, so each multiplet consists of a single $\EC(8)$ orbit.  Calculationally optimal choices of $Q$ are $\la 1,-3,1\ra$, $\la 1, -7 , 1\ra $ respectively.  Again, this is consistent with what was previously found.  Comparing with the tables in \cite{Scott2010,Scott:2017}, one finds that the Scott--Grassl orbit $8a$ is $[8,3,\la 1, -7, 1\ra]$ in our notation, and orbit $8b$ is $[8,1,\la 1, -3, 1\ra]$.  Finally, the fact that $1 \div 3$ implies $E_{\dBr{8,1,1}} \subseteq E_{\dBr{8,1,3}}$, while  the fact that $3\ndiv 2 \times 1$ and $\Delta_0 \not\equiv 1 \Mod{8}$ means that $E_{\dBr{8,1,3}} \nsubseteq E_{\dBr{8,1,1}}$.  So $E_{\dBr{8,1,1}}$ is a proper subfield of $E_{\dBr{8,1,3}}$

Since they only depend on the divisors of $f_j$ and the class numbers $h_{K,f}$, and since these quantities are constant along the rows of the dimension grid, it follows that the number of Galois multiplets, and the number of $\EC(d)$ orbits within each multiplet, are constant along each row.  Thus, moving along the bottom row  one finds that each of the  sequence of admissible pairs $(4,1), (11,3), (29,8), \dots $ gives rise to one Galois multiplet, comprising one $\EC(d)$ orbit.  Similarly, moving along the next-to-bottom row one finds that each of the sequence of admissble pairs $(8,1), (55,7), (377, 48), \dots$ gives rise to two Galois multiplets, each comprising one $\EC(d)$ orbit.

In \Cref{ap:sicdata} we give,
on the assumption that Conjectures~\ref{cnj:tci},~\ref{conj:mrmvc}, and~\ref{conj:RayClassField3} 
and Assumptions~\ref{assum:1}, \ref{assum:2}, and \ref{assum:3} are all true, 
a complete  listing of Galois multiplets and $\EC(d)$ orbits of $1$-SICs for $d\le 100$.  Using this table one can also find the number of Galois multiplets and $\EC(d)$ orbits  for any admissible tuple $(d,r)\sim(K,j,m)$ such that $d_j\le 100$.

In this way, under our conjectures and assumptions, one can quickly compute the number of multiplets and orbits for  much larger dimensions.
For instance when $d=10^6$ one finds  that there is a single Galois multiplet with class number $14\, 800$.  By contrast, when $d=10^6+3$ one finds that there are $40$ Galois multiplets  with conductors and class numbers as in \Cref{tab:ConductorsAndClassNumbers}.
\begin{table}[htb]
    \centering
    \begin{tabular}{ccc  ccc  ccc ccc cc}
        \toprule
        $f$ & $h_{K,f}$ & &  
        $f$ & $h_{K,f}$ & &
        $f$ & $h_{K,f}$ & &
        $f$ & $h_{K,f}$ & &
        $f$ & $h_{K,f}$ 
        \\
        \cmidrule(lr){1-2}
        \cmidrule(lr){4-5}
        \cmidrule(lr){7-8}
        \cmidrule(lr){10-11}
        \cmidrule(lr){13-14}
       1 & 1 &&  2 & 1 &&  4 & 1 &&  5 & 2 &&  8 & 2
        \\
        10 & 2 &&  16 & 4 && 20 & 4 &&  25 & 10 && 40 & 8
        \\
        50 & 10 &&  53 & 52 &&  80 & 16 &&  100 & 20 &&  106 &  52
        \\
        125 & 50 && 200 & 40 &&  212 & 52 &&  250 & 50 &&  265 & 104
        \\
        400 & 80 &&  424 & 104 &&  500 & 100 &&  530 & 104 &&  848 & 208
        \\
        1\,000 & 200 && 1\,060 & 208 &&  1\,325 & 520  &&  2\,000 & 400 &&  2\,120 & 416
        \\
        2\,650 & 520 &&  4\,240 & 832 &&  5\,300 & 1\,040 && 6\,625 & 2\,600 &&  10\,600 & 2\,080
        \\
        13\,250 & 2\,600 &&  21\,200 & 4\,160 &&  26\,500 & 5\,200 &&  53\,000 & 10\,400 &&  106\,000 & 20\,800
        \\
        \bottomrule
    \end{tabular}
    \caption{Conductors and class numbers for $1$-SICs in dimension $10^6+3$, assuming Conjectures~\ref{cnj:tci},~\ref{conj:mrmvc}, and~\ref{conj:RayClassField3}  and Assumptions~\ref{assum:1}, \ref{assum:2}, and \ref{assum:3} are true.  \label{tab:ConductorsAndClassNumbers}}
\end{table}
The  number of multiplets as a function of dimension is plotted in \Cref{fig:nummultiplets}, while the number of $\EC(d)$ orbits is plotted in \Cref{fig:numorbs}.
\begin{figure}[htb]
\begin{center}
\includegraphics[width=15 cm]{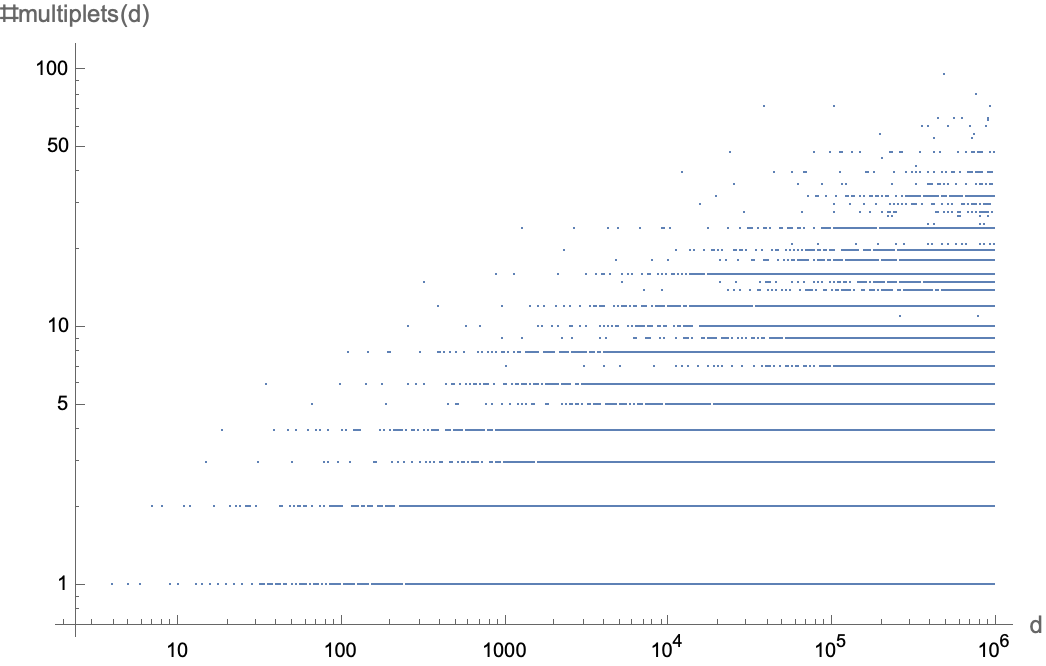}
\end{center}
\caption{Number of Galois multiplets of $1$-SICs as a function of dimension, assuming Conjectures~\ref{cnj:tci},~\ref{conj:mrmvc}, and~\ref{conj:RayClassField3} and Assumptions~\ref{assum:1}, \ref{assum:2}, and \ref{assum:3} are true.\label{fig:nummultiplets}}
\end{figure}
\begin{figure}[htb]
\begin{center}
\includegraphics[width=15 cm]{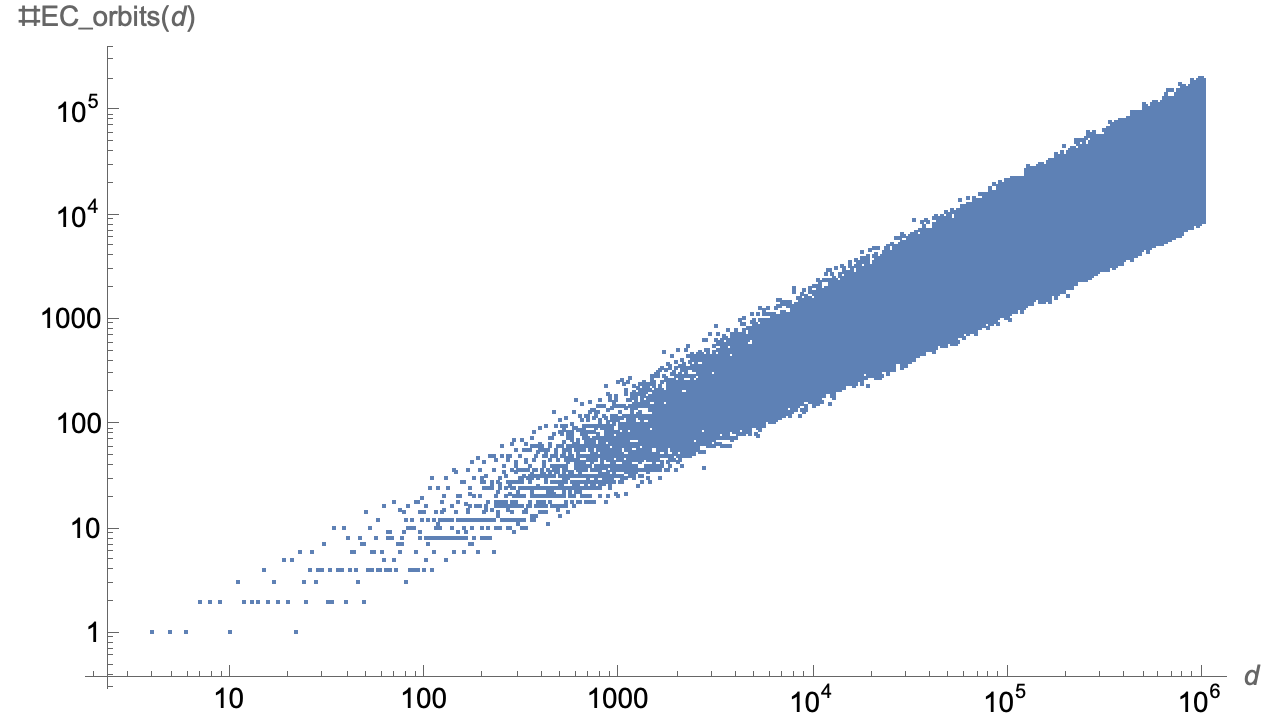}
\end{center}
\caption{Number of $\EC(d)$-orbits of $1$-SICs as a function of dimension, assuming Conjectures~\ref{cnj:tci},~\ref{conj:mrmvc}, and~\ref{conj:RayClassField3} and Assumptions~\ref{assum:1}, \ref{assum:2}, and \ref{assum:3} are true.\label{fig:numorbs}}
\end{figure}

We conclude with a few observations concerning the distribution of $r$-SICs with $r>1$.  A pair $(d,r)$ is admissible if and only if $0<r<(d-1)/2$ and
\eag{
nr(d-r) &= d^2-1
\label{eq:nrdrd2m1}
}
for some integer $n>4$ (see \Cref{dfn:admissiblePair} and discussion following).  If $r=1$, then solutions to \eqref{eq:nrdrd2m1} exist for every $d>3$.  This is far from being the case when $r>1$.
Thus, one finds that there are only $1\,153$ dimensions $d$ less than $10^6$ for which there exist admissible pairs $(d,r)$ with $r>1$.  Moreover, in almost of all of these cases there is only one pair with $r>1$, the only exceptions for $d\le 10^6$ being the five dimensions 29, 71, 239, 3\,191, 60\,761 where there are exactly two pairs.  See Figures~\ref{fig:meffdist1},~\ref{fig:meffdist2} for the distribution of dimensions up to $d=10^6$ and Table~\ref{tab:my_label} for the first 30 solutions to ~\eqref{eq:nrdrd2m1} with $r>1$.
\begin{figure}[htb]
\begin{center}
\includegraphics[width=15 cm]{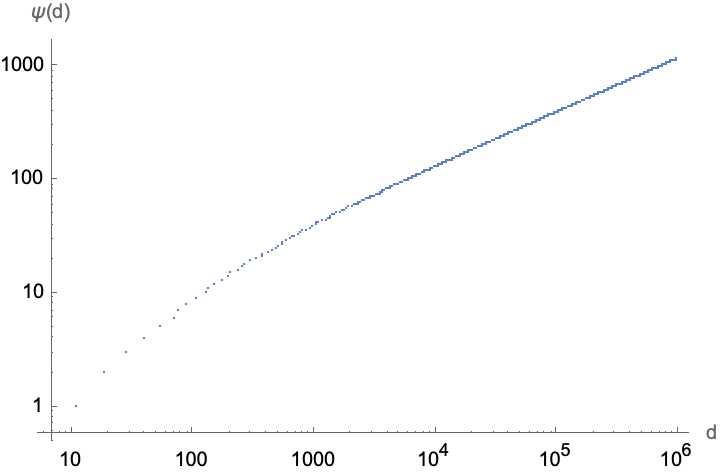}
\end{center}
\caption{Plot of $\psi(d)$ against $d$, where $\psi(d)$ is the number of dimensions less than or equal to $d$ in which there occur  $r$-SICs with $r>1$, assuming Conjectures~\ref{cnj:tci} and~\ref{conj:mrmvc} are true.\label{fig:meffdist1}}
\end{figure}
\begin{figure}[htb]
\begin{center}
\includegraphics[width=15 cm]{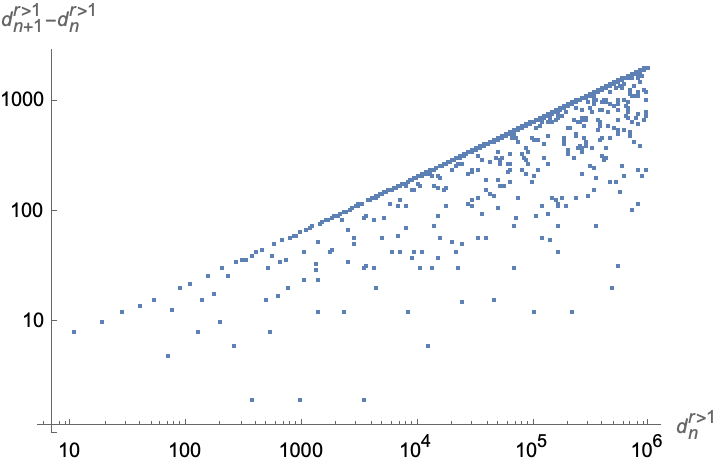}
\end{center}
\caption{Plot of $d^{r>1}_{n+1}-d^{r>1}_n$ against $d^{r>1}_n$, where $d^{r>1}_1, d^{r>1}_2, \dots$ is the increasing sequence of   dimensions in which there occur $r$-SICs with $r>1$, assuming Conjectures~\ref{cnj:tci}, and~\ref{conj:mrmvc} are true.\label{fig:meffdist2}}
\end{figure}
\begin{table}[htb]
    \centering
    \begin{tabular}{ccccc  c c c c c c
    c c c c c c}
        \toprule
        $d$ & $r$ & $K$ &  $j$ & $m$
        & &
                $d$ & $r$ & $K$ &  $j$ & $m$
        &&  $d$ & $r$ & $K$ & $j$ & $m$
        \\
        \cmidrule(lr){1-5} \cmidrule(lr){7-11} \cmidrule(lr){13-17}
        11 & 3 & $\mbb{Q}(\sqrt{5})$  & 1 & 2 
        &&  
        109 & 10 & $\mbb{Q}(\sqrt{6})$ & 1 & 2
        & & 
        271 & 16 &$\mbb{Q}(\sqrt{7})$ &  $1$ & $2$
        \\
        19 & 4 & $\mbb{Q}(\sqrt{3})$  & 1 & 2 
        &&  
        131 & 11 & $\mbb{Q}(\sqrt{13})$  & 1 & 2
        & &
        305 & 17 &$\mbb{Q}(\sqrt{285})$  & $1$ & $2$
        \\
        29 & 5 & $\mbb{Q}(\sqrt{21})$ & 1 & 2
        &&  
        139 & 24 & $\mbb{Q}(\sqrt{21})$  & 1 & 3
        & &
        341 & 18 &$\mbb{Q}(\sqrt{5})$  & $3$ & $2$
        \\
         & 8 & $\mbb{Q}(\sqrt{5})$  & 1 & 3
        &&  
        155 & 12 & $\mbb{Q}(\sqrt{35})$  & 1 & 2
        & &
        377 & 48 &$\mbb{Q}(\sqrt{5})$  & $2$ & $3$
        \\
        41 & 6 & $\mbb{Q}(\sqrt{2})$ & 1 & 2
        &&  
        181 & 13 & $\mbb{Q}(\sqrt{165})$  & 1 & 2
        & &
        379 & 19 &$\mbb{Q}(\sqrt{357})$  & $1$ & $2$
        \\
        55 & 7 & $\mbb{Q}(\sqrt{5})$  & 2 & 2
        & & 
        199 & 55 &$\mbb{Q}(\sqrt{5})$  & $1$ & $5$
        & &
        419 & 20 &$\mbb{Q}(\sqrt{11})$  & $1$ & $2$
        \\
        71 & 8 & $\mbb{Q}(\sqrt{15})$  & 1 & 2
        & &
        209 & 14 &$\mbb{Q}(\sqrt{3})$  & $2$ & $2$
        & &
        461 & 21 &$\mbb{Q}(\sqrt{437})$  & $1$ & $2$
        \\
         & 15 & $\mbb{Q}(\sqrt{3})$  & 1 & 3
        & &
        239 & 15 &$\mbb{Q}(\sqrt{221})$  & $1$ & $2$
        & &
        505 & 22 &$\mbb{Q}(\sqrt{30})$  & $1$ & $2$
        \\
        76 & 21 & $\mbb{Q}(\sqrt{5})$  & 1 & 4
        & &
         & 35 &$\mbb{Q}(\sqrt{2})$  & $1$ & $3$
        & &
        521 & 144 &$\mbb{Q}(\sqrt{5})$  & $1$ & $6$
        \\
        89 & 9 & $\mbb{Q}(\sqrt{77})$  & 1 & 2
        & &
        265 & 56 &$\mbb{Q}(\sqrt{3})$  & $1$ & $4$
        & &
        551 & 23 &$\mbb{Q}(\sqrt{21})$  & $2$ & $2$
        \\
        \bottomrule
    \end{tabular}
    \caption{The first 30 solutions to ~\eqref{eq:nrdrd2m1} with $r>1$.  \label{tab:my_label}}
\end{table}

\subsection{Symmetries}
\label{sc:SICSymmetries}
The purpose of this subsection is to prove some of the empirical observations in \Cref{ssc:symgp}, regarding the symmetry group of a $1$-SIC, and to generalize them to an arbitrary $r$-SIC.  We begin with a summary of our main results.
\begin{defn}[$\ESLsym{s}$, $\ESLposSym{s}$, $\ESLzaunSym{s}$]\label{df:ESLStabilizers}
    Let $s=(t,G,g)$ be a fiducial datum containing the admissible tuple $t=(d,r,Q) \sim  (K,j,m,Q)$.  We define
    \eag{
    \ESLsym{s}&= \GLMorph_s(\stabQGen{t}),
    \\
    \ESLposSym{s}&= \GLMorph_s(\posGen{t}),
    \\
    \ESLzaunSym{s}&= \GLMorph_s(\zaunerGen{t})
    }
(see Definitions~\ref{dfn:AssociatedStabilizers},~\ref{dfn:GLHomomorphism}
for definitions of $\stabQGen{t}$, $\posGen{t}$, $\zaunerGen{t}$, $\GLMorph_s$).
\end{defn}
\begin{defn}[unitary/anti-unitary type]\label{dfn:antiUnitaryType}
   Let $t=(d,r,Q)\sim (K,j,m,Q)$ be an admissible tuple. Let  $f$ the conductor of $Q$.  We say the tuple $t$ is of \emph{anti-unitary type} if  all three of the following conditions are satisfied:
   \begin{enumerate}
       \item $d_1-3$ is a perfect square,
       \item $j_{\rm{min}}(f)$ is odd,
       \item $f\sqrt{d_{j_{\rm{min}}}(f)-3}$ divides $f_{j_{\rm{min}}(f)}$.
   \end{enumerate}
   Otherwise, we say it is of \emph{unitary type}.
\end{defn}
\begin{rmkb}
    Note that it follows from \Cref{thm:negunitcondA} that if $d_1-3$ is a perfect square and $j_{\rm{min}}(f)$ is odd then $d_{j_{\rm{min}}}(f)-3$ is also a perfect square.
 \end{rmkb}
Our first result says that  $\GLMorph_s$ maps stabilizers of forms into stabilizers of fiducials:
\begin{thm}\label{thm:GLMorphHomOfQStabilizerGroup}
    Let $s=(t,G,g)$ be a fiducial datum containing the admissible tuple $t=(d,r,Q)\sim(K,j,m,Q)$.  The restriction of $\GLMorph_s$ to $\mcl{S}(Q)$ is a homomorphism of $\mcl{S}(Q)$ onto $\la \ESLsym{s} \ra$.  Moreover,  $\la \ESLsym{s} \ra \subseteq \StabPiESL{\Pi_s}$.
\end{thm}
\begin{rmkb}
 See  \Cref{dfn:sogppi} for definition of $\StabPiESL{\Pi_s}$.
    In every case where the  groups have been calculated one finds  in fact   $\la \ESLsym{s} \ra = \StabPiESL{\Pi_s}$.  
\end{rmkb}
\begin{proof}
    See below.
\end{proof}
Our second result establishes some basic properties of the matrices $\ESLsym{s}$, $\ESLzaunSym{s}$.
\begin{thm}\label{thm:RsRzsProperties}
    Let $s=(t,G,g)$ be a fiducial datum containing the admissible tuple $t=(d,r,Q)\sim(K,j,m,Q)$. Then  
    \begin{enumerate}
        \item \label{it:RsProperties}$\ESLsym{s}$ has the properties
        \begin{center}
        \begin{tabular}{
        >{\centering} p{2.1 cm}
        >{\centering} p{1 cm}
        >{\centering} p{2 cm} 
        >{\centering} p{2.4 cm} 
        >{\centering} p{2.6 cm}
        >{\centering\arraybackslash} p{3 cm}}
        \toprule
        Type of $t$ & $d$ & $\Det(\ESLsym{s})$ & $\Tr(\ESLsym{s})$ & order of $\ESLsym{s}$ & order of $U_{\ESLsym{s}}\!\la I \ra$
        \\
        \midrule
        unitary & odd & $1$&$d_{j_{\rm{min}}(f)}-1$&$n_t(2m+1)$ &  $n_t(2m+1)$
        \\
        unitary & even & $1$&$d_{j_{\rm{min}}(f)}-1$& $2n_t(2m+1)$ & $n_t(2m+1)$
        \\
        anti-unitary & odd &$-1$&$-\sqrt{d_{j_{\rm{min}}(f)}-3}$& $2n_t(2m+1)$ &  $2n_t(2m+1)$
        \\
        anti-unitary & even & $-1$&
        $-\sqrt{d_{j_{\rm{min}}(f)}-3}$
        &$4n_t(2m+1)$ & $2n_t(2m+1)$
        \\
        \bottomrule
        \end{tabular}
    \end{center}

    \vspace{0.2 cm}

    \item \label{it:RzsProperties} $\ESLzaunSym{s}$ has the properties
     \begin{center}
        \begin{tabular}{
        >{\centering} p{2.1 cm}
        >{\centering} p{0.8 cm}
        >{\centering} p{0.8 cm} 
        >{\centering} p{1 cm} 
        >{\centering} p{2.5 cm} 
        >{\centering} p{2.5 cm} 
        >{\centering\arraybackslash} p{3 cm}}
        \toprule
        Type of $t$ 
        & $d$ 
        & $\ESLzaunSym{s}$
        & $\Det(\ESLzaunSym{s})$ 
        & $\Tr(\ESLzaunSym{s})$ 
        & order of $\ESLzaunSym{s}$ 
        & order of $U_{\ESLzaunSym{s}}\!\la I \ra$
        \\
        \midrule
        unitary & odd & $\ESLsym{s}^{n_t}$ &$1$&$d_j-1$&$2m+1$ &  $2m+1$
        \\
        unitary & even & $\ESLsym{s}^{n_t}$ & $1$&$d_j-1$ & $2(2m+1)$ & $2m+1$
        \\
        anti-unitary & odd & $\ESLsym{s}^{2n_t}$ &$1$&$d_j-1$& $2m+1$ &  $2m+1$
        \\
        anti-unitary & even & $\ESLsym{s}^{2n_t}$ & $1$&
        $d_j-1$
        &$2(2m+1)$ & $2m+1$
        \\
        \bottomrule
        \end{tabular}
    \end{center}
    \end{enumerate}
(see \Cref{dfn:level} for the level, $n_t$).
\end{thm}
\begin{rmkb}
    If it is true that $\la U_{\ESLsym{s}}\!\la I \ra \ra = \StabPiESL{\Pi_s}$, as the empirical observations suggest, then this result means:
\begin{enumerate}
    \item $\mcl{S}(\Pi_s)$ contains a projective anti-unitary if and only if $t$ is of anti-unitary type;
    \item $\mcl{S}(\Pi_s)$ is cyclic order $n_t(2m+1)$ if $t$ is of unitary type, and cyclic order $2n_t(2m+1)$ if $t$ is of anti-unitary type.
\end{enumerate}
\end{rmkb}
It follows from the above that, if $m=1$, then $U_{\ESLzaunSym{s}}\!\la I \ra$ is a canonical order 3 projective unitary (see \Cref{df:canonicalOrder3}). 
Our third main result establishes criteria for type-$z$ and type-$a$ orbits.
\begin{thm}\label{thm:FzAndFaSymmetries}
    Let $s=(t,G,g)$ be a fiducial datum containing the admissible tuple $t=(d,1,Q)\sim(K,j,1,Q)$.
    \begin{enumerate}
        \item if $d\not\equiv 3 \Mod{9}$, then $\ESLzaunSym{s}$ is conjugate to $F_{\rm{z}}$;
        \item if $d\equiv 3 \Mod{9}$, then there exist both type-$a$ and type-$z$ orbits.  Specifically:
        \begin{enumerate}
            \item if $f_j/f \equiv 0 \Mod{3}$, then $\ESLzaunSym{s}$ is conjugate to $F_{\rm{a}}$;
            \item if $f_j/f \not\equiv 0 \Mod{3}$, then $\ESLzaunSym{s}$ is conjugate to $F_{\rm{z}}$.
        \end{enumerate}
    \end{enumerate}
\end{thm}
\begin{rmkb}
    This result explains why one gets both type-$z$ and type-$a$ orbits when $d\equiv 3\Mod{9}$.  
    If one makes the additional assumption that $\la \ESLsym{s}\ra = \StabPiOverlap{\Pi}$ (as is the case in every instance where the groups have been explicitly calculated), then it also explains why one never finds type-$a'$ orbits when $d\equiv 6\Mod{9}$.  
\end{rmkb}
\begin{proof}
    See below.
\end{proof}
\begin{proof}[Proof of \Cref{thm:GLMorphHomOfQStabilizerGroup}.]
    Suppose $M_1$, $M_2\in \mcl{S}(Q)$. 
    Then it follows from Lemmas~\ref{lm:jMproperties} and~\ref{lm:lactqrt} that
    \eag{
     j_{(M_1M_2)^{-1}}(\qrt_t) &= j_{M^{-1}_2}(M^{-1}_1\qrt_t) j_{M^{-1}_1}(\qrt_t) = j_{M^{-1}_2}(\qrt_t) j_{M^{-1}_1}(\qrt_t),
    }
    implying
    \eag{
    \GLMorph_s(M_1M_2) &= \sgn\left(j_{(M_1M_2)^{-1}}(\qrt_t)\right) H\vpu{-1}_g G^{-1} [M_1M_2]\vpu{-1}_{\db} G H^{-1}_g
    =\GLMorph_s(M_1)\GLMorph_s(M_2).
    }
    So the restriction of $\GLMorph_s$ to $\mcl{S}(Q)$ is a homomorphism.
    It follows from \Cref{tm:symgp} that $\mcl{S}(Q)  = \la -I, \stabQGen{t}\ra$. 
    The fact that $j_{-I}(\qrt_t) = -1$ implies $\GLMorph_s(-I) = I$.  
    So $\GLMorph_s(\mcl{S}(Q)) = \la \ESLsym{s}\ra$.  Finally, it follows from \Cref{thm:MTransformedFiducials} that
    \eag{
       U^{\dagger}_{\ESLsym{s}}\Pi\vpu{\dagger}_{s\vpu{-1}} U\vpu{\dagger}_{\ESLsym{s}}&=\Pi_{s\vpu{-1}_{L_t}} = \Pi\vpu{\dagger}_s
    }
    implying $\ESLsym{s}\in \StabPiESL{\Pi_s}$.
\end{proof}
Before proving \Cref{thm:RsRzsProperties}, we need the following lemma.
\begin{lem}\label{lm:RsTermsLt}
Let $s=(t,G,g)$ be a fiducial datum containing the admissible tuple $t=(d,r,Q)\sim(K,j,m,Q)$.  Then
    \eag{
\ESLsym{s} &= \GLMorph_s(\stabQGen{t}) = 
\begin{cases}
    H\vpu{-1}_g G^{-1}[\stabQGen{t}]\vpu{-1}_{\db} G H^{-1}_g
    \qquad & \text{if $t$ is of unitary type,}
    \\
   - H\vpu{-1}_g G^{-1}[\stabQGen{t}]\vpu{-1}_{\db} G H^{-1}_g
    \qquad & \text{if $t$ is of anti-unitary type.}
    \label{eq:ESLsymsExpn}
\end{cases}
}
\end{lem}
\begin{proof}
    Let $f$ be the conductor of $Q$.
It follows from Theorems~\ref{thm:negunitcondb},~\ref{thm:canrep} and~\ref{tm:symgp} that $\Det(\stabQGen{t}) = \Nm(\un_f)= -1$ if and only if the tuple $t$ is of anti-unitary type.
 Also, by assumption, $\Tr(\stabQGen{t}^{-1}) >0$ 
 (see \Cref{dfn:AssociatedStabilizers}).  In view of \Cref{lem:fixedinda}, this means $\qrt_t \in \DD_{\stabQGen{t}^{-1}}$.  Referring to \Cref{def:sl2ldmndf}, we deduce
that $\sgn(j_{\stabQGen{t}^{-1}}(\qrt_t)) =\Det( \stabQGen{t}^{-1})$. \Cref{eq:ESLsymsExpn} follows from this and Definitions~\ref{dfn:GLHomomorphism},~\ref{df:ESLStabilizers}.
\end{proof}
\begin{proof}[Proof of \Cref{thm:RsRzsProperties}.]
We will first prove statement~\eqref{it:RsProperties}. 
It follows from \Cref{lm:RsTermsLt} that
\eag{
\Det(\ESLsym{s}) &= \Det\left([\stabQGen{t}]_{\db}\right)
=\begin{cases} 
1 \qquad & \text{if $t$ is of unitary type,}
\\
-1 \qquad & \text{if $t$ is of anti-unitary type.}
\end{cases}
}
Taking account of \Cref{lem:towerbasic} and Theorems~\ref{thm:negunitcondA}, ~\ref{thm:canrep} and~\ref{tm:symgp}, it also follows that
\eag{
\Tr(\ESLsym{s}) &=
\begin{cases}
\Tr\left(\stabQGen{t}\right) \qquad & \text{if $t$ is of unitary type,}
\\
-\Tr\left(\stabQGen{t}\right) \qquad & \text{if $t$ is of anti-unitary type,}
\end{cases}
\nn
&=
\begin{cases}
\Tr\left(\vn_f\right) \qquad & \text{if $t$ is of unitary type,}
\\
-\Tr\left(\un_f\right) \qquad & \text{if $t$ is of anti-unitary type,}
\end{cases}
\nn
&= \begin{cases}
d_{j_{\rm{min}}(f)}-1 \qquad & \text{if $t$ is of unitary type,}
\\
-\sqrt{d_{j_{\rm{min}}(f)}-3} \qquad & \text{if $t$ is of anti-unitary type.}
\end{cases}
}
It remains to calculate the orders of $\ESLsym{s}$, $U_{\ESLsym{s}}\!\la I \ra$. We consider the four cases separately.

\textit{Case 1:  $t$ is of unitary type and $d$ is odd.}  It follows from \Cref{lm:RsTermsLt} and \Cref{tm:symgp} that  
\begin{alignat}{3}
\ESLsym{s}^{\ell} &= I & \qquad &\iff & \qquad &\stabQGen{t}^{\ell} \equiv  I \Mod{d}
\nn
&&\qquad &\iff&\qquad & \stabQGen{t}^{\ell} \in \mcl{S}_d(Q)
\nn
&&\qquad &\iff&\qquad & \stabQGen{t}^{\ell} \in \la A_t\ra
\nn
&&\qquad &\iff&\qquad & \ell  n_t (2m+1) \mid \ell.
\end{alignat}
So $\ESLsym{s}$ is order $n_t(2m+1)$.  In view of  \Cref{thm:symplecticKernel}, this is also the order of $U_{\ESLsym{s}}\!\la I \ra$.

\textit{Case 2: $t$ is of unitary type and $d$ is even.} Let $\ell$ be the order of $\ESLsym{s}$.
It follows from \eqref{eq:ESLsymsExpn} and \Cref{tm:symgp} that
\begin{alignat}{3}
\ESLsym{s}^{\ell} &= I & \qquad &\iff & \qquad &\stabQGen{t}^{\ell} \equiv I \Mod{\db}
\nn
&&\qquad &\implies &\qquad & \stabQGen{t}^{\ell} \equiv  I \Mod{d}
\nn
&&\qquad &\iff &\qquad & \stabQGen{t}^\ell \in \la A_t\ra
\end{alignat}
implying $\stabQGen{t}^\ell$ is a power of $A_t$.  Since $A_t \equiv  (d+1)I \Mod{\db}$, we must in fact have 
$\stabQGen{t}^\ell=A_t^2 =\stabQGen{t}^{2n_t(2m+1)}$.  So $\ell=2n_t(2m+1)$.  Since $\stabQGen{t}^{n_t(2m+1)}=(d+1)I$, it follows from \Cref{thm:symplecticKernel} that the order of $U_{\ESLsym{s}}\!\la I \ra$ is $n_t(2m+1)$.

\textit{Case 3: $t$ is of anti-unitary type and $d$ is odd.}  As in Case 1, we have $\ESLsym{s}^\ell = I \iff \stabQGen{t}^\ell \in \la A_t\ra $. The fact that $t$ is of anti-unitary type means $\vn = \un^2$.  In view of \Cref{tm:symgp}, this means $A_t = \stabQGen{t}^{2n_t(2m+1)}$.  So $\ESLsym{s}$ is  order $2n_t(2m+1)$.   In view of  \Cref{thm:symplecticKernel}, this is also the order of $U_{\ESLsym{s}}\!\la I \ra$.

\textit{Case 4: $t$ is of anti-unitary type and $d$ is even.} Let $\ell$ be the order of $\ESLsym{s}$.  As in Case 2, we must have $\stabQGen{t}^\ell=A_t^2$.  As in Case 3, we have  $A_t = \stabQGen{t}^{2n_t(2m+1)}$.  So $\stabQGen{s}$ is order $4n_t(2m+1)$.  Since $\stabQGen{t}^{n_t(2m+1)}=(d+1)I$, it follows from \Cref{thm:symplecticKernel} that the order of $U_{\ESLsym{s}}\!\la I \ra$ is $2n_t(2m+1)$.

We will now prove statement~\eqref{it:RzsProperties}.  It follows from \Cref{lm:RsTermsLt} and \Cref{tm:symgp} that
\eag{
\ESLzaunSym{s}&= \GLMorph_s(\zaunerGen{t})
\nn
&=
\begin{cases}
    \GLMorph_s(\stabQGen{t}^{n_t})  \qquad & \text{$t$ is of unitary type,}
    \\
    \GLMorph_s(\stabQGen{t}^{2n_t})  \qquad & \text{$t$ is of anti-unitary type,}
\end{cases}
\nn
&= 
\begin{cases}
    H\vpu{-1}_g G^{-1}[L^{n_t}_t]\vpu{-1}_{\db}G H^{-1}_g \qquad & \text{$t$ is of unitary type,}
    \\
    H\vpu{-1}_g G^{-1}[L^{2n_t}_t]\vpu{-1}_{\db}G H^{-1}_g \qquad & \text{$t$ is of anti-unitary type.}
\end{cases}
}
It follows that 
\eag{
\ESLzaunSym{s}&=
\begin{cases}
    \ESLsym{s}^{n_t} \qquad & \text{$t$ is of unitary type,}
    \\
        \ESLsym{s}^{2n_t} \qquad & \text{$t$ is of anti-unitary type,}
\end{cases}
\label{eq:ESLzaunSymTermsESLsym}
}
and that $\Det(\ESLsym{s}) = 1$ irrespective of whether $t$ is of unitary or anti-unitary type.  It also follows that
\eag{
\Tr(\ESLsym{s}) &=
\begin{cases}
    \Tr(L^{n_t}_t) \qquad & \text{$t$ is of unitary type,}
    \\
    \Tr(L^{2n_t}_t) \qquad & \text{$t$ is of anti-unitary type,}
\end{cases}
\nn
&=\Tr(\vn_f^{n_t}) 
\nn
&= \Tr(\vn^j)
\nn
&= d_j-1.
}
Finally, \Cref{eq:ESLzaunSymTermsESLsym} in conjunction with statement~\ref{it:RsProperties} implies, firstly, that the order of $\ESLzaunSym{s}$ is $2m+1$ if $d$ is odd and $2(2m+1)$ if $d$ is even; and, secondly, that the order of $U_{\ESLzaunSym{s}}$ is $2m+1$ irrespective of the values of $d$, $t$.
\end{proof}
Before proving \Cref{thm:FzAndFaSymmetries}, we need the following lemma.
\begin{lem}\label{lm:fjCoprimeTod}
    Let $t=(d,r,Q)\sim(K,j,m,Q)$ be an admissible tuple. Then
    \eag{
      \gcd(f_j,d_j) &\equiv 
      \begin{cases}
          1 \qquad & d_j \not\equiv 3 \Mod{9},
          \\
          3 \qquad & d_j \equiv 3 \Mod{9}.
      \end{cases}
    }
\end{lem}
\begin{proof}
        Suppose $p$ is a prime divisor of $\gcd(f_j,d_j)$.  Then $p$ divides both $d_j$ and $(d_j-3)(d_j+1) = f_j^2\Delta_0$ where $\Delta_0$ is the discriminant of $K$.  Since $d_j$ is coprime to $d_j+1$, it must in fact be the case that $p$ divides both $d_j$ and $d_j+3$.  It follows that $p=3$ and $d_j$ is a multiple of 3.  We have thus shown that $\gcd(f_j,d_j)$ is a power of 3.  In particular, if $d_j$ is not a multiple of 3 then $f_j$ is coprime to $d$.

    Suppose $d_j$ is a multiple of 3. 
 We can write $d_j=3t$ for some integer $t>1$. Then $f_j^2\Delta_0 = 3(t-1)(3t+1)$, from which it can be seen that $f_j$ is a multiple of $3$ if and only if $t-1$ is a multiple of 3.  Equivalently, $f_j$ is a multiple of $3$ if and only if $d_j \equiv 3 \Mod{9}$. Combined with the result proved in the last paragraph this means: (a) if $d_j \not\equiv 3 \Mod{9}$, then $\gcd(f_j,d_j) = 1$, and (b) if $d_j \equiv 3 \Mod{9}$, then $\gcd(f_j,d_j) = 3$. 
\end{proof}
\begin{proof}[Proof of \Cref{thm:FzAndFaSymmetries}.]
    It follows from \Cref{thm:RsRzsProperties} that $\Det(\ESLzaunSym{s}) = 1$ and $\Tr(\ESLzaunSym{s}) = d-1$. In view of \Cref{thm:CanonicalOrder3ConjugacyClasses}, this means that $\ESLzaunSym{s}$ is conjugate to $F_z$ if $d\not\equiv 3$ or $6\Mod{9}$.  
    
    Suppose, on the other hand, that $d\equiv 3\Mod{9}$ (respectively, $d\equiv 6\Mod{9}$). Then it follows from Theorems \ref{thm:CanonicalOrder3ConjugacyClasses} and~\ref{thm:CriterionForFaFaprime} that $\ESLzaunSym{s}$ is conjugate to $F_a$ (respectively $F'_a$) if  $\ESLzaunSym{s} \equiv  I\Mod{3}$, and to $F_z$ otherwise.  To find the condition for this to be so, observe that Theorems~\ref{tm:symgp},~\ref{thm:RsRzsProperties}, and \Cref{lm:RsTermsLt} imply
    \eag{
        \ESLzaunSym{s}&= \begin{cases}
            \ESLsym{s}^{n_t} \qquad & \text{$t$ is of unitary type,}
            \\
            \ESLsym{s}^{2n_t}  & \text{$t$ is of anti-unitary type,}
        \end{cases}
        \nn
        &=\begin{cases}
            H\vpu{-1}_g G^{-1}[\stabQGen{t}^{n_t}]_{\db}G H^{-1}_g \qquad & \text{$t$ is of unitary type,}
            \\
             H\vpu{-1}_g G^{-1}[\stabQGen{t}^{2n_t}]_{\db}G H^{-1}_g  & \text{$t$ is of anti-unitary type,}
        \end{cases}
\nn
    &= \begin{cases}
            H\vpu{-1}_g G^{-1}[\canrep_Q(\un_f^{n_t}]_{\db}G H^{-1}_g \qquad & \text{$t$ is of unitary type,}
            \\
             H\vpu{-1}_g G^{-1}[\canrep_Q(\un_f^{2n_t})]_{\db}G H^{-1}_g  & \text{$t$ is of anti-unitary type,}
        \end{cases}
\nn
        &= H\vpu{-1}_g G^{-1}[\canrep_Q(\vn_f^{n_t}]_{\db}G H^{-1}_g
\nn
        &= H\vpu{-1}_g G^{-1}[\canrep_Q(\vn^j]_{\db}G H^{-1}_g.
    }
Taking account of \Cref{cr:etaqmodn} and the fact that $\db$ is divisible by $3$, this means
\begin{alignat}{3}
    && \qquad \ESLzaunSym{s}&\equiv  I \Mod{3}
    \nn
    &\iff& \qquad [\canrep_Q(\vn^j)]_{\db} &\equiv  I \Mod{3}
    \nn
    &\iff& \qquad \canrep_Q(\vn^j) &\equiv  I \Mod{3}
    \nn
    &\iff& \qquad \vn^j-1 &\in 3\mcl{O}_f
    \nn
    &\iff &\qquad \frac{d-3-f_j\Delta_0}{2} +f_j\left(\frac{\Delta_0+\sqrt{\Delta_0}}{2}\right)&= 3\left(p_1 + p_2f\left(\frac{\Delta_0+\sqrt{\Delta_0}}{2}\right)\right)
\end{alignat}
for some $p_1, p_2\in \mbb{Z}$. The fact that  $f_j^2\Delta_0 = (d-3)(d+1)$ means $f_j\Delta_0 \equiv  d-3 \Mod{6}$.  We conclude that $\ESLzaunSym{s} \equiv  I\Mod{3}$ if and only if $f_j/f \equiv 0 \Mod{3}$.  Since $d \equiv  0 \Mod{3}$, it follows from \Cref{lm:fjCoprimeTod} that $f_j/f$
 is coprime to $3$ if $d\equiv 6 \Mod{9}$ but divisible by $3$ if $d\equiv 3 \Mod{9}$.  So $\ESLzaunSym{s}$ is necessarily conjugate to $F_z$ if $d\equiv 6 \Mod{9}$.  On the other hand, if $d\equiv 3\Mod{9}$, then there are values of $f$ such that $f_j/f\equiv0\Mod{9}$, implying $\ESLzaunSym{s}$ is conjugate to $F_a$, and others such that $f_j/f \not\equiv 0\Mod{9}$, implying  $\ESLzaunSym{s}$ is conjugate to $F_z$.
 \end{proof}

 \subsection{Alignment}\label{ssc:alignment}
We now come to the phenomenon of \emph{SIC alignment}\cite{Appleby:2017,Andersson:2019}.  In the rank 1 case it is empirically observed, in every case examined, that, up to a sign, the squares of the \normalizedOverlapsText{} at position $d_j$ in a dimension tower reappear among the \normalizedOverlapsText{} at position $d_{2j}$. We will show that this phenomenon is a consequence of our results within our conjectural framework. 
Moreover, we will show that it generalizes to a relation between the \normalizedOverlapsText{} at positions $d_j$ and $d_{nj}$ in the tower, for any integer $n$ coprime to $3$.  

We first state the  main result of this subsection.
\begin{thm}\label{thm:alignment}
    Let $s=(t,G,g)$ be a fiducial datum containing the rank 1 admissible tuple $t=(d_j,1,Q)\sim (K,j,1,Q)$.  Let $n$ an a positive integer coprime to $3$.  Define  $t'=(d_{nj},1,Q)\sim (K,nj,1,Q)$,  $s'=(t',G,g)$,  and $\kappa =d_{nj}/d_j$.  Then $t'$ is an admissible tuple and $\kappa$ is an integer.  If $s'$ is also a fiducial datum, then
    \eag{
    \normalizedOverlapC{s'}{\kappa \mbf{p}} &= 
        (-1)^{\ell(\mbf{p})} \normalizedOverlapC{s}{\mbf{p}}^n
    }
     for all $\mbf{p}\not\equiv \zero \Mod{d_j}$,
    where
    \eag{
    \ell(\mbf{p}) &=
    \begin{cases}
        n+(1+p_1)(1+p_2) \qquad & \text{if $d_j$ is even and $n\equiv \pm 2$ or $\pm 5 \Mod{12}$}\, ,
        \\
        n+1 \qquad & \text{otherwise}\, .
    \end{cases}
    }
\end{thm}
\begin{rmkb}
    In the rank $1$ case the, empirical observations suggest that, if $t$ is an admissible tuple, then $(t,G,g)$ is a fiducial datum if and only if $\Det(G) = \pm 1$ and $g(\sqrt{\Delta_0}) = -\sqrt{\Delta_0}$.  If that is true, then $s'$ is automatically a fiducial datum, so this requirement does not need to be imposed as an additional assumption. 
\end{rmkb}
Before proving \Cref{thm:alignment}, we need to establish the following technical result:
\begin{lem}\label{lem:PreliminaryToAlignmentLemma}
    Let $t=(d,1,Q)\sim(K,j,1,Q)$ be an admissible tuple, and let $n$ be a positive integer coprime to 3.  Then $d_{nj}$ is divisible by $d_j = d$ and
    \eag{
    \frac{f_{nj}d_{nj}(1+d_{nj})}{f_j d_j}
    &\equiv
    \begin{cases}
        n(1+d_j) + d_j  \Mod{2d_j}
              \qquad & \text{if $d_j$ is even and $n\equiv \pm 2$ or $\pm 5\Mod{12}$,}\,
    \\
       n(1+d_j)\Mod{2d_j}
              \qquad & \text{otherwise.}\,
    \end{cases}
    }
\end{lem}
\begin{proof}
    The fact that $d_{nj}$ is divisible by $d_j$ follows from \Cref{thm:dnjfnjchebyexpns} and \Cref{lm:tustarprops}.

    Suppose $d_j$ is odd.  It follows from \Cref{lem:dfdelprops} that $d_{nj}$ is odd. So $d_{nj}/d_j$ is odd.  \Cref{thm:dnjfnjchebyexpns} and \Cref{lm:tustarprops} imply
    \eag{
      \frac{d_{nj}}{d_j} &= \frac{T^{*}_n(d_j)}{d_j}\equiv
      \begin{cases}
          n \Mod{d_j} \qquad &\text{if $n \equiv 1\Mod{3}$,}
          \\
         - n \Mod{d_j}  \qquad &\text{if $n\equiv 2 \Mod{3}$.}
      \end{cases}
    }
    The fact that $d_{nj}/d_j$ is odd means that, if $n$ is also odd, we must have $d_{nj}/d_j \equiv \pm n\Mod{2d_j}$, while if $n$ is even, we must have $d_{nj}/d_j \equiv\pm n + d_j\Mod{2d_j}$.  In other words:
    \eag{
         \frac{d_{nj}}{d_j} &\equiv 
      \begin{cases}
          n \Mod{2d_j}   \qquad &\text{if $n \equiv 1 \Mod{6}$,}
          \\
          -n + d_j  \Mod{2d_j} \qquad &\text{if $n \equiv2\Mod{6}$,}
          \\
          n + d_j  \Mod{2d_j} \qquad &\text{if $n \equiv4 \Mod{6}$,}
          \\
          -n   \Mod{2d_j}  \qquad &\text{if $n \equiv5\Mod{6}$.}
      \end{cases}
    }
    Turning to the ratio $f_{nj}/f_j$, 
    \Cref{thm:dnjfnjchebyexpns} and \Cref{lm:tustarprops}  imply
    \eag{
      \frac{f_{nj}}{f_j}
      =
      U^{*}_n(d_j)
      &\equiv
      \begin{cases}
          1 \Mod{d_j} \qquad &\text{if $n\equiv1\Mod{3}$,}
          \\
          -1 \Mod{d_j} \qquad &\text{if $n\equiv2 \Mod{3}$.} 
      \end{cases}
      \label{eq:UstardjOddExpn}
    }
    It follows from \Cref{lm:tustarprops} that $U^{*}_1(d_j)$ is odd, $U^{*}_2(d_j)$ is even, and $U^{*}_n(d_j) \equiv U^{*}_{n-2}(d_j)\Mod{2}$ for all $n>2$.  Consequently $U^{*}_n(d_j) \equiv n\Mod{2}$  for every positive integer $n$.  In conjunction with \eqref{eq:UstardjOddExpn}, this means that, if $n$ is odd, then $f_{nj}/f_j \equiv \pm 1\Mod{2d_j}$, while if $n$ is even, then $f_{nj}/f_j \equiv \pm 1 + d_j\Mod{2d_j}$.  In other words:
    \eag{
    \frac{f_{nj}}{f_j} &\equiv 
    \begin{cases}
        1 \Mod{2d_j} \qquad & \text{if $n\equiv1\Mod{6}$,}
        \\
        -1+d_j \Mod{2d_j} \qquad & \text{if $n\equiv2\Mod{6}$, }
        \\
        1 + d_j \Mod{2d_j} \qquad & \text{if $n\equiv4\Mod{6}$,}
        \\
        -1 \Mod{2d_j} \qquad & \text{if $n\equiv5\Mod{6}$.}
        \\
    \end{cases}
    }
    Putting these results together, we conclude
    \eag{
    \frac{f_{nj}d_{nj}(1+d_{nj})}{f_j d_j}
    &\equiv
    \begin{cases}
        n(1+nd_j)\Mod{2d_j}  \qquad & \text{if $n\equiv 1 \Mod{6}$,}
        \\
        (-1+d_j)(-n+d_j)(1+d_j(-n+d_j)) \Mod{2d_j} \qquad & \text{if $n\equiv 2 \Mod{6}$,}
        \\
        (1+d_j)(n+d_j)(1+d_j(n+d_j))\Mod{2d_j}  \qquad & \text{if $n\equiv 4 \Mod{6}$,}
        \\
        n(1-nd_j)) \Mod{2d_j} \qquad & \text{if $n\equiv 5 \Mod{6}$.}
    \end{cases}
    \nn
    &\equiv n(1+d_j)\Mod{2d_j}.
    }

    Suppose, on the other hand, that $d_j$ is even.  Then it follows from \Cref{thm:dnjfnjchebyexpns} and \Cref{lm:tustarprops} that
    \eag{
        \frac{d_{nj}}{d_j} = \frac{T^{*}_n(d_j)}{d_j}
       & \equiv \begin{cases}
            n+\left(\frac{n(n-1)}{6} \right)d_j \Mod{2d_j} \qquad & \text{if $n\equiv 1 \Mod{3}$,}
            \\
            -n+\left(\frac{n(n+1)}{6} \right)d_j \Mod{2d_j} \qquad & \text{if $n\equiv2\Mod{3}$,}
        \end{cases}
        \nn
        &\equiv\begin{cases}
            n \Mod{2d_j} \qquad & \text{if $n\equiv 1, 4\Mod{12}$,}
            \\
            -n+d_j \Mod{2d_j}  \qquad & \text{if $n\equiv 2, 5\Mod{12}$,}
            \\
            n+d_j \Mod{2d_j}  \qquad & \text{if $n\equiv 7,10\Mod{12}$,}
            \\
           - n \Mod{2d_j}  \qquad & \text{if $n\equiv 8, 11\Mod{12}$,}
        \end{cases}
    }
   and
    \eag{
     \frac{f_{nj}}{f_j} = U^{*}_n(d_j) 
            & \equiv \begin{cases}
            1+\left(\frac{n-1}{3} \right)d_j \Mod{2d_j} \qquad & \text{if $n\equiv 1\Mod{3}$,}
            \\
            -1+\left(\frac{n+1}{3} \right)d_j \Mod{2d_j} \qquad & \text{if $n\equiv 2\Mod{3}$,}
        \end{cases}
        \nn
        &\equiv \begin{cases}
            1 \Mod{2d_j} \qquad & \text{if $n\equiv 1, 7\Mod{12}$,}
            \\
            -1+d_j \Mod{2d_j}\qquad & \text{if $n\equiv2, 8\Mod{12}$,}
            \\
            1+d_j \Mod{2d_j} \qquad & \text{if $n\equiv4, 10\Mod{12}$,}
            \\
           - 1 \Mod{2d_j} \qquad & \text{if $n\equiv5, 11\Mod{12}$.}
        \end{cases}
    }
    Combining these results, we find
    \eag{
        \frac{f_{nj}d_{nj}(1+d_{nj})}{f_j d_j}
    &\equiv
    \begin{cases}
         n(1+nd_j) \Mod{2d_j} \qquad & \text{if $n\equiv 1 \Mod{12}$,}
        \\
        (-n+d_j)(-1+d_j)(1+d_j(-n+d_j))\Mod{2d_j} \qquad & \text{if $n\equiv 2  \Mod{12}$,}
        \\
        n(1+d_j)(1+nd_j) \Mod{2d_j} \qquad & \text{if $n\equiv 4  \Mod{12}$,}
        \\
        (n-d_j)(1+d_j(-n+d_j))\Mod{2d_j} \qquad & \text{if $n\equiv 5  \Mod{12}$,}
        \\
        (n+d_j)(1+d_j(n+d_j)) \Mod{2d_j}\qquad & \text{if $n\equiv 7  \Mod{12}$,}
        \\
         n(1-d_j) (1-nd_j)\Mod{2d_j}\qquad & \text{if $n\equiv 8  \Mod{12}$,}
        \\
        (n+d_j)(1+d_j)(1+d_j(n+d_j)) \Mod{2d_j}\qquad & \text{if $n\equiv 10  \Mod{12}$,}
        \\
        n(1-nd_j) \Mod{2d_j} \qquad & \text{if $n\equiv 11  \Mod{12}$,}
    \end{cases}
    \nn
      &\equiv  \begin{cases}
         n(1+d_j)\Mod{2d_j}  \qquad & \text{if $n\equiv\pm 1, \pm 4 \Mod{12}$,}
        \\
         n(1+d_j) + d_j \Mod{2d_j} \qquad & \text{if $n\equiv\pm 2, \pm 5\Mod{12}$,}
        \end{cases}
    }
    completing the proof.
\end{proof}
\begin{proof}[Proof of \Cref{thm:alignment}] The fact that $t'$ is admissible is immediate.  Referring to \Cref{dfn:SFKPhase}, we see
\eag{\label{eq:SFPhaseOftPrimeKappapExpression}
\SFPhase{t'}{\kappa\mbf{p}} &= (-1)^{s_{d_{nj}}(\kappa \mbf{p})}e^{-\frac{\pi i}{12}\rade(A_{t'})}\rtu_{d_{nj}}^{-\frac{f_{nj}}{f}Q(\kappa\mbf{p})}
}
where $f$ is the conductor of $Q$.  We first show
\eag{\label{eq:sPowerdSubnjExpression}
(-1)^{s_{d_{nj}}(\kappa \mbf{p})} &= (-1)^{n+1}(-1)^{ns_{d_j}(\mbf{p})}.
}
Indeed, if $d_j$ is odd, then it follows from \Cref{lem:dfdelprops} that $d_{nj}$ is odd, implying
\eag{
(-1)^{s_{d_{nj}}(\kappa \mbf{p})}&=-1 =(-1)^{n+1}(-1)^{ns_{d_j}(\mbf{p})}.
}
On the other hand, if $d_j$ is even, then it follows from Lemmas \ref{lm:tustarprops} and~\ref{lem:dfdelprops} that $d_{nj}$ is even and $\kappa= T^{*}_n(d_j)/d_j \equiv n\Mod{2}$, implying
\eag{
(-1)^{s_{d_{nj}}(\kappa \mbf{p})}&= (-1)^{(1+np_1)(1+np_2)} = (-1)^{n+1}(-1)^{n(1+p_1)(1+p_2)} = (-1)^{n+1}(-1)^{ns_{d_j}(\mbf{p})}.
}
Turning to the second factor on the right hand side of \eqref{eq:SFPhaseOftPrimeKappapExpression}, observe that it follows from \Cref{tm:symgp} that 
\eag{\label{eq:AtPrimePowerOfAt}
A\vpu{n}_{t'}& =\canrep_Q(\vn^{3nj}) =  A_t^n
}
and
\eag{
\Tr(A_{t})&= \Tr(\vn^{3j}) = d_{3j}-1 > 1.
}
\Cref{lem:RademacherProperties} consequently implies
\eag{
e^{-\frac{\pi i}{12}\rade(A_{t'})} &= 
e^{-\frac{\pi i}{12}\rade(A_{t}^n)}
=e^{-\frac{n\pi i}{12}\rade(A_{t})}.
}
Finally, using \Cref{lem:PreliminaryToAlignmentLemma}, the last
 factor on the right hand side of \eqref{eq:SFPhaseOftPrimeKappapExpression} becomes
 \eag{
 \rtu_{d_{nj}}^{-\frac{f_{nj}}{f}Q(\kappa \mbf{p})} &= e^{-\left(\frac{\pi i}{d_j}\right)\left(\frac{f_{nj}d_{nj}(1+d_{nj})}{d_{j}f_j}\right)\left(\frac{f_j}{f}Q(\mbf{p})\right)}
 \nn
 &=\begin{cases}
     e^{-\left(\frac{\pi i}{d_j}\right)\left(n(1+d_j)+d_j\right)\left(\frac{f_j}{f}Q(\mbf{p})\right)} \qquad & \text{if $d_j$ is even and $n\equiv \pm 2$ or  $\pm 5\Mod{12}$,}
     \\
     e^{-\left(\frac{\pi i}{d_j}\right)\left(n(1+d_j)\right)\left(\frac{f_j}{f}Q(\mbf{p})\right)}\qquad & \text{otherwise,}
 \end{cases}
 \nn
 &=\begin{cases}
     (-1)^{\frac{f_j}{f}Q(\mbf{p})}\rtu_{d_j}^{-\frac{nf_j}{f}Q(\mbf{p}}\qquad & \text{if $d_j$ is even and $n\equiv \pm 2$ or  $\pm 5\Mod{12}$,}
     \\
     \rtu_{d_j}^{-\frac{nf_j}{f}Q(\mbf{p}}\qquad & \text{otherwise.}
 \end{cases}
 }
      It follows from \Cref{lem:fjfqcmpsoddeven} that, if $d_j$ is even, then  
     \eag{
     (-1)^{\frac{f_j}{f}Q(\mbf{p})} &= (-1)^{p_1^2+p_1p_2+p_2^2} = (-1)^{n+1}(-1)^{n+(1+p_1)(1+p_2)}.
     }
      Hence
      \eag{
       \rtu_{d_{nj}}^{-\frac{f_{nj}}{f}Q(\kappa \mbf{p})} &= (-1)^{n+1}(-1)^{\ell(\mbf{p})}
        \rtu_{d_j}^{-\frac{nf_j}{f}Q(\mbf{p})}.
      }
      Combining these results, we deduce
\eag{
\SFPhase{t'}{\kappa\mbf{p}} &= 
(-1)^{\ell(\mbf{p})}(-1)^{ns_d(\mbf{p})}
e^{-\frac{n\pi i}{12}\rade(A_{t})}
        \rtu_{d_j}^{-\frac{nf_j}{f}Q(\mbf{p})}
= (-1)^{\ell(\mbf{p})}\left(\SFPhase{t}{\mbf{p}}\right)^n.
}
Using \eqref{eq:AtPrimePowerOfAt}, together with the fact that $\qrt_{t'}=\qrt_{Q,+}=\qrt_t$ (see \Cref{dfn:GhostOverlaps}), we also find
\eag{
\sfc{d_{nj}^{-1} \kappa\mbf{p}}{A_{t'}}{\qrt_{t'}} &= \sfc{d_j^{-1}\mbf{p}}{A^n_t}{\qrt_t}.
}
It follows from~\eqref{eq:coboundary} that, for all $\tau\in \mbb{H}$,
\eag{
    \sfc{d^{-1}_j \mbf{p}}{A^n_t}{\tau}
    &=\frac{\varpi\left(\pt{d^{-1}_j \mbf{p}}{A^n_{t}\tau},A^n_t \tau\right)}{\varpi\left(\pt{d^{-1}_j \mbf{p}}{\tau},\tau\right)}
    \nn
    &= \frac{\varpi\left(\pt{d^{-1}_j \mbf{p}}{A^n_{t}\tau},A^n_t \tau\right)}{\varpi\left(\pt{d^{-1}_j \mbf{p}}{A^{n-1}_t\tau},A^{n-1}_t\tau\right)}
    \times
    \dots
    \times
    \frac{\varpi\left(\pt{d^{-1}_j \mbf{p}}{A_{t}\tau},A_t \tau\right)}{\varpi\left(\pt{d^{-1}_j \mbf{p}}{\tau},\tau\right)}
     }
     Taking the limit as $\tau \to \qrt_t$ and using the fact that $\qrt_t$ is a fixed point of $A_t$ we deduce
     \eag{
     \sfc{d_{nj}^{-1} \kappa\mbf{p}}{A_{t'}}{\qrt_{t'}}
     &= \left(\sfc{d_j^{-1}\mbf{p}}{A_t}{\qrt_{t}}\right)^n.
     }
     Hence
     \eag{
    \normalizedGhostOverlapC{t'}{\kappa\mbf{p}} = 
    \SFPhase{t'}{\kappa\mbf{p}}
    \sfc{d_{nj}^{-1}\kappa \mbf{p}}{A_{t'}}{\qrt_{t'}}=(-1)^{\ell(\mbf{p})}\left(\SFPhase{t}{\mbf{p}}\right)^n \left(\sfc{d_j^{-1}\mbf{p}}{A_t}{\qrt_{t}}\right)^n = (-1)^{\ell(\mbf{p})}
    \left(\normalizedGhostOverlapC{t}{\mbf{p}}\right)^n
\label{eq:NormalizedGhostPowerFormula}
     }
     It  follows from \Cref{cor:AlternativeDefinitionOfCandidateGhostOverlaps} and \Cref{lm:dnjPlus1Expression} that
     \eag{
     \frac{\ghostOverlapC{t'}{\kappa \mbf{p}}}{\left(\ghostOverlapC{t}{\mbf{p}}\right)^n}
     &= 
     q \left(\frac{\normalizedGhostOverlapC{t'}{\kappa\mbf{p}}}{\left(\normalizedGhostOverlapC{t}{\mbf{p}}\right)^n}\right),
&
     \frac{\overlapC{s'}{\kappa \mbf{p}}}{\left(\overlapC{s}{\mbf{p}}\right)^n}
     &=
          q \left(\frac{\normalizedOverlapC{s'}{\kappa\mbf{p}}}{\left(\normalizedOverlapC{s}{\mbf{p}}\right)^n}\right),
          \label{eq:NormalizedOverlapRatios}
     }
where
\eag{
q &= 
\begin{cases}
\frac{(d_j+1)^{\frac{n}{2}}}{d_{\frac{nj}{2}}-1}
\qquad & n \equiv 0\Mod{2},
\\
\frac{(d_j+1)^{\frac{n-1}{2}}}{1+\sum_{r=1}^{\frac{n-1}{2}}(-1)^r d_{rj}} \qquad & n\equiv 1\Mod{4},
\\
\frac{(d_j+1)^{\frac{n-1}{2}}}{2+\sum_{r=1}^{\frac{n-1}{2}}(-1)^r d_{rj}} \qquad & n\equiv 1\Mod{4}.
\end{cases}
}
The fact that $q\in \mbb{Q}$ means we can use \Cref{lm:OverlapTermsGhostOverlap} in conjunction with ~\eqref{eq:NormalizedGhostPowerFormula} and~\eqref{eq:NormalizedOverlapRatios} to deduce
\eag{
\frac{\normalizedOverlapC{s'}{\kappa\mbf{p}}}{\left(\normalizedOverlapC{s}{\mbf{p}}\right)^n}
&= q^{-1}g\left(\frac{\ghostOverlapC{t'}{\kappa GH_g^{-1}G^{-1} \mbf{p}}}{\left(\ghostOverlapC{t}{ GH_g^{-1}G^{-1}\mbf{p}}\right)^n}\right)
=g\left(\frac{\normalizedGhostOverlapC{t'}{\kappa GH_g^{-1}G^{-1} \mbf{p}}}{\left(\normalizedGhostOverlapC{t}{ GH_g^{-1}G^{-1}\mbf{p}}\right)^n}\right)
=(-1)^{\ell(GH_g^{-1}G^{-1}\mbf{p})}.
}
     
     It remains to show that
     \eag{\label{eq:lOfpTransform}
     (-1)^{\ell(GH_g^{-1}G^{-1}\mbf{p})} &= (-1)^{\ell(\mbf{p})}.
     }
     The statement is immediate if $d_j$ is odd, since then $\ell(\mbf{p}) = n+1$ independently of $\mbf{p}$.  Suppose, on the other hand, that $d_j$ is even.  Let 
     \eag{
     M & = GH_g^{-1} G^{-1} = \bmt \ma & \mb \\ \mc &\md \emt.
     }
     The fact that $M \in \GLtwo{\mbb{Z}/2d_j\mbb{Z}}$ means $\Det M$ is odd.  It also means that at least one of $\ma, \mc$, and at least one of $\mb, \md$ must be odd,  implying  
     $(1+\ma)(1+\mc)$ and $(1+\mb)(1+\md)$ are both even.  Hence
     \eag{
        (-1)^{(1+(M\mbf{p})_1)(1+(M\mbf{p})_2)}
        &= 
        (-1)^{1+(\ma p_1 + \mb p_2) + (\mc p_1 + \md p_2) + (\ma p_1 + \mb p_2)  (\mc p_1 + \md p_2) }
        \nn
        &= (-1)^{1+(\ma + \mc + \ma \mc)p_1 + (\mb + \md + \mb \md) p_2 + (\ma \md + \mb \mc) p_1 p_2}
        \nn
        &= (-1)^{1+p_1+p_2 + (1+\ma)(1+\mc)p_1 + (1+\mb)(1+\md) p_2 + \Det(M) p_1p_2}
        \nn
        &= (-1)^{(1+p_1)(1+p_2)}.
     }
     \Cref{eq:lOfpTransform} now follows.  Hence
     $\normalizedOverlapC{s'}{\kappa \mbf{p}}
     =(-1)^{\ell(\mbf{p})}\left(\normalizedOverlapC{s}{\mbf{p}}\right)^n$.
\end{proof}

\section{Necromancy and numerical computation}
\label{sec:nec}

Suppose one wishes to use the preceding conjectures for constructing ghosts to compute an \textit{explicit} $r$-SIC fiducial, either exactly or just a numerical approximation. 
We call any procedure for doing so \textit{necromancy}, so-named because it ``reanimates'' the ghost fiducial as a $r$-SIC fiducial. 
In this section, we describe a method for necromancy specialized to $1$-SICs. 

There is a straightforward brute-force algorithm to achieve necromancy. 
Using \eqref{eq:ghostoverlapformula}, we first compute a ghost fiducial to arbitrary precision. 
We wish to round this numerical approximation into the closest point in a candidate number field to identify an exact representation of the ghost fiducial. 
Powerful tools with polynomial runtimes such as the Lenstra--Lenstra--Lovász (LLL) lattice basis reduction algorithm or other integer relation algorithms can be employed here, though we note that finding a closest vector in a lattice is not believed to be efficient in general. 
A conjecture for the specific number field is provided by the existing conjectures~\cite{Appleby:2020,Kopp2020b,Kopp2020c}. 
Then a Galois automorphism that flips the sign of $\sqrt{\Delta_0}$ can be applied to compute the $1$-SIC fiducial. 
If the conjectures are correct, then with enough starting precision and patience for finding the exact representation, the result should be an exact expression for a $1$-SIC fiducial. 
This exact expression can then be evaluated to any precision that one likes for numerical approximation. 

Unfortunately this approach is impractical for two reasons.  
The first difficulty is that the relevant number field typically has very high degree, and the precision required to round into this field would therefore be impractically large for even storing a $1$-SIC fiducial for modestly large $d$.  
Second, the convergence of LLL or similar algorithms is either far too slow in practice with such high-degree number fields or yields too weak a guarantee on accuracy with the ``true'' closest point to allow a direct computation of a ghost in this fashion. 

To circumvent these difficulties, we now describe an alternative heuristic approach to necromancy that avoids the complexity bottleneck by working entirely in a (typically much lower degree) ring class field. 
One still requires very high precision already for moderately large dimensions. 
For example, we required $\sim 10^5$ digit precision in dimension $d = 100$ to implement this approach in detail. 
However, this is not impractically large for a modern laptop. 
The price for this reduced complexity is that the method itself is more involved. 

We stress that the procedure for necromancy discussed in this section is at present only a heuristic, even assuming our conjectures. 
However, we believe it should be possible to specify a complete algorithm which provably converges to a $1$-SIC assuming only our conjectures. 
It is likely that the methods here extend to $r$-SICs for $r>1$, but we have not attempted to systematically approach numerical computation for $r>1$, so we leave this case to future work and focus on $1$-SICs for the rest of this section. 

Let us first provide a high-level overview of our method of necromancy. 
The first step is to calculate a numerical estimate of a ghost fiducial associated to an admissible tuple $t = (d,1,Q)$. 
In \Cref{ssc:CalculatingShinFunction} we show how this can be done using the integral representation of the double sine function~\cite{Shintani1977,Kopp2020d}. 
It is not practical to use numerical integration to achieve the high precision required by the subsequent steps. 
We therefore use it to calculate an initial, low precision  fiducial, and then amplify its precision using Newton's method, as described in \Cref{subsec:precisionbump}. 
We then describe in \Cref{ssc:ghostinvs} how to (numerically) compute a set of invariants that we call \emph{ghost invariants}. 
The exact versions of these ghost invariants conjecturally live in the ring class field $H_t$ associated to $t$ (recall \Cref{df:RingClassField}). 
Since this is a field extension with substantially lower degree than the field containing the \ghostOverlapsText{}, we can find an exact representation of the ghost invariants without too much difficulty using an integer relation algorithm. 
Importantly, the ghost invariants contain enough information about the original \ghostOverlapsText{} to reconstruct them up to an action of the Galois group. 
Thus, as we show in \Cref{ssc:reconstructSICoverlaps}, we can find exact Galois conjugates for the ghost invariants and from them obtain the subsequent ``SIC invariants''. 
The SIC invariants are at this point specified as exact numbers in a number field, but to make the remaining steps computationally tractable, we again resort to numerical approximations after the conversion step of passing from ghost to SIC. 
Notably, we no longer need the ultra-high precision required in the initial rounding step.
From the SIC invariants, we can numerically reconstruct the \SICoverlapsText{}, again up to an action of the Galois group. 
While in principle we can do this on each separate Galois orbit and then combine the results, we can apply an additional heuristic to avoid having to match Galois actions across multiple orbits. 
In \Cref{ssc:convexoptim} we describe how a convex optimization on a single maximal orbit, in every case tried, avoids the need to search for the unknown Galois action across multiple orbits. 
The output of this procedure is a $1$-SIC fiducial vector of moderate precision; to gain confidence that this is a true $1$-SIC fiducial vector, one can again use Newton's method to enhance the precision to any desired level. 

\subsection{Numerical calculations}

We have implemented the necromancy method described here in a free open-source Julia package called \texttt{Zauner.jl}~\cite{SteveGitHub}, which makes essential use of \texttt{Hecke.jl}~\cite{Hecke}. 
Using this implementation, we have calculated numerical approximations to all of the $1$-SICs in every dimension up to $d = 20$, with the exception of the $1$-SIC in $d=12$, which has $F_a$ symmetry. 
We have not yet implemented necromancy for the $F_a$-symmetric case because of how the Galois group structure is presently computed in our software. 
However, including $F_a$ orbits should be a relatively straightforward extension that we hope to implement soon. 

To see if we could achieve a moderately large dimension with our approach, we also used this algorithm to compute four numerical $1$-SICs in dimension $d=100$, three of which are new. 
This was somewhat challenging computationally, and it seems likely that our current ideas will need refinement to allow computation of a $1$-SIC in, say, $d=1000$. 

The only required input to the \texttt{necromancy} function is an admissible tuple $(d,1,Q)$, and in principle this function requires no fine-tuning to output a $1$-SIC. 
As a practical matter however, the bottleneck in extending our current implementation to more dimensions is the convergence of the integer relation algorithm. 
The efficacy of the method hinges on convergence to the correct element of the ring class field for every ghost invariant. 
We find in practice that this is challenging to achieve within our current heuristics. 
We hope to improve the speed, convergence, and generality of this code (in particular to $F_a$ orbits and $r > 1$) in future versions.

Throughout the remainder of this section, we are implicitly working with a fixed admissible tuple $(d,1,Q)$. We further assume that the twist $G=I$, so that the fiducial datum is   $(d,1,Q,I,g)$ for some  $g$.
To ease notation, nowhere in this section do we explicitly label the dependence on the tuple or datum.  In particular the generators of $\mcl{S}(Q)$, $\mcl{S}_d(Q)$ specified in \Cref{dfn:AssociatedStabilizers} will simply be denoted $L$, $A$ respectively, and the root of $Q$ appearing in \Cref{dfn:GhostOverlaps} will simply be denoted $\rho$.

\subsection{Calculating the Shintani--Faddeev modular cocycle}\label{ssc:CalculatingShinFunction}

The first step in our necromancy method is to compute a ghost fiducial vector. 
In Definitions~\ref{dfn:GhostOverlaps},~\ref{dfn:CandidateGhostAndSICFiducials}, 
a ghost fiducial corresponding to the admissible tuple $(d,r,Q)$ is expressed in terms of the \SFKShort{} modular cocycle  $\sfc{d^{-1}\mbf{p}}{A}{\qrt}$. 
This in turn is expressed in terms of the \SFKShort{} Jacobi cocycle $\sfj{A}{z}{\tau}$ via \eqref{eq:shindf}. 
For $\tau\in \mbb{H}$ the latter is given by a ratio of $q$-Pochhammer symbols via \eqref{eq:sfjfrm}. 
However, we need its values for $\tau \in \mbb{R}$. 
Although these can be obtained by taking a limit, it is numerically more efficient to calculate them directly, via an integral representation, using a procedure we now describe.

Let $L=\smt{\ma & \mb \\ \mc & \md}\in \SL_2(\mbb{Z})$ be arbitrary. Recall (\Cref{dfn:GeneratorsOfSL2Z}) that $\SL_2(\mbb{Z})$ is generated by $T=\smt{1 & 1 \\ 0 & 1}$, $S=\smt{0 & -1\\ 1 & 0}$.
If $\mc =0$ then either $L=T^k$, in which case $\sfj{L}{z}{\tau}=1$ for all $z$, $\tau$, or $L=-T^k$, in which case the function is singular for all $\tau \in \mbb{R}$. 
If $\mc<0$ we can express $\sigma_L$ in terms of $\sigma_{L^{-1}}$. 
It is therefore sufficient to give a procedure for calculating the function when $\mc>0$. 
For this we use Theorem~\ref{thm:tsaltexpn} to deduce the existence of a sequence of integers $r_1,\dots , r_{n+1}$ such that 
\eag{
L&= T^{r_1}S \dots ST^{r_{n+1}}
\label{eq:lexpn}
}
where $r_j\ge 2$ for $1 < j < n+1$. 
The theorem also shows that if $L_j = T^{r_j}S \dots S T^{r_{n+1}}$ then $\DD_L = \DD_{L_1} \subset \DD_{L_2} \dots \subset \DD_{L_{n+1}}$. 
Consequently
\eag{
\sfj{L}{z}{\tau}&= \sfj{T^{r_1}S}{\frac{z}{j_{L_2}(\tau)}}{L_2\cdot\tau} \dots 
 \sfj{T^{r_n}S}{\frac{z}{j_{L_{n+1}}(\tau)}}{L_{n+1}\cdot\tau}\sfj{L_{n+1}}{z}{\tau}
}
for all $\tau \in \DD_L$. 
Using
\eag{
\sfj{L_{n+1}}{\tau}{z}& = \sfj{T^{r_{n+1}}}{z}{\tau} = 1
}
and
\eag{
\sfj{T^mS}{z}{\tau} &= \sfj{T^m}{\frac{z}{j_S(\tau)}}{S\cdot\tau} \sfj{S}{z}{\tau} =\sfj{S}{z}{\tau}, 
}
this becomes
\eag{
\sfj{L}{z}{\tau}&= \sfj{S}{\frac{z}{j_{L_2}(\tau)}}{L_2\cdot\tau} \dots 
 \sfj{S}{\frac{z}{j_{L_{n+1}}(\tau)}}{L_{n+1}\cdot\tau}.
\label{eq:sfjldecom}
}
for all $\tau \in \DD_L$.
The problem thus reduces to calculating $\sfj{S}{z}{\tau}$.  This can be done using the formula~\cite{Shintani1977,Kopp2020d}
\eag{
\sfj{S}{z}{\tau}&= \frac{e^{\frac{\pi i}{12\tau}\left(6z^2+6(1-\tau)z+\tau^2-3\tau+1\right)}}{\dbs(z+1,\tau)}
\label{eq:SFJacobiCocycleTermsDoubleSine}
}
where $\dbs(z+1,\tau)$ is the double sine function.
Note that the double sine function as usually defined has three arguments; we are employing the shorthand~\cite{Kopp2020d} $\dbs(z,\tau,1) = \dbs(z,\tau)$.  Note also that we are using the definition of Shintani~\cite{Shintani1977} and Kurokawa and Koyama~\cite{Kurokawa:2003} which is prevalent in the mathematics literature, as opposed to the definition of Ponsot~\cite{Ponsot:2004} which is more prevalent in the physics literature.
We can calculate $\dbs(z+1,\tau)$ explicitly using the integral representation\cite{Ponsot:2004,Kopp2020d},
\eag{
\dbs(z+1,\tau)&=\exp\left(-\int_0^{\infty}\left(\frac{\sinh\left(\frac{\tau-1-2z}{2}\right)t}{2\sinh\left(\frac{t}{2}\right)\sinh\left(\frac{\tau t}{2}\right)}-\frac{\tau-1-2z}{\tau t}\right)\frac{dt}{t}\right)
\label{eq:dsintrep}
}
valid for $\re(\tau)>0$ and $-1< \re(z) < \re(\tau)$. 
To use this integral representation in ~\eqref{eq:sfjldecom}, one needs, firstly, that $\re(L_r\cdot\tau) >0$ for $r=2$, \dots, $n+1$. 
A sufficient condition for that to be true is that $\re(j_L(\tau))>0$. 
In particular, it is true for the case that interests us, $\tau \in \DD_L \cap \mbb{R}$. 
Indeed, let $\tau = x+i y$ with $x$, $y\in \mbb{R}$. 
Then, in the notation of Theorem~\ref{thm:tsaltexpn}, $\re(j_L(\tau))>0$ implies $x+\md_r/\mc_r >0$ for $r=2, \ldots, n+1$ and, consequently,
\eag{
\re\left(L_r\cdot\tau\right)&=
\begin{cases}
\frac{\mc_r}{\mc_{r+1}}\left(\frac{\left(x+\frac{\md_r}{\mc_r}\right)\left(x+\frac{\md_{r+1}}{\mc_{r+1}}\right)+y^2}{\left(x+\frac{\md_{r+1}}{\mc_{r+1}}\right)^2+y^2}\right) \qquad &r=2,\dots , n
\\
\mc_{n+1}\left(x+\frac{\md_{n+1}}{\mc_{n+1}}\right)\qquad & r = n+1
\end{cases}
\nn
&>0
}
To deal with the problem that $\re\left(z/j_{L_r}(\tau) \right)$ may not be in the required interval we may use the fact 
\eag{
\sfj{S}{z+m_1\tau+m_2}{\tau} &=\frac{\qp_{m_1}(z,\tau)}{\qp_{-m_2}\left(\frac{z}{\tau},-\frac{1}{\tau}\right)}\sfj{S}{z}{\tau}
}
for all $m_1$, $m_2\in \mbb{Z}$ and all $\tau\in \DD_S$.

If we can calculate one $1$-SIC on a given extended Clifford group orbit, then we can easily calculate all the others by applying the appropriate unitary or anti-unitary transformation. 
To make the most efficient use of available resources, one should choose the quadratic form $Q$ appearing in the admissible tuple $t=(d,1,Q)$ in such a way as to minimize the length of the expansion in ~\eqref{eq:lexpn}. 
It follows from Theorem~\ref{thm:effexpn} that to do this we need to choose $Q$ to be HJ-reduced and such that the HJ-continued fraction expansion of $\beta_{Q,+}$ has minimal period. 
The HJ-reduced continued fraction then tells us the integers $r_1$, \dots, $r_{n+1}$ in the expansion in ~\eqref{eq:lexpn} (see Appendix~\ref{ap:hjcont} for the definitions and a summary of the relevant properties of HJ-reduced forms and HJ-continued fractions). 
Appropriate choices of $Q$ are tabulated in Appendix~\ref{ap:sicdata} for $d=4$ to $100$.

\subsection{Precision enhancement with Newton's method}
\label{subsec:precisionbump}
The method described below does not require us to calculate the full set of $d^2$ \ghostOverlapsText{}. 
It does, however, require us to calculate a subset of size $\gg d \log(d)^c$ for some absolute constant $c$. 
To convert this subset into the corresponding subset of \SICoverlapsText{} requires the \ghostOverlapsText{} to be calculated to very high precision ($10^5$ digit precision in dimension 100). 
Doing this using ~\eqref{eq:dsintrep} alone would be prohibitively slow. 
We therefore use the alternative method we now describe.
\begin{lem}
    Let $\tilde{\Pi}$ be a ghost $1$-SIC fiducial.  Then there exists $|\psi\ra \in \mcl{L}(\mbb{C}^d)$ such that
    \eag{
    \tilde{\Pi} &=\lambda |\psi\ra \la \psi | U_P
    }
    where $\lambda = \la \psi |U_P|\psi\ra  = \pm 1$.
\end{lem}
\begin{rmkb}
    We will refer to $|\psi\ra$ as the ghost SIC fiducial vector.
\end{rmkb}
\begin{proof}
    The fact that $\tilde{\Pi}$ is rank 1 means we can write it in the form 
\eag{
\tilde{\Pi} &= |\psi\ra \la \phi|
}
for some pair of vectors satisfying $\la \phi|\psi\ra = 1$.  The fact that $\tilde{\Pi}_s$ is a \pprj{} (see \Cref{def:pProjector} and discussion following) means
\eag{
|\phi\ra \la \psi| &= \left(|\psi\ra \la \phi|\right)^{\dagger} = U_P|\psi\ra \la \phi|U_P &&\implies & |\phi\ra &= \left(\frac{\la \phi |U_P | \psi\ra}{\la \psi|\psi\ra}\right) U_P|\psi\ra.
}
So
\eag{
\tilde{\Pi} &= \lambda |\psi\ra \la \psi | U_P
}
for some $\lambda$.  The fact that $\Tr(\tilde{\Pi}) =1$ means $\lambda \la \psi|U_P|\psi\ra=1$.  Our freedom to make  the replacements $|\psi\ra \to \sqrt{|\lambda|} | \psi\ra$ and $\lambda\to \lambda/|\lambda|$ means we can choose $\lambda$, $|\psi\ra$ so that $\lambda = \la \psi |U_P|\psi\ra  = \pm 1$.
\end{proof}
Expressed in terms of the ghost SIC fiducial vector,
~\eqref{eq:ghostProjectorDef} becomes
\eag{
\lambda |\psi\ra \la \psi | U_P &= \frac{1}{d}\sum_{\mbf{p}}
\ghostOverlap_{\mbf{p}}D_{\mbf{p}}.
}
(setting $G=I$, and dropping the $t$-label).
The simplest case is when none of the components of $|\psi\rangle$ is zero. 
In that case, in order to calculate the $2d-2$ real numbers determining the complex vector $|\psi\rangle$ (up to an overall constant) it suffices to know $2d$ of the numbers $\ghostOverlap_{\mbf{p}}$. 
Specifically, let
\eag{
 \chi_{0,j} &= \frac{1}{d}\sum_{k=0}^{d-1}\rtus_d^{jk}\ghostOverlap_{0,k},
 &
  \chi_{1,j} &= \frac{1}{d}\sum_{k=0}^{d-1}\rtu_d^k\rtus_d^{jk}\ghostOverlap_{1,k}.
}
The vector $|\psi\rangle$ is only determined up to an arbitrary phase. 
We may therefore assume, without loss of generality, that $\la 0 |\psi\rangle$ is positive real. 
We then have
 \eag{
 \la j |\psi\rangle
 &=
 \sqrt{|\chi_{0,0}|}
 \prod_{k=0}^{j-1}\left(\frac{\chi_{1,k}}{\chi_{0,k}}\right),
 }
with the convention that $\prod_{k=0}^{-1} f(k) = 1$. 
Our strategy is therefore to calculate low precision approximations to the $2d$ numbers $\ghostOverlap_{0,k}$, $\ghostOverlap_{1,k}$ using the integral representation and then use these to calculate a low precision approximation to the vector $|\psi\rangle$. 
We then apply Newton's method to the system of equations
\eag{
\la \psi |U_P D_{\mbf{p}} |\psi\ra \la \psi |U_P D_{-\mbf{p}} |\psi\ra 
&=\frac{d\delta^{(d)}_{\mbf{p},\boldsymbol{0}}+1}{d+1}
}
to calculate a high precision approximation to $|\psi\rangle$ which in turn can be used to calculate high precision approximations to the numbers $\ghostOverlap_{\rm{p}}$. 
If it should happen that one or more of the components of $|\psi\rangle$ is zero then it may be necessary to calculate more than $2d$ low precision \ghostOverlapsText{}. 
However in no case is it necessary to calculate the numerical integral to high precision.

\subsection{Ghost invariants}
\label{ssc:ghostinvs}
We next describe our method for constructing a set of numbers in the ring class field $H$ which fully specify the \ghostOverlapsText{} defined from the admissible tuple $(d,1,Q)$ on a maximal Galois orbit. 

The method relies on the following empirical observation\footnote{While we do not have a proof of this observation at present, we believe it can be proven from our broader framework of conjectures and hope to show this in an upcoming paper.}:
\begin{empirical}
    If $\GE$ is the field generated by the \ghostOverlapsText{}, and if $H$ is the ring class field, then there is an isomorphism of $\Gal(\GE/H)$ onto $\mcl{M}/\mcl{S}$, where $\mcl{S}$ is the symmetry group (i.e.,\ the set of $G\in \GL_2(\mbb{Z}/\bar{d}\mbb{Z})$ such that $\ghostOverlap_{G\mbf{p}} = \ghostOverlap_{\mbf{p}}$ for all $\mbf{p}$) and $\mcl{M}$ is a maximal abelian subgroup of $\GL_2(\mbb{Z}/\bar{d}\mbb{Z})$ containing $\mcl{S}$. 
If $h\in \Gal(\GE/H)$ and $F_h\mcl{S}$ is the corresponding element of $\GL_2(\mbb{Z}/\bar{d}\mbb{Z})$, then
\eag{
h(\ghostOverlap_{\mbf{p}})& = \ghostOverlap_{F_h\mbf{p}}.
}
for all $\mbf{p}$. 
\end{empirical}
\begin{rmkb}
This isomorphism was originally noted empirically by studying the known examples of $1$-SICs~\cite{Appleby:2013,Appleby:2018}. 
For a type $z$ orbit, there is only one maximal abelian subgroup containing $\mcl{S}$ (namely, the centralizer of $\mcl{S}$). 
The characterization of $\mcl{M}$ in the case of a type $a$ orbits will appear in future work.
\end{rmkb}

The key to our method is that, using this isomorphism, one can calculate the action of $\Gal(\GE/H)$ on the \ghostOverlapsText{} without knowing them exactly.

Let $\ghostOverlap_{\mbf{p}_1},\dots, \ghostOverlap_{\mbf{p}_{n}}$ be a maximal orbit of \ghostOverlapsText{} under the action of $\Gal(\GE/H)$.
Suppose that each element of the maximal orbit generates the full field $\GE$, and therefore is not stabilized by a non-identity element of $\Gal(\GE/H)$. 
In the case of a maximal order, this provably follows from the Stark Conjectures (see Subsection~\ref{ssc:StarkConjectures}); for non-maximal orders we heuristically assume it holds. 
Choose
 $L_1$, \dots, $L_m\in \mcl{M}$ such that
 \begin{enumerate}
     \item each  $L_j \mcl{S}$ is order $n_j = \prm_j^{r_j}$ in $\mcl{M}/\mcl{S}$ with $\prm_j$ prime and $r_j$ a positive integer, and
     \item $\mcl{M}/\mcl{S}$ is isomorphic to the direct product $\left< L_1\mcl{S}\right> \times \dots \times \left< L_m\mcl{S}\right>$.
 \end{enumerate}
So $n_1 \dots n_m = n$, where $n$ is  the order of $\Gal(\GE/H)$. 
Let $\Gg_j$ be the element of $\Gal(\GE/H)$ such that 
\eag{
\Gg_j(\ghostOverlap_{\mbf{p}}) = \ghostOverlap_{L_j\mbf{p}}
\label{eq:gjAction}
}
for all $\mbf{p}$, and let $\GE_j=\{c\in \GE\colon \Gg_k(c) = c \text{ for all } k\neq j\}$.  
Finally choose some fixed element of the orbit, say $\ghostOverlap_{\mbf{p}_1}$, and define
\eag{
\ghostOverlap_{s_1,\dots, s_m} &= \ghostOverlap_{L_1^{s_1}\dots L_m^{s_m}\mbf{p}_1}.
\label{eq:nus1smdf}
}
Then the map $(s_1,\dots, s_m) \to \ghostOverlap_{s_1,\dots, s_m}$ is a bijective correspondence between $\mbb{Z}/n_1\mbb{Z} \times \dots \times \mbb{Z}/n_m\mbb{Z}$ and the orbit. 
It follows that if we define sets
\eag{
\GK_{j,t} &= \{\ghostOverlap_{s_1, \dots, s_m}\colon s_j = t\}
}
then their intersection is the single-element set
\eag{
\bigcap_{j=1}^{m}\GK_{j,t_j} &= \{\ghostOverlap_{t_1,\dots, t_m}\}.
\label{eq:Kintersect}
}
Let
\eag{
\GPly_{j,t}(x) &= \prod_{\ghostOverlap\in \GK_{j,t}} (x-\ghostOverlap) = \sum_{l=0}^{n/n_j} \Gc_{j,t,l}x^l.
\label{eq:PjtDefinition}
}
Then the coefficients $\Gc_{j,t,l}$ are all in $\GE_j$.%
\footnote{In what follows we define $\Gc_{j,t,l}$ to be (essentially) the elementary symmetric polynomials of the \ghostOverlapsText{}. 
However, any basis of symmetric polynomials would suffice, and in our numerical calculations we have used power sums instead. 
Using power sums offers some computational advantages, but adds extra steps of changing bases to an already lengthy procedure, so we omit these details.} 
So one approach to the problem of calculating the \SICoverlapsText{} would be to calculate exact versions of the coefficients $\Gc_{j,t,l}$ using an integer-relation algorithm, transform them using a $\sqrt{\Delta_0}$--sign switching automorphism, find the roots of the transformed polynomial, and then find the \SICoverlapsText{} using the transformed version of ~\eqref{eq:Kintersect}. 
Provided $m>1$, this method would be more efficient than using an integer relation algorithm to calculate exact expressions for the $\ghostOverlap_{\mbf{p}}$ since $\GE_j$ is then lower degree than $\GE$. 
However, we can do better than that, by defining a set of numbers which are all in $H$, instead of $\GE_j$, and which are therefore easier to calculate from their numerical counterparts using an integer relation algorithm. 
We refer to these numbers as \emph{ghost invariants}.

Our method relies on the following fact.
\begin{lem}\label{lm:l0jexist}
For each $j$ there exists at least one index $\ell$ such that the  numbers $\{\Gc_{j,t,\ell}\colon t=0,1, \dots, n_j-1\}$  are distinct non-zero. 
If $\ell$ is such an index, then $\GE_j = H(\Gc_{j,0,\ell})$.
\end{lem}
\begin{proof}
Suppose there were no such index $\ell$. 
Since $\Gal(\GE_j/H) = \la \Gg_j\ra$ is cyclic order $n_j$, where $n_j$ is a power of a prime, and since $\Gc_{j,t,\ell}=\Gg_j^t(\Gc_{j,0,\ell})$, it would follow that for each $\ell$ there existed a positive integer $r(\ell)<r_j $ such that
\eag{
\Gg_j^{q_j^{r(\ell)}}(\Gc_{j,t,\ell}) = \Gc_{j,t,\ell}
}
for all $t$.  Defining $r = \max_\ell r(\ell)$, this would mean that
\eag{
\Gg_j^{q_j^{r}}(\Gc_{j,t,\ell}) = \Gc_{j,t,\ell}
}
for all $t$, $\ell$,
This in turn would mean $\GK_{j,q_j^r}=\GK_{j,0}$, contradicting the fact that the sets $\GK_{j,0\vphantom{q_j^{r_j}}}, \ldots, \GK_{j,q_j^{r_j}-1}$ are disjoint.

To prove the second statement, suppose $H(\Gc_{j,0,\ell})$ were a proper subfield of $\GE_j$. 
Then it would be fixed by $\Gg_j^{q_j^{r}}$, for some positive integer $r<r_j$. 
But that would mean $\Gc_{j,q_j^r,\ell} = \Gc_{j,0\vphantom{q_j^r},\ell}$, contradicting the fact that the  coefficients are distinct.
\end{proof}
In view of Lemma~\ref{lm:l0jexist}, we can choose for  each $j$ an index  $\ell_{j}$ such that $\GE_j=H(\Gc_{j,0,\ell_{j}})$. 
We then have, for each $\ell$, a set of numbers $\Ga_{j,r,\ell}\in H$ such that
\eag{
\Gc_{j,0,\ell}&= \sum_{r=0}^{n_j-1} \Ga_{j,r,\ell} \Gc_{j,0,\ell_{j}}^r.
}
Repeatedly applying $\Gg_j$ to both sides, it follows that
\eag{
\bmt \Gc_{j,0,\ell} \\ \vdots \\ \Gc_{j,n_j-1,\ell}\emt
&=
\GV_j \bmt \Ga_{j,0,\ell} \\ \vdots \\ \Ga_{j,n_j-1,\ell}\emt
\label{eq:cvaeqns}
}
where $\GV_j$ is the Vandermonde matrix
\eag{
\GV_j &= \bmt 1 &  \Gc_{j,0,\ell_{j}} & \dots & \Gc^{n_j-1}_{j,0,\ell_{j}}
\\
\vdots & \vdots & & \vdots
\\
1 &  \Gc_{j,n_j-1,\ell_{j}} & \dots & \Gc^{n_j-1}_{j,n_j-1,\ell_{j}}
\emt.
}
In the case $\ell=\ell_{j}$ we also have
\eag{
\Gc_{j,1,\ell_{j}}&= \sum_{r=0}^{n_j-1} \Gb_{j,r} \Gc_{j,0,\ell_{j}}^r
}
for some $\Gb_{j,r}\in H$, from which it follows that
\eag{
\bmt \Gc_{j,1,\ell_{j}} \\ \vdots \\ \Gc_{j,n_j-1,\ell_{j}} \\ \Gc_{j,0,\ell_{j}}\emt
&=
\GV_j \bmt \Gb_{j,0} \\ \vdots \\ \Gb_{j,n_j-1}\emt.
\label{eq:cvbeqns}
}
Equation ~\eqref{eq:cvbeqns} implies
\eag{
\Gg_j\big( \Gc_{j,t,\ell_{j}} \big) = \GQ_j\big(\Gc_{j,t,\ell_{j}}\big)
}
for all $t$, where 
\eag{
\GQ_j (x) &= \sum_{u=0}^{n_j-1} \Gb_{j,u}x^u
\label{eq:galply}
}
Following the terminology of ~\cite{Appleby:2022}, we refer to the $\GQ_j(x)$ as \emph{Galois polynomials}.

Numerical approximations to the numbers $\Ga_{j,t,\ell}$, $\Gb_{j,t}$ can be obtained from ~\eqref{eq:cvaeqns} and ~\eqref{eq:cvbeqns} by solving the linear system. 
Stable solution of Vandermonde linear systems can be done with only $O(n_j^2)$ operations~\cite{Bjorck1970, Higham1987}. 

Finally, define numbers $\Ge_{j,t}$ by
\eag{
\Ge_{j,0}+\Ge_{j,1}x + \dots + \Ge_{j,n_j-1} x^{n_j-1} +x^{n_j}
=\prod_{t=0}^{n_j-1}(x-\Gc_{j,t,\ell_{j}})
}
The numbers $\Ga_{j,t,\ell}$, $\Gb_{j,t}$, $\Ge_{j,t}$ are the ghost invariants we introduced earlier. 
They are all in the ring class field $H$, and exact versions can be calculated using an integer relation algorithm starting from high precision numerical approximations. 
It will be seen that there are $n_1 + \dots + n_m$ invariants $\Ge_{j,t}$, $n_1 + \dots + n_m$ invariants $\Gb_{j,t}$, and $nm-(n_1 + \dots + n_m)$ invariants $\Ga_{j,t,\ell}$ with $\ell\neq \ell_j$, or $nm + n_1 + \dots n_m$ invariants in total. By contrast there are only $n$ ghost overlaps $\ghostOverlap_{\mbf{p}_1},\dots, \ghostOverlap_{\mbf{p}_{n}}$ on the orbit.  However, the numbers $\ghostOverlap_{\mbf{p}_k}$ are in the much larger field $\GE$, and exact values are correspondingly harder to calculate.  Note that we lose some information in going from the ghost overlaps to the ghost invariants.  However, from a knowledge of the Galois invariants it is possible to narrow down the corresponding ghost overlaps to a manageable set of alternatives, from which the correct set can be extracted by a process of trial and error.  The procedure is described in the next section, when we describe the process of constructing a SIC from the appropriate Galois conjugates of the invariants.

\subsection{Constructing the SIC overlaps} 
\label{ssc:reconstructSICoverlaps}
Now let $\SE$ be the field generated by the \SICoverlapsText{}, and let $\gsw$ be any automorphism which switches the sign of $\sqrt{\Delta_0}$. 
For the sake of simplicity, assume $\gsw(\rtus_d) = \rtus_d$ (although it is straightforward to construct a modified version of the argument which works when this condition is not satisfied). 
Under our conjectures and assumptions (see also \cite{Appleby:2022b}), the automorphism $g$ maps $\GE$ onto $\SE$, and the map $h\mapsto \gsw h \gsw^{-1}$ is an isomorphism of $\Gal(\GE/H)$ onto $\Gal(\SE/H)$. 
Define
\eag{
\SP &= g(\GP), & 
\overlap_{\mbf{p}} &= \gsw(\ghostOverlap_{\mbf{p}}), &
\overlap_{s_1,\dots, s_m} &= \gsw(\ghostOverlap_{s_1,\dots, s_m}), & 
\Sc_{j,t,l}&=\gsw(\Gc_{j,t,l}), & 
\Sg_j&= \gsw \Gg_j \gsw^{-1},
}
where $\GP$ is the ghost projector with which we started. Then $\SP$ is a $1$-SIC projector, and $\overlap_{\mbf{p}} =  \Tr(\SP D^{\dagger}_{\mbf{p}})$ are the corresponding \SICoverlapsText{}.  Furthermore, it follows from ~\eqref{eq:gjAction}, \eqref{eq:nus1smdf} that
\eag{
\Sg_j(\overlap_{\mbf{p}})&= \overlap_{L_j\mbf{p}}
}
for all $j$, $\mbf{p}$, and
\eag{
\overlap_{s_1,\dots , s_m} &= \Sg_1^{s_1}\dots \Sg_m^{s_m}(\overlap_{\mbf{p}_1}) = \overlap_{L_1^{s_1} \dots L_m^{s_m}\mbf{p}_1}
\label{eq:sops1tosmDefinition}
}
for all $s_1, \dots, s_m$.
Also define $\SK_{j,t}$ to be the set of roots of the equation
\eag{
\sum_{\ell=0}^{n/n_j} \Sc_{j,t,\ell} x^\ell = 0.
}
Then it follows from ~\eqref{eq:Kintersect} and \eqref{eq:PjtDefinition} that
\eag{
\bigcap_{j=1}^{m}\SK_{j,t_j} &= \{\overlap_{t_1,\dots, t_m}\}.
\label{eq:intersectionKtildejtj}
}
Of course, the fact that we only have numerical approximations for $\GP$, $\ghostOverlap_{\mbf{p}}$, and $c_{j,t,\ell}$ means that we cannot calculate $\SP$, $\overlap_{\mbf{p}}$, or $\Sc_{j,t,\ell}$ directly. 
However, we do have exact expressions for the ghost invariants.  Moreover, the fact that the ghost invariants are all in $H$ means   it is straightforward to calculate the corresponding \emph{SIC invariants}
\eag{
\Sa_{j,t,\ell}& = \gsw(\Ga_{j,t,\ell}), & \Sb_{j,t}&=\gsw(\Gb_{j,t}), & \Se_{j,t} &= \gsw(\Ge_{j,t}).
}
The fact that one loses some information in going from the ghost fiducial to the ghost invariants  means the SIC invariants do not specify $\SP$ unambiguously. 
However, they do contain enough information to enable us to construct a set of candidate operators, from which a set of $1$-SIC fiducial projectors which includes $\SP$ can be extracted without too much difficulty, using the method we now describe. 

Let
\eag{
\SQ_j (x) &= \sum_{u=0}^{n_j-1} \Sb_{j,u}x^u.
\label{eq:tgalply}
}
Then
\eag{
\Sg_j\big(\Sc_{j,t,\ell_{j}}\big) &= \SQ_j\big(\Sc_{j,t,\ell_{j}}\big).
}
We now find the  $\Sc_{j,t,\ell_j}$ using the fact that they are the roots of the equation 
\eag{
\sum_{t=0}^{n_j-1} \Se_{j,t}x^t &=0.
}
The problem is, of course,  that we do not \emph{ab initio} know  which particular root is equal to which particular coefficient $\Sc_{j,t,\ell_j}$.  To deal with this problem, we choose a root at random and label it  $\Scp_{j,0,\ell_{j}}$.  We then define 
 $\Scp_{j,t,\ell_{j}}$ recursively by
\eag{
\Scp_{j,t+1,\ell_{j}}&= \SQ_j\big(\Scp_{j,t,\ell_{j}}\big)
}
so that
\eag{
\Sc_{j,t,\ell_{j}}&=\Scp_{j,t+r_j,\ell_{j}}
}
for all $t$, $j$ and  some unknown $j$-dependent constant $r_j$.   We then extend  the definition of $\Scp_{j,t,\ell}$ to arbitrary values of $\ell$ by setting
\eag{
\Scp_{j,t,\ell}&=\Sc_{j,t-r_j,\ell}
}
for all $j, t, \ell$. We then have
\eag{
\bmt \Scp_{j,0,\ell} \\ \vdots \\ \Scp_{j,n_j-1,\ell}\emt
&=
\SVp_j \bmt \Sa_{j,0,\ell} \\ \vdots \\ \Sa_{j,n_j-1,\ell}\emt
\label{eq:ctildeprimejtlTermsatildejtl}
}
for all $j, \ell$, where 
\eag{
\SVp_j&=\bmt 1 &  \Scp_{j,0,\ell_{j}} & \dots & \Scp^{n_j-1}_{j,0,\ell_{j}}
\\
\vdots & \vdots & & \vdots
\\
1 &  \Scp_{j,n_j-1,\ell_{j}} & \dots & \Scp^{n_j-1}_{j,n_j-1,\ell_{j}}
\emt
}
\Cref{eq:ctildeprimejtlTermsatildejtl} together with our knowledge of the $\Scp_{j,t,\ell_j}$ and $\Sa_{j,t,\ell}$ enables us to calculate  $\Scp_{j,t,\ell}$ for all triples $(j,t,\ell)$.

Now let  $\SKp_{j,t}$ be the roots of the equation
\eag{
\sum_{\ell=0}^{n/n_j} \Scp_{j,t,\ell} x^\ell = 0.
}
Then $\SK_{j,t}=\SKp_{j,t+r_j}$. In view of \eqref{eq:intersectionKtildejtj}, this means that if we define $\overlapPrime_{t_1,\dots,t_m}$ to be the unique member of the set
\eag{
\bigcap_{j=1}^m\SKp_{j,t_j},
}
then $\overlap_{t_1,\dots,t_m} = \overlapPrime_{t_1+r_1,\dots, t_m + r_m}$ for all $t_1, \dots, t_m$.  

Using the values of the $\Scp_{j,t,\ell}$, we can calculate the values of the quantities $\overlapPrime_{t_1,\dots, t_m}$, which in turn fixes the quantities $\overlap_{t_1,\dots,t_m}$ up  to the unknown shifts $r_1, \dots, r_m$.
It is convenient to rephrase this slightly. Given an arbitrary element  $\mbf{p}$   of the $\mcl{M}/\mcl{S}$-orbit of $\mbf{p}_1$, choose $t_1, \dots, t_m$ and $M\in\mcl{M}$ such that $\mbf{p} = L_1^{t_1} \dots L_m^{t_m} M \mbf{p}_1$ and define
\eag{
\overlapPrime_{\mbf{p}} &= \overlapPrime_{t_1, \dots, t_m}.
}
In view of \eqref{eq:sops1tosmDefinition} we then have
\eag{
\overlap_{\mbf{p}}&= \overlap_{L_1^{t_1} \dots L_m^{t_m}M \mbf{p}_1} = \overlap_{L_1^{t_1} \dots L_m^{t_m} \mbf{p}_1} = \overlap_{t_1, \dots, t_m} = \overlapPrime_{t_1+r_1, \ldots, t_m +r_m} = \overlapPrime_{L\mbf{p}}
}
where 
\eag{
L=L^{r_1}\dots L^{r_m}
}
is an  element of $\mcl{M}/\mcl{S}$, which we determine by trial-and-error.  

There are two ways in which we can reduce the size of the search space. In the first place, if we only want \emph{some} $1$-SIC on the same $\EC(d)$ orbit as $\Pi$, not necessarily $\Pi$ itself, we can use the fact that if $L'L^{-1} \in \ESL_2(\mbb{Z}/\bar{d}\mbb{Z})$ then $\{\overlapPrime_{L'\mbf{p}}\}$ is a set of \SICoverlapsText{} if and only if $\{\overlapPrime_{L\mbf{p}}\}$ is.  It is consequently only necessary to test one element from each coset  $\mcl{M}/\bigl(\mcl{M}\cap\ESL_2(\mbb{Z}/\bar{d}\mbb{Z})\bigr)$. 

The search space can be further reduced using the following condition. Let $\mcl{P}$ be the set $\{M\mbf{p}_1\colon  M\in \mcl{M}\}$, and define
\eag{
B &=\frac{1}{d}\sum_{\mbf{p}\in\mcl{P}}\overlap_{\mbf{p}} D_{\mbf{p}}.
}
The fact that $\mcl{M}$ is maximal abelian means $-I \in \mcl{M}$, which in turn implies $B$ is Hermitian, and $\Tr(B(\Pi-B))=0$.  Let $\lambda_{\rm{max}}$ be  the largest eigenvalue of $B$.  Then
\eag{
\lambda_{\rm{max}} &\ge \Tr(B\Pi) = \Tr(B^2) = \frac{|\mcl{P}|}{d(d+1)}
}
where $|\mcl{P}|$ is the cardinality of $\mcl{P}$.  We may therefore remove from consideration any set of candidate \overlapsText{} which do not satisfy this requirement.

\subsection{Convex optimization}
\label{ssc:convexoptim}
At this point, we have managed to construct an orbit of $\mcl{M}$ acting on $\p \in (\Z/\db\Z)^2$ and for which we know numerical approximations of the associated $\overlap_\p$ up to a shift by the action of an unknown element $M \in \mcl{M}$. 
In general, the group action of $\mcl{M}$ will split the $\p$ into more than one orbit. 
In that case, the $k$th orbit can be learned in the above fashion. 
However, the unknown matrix shift will in general be different for each orbit, so it must be guessed \emph{simultaneously} across all orbits if one wishes to reconstruct all $d^2-1$ nontrivial values of $\overlap_\p$. 
The total number of possibilities will in general grow exponentially in the number of orbits, which is an undesirably large search space. 
It might be that there are efficient ways to search this space, but we are not aware of any. 
We will instead describe an alternative method based on convex optimization that requires only knowing a single (correctly shifted) orbit of sufficient size. 

Our idea is to use \emph{low-rank matrix recovery}, which is a matrix analog of the better-known method of compressive sensing. 
For simplicity, consider a square $d\times d$ matrix $X \in \mathcal{L}(\mathbb{C}^d)$. 
In the low-rank matrix recovery problem, one is given a collection of \emph{measurements} $A_k \in \mathcal{L}(\mathbb{C}^d)$ and one learns the value of the \emph{observations} $b_k$, which are linear functions $b_k = \Tr(A_k X)$ obtained by acting on an unknown matrix $X$. 
Written in a vector notation, we have $\boldsymbol{b} = \boldsymbol{A}(X)$ where the measurements are now a linear map $\boldsymbol{A}: \mathcal{L}(\mathbb{C}^d) \to \mathbb{C}^m$ and the observations are $\boldsymbol{b}\in \mathbb{C}^m$. 
One wishes to recover $X$ from knowledge of measurements $\boldsymbol{A}$ and observations $\boldsymbol{b}$. 
If $m = d^2 = \mathrm{rank}(\boldsymbol{A})$, then the solution is trivially $X = \boldsymbol{A}^{-1}\boldsymbol{b}$. 

Suppose that our measurements are expensive to implement, and we wish to recover $X$ from $m \ll d^2$ observations, and hence $\mathrm{rank}(\boldsymbol{A}) \le m$. 
In this case, it is clear that there is no unique solution since $X$ and $X+Y$ are indistinguishable for nonzero $Y$ in the null space of $\boldsymbol{A}$. 

In many cases of interest, one has the additional knowledge that the unknown matrix $X$ has rank $r \ll d$. 
One might hope that this fact would allow efficiently recovering $X$ from far fewer measurements, despite the original problem being ill-posed. 
After all, information-theoretically there are only $O(r d)$ parameters specifying $X$. 
This intuition turns out to be correct: It suffices to make $m = O\bigl(r d (\log d)^c\bigr)$ observations from a typical set of certain randomized measurements~\cite{Gross2010, Gross2011}, where $c$ is an absolute constant. 
The algorithm which efficiently reconstructs $X$ from only $m$ observations is a convex relaxation of the naive algorithm that finds a minimal-rank matrix among those consistent with the observations. 
The convex relaxation (which can be cast as a semidefinite program) returns the answer:
\begin{align}
    X^\star = \arg\min_Y \| Y \|_1 \mbox{ s.t.\ } \boldsymbol{A}(Y) = \boldsymbol{b},
\end{align}
where $\|Y\|_1$ is the Schatten 1-norm, or the sum of the singular values of $Y$. 

The final step in our necromancy procedure is now clear. 
Choose a specific large orbit $\mcl{P}$ and reconstruct the $\overlap_\p$ along that orbit, up to an unknown matrix $M \in \mcl{M}$, using the methods above. 
In our examples, we have always chosen a maximal orbit, and in every case we've observed this orbit was unique and had a size many times larger than $d$. 
If our prior computations and conjectures are correct, then our unknown $1$-SIC fiducial projector $\Pi$ satisfies the constraints
\begin{align}
\label{eq:SDPcons}
    \overlap_{M\p} =  \Tr\bigl(D^\dagger_\p \Pi\bigr), \  \p \in \mcl{P}.
\end{align}
We also know that $\Pi = \Pi^\dagger$ and $\Pi \succeq 0$. 
Since $\Pi$ is positive semidefinite, the Schatten 1-norm becomes simply the trace, so we do not enforce the constraint $\Tr(\Pi) = 1$ to avoid trivializing the objective function. 
We then output the matrix that minimizes $\Tr(\Pi)$ subject to the constraints (\ref{eq:SDPcons}) as well as the Hermitian and positive semidefinite constraints. 
If the result is not (numerically) a rank-1 matrix, we simply try another candidate matrix $M$ and run the semidefinite program again. 
One can gain further confidence in a numerical solution by applying Newton's method, similar to \Cref{subsec:precisionbump}, to enhance the precision to any desired level.

The number of candidate solutions that must be checked is always at most the order of $\abs{\mcl{M}}$, which is at most polynomial in $d$. 
The semidefinite program also runs in polynomial time in $d$. 
Therefore, at least in this formal sense, this procedure is efficient with respect to the dimension. 

\appendix

\section{Alternative fiducial data}
\label{ap:alternativeFiducialData}
In \Cref{dfn:SFKPhase}, we defined the \SFKShort{} phase corresponding to the admissible tuple $t=(d,r,Q)\sim(K,j,m,Q)$ by
\eag{
\SFPhase{t}{\mbf{p}}= (-1)^{s_d(\mbf{p})} e^{-\tfrac{\pi i}{12} \rade(A_t)} \rtu_d^{- \tfrac{f_{jm}}{f} Q(\mbf{p})}
}
with
\eag{
s_d(\mbf{p}) = d+(1+d)(1+p_1)(1+p_2).
}
However, as we will see below, it would have been possible to have replaced $s_d(\mbf{p})$ with the more general expression
\eag{\label{eq:signMoreGeneralExpression}
s_d(\svc,\mbf{p}) &= d+(1+d)(1+p_1)(1+p_2) +(1+d) \la \svc,\mbf{p}\ra
}
for arbitrary $\svc\in \mbb{Z}^2$.  
Similarly, in \Cref{dfn:GhostOverlaps}, we defined the \candidateNormalizedGhostOverlapsText{} corresponding to the admissible tuple $t=(d,r,Q)\sim(K,j,m,Q)$ by
\eag{
    \normalizedGhostOverlapC{t}{\mbf{p}} = 
    \SFPhase{t}{\mbf{p}}
    \sfc{d^{-1} \mbf{p}}{A_t}{\qrt_{Q,+}}.
}
However, as we will see below, it would have been possible to replace $\qrt_{Q,+}$ with $\qrt_{Q,-}$.  The reason we did not  make either of these choices in the main text is because they do not lead to new $r$-SICs, as we now show.

In this appendix, we modify the notation used elsewhere in the paper, so as to include an explicit dependence on $\svc$ and   root $\qrt_{Q,\pm}$.   Let $s=(t,G,g)$ be a fiducial datum containing the admissible tuple $t=(d,r,Q)\sim(K,j,m,Q)$, let $f$ be the conductor of $Q$, let $\qrt$ be either of the two roots of $Q$, and let $\qrt, \svc\in \mbb{Z}^2$.  We define 
\eag{\label{eq:SFPhaseExtended}
\SFPhase{t,\svc}{\mbf{p}}&= (-1)^{s_d(\svc,\mbf{p})} e^{-\tfrac{\pi i}{12} \rade(A_t)} \rtu_d^{- \tfrac{f_{jm}}{f} Q(\mbf{p})}
\\\label{eq:ghostOverlapExtended}
\intertext{where $s_d(\svc,\mbf{p})$ is as given by \eqref{eq:signMoreGeneralExpression},}
\normalizedGhostOverlapC{t,\svc,\qrt}{\mbf{p}}& = 
    \SFPhase{t,\svc}{\mbf{p}}
    \sfc{d^{-1} \mbf{p}}{A_t}{\qrt},
\\\label{eq:ghostProjectorExtended}
\tilde{\Pi}_s(\svc,\qrt) &= \frac{r}{d} I + \frac{1}{d\sqrt{d_j+1}}\sum_{\mbf{p}\notin d\mbb{Z}^2}\normalizedGhostOverlapC{t,\svc,\qrt}{G\mbf{p}} D\vpu{t}_{\mbf{p}}
\\\label{eq:PisDefAlternativeData}
\Pi_s(\svc,\qrt) &= g\left((\tilde{\Pi}_s(\svc,\qrt)\right).
}
The purpose of this appendix is to prove
\begin{thm}\label{thm:AlternativeData} Assume Conjectures~\ref{cnj:tci} and~\ref{conj:RayClassField3} are true.
Let $s=(d,r,Q,G,g)$ be a fiducial datum, let $\qrt$ be either of the two roots of $Q$, and let $\svc$ be any element of $\mbb{Z}^2$.
    Then there exists a form $Q'$ such that
    \eag{
    \Pi_s(\svc,\qrt) &= \Pi_{s'}(\zero,\qrt_{Q',+})
    }
where $s'=(d,r,Q',G,g)$.
\end{thm}
It follows from this result that there is no loss of generality if, as in the rest of the paper, we confine ourselves to the case $\svc=\zero$, $\qrt=\qrt_{Q,+}$.

In order to prove \Cref{thm:AlternativeData} we need first to prove the following lemma.
\begin{lem}\label{lem:UpUMWactonPi} Assume Conjectures~\ref{cnj:tci} and~\ref{conj:RayClassField3} are true.
Let $s=(d,r,Q,G,g)$ be a fiducial datum, and let  $\mbf{w}, \mbf{w}' \in \mbb{Z}^2$.    Then $s'=(d,r,-Q,G,g)$ is also a fiducial datum, and
    \eag{
U\vpu{\dagger}_{P}\Pi_s(\svc,\qrt_{Q,\pm})U^{\dagger}_{P} &= \Pi_{s'}(\svc,\qrt_{-Q,\mp})
\label{eq:PiNegativeQResult}
\\
U\vpu{\dagger}_{M_{G^{-1}\svc'}}\Pi_s(\svc,\qrt_{Q,\pm})U^{\dagger}_{M_{G^{-1}\svc'} }
&= \Pi_s(\svc+\svc',\qrt_{Q,\pm})
\label{eq:ChangesvcInPisExpression}
    }
     where
 $M_{\svc'} =\smt{1 & w'_1 d \\ w'_2 d & 1}$ and $P$ is the parity matrix (see \Cref{dfn:parityMatrix}).
\end{lem}
\begin{proof}
Let $t=(d,r,Q)$, $t'=(d,r,-Q)$. Assuming \Cref{cnj:tci} is true, the fact that $Q$ and $-Q$ have the same discriminant means $\mcl{Z}_t=\mcl{Z}_{t'}$ (see \Cref{dfn:shift}).  Also it follows from \Cref{conj:RayClassField3} that $\sicField_t=\sicField_{t'}$.  So $G$ satisfies \eqref{eq:TwistCondition} for some $\shift\in \mcl{Z}_{t'}$ and $g\in \Gal(\sicField_t/\mbb{Q})$.  It follows that $s'$ is a fiducial datum.

 It is easily seen that $A\vpu{-1}_{t'} = A^{-1}_t$ and $\qrt_{-Q,\pm}= \qrt_{Q,\mp}$.  In view of \Cref{lem:RademacherProperties} this means
\eag{
\SFPhase{t',\svc}{\mbf{p}}&=(-1)^{s_d(\svc,\mbf{p})}e^{-\frac{\pi i}{12}\rade(A^{-1}_{t})}\rtu_d^{\frac{f_{jm}}{f}Q}(\mbf{p})
=\left( \SFPhase{t,\svc}{\mbf{p}} \right)^{-1},
}
for all $\mbf{p}\in \mbb{Z}^2$, while it follows from \Cref{lm:sfam1sfaeq1} that
\eag{
\sfc{d^{-1} \mbf{p}}{A_{t'}}{\qrt_{-Q,\mp}} 
&=
\left(\sfc{d^{-1} \mbf{p}}{A_{t}}{\qrt_{Q,\pm}} \right)^{-1}
}
for all $\mbf{p}\in \mbb{Z}^2$.  Hence
\eag{
\normalizedGhostOverlapC{t',\svc,\qrt_{-Q,\mp}}{\mbf{p}}& =
\SFPhase{t',\svc}{\mbf{p}}
    \sfc{d^{-1} \mbf{p}}{A_{t'}}{\qrt_{-Q,\mp}}
    =
    \left(
    \normalizedGhostOverlapC{t,\svc,\qrt_{Q,\pm}}{\mbf{p}}
    \right)^{-1}.
}
Taking account of \Cref{thm:nupnumpeq1} we deduce
\eag{
\normalizedGhostOverlapC{t',\svc,\qrt_{-Q,\mp}}{\mbf{p}}& =
\normalizedGhostOverlapC{t,\svc,\qrt_{Q,\pm}}{-\mbf{p}}
}
for all $\mbf{p}\in \mbb{Z}^2$ such that $\mbf{p}\notin d\mbb{Z}^2$.  Consequently
\eag{
\tilde{\Pi}_{s'}(\svc,\qrt_{-Q,\mp})&= \frac{r}{d}I + \frac{1}{d\sqrt{d_j+1}}\sum_{\mbf{p}\notin d\mbb{Z}^2}\normalizedGhostOverlapC{t',\svc,\qrt_{-Q,\mp}}{G\mbf{p}}D_{\mbf{p}}
\nn
&=\frac{r}{d}I + \frac{1}{d\sqrt{d_j+1}}\sum_{\mbf{p}\notin d\mbb{Z}^2}\normalizedGhostOverlapC{t,\svc,\qrt_{Q,\pm}}{-G\mbf{p}}D_{\mbf{p}}
\nn
&=
U\vpu{\dagger}_P\tilde{\Pi}_{s}(\svc,\qrt_{Q,\pm})U^{\dagger}_P.
\label{eq:GhostProjectorNegativeQ}
}
It follows from \eqref{eq:ParityMatrixElements} that the matrix elements of $U_P$ are all in $\mbb{Z}$.  So ~\eqref{eq:GhostProjectorNegativeQ} and 
\eqref{eq:PisDefAlternativeData}  imply
\eag{
\Pi_{s'}(\svc,\qrt_{-Q,\mp})&=
g\left(  \tilde{\Pi}_{s'}(\svc,\qrt_{-Q,\mp})\right)
= U\vpu{\dagger}_P 
\Pi_{s}(\svc,\qrt_{Q,\pm})U^{\dagger}_P,
}
thereby proving \eqref{eq:PiNegativeQResult}.  Turning to  \eqref{eq:ChangesvcInPisExpression}, observe that the statement is trivial if $d$ is odd.  Suppose, on the other hand, that $d$ is even.  Then it follows from ~\eqref{eq:ghostProjectorExtended}  that
\eag{
U\vpu{\dagger}_{M_{G^{-1}\svc'}}\tilde{\Pi}_s(\svc,\qrt_{Q,\pm})U^{\dagger}_{M_{G^{-1}\svc'}}
&=
\frac{r}{d}I +\frac{1}{d\sqrt{d_j+1}}\sum_{\mbf{p}\notin d\mbb{Z}^2}
\normalizedGhostOverlapC{t,\svc,\qrt_{Q,\pm}}{GM^{-1}_{G^{-1}\svc'}\mbf{p}} D_{\mbf{p}}.
}
Observe that $M^{-1}_{G^{-1}\svc'} = M\vpu{-1}_{G^{-1}\svc'}$ as an element of $\SLtwo{\mbb{Z}/2d\mbb{Z}}$.  Since $M_{G^{-1}\svc'} \equiv I \Mod{2}$, it follows that 
\eag{
(-1)^{s_d(\svc,GM^{-1}_{G^{-1}\svc'} \mbf{p})} &= 
(-1)^{s_d(\svc,G\mbf{p})}.
}
It follows from \Cref{cor:fjmfqcmpsoddeven} that
\eag{
\rtu_d^{-\frac{f_{jm}}{f}Q\left(GM^{-1}_{G^{-1}\svc'}\mbf{p}\right)}
&=\rtu_d^{-a(GM_{G^{-1}\svc'}\mbf{p})_1^2 - b (GM_{G^{-1}\svc'}\mbf{p})_1(GM_{G^{-1}\svc'}\mbf{p})_2 - c(GM_{G^{-1}\svc'}\mbf{p})_2^2}
}
where $a$, $b$, $c$ are  odd. Setting $G=\smt{\ma & \mb \\ \mc & \md}$, we have
\eag{
\rtu_d^{-a(GM_{G^{-1}\svc'}\mbf{p})_1^2}
&= 
\rtu_d^{-a\left((G\mbf{p})_1+d(\mb (G^{-1}\svc')_2 p_1+\ma (G^{-1}\svc')_1p_2)\right)^2} = \rtu_d^{-a(G\mbf{p})_1^2},
\\
\rtu_d^{-c(GM_{G^{-1}\svc'}\mbf{p})_2^2}
&= 
\rtu_d^{-c\left((G\mbf{p})_2 + d(\md (G^{-1}\svc')_2p_1+\mc (G^{-1}\svc')_1 p_2)\right)^2}= \rtu_d^{-c(G\mbf{p})_2^2},
}
and
\eag{
&\rtu_d^{- b (GM_{G^{-1}\svc'}\mbf{p})_1(GM_{G^{-1}\svc'}\mbf{p})_2}
\nn
&\hspace{0.5 cm}=
\rtu_d^{
-b\left((G\mbf{p})_1+d(\mb (G^{-1}\svc')_2 p_1+\ma (G^{-1}\svc')_1p_2)\right)\left((G^{-1}\mbf{p})_2 + d(\md (G^{-1}\svc')_2p_1+\mc (G^{-1}\svc')_1 p_2)\right)}
\nn
&\hspace{0.5 cm}=
(-1)^{(G\mbf{p})_1(\md (G^{-1}\svc')_2 p_1+\mc (G^{-1}\svc')_1 p_2)+(G\mbf{p})_2(\mb (G^{-1}\svc')_2 p_1 +\ma (G^{-1}\svc')_1 p_2) }
\rtu_d^{
-b(G\mbf{p})_1(G\mbf{p})_2
}
\nn
&\hspace{0.5 cm}=(-1)^{(\ma p_1+\mb p_2)(\md (G^{-1}\svc')_2 p_1+\mc (G^{-1}\svc')_1 p_2)+(\mc p_1 + \md p_2)(\mb (G^{-1}\svc')_2 p_1 +\ma (G^{-1}\svc')_1 p_2) }
\rtu_d^{
-b(G\mbf{p})_1(G\mbf{p})_2
}
\nn
&\hspace{0.5 cm}= (-1)^{(\ma \md + \mb \mc)( (G^{-1}\svc')_2p_1^2+ (G^{-1}\svc')_1p_2^2)}
\rtu_d^{
-b(G\mbf{p})_1(G\mbf{p})_2
}
\nn
&\hspace{0.5 cm}= (-1)^{\Det (G) ((G^{-1}\svc')_2p\vpu{'}_1 + (G^{-1}\svc')_1p\vpu{'}_2)}
\rtu_d^{
-b(G\mbf{p})_1(G\mbf{p})_2
}
\nn
&\hspace{0.5 cm}= (-1)^{\Det (G) \la G^{-1}\svc',\mbf{p}\ra}\rtu_d^{
-b(G\mbf{p})_1(G\mbf{p})_2}
\nn
&\hspace{0.5 cm}= (-1)^{\la \svc',G\mbf{p}\ra}\rtu_d^{
-b(G\mbf{p})_1(G\mbf{p})_2}.
}
Putting all this together, we conclude
\eag{
\rtu_d^{-\frac{f_{jm}}{f}Q(GM^{-1}_{G^{-1}\svc'}\mbf{p})}
&=(-1)^{\la \svc',G\mbf{p}\ra}\rtu_d^{-\frac{f_{jm}}{f}Q(G\mbf{p})},
}
and, consequently,
\eag{
\SFPhase{t,\svc}{GM^{-1}_{G^{-1}\svc'}\mbf{p}}&=(-1)^{s_d\left(\svc,GM^{-1}_{G^{-1}\svc'}\mbf{p}\right)}e^{-\frac{\pi i}{12}\rade(A_t)}\rtu_d^{-\frac{f_{jm}}{f}Q\left(GM^{-1}_{G^{-1}\svc'}\mbf{p}\right)}
\nn
&=(-1)^{s_d(\svc,G\mbf{p})+\la \svc',G\mbf{p}\ra}e^{-\frac{\pi i}{12}\rade(A_t)} \rtu_d^{-\frac{f_{jm}}{f}Q\left(G\mbf{p}\right)}
\nn
&= (-1)^{s_d(\svc+\svc',G\mbf{p})}e^{-\frac{\pi i}{12}\rade(A_t)} \rtu_d^{-\frac{f_{jm}}{f}Q\left(G\mbf{p}\right)}
\nn
&= \SFPhase{t,\svc+\svc'}{G\mbf{p}}.
}
The fact that $M^{-1}_{G^{-1}\svc'} \equiv  I \Mod{d}$ implies, in view of \Cref{lm:shinperiodicity}, that
\eag{
\sfc{d^{-1}GM^{-1}_{G^{-1}\svc'} \mbf{p}}{A_t}{\qrt_{Q,\pm}} &= \sfc{d^{-1}G \mbf{p}}{A_t}{\qrt_{Q,\pm}}
}
for all $\mbf{p}\not\equiv\zero \Mod{d}$.
Hence
\eag{
\normalizedGhostOverlapC{t,\svc,\qrt_{Q,\pm}}{GM^{-1}_{G^{-1}\svc'}\mbf{p}} &=
\SFPhase{t,\svc}{GM^{-1}_{G^{-1}\svc'}\mbf{p}}\sfc{d^{-1}GM^{-1}_{G^{-1}\svc'} \mbf{p}}{A_t}{\qrt_{Q,\pm}}
\nn
&= \SFPhase{t,\svc+\svc'}{G\mbf{p}}\sfc{d^{-1}G \mbf{p}}{A_t}{\qrt_{Q,\pm}}
\nn
&= \normalizedGhostOverlapC{t,\svc+\svc',\qrt_{Q,\pm}}{G\mbf{p}}.
}
Consequently
\eag{
U\vpu{\dagger}_{M_{G^{-1}\svc'}}\tilde{\Pi}_s(\svc,\qrt_{Q,\pm})U^{\dagger}_{M_{G^{-1}\svc'}}
&=\frac{r}{d}I +\frac{1}{d\sqrt{d_j+1}}\sum_{\mbf{p}\notin d\mbb{Z}^2}
\normalizedGhostOverlapC{t,\svc+\svc',\qrt_{Q,\pm}}{G\mbf{p}} D_{\mbf{p}}
\nn
&= \tilde{\Pi}_s(\svc+\svc',\qrt_{Q,\pm}).
}
Write
\eag{
G^{-1}\svc' &= \bmt \ell_1 \\ \ell_2 \emt
}
for $\ell_1, \ell_2\in \mbb{Z}/\db\mbb{Z}$.  It follows from ~\cite[Thm.~1]{Appleby2005} that
\eag{
U_{M_{G^{-1}\svc'}} &= e^{i\theta} X^{-\frac{d\ell_1}{2}}Z^{-\frac{d\ell_2}{2}}
}
for some $\theta\in \mbb{R}$ such that $e^{i\theta}\in \mbb{Q}(\rtu_d)$ (see \Cref{def:WHGroup} for $X$, $Z$).  We have 
\eag{
X^{-\frac{d\ell_1}{2}}Z^{-\frac{d\ell_2}{2}} &= \sum_{j=0}^{d} (-1)^{\ell_2 j}\left| j+\tfrac{\ell_1d}{2}\right>\left< \vphantom{\tfrac{\ell_1d}{2}} j \right|.
}
So the matrix elements  of $X^{-\frac{d\ell_1}{2}}Z^{-\frac{d\ell_2}{2}}$ are all in $\mbb{Z}$.
Hence
\eag{
\Pi_s(\svc+\svc',\qrt_{Q,\pm}) &= 
g\left( \tilde{\Pi}_s(\svc+\svc',\qrt_{Q,\pm})
\right)
\nn
&= g\left( U\vpu{\dagger}_{M_{G^{-1}\svc'}}\tilde{\Pi}_s(\svc,\qrt_{Q,\pm})U^{\dagger}_{M_{G^{-1}\svc'}}
\right)
\nn
&= g\left(\left( X^{-\frac{d\ell_1}{2}}Z^{-\frac{d\ell_2}{2}}\right)\tilde{\Pi}_s(\svc,\qrt_{Q,\pm})\left(X^{-\frac{d\ell_1}{2}}Z^{-\frac{d\ell_2}{2}}\right)^{\dagger}
\right)
\nn
&=\left( X^{-\frac{d\ell_1}{2}}Z^{-\frac{d\ell_2}{2}}\right) g\left(\tilde{\Pi}_s(\svc,\qrt_{Q,\pm})\right)\left(X^{-\frac{d\ell_1}{2}}Z^{-\frac{d\ell_2}{2}}\right)^{\dagger}
\nn
&= U\vpu{\dagger}_{M_{G^{-1}\svc'}}\Pi_s(\svc,\qrt_{Q,\pm})U^{\dagger}_{M_{G^{-1}\svc'}},
}
completing the proof.
\end{proof}
We are now ready to prove the main result of this Appendix.
\begin{proof}[Proof of \Cref{thm:AlternativeData}]
    We first show that, given an arbitrary fiducial datum $s=(d,r,Q,G,g)$ there exists a datum $s'=(d,r,Q',G,g)$  such that 
    \eag{\label{eq:FiducialWithOtherRoot}
     \Pi_s(\zero,\qrt_{Q,-}) &= \Pi_{s'} (\zero, \qrt_{Q',+}).
    }
    Indeed, \Cref{lem:UpUMWactonPi} implies
    \eag{
     \Pi_s(\zero,\qrt_{Q,-}) &=  U^{\dagger}_P \Pi_{s''}(\zero,\qrt_{-Q,+})U\vpu{\dagger}_P
    }
    where $s''$ is the datum $(d,r,-Q,G,g)$. It then follows from Theorems~\ref{thm:GLHomomorphismSurjective} and~\ref{thm:MTransformedFiducials} that there exists $R\in \GLtwo{\mbb{Z}}$ such that
    \eag{
    U^{\dagger}_P \Pi_{s''}(\zero,\qrt_{-Q,+})U\vpu{\dagger}_P &= \Pi_{s''_R}(\zero,\qrt_{-Q_R,+}).
    }
    \Cref{eq:FiducialWithOtherRoot} then follows by combining these statements and setting $Q'=-Q_R$, $s'=s''_R$. 

    Now consider the general case. Given an arbitrary datum $s=(d,r,Q,G,g)$, an arbitrary root $\qrt_{Q,\pm}$, and arbitrary $\svc\in \mbb{Z}^2$, it follows from \Cref{lem:UpUMWactonPi} and the result just proved that
    \eag{
    \Pi_s(\svc,\qrt_{Q,\pm}) &=U\vpu{\dagger}_{M_{G^{-1}\svc}}\Pi_s(\zero,\qrt_{Q,\pm}) U^{\dagger}_{M_{G^{-1}\svc}}
    \nn
    &= U\vpu{\dagger}_{M_{G^{-1}\svc}}\Pi_{s''}(\zero,\qrt_{Q'',+}) U^{\dagger}_{M_{G^{-1}\svc}}
    }
    for some fiducial datum $s''=(d,r,Q'',G,g)$.  It then follows from Theorems~\ref{thm:GLHomomorphismSurjective} and~\ref{thm:MTransformedFiducials} that there exists $R\in \GLtwo{\mbb{Z}}$ such that
    \eag{
    U\vpu{\dagger}_{M_{G^{-1}\svc}}\Pi_{s''}(\zero,\qrt_{Q'',+}) U^{\dagger}_{M_{G^{-1}\svc}}
    &= \Pi_{s''_R}(\zero,\qrt_{Q''_R,+}).
    }
    Setting $s'=s''_R$, $Q'=Q''_R$ the result follows.
\end{proof}

\section{Canonical order \texorpdfstring{$3$}{3} unitaries}\label{ap:canord3}
The purpose of this appendix is to prove \Cref{thm:CanonicalOrder3ConjugacyClasses}, characterizing the conjugacy classes of the elements of $\ESLtwo{\mbb{Z}/\db\mbb{Z}}$ having trace equal to $d-1$. 
Our starting point is Lemma 9.2 in Bos and Waldron~\cite{Bos:2019}, which describes the conjugacy classes of the elements of $\SLtwo{\mbb{Z}/d\mbb{Z}}$ having trace equal to $-1$. 
We proceed in two steps: 
\begin{enumerate}
    \item \label{it:ExtendBosWaldronToESL} We first use Bos and Waldron's result to prove an analogous result for the  elements of $\ESLtwo{\mbb{Z}/d\mbb{Z}}$ having trace equal to $-1$.
    \item We then use this to prove the result for the elements of $\ESLtwo{\mbb{Z}/\db\mbb{Z}}$ having trace equal to $d-1$.
\end{enumerate}

Let $F_z$, $F_a$, $F'_a$ be as specified in \Cref{dfn:FzFaFpaMatrices}, and let $\bar{F}_z$, $\bar{F}_a$, $\bar{F}'_a$ be their reductions modulo $d$.  
Thus
\eag{
\bar{F}_z&= \bmt 0 &  -1 \\ 1 & -1\emt,
&
\bar{F}_a &= \bmt 1 & 3 \\ \frac{d-3}{3} & -2\emt,
&
\bar{F}'_a &= \bmt 1 & 3 \\ \frac{2d-3}{3} & -2\emt.
}

We begin by stating the result of Bos and Waldron on which we rely. 
As in the rest of this paper we restrict ourselves to the case $d\ge 4$.

\begin{thmb}[Lemma 9.2 in Bos and Waldron~\cite{Bos:2019}]
    The set of matrices in $\SLtwo{\mbb{Z}/d\mbb{Z}}$ having trace equal to $-1$ consists of
    \begin{enumerate}[itemsep=0.1 cm]
        \item The single conjugacy class $[\bar{F}_z]$ if $d\not\equiv 0 \Mod{3}$,
        \item The two disjoint conjugacy classes $[\bar{F}\vpu{-1}_z]$, $[\bar{F}^{-1}_z]$ if $d \equiv 0 \Mod{9}$,
        \item The three disjoint conjugacy classes $[\bar{F}\vpu{-1}_z]$, $[\bar{F}^{-1}_z]$, $[\bar{F}\vpu{-1}_a]$ if $d \equiv 3 \Mod{9}$,
        \item The three disjoint conjugacy classes $[\bar{F}\vpu{-1}_z]$, $[\bar{F}^{-1}_z]$, $[\bar{F}'_a]$ if $d \equiv 6 \Mod{9}$,
    \end{enumerate}
    where the notation $[G]$ means ``conjugacy class of $G$ considered as an element of $\SLtwo{\mbb{Z}/d\mbb{Z}}$''.
\end{thmb}
\begin{rmkb}
    Note that Bos and Waldron state a  more general version of the lemma applicable to all $d\ge 2$; in particular to $d=3$, which requires special treatment.
\end{rmkb}
The next result says that extending to $\ESLtwo{\mbb{Z}/d\mbb{Z}}$ reduces the number of conjugacy classes. 

\begin{lem}\label{lm:canord3pre}
    The set of matrices in $\ESLtwo{\mbb{Z}/d\mbb{Z}}$ having determinant equal to $+1$ and  trace equal to $-1$ consists of
    \begin{enumerate}[itemsep=0.1 cm]
        \item 
        \label{it:eslfzconjugacydneq36}The single conjugacy class $[\bar{F}_z]$ if $d\not\equiv 3,6 \Mod{9}$.
        \item 
        \label{it:eslfzconjugacydeq6}The two disjoint conjugacy classes $[\bar{F}\vpu{-1}_z]$,  $[\bar{F}\vpu{-1}_a]$ if $d \equiv 3 \Mod{9}$,
        \item The two disjoint conjugacy classes $[\bar{F}\vpu{-1}_z]$, $[\bar{F}'_a]$ if $d \equiv 6 \Mod{9}$,
    \end{enumerate}
    where the notation $[G]$ means ``conjugacy class of $G$ considered as an element of $\ESLtwo{\mbb{Z}/d\mbb{Z}}$''.
\end{lem}
\begin{proof}
 Let
 \eag{
 M &= \bmt 0 & 1 \\ 1 & 0 \emt 
 }
    Then 
    \eag{
    M\bar{F}\vpu{-1}_z M^{-1} &= \bar{F}^{-1}_z,
    }
   implying $[\bar{F}^{-1}_z] = [\bar{F}\vpu{-1}_z]$ for all $d$.  
   
   To see that $[\bar{F}_z]$ and $[\bar{F}_a]$ are disjoint when $d \equiv 3 \Mod{9}$, assume on the contrary that
   \eag{
    \bar{F}_z & = G \bar{F}_a G^{-1}
   }
   for some $G\in \ESLtwo{\mbb{Z}/d\mbb{Z}}$.  Since $\bar{F}_a \equiv I \Mod{3}$, it would follow that $\bar{F}_z \equiv I \Mod{3}$, which is a contradiction.

   The fact that $[\bar{F}\vpu{'}_z]$ and $[\bar{F}'_a]$ are disjoint when when $d \equiv 6 \Mod{9}$ is proved similarly.
\end{proof}
We  now use this to prove \Cref{thm:CanonicalOrder3ConjugacyClasses}.
If $d$ is odd, then $\bar{F}\vpu{'}_z = F\vpu{'}_z$, $\bar{F}\vpu{'}_a = F\vpu{'}_a$, $\bar{F}'_a = F'_a$, and the theorem is an immediate consequence of \Cref{lm:canord3pre}.  

Suppose, on the other hand, that $d$ is even.  Making the replacement $d \mapsto 2d$ in  \Cref{lm:canord3pre}, and using the fact that $d \equiv 3 \Mod{9} \iff 2d \equiv 6 \Mod{9}$ and $d \equiv 6 \Mod{9} \iff 2d\equiv 3 \Mod{9}$, we find
that the set of matrices in $\ESLtwo{\mbb{Z}/\db\mbb{Z}}$
having determinant equal to  $+1$ and trace equal to $-1$ consists of:
\begin{enumerate}
    \item The single conjugacy class
    \eag{
    \left[\bmt 0 & -1 \\ 1 & -1 \emt\right]
    }
    if $d \not\equiv 3,6 \Mod{9}$;
    \item The two disjoint conjugacy classes 
    \eag{
    &\left[\bmt 0 & -1 \\ 1 & -1 \emt \right],  & \left[\bmt 1 & 3 \\ \frac{4d-3}{3} & -2\emt\right]
    }
    if $d \equiv 3 \Mod{9}$;
    \item The two disjoint conjugacy classes 
    \eag{
    &\left[\bmt 0 & -1 \\ 1 & -1 \emt \right],  & \left[\bmt 1 & 3 \\ \frac{2d-3}{3} & -2\emt\right]
    }
    if $d \equiv 6 \Mod{9}$.
\end{enumerate}
Now suppose that $F\in \ESLtwo{\mbb{Z}/\db\mbb{Z}}$ has determinant $+1$ and trace equal to $d-1$.  Then $G=(d+1)F$ has trace equal to $-1$, so the result just proved implies that 
\eag{
F\in [(d+1)G] &= \left[\bmt 0 & d-1 \\ d+1 & d-1\emt\right]
\\
\intertext{if $d\not\equiv 3,6 \Mod{9}$,}
F\in [(d+1)G] &= \left[\bmt 0 & d-1 \\ d+1 & d-1\emt\right] &&\text{or} & & \left[\bmt d+1 & d+3 \\ \frac{d-3}{3} & -2\emt\right]
\\
\intertext{if $d \equiv 3 \Mod{9}$, and}
F\in [(d+1)G] &= \left[\bmt 0 & d-1 \\ d+1 & d-1\emt\right] &&\text{or} & & \left[\bmt d+1 & d+3 \\ \frac{5d-3}{3} & -2\emt\right]
}
if $d \equiv 6 \Mod{9}$.  Using
\eag{
\bmt 1 & 0 \\ d & 1 \emt \bmt d+1 & d+3 \\ \frac{d-3}{3} & -2\emt
\bmt 1 & 0 \\ d & 1 \emt^{-1} &=
\bmt 1 & d+3 \\ \frac{4d-3}{3} & d-2\emt = 
F_a,
\\
\bmt 1 & 0 \\ d & 1 \emt \bmt d+1 & d+3 \\ \frac{5d-3}{3} & -2\emt
\bmt 1 & 0 \\ d & 1 \emt^{-1} &=
\bmt 1 & d+3 \\ \frac{2d-3}{3} & d-2\emt
=F'_a,
}
we deduce
\eag{
F &\in [F_z]  & & \text{if $d\not\equiv 3,6 \Mod{9}$},
\\
F&\in [F_z] \text{ or } [F_a] && \text{if $d\equiv 3 \Mod{9}$},
\\
F&\in [F_z] \text{ or }  [F'_a] && \text{if $d\equiv 6 \Mod{9}$}.
}
The fact that $[F_z]$, $[F_a]$ are disjoint when $d\equiv 3 \Mod{9}$ follows from the fact that $F_a=I\not\equiv F_z \Mod{3}$.  Similarly, if $d\equiv 6 \Mod{9}$, then $F'_a=I\not\equiv F\vpu{'}_z \Mod{3}$, implying $[F\vpu{'}_z]$, $[F'_a]$ are disjoint.
 \Cref{thm:CanonicalOrder3ConjugacyClasses} now follows.

\section{Hirzebruch--Jung continued fractions} \label{ap:hjcont}
The purpose of this appendix is to show how continued fraction expansions can be used to choose a quadratic form which minimizes the length of the expansion on the right hand side of \eqref{eq:lexpn}
Expansions of the form
\eag{
[k_1,k_2,k_3,k_4,\dots]_{+} &=k_1 +\cfrac{1}{k_2 + \cfrac{1}{k_3+\cfrac{1}{k_4+\cdots}}}
}
are extremely well-known and are described in considerable detail in standard texts such as ~\cite{Hardy2009,Davenport2008}. Following Popescu-Pampu~\cite{Popescu-Pampu:2007} we refer to them as \emph{Euclidean (E-) continued fractions}. 
In this paper we need a different kind of expansion, which Popescu-Pampu refers to as a \emph{Hirzebruch--Jung (HJ-) continued fraction},  of the form
\eag{
[k_1,k_2,k_3,k_4,\dots]_{-} &=k_1 -\cfrac{1}{k_2 - \cfrac{1}{k_3-\cfrac{1}{k_4-\cdots}}}\,.
}
In the literature\cite{Jung:1908,Hirzebruch:1953,Hirzebruch:1971,Hirzebruch:1973,Shintani1976,Adler:1984,Myerson:1987,Finkelshtein:1993,Katok:1996,Katok:2003,Popescu-Pampu:2007,Luzzi:2007,Lopez:2015,Bjorklund:2020} such fractions are also described as backwards, negative-regular, minus, reduced regular, and by-excess continued fractions. 
For the convenience of the reader in this appendix we collect their essential properties.
Since we are only concerned with HJ-continued fractions in this paper, we will drop the subscript, and simply denote them $[k_1,k_2,k_3,k_4,\dots]$. 
We also review the related concept of a \emph{Hirzebruch--Jung (HJ-) reduced form}.

For all $x\in \mbb{Q}$ (respectively $x\in \mbb{R}\setminus\mbb{Q}$), there exists a unique finite (respectively infinite) sequence of integers $k_j$ such that $x=[k_1,k_2\dots ]$, the $k_j \ge 2$ for all $j\ge 2$, and (in case the sequence is infinite) there is no integer $m$ such that $k_j = 2$ for all $j\ge m$.

Define = $(k_1,k_2,\dots, k_n)$ recursively by
\eag{
(k_1,k_2, \dots, k_n) &=\begin{cases}
k_1 \qquad & n=1,
\\
k_1k_2-1\qquad & n=2,
\\
(k_1,k_2,\dots, k_{n-1})k_n -(k_1,k_2\dots, k_{n-2}) \qquad &n>2.
\end{cases}
}
We refer to these quantities as \textit{HJ-convergents}. They can be calculated using the following modified version of Euler's rule: First take the product of all $n$ numbers $k_1,\dots, k_n$, then subtract all products obtained by omitting a pair of adjacent numbers, then add all products  obtained by omitting two different pairs of adjacent numbers, and so on.
In particular, they are symmetric under reversal: 
\eag{
(k_1,k_2,\dots ,k_{n-1}, k_n)& = (k_n, k_{n-1},\dots , k_2, k_1).
}
One has
\eag{
[k_1,k_2, \dots, k_n]&=
\begin{cases}
(k_1) \qquad & n=1,
\\
\frac{(k_1,k_2, \dots, k_n)}{(k_2,\dots, k_n)} \qquad & n\ge 2,
\end{cases}
}
for finite HJ-expansions, 
\eag{
[k_1,k_2, \dots, k_{n}, \dots]&=
\begin{cases}
\frac{(k_1) [k_{n+1}, k_{n+2}, \dots]-1}{[k_{n+1}, k_{n+2}, \dots]} \vphantom{\Big( }\qquad & n=1,
\\ 
    \frac{(k_1, k_2) [k_{n+1}, k_{n+2}, \dots]-(k_1)}{(k_2)[k_{n+1}, k_{n+2}, \dots] - 1}\vphantom{\Big( }\qquad & n=2,
\\
    \frac{(k_1, \dots, k_n) [k_{n+1}, k_{n+2}, \dots]-(k_1, \dots, k_{n-1})}{(k_2,\dots, k_n)[k_{n+1}, k_{n+2}, \dots] - (k_2, \dots, k_{n-1})}\vphantom{\Big( }\qquad & n>2,
\end{cases}
}
for infinite HJ-expansions, and
\eag{
T^{k_1}S T^{k_2}S \cdots T^{k_n}S &=
 \begin{cases}
 \bmt (k_1) & -1 \\ 1 & 0\emt \qquad &n=1,
 \\
 \bmt (k_1,k_2) & -(k_1) \\ (k_2) & -1\emt
 \qquad & n=2,
 \\
 \bmt(k_1,\dots,k_n) & -(k_1,\dots, k_{n-1}) 
 \\ (k_2,\dots, k_{n}) & -(k_2,\dots, k_{n-1})\emt \qquad & n\ge 3.
 \end{cases}
\label{eq:tk1sdtstknsexpn}
}
In particular
\eag{
T^{k_1}S T^{k_2}S \cdots T^{k_n}S.\left([k_{n+1},k_{n+2},\dots]\right) &= [k_1, k_2, \dots]
\label{eq:TSActionOnHJContinuedFraction}
}
for all $n$.
This last relation is the reason HJ-continued fractions are relevant to this paper.

A continued fraction is said to be \emph{periodic} if it is infinite and of the form
\eag{
 [j_1,j_2,\dots,j_m,\overline{k_1,k_2,\dots,k_n}]
 &=
 [j_1,j_2,\dots,j_m,k_1,\dots,k_n,k_1,\dots,k_n,k_1,\dots,k_n,\dots]
}
It is said to be \emph{purely periodic}  if it is of the form $[\overline{k_1,\dots,k_n}]$. If $k_1, \dots, k_n$ doesn't break into two or more identical subsequences, then we say that $n$ is the \emph{period} of $[\overline{k_1,\dots,k_n}]$.

The HJ-continued fraction expansion of a real number is periodic if and only if it is  an irrational element of a real quadratic field.    Let $\qrt$ be such a number. Then its HJ-continued fraction expansion is purely periodic if and only if  $\qrt > 1>\qrt' >0$, where $\qrt'$ is its Galois conjugate.

There is a close connection between purely periodic HJ-continued fractions and a class of quadratic forms that we now define.  A form $Q=\la a, b, c\ra$ with discriminant $\Delta = b^2-4ac$ is reduced in the ordinary sense~\cite{Buchmann:2007,Buell1989}, or (as we will say) \emph{Euclidean (E-) reduced}, if 
\eag{
0 < \sqrt{\Delta}-b < 2|a| < \sqrt{\Delta}+b.
}
Forms of this type have a connection with Euclidean continued fractions. Specifically, the number $\tfrac{b+\sqrt{\Delta}}{2|a|}$ has a purely periodic E-continued fraction expansion if and only if $Q$ is E-reduced.  Correspondingly we say $Q$ is \emph{Hirzebruch--Jung (HJ-) reduced} if
\eag{
0 < -\sqrt{\Delta}-b < 2|a| < \sqrt{\Delta}-b.
}
The number $\tfrac{-b+\sqrt{\Delta}}{2|a|}$ has a purely periodic HJ-continued fraction expansion if and only if $Q$ is HJ-reduced.

Let $Q=\la a,b, c\ra$ and $J=\smt{1 & 0 \\ 0 & -1}$. 
Then $Q_J =\la -a, b, -c\ra$ (see \Cref{ssec:quadraticfieldsforms}). 
It follows that $Q_J$ is E-reduced (respectively HJ-reduced) if and only if $Q$ is E-reduced (respectively HJ-reduced), in which case they define the same purely periodic E-continued (respectively HJ-continued) fraction. We may therefore, without loss of generality, confine  ourselves to E-reduced (respectively HJ-reduced) forms $\la a, b, c\ra$ for which $a>0$. 
In the following this restriction will be assumed without comment. 
The map taking $Q$ to $-\qrt_{Q,-}$ (respectively $\qrt_{Q,+}$) is then a bijective correspondence of the set of E-reduced (respectively HJ-reduced) forms onto the set of purely periodic E-continued (respectively HJ-continued) fractions.

There is a well-known algorithm~\cite{Buchmann:2007,Buell1989} for calculating the complete set of E-reduced forms on a given $\GLtwo{\mbb{Z}}$ orbit. 
We now show how this can be used to construct the complete set of HJ-reduced forms on the same orbit. 
Let $W$ be the forms $\la a, b,c\ra$ on a $\GLtwo{\mbb{Z}}$ orbit for which $a>0$, and let 
\eag{
W_{\rm{E}} &= \{Q\in W\colon \qrt_{Q,-}< -1 < 0< \qrt_{Q,+} <1 \}
\\
W_{\rm{HJ}}&=\{Q\in W \colon 0 < \qrt_{Q,-} < 1 < \qrt_{Q,+}\}
}
(where $\qrt_{Q,\pm}$ are given by \eqref{dfn:qrtdf}).
Then $W_{\rm{E}}$ (respectively $W_{\rm{HJ}}$) is precisely the set of E-reduced (respectively HJ-reduced) forms in $W$. Also define, for $n=0,1,\dots$,
\eag{
W_{\rm{E}}^{(n)} &= \left\{Q\in W\colon \qrt_{Q,-}< -1 < 0< \qrt_{Q,+} <\frac{1}{n+1}\right\},
\\
W_{\rm{HJ}}^{(n)} &= \left\{Q\in W\colon \frac{n}{n+1}<\qrt_{Q,-}<\frac{n+1}{n+2} < 1 < \qrt_{Q,+}\right\}.
}
Then
\eag{
W_{\rm{E}} = W_{\rm{E}}^{(0)} \supseteq W_{\rm{E}}^{(1)} \supseteq \dots,  &&& \bigcap_{n=0}^{\infty} W_{\rm{E}}^{(n)} = \emptyset,
\\
W_{\rm{HJ}}^{(n)}\cap W_{\rm{HJ}}^{(n')} = \emptyset \qquad \text{if $n\neq n'$}, &&& \bigcup_{n=0}^{\infty} W_{\rm{HJ}}^{(n)}  = W_{\rm{HJ}}.
}
Since $W_{\rm{E}}$ is finite, non-empty there must exist $n_0 \in \mbb{N}$ such that $W_{\rm{E}}^{(n)} = \emptyset$ if and only if $n\ge n_0$.
Let
\eag{
L_n&=\bmt n & -1 \\  n+1 & -1\emt. 
}
Then $x<-1$ if and only if $n/(n+1) < L_n \cdot x < (n+1)/(n+2)$,  and $0<x <1/(n+1)$ if and only if $1 < L_n \cdot x$. In view of \Cref{{lm:lactqrt}}, this means the map $Q\to Q_{L^{-1}_n}$ is a bijection of $W_{\rm{E}}^{(n)}$ onto $W_{\rm{HJ}}^{(n)}$.  It follows that the cardinality of $W_{\rm{HJ}}$ is at most $n_0$ times the cardinality of $W_{\rm{E}}$.  It also provides an algorithm for calculating the set $W_{\rm{HJ}}$, given the set $W_{\rm{E}}$.

\begin{lem}\label{eq:convineq}
Suppose $k_j\ge 2$ for all $j$.  Then
\eag{
(k_1,\dots, k_n) > (k_1, \dots, k_{n-1}) > \dots >(k_1,k_2) > (k_1) >1
}
\end{lem}
\begin{proof}
Straightforward consequence of the definition.
\end{proof}
\begin{lem}\label{lm:tksprdprops}
Suppose
\eag{
T^{k_1}S T^{k_2} S \cdots T^{k_n}S & = L^m
\label{eq:tkseqlm}
}
for some sequence of integers $k_1, k_2, \dots, k_n$ all greater than 1 and some positive integer $m$.  Then $n= \ell m$ for some positive integer $\ell$,   $k_{j+\ell}=k_j$ for $j=1,\dots , n-\ell$, and 
\eag{
L &=
\begin{cases}
T^{k_1} S \cdots T^{k_{\ell}}S \qquad & \text{if $m$ is odd},
\\
\pm T^{k_1} S \cdots T^{k_{\ell}}S \qquad & \text{if $m$ is even}.
\end{cases}
\label{eq:ltrmstk1skr}
}
\end{lem}
\begin{proof}
Let $\bar{T}$, $\bar{S}$ be the images of $T$, $S$ under the canonical projection $h\colon \SLtwo{\mbb{Z}} \to \PSLtwo{\mbb{Z}}$, and let $\bar{R} = \bar{T}\bar{S}$.  Then (see for example ~\cite{Alperin:1993}) $\bar{S}$ is order $2$, $\bar{R}$ is order $3$, and every element of $\PSLtwo{\mbb{Z}}$ is either the identity or else has a unique alternating expansion of the form $\bar{M}_1 \cdots \bar{M}_n$ where: (1) each $\bar{M}_j$ is either $\bar{S}$ or $\bar{R}^k$ for $k=1$ or $2$, and (2) terms equal to $\bar{S}$ alternate with terms equal to a non-zero power of $\bar{R}$. Also define $\bar{S}_1=\bar{S}$, $\bar{S}_2=\bar{S}\bar{R}\bar{S}$, $\bar{S}_2=\bar{S}\bar{R}\bar{S}\bar{R}\bar{S}$, \dots.
Now let  $\bar{L} = h(L)$ and let $\bar{L} = \bar{M}_1 \cdots \bar{M}_q$ be its expansion in terms of alternating powers of $\bar{S}$ and $\bar{R}$. It follows from ~\eqref{eq:tkseqlm} that
\eag{
\bar{L}^m = \bar{R}\bar{S}_{k_1-1}\bar{R}^2 \bar{S}_{k_2-1} \bar{R}^2\cdots \bar{R}^2\bar{S}_{k_n-1} \bar{R}.
\label{eq:lmexpn}
}
So $\bar{M}_1 = \bar{M}_q= \bar{R}$, and
\eag{
L&= \bar{R} \bar{S}_{\kappa_1-1} \bar{R}^2 \cdots \bar{S}_{\kappa_{\ell}-1}\bar{R}
}
for some sequence of integers $\kappa_1, \dots, \kappa_{\ell}$ all greater than one.  Equation ~\eqref{eq:lmexpn} and the uniqueness of the alternating expansion then imply $n=m\ell$, $k_j=\kappa_j $ for $j=1, \dots , \ell$, and $k_{j+\ell} = k_j$ for $j=1,\dots, n-\ell$. Equation ~\eqref{eq:ltrmstk1skr} then follows.
\end{proof}

\begin{thm}\label{thm:hjfrmlqexpn}
Let $Q=\la a, b, c\ra$ be a form with $a>0$, let $f$ be its conductor, and let
\eag{
\qrt_{Q,+} &= [j_1,\dots, j_{q},\overline{k_1, \dots, k_{p}}]
}
(where we set $q=0$ if $\qrt_{Q,+}$ has a purely periodic expansion equal to $[\overline{k_1,\dots, k_{p}}]$).
Assume the sequences $j_1, \dots, j_{q}$ and $k_1, \dots, k_{p}$ are as short as possible (i.e., $k_1, \dots , k_{p}$ is not the conjunction of two or more identical subsequences, and $j_{q}\neq k_{p}$). Then
\eag{
\canrep_Q(\vn_f) &= (T^{j_1}S \cdots T^{j_{q}}S )
(T^{k_1}S \cdots T^{k_{p}}S)(T^{j_1}S \cdots T^{j_{q}}S )^{-1}.
}
\end{thm}
\begin{proof}
Assume, to begin with, that $q=0$ and $\qrt_{Q,+}= [\overline{k_1, \dots, k_{p}}]$.  Let
\eag{
M &= T^{k_1}S \cdots T^{k_{p}}S.
}
Then it follows from ~\eqref{eq:tk1sdtstknsexpn} that
\eag{
M\qrt_{Q,+}& = \qrt_{Q,+}.
}
In view of \Cref{lm:lactqrt}, this means
$M \in \mcl{S}(Q)$. 
It then follows from \Cref{tm:symgp} that
\eag{
M &= s_1\canrep_Q(\vn_f^{s_2\lambda}) =
s_1\left(\frac{d_{\lambda j_{\rm{min}}(f)} -1}{2}I + \frac{s_2 f_{\lambda j_{\rm{min}}(f)}}{f}S Q\right),
}
where $\lambda$ is a positive integer, and $s_1, s_2$ are signs. 
Comparing this expression with ~\eqref{eq:tk1sdtstknsexpn},
one sees:
\eag{
s_1(d_{\lambda j_{\rm{min}}(f)}-1) = \Tr(M) &= 
\begin{cases}
(k_1) \qquad & p=1,
\\
(k_1,k_2)-1 \qquad & p=2,
\\
(k_1,\dots, k_{p})-(k_2,\dots, k_{p-1}) \qquad & p\ge 3;
\end{cases}
\\
\frac{s_1s_2f_{\lambda j_{\rm{min}}(f)}}{f} = M_{21}
&=
\begin{cases}
1 \qquad & p=1,
\\
(k_2,\dots, k_{p}) \qquad & p\ge 2.
\end{cases}
}
In view of Lemma~\ref{eq:convineq}, this means $s_1=s_2 = +1$.  So $M = \left(\canrep_Q(\vn_f)\right)^{\lambda}$. 
It then follows from Lemma~\ref{lm:tksprdprops} and the assumption that the sequence $k_1, \dots, k_{p}$ is not the conjunction of two or more identical subsequences that $\lambda=1$ and $M = \canrep_Q(\vn_f)$.

Now suppose $q\ge 1$. 
Let 
\eag{
M &= T^{k_1}S \cdots T^{k_{p}}S,
\\
N&=T^{j_1}S \cdots T^{j_{q}}S,
}
and let $Q'=\la a', b', c'\ra$ be the unique form such that $[\overline{k_1,\dots, k_{p}}] = \qrt_{Q',+}$ and $a'>0$. 
It follows from the result just proved that $M = \canrep_{Q'}(\vn_f)$, while it follows from ~\eqref{eq:TSActionOnHJContinuedFraction} that
\eag{
N \qrt_{Q',+} & = \qrt_{Q,+}.
}
In view of \Cref{{lm:lactqrt}}. this means  $Q = Q'_{N^{-1}}$.  Hence
\eag{
\canrep_Q(\vn_f) &= \frac{d_{r_f} -1}{2}I +\frac{f_{r_f}}{f} S(N^{-1})^{\rm{T}}Q' N^{-1} 
= NM N^{-1}, 
}
completing the proof.
\end{proof}

\begin{thm}\label{thm:tsaltexpn}
Let $L=\smt{\mc_1 & \md_1 \\ \mc_2 & \md_2}$ be any element of $\SLtwo{\mbb{Z}}$.  Then the following statements are equivalent:
\begin{enumerate}
    \item $\mc_2 >0$.
    \item There exists an integer $n\ge 1$ and  sequence of integers $r_1, r_2, \dots, r_{n+1}$ such that $r_i\ge 2$ unless $i=1$ or $n+1$,  and 
    \eag{
    L &= T^{r_1} S T^{r_2} S \cdots S T^{r_{n}} S T^{r_{n+1}}.
    \label{eq:ltsdecom}
    }
\end{enumerate}
If these conditions are satisfied, the integer $n$ and sequence $r_1, \dots , r_{n+1}$ are unique. 

Continue to assume the conditions are satisfied.  Define 
\eag{
L_i &= \bmt \mc_i & \md_i \\ \mc_{i+1} & \md_{i+1} \emt = T^{r_i}S\cdots T^{r_{n+1}}.
\label{eq:limidf}
} 
Then
\eag{
\mc_2 > \mc_3 > \dots  > \mc_{n+1} ,
}
\eag{
\frac{\md_2}{\mc_2}  < \frac{\md_3}{\mc_3} < \dots < \frac{\md_{n+1}}{\mc_{n+1}},
}
\eag{
\DD_{L} = \DD_{L_1} \subset \DD_{L_2} \subset \dots \subset \DD_{L_{n+1}} = \mbb{C}.
}
\end{thm}
\begin{proof}
Aside from uniqueness this is proved in ~\cite{Kopp2020d}.  To prove uniqueness suppose
\eag{
T^{r_1} S T^{r_2} S \cdots S T^{r_{n}} S T^{r_{n+1}}
&= 
T^{r'_1} S T^{r'_2} S \cdots S T^{r'_{n}} S T^{r'_{n'+1}}.
\label{eq:tspweqtsppw}
}
Assume to begin with that $r_i$, $r'_i\ge 2$ for all $i$, including $i=1$, $n+1$.
Let $h\colon \SLtwo{\mbb{Z}}\to \PSLtwo{\mbb{Z}}$ and $\bar{S}$, $\bar{T}$, $\bar{R}$ and $\bar{S}_i$ be as in the proof of Lemma~\ref{lm:tksprdprops}.  Then
\eag{
\bar{R}\bar{S}_{r_1-1} \bar{R}^2 \cdots \bar{R}^2 \bar{S}_{r_{n+1}-1} \bar{R}\bar{S}_1
&= 
\bar{R}\bar{S}_{r'_1-1} \bar{R}^2 \cdots \bar{R}^2 \bar{S}_{r'_{n'+1}-1} \bar{R}\bar{S}_1,
}
which, in view of the uniqueness of the alternating expansion~\cite{Alperin:1993}, means $n=n'$ and $r_i = r'_i$ for all $i$.  In the general case, choose integers $m$, $\ell$ such that $r_1+m$, $r'_1 + m$, $r_{n+1}+\ell$, $r'_{n'+1}+\ell\ge 2$. 
Then multiplying both sides of ~\eqref{eq:tspweqtsppw} by $T^{m}$ on the left and $T^{\ell}$ on the right gives
\eag{
T^{r_1+m} S T^{r_2} S \cdots S T^{r_{n}} S T^{r_{n+1}+\ell}
&= 
T^{r'_1+m} S T^{r'_2} S \cdots S T^{r'_{n}} S T^{r'_{n'+1}+\ell}.
}
Applying the result just proved, the claim follows.
\end{proof}
\begin{defn}[canonical expansion; length]\label{dfn:canonicalExpansion}
We refer to the expression on the right hand side of ~\eqref{eq:ltsdecom} as the \textit{canonical expansion} of $L$, and we refer to the integer $n$ as its \textit{length}.
\end{defn}
\begin{lem}\label{lm:rtsrels}
Let $R=TS$.  Then
\eag{
T &= -RS,
\\
T^{-1} &= -SR^2,
\\
\intertext{and, for every positive integer $m$,}
ST^{-m}S &= -(TST)^m
\label{eq:STnegativepowerSIdentity}
}
\end{lem}
\begin{proof}
Straightforward consequences of the relations $R^3=S^2 = -I$.
\end{proof}
\begin{thm}\label{thm:effexpn}
Let $\mcl{F}$ be a  $\GLtwo{\mbb{Z}}$ orbit of forms,  let $\mcl{F}_{+}$ be the subset consisting of $\la a, b,c\ra\in \mcl{F}$ for which $a>0$, and let $\mcl{F}_{\rm{HJ}}$ be the set of HJ-reduced forms in $\mcl{F}_{+}$. Let $p_{\rm{min}}$ be the minimum value of the period $p$ of  $\qrt_{Q,+}=[\overline{k_1,\dots, k_p}]$ as $Q$ ranges over the  set $\mcl{F}_{\rm{HJ}}$.   
Then for all $Q\in \mcl{F}_{+}$, the length of $\canrep_{Q}(\vn_f)$ is greater than or equal to $p_{\rm{min}}$.  
\end{thm}
\begin{rmkb}
For a given admissible tuple $t$ we can use this result together with
\cref{tm:symgp} to find a matrix $M\in \GLtwo{\mbb{Z}}$ for which $\posGen{t_M}$ has minimum length.
\end{rmkb}
\begin{proof}
Let $Q\in \mcl{F}_{+}$ be arbitrary. 
The statement is immediate if $Q$ is HJ-reduced, so assume not.  
Let  $L=\canrep_Q(\vn_f)$, and  $\qrt_{Q,+} = [j_1, \dots, j_{q},\overline{k_1, \dots, k_{p}}]$ where the sequences  $j_1,\dots , j_{q}$ and $k_1, \dots, k_{p}$ are chosen as short as possible. 
It follows from Theorem~\ref{thm:hjfrmlqexpn} that
\eag{
L &= N M N^{-1}
}
where
\eag{
N&= T^{j_1}S \cdots T^{j_{q}}S 
\\ M&=T^{k_1}S \cdots T^{k_{p}}S
\\
\intertext{and from Theorem~\ref{thm:tsaltexpn} that}
L&= T^{r_1} S \cdots T^{r_n}ST^{r_{n+1}}
}
where  $r_i\ge 2$ for $1 < i \le n$. 
We have
\eag{
L N & = N M,
}
\eag{
LN &=T^{r_1}S \cdots T^{r_n} S T^{m} S T^{j_2} S \cdots T^{j_q}S
\label{eq:lneqnm}
\\
\intertext{where $m=r_{n+1} + j_1$, and}
NM&= T^{j_1}S \cdots T^{j_q}S T^{k_1}S \cdots T^{k_p}S
\label{eq:nmeqnm}
}
The expression on the right hand side of ~\ref{eq:nmeqnm} is the canonical expansion of $NM$ with length equal to $\length(N)+\length(M)$. 
If $m\ge 2$, then the expression on the right hand side of ~\eqref{eq:lneqnm} is the canonical expansion of $LN=NM$, which, in view of Theorem~\ref{thm:tsaltexpn}, means $\length(L)+\length(N) = \length(N)+\length(M)$, implying $\length(L) = \length(M) \ge p_{\rm{min}}$.
Suppose $m<2$.  
There are three cases to consider:
 
\spc
 
\noindent \textbf{Case 1.} $m=1 $. Then 
\eag{
LN
 &=
 T^{r_1}S \cdots T^{r_n} S T S T^{j_2} S \cdots T^{j_q}S.
}
Using ~\eqref{eq:STnegativepowerSIdentity} this becomes
\eag{
LN &= 
T^{r_1}S \cdots T^{r_n-1} S T^{j_2-1} S \cdots T^{j_q}S.
}
One goes on in this way, making repeated applications of ~\eqref{eq:STnegativepowerSIdentity}, until one obtains an expansion in canonical form.  
Each application of ~\eqref{eq:STnegativepowerSIdentity} reduces the number of $S$ operators, so the length of the resulting expansion will be less than $n+q$. 
It follows that $\length(L)+\length(N)> \length(LN) =\length(NM) = \length(N)+\length(M)$, implying  $\length(L)> \length(M) \ge p_{\rm{min}}$.
 
\spc

\noindent \textbf{Case 2.} $m=0 $. Then  
\eag{
LN
 &=- T^{r_1}S \cdots  T^{r_{n-1}}S T^{r_n+j_2}  ST^{j_3}S \cdots T^{j_q}S.
 \label{eq:LNExpansion}
}
 Writing $LN=NM = \smt{\ma & \mb \\ \mc & \md}$, Theorem~\ref{thm:tsaltexpn} applied to the expansion on the right hand side of \eqref{eq:nmeqnm}  implies $\mc>0$, while the same theorem applied to the expansion on the right hand side of \eqref{eq:LNExpansion} implies $\mc<0$.  It follows that this case is not possible.
 
\spc
 
 \noindent \textbf{Case 3.} $m<0$.  It follows from Lemma~\ref{lm:rtsrels} that 
\eag{
LN
 &=- T^{r_1}S \cdots T^{r_n+1} S (T^{2}
 S)^{|m|-1} T^{j_2+1} S \cdots T^{j_q}S,
}
which is not possible for the same reason that Case 2 is not possible.
\end{proof}

\section{Shintani--Faddeev Jacobi cocycle} 
\label{sec:SFJCocycleAppendix}

The domain of the Shintani--Faddeev Jacobi cocycle given in \Cref{df:shinfadjacocycle} is a slight extension of that given in ~\cite{Kopp2020d}.  The purpose of this Appendix is, firstly, to explain that extension; secondly, to explain how $\sigma_{M}(z,\tau)$ is defined when $\tau\notin\mbb{H}$; and thirdly, to prove the cocycle condition ~\eqref{eq:sfjcocyclerelInt}.

Let $M=\smt{\ma & \mb \\ \mc & \md }\in \SLtwo{\mbb{Z}}$.  In ~\cite{Kopp2020d} the domain of $\sfjM{M}{z}{\tau}$ is denoted $\mbb{C}\times \mathbb{D}_M$.  If $\mc \neq 0$, or if $\mc = 0$ and $\md>0$, then $\mathbb{D}_M = \DD_M$.  If, on the other hand, $\mc=0$ and $\md<0$, then $\mathbb{D}_M = \mbb{H}$, whereas $\DD_M = \mbb{C}\setminus\mbb{R} = \HH \cup (-\HH)$.
This discrepancy is due to the fact that, in ~\cite{Kopp2020d}, the Shintani--Faddeev Jacobi cocycle was defined by analytic continuation in the $\tau$ variable from the upper half-plane, so it is defined for $\tau$ in a connected subset of $\C$. 

We may define $\sfjM{M}{z}{\tau}$ on $\mbb{C}\times \DD_M$ by first extending the definition of the infinite variant $q$-Pochhammer symbol (\Cref{dfn:variantqPochhammer}) from $\tau \in \HH$ to $\tau \HH \cup (-\HH)$. As it turns out, sending $n \to -\infty$ in the definition of the finite $q$-Pochhammer symbol \Cref{eq:finiteqpoch} naturally produces a convergent series on the lower half-plane, just as sending $n \to \infty$ produces a convergent series on the upper half-plane.

\begin{defn}\label{defn:pochupperlower}
    Let $z \in \C$ and $\tau \in \DD_{-I} = \HH \cup (-\HH)$. Set
    \begin{equation}
    \varpi(z,\tau) = 
    \begin{cases}
        \ds\prod_{j=0}^\infty \left(1 - e^{2\pi i (z+j\tau)}\right) & \mbox{ if } \tau \in \HH, \\
        \ds\prod_{j=1}^\infty \left(1 - e^{2\pi i (z-j\tau)}\right)^{-1} & \mbox{ if } \tau \in -\HH.
    \end{cases}
    \end{equation}
\end{defn}

We now define the Shintani--Faddeev Jacobi cocycle in two cases, where the definition in the first case is the same as in \cite[Defn.~4.17]{Kopp2020d}. The main theorems of this appendix, \Cref{thm:coboundaryall} and \Cref{thm:cocycleall}, will then establish the compatibility of the two cases in this definition.

\begin{defn}\label{defn:sfjall}
    Let $M \in \SL_2(\Z)$, $z \in \C$, and $\tau \in \DD_M$. 
    \begin{itemize}
        \item[(1)]
    If $M \neq -T^k$ for any $k \in \Z$, we define $\sfjM{M}{z}{\tau}$ to be the meromorphic continuation of ~\eqref{eq:sfjfrm} from $\HH$ to $\DD_M$, which exists by \cite[Thm.~4.29]{Kopp2020d}.
        \item[(2)]
    If $M = -T^k$ for some $k \in \Z$, we define it by
    \begin{equation}
        \sfjM{M}{z}{\tau}
        = \frac{\varpi\!\left(\frac{z}{j_M(\tau)},M\cdot\tau\right)}{\varpi(z,\tau)}
        = \frac{\varpi(-z,\tau)}{\varpi(z,\tau)}.
    \end{equation}
    \end{itemize}
\end{defn}

The next lemma establishes a basic functional equation relating the behavior of the function $\varpi$ on the upper and lower half-planes under the transformation $\tau \mapsto -\tau$.

\begin{lem}\label{lem:pochupperlower}
    The identity of meromorphic functions
    \begin{equation}
        \varpi(z,\tau) = \varpi(z-\tau,-\tau)^{-1}
    \end{equation}
    holds for $z \in \C$ and $\tau \in \DD_{-I} = \HH \cup (-\HH)$.
\end{lem}
\begin{proof}
    This is an immediate consequence of \Cref{defn:pochupperlower}.
\end{proof}

We also establish a functional equation that relates the behavior of $\sigma_S$ on the the upper and lower half-planes, this time under the transformation $\tau \mapsto \frac{1}{\tau}$.

\begin{lem}\label{lem:SFJcocycleSzOvertau}
    The identity of meromorphic functions
    \begin{equation}\label{eq:SFJcocycleSzOvertau}
        \sfjM{S}{\frac{z}{\tau}}{\frac{1}{\tau}}
        = \frac{e^{2\pi i z}-1}{e^{2\pi i z/\tau}-1} \cdot \sfjM{S}{z}{\tau}
    \end{equation}
    holds for all $z \in \C$ and $\tau \in \DD_S$.
\end{lem}
\begin{proof}
    Using the identities (given in \cite[Sec.~4.8]{Kopp2020d})
\eag{
\dbs(z+1,\tau)& = \frac{\dbs(z,\tau)}{2 \sin\left(\frac{\pi z}{\tau}\right)}
\label{eq:doubleSineQuasiPeriodicity1}
\\
\dbs(z+\tau,\tau)& = \frac{\dbs(z,\tau)}{2 \sin\left(\pi z\right)}
\label{eq:doubleSineQuasiPeriodicity2}
\\
\dbs\!\left(\frac{z}{\tau},\frac{1}{\tau}\right) &= \dbs(z,\tau)
}  
and the fact that $\tau \in \DD_S \iff \tau^{-1} \in \DD_S$ in ~\eqref{eq:SFJacobiCocycleTermsDoubleSine}, we find
\eag{
\sfjM{S}{\frac{z}{\tau}}{\frac{1}{\tau}} &= \frac{e^{\frac{\pi i \tau}{12}\left(\frac{6z^2}{\tau^2}+6\left(1-\frac{1}{\tau}\right)\frac{z}{\tau} + \frac{1}{\tau^2}-\frac{3}{\tau}+1\right)}}{\dbs\left(\frac{z}{\tau}+1,\frac{1}{\tau}\right)}
\nn
&= \frac{e^{\frac{\pi i}{12\tau}\left(6z^2 -6(1-\tau)z+\tau^2-3\tau+1\right)}}{\dbs(z+\tau,\tau)}
\nn
&= \frac{2\sin (\pi z) e^{\frac{\pi i (\tau-1)z}{\tau}} e^{\frac{\pi i}{12\tau}\left(6z^2+6(1-\tau)z+\tau^2-3\tau+1\right)}}{\dbs(z,\tau)}
\nn
&= \frac{\sin(\pi z)}{\sin\!\left(\frac{\pi z}{\tau}\right)} e^{\frac{\pi i (\tau-1)z}{\tau}}\sfjM{S}{z}{\tau} 
\nn
&= \frac{e^{2\pi i z}-1}{e^{2\pi i z/\tau}-1} \cdot \sfjM{S}{z}{\tau}
}
for all $z\in \mbb{C}$, $\tau \in \DD_S$, which establishes ~\eqref{eq:SFJcocycleSzOvertau}. 
\end{proof}

Using the above lemmas, we will now establish that the condition in \Cref{defn:sfjall}(2) holds true for all $M \in \SL_2(\Z)$, not just those of the form $M = -T^k$, giving a simple expression for $\sfjM{M}{z}{\tau}$ on $\C \times (\HH \cup -\HH)$.

\begin{thm}\label{thm:coboundaryall}
    Let $M \in \SL_2(\Z)$.
    The identity of meromorphic functions
    \begin{equation}\label{eq:coboundaryuppperlower}
        \sfjM{M}{z}{\tau}
        = \frac{\varpi\!\left(\frac{z}{j_M(\tau)},M\cdot\tau\right)}{\varpi(z,\tau)}
    \end{equation}
    holds for $z \in \C$ and $\tau \in \DD_{-I} = \HH \cup (-\HH)$.
\end{thm}
\begin{proof}
    The identity \eqref{eq:coboundaryuppperlower} holds for $\tau \in \HH$ by \Cref{df:shinfadjacocycle}, so it suffices to prove this identity for $\tau \in -\HH$.
    In the case $M = S$, we have
    \begin{align}
        \sfjM{S}{z}{\tau}
        &= \frac{e^{2\pi i z/\tau}-1}{e^{2\pi i z}-1} \cdot \sfjM{S}{\frac{z}{\tau}}{\frac{1}{\tau}} && \mbox{(by \Cref{lem:SFJcocycleSzOvertau})} 
        \nn
        &= \frac{e^{2\pi i z/\tau}-1}{e^{2\pi i z}-1} \cdot \frac{\varpi(z,-\tau)}{\varpi\!\left(\frac{z}{\tau},\frac{1}{\tau}\right)} && \mbox{(by \eqref{eq:sfjfrm})}
        \nn
        &= \frac{e^{2\pi i z/\tau}-1}{e^{2\pi i z}-1} \cdot \frac{\varpi(z+\tau,\tau)^{-1}}{\varpi\!\left(\frac{z}{\tau}-\frac{1}{\tau},-\frac{1}{\tau}\right)^{-1}} && \mbox{(by \Cref{lem:pochupperlower})}
        \nn
        &= \frac{e^{2\pi i z/\tau}-1}{e^{2\pi i z}-1} \cdot \frac{\varpi\!\left(\frac{z}{\tau}-\frac{1}{\tau},-\frac{1}{\tau}\right)}{\varpi(z+\tau,\tau)}
        \nn
        &= \frac{\varpi\!\left(\frac{z}{\tau},-\frac{1}{\tau}\right)}{\varpi(z,\tau)}, \label{eq:coboundaryS}
    \end{align}
    proving \eqref{eq:coboundaryuppperlower} in the case $M=S$.
    We now prove \eqref{eq:coboundaryuppperlower} for arbitrary $M = \smmattwo{\ma}{\mb}{\mc}{\md}$ by splitting into cases based on the sign of $\mc$.

    \spc \noindent \textbf{Case 1.} Suppose $\mc = 0$.
    Then either $M = T^k$ or $M = -T^k$ for some $k \in \Z$.
    The identity \eqref{eq:coboundaryuppperlower} holds trivially in the case $M = T^k$, because
    \begin{align}
        \sfjM{T^k}{z}{\tau} = 1 = \frac{\varpi(z,\tau)}{\varpi(z,\tau)} = \frac{\varpi\!\left(\frac{z}{1},\tau+k\right)}{\varpi(z,\tau)} = \frac{\varpi\!\left(\frac{z}{j_{T^k}(\tau)},T^k\cdot\tau\right)}{\varpi(z,\tau)}.
    \end{align}
    It also holds for $M = -T^k$ by \Cref{defn:sfjall}(2).

    \spc \noindent \textbf{Case 2.} Suppose $\mc > 0$.
    We proceed as in the proof of \cite[Thm.~4.29]{Kopp2020d}.
    Divide $\mc$ by $\ma$ with negative remainder to obtain $\ma = \mc k-\mc'$ for some $k\in \Z$ and some $\mc' \in \Z$ with $0 \leq \mc' < \mc$. Set
    \begin{equation}
        M' = \mattwo{\ma'}{\mb'}{\mc'}{\md'} = S^{-1}T^{-k}M = \mattwo{\mc}{\md}{\mc k-\ma}{\md k-\mb}.
    \end{equation}
    Thus $M = T^k S M'$, and the cocycle condition known to be valid for $\tau \in \HH$ tells us that
    \begin{align}
        \sfjM{M}{z}{\tau}
        &= 
        \sfjM{T^k}{\frac{z}{j_{SM'}(\tau)}}{SM'\cdot\tau}
        \sfjM{S}{\frac{z}{j_{S}(\tau)}}{M'\cdot\tau}
        \sfjM{M'}{z}{\tau}
        \nn
        &= 
        \sfjM{S}{\frac{z}{j_{M'}(\tau)}}{M'\cdot\tau}
        \sfjM{M'}{z}{\tau}. \label{eq:cocyclestep}
    \end{align}
    By analytic continuation, this relation holds for $\tau \in \DD_M \cap (M')^{-1}\cdot\DD_S \cap \DD_M$. We have $\mc\md'-\md\mc' = \mc(\md k-\mb)-\md(\mc k-\ma) = \ma\md-\mb\mc = 1$, so (unless $\mc' = 0$) $\frac{\md'}{\mc'} - \frac{\md}{\mc} = \frac{1}{\mc \mc'} > 0$, and thus $\frac{\md'}{\mc'} > \frac{\md}{\mc}$. Therefore
    \begin{align}
        \DD_M &= \C \setminus \left(-\infty, -\frac{\md}{\mc}\right], \\
        (M')^{-1}\cdot\DD_S &= \C \setminus \left(-\frac{\md'}{\mc'}, -\frac{\md}{\mc}\right] \supseteq \DD_M, \\ 
        \DD_{M'} &= \C \setminus \left(-\infty, -\frac{\md'}{\mc'}\right] \supseteq \DD_M,
    \end{align}
    so in fact \eqref{eq:cocyclestep} holds for $\tau \in \DD_M$. (In the omitted case $\mc' =0$, we must have $M' = T^k$ for some $k \in \Z$, and again the relation holds for $\tau \in \DD_M$.) In particular, it holds on $-\HH$, so it may be used to prove \eqref{eq:coboundaryuppperlower} by induction on $\mc$, using Case 1 as the base case and \eqref{eq:coboundaryS} to prove the inductive step. To be explicit, applying the inductive hypothesis and \eqref{eq:coboundaryS} to \eqref{eq:cocyclestep} gives
    \begin{align}
        \sfjM{M}{z}{\tau}
        &=
        \frac{\varpi\!\left(\frac{z}{j_{SM'}(\tau)},SM'\cdot\tau\right)}{\varpi\!\left(\frac{z}{j_{M'}(\tau)},M'\cdot\tau\right)}
        \cdot
        \frac{\varpi\!\left(\frac{z}{j_{M'}(\tau)},M'\cdot\tau\right)}{\varpi\!\left(z,\tau\right)}
        =
        \frac{\varpi\!\left(\frac{z}{j_{SM'}(\tau)},SM'\cdot\tau\right)}{\varpi\!\left(z,\tau\right)}
        \nn
        &=
        \frac{\varpi\!\left(\frac{z}{j_{T^k SM'}(\tau)},T^k SM'\cdot\tau\right)}{\varpi\!\left(z,\tau\right)}
        = 
        \frac{\varpi\!\left(\frac{z}{j_{M}(\tau)},M\cdot\tau\right)}{\varpi\!\left(z,\tau\right)}.
    \end{align}
    
    \spc \noindent \textbf{Case 3.} Suppose $\mc < 0$.
    The relation
    \begin{equation}
        1 = \sfjM{M^{-1}M}{z}{\tau} = \sfjM{M^{-1}}{\frac{z}{j_M(\tau)}}{M\cdot\tau} \sfjM{M}{z}{\tau}
    \end{equation}
    is valid for $\tau \in M^{-1}\cdot\DD_{M^{-1}} \cap \DD_M = \DD_M$. Also, $M^{-1} = \smmattwo{\md}{-\mb}{-\mc}{\ma}$, so Case 2 applies to $M^{-1}$. Thus,
    \begin{equation}
        \sfjM{M}{z}{\tau}
        = \sfjM{M^{-1}}{\frac{z}{j_M(\tau)}}{M\cdot\tau}^{-1}
        = \frac{\varpi\!\left(z,\tau\right)^{-1}}{\varpi\!\left(\frac{z}{j_{M}(\tau)},M\cdot\tau\right)^{-1}}
        = \frac{\varpi\!\left(\frac{z}{j_{M}(\tau)},M\cdot\tau\right)}{\varpi\!\left(z,\tau\right)}
    \end{equation}
    for $\tau \in -\HH$.
\end{proof}

Finally, we prove that the cocycle relation for the Shintani--Faddeev Jacobi cocycle holds true everywhere, including in those cases not covered in \cite{Kopp2020d}.

\begin{thm}\label{thm:cocycleall}
    Let $M, M' \in \SL_2(\Z)$. The identity of meromorphic functions
    \begin{equation}
        \sfjM{MM'}{z}{\tau}
        = \sfjM{M}{\frac{z}{j_{M'}}}{M'\cdot\tau} \sfjM{M'}{z}{\tau}
    \end{equation}
    holds for $z \in \C$ and $\tau \in \DD_{M,M'} := \DD_{MM'} \cap (M')^{-1}\cdot\DD_M \cap \DD_{M'}$.
\end{thm}
\begin{proof}
    The cocycle relation holds for $\tau \in \HH \cup (-\HH)$ as an immediate consequence of \Cref{thm:coboundaryall}. It then holds for $\tau \in \DD_{MM'}$ by analytic continuation.
\end{proof}

\section{Real quadratic fields with an odd-trace unit}\label{ap:oddtrace}

\Cref{thm:cofinal} splits the characterization of abelian extensions of real quadratic fields (conjecturally) generated by $r$-SICs into two cases, and it would be nice to know how often each case occurs. We give a partial result, showing that the trace of the fundamental unit is odd (the case when the full maximal abelian extension is conjecturally attained) a positive proportion of the time. 
Thus, for a positive proportion of real quadratic fields, $r$-SICs may be viewed as a geometric solution to Hilbert's twelfth problem (albeit a conjectural one, depending on the Stark--Tate Conjecture and the Twisted Convolution Conjecture).

If a quadratic field contains an odd trace unit, then elementary methods show that its discriminant must obey a congruence restriction modulo $8$. 
\begin{lem}\label{lem:fivemodeight}
If $K$ is a real quadratic field containing a unit of odd trace, then $\disc K \con 5 \Mod{8}$.
\end{lem}
\begin{proof}
Let $\Delta = \disc{K}$. If $\Delta$ is even, then $\OO_K = \Z + \frac{\sqrt{\Delta}}{2}\Z$, so $\Tr(\alpha)$ is even for all $\alpha \in \OO_K$. 
Thus, $\Delta$ must be odd; indeed, $\Delta \con 1 \Mod{4}$ because $\Delta$ is a discriminant.

The ring of integers of $K$ is therefore
\begin{equation}
    \OO_K = \Z + \frac{1+\sqrt{\Delta}}{2}\Z.
\end{equation}
The quotient ring $\OO_K/2\OO_K$ is represented by the congruence classes of $0, 1, \frac{-1+\sqrt{\Delta}}{2},$ and $\frac{1+\sqrt{\Delta}}{2}$. If $\e$ is a unit of odd trace in $\OO_K^\times$, then
\begin{align}
    \e &\con \frac{\pm 1+\sqrt{\Delta}}{2} \Mod{2\OO_K}; \\
    \e' &\con \frac{\pm 1-\sqrt{\Delta}}{2} \Mod{2\OO_K}.
\end{align}
Thus, $\Nm(\e)=\e\e' \con -\frac{1-\Delta}{4} \Mod{2\OO_K}$.
If $\Delta \con 1 \Mod{8}$, then it follows that $2|\Nm(\e)$, which contradicts the fact that $\e$ is a unit. Thus, $\Delta \con 5 \Mod{8}$.
\end{proof}

\begin{lem}\label{lem:cubicfields}
Let $K$ be a real quadratic field such that $\disc K \con 5 \Mod{8}$. Consider the following conditions:
\begin{itemize}
    \item[(1)] There exists no cubic number field $L$ such that $\disc L = \disc K$ or $\disc L = 4\disc K$.
    \item[(2)] The field $K$ has a unit of odd trace.
    \item[(3)] There exists no cubic number field $L$ such that $\disc L = 4\disc K$.
\end{itemize}
Then (1) implies (2), and (2) implies (3).
\end{lem}
\begin{proof}
Consider the exact sequence of ray class groups
\begin{equation}
1 \to \frac{\OO_K^\times}{\U_{2}(\OO_K)} \xrightarrow{\iota} (\OO_K/2\OO_K)^\times \to \Cl_2(\OO_K) \xrightarrow{\pi} \Cl(\OO_K) \to 1.
\end{equation}
This is the exact sequence given in \cite[Thm.\ 5.4]{Kopp2020b}, specialized to the case when $\OO=\OO'=\mm'=\OO_K$, $\mm=2\OO_K$, and $\rS=\rS'=\{\}$. The group $(\OO_K/2\OO_K)^\times \isom \Z/3\Z$ because $\disc K \equiv 5 \Mod{8}$. The unit group $\OO_K^\times$ has an element of odd trace if and only if the the map $\iota$ is an isomorphism, that is, if and only if the map $\pi$ is an isomorphism. Setting $h_K = \abs{\Cl(\OO_K)}$, then exactly one of the following is true:
\begin{itemize}
    \item[(A)] $\abs{\Cl_2(\OO_K)} = h_K$, and $K$ has a unit of odd trace; or
    \item[(B)] $\abs{\Cl_2(\OO_K)} = 3h_K$, and $K$ does not have a unit of odd trace.
\end{itemize}

Let $\phi$ be any nontrivial group homomorphism $\phi : \Cl_2(\OO_K) \to \Z/3\Z$. 
By the existence theorem of class field theory and the Galois correspondence, there exists a cubic subextension $M/K$ of the ray class field $H_2/K$ corresponding to $\ker(\phi)$ under the Galois correspondence. The field $M$ is sextic over $\Q$, with $\Gal(M/\Q) \isom S_3$. Pick a cubic subfield $L$ of $M$; $L$ is unique up to isomorphism. Using the conductor-discriminant formula, one can show that $\disc L = \disc K$ if $\phi$ factors through the map $\pi$, and $\disc L = 4\disc K$ otherwise. Moreover, this correspondence defines bijections
\begin{align}
    \{\phi : \Cl(\OO_K) \surj \Z/3\Z\} &\longleftrightarrow \{\mbox{cubic } L/\Q : \disc(L) = \disc(K)\}, \\
    \{\phi : \Cl_2(\OO_K) \surj \Z/3\Z, \phi \neq \pi \circ \phi'\} &\longleftrightarrow \{\mbox{cubic } L/\Q : \disc(L) = 4\disc(K)\}. \label{eq:cubicfieldunfactor}
\end{align}

We now prove that (1) implies (2). Suppose that there is no cubic number field $L$ with $\disc L = \disc K$ or $\disc L = 3\disc K$.  Thus, there is no nontrivial group homomorphism $\phi : \Cl_2(\OO_K) \to \Z/3\Z$. Since $\Cl_2(\OO_K)$ is abelian, this means that $3 \ndiv \abs{\Cl_2(\OO_K)}$. It then follows that we are in case (A), and $K$ has a unit of odd trace.

Finally, we prove that (2) implies (3). Suppose that $K$ has a unit of odd trace, so we are in case (A): $\abs{\Cl_2(\OO_K)} = h_K$. Thus, every nontrivial group homomorphism $\phi : \Cl_2(\OO_K) \to \Z/3\Z$ factors through $\Cl(\OO_K)$. By \eqref{eq:cubicfieldunfactor}, there is no cubic number field $L$ with $\disc(L) = 4\disc(K)$.
\end{proof}

In the authors' understanding, analytic number theory and arithmetic statistics have not yet produced techniques capable of finding the exact asymptotic density of the number of real quadratic fields with a unit of odd trace. 
However, techniques for counting quadratic and cubic fields may be used to give upper and lover bounds using \Cref{lem:fivemodeight} and \Cref{lem:cubicfields}. 
The following proof is based on ideas suggested by Frank Thorne \cite{ThorneCorrespondence} and Jiuya Wang \cite{WangCorrespondence}.
\begin{proof}[Proof of \Cref{thm:oddtracecount}]
In the proof, we will use the notation $\Delta_K = \disc K$ for the field discriminant and 
\begin{align}
    N_2(X;\text{[conditions]}) &= \#\{K : [K : \Q]=2, 0 < \Delta_K < X, \text{ [conditions]}\}, \\
    N_3(X;\text{[conditions]}) &= \#\{L : [L : \Q]=3, 0 < \Delta_L < X, \text{ [conditions]}\}
\end{align}
for counts of quadratic and cubic fields of positive discriminant satisfying certain conditions. We write $N_2(X)$ and $N_3(X)$ if there are no conditions.

An integer $\Delta>1$ is the discriminant of a quadratic field if and only if it satisfies the congruence conditions $\Delta \con 1,5,8,9,12,13 \Mod{16}$ and $p^2\ndiv\Delta$ for all odd primes $p$. 
By the following standard sieve-theoretical calculation, taking $\mu(d)$ to be the M\"{o}bius function and taking $r_d$ to be the smallest positive solution to $r_d \equiv n_0 \Mod{16}$ and $r_d \equiv 0 \Mod{d^2}$,
\begin{align}
    \{n \leq X : n \equiv n_0 \Mod{16}, p^2\ndiv n \mbox{ for } p \mbox{ odd}\}
    &= \sum_{\substack{1\leq d \leq X^{1/2} \\ 2\,\ndiv\,d}} \mu(d)\floor{\frac{X-r_d}{16d^2}} \\
    &= \sum_{\substack{1\leq d \leq X^{1/2} \\ 2\,\ndiv\,d}} \mu(d)\frac{X}{16d^2} + O(X^{1/2}) \\
    &= \left(\sum_{\substack{1\leq d \\ 2\,\ndiv\,d}} \frac{\mu(d)}{16d^2} + O(X^{-1/2})\right)X + O(X^{1/2}) \\
    &= \left(\frac{1}{16}\prod_{p \neq 2} \left(1-\frac{1}{p^2}\right)\right)X + O(X^{1/2}) \\
    &= \frac{1}{12\zeta(2)}X + O(X^{1/2}).
\end{align}
Thus, taking $n_0 \in \{1,5,8,9,12,13\}$ and $n_0 \in \{5,13\}$, respectively, we have
\begin{align}
    N_2(X)
    &= \frac{1}{2\zeta(2)}X + O(X^{1/2}), \mbox{ and} \label{eq:allquadcount} \\
    N_2(X; \Delta_K \equiv 5 \Mod{8})
    &= \frac{1}{6\zeta(2)}X + O(X^{1/2}). \label{eq:fivemodeightcount}
\end{align}
Thus, by using \Cref{lem:fivemodeight} and dividing these two asymptotic equalities \eqref{eq:fivemodeightcount} and \eqref{eq:allquadcount}, we obtain the upper bound
\begin{equation}
    \frac{N_2(X; 2\ndiv\Tr(\e_K))}{N_2(X)}
    \leq \frac{N_2(X; \Delta_K \equiv 5 \Mod{8})}{N_2(X)}
    = \frac{1}{3} + O(X^{-1/2}). \label{eq:upperbound}
\end{equation}

To obtain the lower bound, we appeal to the results of Taniguchi and Thorne \cite{Taniguchi:2013} on counting cubic fields with local restrictions. If $L$ is a cubic field, then $\OO_L \tensor \Z_2$ a ``maximal cubic ring over $\Z_2$,'' that is, a product of valuation rings of finite extensions of $\Q_2$ whose degrees sum to $3$. There are exactly $10$ maximal cubic rings over $\Z_2$, shown in the following table, which is based on the tables given in \cite[p.\ 2487--2488]{Taniguchi:2013} and on the database of Jones and Roberts \cite{jr}. In the table, $\omega$ is a root of $x^2+x+1=0$, and $\alpha$ is a root of $x^3-x-1=0$.

\begin{center}
\bgroup
\def\arraystretch{1.2}
\begin{tabular}{p{4 cm} p{5.5 cm} p{2 cm}}
     \toprule
     $\OO_L \tensor \Z_2$ & forced congruence conds.\ & density wt.\ \\
     \midrule
     $\Z_2 \times \Z_2 \times \Z_2$ & $\Delta_L \equiv 1 \Mod{8}$ & $1/6$ \\
     $\Z_2 \times \Z_2[\omega]$ & $\Delta_L \equiv 5 \Mod{8}$ & $1/2$ \\
     $\Z_2[\alpha]$ & $\Delta_L \equiv 1 \Mod{8}$ & $1/3$ \\
     $\Z_2[\sqrt[3]{2}]$ & $\Delta_L \equiv 20 \Mod{32}$ & $1/4$ \\   
     $\Z_2 \times \Z_2[\sqrt{-1}]$ & $\Delta_L \equiv 28 \Mod{32}$ & $1/4$ \\   
     $\Z_2 \times \Z_2[\sqrt{3}]$ & $\Delta_L \equiv 12 \Mod{32}$ & $1/4$ \\  
     $\Z_2 \times \Z_2[\sqrt{2}]$ & $\Delta_L \equiv 8 \Mod{64}$ & $1/8$ \\   
     $\Z_2 \times \Z_2[\sqrt{-2}]$ & $\Delta_L \equiv 56 \Mod{64}$ & $1/8$ \\  
     $\Z_2 \times \Z_2[\sqrt{6}]$ & $\Delta_L \equiv 24 \Mod{64}$ & $1/8$ \\  
     $\Z_2 \times \Z_2[\sqrt{-6}]$ & $\Delta_L \equiv 40 \Mod{64}$ & $1/8$\\
     \bottomrule
\end{tabular}
\egroup
\end{center}

In particular, the congruence condition $\Delta_L \equiv 5 \Mod{8}$ is equivalent to the local restriction $\OO_L \tensor \Z_2 = \Z_2 \times \Z_2[\omega]$, and the congruence condition $\Delta_L \equiv 20 \Mod{32}$ is equivalent to the local restriction $\OO_L \tensor \Z_2 = \Z_2[\sqrt[3]{2}]$. The following asymptotics are thus special cases of \cite[Thm.\ 1.3]{Taniguchi:2013}\footnote{Note that the more general result \cite[Thm.\ 6.2]{Taniguchi:2013} also allows one to impose additional congruence restrictions, but we don't need to do this.}, with the secondary term absorbed into the error term:
\begin{align}
    N_3(X; \Delta_L \equiv 5 \Mod{8}) \label{eq:n3dummy1}
    &= \frac{C^+(\SSS_{\Z_2 \times \Z_2[\omega]})}{12\zeta(3)}X + O(X^{5/6}), \\ 
    N_3(X; \Delta_L \equiv 20 \Mod{32}) \label{eq:n3dummy2}
    &= \frac{C^+(\SSS_{\Z_2[\sqrt[3]{2}]})}{12\zeta(3)}X + O(X^{5/6}). 
\end{align}
The constants $C^+(\SSS_{\Z_2 \times \Z_2[\omega]})$ and $C^+(\SSS_{\Z_2[\sqrt[3]{2}]})$ are computed from the ``density weights'' in the table:
\begin{align}
    C^+(\SSS_{\Z_2 \times \Z_2[\omega]})
    &= \frac{1/2}{1/6+1/2+1/3+1/4+1/4+1/4+1/8+1/8+1/8+1/8} = \frac{2}{9},
    \\
    C^+(\SSS_{\Z_2[\sqrt[3]{2}]})
    &= \frac{1/4}{1/6+1/2+1/3+1/4+1/4+1/4+1/8+1/8+1/8+1/8} = \frac{1}{9}.
\end{align}
However, these are not actually the asymptotics we want---we should also be 
removing
non-fundamental discriminants by imposing the condition that $p^2 \ndiv \Delta_L$ for odd primes $p$.
The condition that $p^2 \ndiv \Delta_L$ is equivalent to the condition that $\OO_L \tensor \Z_p$ is not totally ramified at $p$; see \cite[Sec.~6.1]{Taniguchi:2013}.
Let $C^+_{\rm ntr}(p)$ denote the local density of non-totally ramified $\OO_L \tensor \Z_p$.
The following asymptotic formulas are also special cases of \cite[Thm.\ 1.3]{Taniguchi:2013}, treating $Y$ as a constant.
\begin{align}
    N_3\!\left(X; \begin{array}{c} \Delta_L \equiv 5 \Mod{8}, \\ p^2 \ndiv \Delta_L \mbox{ for odd } p \leq Y \end{array}\!\right)
    &= \frac{C^+(\SSS_{\Z_2 \times \Z_2[\omega]})}{12\zeta(3)}\left(\prod_{\substack{p \leq Y \\ 2 \,\ndiv\, p}} C^+_{\rm ntr}(p)\right)X + O_Y(X^{5/6}), \\
    N_3\!\left(X; \begin{array}{c} \Delta_L \equiv 20 \Mod{32}, \\ p^2 \ndiv \Delta_L \mbox{ for odd } p \leq Y \end{array}\!\right)
    &= \frac{C^+(\SSS_{\Z_2[\sqrt[3]{2}]})}{12\zeta(3)}\left(\prod_{\substack{p \leq Y \\ 2 \,\ndiv\, p}} C^+_{\rm ntr}(p)\right)X + O_Y(X^{5/6}).
\end{align}
The $C^+_{\rm ntr}(p)$ are calculated using the table in \cite[Sec.~6.2]{Taniguchi:2013} to be
\begin{equation}
    C^+_{\rm ntr}(p) 
    = \frac{1/6+1/2+1/3+1/p}{1/6+1/2+1/3+1/p+1/p^2}
    = \frac{1+p^{-1}}{1+p^{-1}+p^{-2}}
    = \frac{(1-p^{-3})^{-1}}{(1-p^{-2})^{-1}},
\end{equation}
and hence
\begin{equation}
    \prod_{\substack{p \leq Y \\ 2 \,\ndiv\, p}} C^+_{\rm ntr}(p)
    = \frac{1-2^{-3}}{1-2^{-2}} \cdot \frac{\zeta(3)}{\zeta(2)} + O(Y^{-1})
    = \frac{7\zeta(3)}{6\zeta(2)} + O(Y^{-1}).
\end{equation}
Thus, we obtain the asymptotic formulas
\begin{align}
    N_3\!\left(X; \begin{array}{c} \Delta_L \equiv 5 \Mod{8}, \\ p^2 \ndiv \Delta_L \mbox{ for odd } p \leq Y \end{array}\!\right) 
    &= \frac{7}{162\zeta(2)}X + O(X/Y) + O_Y(X^{5/6}), \\
    N_3\!\left(X; \begin{array}{c} \Delta_L \equiv 20 \Mod{32}, \\ p^2 \ndiv \Delta_L \mbox{ for odd } p \leq Y \end{array}\!\right) 
    &= \frac{7}{324\zeta(2)}X + O(X/Y) + O_Y(X^{5/6}).
\end{align}
By sending $Y \to \infty$ (sufficiently slowly compared to $X$), we have
\begin{align}
    N_3\!\left(X; \begin{array}{c} \Delta_L \equiv 5 \Mod{8}, \\ \Delta_L \mbox{ fundamental} \end{array}\!\right) 
    &= \frac{7}{162\zeta(2)}X + o(X), \\
    N_3\!\left(X; \begin{array}{c} \Delta_L \equiv 20 \Mod{32}, \\ \Delta_L \mbox{ fundamental} \end{array}\!\right) 
    &= \frac{7}{324\zeta(2)}X + o(X).
\end{align}
By \Cref{lem:cubicfields} (specifically the fact that (1) implies (2)), we have the bound
\begin{align}
    N_2(X; & 2|\Tr(\e_K), \Delta_K \equiv 5 \Mod{8}) \\
    &\leq N_3\!\left(X; \begin{array}{c} \Delta_L \equiv 5 \Mod{8}, \\ \Delta_L \mbox{ fundamental} \end{array}\!\right) + N_3\!\left(X; \begin{array}{c} \Delta_L \equiv 20 \Mod{32}, \\ \Delta_L \mbox{ fundamental} \end{array}\!\right) \\
    &= \frac{7}{162\zeta(2)}X + \frac{7}{324\zeta(2)}(4X) + o(X) \\
    &= \frac{7}{54\zeta(2)}X + o(X).
\end{align}
Thus, using \Cref{lem:fivemodeight},
\begin{align}
    N_2(X; 2\ndiv\Tr(\e_K))
    &= N_2(X, \Delta_K \equiv 5 \Mod{8}) - N_2(X; 2\div\Tr(\e_K), \Delta_K \equiv 5 \Mod{8}) \\
    &\geq \frac{1}{6\zeta(2)}X - \frac{7}{54\zeta(2)}X + o(X) \\
    &= \frac{1}{27\zeta(2)}X + o(X). \label{eq:oddtracelowercount}
\end{align}
By dividing the asymptotic formulas \eqref{eq:oddtracelowercount} and \eqref{eq:allquadcount}, we obtain the lower bound
\begin{align}
    \frac{N_2(X;2\ndiv\Tr(\e_K))}{N_2(X)} \geq \frac{2}{27} + o(1). \label{eq:lowerbound}
\end{align}
Combining \eqref{eq:upperbound} and \eqref{eq:lowerbound} completes the proof.
\end{proof}

The upper bound in \Cref{thm:oddtracecount} is fairly trivial, as it only uses the congruence restriction $\disc K \equiv 5 \Mod{8}$; one might hope to get a better upper bound using the fact that ``(2) implies (3)'' from \Cref{lem:cubicfields}, which we did not use at all! It is not clear how to so do at present, as we would need some additional result to tell us that the cubic fields with $\disc L \equiv 20 \Mod{32}$ hit enough \textit{distinct} discriminants.

Numerical evidence suggests that the true asymptotic density of real quadratic fields with a unit of odd trace is about $2/9$ (or $22.2\%$), that is, about $2/3$ (or $66.7\%$) of the real quadratic fields with discriminant congruent to $5$ modulo $8$. According to a calculation performed in Mathematica, among real quadratic fundamental discriminants $\Delta = 8k-3$ for integers $k \in [10^{10}, 10^{10}+10^{6}]$, about $66.9\%$ have a unit of odd trace. Calculations involving smaller discriminants suggest a positive bias in the count of such discriminants up to $X$ that is going to zero slower than $X^{1-\delta}$ for any $\delta>0$.

Finally, we give a brief comparison to the problem of solubility of the negative ``Pell'' equation. 
The existence of a unit of negative norm in the real quadratic field $\Q(\sqrt{D})$ is equivalent to the existence of an integer solution to the equation $x^2-Dy^2=-1$, whereas the existence of a unit of odd trace in the real quadratic field $\Q(\sqrt{D})$ with $D \equiv 1 \Mod{4}$ is equivalent to the existence of an integer solution to the equation $x^2+xy+\frac{1-D}{4}y^2=1$ with $y$ odd. 
In the former case, however, the asymptotic density of such $D$ is zero, and this leads to additional complications. 
Nonetheless, as in our problem, upper and lower bounds of the same order of magnitude can be given on the number of such $D$ up to $X$; this was done by Fouvry and Kl\"{u}ners \cite[Thm.\ 1]{Fouvry:2010}. 
The narrow class group of $\OO_K$ plays a similar role in their work as does the ray class group modulo $2$ in our \Cref{lem:cubicfields}.

\section{\texorpdfstring{$1$}{1}-SIC data tables}\label{ap:sicdata}

In this appendix, we collect tables containing the algebraic data canonically specifying ghost $1$-SICs, and non-canonically specifying $1$-SICs, in dimensions $d = 4$--$100$. 
This list is conjecturally complete for all Weyl--Heisenberg covariant $1$-SICs; there is exactly one row corresponding to each predicted $\EC(d)$-orbit. 

We have numerically computed an approximate ghost SIC using the \SFKFull{} modular cocycle and checked that the \ghostOverlapsText{} satisfy 
\eag{
    \biggl|\Tr(\tilde{\Pi}_\p \tilde{\Pi}_{\zero}) - \frac{(1-\delta^d_{\p,\zero}) + \delta^d_{\p,\zero}(d+1)}{d+1}\biggr| < 10^{-66}
}
in all of the following cases:
\begin{enumerate}
    \item For the 251 rows corresponding to $d \le 60$; 
    \item For the 39 rows with $60 < d < 100$ where $Q = \langle1,1-d,1\rangle$;
    \item For the 4 rows with $d = 100$. 
\end{enumerate}
For $d\le 20$ and for $d = 100$ we have also used our necromancy procedure to numerically compute the set of associated $1$-SICs. 
See \cref{sec:nec} for more details about our numerical calculations. 

Each ghost fiducial is specified by an admissible tuple $t = (d,1,Q)$ with dimension $d$ and integer binary quadratic form $Q$. 
The relations between $t$ and the remaining data in the table are as follows. 
First factorize $(d+1)(d-3)=f_d^2 \Delta_0$ where $\Delta_0$ is a fundamental discriminant. 
Then $Q$ (conjecturally) yields a valid $1$-SIC if $\disc(Q) = f^2 \Delta_0$ where $f\div f_d$.
While $d$ and $Q$ determine all other data needed for constructing a ghost, the additional columns are included for convenience, since they contain additional algebraic data that may be ``difficult'' to compute, for example requiring integer factoring or finding a fundamental unit. 
The column $\Delta_0$ contains the fundamental discriminant of $Q$ and the column $f$ the conductor of $Q$. 
The columns $h$ and $\Cl(\mathcal{O}_f)$ contain the class number and class group respectively. 
The column $\Gal(E_s^{(1)}/H)$ is the Galois group of the candidate overlap field ramified at the first infinite place. 
(This is isomorphic to $\Gal(E_t^{(2)}/H)$, the Galois group of the candidate \textit{ghost} overlap field ramified at the \textit{second} infinite place.) 
The column $L^n$ is of the form where $L$ is the positive-trace generator of the stability group of $Q$ with the same sign as $Q$, as defined in \Cref{dfn:AssociatedStabilizers}, and $A = L^n$. 
The column `$\text{a.u.}$' is marked `$\text{Y}$' if there is an anti-unitary symmetry.
Finally, $\ell(A)$ is the length of the word expansion of $A$ using the Hirzebruch--Jung (negative regular) reduction into the standard ($S$ and $T$) generators of $\SLtwo{\Z}$. 
This is one measure of the complexity of constructing the actual ghost fiducial vector for that input. 
The forms $Q$ in this list were selected among class representatives to minimize this complexity, although this choice is not generally unique. 
In order to write down a ghost $1$-SIC explicitly, one must also choose a twist $G$, which may canonically be taken to be the identity matrix $G = \smmattwo1001$; the choice of twist does not affect the $\EC(d)$-orbit.
The data in each row are sufficient to compute a ghost fiducial, but to fully specify a $1$-SIC, one must additionally make an arbitrary, non-canonical choice of a sign-switching Galois automorphism $g$. %

\renewcommand{\arraystretch}{1.5}
\input{figures/data}

\clearpage
\bibliographystyle{habbrv}
\bibliography{refs}
\end{document}

%% file: figures/data.tex
\begin{center}
\begin{longtable}{tttttttttt}
\toprule
d & \Delta_0 & f & h & \Cl(\mathcal{O}_f) & \Gal\bigl(E_s^{(1)}/H\bigr) & Q & L^n & \text{a.u.} & \ell \\
\midrule
4&5&1& 1 &  C_{1}  & C_{2}^{2} &\langle1,-3,1\rangle&\smt{2&-1\\1&-1}^{6}&\text{Y}&4\\
\midrule
5&12&1& 1 &  C_{1}  & C_{8} &\langle1,-4,1\rangle&\smt{4&-1\\1&0}^{3}&&4\\
\midrule
6&21&1& 1 &  C_{1}  & C_{2}\times{}C_{6} &\langle1,-5,1\rangle&\smt{5&-1\\1&0}^{3}&&4\\
\midrule
7&8&1& 1 &  C_{1}  & C_{6} &\langle2,-4,1\rangle&\smt{3&-1\\2&-1}^{6}&\text{Y}&7\\
&&2& 1 &  C_{1}  & C_{2}\times{}C_{6} &\langle1,-6,1\rangle&\smt{6&-1\\1&0}^{3}&&4\\
\midrule
8&5&1& 1 &  C_{1}  & C_{2}\times{}C_{4} &\langle1,-3,1\rangle&\smt{2&-1\\1&-1}^{12}&\text{Y}&7\\
&&3& 1 &  C_{1}  & C_{2}\times{}C_{4}^{2} &\langle1,-7,1\rangle&\smt{7&-1\\1&0}^{3}&&4\\
\midrule
9&60&1& 2 & C_{2}  & C_{3}\times{}C_{6} &\langle1,-8,1\rangle&\smt{8&-1\\1&0}^{3}&&4\\
&&& & & &\langle5,-10,2\rangle&\smt{9&-2\\5&-1}^{3}&&7\\
\midrule
10&77&1& 1 &  C_{1}  & C_{2}\times{}C_{24} &\langle1,-9,1\rangle&\smt{9&-1\\1&0}^{3}&&4\\
\midrule
11&24&1& 1 &  C_{1}  & C_{40} &\langle3,-6,1\rangle&\smt{11&-2\\6&-1}^{3}&&7\\
&&2& 2 & C_{2}  & C_{40} &\langle1,-10,1\rangle&\smt{10&-1\\1&0}^{3}&&4\\
&&& & & &\langle3,-12,4\rangle&\smt{11&-4\\3&-1}^{3}&&7\\
\midrule
12&13&1& 1 &  C_{1}  &  C_{2}^{4} &\langle3,-5,1\rangle&\smt{4&-1\\3&-1}^{6}&\text{Y}&10\\
&&3& 1 &  C_{1}  & C_{2}^{3}\times{}C_{6} &\langle1,-11,1\rangle&\smt{11&-1\\1&0}^{3}&&4\\
\midrule
13&140&1& 2 & C_{2}  & C_{4}\times{}C_{12} &\langle1,-12,1\rangle&\smt{12&-1\\1&0}^{3}&&4\\
&&& & & &\langle7,-14,2\rangle&\smt{13&-2\\7&-1}^{3}&&7\\
\midrule
14&165&1& 2 & C_{2}  & C_{2}\times{}C_{6}^{2} &\langle1,-13,1\rangle&\smt{13&-1\\1&0}^{3}&&4\\
&&& & & &\langle5,-15,3\rangle&\smt{14&-3\\5&-1}^{3}&&7\\
\midrule
15&12&1& 1 &  C_{1}  & C_{24} &\langle1,-4,1\rangle&\smt{4&-1\\1&0}^{6}&&7\\
&&2& 1 &  C_{1}  & C_{2}\times{}C_{24} &\langle4,-8,1\rangle&\smt{15&-2\\8&-1}^{3}&&7\\
&&4& 2 & C_{2}  & C_{2}\times{}C_{24} &\langle1,-14,1\rangle&\smt{14&-1\\1&0}^{3}&&4\\
&&& & & &\langle11,-18,3\rangle&\smt{16&-3\\11&-2}^{3}&&10\\
\midrule
16&221&1& 2 & C_{2}  & C_{2}\times{}C_{8}^{2} &\langle1,-15,1\rangle&\smt{15&-1\\1&0}^{3}&&4\\
&&& & & &\langle7,-19,5\rangle&\smt{17&-5\\7&-2}^{3}&&10\\
\midrule
17&28&1& 1 &  C_{1}  & C_{96} &\langle2,-6,1\rangle&\smt{17&-3\\6&-1}^{3}&&7\\
&&3& 2 & C_{2}  & C_{96} &\langle1,-16,1\rangle&\smt{16&-1\\1&0}^{3}&&4\\
&&& & & &\langle9,-18,2\rangle&\smt{17&-2\\9&-1}^{3}&&7\\
\midrule
18&285&1& 2 & C_{2}  & C_{3}\times{}C_{6}^{2} &\langle1,-17,1\rangle&\smt{17&-1\\1&0}^{3}&&4\\
&&& & & &\langle13,-21,3\rangle&\smt{19&-3\\13&-2}^{3}&&10\\
\midrule
19&5&1& 1 &  C_{1}  & C_{18} &\langle1,-3,1\rangle&\smt{2&-1\\1&-1}^{18}&\text{Y}&10\\
&&2& 1 &  C_{1}  & C_{3}\times{}C_{18} &\langle4,-6,1\rangle&\smt{5&-1\\4&-1}^{6}&\text{Y}&13\\
&&4& 1 &  C_{1}  & C_{6}\times{}C_{18} &\langle5,-10,1\rangle&\smt{19&-2\\10&-1}^{3}&&7\\
&&8& 2 & C_{2}  & C_{6}\times{}C_{18} &\langle1,-18,1\rangle&\smt{18&-1\\1&0}^{3}&&4\\
&&& & & &\langle5,-20,4\rangle&\smt{19&-4\\5&-1}^{3}&&7\\
\midrule
20&357&1& 2 & C_{2}  & C_{2}^{3}\times{}C_{24} &\langle1,-19,1\rangle&\smt{19&-1\\1&0}^{3}&&4\\
&&& & & &\langle7,-21,3\rangle&\smt{20&-3\\7&-1}^{3}&&7\\
\midrule
21&44&1& 1 &  C_{1}  &  C_{2}^{2}\times{}C_{24} &\langle5,-8,1\rangle&\smt{22&-3\\15&-2}^{3}&&10\\
&&3& 4 & C_{4}  & C_{2}\times{}C_{6}^{2} &\langle1,-20,1\rangle&\smt{20&-1\\1&0}^{3}&&4\\
&&& & & &\langle5,-24,9\rangle&\smt{22&-9\\5&-2}^{3}&&10\\
&&& & & &\langle11,-22,2\rangle&\smt{21&-2\\11&-1}^{3}&&7\\
&&& & & &\langle9,-24,5\rangle&\smt{22&-5\\9&-2}^{3}&&10\\
\midrule
22&437&1& 1 &  C_{1}  & C_{2}\times{}C_{120} &\langle1,-21,1\rangle&\smt{21&-1\\1&0}^{3}&&4\\
\midrule
23&120&1& 2 & C_{2}  & C_{176} &\langle6,-12,1\rangle&\smt{23&-2\\12&-1}^{3}&&7\\
&&& & & &\langle3,-12,2\rangle&\smt{23&-4\\6&-1}^{3}&&7\\
&&2& 4 & C_{2}^{2}  & C_{176} &\langle1,-22,1\rangle&\smt{22&-1\\1&0}^{3}&&4\\
&&& & & &\langle19,-10,-5\rangle&\smt{16&5\\19&6}^{3}&&13\\
&&& & & &\langle8,-24,3\rangle&\smt{23&-3\\8&-1}^{3}&&7\\
&&& & & &\langle15,0,-8\rangle&\smt{11&8\\15&11}^{3}&&10\\
\midrule
24&21&1& 1 &  C_{1}  & C_{2}\times{}C_{4}\times{}C_{12} &\langle1,-5,1\rangle&\smt{5&-1\\1&0}^{6}&&7\\
&&5& 2 & C_{2}  & C_{2}^{2}\times{}C_{4}\times{}C_{12} &\langle1,-23,1\rangle&\smt{23&-1\\1&0}^{3}&&4\\
&&& & & &\langle17,-27,3\rangle&\smt{25&-3\\17&-2}^{3}&&10\\
\midrule
25&572&1& 2 & C_{2}  & C_{5}\times{}C_{40} &\langle1,-24,1\rangle&\smt{24&-1\\1&0}^{3}&&4\\
&&& & & &\langle13,-26,2\rangle&\smt{25&-2\\13&-1}^{3}&&7\\
\midrule
26&69&1& 1 &  C_{1}  & C_{2}\times{}C_{12}^{2} &\langle3,-9,1\rangle&\smt{26&-3\\9&-1}^{3}&&7\\
&&3& 3 & C_{3}  & C_{2}\times{}C_{12}^{2} &\langle1,-25,1\rangle&\smt{25&-1\\1&0}^{3}&&4\\
&&& & & &\langle17,-3,-9\rangle&\smt{14&9\\17&11}^{3}&&10\\
&&& & & &\langle5,-29,11\rangle&\smt{27&-11\\5&-2}^{3}&&10\\
\midrule
27&168&1& 2 & C_{2}  & C_{9}\times{}C_{18} &\langle7,-14,1\rangle&\smt{27&-2\\14&-1}^{3}&&7\\
&&& & & &\langle11,-16,2\rangle&\smt{29&-4\\22&-3}^{3}&&13\\
&&2& 4 & C_{2}^{2}  & C_{9}\times{}C_{18} &\langle1,-26,1\rangle&\smt{26&-1\\1&0}^{3}&&4\\
&&& & & &\langle7,14,-17\rangle&\smt{6&17\\7&20}^{3}&&7\\
&&& & & &\langle19,-8,-8\rangle&\smt{17&8\\19&9}^{3}&&10\\
&&& & & &\langle11,-32,8\rangle&\smt{29&-8\\11&-3}^{3}&&10\\
\midrule
28&29&1& 1 &  C_{1}  & C_{2}^{2}\times{}C_{6}^{2} &\langle5,-7,1\rangle&\smt{6&-1\\5&-1}^{6}&\text{Y}&16\\
&&5& 2 & C_{2}  & C_{2}^{3}\times{}C_{6}^{2} &\langle1,-27,1\rangle&\smt{27&-1\\1&0}^{3}&&4\\
&&& & & &\langle13,-33,7\rangle&\smt{30&-7\\13&-3}^{3}&&13\\
\midrule
29&780&1& 4 & C_{2}^{2}  & C_{280} &\langle1,-28,1\rangle&\smt{28&-1\\1&0}^{3}&&4\\
&&& & & &\langle15,-30,2\rangle&\smt{29&-2\\15&-1}^{3}&&7\\
&&& & & &\langle10,-30,3\rangle&\smt{29&-3\\10&-1}^{3}&&7\\
&&& & & &\langle6,-30,5\rangle&\smt{29&-5\\6&-1}^{3}&&7\\
\midrule
30&93&1& 1 &  C_{1}  &  C_{2}\times{}C_{6}\times{}C_{24} &\langle7,-11,1\rangle&\smt{31&-3\\21&-2}^{3}&&10\\
&&3& 3 & C_{3}  & C_{2}\times{}C_{6}\times{}C_{24} &\langle1,-29,1\rangle&\smt{29&-1\\1&0}^{3}&&4\\
&&& & & &\langle19,-1,-11\rangle&\smt{15&11\\19&14}^{3}&&10\\
&&& & & &\langle7,-33,9\rangle&\smt{31&-9\\7&-2}^{3}&&10\\
\midrule
31&56&1& 1 &  C_{1}  & C_{10}\times{}C_{30} &\langle2,-8,1\rangle&\smt{31&-4\\8&-1}^{3}&&7\\
&&2& 2 & C_{2}  & C_{10}\times{}C_{30} &\langle8,-16,1\rangle&\smt{31&-2\\16&-1}^{3}&&7\\
&&& & & &\langle11,2,-5\rangle&\smt{13&10\\22&17}^{3}&&13\\
&&4& 4 & C_{4}  & C_{10}\times{}C_{30} &\langle1,-30,1\rangle&\smt{30&-1\\1&0}^{3}&&4\\
&&& & & &\langle13,-34,5\rangle&\smt{32&-5\\13&-2}^{3}&&10\\
&&& & & &\langle25,-36,4\rangle&\smt{33&-4\\25&-3}^{3}&&13\\
&&& & & &\langle5,-34,13\rangle&\smt{32&-13\\5&-2}^{3}&&10\\
\midrule
32&957&1& 2 & C_{2}  & C_{2}\times{}C_{16}^{2} &\langle1,-31,1\rangle&\smt{31&-1\\1&0}^{3}&&4\\
&&& & & &\langle11,-33,3\rangle&\smt{32&-3\\11&-1}^{3}&&7\\
\midrule
33&1020&1& 4 & C_{2}^{2}  & C_{2}\times{}C_{120} &\langle1,-32,1\rangle&\smt{32&-1\\1&0}^{3}&&4\\
&&& & & &\langle17,-34,2\rangle&\smt{33&-2\\17&-1}^{3}&&7\\
&&& & & &\langle23,-36,3\rangle&\smt{34&-3\\23&-2}^{3}&&10\\
&&& & & &\langle29,-40,5\rangle&\smt{36&-5\\29&-4}^{3}&&16\\
\midrule
34&1085&1& 2 & C_{2}  & C_{2}\times{}C_{288} &\langle1,-33,1\rangle&\smt{33&-1\\1&0}^{3}&&4\\
&&& & & &\langle7,-35,5\rangle&\smt{34&-5\\7&-1}^{3}&&7\\
\midrule
35&8&1& 1 &  C_{1}  & C_{6}\times{}C_{12} &\langle2,-4,1\rangle&\smt{3&-1\\2&-1}^{12}&\text{Y}&13\\
&&2& 1 &  C_{1}  & C_{2}\times{}C_{6}\times{}C_{12} &\langle1,-6,1\rangle&\smt{6&-1\\1&0}^{6}&&7\\
&&3& 1 &  C_{1}  & C_{2}\times{}C_{6}\times{}C_{24} &\langle7,-10,1\rangle&\smt{37&-4\\28&-3}^{3}&&13\\
&&4& 1 &  C_{1}  & C_{2}\times{}C_{6}\times{}C_{24} &\langle4,-12,1\rangle&\smt{35&-3\\12&-1}^{3}&&7\\
&&6& 2 & C_{2}  & C_{2}\times{}C_{6}\times{}C_{24} &\langle9,-18,1\rangle&\smt{35&-2\\18&-1}^{3}&&7\\
&&& & & &\langle4,-20,7\rangle&\smt{37&-14\\8&-3}^{3}&&10\\
&&12& 4 & C_{4}  & C_{2}\times{}C_{6}\times{}C_{24} &\langle1,-34,1\rangle&\smt{34&-1\\1&0}^{3}&&4\\
&&& & & &\langle16,8,-17\rangle&\smt{13&17\\16&21}^{3}&&13\\
&&& & & &\langle4,-36,9\rangle&\smt{35&-9\\4&-1}^{3}&&7\\
&&& & & &\langle28,-12,-9\rangle&\smt{23&9\\28&11}^{3}&&13\\
\midrule
36&1221&1& 4 & C_{4}  & C_{2}\times{}C_{6}^{3} &\langle1,-35,1\rangle&\smt{35&-1\\1&0}^{3}&&4\\
&&& & & &\langle15,-39,5\rangle&\smt{37&-5\\15&-2}^{3}&&10\\
&&& & & &\langle25,-11,-11\rangle&\smt{23&11\\25&12}^{3}&&10\\
&&& & & &\langle5,-39,15\rangle&\smt{37&-15\\5&-2}^{3}&&10\\
\midrule
37&1292&1& 4 & C_{4}  & C_{12}\times{}C_{36} &\langle1,-36,1\rangle&\smt{36&-1\\1&0}^{3}&&4\\
&&& & & &\langle11,-40,7\rangle&\smt{38&-7\\11&-2}^{3}&&10\\
&&& & & &\langle19,0,-17\rangle&\smt{18&17\\19&18}^{3}&&7\\
&&& & & &\langle23,2,-14\rangle&\smt{17&14\\23&19}^{3}&&10\\
\midrule
38&1365&1& 4 & C_{2}^{2}  & C_{2}\times{}C_{18}^{2} &\langle1,-37,1\rangle&\smt{37&-1\\1&0}^{3}&&4\\
&&& & & &\langle13,-39,3\rangle&\smt{38&-3\\13&-1}^{3}&&7\\
&&& & & &\langle33,-45,5\rangle&\smt{41&-5\\33&-4}^{3}&&16\\
&&& & & &\langle11,-43,11\rangle&\smt{40&-11\\11&-3}^{3}&&10\\
\midrule
39&40&1& 2 & C_{2}  &  C_{2}\times{}C_{4}\times{}C_{12} &\langle6,-8,1\rangle&\smt{7&-1\\6&-1}^{6}&\text{Y}&19\\
&&& & & &\langle3,-8,2\rangle&\smt{7&-2\\3&-1}^{6}&\text{Y}&13\\
&&2& 2 & C_{2}  &  C_{2}^{2}\times{}C_{4}\times{}C_{12} &\langle9,-14,1\rangle&\smt{40&-3\\27&-2}^{3}&&10\\
&&& & & &\langle3,-14,3\rangle&\smt{40&-9\\9&-2}^{3}&&10\\
&&3& 2 & C_{2}  & C_{2}\times{}C_{12}^{2} &\langle10,-20,1\rangle&\smt{39&-2\\20&-1}^{3}&&7\\
&&& & & &\langle2,-20,5\rangle&\smt{39&-10\\4&-1}^{3}&&7\\
&&6& 4 & C_{2}^{2}  & C_{2}\times{}C_{12}^{2} &\langle1,-38,1\rangle&\smt{38&-1\\1&0}^{3}&&4\\
&&& & & &\langle31,-18,-9\rangle&\smt{28&9\\31&10}^{3}&&13\\
&&& & & &\langle8,-48,27\rangle&\smt{43&-27\\8&-5}^{3}&&13\\
&&& & & &\langle8,24,-27\rangle&\smt{7&27\\8&31}^{3}&&7\\
\midrule
40&1517&1& 2 & C_{2}  & C_{2}\times{}C_{4}^{2}\times{}C_{24} &\langle1,-39,1\rangle&\smt{39&-1\\1&0}^{3}&&4\\
&&& & & &\langle17,-49,13\rangle&\smt{44&-13\\17&-5}^{3}&&13\\
\midrule
41&1596&1& 8 & C_{2}\times{}C_{4}  & C_{560} &\langle1,-40,1\rangle&\smt{40&-1\\1&0}^{3}&&4\\
&&& & & &\langle5,-44,17\rangle&\smt{42&-17\\5&-2}^{3}&&10\\
&&& & & &\langle14,-42,3\rangle&\smt{41&-3\\14&-1}^{3}&&7\\
&&& & & &\langle17,-44,5\rangle&\smt{42&-5\\17&-2}^{3}&&10\\
&&& & & &\langle6,-42,7\rangle&\smt{41&-7\\6&-1}^{3}&&7\\
&&& & & &\langle34,-10,-11\rangle&\smt{25&11\\34&15}^{3}&&13\\
&&& & & &\langle21,0,-19\rangle&\smt{20&19\\21&20}^{3}&&7\\
&&& & & &\langle13,-46,10\rangle&\smt{43&-10\\13&-3}^{3}&&13\\
\midrule
42&1677&1& 4 & C_{4}  & C_{2}\times{}C_{6}^{3} &\langle1,-41,1\rangle&\smt{41&-1\\1&0}^{3}&&4\\
&&& & & &\langle19,-47,7\rangle&\smt{44&-7\\19&-3}^{3}&&13\\
&&& & & &\langle29,-13,-13\rangle&\smt{27&13\\29&14}^{3}&&10\\
&&& & & &\langle7,-47,19\rangle&\smt{44&-19\\7&-3}^{3}&&13\\
\midrule
43&440&1& 2 & C_{2}  & C_{14}\times{}C_{42} &\langle11,-22,1\rangle&\smt{43&-2\\22&-1}^{3}&&7\\
&&& & & &\langle17,-24,2\rangle&\smt{45&-4\\34&-3}^{3}&&13\\
&&2& 4 & C_{2}^{2}  & C_{14}\times{}C_{42} &\langle1,-42,1\rangle&\smt{42&-1\\1&0}^{3}&&4\\
&&& & & &\langle11,22,-29\rangle&\smt{10&29\\11&32}^{3}&&7\\
&&& & & &\langle37,-24,-8\rangle&\smt{33&8\\37&9}^{3}&&16\\
&&& & & &\langle17,-48,8\rangle&\smt{45&-8\\17&-3}^{3}&&10\\
\midrule
44&205&1& 2 & C_{2}  & C_{2}^{3}\times{}C_{120} &\langle5,-15,1\rangle&\smt{44&-3\\15&-1}^{3}&&7\\
&&& & & &\langle7,-17,3\rangle&\smt{47&-9\\21&-4}^{3}&&16\\
&&3& 4 & C_{4}  & C_{2}^{3}\times{}C_{120} &\langle1,-43,1\rangle&\smt{43&-1\\1&0}^{3}&&4\\
&&& & & &\langle7,-47,13\rangle&\smt{45&-13\\7&-2}^{3}&&10\\
&&& & & &\langle9,27,-31\rangle&\smt{8&31\\9&35}^{3}&&7\\
&&& & & &\langle27,-3,-17\rangle&\smt{23&17\\27&20}^{3}&&10\\
\midrule
45&1932&1& 4 & C_{2}^{2}  & C_{3}\times{}C_{6}\times{}C_{24} &\langle1,-44,1\rangle&\smt{44&-1\\1&0}^{3}&&4\\
&&& & & &\langle23,-46,2\rangle&\smt{45&-2\\23&-1}^{3}&&7\\
&&& & & &\langle31,-48,3\rangle&\smt{46&-3\\31&-2}^{3}&&10\\
&&& & & &\langle41,-54,6\rangle&\smt{49&-6\\41&-5}^{3}&&19\\
\midrule
46&2021&1& 3 & C_{3}  & C_{2}\times{}C_{528} &\langle1,-45,1\rangle&\smt{45&-1\\1&0}^{3}&&4\\
&&& & & &\langle19,-49,5\rangle&\smt{47&-5\\19&-2}^{3}&&10\\
&&& & & &\langle5,-49,19\rangle&\smt{47&-19\\5&-2}^{3}&&10\\
\midrule
47&33&1& 1 &  C_{1}  & C_{736} &\langle4,-7,1\rangle&\smt{51&-8\\32&-5}^{3}&&13\\
&&2& 1 &  C_{1}  & C_{736} &\langle3,-12,1\rangle&\smt{47&-4\\12&-1}^{3}&&7\\
&&4& 2 & C_{2}  & C_{736} &\langle12,-24,1\rangle&\smt{47&-2\\24&-1}^{3}&&7\\
&&& & & &\langle4,16,-17\rangle&\smt{7&34\\8&39}^{3}&&7\\
&&8& 4 & C_{2}^{2}  & C_{736} &\langle1,-46,1\rangle&\smt{46&-1\\1&0}^{3}&&4\\
&&& & & &\langle37,-22,-11\rangle&\smt{34&11\\37&12}^{3}&&13\\
&&& & & &\langle3,-48,16\rangle&\smt{47&-16\\3&-1}^{3}&&7\\
&&& & & &\langle31,2,-17\rangle&\smt{22&17\\31&24}^{3}&&13\\
\midrule
48&5&1& 1 &  C_{1}  &  C_{2}\times{}C_{8}^{2} &\langle1,-3,1\rangle&\smt{2&-1\\1&-1}^{24}&\text{Y}&13\\
&&3& 1 &  C_{1}  & C_{2}\times{}C_{8}\times{}C_{24} &\langle1,-7,1\rangle&\smt{7&-1\\1&0}^{6}&&7\\
&&7& 1 &  C_{1}  &  C_{2}\times{}C_{8}^{2} &\langle11,-17,1\rangle&\smt{49&-3\\33&-2}^{3}&&10\\
&&21& 4 & C_{4}  & C_{2}^{2}\times{}C_{8}\times{}C_{24} &\langle1,-47,1\rangle&\smt{47&-1\\1&0}^{3}&&4\\
&&& & & &\langle11,-51,9\rangle&\smt{49&-9\\11&-2}^{3}&&10\\
&&& & & &\langle41,-55,5\rangle&\smt{51&-5\\41&-4}^{3}&&16\\
&&& & & &\langle9,-51,11\rangle&\smt{49&-11\\9&-2}^{3}&&10\\
\midrule
49&92&1& 1 &  C_{1}  & C_{14}\times{}C_{42} &\langle2,-10,1\rangle&\smt{49&-5\\10&-1}^{3}&&7\\
&&5& 6 & C_{6}  & C_{14}\times{}C_{42} &\langle1,-48,1\rangle&\smt{48&-1\\1&0}^{3}&&4\\
&&& & & &\langle7,-54,22\rangle&\smt{51&-22\\7&-3}^{3}&&13\\
&&& & & &\langle11,-54,14\rangle&\smt{51&-14\\11&-3}^{3}&&10\\
&&& & & &\langle25,-50,2\rangle&\smt{49&-2\\25&-1}^{3}&&7\\
&&& & & &\langle14,-54,11\rangle&\smt{51&-11\\14&-3}^{3}&&10\\
&&& & & &\langle22,-54,7\rangle&\smt{51&-7\\22&-3}^{3}&&13\\
\midrule
50&2397&1& 2 & C_{2}  & C_{10}\times{}C_{120} &\langle1,-49,1\rangle&\smt{49&-1\\1&0}^{3}&&4\\
&&& & & &\langle17,-51,3\rangle&\smt{50&-3\\17&-1}^{3}&&7\\
\midrule
51&156&1& 2 & C_{2}  & C_{2}\times{}C_{288} &\langle10,-14,1\rangle&\smt{53&-4\\40&-3}^{3}&&13\\
&&& & & &\langle5,-14,2\rangle&\smt{53&-8\\20&-3}^{3}&&10\\
&&2& 4 & C_{4}  & C_{2}\times{}C_{288} &\langle13,-26,1\rangle&\smt{51&-2\\26&-1}^{3}&&7\\
&&& & & &\langle5,-28,8\rangle&\smt{53&-16\\10&-3}^{3}&&13\\
&&& & & &\langle23,-16,-4\rangle&\smt{41&8\\46&9}^{3}&&19\\
&&& & & &\langle20,-8,-7\rangle&\smt{33&14\\40&17}^{3}&&13\\
&&4& 8 & C_{2}\times{}C_{4}  & C_{2}\times{}C_{288} &\langle1,-50,1\rangle&\smt{50&-1\\1&0}^{3}&&4\\
&&& & & &\langle15,-54,7\rangle&\smt{52&-7\\15&-2}^{3}&&10\\
&&& & & &\langle12,-60,23\rangle&\smt{55&-23\\12&-5}^{3}&&13\\
&&& & & &\langle7,-54,15\rangle&\smt{52&-15\\7&-2}^{3}&&10\\
&&& & & &\langle13,26,-35\rangle&\smt{12&35\\13&38}^{3}&&7\\
&&& & & &\langle21,-54,5\rangle&\smt{52&-5\\21&-2}^{3}&&10\\
&&& & & &\langle35,-16,-16\rangle&\smt{33&16\\35&17}^{3}&&10\\
&&& & & &\langle5,-54,21\rangle&\smt{52&-21\\5&-2}^{3}&&10\\
\midrule
52&53&1& 1 &  C_{1}  & C_{2}^{2}\times{}C_{12}^{2} &\langle7,-9,1\rangle&\smt{8&-1\\7&-1}^{6}&\text{Y}&22\\
&&7& 3 & C_{3}  & C_{2}^{3}\times{}C_{12}^{2} &\langle1,-51,1\rangle&\smt{51&-1\\1&0}^{3}&&4\\
&&& & & &\langle17,-59,13\rangle&\smt{55&-13\\17&-4}^{3}&&16\\
&&& & & &\langle13,-59,17\rangle&\smt{55&-17\\13&-4}^{3}&&16\\
\midrule
53&12&1& 1 &  C_{1}  & C_{312} &\langle1,-4,1\rangle&\smt{4&-1\\1&0}^{9}&&10\\
&&3& 1 &  C_{1}  & C_{936} &\langle9,-12,1\rangle&\smt{56&-5\\45&-4}^{3}&&16\\
&&5& 2 & C_{2}  & C_{936} &\langle6,-18,1\rangle&\smt{53&-3\\18&-1}^{3}&&7\\
&&& & & &\langle3,-18,2\rangle&\smt{53&-6\\9&-1}^{3}&&7\\
&&15& 6 & C_{6}  & C_{936} &\langle1,-52,1\rangle&\smt{52&-1\\1&0}^{3}&&4\\
&&& & & &\langle22,-62,13\rangle&\smt{57&-13\\22&-5}^{3}&&13\\
&&& & & &\langle9,-60,25\rangle&\smt{56&-25\\9&-4}^{3}&&16\\
&&& & & &\langle27,-54,2\rangle&\smt{53&-2\\27&-1}^{3}&&7\\
&&& & & &\langle25,-60,9\rangle&\smt{56&-9\\25&-4}^{3}&&16\\
&&& & & &\langle13,-62,22\rangle&\smt{57&-22\\13&-5}^{3}&&13\\
\midrule
54&2805&1& 4 & C_{2}^{2}  & C_{3}\times{}C_{18}^{2} &\langle1,-53,1\rangle&\smt{53&-1\\1&0}^{3}&&4\\
&&& & & &\langle37,-57,3\rangle&\smt{55&-3\\37&-2}^{3}&&10\\
&&& & & &\langle11,-55,5\rangle&\smt{54&-5\\11&-1}^{3}&&7\\
&&& & & &\langle13,-59,13\rangle&\smt{56&-13\\13&-3}^{3}&&13\\
\midrule
55&728&1& 2 & C_{2}  & C_{8}\times{}C_{120} &\langle14,-28,1\rangle&\smt{55&-2\\28&-1}^{3}&&7\\
&&& & & &\langle7,-28,2\rangle&\smt{55&-4\\14&-1}^{3}&&7\\
&&2& 4 & C_{2}^{2}  & C_{8}\times{}C_{120} &\langle1,-54,1\rangle&\smt{54&-1\\1&0}^{3}&&4\\
&&& & & &\langle43,-26,-13\rangle&\smt{40&13\\43&14}^{3}&&13\\
&&& & & &\langle8,-64,37\rangle&\smt{59&-37\\8&-5}^{3}&&13\\
&&& & & &\langle8,-56,7\rangle&\smt{55&-7\\8&-1}^{3}&&7\\
\midrule
56&3021&1& 6 & C_{6}  & C_{2}^{3}\times{}C_{12}^{2} &\langle1,-55,1\rangle&\smt{55&-1\\1&0}^{3}&&4\\
&&& & & &\langle23,-59,5\rangle&\smt{57&-5\\23&-2}^{3}&&10\\
&&& & & &\langle25,-61,7\rangle&\smt{58&-7\\25&-3}^{3}&&13\\
&&& & & &\langle19,19,-35\rangle&\smt{18&35\\19&37}^{3}&&7\\
&&& & & &\langle7,-61,25\rangle&\smt{58&-25\\7&-3}^{3}&&13\\
&&& & & &\langle5,-59,23\rangle&\smt{57&-23\\5&-2}^{3}&&10\\
\midrule
57&348&1& 2 & C_{2}  &  C_{6}^2\times{}C_{18} &\langle13,-20,1\rangle&\smt{58&-3\\39&-2}^{3}&&10\\
&&& & & &\langle17,-22,2\rangle&\smt{61&-6\\51&-5}^{3}&&19\\
&&3& 6 & C_{6}  & C_{2}\times{}C_{18}^{2} &\langle1,-56,1\rangle&\smt{56&-1\\1&0}^{3}&&4\\
&&& & & &\langle18,-66,17\rangle&\smt{61&-17\\18&-5}^{3}&&13\\
&&& & & &\langle9,-60,13\rangle&\smt{58&-13\\9&-2}^{3}&&10\\
&&& & & &\langle29,-58,2\rangle&\smt{57&-2\\29&-1}^{3}&&7\\
&&& & & &\langle13,-60,9\rangle&\smt{58&-9\\13&-2}^{3}&&10\\
&&& & & &\langle17,-66,18\rangle&\smt{61&-18\\17&-5}^{3}&&13\\
\midrule
58&3245&1& 4 & C_{4}  & C_{2}\times{}C_{840} &\langle1,-57,1\rangle&\smt{57&-1\\1&0}^{3}&&4\\
&&& & & &\langle7,-61,17\rangle&\smt{59&-17\\7&-2}^{3}&&10\\
&&& & & &\langle49,-33,-11\rangle&\smt{45&11\\49&12}^{3}&&16\\
&&& & & &\langle17,-61,7\rangle&\smt{59&-7\\17&-2}^{3}&&10\\
\midrule
59&840&1& 4 & C_{2}^{2}  & C_{1160} &\langle15,-30,1\rangle&\smt{59&-2\\30&-1}^{3}&&7\\
&&& & & &\langle23,-32,2\rangle&\smt{61&-4\\46&-3}^{3}&&13\\
&&& & & &\langle5,-30,3\rangle&\smt{59&-6\\10&-1}^{3}&&7\\
&&& & & &\langle19,2,-11\rangle&\smt{27&22\\38&31}^{3}&&13\\
&&2& 8 & C_{2}^{3}  & C_{1160} &\langle1,-58,1\rangle&\smt{58&-1\\1&0}^{3}&&4\\
&&& & & &\langle15,30,-41\rangle&\smt{14&41\\15&44}^{3}&&7\\
&&& & & &\langle23,18,-33\rangle&\smt{20&33\\23&38}^{3}&&10\\
&&& & & &\langle55,-70,7\rangle&\smt{64&-7\\55&-6}^{3}&&22\\
&&& & & &\langle20,20,-37\rangle&\smt{19&37\\20&39}^{3}&&7\\
&&& & & &\langle12,-60,5\rangle&\smt{59&-5\\12&-1}^{3}&&7\\
&&& & & &\langle11,-62,11\rangle&\smt{60&-11\\11&-2}^{3}&&10\\
&&& & & &\langle40,0,-21\rangle&\smt{29&21\\40&29}^{3}&&13\\
\midrule
60&3477&1& 4 & C_{4}  & C_{2}^{3}\times{}C_{6}\times{}C_{24} &\langle1,-59,1\rangle&\smt{59&-1\\1&0}^{3}&&4\\
&&& & & &\langle17,-65,11\rangle&\smt{62&-11\\17&-3}^{3}&&10\\
&&& & & &\langle41,-19,-19\rangle&\smt{39&19\\41&20}^{3}&&10\\
&&& & & &\langle11,-65,17\rangle&\smt{62&-17\\11&-3}^{3}&&10\\
\midrule
61&3596&1& 6 & C_{6}  & C_{20}\times{}C_{60} &\langle1,-60,1\rangle&\smt{60&-1\\1&0}^{3}&&4\\
&&& & & &\langle19,-66,10\rangle&\smt{63&-10\\19&-3}^{3}&&13\\
&&& & & &\langle25,-64,5\rangle&\smt{62&-5\\25&-2}^{3}&&10\\
&&& & & &\langle31,-62,2\rangle&\smt{61&-2\\31&-1}^{3}&&7\\
&&& & & &\langle5,-64,25\rangle&\smt{62&-25\\5&-2}^{3}&&10\\
&&& & & &\langle50,-14,-17\rangle&\smt{37&17\\50&23}^{3}&&13\\
\midrule
62&413&1& 1 &  C_{1}  & C_{2}\times{}C_{30}^{2} &\langle7,-21,1\rangle&\smt{62&-3\\21&-1}^{3}&&7\\
&&3& 4 & C_{4}  & C_{2}\times{}C_{30}^{2} &\langle1,-61,1\rangle&\smt{61&-1\\1&0}^{3}&&4\\
&&& & & &\langle9,51,-31\rangle&\smt{5&31\\9&56}^{3}&&16\\
&&& & & &\langle7,-63,9\rangle&\smt{62&-9\\7&-1}^{3}&&7\\
&&& & & &\langle29,11,-31\rangle&\smt{25&31\\29&36}^{3}&&16\\
\midrule
63&60&1& 2 & C_{2}  & C_{3}^{2}\times{}C_{6}^{2} &\langle1,-8,1\rangle&\smt{8&-1\\1&0}^{6}&&7\\
&&& & & &\langle5,-10,2\rangle&\smt{9&-2\\5&-1}^{6}&&13\\
&&2& 2 & C_{2}  & C_{3}\times{}C_{6}^{3} &\langle4,-16,1\rangle&\smt{63&-4\\16&-1}^{3}&&7\\
&&& & & &\langle3,-18,7\rangle&\smt{67&-28\\12&-5}^{3}&&13\\
&&4& 4 & C_{2}^{2}  & C_{3}\times{}C_{6}^{3} &\langle16,-32,1\rangle&\smt{63&-2\\32&-1}^{3}&&7\\
&&& & & &\langle28,-36,3\rangle&\smt{67&-6\\56&-5}^{3}&&19\\
&&& & & &\langle21,6,-11\rangle&\smt{25&22\\42&37}^{3}&&13\\
&&& & & &\langle12,12,-17\rangle&\smt{19&34\\24&43}^{3}&&10\\
&&8& 8 & C_{2}\times{}C_{4}  & C_{3}\times{}C_{6}^{3} &\langle1,-62,1\rangle&\smt{62&-1\\1&0}^{3}&&4\\
&&& & & &\langle44,-76,11\rangle&\smt{69&-11\\44&-7}^{3}&&16\\
&&& & & &\langle49,-30,-15\rangle&\smt{46&15\\49&16}^{3}&&13\\
&&& & & &\langle11,-76,44\rangle&\smt{69&-44\\11&-7}^{3}&&16\\
&&& & & &\langle43,-66,3\rangle&\smt{64&-3\\43&-2}^{3}&&10\\
&&& & & &\langle28,12,-33\rangle&\smt{25&33\\28&37}^{3}&&13\\
&&& & & &\langle53,-36,-12\rangle&\smt{49&12\\53&13}^{3}&&16\\
&&& & & &\langle28,-68,7\rangle&\smt{65&-7\\28&-3}^{3}&&13\\
\midrule
64&3965&1& 4 & C_{2}^{2}  & C_{2}\times{}C_{32}^{2} &\langle1,-63,1\rangle&\smt{63&-1\\1&0}^{3}&&4\\
&&& & & &\langle31,-73,11\rangle&\smt{68&-11\\31&-5}^{3}&&19\\
&&& & & &\langle13,-65,5\rangle&\smt{64&-5\\13&-1}^{3}&&7\\
&&& & & &\langle43,3,-23\rangle&\smt{30&23\\43&33}^{3}&&16\\
\midrule
65&4092&1& 8 & C_{2}\times{}C_{4}  & C_{4}\times{}C_{12}\times{}C_{24} &\langle1,-64,1\rangle&\smt{64&-1\\1&0}^{3}&&4\\
&&& & & &\langle21,-72,13\rangle&\smt{68&-13\\21&-4}^{3}&&16\\
&&& & & &\langle11,44,-49\rangle&\smt{10&49\\11&54}^{3}&&7\\
&&& & & &\langle13,-72,21\rangle&\smt{68&-21\\13&-4}^{3}&&16\\
&&& & & &\langle22,-66,3\rangle&\smt{65&-3\\22&-1}^{3}&&7\\
&&& & & &\langle39,-6,-26\rangle&\smt{35&26\\39&29}^{3}&&10\\
&&& & & &\langle33,0,-31\rangle&\smt{32&31\\33&32}^{3}&&7\\
&&& & & &\langle7,-68,19\rangle&\smt{66&-19\\7&-2}^{3}&&10\\
\midrule
66&469&1& 3 & C_{3}  &  C_{2}^{3}\times{}C_{120} &\langle15,-23,1\rangle&\smt{67&-3\\45&-2}^{3}&&10\\
&&& & & &\langle3,-23,5\rangle&\smt{67&-15\\9&-2}^{3}&&10\\
&&& & & &\langle5,-23,3\rangle&\smt{67&-9\\15&-2}^{3}&&10\\
&&3& 6 & C_{6}  & C_{2}\times{}C_{6}\times{}C_{120} &\langle1,-65,1\rangle&\smt{65&-1\\1&0}^{3}&&4\\
&&& & & &\langle27,-75,13\rangle&\smt{70&-13\\27&-5}^{3}&&13\\
&&& & & &\langle41,-11,-25\rangle&\smt{38&25\\41&27}^{3}&&10\\
&&& & & &\langle61,-77,7\rangle&\smt{71&-7\\61&-6}^{3}&&22\\
&&& & & &\langle5,-69,27\rangle&\smt{67&-27\\5&-2}^{3}&&10\\
&&& & & &\langle13,-75,27\rangle&\smt{70&-27\\13&-5}^{3}&&13\\
\midrule
67&17&1& 1 &  C_{1}  & C_{11}\times{}C_{66} &\langle2,-5,1\rangle&\smt{9&-2\\4&-1}^{6}&\text{Y}&16\\
&&2& 1 &  C_{1}  & C_{11}\times{}C_{66} &\langle8,-10,1\rangle&\smt{9&-1\\8&-1}^{6}&\text{Y}&25\\
&&4& 1 &  C_{1}  & C_{22}\times{}C_{66} &\langle13,-18,1\rangle&\smt{69&-4\\52&-3}^{3}&&13\\
&&8& 2 & C_{2}  & C_{22}\times{}C_{66} &\langle17,-34,1\rangle&\smt{67&-2\\34&-1}^{3}&&7\\
&&& & & &\langle13,10,-19\rangle&\smt{23&38\\26&43}^{3}&&10\\
&&16& 4 & C_{4}  & C_{22}\times{}C_{66} &\langle1,-66,1\rangle&\smt{66&-1\\1&0}^{3}&&4\\
&&& & & &\langle53,-18,-19\rangle&\smt{42&19\\53&24}^{3}&&13\\
&&& & & &\langle17,-68,4\rangle&\smt{67&-4\\17&-1}^{3}&&7\\
&&& & & &\langle52,-20,-19\rangle&\smt{43&19\\52&23}^{3}&&13\\
\midrule
68&4485&1& 4 & C_{2}^{2}  & C_{2}^{3}\times{}C_{288} &\langle1,-67,1\rangle&\smt{67&-1\\1&0}^{3}&&4\\
&&& & & &\langle23,-69,3\rangle&\smt{68&-3\\23&-1}^{3}&&7\\
&&& & & &\langle57,-75,5\rangle&\smt{71&-5\\57&-4}^{3}&&16\\
&&& & & &\langle15,-75,19\rangle&\smt{71&-19\\15&-4}^{3}&&10\\
\midrule
69&4620&1& 8 & C_{2}^{3}  & C_{2}\times{}C_{528} &\langle1,-68,1\rangle&\smt{68&-1\\1&0}^{3}&&4\\
&&& & & &\langle35,-70,2\rangle&\smt{69&-2\\35&-1}^{3}&&7\\
&&& & & &\langle47,-72,3\rangle&\smt{70&-3\\47&-2}^{3}&&10\\
&&& & & &\langle61,-78,6\rangle&\smt{73&-6\\61&-5}^{3}&&19\\
&&& & & &\langle14,-70,5\rangle&\smt{69&-5\\14&-1}^{3}&&7\\
&&& & & &\langle10,-70,7\rangle&\smt{69&-7\\10&-1}^{3}&&7\\
&&& & & &\langle17,-76,17\rangle&\smt{72&-17\\17&-4}^{3}&&16\\
&&& & & &\langle21,-84,29\rangle&\smt{76&-29\\21&-8}^{3}&&13\\
\midrule
70&4757&1& 5 & C_{5}  & C_{2}\times{}C_{6}^{2}\times{}C_{24} &\langle1,-69,1\rangle&\smt{69&-1\\1&0}^{3}&&4\\
&&& & & &\langle31,-75,7\rangle&\smt{72&-7\\31&-3}^{3}&&13\\
&&& & & &\langle41,-1,-29\rangle&\smt{35&29\\41&34}^{3}&&10\\
&&& & & &\langle11,-73,13\rangle&\smt{71&-13\\11&-2}^{3}&&10\\
&&& & & &\langle31,13,-37\rangle&\smt{28&37\\31&41}^{3}&&13\\
\midrule
71&136&1& 2 & C_{2}  & C_{1680} &\langle2,-12,1\rangle&\smt{71&-6\\12&-1}^{3}&&7\\
&&& & & &\langle5,-14,3\rangle&\smt{77&-18\\30&-7}^{3}&&16\\
&&2& 4 & C_{4}  & C_{1680} &\langle8,-24,1\rangle&\smt{71&-3\\24&-1}^{3}&&7\\
&&& & & &\langle3,-26,11\rangle&\smt{74&-33\\9&-4}^{3}&&16\\
&&& & & &\langle15,-28,4\rangle&\smt{77&-12\\45&-7}^{3}&&13\\
&&& & & &\langle11,-26,3\rangle&\smt{74&-9\\33&-4}^{3}&&16\\
&&3& 4 & C_{4}  & C_{1680} &\langle18,-36,1\rangle&\smt{71&-2\\36&-1}^{3}&&7\\
&&& & & &\langle11,-38,5\rangle&\smt{73&-10\\22&-3}^{3}&&13\\
&&& & & &\langle2,-36,9\rangle&\smt{71&-18\\4&-1}^{3}&&7\\
&&& & & &\langle5,-38,11\rangle&\smt{73&-22\\10&-3}^{3}&&13\\
&&6& 8 & C_{2}\times{}C_{4}  & C_{1680} &\langle1,-70,1\rangle&\smt{70&-1\\1&0}^{3}&&4\\
&&& & & &\langle5,-74,29\rangle&\smt{72&-29\\5&-2}^{3}&&10\\
&&& & & &\langle47,14,-25\rangle&\smt{28&25\\47&42}^{3}&&13\\
&&& & & &\langle44,-12,-27\rangle&\smt{41&27\\44&29}^{3}&&10\\
&&& & & &\langle55,-34,-17\rangle&\smt{52&17\\55&18}^{3}&&13\\
&&& & & &\langle11,-76,20\rangle&\smt{73&-20\\11&-3}^{3}&&10\\
&&& & & &\langle9,-72,8\rangle&\smt{71&-8\\9&-1}^{3}&&7\\
&&& & & &\langle20,-76,11\rangle&\smt{73&-11\\20&-3}^{3}&&10\\
\midrule
72&5037&1& 4 & C_{4}  & C_{2}\times{}C_{6}\times{}C_{12}^{2} &\langle1,-71,1\rangle&\smt{71&-1\\1&0}^{3}&&4\\
&&& & & &\langle21,-75,7\rangle&\smt{73&-7\\21&-2}^{3}&&10\\
&&& & & &\langle49,-23,-23\rangle&\smt{47&23\\49&24}^{3}&&10\\
&&& & & &\langle7,-75,21\rangle&\smt{73&-21\\7&-2}^{3}&&10\\
\midrule
73&5180&1& 4 & C_{2}^{2}  & C_{24}\times{}C_{72} &\langle1,-72,1\rangle&\smt{72&-1\\1&0}^{3}&&4\\
&&& & & &\langle37,-74,2\rangle&\smt{73&-2\\37&-1}^{3}&&7\\
&&& & & &\langle61,-80,5\rangle&\smt{76&-5\\61&-4}^{3}&&16\\
&&& & & &\langle67,-84,7\rangle&\smt{78&-7\\67&-6}^{3}&&22\\
\midrule
74&213&1& 1 &  C_{1}  & C_{2}\times{}C_{36}^{2} &\langle3,-15,1\rangle&\smt{74&-5\\15&-1}^{3}&&7\\
&&5& 6 & C_{6}  & C_{2}\times{}C_{36}^{2} &\langle1,-73,1\rangle&\smt{73&-1\\1&0}^{3}&&4\\
&&& & & &\langle53,-5,-25\rangle&\smt{39&25\\53&34}^{3}&&13\\
&&& & & &\langle51,-87,11\rangle&\smt{80&-11\\51&-7}^{3}&&16\\
&&& & & &\langle25,25,-47\rangle&\smt{24&47\\25&49}^{3}&&7\\
&&& & & &\langle11,-87,51\rangle&\smt{80&-51\\11&-7}^{3}&&16\\
&&& & & &\langle53,5,-25\rangle&\smt{34&25\\53&39}^{3}&&13\\
\midrule
75&152&1& 1 &  C_{1}  &  C_{40}^{2} &\langle11,-14,1\rangle&\smt{79&-6\\66&-5}^{3}&&19\\
&&2& 2 & C_{2}  &  C_{40}^{2} &\langle17,-26,1\rangle&\smt{76&-3\\51&-2}^{3}&&10\\
&&& & & &\langle19,0,-8\rangle&\smt{37&24\\57&37}^{3}&&13\\
&&3& 4 & C_{4}  & C_{10}\times{}C_{120} &\langle19,-38,1\rangle&\smt{75&-2\\38&-1}^{3}&&7\\
&&& & & &\langle26,4,-13\rangle&\smt{33&26\\52&41}^{3}&&13\\
&&& & & &\langle29,-40,2\rangle&\smt{77&-4\\58&-3}^{3}&&13\\
&&& & & &\langle26,-4,-13\rangle&\smt{41&26\\52&33}^{3}&&13\\
&&6& 8 & C_{2}\times{}C_{4}  & C_{10}\times{}C_{120} &\langle1,-74,1\rangle&\smt{74&-1\\1&0}^{3}&&4\\
&&& & & &\langle44,-4,-31\rangle&\smt{39&31\\44&35}^{3}&&10\\
&&& & & &\langle71,-88,8\rangle&\smt{81&-8\\71&-7}^{3}&&25\\
&&& & & &\langle44,4,-31\rangle&\smt{35&31\\44&39}^{3}&&10\\
&&& & & &\langle19,38,-53\rangle&\smt{18&53\\19&56}^{3}&&7\\
&&& & & &\langle11,62,-37\rangle&\smt{6&37\\11&68}^{3}&&19\\
&&& & & &\langle29,22,-43\rangle&\smt{26&43\\29&48}^{3}&&10\\
&&& & & &\langle11,-84,36\rangle&\smt{79&-36\\11&-5}^{3}&&19\\
\midrule
76&5621&1& 6 & C_{6}  & C_{2}^{3}\times{}C_{18}^{2} &\langle1,-75,1\rangle&\smt{75&-1\\1&0}^{3}&&4\\
&&& & & &\langle31,-79,5\rangle&\smt{77&-5\\31&-2}^{3}&&10\\
&&& & & &\langle23,-89,25\rangle&\smt{82&-25\\23&-7}^{3}&&16\\
&&& & & &\langle11,55,-59\rangle&\smt{10&59\\11&65}^{3}&&7\\
&&& & & &\langle61,3,-23\rangle&\smt{36&23\\61&39}^{3}&&16\\
&&& & & &\langle5,-79,31\rangle&\smt{77&-31\\5&-2}^{3}&&10\\
\midrule
77&5772&1& 8 & C_{2}\times{}C_{4}  & C_{2}\times{}C_{6}\times{}C_{120} &\langle1,-76,1\rangle&\smt{76&-1\\1&0}^{3}&&4\\
&&& & & &\langle7,-82,34\rangle&\smt{79&-34\\7&-3}^{3}&&13\\
&&& & & &\langle26,-78,3\rangle&\smt{77&-3\\26&-1}^{3}&&7\\
&&& & & &\langle34,-82,7\rangle&\smt{79&-7\\34&-3}^{3}&&13\\
&&& & & &\langle6,-78,13\rangle&\smt{77&-13\\6&-1}^{3}&&7\\
&&& & & &\langle29,28,-43\rangle&\smt{24&43\\29&52}^{3}&&10\\
&&& & & &\langle39,0,-37\rangle&\smt{38&37\\39&38}^{3}&&7\\
&&& & & &\langle14,-82,17\rangle&\smt{79&-17\\14&-3}^{3}&&10\\
\midrule
78&237&1& 1 &  C_{1}  & C_{2}\times{}C_{6}\times{}C_{12}^{2} &\langle13,-17,1\rangle&\smt{81&-5\\65&-4}^{3}&&16\\
&&5& 6 & C_{6}  & C_{2}\times{}C_{6}\times{}C_{12}^{2} &\langle1,-77,1\rangle&\smt{77&-1\\1&0}^{3}&&4\\
&&& & & &\langle13,59,-47\rangle&\smt{9&47\\13&68}^{3}&&16\\
&&& & & &\langle31,-91,19\rangle&\smt{84&-19\\31&-7}^{3}&&13\\
&&& & & &\langle53,-25,-25\rangle&\smt{51&25\\53&26}^{3}&&10\\
&&& & & &\langle19,-91,31\rangle&\smt{84&-31\\19&-7}^{3}&&13\\
&&& & & &\langle25,35,-47\rangle&\smt{21&47\\25&56}^{3}&&16\\
\midrule
79&380&1& 2 & C_{2}  & C_{26}\times{}C_{78} &\langle5,-20,1\rangle&\smt{79&-4\\20&-1}^{3}&&7\\
&&& & & &\langle13,-22,2\rangle&\smt{83&-8\\52&-5}^{3}&&13\\
&&2& 4 & C_{4}  & C_{26}\times{}C_{78} &\langle20,-40,1\rangle&\smt{79&-2\\40&-1}^{3}&&7\\
&&& & & &\langle8,-44,13\rangle&\smt{83&-26\\16&-5}^{3}&&19\\
&&& & & &\langle5,-40,4\rangle&\smt{79&-8\\10&-1}^{3}&&7\\
&&& & & &\langle13,-44,8\rangle&\smt{83&-16\\26&-5}^{3}&&19\\
&&4& 8 & C_{2}\times{}C_{4}  & C_{26}\times{}C_{78} &\langle1,-78,1\rangle&\smt{78&-1\\1&0}^{3}&&4\\
&&& & & &\langle13,-88,32\rangle&\smt{83&-32\\13&-5}^{3}&&13\\
&&& & & &\langle16,-96,49\rangle&\smt{87&-49\\16&-9}^{3}&&13\\
&&& & & &\langle32,-88,13\rangle&\smt{83&-13\\32&-5}^{3}&&13\\
&&& & & &\langle61,-38,-19\rangle&\smt{58&19\\61&20}^{3}&&13\\
&&& & & &\langle7,-82,23\rangle&\smt{80&-23\\7&-2}^{3}&&10\\
&&& & & &\langle16,48,-59\rangle&\smt{15&59\\16&63}^{3}&&7\\
&&& & & &\langle23,-82,7\rangle&\smt{80&-7\\23&-2}^{3}&&10\\
\midrule
80&77&1& 1 &  C_{1}  & C_{8}^{2}\times{}C_{24} &\langle1,-9,1\rangle&\smt{9&-1\\1&0}^{6}&&7\\
&&3& 2 & C_{2}  & C_{2}\times{}C_{8}^{2}\times{}C_{24} &\langle9,-27,1\rangle&\smt{80&-3\\27&-1}^{3}&&7\\
&&& & & &\langle19,3,-9\rangle&\smt{35&27\\57&44}^{3}&&16\\
&&9& 6 & C_{6}  & C_{2}\times{}C_{8}^{2}\times{}C_{24} &\langle1,-79,1\rangle&\smt{79&-1\\1&0}^{3}&&4\\
&&& & & &\langle19,-85,13\rangle&\smt{82&-13\\19&-3}^{3}&&13\\
&&& & & &\langle9,69,-41\rangle&\smt{5&41\\9&74}^{3}&&16\\
&&& & & &\langle73,-55,-11\rangle&\smt{67&11\\73&12}^{3}&&22\\
&&& & & &\langle9,-87,37\rangle&\smt{83&-37\\9&-4}^{3}&&16\\
&&& & & &\langle63,-21,-23\rangle&\smt{50&23\\63&29}^{3}&&13\\
\midrule
81&6396&1& 12 & C_{2}\times{}C_{6}  & C_{27}\times{}C_{54} &\langle1,-80,1\rangle&\smt{80&-1\\1&0}^{3}&&4\\
&&& & & &\langle33,-84,5\rangle&\smt{82&-5\\33&-2}^{3}&&10\\
&&& & & &\langle25,-86,10\rangle&\smt{83&-10\\25&-3}^{3}&&13\\
&&& & & &\langle41,-82,2\rangle&\smt{81&-2\\41&-1}^{3}&&7\\
&&& & & &\langle66,-18,-23\rangle&\smt{49&23\\66&31}^{3}&&13\\
&&& & & &\langle5,-84,33\rangle&\smt{82&-33\\5&-2}^{3}&&10\\
&&& & & &\langle55,-84,3\rangle&\smt{82&-3\\55&-2}^{3}&&10\\
&&& & & &\langle11,-84,15\rangle&\smt{82&-15\\11&-2}^{3}&&10\\
&&& & & &\langle53,-6,-30\rangle&\smt{43&30\\53&37}^{3}&&16\\
&&& & & &\langle71,-90,6\rangle&\smt{85&-6\\71&-5}^{3}&&19\\
&&& & & &\langle53,6,-30\rangle&\smt{37&30\\53&43}^{3}&&16\\
&&& & & &\langle15,-84,11\rangle&\smt{82&-11\\15&-2}^{3}&&10\\
\midrule
82&6557&1& 3 & C_{3}  & C_{2}\times{}C_{1680} &\langle1,-81,1\rangle&\smt{81&-1\\1&0}^{3}&&4\\
&&& & & &\langle11,-87,23\rangle&\smt{84&-23\\11&-3}^{3}&&10\\
&&& & & &\langle23,-87,11\rangle&\smt{84&-11\\23&-3}^{3}&&10\\
\midrule
83&105&1& 2 & C_{2}  & C_{2296} &\langle4,-11,1\rangle&\smt{85&-8\\32&-3}^{3}&&10\\
&&& & & &\langle2,-11,2\rangle&\smt{85&-16\\16&-3}^{3}&&13\\
&&2& 2 & C_{2}  & C_{2296} &\langle16,-22,1\rangle&\smt{85&-4\\64&-3}^{3}&&13\\
&&& & & &\langle3,-24,13\rangle&\smt{89&-52\\12&-7}^{3}&&13\\
&&4& 4 & C_{2}^{2}  & C_{2296} &\langle21,-42,1\rangle&\smt{83&-2\\42&-1}^{3}&&7\\
&&& & & &\langle39,-30,-5\rangle&\smt{71&10\\78&11}^{3}&&25\\
&&& & & &\langle3,-42,7\rangle&\smt{83&-14\\6&-1}^{3}&&7\\
&&& & & &\langle28,0,-15\rangle&\smt{41&30\\56&41}^{3}&&13\\
&&8& 8 & C_{2}^{3}  & C_{2296} &\langle1,-82,1\rangle&\smt{82&-1\\1&0}^{3}&&4\\
&&& & & &\langle21,42,-59\rangle&\smt{20&59\\21&62}^{3}&&7\\
&&& & & &\langle69,-48,-16\rangle&\smt{65&16\\69&17}^{3}&&16\\
&&& & & &\langle39,-96,16\rangle&\smt{89&-16\\39&-7}^{3}&&16\\
&&& & & &\langle28,-84,3\rangle&\smt{83&-3\\28&-1}^{3}&&7\\
&&& & & &\langle12,-84,7\rangle&\smt{83&-7\\12&-1}^{3}&&7\\
&&& & & &\langle23,48,-48\rangle&\smt{17&48\\23&65}^{3}&&10\\
&&& & & &\langle48,0,-35\rangle&\smt{41&35\\48&41}^{3}&&10\\
\midrule
84&85&1& 2 & C_{2}  &  C_{2}^{4}\times{}C_{6}^{2} &\langle9,-11,1\rangle&\smt{10&-1\\9&-1}^{6}&\text{Y}&28\\
&&& & & &\langle3,-11,3\rangle&\smt{10&-3\\3&-1}^{6}&\text{Y}&16\\
&&3& 2 & C_{2}  &  C_{2}^{3}\times{}C_{6}^{3} &\langle19,-29,1\rangle&\smt{85&-3\\57&-2}^{3}&&10\\
&&& & & &\langle27,-3,-7\rangle&\smt{46&21\\81&37}^{3}&&19\\
&&9& 6 & C_{6}  & C_{2}^{3}\times{}C_{6}^{3} &\langle1,-83,1\rangle&\smt{83&-1\\1&0}^{3}&&4\\
&&& & & &\langle37,-89,7\rangle&\smt{86&-7\\37&-3}^{3}&&13\\
&&& & & &\langle9,-87,19\rangle&\smt{85&-19\\9&-2}^{3}&&10\\
&&& & & &\langle17,51,-63\rangle&\smt{16&63\\17&67}^{3}&&7\\
&&& & & &\langle19,-87,9\rangle&\smt{85&-9\\19&-2}^{3}&&10\\
&&& & & &\langle7,-89,37\rangle&\smt{86&-37\\7&-3}^{3}&&13\\
\midrule
85&7052&1& 4 & C_{4}  & C_{8}\times{}C_{288} &\langle1,-84,1\rangle&\smt{84&-1\\1&0}^{3}&&4\\
&&& & & &\langle58,-98,11\rangle&\smt{91&-11\\58&-7}^{3}&&16\\
&&& & & &\langle43,0,-41\rangle&\smt{42&41\\43&42}^{3}&&7\\
&&& & & &\langle11,-98,58\rangle&\smt{91&-58\\11&-7}^{3}&&16\\
\midrule
86&7221&1& 10 & C_{10}  & C_{2}\times{}C_{42}^{2} &\langle1,-85,1\rangle&\smt{85&-1\\1&0}^{3}&&4\\
&&& & & &\langle21,51,-55\rangle&\smt{17&55\\21&68}^{3}&&16\\
&&& & & &\langle53,15,-33\rangle&\smt{35&33\\53&50}^{3}&&10\\
&&& & & &\langle41,13,-43\rangle&\smt{36&43\\41&49}^{3}&&19\\
&&& & & &\langle25,-89,7\rangle&\smt{87&-7\\25&-2}^{3}&&10\\
&&& & & &\langle29,-87,3\rangle&\smt{86&-3\\29&-1}^{3}&&7\\
&&& & & &\langle7,-89,25\rangle&\smt{87&-25\\7&-2}^{3}&&10\\
&&& & & &\langle41,-95,11\rangle&\smt{90&-11\\41&-5}^{3}&&19\\
&&& & & &\langle35,-89,5\rangle&\smt{87&-5\\35&-2}^{3}&&10\\
&&& & & &\langle17,59,-55\rangle&\smt{13&55\\17&72}^{3}&&16\\
\midrule
87&1848&1& 4 & C_{2}^{2}  & C_{2}\times{}C_{840} &\langle22,-44,1\rangle&\smt{87&-2\\44&-1}^{3}&&7\\
&&& & & &\langle11,-44,2\rangle&\smt{87&-4\\22&-1}^{3}&&7\\
&&& & & &\langle38,-48,3\rangle&\smt{91&-6\\76&-5}^{3}&&19\\
&&& & & &\langle19,-48,6\rangle&\smt{91&-12\\38&-5}^{3}&&13\\
&&2& 8 & C_{2}^{3}  & C_{2}\times{}C_{840} &\langle1,-86,1\rangle&\smt{86&-1\\1&0}^{3}&&4\\
&&& & & &\langle67,-42,-21\rangle&\smt{64&21\\67&22}^{3}&&13\\
&&& & & &\langle11,66,-69\rangle&\smt{10&69\\11&76}^{3}&&7\\
&&& & & &\langle57,-96,8\rangle&\smt{91&-8\\57&-5}^{3}&&13\\
&&& & & &\langle79,-60,-12\rangle&\smt{73&12\\79&13}^{3}&&22\\
&&& & & &\langle59,-90,3\rangle&\smt{88&-3\\59&-2}^{3}&&10\\
&&& & & &\langle19,-94,19\rangle&\smt{90&-19\\19&-4}^{3}&&10\\
&&& & & &\langle56,0,-33\rangle&\smt{43&33\\56&43}^{3}&&16\\
\midrule
88&7565&1& 4 & C_{2}^{2}  & C_{2}\times{}C_{4}^{2}\times{}C_{120} &\langle1,-87,1\rangle&\smt{87&-1\\1&0}^{3}&&4\\
&&& & & &\langle43,-99,13\rangle&\smt{93&-13\\43&-6}^{3}&&22\\
&&& & & &\langle73,-95,5\rangle&\smt{91&-5\\73&-4}^{3}&&16\\
&&& & & &\langle61,1,-31\rangle&\smt{43&31\\61&44}^{3}&&16\\
\midrule
89&860&1& 2 & C_{2}  & C_{2640} &\langle10,-30,1\rangle&\smt{89&-3\\30&-1}^{3}&&7\\
&&& & & &\langle5,-30,2\rangle&\smt{89&-6\\15&-1}^{3}&&7\\
&&3& 8 & C_{2}\times{}C_{4}  & C_{2640} &\langle1,-88,1\rangle&\smt{88&-1\\1&0}^{3}&&4\\
&&& & & &\langle9,78,-46\rangle&\smt{5&46\\9&83}^{3}&&16\\
&&& & & &\langle10,-90,9\rangle&\smt{89&-9\\10&-1}^{3}&&7\\
&&& & & &\langle41,14,-46\rangle&\smt{37&46\\41&51}^{3}&&16\\
&&& & & &\langle18,54,-67\rangle&\smt{17&67\\18&71}^{3}&&7\\
&&& & & &\langle37,-102,18\rangle&\smt{95&-18\\37&-7}^{3}&&16\\
&&& & & &\langle45,0,-43\rangle&\smt{44&43\\45&44}^{3}&&7\\
&&& & & &\langle18,-102,37\rangle&\smt{95&-37\\18&-7}^{3}&&16\\
\midrule
90&7917&1& 4 & C_{2}^{2}  & C_{3}\times{}C_{6}^{2}\times{}C_{24} &\langle1,-89,1\rangle&\smt{89&-1\\1&0}^{3}&&4\\
&&& & & &\langle61,-93,3\rangle&\smt{91&-3\\61&-2}^{3}&&10\\
&&& & & &\langle13,-91,7\rangle&\smt{90&-7\\13&-1}^{3}&&7\\
&&& & & &\langle37,-105,21\rangle&\smt{97&-21\\37&-8}^{3}&&13\\
\midrule
91&2024&1& 6 & C_{6}  & C_{2}\times{}C_{6}\times{}C_{12}^{2} &\langle23,-46,1\rangle&\smt{91&-2\\46&-1}^{3}&&7\\
&&& & & &\langle5,-48,14\rangle&\smt{93&-28\\10&-3}^{3}&&13\\
&&& & & &\langle7,-48,10\rangle&\smt{93&-20\\14&-3}^{3}&&10\\
&&& & & &\langle35,-48,2\rangle&\smt{93&-4\\70&-3}^{3}&&13\\
&&& & & &\langle10,-48,7\rangle&\smt{93&-14\\20&-3}^{3}&&10\\
&&& & & &\langle14,-48,5\rangle&\smt{93&-10\\28&-3}^{3}&&13\\
&&2& 12 & C_{2}\times{}C_{6}  & C_{2}\times{}C_{6}\times{}C_{12}^{2} &\langle1,-90,1\rangle&\smt{90&-1\\1&0}^{3}&&4\\
&&& & & &\langle5,-94,37\rangle&\smt{92&-37\\5&-2}^{3}&&10\\
&&& & & &\langle40,-96,7\rangle&\smt{93&-7\\40&-3}^{3}&&13\\
&&& & & &\langle8,-96,35\rangle&\smt{93&-35\\8&-3}^{3}&&10\\
&&& & & &\langle40,16,-49\rangle&\smt{37&49\\40&53}^{3}&&13\\
&&& & & &\langle56,-16,-35\rangle&\smt{53&35\\56&37}^{3}&&10\\
&&& & & &\langle85,-104,8\rangle&\smt{97&-8\\85&-7}^{3}&&25\\
&&& & & &\langle17,-100,28\rangle&\smt{95&-28\\17&-5}^{3}&&13\\
&&& & & &\langle29,-98,13\rangle&\smt{94&-13\\29&-4}^{3}&&16\\
&&& & & &\langle4,-92,23\rangle&\smt{91&-23\\4&-1}^{3}&&7\\
&&& & & &\langle29,40,-56\rangle&\smt{25&56\\29&65}^{3}&&16\\
&&& & & &\langle28,-100,17\rangle&\smt{95&-17\\28&-5}^{3}&&13\\
\midrule
92&8277&1& 6 & C_{6}  & C_{2}^{3}\times{}C_{528} &\langle1,-91,1\rangle&\smt{91&-1\\1&0}^{3}&&4\\
&&& & & &\langle13,-101,37\rangle&\smt{96&-37\\13&-5}^{3}&&13\\
&&& & & &\langle53,3,-39\rangle&\smt{44&39\\53&47}^{3}&&10\\
&&& & & &\langle31,31,-59\rangle&\smt{30&59\\31&61}^{3}&&7\\
&&& & & &\langle53,-3,-39\rangle&\smt{47&39\\53&44}^{3}&&10\\
&&& & & &\langle61,-15,-33\rangle&\smt{53&33\\61&38}^{3}&&13\\
\midrule
93&940&1& 6 & C_{6}  &  C_{2}^{2}\times{}C_{10}\times{}C_{30} &\langle21,-32,1\rangle&\smt{94&-3\\63&-2}^{3}&&10\\
&&& & & &\langle7,-32,3\rangle&\smt{94&-9\\21&-2}^{3}&&10\\
&&& & & &\langle9,-34,6\rangle&\smt{97&-18\\27&-5}^{3}&&13\\
&&& & & &\langle27,-20,-5\rangle&\smt{76&15\\81&16}^{3}&&19\\
&&& & & &\langle21,4,-11\rangle&\smt{40&33\\63&52}^{3}&&13\\
&&& & & &\langle3,-32,7\rangle&\smt{94&-21\\9&-2}^{3}&&10\\
&&3& 12 & C_{2}\times{}C_{6}  & C_{2}\times{}C_{30}^{2} &\langle1,-92,1\rangle&\smt{92&-1\\1&0}^{3}&&4\\
&&& & & &\langle55,-10,-38\rangle&\smt{51&38\\55&41}^{3}&&10\\
&&& & & &\langle26,-98,11\rangle&\smt{95&-11\\26&-3}^{3}&&10\\
&&& & & &\langle77,-100,5\rangle&\smt{96&-5\\77&-4}^{3}&&16\\
&&& & & &\langle11,-98,26\rangle&\smt{95&-26\\11&-3}^{3}&&10\\
&&& & & &\langle55,10,-38\rangle&\smt{41&38\\55&51}^{3}&&10\\
&&& & & &\langle47,-94,2\rangle&\smt{93&-2\\47&-1}^{3}&&7\\
&&& & & &\langle65,20,-31\rangle&\smt{36&31\\65&56}^{3}&&19\\
&&& & & &\langle73,-24,-27\rangle&\smt{58&27\\73&34}^{3}&&13\\
&&& & & &\langle89,-108,9\rangle&\smt{100&-9\\89&-8}^{3}&&28\\
&&& & & &\langle22,-98,13\rangle&\smt{95&-13\\22&-3}^{3}&&13\\
&&& & & &\langle14,-110,65\rangle&\smt{101&-65\\14&-9}^{3}&&19\\
\midrule
94&8645&1& 4 & C_{2}^{2}  & C_{2}\times{}C_{2208} &\langle1,-93,1\rangle&\smt{93&-1\\1&0}^{3}&&4\\
&&& & & &\langle85,-105,7\rangle&\smt{99&-7\\85&-6}^{3}&&22\\
&&& & & &\langle19,-95,5\rangle&\smt{94&-5\\19&-1}^{3}&&7\\
&&& & & &\langle17,-99,17\rangle&\smt{96&-17\\17&-3}^{3}&&10\\
\midrule
95&552&1& 2 & C_{2}  & C_{2}\times{}C_{18}\times{}C_{72} &\langle6,-24,1\rangle&\smt{95&-4\\24&-1}^{3}&&7\\
&&& & & &\langle3,-24,2\rangle&\smt{95&-8\\12&-1}^{3}&&7\\
&&2& 4 & C_{2}^{2}  & C_{2}\times{}C_{18}\times{}C_{72} &\langle24,-48,1\rangle&\smt{95&-2\\48&-1}^{3}&&7\\
&&& & & &\langle31,10,-17\rangle&\smt{37&34\\62&57}^{3}&&13\\
&&& & & &\langle29,2,-19\rangle&\smt{45&38\\58&49}^{3}&&13\\
&&& & & &\langle3,-48,8\rangle&\smt{95&-16\\6&-1}^{3}&&7\\
&&4& 8 & C_{2}\times{}C_{4}  & C_{2}\times{}C_{18}\times{}C_{72} &\langle1,-94,1\rangle&\smt{94&-1\\1&0}^{3}&&4\\
&&& & & &\langle16,72,-57\rangle&\smt{11&57\\16&83}^{3}&&19\\
&&& & & &\langle73,-100,4\rangle&\smt{97&-4\\73&-3}^{3}&&13\\
&&& & & &\langle31,42,-57\rangle&\smt{26&57\\31&68}^{3}&&19\\
&&& & & &\langle3,-96,32\rangle&\smt{95&-32\\3&-1}^{3}&&7\\
&&& & & &\langle76,-4,-29\rangle&\smt{49&29\\76&45}^{3}&&16\\
&&& & & &\langle12,-108,59\rangle&\smt{101&-59\\12&-7}^{3}&&13\\
&&& & & &\langle76,4,-29\rangle&\smt{45&29\\76&49}^{3}&&16\\
\midrule
96&9021&1& 8 & C_{8}  & C_{2}^{2}\times{}C_{16}\times{}C_{48} &\langle1,-95,1\rangle&\smt{95&-1\\1&0}^{3}&&4\\
&&& & & &\langle39,-99,5\rangle&\smt{97&-5\\39&-2}^{3}&&10\\
&&& & & &\langle65,-109,11\rangle&\smt{102&-11\\65&-7}^{3}&&16\\
&&& & & &\langle15,-99,13\rangle&\smt{97&-13\\15&-2}^{3}&&10\\
&&& & & &\langle65,-31,-31\rangle&\smt{63&31\\65&32}^{3}&&10\\
&&& & & &\langle55,1,-41\rangle&\smt{47&41\\55&48}^{3}&&10\\
&&& & & &\langle11,-109,65\rangle&\smt{102&-65\\11&-7}^{3}&&16\\
&&& & & &\langle5,-99,39\rangle&\smt{97&-39\\5&-2}^{3}&&10\\
\midrule
97&188&1& 1 &  C_{1}  & C_{32}\times{}C_{96} &\langle2,-14,1\rangle&\smt{97&-7\\14&-1}^{3}&&7\\
&&7& 8 & C_{8}  & C_{32}\times{}C_{96} &\langle1,-96,1\rangle&\smt{96&-1\\1&0}^{3}&&4\\
&&& & & &\langle11,-106,46\rangle&\smt{101&-46\\11&-5}^{3}&&19\\
&&& & & &\langle38,-110,19\rangle&\smt{103&-19\\38&-7}^{3}&&13\\
&&& & & &\langle23,-106,22\rangle&\smt{101&-22\\23&-5}^{3}&&13\\
&&& & & &\langle49,-98,2\rangle&\smt{97&-2\\49&-1}^{3}&&7\\
&&& & & &\langle22,-106,23\rangle&\smt{101&-23\\22&-5}^{3}&&13\\
&&& & & &\langle62,6,-37\rangle&\smt{45&37\\62&51}^{3}&&13\\
&&& & & &\langle11,84,-49\rangle&\smt{6&49\\11&90}^{3}&&19\\
\midrule
98&1045&1& 4 & C_{4}  & C_{2}\times{}C_{42}^{2} &\langle11,-33,1\rangle&\smt{98&-3\\33&-1}^{3}&&7\\
&&& & & &\langle3,-35,15\rangle&\smt{101&-45\\9&-4}^{3}&&16\\
&&& & & &\langle5,-35,9\rangle&\smt{101&-27\\15&-4}^{3}&&10\\
&&& & & &\langle3,29,-17\rangle&\smt{5&51\\9&92}^{3}&&16\\
&&3& 8 & C_{2}\times{}C_{4}  & C_{2}\times{}C_{42}^{2} &\langle1,-97,1\rangle&\smt{97&-1\\1&0}^{3}&&4\\
&&& & & &\langle43,-103,7\rangle&\smt{100&-7\\43&-3}^{3}&&13\\
&&& & & &\langle29,-113,29\rangle&\smt{105&-29\\29&-8}^{3}&&13\\
&&& & & &\langle43,17,-53\rangle&\smt{40&53\\43&57}^{3}&&13\\
&&& & & &\langle81,-57,-19\rangle&\smt{77&19\\81&20}^{3}&&16\\
&&& & & &\langle63,9,-37\rangle&\smt{44&37\\63&53}^{3}&&16\\
&&& & & &\langle11,77,-79\rangle&\smt{10&79\\11&87}^{3}&&7\\
&&& & & &\langle63,-9,-37\rangle&\smt{53&37\\63&44}^{3}&&16\\
\midrule
99&24&1& 1 &  C_{1}  & C_{3}^{2}\times{}C_{120} &\langle3,-6,1\rangle&\smt{11&-2\\6&-1}^{6}&&13\\
&&2& 2 & C_{2}  & C_{3}^{2}\times{}C_{120} &\langle1,-10,1\rangle&\smt{10&-1\\1&0}^{6}&&7\\
&&& & & &\langle3,-12,4\rangle&\smt{11&-4\\3&-1}^{6}&&13\\
&&4& 2 & C_{2}  & C_{3}\times{}C_{6}\times{}C_{120} &\langle4,-20,1\rangle&\smt{99&-5\\20&-1}^{3}&&7\\
&&& & & &\langle3,-24,16\rangle&\smt{109&-80\\15&-11}^{3}&&19\\
&&5& 2 & C_{2}  & C_{3}\times{}C_{6}\times{}C_{120} &\langle19,-26,1\rangle&\smt{101&-4\\76&-3}^{3}&&13\\
&&& & & &\langle23,-28,2\rangle&\smt{105&-8\\92&-7}^{3}&&25\\
&&10& 4 & C_{2}^{2}  & C_{3}\times{}C_{6}\times{}C_{120} &\langle25,-50,1\rangle&\smt{99&-2\\50&-1}^{3}&&7\\
&&& & & &\langle8,-56,23\rangle&\smt{105&-46\\16&-7}^{3}&&16\\
&&& & & &\langle19,14,-29\rangle&\smt{35&58\\38&63}^{3}&&10\\
&&& & & &\langle43,-54,3\rangle&\smt{103&-6\\86&-5}^{3}&&19\\
&&20& 8 & C_{2}\times{}C_{4}  & C_{3}\times{}C_{6}\times{}C_{120} &\langle1,-98,1\rangle&\smt{98&-1\\1&0}^{3}&&4\\
&&& & & &\langle47,24,-48\rangle&\smt{37&48\\47&61}^{3}&&16\\
&&& & & &\langle4,-100,25\rangle&\smt{99&-25\\4&-1}^{3}&&7\\
&&& & & &\langle32,48,-57\rangle&\smt{25&57\\32&73}^{3}&&16\\
&&& & & &\langle67,-102,3\rangle&\smt{100&-3\\67&-2}^{3}&&10\\
&&& & & &\langle19,-104,16\rangle&\smt{101&-16\\19&-3}^{3}&&13\\
&&& & & &\langle75,0,-32\rangle&\smt{49&32\\75&49}^{3}&&13\\
&&& & & &\langle16,-104,19\rangle&\smt{101&-19\\16&-3}^{3}&&13\\
\midrule
100&9797&1& 4 & C_{4}  & C_{2}^{2}\times{}C_{10}\times{}C_{120} &\langle1,-99,1\rangle&\smt{99&-1\\1&0}^{3}&&4\\
&&& & & &\langle29,-103,7\rangle&\smt{101&-7\\29&-2}^{3}&&10\\
&&& & & &\langle43,-123,31\rangle&\smt{111&-31\\43&-12}^{3}&&16\\
&&& & & &\langle7,-103,29\rangle&\smt{101&-29\\7&-2}^{3}&&10\\
\bottomrule
\end{longtable}
\end{center}

%% file: main_v2.bbl
\begin{thebibliography}{100}

\bibitem{Adler:1984}
R.~L. Adler and L.~Flatto.
\newblock The backward continued fraction map and geodesic flow.
\newblock {\em Ergodic Theory and Dynamical Systems}, 4:487--492, 1984.

\bibitem{Alperin:1993}
R.~C. Alperin.
\newblock $\rm{PSL}_2({Z})= {Z}_2* {Z}_3$.
\newblock {\em The American Mathematical Monthly}, 100:385--386, 1993.

\bibitem{Andersson:2019}
O.~Andersson and I.~Dumitru.
\newblock Aligned {SICs} and embedded tight frames in even dimensions.
\newblock {\em Journal of Physics A: Mathematical and Theoretical}, 52:425302,
  2019.

\bibitem{Appleby2005}
D.~M. Appleby.
\newblock Symmetric informationally complete positive-operator valued measures
  and the extended {C}lifford group.
\newblock {\em J. Math. Phys.}, 46:052107, 2005.

\bibitem{Appleby:2007a}
D.~M. Appleby.
\newblock Symmetric informationally complete measurements of arbitrary rank.
\newblock {\em Opt. Spect.}, 103:416--428, 2007.

\bibitem{Appleby:2009}
D.~M. Appleby.
\newblock Properties of the extended {C}lifford group with applications to
  {SIC-POVM}s and {MUB}s, 2009.
\newblock Available at \href{https://arxiv.org/abs/0909.5233}{arXiv:0909.5233}.

\bibitem{Appleby:2013}
D.~M. Appleby, H.~Yadsan-Appleby, and G.~Zauner.
\newblock {G}alois automorphisms of a symmetric measurement.
\newblock {\em Quant. Inf. Comput.}, 13:672--720, 2013.

\bibitem{Appleby:2017}
M.~Appleby, I.~Bengtsson, I.~Dumitru, and S.~Flammia.
\newblock Dimension towers of {SICs}. {I}. {A}ligned {SIC}s and embedded tight
  frames.
\newblock {\em J. Math. Phys.}, 58(11):112201, 2017.

\bibitem{Appleby:2019b}
M.~Appleby, I.~Bengtsson, S.~Flammia, and D.~Goyeneche.
\newblock Tight frames, {H}adamard matrices, and {Z}auner's conjecture.
\newblock {\em J. Phys. A}, 52:295301, 2019.

\bibitem{Appleby:2022}
M.~Appleby, I.~Bengtsson, M.~Grassl, M.~Harrison, and G.~McConnell.
\newblock {SIC-POVMs} from {S}tark units: Prime dimensions $n^2+3$.
\newblock {\em J. Math. Phys.}, 63:112205, 2022.

\bibitem{Appleby:2018}
M.~Appleby, T.-Y. Chien, S.~T. Flammia, and S.~Waldron.
\newblock Constructing exact symmetric informationally complete measurements
  from numerical solutions.
\newblock {\em J. Phys. A}, 51:165302, 2018.

\bibitem{Appleby:2022b}
M.~Appleby, S.~Flammia, and G.~Kopp.
\newblock Ghost {SICs} and the {W}igner function, 2025.
\newblock Forthcoming.

\bibitem{Appleby:2017a}
M.~Appleby, S.~Flammia, G.~McConnell, and J.~Yard.
\newblock {SIC}s and algebraic number theory.
\newblock {\em Found. Phys.}, 47:1042--1059, 2017.

\bibitem{Appleby:2020}
M.~Appleby, S.~Flammia, G.~McConnell, and J.~Yard.
\newblock Generating ray class fields of real quadratic fields via complex
  equiangular lines.
\newblock {\em Acta Arithmetica}, 192(3):211--233, 2020.

\bibitem{Appleby2022c}
M.~Appleby, S.~T. Flammia, and G.~S. Kopp.
\newblock Representations of {G}alois groups of {SIC}s, 2025.
\newblock Forthcoming.

\bibitem{Bengtsson:2024}
I.~Bengtsson, M.~Grassl, and G.~McConnell.
\newblock {SIC-POVMs} from {Stark} units: Dimensions $n^2+ 3= 4p$, $p$ prime,
  2024, 2403.02872.

\bibitem{Bengtsson:2022}
I.~Bengtsson and B.~Srivastava.
\newblock Dimension towers of {SICs}: {II}. some constructions.
\newblock {\em Journal of Physics A: Mathematical and Theoretical}, 55:215302,
  2022.

\bibitem{Bjorck1970}
{\ringabove{A}}.~Bj\"{o}ck and V.~Pereyra.
\newblock Solution of {V}andermonde systems of equations.
\newblock {\em Mathematics of Computation}, 24(112):893, Oct. 1970.

\bibitem{Bjorklund:2020}
C.~Bjorklund and M.~Litman.
\newblock Error approximation for backwards and simple continued fractions.
\newblock {\em Research in Number Theory}, 10:1--26, 2024.

\bibitem{Bos:2019}
L.~Bos and S.~Waldron.
\newblock {SIC}s and the elements of canonical order 3 in the {C}lifford group.
\newblock {\em J. Phys. A}, 52:105301, 2019.

\bibitem{Buchmann:2007}
J.~Buchmann and U.~Vollmer.
\newblock {\em Binary Quadratic Forms: An Algorithmic Approach}.
\newblock Algorithms and Computation in Mathematics, no. 20. Springer, 2007.

\bibitem{Buell1989}
D.~A. Buell.
\newblock {\em Binary Quadratic Forms: Classical Theory and Modern
  Computations}.
\newblock Springer-Verlag, 1989.

\bibitem{Casazza:2011}
P.~G. Casazza, M.~Fickus, D.~G. Mixon, Y.~Wang, and Z.~Zhou.
\newblock Constructing tight fusion frames.
\newblock {\em Applied and Computational Harmonic Analysis}, 30:175--187, 2011.

\bibitem{Chen:2015}
B.~Chen, T.~Li, and S.-M. Fei.
\newblock General {SIC} measurement-based entanglement detection.
\newblock {\em Quantum Information Processing}, 14:2281--2290, 2015.

\bibitem{Closset:2019}
C.~Closset and H.~Kim.
\newblock Three-dimensional $\mcl{N}=2$ supersymmetric gauge theories and
  partition functions on {S}eifert manifolds: A review.
\newblock {\em Int. J. Mod. Phys. A}, 34:1930011, 2019.

\bibitem{Cuffaro:2024}
G.~Cuffaro and C.~A. Fuchs.
\newblock Quantum states with maximal magic, 2024.
\newblock Available at
  \href{https://arxiv.org/abs/2412.21083}{arXiv:2412.21083}.

\bibitem{Dai:2022}
H.~Dai, S.~Fu, and S.~Luo.
\newblock Detecting magic states via characteristic functions.
\newblock {\em International Journal of Theoretical Physics}, 61:35, 2022.

\bibitem{Davenport2008}
J.~H. Davenport.
\newblock {\em The Higher Arithmetic: An Introduction to the Theory of
  Numbers}.
\newblock Cambridge University Press, eighth edition, 2008.

\bibitem{DeBrota:2020}
J.~B. DeBrota, C.~A. Fuchs, and B.~C. Stacey.
\newblock Symmetric informationally complete measurements identify the
  irreducible difference between classical and quantum.
\newblock {\em Phys. Rev. Res.}, 2:013074, 2020.

\bibitem{Delsarte:1975}
P.~Delsarte, J.~M. Goethals, and J.~J. Seidel.
\newblock Bounds for systems of lines, and {J}acobi polynomials.
\newblock {\em Philips Research Reports}, 30:91, 1975.

\bibitem{Diamond:2005}
F.~Diamond and J.~Shurman.
\newblock {\em A First Course in Modular Forms}.
\newblock Springer, 2005.

\bibitem{Dimofte:2015}
T.~Dimofte.
\newblock Complex {C}hern--{S}imons theory at level k via the 3d--3d
  correspondence.
\newblock {\em Commun. Math. Phys.}, 339:619--662, 2015.

\bibitem{Dimofte:2010}
T.~D. Dimofte.
\newblock {\em Refined BPS Invariants, {C}hern-{S}imons Theory, and the Quantum
  Dilogarithm}.
\newblock PhD thesis, California Institute of Technology, 2010.

\bibitem{Faddeev:1995}
L.~D. Faddeev.
\newblock Discrete {H}eisenberg-{W}eyl group and modular group.
\newblock {\em Lett. Math. Phys.}, 34:249--254, 1995.

\bibitem{Faddeev:1994}
L.~D. Faddeev and R.~M. Kashaev.
\newblock Quantum dilogarithm.
\newblock {\em Mod. Phys. Lett. A}, 9:427--434, 1994.

\bibitem{Faddeev:2001}
L.~D. Faddeev, R.~M. Kashaev, and A.~Y. Volkov.
\newblock Strongly coupled quantum discrete {L}iouville theory. {I}: Algebraic
  approach and duality.
\newblock {\em Commun. Math. Phys.}, 219:199--219, 2001.

\bibitem{Fannjiang:2020}
A.~Fannjiang and T.~Strohmer.
\newblock The numerics of phase retrieval.
\newblock {\em Acta Numerica}, 29:125--228, 2020.

\bibitem{Feng:2022}
L.~Feng and S.~Luo.
\newblock From stabilizer states to {SIC-POVM} fiducial states.
\newblock {\em Theor. and Math. Phys.}, 213:1747--1761, 2022.

\bibitem{Fickus:2017}
M.~Fickus, J.~Jasper, D.~G. Mixon, and C.~E. Watson.
\newblock A brief introduction to equi-chordal and equi-isoclinic tight fusion
  frames.
\newblock In {\em Wavelets and Sparsity XVII}, volume 10394, pages 186--194.
  SPIE, 2017.

\bibitem{Hecke}
C.~Fieker, W.~Hart, T.~Hofmann, and F.~Johansson.
\newblock Nemo/{H}ecke: Computer algebra and number theory packages for the
  {J}ulia programming language.
\newblock In {\em Proceedings of the 2017 ACM on International Symposium on
  Symbolic and Algebraic Computation}, ISSAC '17, pages 157--164, New York, NY,
  USA, 2017. ACM.

\bibitem{Finkelshtein:1993}
Y.~Y. Finkel'shtein.
\newblock Klein polygons and reduced regular continued fractions.
\newblock {\em Russ. Math. Surv.}, 48:198--200, 1993.

\bibitem{SteveGitHub}
S.~T. Flammia.
\newblock \texttt{Zauner.jl}.
\newblock available at
  \href{https://github.com/sflammia/Zauner.jl}{github.com/sflammia/Zauner.jl},
  2024.

\bibitem{Fouvry:2010}
E.~Fouvry and J.~Kl\"{u}ners.
\newblock On the negative {P}ell equation.
\newblock {\em Ann. Math.}, 172(3):2035--2104, 2010.

\bibitem{Fuchs:2017}
C.~A. Fuchs, M.~C. Hoang, and B.~C. Stacey.
\newblock The {SIC} question: {H}istory and state of play.
\newblock {\em Axioms}, 6:21, 2017.

\bibitem{Fuchs:2013}
C.~A. Fuchs and R.~Schack.
\newblock Quantum-{B}ayesian coherence.
\newblock {\em Rev. Mod. Phys.}, 85:1693--1715, 2013.

\bibitem{Garoufalidis:2021}
S.~Garoufalidis and R.~Kashaev.
\newblock Resurgence of {F}addeev's quantum dilogarithm.
\newblock In A.~Papadopoulos, editor, {\em Topology and Geometry}, pages
  257--272. European Mathematical Society, 2021.

\bibitem{Garoufalidis:2023}
S.~Garoufalidis and R.~Kashaev.
\newblock Quantum dilogarithms over local fields and invariants of 3-manifolds,
  2023.
\newblock Available at
  \href{https://arxiv.org/abs/2306.01331}{arXiv:2306.01331}.

\bibitem{Gour:2014}
G.~Gour and A.~Kalev.
\newblock Construction of all general symmetric informationally complete
  measurements.
\newblock {\em J. Phys. A}, 47:335302, 2014.

\bibitem{Grassl:2021a}
M.~Grassl.
\newblock Computing {SIC-POVMs} using permutation symmetries and {S}tark units,
  26 October 2021.
\newblock Available at the
  \href{https://www.youtube.com/watch?v=2vzS55SjaZI}{Codex seminar youtube
  channel}.

\bibitem{Grassl:2021}
M.~Grassl.
\newblock Computing numerical and exact {SIC-POVM}s, 29 March 2021.
\newblock Available at the
  \href{https://www.youtube.com/watch?v=CGNxSRcqWts}{Jagiellonian University
  seminar youtube channel}.

\bibitem{Graydon:2016a}
M.~A. Graydon and D.~M. Appleby.
\newblock Entanglement and designs.
\newblock {\em J. Phys. A}, 49:33LT02, 2016.

\bibitem{Graydon:2016}
M.~A. Graydon and D.~M. Appleby.
\newblock Quantum conical designs.
\newblock {\em J. Phys. A}, 49:085301, 2016.

\bibitem{Gross2011}
D.~Gross.
\newblock Recovering low-rank matrices from few coefficients in any basis.
\newblock {\em {IEEE} Trans. Inf. Theory}, 57(3):1548--1566, 2011.

\bibitem{Gross2010}
D.~Gross, Y.-K. Liu, S.~T. Flammia, S.~Becker, and J.~Eisert.
\newblock Quantum state tomography via compressed sensing.
\newblock {\em Phys. Rev. Lett.}, 105:150401, 2010.

\bibitem{Hardy2009}
G.~H. Hardy and E.~M. Wright.
\newblock {\em An Introduction to the Theory of Numbers}.
\newblock Oxford University Press, sixth edition, 2009.
\newblock Revised by D. R. Heath-Brown and J. H. Silverman.

\bibitem{Herman:2009}
M.~A. Herman and T.~Strohmer.
\newblock High-resolution radar via compressed sensing.
\newblock {\em Signal Processing, IEEE Transactions on}, 57:2275--2284, 2009.

\bibitem{Higham1987}
N.~J. Higham.
\newblock Error analysis of the {B}j{\"o}rck-{P}ereyra algorithms for solving
  {V}andermonde systems.
\newblock {\em Numer. Math.}, 50:613--632, 1987.

\bibitem{Hilbert:1900}
D.~Hilbert.
\newblock Mathematische probleme.
\newblock {\em G{\"{o}}tttinger Nachrichten}, pages 253--297, 1900.
\newblock Reprinted in \emph{Archiv der Mathematik und Physik} \textbf{3},
  44--63; 213--237 (1901). English translation in \emph{Bulletin of the
  American Mathematical Society} \textbf{8}, 437--479 (1902).

\bibitem{Hirzebruch:1953}
F.~Hirzebruch.
\newblock {\"{U}}ber vierdimensionale {R}iemannsche {F}l{\"{a}}chen
  mehrdeutiger analytischer {F}unktionen von zwei komplexen
  {V}er{\"{a}}nderlichen.
\newblock {\em Math. Ann.}, 126:1--22, 1953.

\bibitem{Hirzebruch:1973}
F.~Hirzebruch.
\newblock {H}ilbert modular surfaces.
\newblock {\em Enseign. Math.}, 19:183--282, 1973.

\bibitem{Hirzebruch:1971}
F.~Hirzebruch, W.~Neumann, and S.~Koh.
\newblock {\em Differentiable Manifolds and Quadratic Forms}.
\newblock Marcel Dekker, Inc. New York, 1971.

\bibitem{Hoggar1998}
S.~G. Hoggar.
\newblock 64 lines from a quaternionic polytope.
\newblock {\em Geom. Dedicata}, 69(3):287--289, 1998.

\bibitem{Horodecki:2022}
P.~Horodecki, {\L}.~Rudnicki, and K.~{\.{Z}}yczkowski.
\newblock Five open problems in quantum information.
\newblock {\em PRX Quantum}, page 010101, 2022.

\bibitem{jr}
J.~W. Jones and D.~P. Roberts.
\newblock A database of local fields.
\newblock {\em J. Symbolic Comput.}, 41(1):80--97, 2006.

\bibitem{Jung:1908}
H.~W.~E. Jung.
\newblock Darstellung der {F}unktionen eines algebraischen {K}{\"{o}}rpers
  zweier unabh{\"{a}}ngigen {V}er{\"{a}}nderlichen $x$, $y$ in der {U}mgebung
  einer {S}telle $x= a$, $y= b$.
\newblock {\em J. Reine Angew. Math.}, 1908:289--314, 1908.

\bibitem{Kashaev:1997}
R.~M. Kashaev.
\newblock The hyperbolic volume of knots from the quantum dilogarithm.
\newblock {\em Lett. Math. Phys.}, 39:269--275, 1997.

\bibitem{Katok:1996}
S.~Katok.
\newblock Coding of closed geodesics after gauss and morse.
\newblock {\em Geom. Dedicata}, 63:123--145, 1996.

\bibitem{Katok:2003}
S.~Katok.
\newblock Continued fractions, hyperbolic geometry and quadratic forms.
\newblock In S.~Katok, A.~Sossinsky, and S.~Tabachnikov, editors, {\em MASS
  Selecta: Teaching and Learning Advanced Undergraduate Mathematics}, pages
  121--160. American Mathematical Society, 2003.

\bibitem{King:2021}
E.~J. King.
\newblock Constructing subspace packings from other packings.
\newblock {\em Linear Algebra Appl.}, 625:68--80, 2021.

\bibitem{Koch2000}
H.~Koch.
\newblock {\em Number Theory. Algebraic Numbers and Functions}.
\newblock Graduate Studies in Mathematics, vol. 24. American Mathematical
  Society, 2000.

\bibitem{Kopp2019}
G.~S. Kopp.
\newblock S{IC}-{POVM}s and the {S}tark conjectures.
\newblock {\em Int. Math. Res. Not. IMRN}, 2021(18):13812--13838, 2021.

\bibitem{Kopp2020d}
G.~S. Kopp.
\newblock The {S}hintani--{F}addeev modular cocycle: {S}tark units from
  $q$-{P}ochhammer ratios, 2024.
\newblock Available at
  \href{https://arxiv.org/abs/2411.06763}{arXiv:2411.06763}.

\bibitem{Kopp2020b}
G.~S. Kopp and J.~C. Lagarias.
\newblock Ray class groups and ray class fields for orders of number fields,
  2022.
\newblock Available at
  \href{https://arxiv.org/abs/2212.09177}{arXiv:2212.09177}.

\bibitem{Kopp2020c}
G.~S. Kopp and J.~C. Lagarias.
\newblock {SIC}s and orders of real quadratic fields, 2024.
\newblock Available at
  \href{https://arxiv.org/abs/2407.08048}{arXiv:2407.08048}.

\bibitem{Kopp2024}
G.~S. Kopp and J.~C. Lagarias.
\newblock Ray class monoids for orders of number fields, 2025.
\newblock Forthcoming.

\bibitem{Kubert:1981}
D.~S. Kubert and S.~Lang.
\newblock {\em Modular units}, volume 244 of {\em Grundlehren der
  Mathematischen Wissenschaften}.
\newblock Springer-Verlag, New York-Berlin, 1981.

\bibitem{Kurokawa:2003}
N.~Kurokawa and S.~Koyama.
\newblock Multiple sine functions.
\newblock {\em Forum Math.}, 15:839--876, 2003.

\bibitem{Lopez:2015}
F.~I.~B. L{\'{o}}pez, V.~N. Efremov, and A.~M.~H. Magdaleno.
\newblock Algorithm for fast calculation of {H}irzebruch-{J}ung continued
  fraction expansions to coding of graph manifolds.
\newblock {\em Applied Mathematics}, 6:1676, 2015.

\bibitem{Luzzi:2007}
L.~Luzzi and S.~Marmi.
\newblock On the entropy of {J}apanese continued fractions.
\newblock {\em Discrete and Continuous Dynamical Systems}, 20:673--711, 2007.

\bibitem{Marcus:1977}
D.~A. Marcus.
\newblock {\em Number fields}.
\newblock Universitext. Springer-Verlag, New York-Heidelberg, 1977.

\bibitem{Murakami:2018}
H.~Murakami and Y.~Yokota.
\newblock {\em Volume Conjecture for Knots}.
\newblock Springer Briefs in Mathematical Physics Vol. 30. Springer, 2018.

\bibitem{Myerson:1987}
G.~Myerson.
\newblock On semi-regular finite continued fractions.
\newblock {\em Arch. Math.}, 48:420--425, 1987.

\bibitem{Neukirch1999}
J.~Neukirch.
\newblock {\em Algebraic Number Theory}.
\newblock Springer Berlin Heidelberg, 1999.

\bibitem{Ponsot:2004}
B.~Ponsot.
\newblock Recent progress in {L}iouville field theory.
\newblock {\em Int. J. Mod. Phys. A}, 19(supp02):311--335, 2004.

\bibitem{Popescu-Pampu:2007}
P.~Popescu-Pampu.
\newblock The geometry of continued fractions and the topology of surface
  singularities.
\newblock In J.-P. Brasselet and T.~Suwa, editors, {\em Singularities in
  geometry and topology: : Proceedings of the third Franco-Japanese Symposium
  on Singularities, September 2004}, Advanced Studies in Pure Mathematics,
  Volume 46, pages 119--195. Mathematical Society of Japan, 2007.
\newblock Available at
  \href{https://arxiv.org/abs/math/0506432}{arXiv:math/0506432}.

\bibitem{Rademacher1955}
H.~Rademacher.
\newblock Zur {T}heorie der {D}edekindschen {S}ummen.
\newblock {\em Math. Z.}, 63:445--463, 1955.

\bibitem{Rademacher:1973}
H.~Rademacher.
\newblock {\em Topics in Analytical Number Theory}.
\newblock Die Grundlehren der mathematischen Wissenschaften. Springer, 1973.

\bibitem{Rademacher:1972}
H.~Rademacher and E.~Grosswald.
\newblock {\em Dedekind sums}.
\newblock The Carus Mathematical Monographs, No. 16. American Mathematical
  Society, 1972.

\bibitem{Renes2004}
J.~M. Renes, R.~Blume-Kohout, A.~J. Scott, and C.~M. Caves.
\newblock Symmetric informationally complete quantum measurements.
\newblock {\em J. Math. Phys.}, 45(6):2171--2180, June 2004.

\bibitem{Scott:2006}
A.~J. Scott.
\newblock Tight informationally complete quantum measurements.
\newblock {\em J. Phys. A}, 39:13507--13530, 2006.

\bibitem{Scott:2017}
A.~J. Scott.
\newblock {SIC}s: Extending the list of solutions, 2017.
\newblock Available at
  \href{https://arxiv.org/abs/1703.03993}{arXiv:1703.03993}.

\bibitem{Scott2010}
A.~J. Scott and M.~Grassl.
\newblock Symmetric informationally complete positive-operator-valued measures:
  A new computer study.
\newblock {\em J. Math. Phys.}, 51(4):042203, 2010, 0910.5784.

\bibitem{Shang:2018}
J.~Shang, A.~Asadian, H.~Zhu, and O.~G{\"u}hne.
\newblock Enhanced entanglement criterion via symmetric informationally
  complete measurements.
\newblock {\em Physical Review A}, 98:022309, 2018.

\bibitem{Shintani1976}
T.~Shintani.
\newblock On evaluation of zeta functions of totally real algebraic number
  fields at non-positive integers.
\newblock {\em J. Fac. Sci., Univ. Tokyo, Sect. IA}, 23:393--417, 1976.

\bibitem{Shintani1977}
T.~Shintani.
\newblock On a {K}ronecker limit formula for real quadratic fields.
\newblock {\em J. Fac. Sci. Univ. Tokyo}, 24:167--199, 1977.

\bibitem{Shintani1977b}
T.~Shintani.
\newblock On certain ray class invariants of real quadratic fields.
\newblock {\em J. Math. Soc. Japan}, 30:139--167, 1977.

\bibitem{Shintani1980}
T.~Shintani.
\newblock A proof of the classical kronecker limit formula.
\newblock {\em Tokyo J. Math.}, 3:191--199, 1980.

\bibitem{Stacey:2021}
B.~C. Stacey.
\newblock {\em A First Course in the Sporadic {SIC}s}.
\newblock Springer Briefs in Mathematical Physics, Vol. 41. Springer, 2021.

\bibitem{Stark1}
H.~M. Stark.
\newblock Values of {L}-functions at $s=1$. {I}. {L}-functions for quadratic
  forms.
\newblock {\em Adv. Math.}, 7(3):301--343, 1971.

\bibitem{Stark2}
H.~M. Stark.
\newblock {L}-functions at $s=1$. {II}. {A}rtin {L}-functions with rational
  characters.
\newblock {\em Adv. Math.}, 17(1):60--92, 1975.

\bibitem{Stark3}
H.~M. Stark.
\newblock {L}-functions at $s=1$. {III}. {T}otally real fields and {H}ilbert's
  twelfth problem.
\newblock {\em Adv. Math.}, 22(1):64--84, 1976.

\bibitem{Starkrealquad}
H.~M. Stark.
\newblock Class fields for real quadratic fields and {$L$}-series at 1.
\newblock In A.~Fr{\"{o}}hlich, editor, {\em Algebraic Number Fields
  ($L$-Functions and Galois Properties): Proceedings of a Symposium.}, pages
  355--374. New York: Academic Press, 1977.

\bibitem{Stark4}
H.~M. Stark.
\newblock {L}-functions at $s=1$. {IV}. {F}irst derivatives at $s=0$.
\newblock {\em Adv. Math.}, 35(3):197--235, 1980.

\bibitem{Szymusiak:2016}
A.~Szymusiak and W.~S{\l}omczy{\'{n}}ski.
\newblock Informational power of the {H}oggar {SIC-POVM}.
\newblock {\em Phys. Rev. A}, 94:012122, 2016.

\bibitem{Tangedal:2007}
B.~A. Tangedal.
\newblock Continued fractions, special values of the double sine function, and
  {S}tark units over real quadratic fields.
\newblock {\em J. Number Theory}, 124(2):291--313, 2007.

\bibitem{Taniguchi:2013}
T.~Taniguchi and F.~Thorne.
\newblock Secondary terms in counting functions for cubic fields.
\newblock {\em Duke Math. J.}, 162(13):2451--2508, 2013.

\bibitem{Tate:1981}
J.~Tate.
\newblock On {S}tark's conjectures on the behavior of {$L(s,\chi)$} at $s=0$.
\newblock {\em J. Fac. Sci. Univ. Tokyo Sect. IA Math.}, 28(3):963--978, 1981.

\bibitem{Tavakoli:2020}
A.~Tavakoli, I.~Bengtsson, N.~Gisin, and J.~M. Renes.
\newblock Compounds of symmetric informationally complete measurements and
  their application in quantum key distribution.
\newblock {\em Phys. Rev. Res.}, 2:043122, 2020.

\bibitem{ThorneCorrespondence}
F.~Thorne, 2021.
\newblock Personal correspondence.

\bibitem{Waldron:2018}
S.~F.~D. Waldron.
\newblock {\em An Introduction to Finite Tight Frames}.
\newblock Birkh{\"{a}}user, 2018.

\bibitem{WangCorrespondence}
J.~Wang, 2024.
\newblock Personal correspondence.

\bibitem{Woronowicz:2000}
S.~Woronowicz.
\newblock Quantum exponential function.
\newblock {\em Reviews in Mathematical Physics}, 12:873--920, 2000.

\bibitem{Yokoi:1968}
H.~Yokoi.
\newblock On real quadratic fields containing units with norm $-1$.
\newblock {\em Nagoya Math. J.}, 33:139--152, 1968.

\bibitem{Zagier:2007a}
D.~Zagier.
\newblock The dilogarithm function.
\newblock In P.~Cartier, P.~Moussa, B.~Julia, and P.~Vanhove, editors, {\em Les
  Houches School of Physics: Frontiers in Number Theory, Physics, and Geometry
  II: On Conformal Field Theories, Discrete Groups and Renormalization: Les
  Houches, France, March 9-21, 2003}, pages 3--65, 2007.

\bibitem{Zauner1999}
G.~Zauner.
\newblock {\em Quantendesigns -- {G}rundz\"{u}ge einer nichtkommutativen
  {D}esigntheorie}.
\newblock PhD thesis, University of Vienna, 1999.
\newblock Available in English translation as: G. Zauner, Quantum Designs:
  Foundations of a Noncommutative Design Theory, \textit{Int. J. Quant. Inf.},
  9(1):445--507, 2011.

\end{thebibliography}
